%% file: coarse_groups.tex
\documentclass[11pt,a4paper, twoside]{book}
\usepackage{newtxtext}
\usepackage[nottoc,notlot,notlof]{tocbibind}
\usepackage{amsthm}
\usepackage{newtxmath}
\usepackage[T1]{fontenc}	
\usepackage[utf8]{inputenc}			% can use \k{}

\usepackage{enumerate}
\usepackage[colorlinks=true,linkcolor=black, citecolor=black]{hyperref}  

\usepackage{imakeidx, multicol}
\makeindex[intoc, columns=3] 

\usepackage[intoc,refpage]{nomencl}
\makenomenclature

\usepackage{fullpage}

\usepackage{mathtools} 					%typsetting tools
\usepackage{microtype} 					%better looking pdf output
\usepackage{bbm}					%more blackboard symbols
\usepackage{bm}						%more boldface symbols
\usepackage[shortcuts]{extdash} 			% hyphen \=/

\usepackage{subcaption}					%necessary for subfigures
\usepackage{accents}
\usepackage{xspace}

\usepackage{xparse} % not needed for LaTeX 2020-10-01 or later

\usepackage{xcolor}
\usepackage{tikz-cd}

\input{ExtraCommands}

\input{CrsCommands}

\theoremstyle{definition}
\newtheorem*{warning}{Warning}
\newtheorem*{problem*}{Problem}
\theoremstyle{remark}
\newtheorem{move}{Move}

\newcommand{\editC}[1]{{\color{olive}#1}} 
\title{An Invitation to Coarse Groups}
\author{Arielle Leitner and Federico Vigolo}

\begin{document}

\frontmatter

\maketitle

\chapter*{Abstract}
% \begin{abstract}
 In this monograph we lay the foundation for a theory of coarse groups and coarse actions. Coarse groups are group objects in the category of coarse spaces, and can be thought of as sets with operations that satisfy the group axioms ``up to uniformly bounded error''. In the first part of this work, we develop the theory of coarse homomorphisms, quotients, and subgroups, and prove that coarse versions of the Isomorphism Theorems hold true. We also initiate the study of coarse actions and show how they relate to the fundamental observation of Geometric Group Theory.
 
 In the second part we explore a selection of specialized topics, such as the study of coarse group structures on set\=/groups, groups of coarse automorphisms and spaces of controlled maps. Here the main aim is to show how the theory of coarse groups connects with classical subjects. These include: number theory; the study of bi\=/invariant metrics on groups; quasimorphisms and stable commutator length; $\Out(F_n)$; topological group actions.
 
 We see this text as an invitation and a stepping stone for further research.
 
 \vspace{10 em}
 
 \begin{center}
  \texttt{Version 1.0.0}
 \end{center}

% \end{abstract}

% \clearpage

\tableofcontents

% \clearpage

\mainmatter

\include{Intro}

\include{BasicTheory}

\include{SelectedTopics}

\appendix %%%%%%%%%%%%%%%%%%%%%%%%

\include{Appendices}

\backmatter

\bibliographystyle{plain}  
\bibliography{BibCoarseGroups}

\printindex

\newpage

\printnomenclature

\end{document}

%% file: ExtraCommands.tex
\newcommand{\CCC}{\mathbb{C}}
\renewcommand{\AA}{\mathbb{A}}

		\newcommand{\NN}{\mathbb{N}}	
			
\newcommand{\QQ}{\mathbb{Q}}		\newcommand{\RR}{\mathbb{R}}	
\renewcommand{\SS}{\mathbb{S}}

		\newcommand{\ZZ}{\mathbb{Z}}	

		\newcommand{\CB}{\mathcal{B}}	
\newcommand{\CC}{\mathcal{C}}		\newcommand{\CD}{\mathcal{D}}	
\newcommand{\CE}{\mathcal{E}}		\newcommand{\CF}{\mathcal{F}}	
\newcommand{\CG}{\mathcal{G}}			
			
		\newcommand{\CL}{\mathcal{L}}	
			
		\newcommand{\CP}{\mathcal{P}}	
		\newcommand{\CR}{\mathcal{R}}	
			
\newcommand{\CU}{\mathcal{U}}

\newcommand{\fka}{\mathfrak{a}}		\newcommand{\fkb}{\mathfrak{b}}	
\newcommand{\fkc}{\mathfrak{c}}		\newcommand{\fkd}{\mathfrak{d}}

\newcommand{\fku}{\mathfrak{u}}

		\newcommand{\fkB}{\mathfrak{B}}	
\newcommand{\fkC}{\mathfrak{C}}			
		\newcommand{\fkF}{\mathfrak{F}}	
			
\newcommand{\fkI}{\mathfrak{I}}			
\newcommand{\fkK}{\mathfrak{F}}

\newcommand{\fkU}{\mathfrak{U}}

\newcommand{\ubar}[1]{\underaccent{\bar}{#1}} 		%\bar under letter

\newcommand{\myllangle}{\langle\!\langle}
\newcommand{\myrrangle}{\rangle\!\rangle}

\DeclarePairedDelimiter\abs{\lvert}{\rvert}		%absolute value
\DeclarePairedDelimiter\norm{\lVert}{\rVert}		%norm
\DeclarePairedDelimiter\angles{\langle}{\rangle}	% used in \scal
\DeclarePairedDelimiter\aangles{\myllangle}{\myrrangle}	% used in \scal
\DeclarePairedDelimiter\paren{(}{)}			%   These are to fix the spacing between 
\DeclarePairedDelimiter\braces{\{}{\}}			%   Also, use with * to get variable size
\DeclarePairedDelimiter\floor{\lfloor}{\rfloor}

	\newcommand{\bigparen}[1]{\paren[\big]{#1}}
	\newcommand{\Bigparen}[1]{\paren[\Big]{#1}}

	\newcommand{\bigbraces}[1]{\braces[\big]{#1}}
	\newcommand{\Bigbraces}[1]{\braces[\Big]{#1}}
\newcommand{\bigmid}{\mathrel{\big|}}			%bigger \mid
\newcommand{\Bigmid}{\mathrel{\Big|}}			%Bigger \mid

\DeclareMathOperator{\st}{{\rm St}} 			%for star

\DeclareMathOperator{\id}{id}				%identity
\DeclareMathOperator{\aut}{Aut}				%automorphisms
\DeclareMathOperator{\out}{Out}				%outer automorphisms

\DeclareMathOperator{\im}{Im}				%imginary part

\DeclareMathOperator{\diag}{diag}
\DeclareMathOperator{\diam}{diam}			%diameter
\DeclareMathOperator{\Hom}{Hom}		
\DeclareMathOperator{\Out}{Out}

\DeclareMathOperator{\GL}{GL}	
\DeclareMathOperator{\Gl}{GL}					
\DeclareMathOperator{\Sl}{SL}	
\theoremstyle{plain}
\newtheorem{thm}{Theorem}[section]				
\newtheorem{prop}[thm]{Proposition}		
\newtheorem{lem}[thm]{Lemma}			\newtheorem{lemma}[thm]{Lemma}			
\newtheorem{cor}[thm]{Corollary}		
\newtheorem{qu}[thm]{Question}			
\newtheorem{conj}[thm]{Conjecture}

\newtheorem{problem}[thm]{Problem}

\newtheorem*{thm*}{Theorem}			\newtheorem*{theorem*}{Theorem}		
\newtheorem*{prop*}{Proposition}		\newtheorem*{proposition*}{Proposition}
\newtheorem*{lem*}{Lemma}			\newtheorem*{lemma*}{Lemma}			
\newtheorem*{cor*}{Corollary}			\newtheorem*{corollary*}{Corollary}
\newtheorem*{qu*}{Question}			\newtheorem*{question*}{Question}
\newtheorem*{conj*}{Conjecture}			\newtheorem*{conjecture*}{Question}
\newtheorem*{fact*}{Fact}
\newtheorem*{claim*}{Claim}

\theoremstyle{definition}
\newtheorem{de}[thm]{Definition}		\newtheorem{definition}[thm]{Definition}	
	
		\newtheorem{convention}[thm]{Convention}

\newtheorem*{de*}{Definition}			\newtheorem{definition*}{Definition}	
\newtheorem*{notation*}{Notation}	
\newtheorem*{conv*}{Convention}			\newtheorem*{convention*}{Convention}

\theoremstyle{remark}
\newtheorem{rmk}[thm]{Remark}			\newtheorem{remark}[thm]{Remark}			
\newtheorem{exmp}[thm]{Example}			\newtheorem{example}[thm]{Example}

\DeclareMathAlphabet{\mathbit}{OT1}{cmr}{bx}{it}

%% file: CrsCommands.tex
\ExplSyntaxOn
\cs_new_protected:Nn \egreg_embolden_command:N
 {\cs_set_eq:cN { __egreg_ \cs_to_str:N #1 : } #1
  \cs_set_protected:Npn #1 { \bm { \use:c { __egreg_ \cs_to_str:N #1 : } } }}
\cs_new_protected:Nn \egreg_embolden_char:n
 {\exp_args:Nc \mathchardef { __egreg_#1: } = \mathcode`#1 \scan_stop:
  \cs_set_protected:cn { __egreg_#1_bold: } { \bm { \use:c { __egreg_#1: } } }
  \char_set_active_eq:nc { `#1 } { __egreg_#1_bold: }
  \mathcode`#1 = "8000 \scan_stop:}
\cs_new_protected:Nn \egreg_embolden:
 {\clist_map_function:nN
   {% fill up
    A,B,C,D,E,F,G,H,I,J,K,L,M,N,O,P,Q,R,S,T,U,V,W,X,Y,Z,
    a,b,c,d,e,f,g,h,i,j,k,l,m,n,o,p,q,r,s,t,u,v,w,x,y,z,
    *,<,>,/,-,1,
   }
   \egreg_embolden_char:n
  \clist_map_function:nN
   {% fill up
    \alpha,\beta,\gamma,\delta,\epsilon,\varepsilon,\zeta,\eta,\theta,\vartheta,\iota,\kappa,\lambda,\mu,\nu,\xi,\pi,\varpi,\rho,\varrho,\sigma,\varsigma,\tau,\upsilon,\phi,\varphi,\chi,\psi,\omega,
    \Gamma,\varGamma,\Delta,\varDelta,\Theta,\varTheta,\Lambda,\varLambda,\Xi,\varXi,\Pi,\varPi,\Sigma,\varSigma,\Upsilon,\varUpsilon,\Phi,\varPhi,\Psi,\varPsi,\Omega,\varOmega,
    \cup,\cap,
    \neq,\cong,\ncong,
    \in,
    \subset,\subseteq,\supset,\supseteq,\subsetneq,\supsetneq,\nsubseteq,\nsupseteq,
    \leq,\geq,\lneq,\gneq,\lhd,\rhd,\trianglelefteq,\trianglerighteq,\triangleright,\trianglerighteq,
   }
   \egreg_embolden_command:N}
\NewDocumentCommand{\crse}{m}
 {\group_begin:
  \egreg_embolden:
  #1
  \group_end:}
\ExplSyntaxOff

\newcommand*{\oldcrse}[1]{\bm{#1}}%\mathbit{#1}}
\newcommand{\cid}{\mathbf{id}}
\newcommand{\cop}{\mathbin{\raisebox{-0.1 ex}{\scalebox{1.2}{$\boldsymbol{\ast}$}}}}
\newcommand{\cact}{\mathbin{\raisebox{-0.1 ex}{\scalebox{1.2}{$\boldsymbol{\cdot}$}}}}
\newcommand{\ccirc}{\mathbin{\raisebox{-0.1 ex}{\scalebox{1.2}{$\boldsymbol{\circ}$}}}}%{\bm\circ}
\DeclareMathOperator{\chom}{Hom_{\mathbf{Crs}}}	
\DeclareMathOperator{\cAut}{Aut_{\mathbf{Crs}}}	
\DeclareMathOperator{\cmor}{Mor_{\mathbf{Crs}}}	
\DeclareMathOperator{\fcmor}{Mor_{\mathbf{FrCrs}}}	%frag crse
\newcommand*{\mor}[1]{\operatorname{Mor}_{\Cat{#1}}}
\DeclareMathOperator{\cim}{\mathbf{Im}}
\DeclareMathOperator{\cker}{\mathbf{ker}}

\newcommand*{\cinversefn}{[\mhyphen]^{\oldcrse{-1}}}
\newcommand*{\cunit}{{\oldcrse{\unit}}} 

\newcommand{\tobj}{{\braces{\rm pt}}}
\newcommand{\terobj}{\tobj}

\newcommand{\assRel}{E_{\rm as} }
\newcommand{\invRel}{E_{\rm inv} }
\newcommand{\idRel}{E_{\rm id} }
\newcommand{\cmp}{\circ}
\newcommand*{\op}[1]{#1^{\scriptscriptstyle\rm op}}

\newcommand{\pts}[1]{{\mathfrak{i}}_{#1}}

\newcommand{\mingrp}{\varcrs[grp]{min}}

\newcommand{\mincrs}{\CE_{\rm min}}
\newcommand{\maxcrs}{\CE_{\rm max}}
\makeatletter
\def\varcrs{\@ifnextchar[{\@withvarcrs}{\@withoutvarcrs}}
\def\@withvarcrs[#1]#2{\CE_{\rm #2}^{\rm #1}}
\def\@withoutvarcrs#1{\CE_{\rm #1}}
\def\varFcrs{\@ifnextchar[{\@withvarFcrs}{\@withoutvarFcrs}}
\def\@withvarFcrs[#1]#2{\CF_{\rm #2}^{\rm #1}}
\def\@withoutvarFcrs#1{\CF_{\rm #1}}
\makeatother

\newcommand*{\ctrl}[2]{#2^{#1}_{\rm ctr}}
\newcommand*{\ctreq}[2]{{\rm CE}(#1,#2)}

\newcommand*{\ctraut}[1]{{\rm CE}(#1)}
\newcommand*{\bijctraut}[1]{{\rm BijCE}(#1)}

\newcommand*{\rel}[1]{\mathrel{{\leftarrow}\boxed{\vphantom{|}\,{\scriptstyle #1}\,}{\rightarrow}}}
\newcommand*{\torel}[1]{\mathrel{{\mhyphen}\boxed{\vphantom{|}\,{\scriptstyle #1}\,}{\rightarrow}}} 
\newcommand{\ceq}{\asymp} 
\newcommand{\csub}{%\underaccent{\tilde}{\subset}}
                    \preccurlyeq}

\newcommand{\closefn}{\approx}

 \DeclareMathOperator{\rd}{red}

\DeclareMathOperator{\qOut}{QOut}	
\newcommand{\Ad}{{\rm Ad}}

\newcommand*{\inversefn}{(\mhyphen)^{-1}}
\newcommand*{\unit}{{e}} 

\newcommand{\dblact}{{}_{=}\alpha}

\newcommand{\Cat}[1]{\ifmmode \text{\normalfont \textbf{#1}} \else {\normalfont \textbf{#1}}\fi}
\newcommand{\mhyphen}{\textnormal{-}}
\newcommand{\variable}{\,\mhyphen\,}
\newcommand*{\ev}{{\rm ev}}

%% file: Intro.tex
\chapter{Introduction} %%%%%%%%%%%%%%%%%%%%%%%
\label{ch:introduction}

\section{Background and Motivation} 

{\bf Coarse geometry.}
Coarse geometry is the study of the large scale geometric features of a space or, more precisely, of those properties that are invariant up to ``uniformly bounded error''.
This language allows us to formalize the idea that two spaces such as $\ZZ^n$ and $\RR^n$ look alike when seen ``from very far away''. In other words, we imagine looking at spaces through a lens that blurs the picture so that all the fine details become invisible and we only recognize the ``coarse'' shape of it. There are many reasons why one may wish to do this. For example, to use topological/analytic techniques on a discrete space like $\ZZ$ by identifying it with a continuous one like $\RR$ or, vice versa, to use combinatorial/algebraic methods on continuous spaces.

The language of coarse geometry allows us to study any space with a notion of ``uniform boundedness'', topological or otherwise. It is convenient to start exploring this idea in the context of metric spaces because the metric provides us with an intuitive meaning of uniform boundedness. Namely, a family $(B_i)_{i\in I}$ of subsets of a metric space $(X,d_X)$ is \emph{uniformly bounded} if there is a uniform upper bound on their diameters: $\diam(B_i)\leq C$ for every $i\in I$.
When this is the case, any two sequences $(x_i)_{i\in I}$ and $(y_i)_{i\in I}$ with $x_i,y_i\in B_i$ coincide up to uniformly bounded error and are hence ``coarsely indistinguishable''.
Analogously, we say that two functions $f_1,f_2\colon Z\to X$ coincide up to uniformly bounded error (a.k.a. they are \emph{close}) if there exists $C\in\RR$ so that $d_X(f_1(z),f_2(z))\leq C$ for all $z\in Z$. Again, we think of close functions as coarsely indistinguishable from one another.

Given two metric spaces $(X,d_X), (Y,d_Y)$ we only wish to consider the functions $f\colon X\to Y$ that preserve uniform boundedness, \emph{i.e.}~which send uniformly bounded subsets of $X$ to uniformly bounded subsets of $Y$. Explicitly, this happens when there exists some increasing control function $\rho\colon [0,\infty)\to [0,\infty)$ such that
\(
 d_Y(f(x),f(x'))\leq \rho\bigparen{d_X(x,x')}
\)
for every $x,x'\in X$. When this is the case, we say that $f\colon X\to Y$ is a \emph{controlled function}. %Putting things together, we may now say that t
Two metric spaces are \emph{coarsely equivalent} if there exist controlled functions $f\colon X\to Y$ and $g\colon Y\to X$ whose compositions $f\circ g$ and $g\circ f$ are close to $\id_Y$ and $\id_X$ respectively. For instance, we now see that $\ZZ^n$ and $\RR^n$ are coarsely equivalent via the natural inclusion and the Cartesian power of the integer floor function.

\

To provide some context and motivation, we note that the coarse geometry of metric spaces is the backbone of Geometric Group Theory. Any finitely generated group can be seen as a discrete geodesic metric space by equipping it with a word metric (equivalently, identifying the group elements with the vertices of a Cayley graph considered with the graph metric). These metrics depend on the choice of a finite generating set. However, they are unique up to coarse equivalence.
As a consequence, it is possible to characterize algebraic properties of groups in terms of coarse geometric invariants of their Cayley graphs. Two outstanding such results are Gromov's theorem showing that a finitely generated group is virtually nilpotent if and only if it has polynomial growth, and Stallings' result that a group has a non trivial splitting over a finite subgroup if and only if its Cayley graph has more than one end.

The Milnor--Schwarz Lemma illustrates another fundamental feature of coarse equivalences. This result states that a finitely generated group $G$ acting properly cocompactly on a metric space $X$ is coarsely equivalent to $X$. For example,  if $G=\pi_1(M)$ is the fundamental group of a compact Riemannian manifold, then the action by deck transformations induces a coarse equivalence between $G$ and the universal cover $\widetilde M$. This important observation provides a seamless dictionary between algebraic properties of groups, their actions, and geometry of metric spaces. This correspondence is one of the main raisons d'être of Geometric Group Theory: the ability to conflate groups and spaces on which they act proves to be a very powerful tool. To mention one exceptional achievement: Agol and Wise exploited this philosophy to prove the virtual Haken Conjecture, a long standing open problem in low\=/dimensional topology.

Other important applications have a more analytic flavor. For instance, Roe proved a version of the Atiyah--Singer index theorem for open manifolds and subsequently formulated a coarse version of the Baum--Connes conjecture. G. Yu proved that a huge class of manifolds satisfies this coarse Baum--Connes Conjecture and, as a consequence, the Novikov Conjecture.
Such a brief introduction cannot comprehensively cover all of coarse geometry and its role in modern mathematics, we therefore redirect the interested reader to more specialised resources
(we recommend \cite{benjamini2013coarse,bridson_metric_2013,cornulier2016metric,de_la_harpe2000topics,drutu_geometric_2018,nowak2012large,roe_lectures_2003,rosendal2017coarse} or the original works of Gromov \cite{gromov_hyperbolic_1987,niblo_geometric_1993,gromov_random_2003,gromov_spaces_2010}).

\begin{rmk}
 In Geometric Group Theory one generally speaks about quasi\=/isometries rather than coarse equivalences. However, it is an easy exercise to show that a controlled function between geodesic metric spaces is coarsely Lipschitz (\emph{i.e.}~the control function $\rho$ can be assumed to be affine $\rho(t)=Lt+A$). Since Cayley graphs are geodesic metric spaces, a coarse equivalence between finitely generated groups is indeed a quasi\=/isometry.
\end{rmk}

As already mentioned, the coarse geometric perspective is useful beyond the realm of metric spaces. For instance, to identify a topological group with any of its cocompact lattices. In this general setting however, it is more complicated to understand what the meaning of ``uniformly bounded'' should be. Fortunately, Roe \cite{roe_lectures_2003} developed an elegant axiomatization of the notion of uniform boundedness in terms of what he calls \emph{coarse structures}.
\footnote{%
Independently, an alternative formalism serving the same purpose was introduced by Protasov and Banakh on the same year \cite{protasov2003ball}. This theory gained traction and is often used in the Ukrainian school.
}

To better understand Roe's approach, it is useful to revisit the metric example.
Given a metric space $(X,d)$, equip the product $X\times X$ with the $\ell_1$\=/metric $d\paren{(x_1,x_2),(x_1',x_2')}\coloneqq d(x_1,x_1')+d(x_2,x_2')$. With this choice of metric, the distance between two points $x,y\in X$ is equal to the distance of the pair $(x,y)$ from the diagonal $\Delta_X\subset X\times X$. It follows that a family $(B_i)_{i\in I}$ of subsets of $X$ is uniformly bounded if and only if the union $\bigcup_{i\in I} B_i\times B_i$ is contained in a bounded neighborhood of $\Delta_X$. This suggests that in order to define a notion of uniform boundedness on a set $X$ it is sufficient to describe the ``bounded neighborhoods'' of the diagonal $\Delta_X\subset X\times X$.

Taking this point of view, a \emph{coarse structure} on a set $X$ is a family $\CE$ of subsets of $X\times X$ satisfying certain axioms that generalize the triangle inequality (Definition~\ref{def:p1:coarse.structure}). These sets ought to be thought of as the bounded neighborhoods of $\Delta_X$. A family $(B_i)_{i\in I}$ of subsets $X$ is uniformly bounded (a.k.a. \emph{controlled}) if the union $\bigcup_{i\in I} B_i\times B_i$ is contained in some $E\in \CE$. The axioms on $\CE$ are chosen so that subsets and uniformly bounded neighborhoods of uniformly bounded sets remain uniformly bounded. A \emph{coarse space} is a set $X$ equipped with a coarse structure $\CE$.
In other words, the slogan is that a coarse space is a space where it makes mathematical sense to say that some property holds ``up to uniformly bounded error''.

The functions between coarse spaces that we wish to consider are those functions that preserve uniform boundedness. Explicitly, if $(X,\CE)$, $(Y,\CF)$ are coarse spaces, we are interested in those functions $f\colon X\to Y$ such that $f\times f\colon X\times X\to Y\times Y$ sends elements in $\CE$ to elements in $\CF$. Such functions are called \emph{controlled}.
The reader may have realized that the definition for a coarse space is a direct analog of the classical notion of uniform space \cite{bourbaki1966general}. This is no surprise: both languages seek to axiomatize some notion of ``uniformity'' whose meaning is clear for metric spaces but elusive in general. Uniform spaces extend the idea of ``uniformly small'' and uniform continuity; coarse spaces extend the idea of ``uniformly bounded''.

As we did for metric spaces, we now wish to identify ``close functions''. Two functions between coarse spaces $f,f'\colon (X,\CE)\to (Y,\CF)$ are \emph{close} if $f(x)$ and $f'(x)$ are uniformly close as $x$ varies, \emph{i.e.}~the sets $\{f(x),f'(x)\}$ are uniformly bounded.  A \emph{coarse map} between coarse spaces $(X,\CE)$, $(Y,\CF)$ is an equivalence class $[f]$ of controlled functions $X\to Y$, where functions are equivalent if they are close. A \emph{coarse equivalence} is a coarse map $[f]\colon (X,\CE)\to (Y,\CF)$ with a \emph{coarse inverse}, \emph{i.e.}~a coarse map $[g]\colon (Y,\CF)\to (X,\CE)$ with $[f]\circ [g]=[\id_Y]$ and $[g]\circ[f]=[\id_X]$.
Two coarse spaces are \emph{coarsely equivalent} if there is a coarse equivalence between them. Coarse geometry may be thought of as the study of properties of coarse spaces that are defined up to coarse equivalence.

To close the circle of ideas and conclude this brief introduction to coarse geometry, we remark that every metric space has a natural \emph{metric coarse structure} $\CE_d$ (consisting of all subsets of bounded neighborhoods of the diagonal).
In this case, it is easy to verify that the notions of uniformly bounded, controlled, close and coarse equivalence for the coarse space $(X,\CE_{d_X})$ coincide with the corresponding notions for the metric space $(X,d_X)$.

\

\noindent{\bf Coarse groups.}
The \emph{category of coarse spaces} \Cat{Coarse} is the category with coarse spaces as objects and coarse maps as morphisms. From now on, we will use bold italic symbols to denote anything related to \Cat{Coarse}, so that $\crse{X}=(X,\CE)$ denotes a coarse space and $\crse{f\colon X\to Y}$ denotes a coarse map (\emph{i.e.}~an equivalence class of controlled maps $f\colon (X,\CE)\to (Y,\CF)$).
In one sentence, a \emph{coarse group} is a group object in the category of coarse spaces.
More concretely, a coarse group is a coarse space $\crse G=(G,\CE)$ with a unit $e\in G$ and coarse maps $\crse{\ast\colon G\times G\to G}$ and $\cinversefn\colon \crse{G\to G}$ that obey the usual group laws ``up to uniformly bounded error''.

It is an empirical observation that the group objects in an interesting category turn out to be very interesting in their own right. For example: algebraic groups, Lie groups and topological groups are all defined as group objects in their relevant categories.
Given this premise, it is natural to be interested in coarse groups!
One bewildering feature we came to appreciate while developing this work is that even basic questions regarding coarse groups end up being connected with the most disparate topics. To mention a few, we see that the study of coarse structures on $\ZZ$ quickly takes us in the realm of number theory; a similar study on more general groups connects with questions regarding the existence of bi\=/invariant metrics. Investigating non\=/abelian free groups leads us towards the theory of (stable) commutator length and $\Out(F_n)$. Banach spaces give rise to all sort of interesting examples, and, to conclude, the study of coarse homomorphisms between coarse groups is a natural generalization of the theory of quasimorphisms. We will provide a more detailed explanation of these aspects in the following sections.

To avoid confusion, from now on we will always add an adjective to the word ``group'' to specify the category we are working in. We will thus have coarse groups, topological groups, algebraic groups, and---rather unconventionally---we will use the term \emph{set\=/group} to denote groups in the classical sense, which are group objects in the category of sets.

\

\noindent{\bf Coarse actions.}
Since their conception by Galois, set\=/group actions have been just as important as the set\=/groups themselves. It is then natural to try to understand what actions should look like in the category of coarse spaces. 
Of course, there is a natural definition of \emph{coarse action} of a coarse group $\crse G$ on a coarse space $\crse Y$.
Namely, it is a coarse map $\crse{\alpha}\colon\crse G\times\crse Y\to \crse Y$ that satisfies the usual axioms for actions of set\=/groups up to uniformly bounded error.

This definition directly generalizes other classical notions, such as that of quasi\=/action.
Apart from its naturality, one interesting aspect of the theory of coarse actions is that it allows us to look at classical results from a different perspective. 
In a sense that we will make precise later on, there is a strong connection between coarse actions and left invariant coarse structures.   
Using this shift of perspective, we can then realize that the Milnor--Schwarz  Lemma is more or less equivalent to the observation that left invariant coarse structures on a set\=/group $G$ are determined by the bounded neighborhoods of the unit $e\in G$. 

\

\noindent{\bf Purpose and motivation.} 
The goal of this monograph is to lay the foundations for a theory of coarse groups. In a large part, we wrote this text as a service to the community: rather than pushing towards a specific application, we made an effort to develop a self\=/contained exposition with emphasis on examples. We also point at some of the many connections between this topic and other subjects, and we pose various problems and questions for future research.

On a first read, it may be helpful to imagine all coarse spaces as metric spaces: the case of metric coarse spaces already contains many of the conceptual difficulties of the general theory of coarse spaces. 
However, we found that Roe's formalism of coarse spaces was of great help in identifying the crucial points in many arguments.

Another reason for embracing Roe's language is that it generalizes effortlessly to various other settings and constructs. This gives rise to a richer, more flexible theory.
For example, any abelian topological group admits a natural coarse structure where the controlled sets are contained in uniformly compact neighborhoods of the diagonal. These topological coarse structures are such that any abelian topological group is coarsely equivalent to any of its cocompact lattices. One significant case is that of the group of rational numbers: $\QQ$ is a cocompact lattice in the ring of adeles $\AA_\QQ$, hence $\QQ$ and $\AA_\QQ$ are isomorphic coarse groups (the former is seen as a discrete topological group, the latter has a natural locally compact topology). This fact is analogous to the obvious remark that $\ZZ$ and $\RR$ are coarsely equivalent and, using the formalism of coarse spaces, it is just as simple to prove.
Our interest in this subject was sparked by Uri Bader's realization that this observation allows for a unified proof of the fact that the spaces of ``quasi\=/endomorphisms'' of $\ZZ$ and $\QQ$ are naturally identified with $\RR$ and $\mathbb{A}_\mathbb{Q}$ respectively \cite{bader20XX}.

The general language of coarse spaces is also well suited for dealing with large (\emph{i.e.}~non finitely generated) groups. This is the same reason why Rosendal adopted it in his recent monograph on the coarse geometry of topological groups \cite{rosendal2017coarse}. Finally, we should also point out that the theory of coarse spaces is commonly used in the operator algebras community in relation with $C^*$\=/algebras of geometric origin. We shall not pursue these topics in this monograph, but we leave them for future work.

\

\section{On this Manuscript}
This text is organized in two separate parts. The first part is mostly concerned with elementary properties and constructions of coarse groups and coarse actions. Here we prove a number of fairly basic ``utility results'' and verify that the various pieces of the theory fit well together and work as expected. This part is for the most part self contained and the proofs are rather detailed. We also made an effort to develop this using both Roe's language of coarse structures \cite{roe_lectures_2003} and the approach in terms of ``controlled coverings'' à la Dydak--Hoffland \cite{dydak_alternative_2008}. In this regard, it is a recommended exercise to the reader to translate statements and proofs of various results from one language to the other. 

Of course, not all the material presented in this part is novel. Almost all the general facts about the category of coarse spaces can be found scattered in the literature. Also various ideas that are important in the development of the theory of coarse groups and coarse actions exist in a more or less implicit form in previous works.
% Besides Roe, other names of mathematicians who played a role in developing the theory of coarse geometry that we use in this text include Bader, Brodskiy, Dikranjan, Dydak, Hoffland, Mitra, Nicas, Protasov, Protasova, Rosendal, Rosenthal, and Zava.
More detailed citations are included throughout the text, but we are sure we missed some, for which we apologise in advance.
While writing this manuscript, we decided to recollect and introduce all the necessary facts and conventions to provide a unified framework with consistent nomenclature and notation.

\

The second part of the text has a rather different flavor. The goal here is to point at possible research directions and connections with more classical subjects. To do so, we selected and developed some specific topics according to our own taste. By design, this part is much less self\=/contained, as it is largely meant to connect with pre-existing literature. For the most part, these topics can be read independently from one another.
Here the proofs are generally briefer than in the first part. We also spend more time in introducing and discussing open questions.

This work only scratches the surface of a theory of coarse groups and their coarse actions. All the topics we considered led us towards all sorts of interesting mathematics: more than we could hope to do ourselves. Rather than as a complete treatment, we see this part of the manuscript as an invitation to further research.

\

We should point out that, despite being inspired by the categorical perspective, we decided to avoid making heavy use of the categorical language in the main body of this text.  Instead, we collected various categorical remarks in the appendix.

\section{List of Findings I: Basic Theory}
{\bf Coarse groups.}
By definition, if $\crse G=(G,\CE)$ is a coarse group and $\cop\colon \crse{G\times G\to G}$ is the coarse multiplication map, any of its representatives $\ast\colon G\times G\to G$ must be controlled. The first step in our study of coarse groups, is to understand when a multiplication function on a coarse space is controlled. To do so, it is fundamental to observe that the multiplication function $\ast$ is controlled if and only if both the left multiplication $(g\ast\mhyphen)$ and the right multiplication $(\mhyphen\ast g)$ are \emph{equi controlled} as $g$ varies in $G$, see Definition~\ref{def:p1:equi controlled}. %To have a feeling for this notion, the reader can keep in mind that 
The intuition is that a family of maps $(f_i)_{i\in I}$ between metric spaces is equi controlled if and only if the maps are all controlled with the same control function $\rho\colon [0,\infty)\to[0,\infty)$.

This elementary observation allows us to produce our first non trivial examples of coarse groups. In fact, if $G$ is a set\=/group equipped with a bi\=/invariant metric $d$, then the left and right multiplication are equi controlled with respect to $\CE_d$. It follows that the group operation $\ast$ is controlled. The inversion $\inversefn$ is an isometry, hence it is controlled as well. It then follows that $(G,\CE_d)$ is a coarse group with coarse operations $\cop$ and $\cinversefn$ (one may observe that the converse is also true: if $(G,\CE_d)$ is a coarse group then $d$ is coarsely equivalent to a bi\=/invariant metric).

\begin{rmk}
 An alternative, more categorical, point of view to the above remark is that groups with bi\=/invariant metrics are group objects in a category of metric spaces (see Appendix~\ref{ch:appendix:metric.groups}). Considering metric coarse structures defines a functor from this category to \Cat{Coarse}, and it is easy to observe that this functor sends group objects to group objects.
\end{rmk}

We can now provide concrete examples and begin to acquire a taste for a theory of coarse groups. The easiest examples are given by choosing an abelian set\=/group and giving it a left invariant metric. This of course includes $\ZZ^n$ and $\RR^n$ with their Euclidean metrics. More generally, we may consider the additive group of any normed ring or vector space, \emph{e.g.}~a Banach space.

On the non\=/abelian side, the word metric associated with a (possibly infinite) conjugation invariant generating set is always bi\=/invariant. If $G=\angles{S}$ is a finitely generated set\=/group, the bi\=/invariant word metric $d_{\overline{S}}$ associated with the conjugation invariant set $\overline{S}=\bigcup_{g\in G}gSg^{-1}$ does not depend on the choice of finite generating set $S$ up to coarse equivalence. That is, such a set\=/group $G$ is naturally equipped with a canonical metric coarse structure $\varcrs{bw}\coloneqq \CE_{d_{\overline{S}}}$ so that $(G,\CE_{d_{\overline{S}}})$ is a coarse group (one can show that this is the finest connected coarse structure equipped with which $G$ is a coarse group, Chapter~\ref{ch:p2:coarse structures on set groups}).

One fundamental example is given by the free group $F_S$ on $S$ generators. Here the bi\=/invariant word metric $d_{\overline{S}}$ can be seen as a ``cancellation'' metric. Namely, an alternative definition for $d_{\overline{S}}$ 
is obtained by declaring that for every word $w\in F_S$ the distance $d_{\overline{S}}(e,w)$ is the minimal number of letters of $w$ that it is necessary to cancel to make it equivalent to the trivial word (the letters can appear in any position in $w$). The metric is then determined by the identity $d_{\overline{S}}(w_1,w_2)=d_{\overline{S}}(e,w_1^{-1}w_2)$. As a coarse group, $(F_S,\CE_{d_{\overline{S}}})$ is considerably more complicated than the abelian examples discussed above because it is not even \emph{coarsely abelian} (\emph{i.e.}~the distance $d_{\overline{S}}(w_1w_2 ,w_2w_1)$ is not uniformly bounded as $w_1,w_2$ range in $F_S$).

\begin{rmk}\label{rmk:intro:trivially coarse}
 Every set\=/group $G$ can be made into a \emph{trivially coarse group} by declaring that the only bounded sets are singletons (\emph{i.e.}~$G$ is given the trivial coarse structure $\mincrs$ consisting only of diagonal subsets). In other words, the coarse group $(G,\mincrs)$ is not coarse at all, because no non\=/trivial error is a bounded error.
 
This degenerate example shows how to see the theory of coarse groups as an extension of the classical theory of set\=/groups: every ``coarse'' statement or definition reduces to the relevant classical notion when specialized to trivially coarse groups. 
 We shall soon see that an important part of Geometric Group Theory can be framed in the setting of coarse actions of trivially coarse groups.
\end{rmk}

\

The above examples of coarse groups are all obtained by equipping some set\=/group $G$ with a coarse structure $\CE$ so that $\crse G=(G,\CE)$ is a coarse group (\emph{i.e.}~$\crse G$ is a \emph{coarsified set\=/group}).
In other words, these are coarse groups whose coarse operations admit representatives that satisfy the group axioms ``on the nose''. However, this need not be the case in general. 
It is not hard to show that it is always possible to find \emph{adapted} representatives so that the identities $g\ast e= e\ast g=g$ and $g\ast g^{-1}= g^{-1}\ast g=e $ hold for every $g\in G$ (Lemma~\ref{lem:p1:adapted representatives exist}). On the other hand, there are coarse groups where every representative for $\cop$ fails to be associative. In view of this, one may also say that coarse groups are coarse spaces equipped with a coarsely associative binary operation that has a unit and inverses.

\begin{rmk}
 It is worth pointing out that the lack of associativity implies that even adapted representatives that satisfy the unit and inverse axiom may be rather nasty. In particular, $\inversefn$ need not be an involution and the left multiplication $(g^{-1}\ast\variable)$ need not be an inverse of the function $(g\ast\variable)$ (in fact, these functions might not even be bijective)!
\end{rmk}

We would like to mention two more properties of coarse groups. The first one is a useful fact that makes it easier to check whether a coarse space endowed with multiplication and inverse operations is a coarse group.
 Namely, if one already knows that the multiplication  function is controlled then the inverse is automatically controlled:

\begin{prop}[Proposition~\ref{prop:p1:coarse.group.iff.heartsuit.and.equi controlled}]
 Let $\crse G=(G,\CE)$ be a coarse space, $e\in G$ and $\ast,\inversefn$ be operations that coarsely satisfy the axioms of associativity, unit and inversion. If $\ast$ is controlled then $\inversefn$ is controlled as well and $\crse G$ endowed with the coarse operations $\cop$ and $\cinversefn$ is a coarse group.
\end{prop}

\begin{rmk}
 This should be contrasted with the topological setting: it is not true that if a set\=/group $G$ is given a topology $\tau$ so that the multiplication function is continuous then $(G,\tau)$ is a topological group. In fact, there are examples where the inversion is not continuous.
\end{rmk}

It is an important property of coarse groups that their coarse structure is determined ``locally''. 
Informally, this means that a family of subsets $U_i\subseteq G$ of a coarse group $\crse G$ is uniformly bounded if and only if there is some bounded set $B$  such that each $U_i$ is contained in some translate of $B$. A more formal statement is analogous to the fact that the topology of a topological group is completely determined by the set of open neighborhoods of the identity:

\begin{prop}[Proposition~\ref{prop:p1:neighbourhoods.of.identity.generate}]\label{prop:intro:det locally}
 Let $\CE,\CE'$ be two coarse structures on a set $G$, and assume that $\ast$, $e$, $\inversefn$ are fixed so that both $(G,\CE)$ and $(G,\CE')$ are coarse groups. Then $\CE\subseteq\CE'$ if and only if every $\CE$\=/bounded set containing $e$ is $\CE'$\=/bounded.
\end{prop}

\begin{rmk}
 Given a set\=/group $G$, it is not hard to explicitly characterize whether a certain family $(U_i)_{i\in I}$ of subsets of $G$ is the family of all $\CE$\=/bounded subsets of $G$ for some coarse structure $\CE$ so that $(G,\CE)$ is a coarse group. Namely, this is the case if and only if $(U_i)_{i\in I}$ is closed under taking subsets, inverses, finite products and arbitrary unions of conjugates (Subsection~\ref{sec:p1:determined locally}).
Since products of compact sets are compact, it follows that every abelian topological group $(G,\tau)$ can be given a unique coarse structure $\CE_\tau$ so that the $\CE_\tau$\=/bounded sets are the relatively compact subsets of $G$ and $(G,\CE_\tau)$ is a coarse group.
For example, this defines a topological coarsification of the additive group of the ring of adeles $\AA_\QQ$. With this coarsification, $\AA_\QQ$ is isomorphic (as a coarse group) to the topological coarsification of the discrete group $\QQ$.
\end{rmk}

\

\noindent{\bf Coarse homomorphisms, subgroups and quotients.}
A \emph{coarse homomorphism} between coarse groups is a coarse map $\crse{f\colon G\to H}$ that is compatible with the coarse group operations. Equivalently, if $f\colon G\to H$ is a representative for $\crse f$ then the image $f(g_1\ast_G g_2)$ must coincide with the product $f(g_1)\ast_H f(g_2)$ up to uniformly bounded error.

Quasimorphisms are an example of coarse homomorphisms. Recall that a $\RR$\=/quasimorphism of a  set\=/group $G$ is a function $\phi\colon G\to \RR$ such that $\abs{\phi(g_1g_2)-\phi(g_1)-\phi(g_2)}$ is uniformly bounded. Let $\CE_{\abs{\mhyphen}}$ be the metric coarse structure on $\RR$ associated with the Euclidean metric.
If $\crse G=(G,\CE)$ is a coarsified set\=/group, then $\crse \phi\colon \crse G\to(\RR,\CE_{\abs{\mhyphen}})$ is a coarse homomorphism as soon as the $\RR$\=/quasimorphism $\phi$ is controlled. 
An examination of the rich literature on $\RR$\=/quasimorphisms and bounded cohomology shows that coarse homomorphisms can differ greatly from usual homomorphisms of set\=/groups.

\begin{rmk}
 The fact that coarse structures on coarse groups are determined locally has some useful consequences for their coarse homomorphisms. Namely:
 \begin{itemize}
  \item Let $\crse G=(G,\CE)$, $\crse H=(H,\CF)$ be coarse groups and let $f\colon G\to H$ be a function so that $f(g_1\ast_G g_2)$ and $f(g_1)\ast_H f(g_2)$ (resp. $f(e_G)$ and $e_H$) coincide up to uniformly bounded error. If $f$ sends bounded sets to bounded sets then it is automatically controlled, hence $\crse f$ is a coarse homomorphism (Proposition~\ref{prop:p1:bornologous map is coarse hom}).
  \item If $\crse{f\colon G\to H}$ is a proper coarse homomorphism (\emph{i.e.}~so that preimages of bounded sets are bounded) then it is a \emph{coarse embedding} (Corollary~\ref{prop:p1:proper.hom.coarse.embedding}).
 \end{itemize} 

\end{rmk}

Naturally, two coarse groups are said to be \emph{isomorphic} if there exist coarse homomorphisms $\crse{f\colon G\to H}$ and $\crse{g\colon H\to G}$ that are coarse inverses of one another.

\

At this point the theory becomes more subtle.
We wish to declare that a coarse subgroup of a coarse group $\crse{G}$ is the image of a coarse homomorphism $\crse{f\colon H\to G}$. 
However, coarse homomorphisms are only defined up to some equivalence relation, hence so are their images. Explicitly, we define the equivalence relation of \emph{asymptoticity} (Definition~\ref{def:p1:asymptotic})  on the set of subsets of a coarse space and we define the \emph{coarse image} of a coarse map $\crse {f\colon X \to Y}$ to be the equivalence class of $f(X)\subseteq Y$ up to asymptoticity. 
For intuition, the reader should keep in mind that if $\crse X=(X,\CE_d)$ is a metric coarse space then two subsets of $X$ are asymptotic if and only if they are a finite Hausdorff distance from one another.

A \emph{coarse subgroup} of $\crse G$ (denoted $\crse{H\leq G}$) is the coarse image of a coarse homomorphism to $\crse G$. 
Alternatively, we can characterize the coarse subgroups of $\crse G=(G,\CE)$ as those equivalence classes $[H]$ where $H\subseteq G$ is a subset so that $[H]=[H\ast_G H]$ and $[H]=[H^{-1}]$ (Proposition~\ref{prop:p1:closed.asymp.classes.are.subgroups}).
Strictly speaking, a coarse subgroup $\crse{ H\leq G}$ is not a coarse group because the underlying set is only defined up to equivalence. However, whenever a representative is fixed $\crse H$ is indeed a coarse group. Moreover, different choices of representatives yield canonically isomorphic coarse groups. We are hence entitled to forget about this subtlety and treat coarse subgroups as coarse groups.

Coarse subgroups can be used to appreciate some interesting behavior in the theory of coarse groups. For example, let $V\subset\RR^d$ be a $k$\=/dimensional affine subspace and let $H\subset \ZZ^d$ be the subset of lattice points that are within distance $\sqrt k$ from $V$. Then $H$ and $V$ are asymptotic subsets of $(\RR^d,\CE_{\norm{\mhyphen}})$. Since $V+V$ is at finite Hausdorff distance from $V$, the equivalence class $[V]$ is a coarse subgroup of $(\RR^d,\CE_{\norm{\mhyphen}})$. It follows that $[H]$ is a coarse subgroup of $(\ZZ^d,\CE_{\norm{\mhyphen}})$ as well. A peculiar feature of $[H]$ is that one can choose $V$ with an irrational slope so that $H$ is not asymptotic to any subgroup of $\ZZ^d$. It is also worthwhile noting that quasicrystals in $\RR^d$ give rise to coarse subgroups as well. We should point out that, although mildly puzzling, these coarse subgroups are not very interesting as coarse group, as they are always isomorphic to $(\RR^k,\CE_{\norm{\mhyphen}})$.

Other exotic examples can be obtained by considering infinite approximate subgroups of set\=/groups. Recall that an approximate subgroup is a subset $A$ of a set\=/group $G$ so that $A=A^{-1}$ and the set of products $AA$ is contained in finitely many translates of $A$ (\emph{i.e.}~$AA\subseteq XA$ for some $X\subset G$ finite). Now, if $\CE$ is a coarse structure on $G$ so that $(G,\CE)$ is a coarse group and every finite subset of $G$ is $\CE$\=/bounded, it follows that for any approximate subgroup $A\subset G$ the asymptotic class $[A]$ is a coarse subgroup of $(G,\CE)$. We will later return to these ideas to construct coarse subgroups of $(F_2,\CE_{d_{\overline{S}}})$.

\

Quotients of coarse groups are obtained by enlarging the coarse structure. Morally, the point of coarse geometry is that points in a coarse space can only be distinguished up to uniformly bounded error. If we enlarge the coarse structure, fewer points can be distinguished, therefore this ``coarser'' coarse space is for all intents and purposes a quotient of the original space.

We then say that a \emph{coarse quotient} of a coarse group $\crse G=(G,\CE)$ is a coarse group $(G,\CF)$ so that $\CE\subseteq \CF$. This definition behaves as expected: a coarse group $\crse Q$ is isomorphic to a coarse quotient of $\crse G$ if and only if there is a coarsely surjective coarse homomorphism $\crse{q\colon G\to Q}$.

\begin{rmk}\label{rmk:intro:silly quotient}
Notice that unlike the case of set\=/groups, a coarse quotient of a coarse group need not be a ``quotient by a coarse subgroup''. For example, the coarse group $(\ZZ,\CE_{\abs{\mhyphen}})$ is a coarse quotient of the trivially coarse group $(\ZZ,\mincrs)$, however there is no coarse subgroup $[H]<(\ZZ,\mincrs)$ so that $(\ZZ,\CE_{\abs{\mhyphen}})$ is obtained by ``modding out $H$'' (recall that the $\mincrs$ is the coarse structure so that points coarsely coincide only if they they are the same, Remark~\ref{rmk:intro:trivially coarse}).
 
 It is still true that if $\crse {N\trianglelefteq G}$ is a \emph{coarsely normal} coarse subgroup of $\crse G$, there is a well\=/defined coarse quotient $\crse{G/N}$ (see below for more on this).
\end{rmk}

\

\noindent{\bf Coarse kernels and Isomorphism Theorems.}
Recall the coarse image of a coarse map is only defined up to asymptoticity. Similarly, we expect that the preimage of a coarse subspace should only be defined up to equivalence. However, the situation is even more delicate, as coarse preimages may not exist at all! 

The reason for this is present in Remark~\ref{rmk:intro:silly quotient}. If we consider the coarse map $\cid_\ZZ\colon(\ZZ,\mincrs)\to (\ZZ,\CE_{\abs{\mhyphen}})$, we see that any $\CE_{\abs{\mhyphen}}$\=/bounded set can be seen as a reasonable preimage of $0$. However, these sets are not asymptotic according to the trivial coarse structure. We are hence unable to make a canonical choice of a coarse subspace of $(\ZZ,\mincrs)$ to play the role of the coarse preimage of $[0]\crse{\subset}(\ZZ,\CE_{\abs{\mhyphen}})$.

Fortunately, for many important examples of coarse maps $\crse{f\colon X\to Y}$ we can define canonical coarse preimages. Namely it is often the case that all the reasonable candidate preimages of some coarse subspace $\crse{Z\subseteq Y}$ are asymptotic to one another. When this happens, we say that this equivalence class is the \emph{coarse preimage} of $\crse Z$ and we denote it by $\crse{f^{-1}(Z)}$ (Definition~\ref{def:p1:coarse preimage}).

We say that a coarse homomorphism of coarse groups $\crse{f\colon G\to H}$ has a \emph{coarse kernel} if $[e_H]\crse{\subseteq H}$ has a coarse preimage. In this case we denote it by $\cker(\crse{f})\coloneqq \crse{f^{-1}}([e_H])$. Given the premise, it is unsurprising that not every coarse homomorphism has a coarse kernel. However, when it does exist it is well behaved: we prove in Subsection~\ref{sec:p1:isomorphism theorems} that appropriate analogs of the isomorphism theorems relating subgroups, kernels and quotients hold true in the coarse setting as well (Theorems~\ref{thm:p1:first iso theorem}, \ref{thm:p1:second iso theorem} and \ref{thm:p1:third iso theorem}).

\

\noindent{\bf Coarse actions.}
As announced, a \emph{coarse action} $\crse{\alpha\colon G\curvearrowright Y}$ of a coarse group $\crse G=(G,\CE)$ on a coarse space $\crse Y=(Y,\CF)$ is a coarse map $\crse \alpha\colon \crse{G\times Y\to Y}$ that is coarsely compatible with the coarse group operations. That is, any representative $\alpha$ for $\crse \alpha$ is so that $\alpha(g_1,\alpha(g_2,y))$ is uniformly close to $\alpha(g_1\ast g_2,y)$ and $\alpha(e,y)$ is uniformly close to $y$ for every $g_1,g_2\in G$, $y\in Y$. Once again, in order for $\alpha$ to be controlled it is necessary that the functions $\alpha(g,\variable)$ and $\alpha(\variable,y)$ be equi controlled.

\begin{rmk}
 If $G$ is a set\=/group and $\alpha\colon G\curvearrowright(Y,d_Y)$ is an action by isometries ---or, more generally, a quasi-action--- then $\crse \alpha\colon(G,\mincrs)\curvearrowright(Y,\CE_{d_Y})$ is a coarse action of the trivially coarse group. However, it is generally not true that if $d_G$ is a bi\=/invariant metric on $G$ then $\crse \alpha\colon(G,\CE_{d_G})\curvearrowright(Y,\CE_{d_Y})$ is a coarse action: in order to make sure that $\alpha$ is controlled it is necessary that $d_Y(\alpha(g_1,y),\alpha(g_2,y))$ is uniformly bounded in terms of $d_G(g_1,g_2)$.
\end{rmk}

Prototypical examples of coarse group actions of a coarse group $\crse G$ include the actions of $\crse G$ on itself by left multiplication (more on this later) and by conjugation. A useful feature of the latter is that it allows us to define coarse abelianity, normality etc. in terms of (coarse) fixed point properties.

\

Given two coarse actions $\crse{\alpha\colon G\curvearrowright Y}$, $\crse{\alpha'\colon G\curvearrowright Y'}$ of the same coarse group $\crse G$, we say that a coarse map $\crse{f\colon Y\to Y'}$ is \emph{coarsely equivariant} if it intertwines the $\crse G$\=/actions up to uniformly bounded error (Definition~\ref{def:p1:coarse equivariant}). As we did for coarse groups, we say that a \emph{quotient coarse action} of $\crse \alpha\colon (G,\CE)\curvearrowright(Y,\CF)$ is a coarse action $\crse \alpha\colon (G,\CE)\curvearrowright(Y,\CF')$ where $\CF'\supseteq \CF$ is some coarser coarse structure on $Y$ so that the functions $\alpha(g,\variable)$ are still equi controlled (this is necessary to ensure that $\alpha$ is still controlled). Once again, this is a consistent definition because a coarsely equivariant map $\crse f\colon (Y,\CF)\to \crse Z$ is coarsely surjective if and only if it factors through a coarsely equivariant coarse equivalence $(Y,\CF')\cong \crse Z$.

Given a coarse subgroup $[H]=\crse{H< G}=(G,\CE)$, we can define a coarse action of $\crse G$ on the \emph{coarse coset space} $\crse{G/H}$ by considering an appropriate quotient of the coarse action by left multiplication $\crse G\curvearrowright \crse G$. Namely, $\crse{G/H}$ is the set $G$ equipped with the smallest coarse structure containing $\CE$ and such that $H$ is bounded and $\crse{G\curvearrowright G/H}$ is a coarse action (it is interesting to note that, as in Remark~\ref{rmk:intro:silly quotient}, we may also consider coset spaces $\crse{G/A}$ where $\crse{A\subset G}$ is any coarse subspace, not necessarily a coarse subgroup).
The following result is an ingredient in the proof of the First Isomorphism Theorem.

\begin{thm}[Theorem~\ref{thm:p1:quotient by normal is a group}]
 Let $\crse{H\leq G}$ be a coarse subgroup. The coarse space $\crse{G/H}$ can be made into a coarse group with coarse operations that are compatible with the coarse action $\crse{G\curvearrowright G/H}$ if and only if $\crse H$ is coarsely normal in $\crse G$.
\end{thm}

\

Something interesting happens when looking at orbit maps and pull\=/backs of coarse structures. Let $\crse\alpha\colon(G,\CE)\curvearrowright (Y,\CF)$ be a coarse action and fix a point $y\in Y$. We may then consider the orbit map $\alpha_{y}\colon G\to Y$ given by $\alpha_{y}(g)\coloneqq \alpha(g,y)$. We can use this function to define a \emph{pull\=/back coarse structure} $\alpha_{y}^\ast(\CF)$ on $G$. 
It follows from the definition of pull\=/back coarse structure that the orbit map $\alpha_{y}$ defines a coarse equivalence between $(G,\alpha_{y}^\ast(\CF))$ and the coarse image $\crse{\alpha_{y}(G) \subseteq Y}$.
For a set\=/group $G$ acting by isometries on a metric space $(Y,d_Y)$, this is equivalent to defining a (pseudo) metric on $G$ by letting $d_G(g_1,g_2)\coloneqq d_Y(\alpha(g_1,y),\alpha(g_2,y))$ and giving $G$ the resulting metric coarse structure $\CE_{d_G}$.

It follows from the construction that $\CE\subseteq \alpha_{y}^\ast(\CF)$ and that the left multiplication defines a coarse action $(G,\CE)\curvearrowright(G,\alpha_{y}^\ast(\CF))$. The orbit map $\crse{\alpha_{y}}\colon (G,\alpha_{y}^\ast(\CF))\to \crse Y$ is coarsely equivariant. In turn, if $\crse{\alpha_{y}}$ is coarsely surjective---in which case we say that the action is \emph{cobounded}---it follows that $(G,\CE)\curvearrowright(Y,\CF)$ is a coarse quotient of the coarse action by left multiplication of $(G,\CE)$ on itself.

It is also simple to verify that if $y'$ is a point close to $y$ then the two orbit maps give rise to the same pull\=/back coarse structure. More generally one can show with a bit of work that if $\crse {G\curvearrowright Y}$ is a cobounded coarse action then the pull\=/back coarse structure is uniquely defined up to conjugation. Putting these observations together, one obtains the following characterization of cobounded coarse actions:

\begin{prop}[Proposition~\ref{prop:p1:cobounded coarse actions isom left mult}]
 Let $\crse G=(G,\CE)$ be a coarse group. There is a natural correspondence:
  \[
  \braces{\text{\upshape cobounded coarse actions of }\crse G}/_{\text{\upshape coarsely equivariant coarse equivalence}}
 \]
 and
 \[
  \braces{\CF\text{\upshape coarse structure}\mid \CE\subseteq\CF,\text{\upshape left multiplication is a coarse action }(G,\CE)\curvearrowright(G,\CF)}/_{\text{\upshape conjugation}}.
  \]
\end{prop}

\

Recall that group coarse structures are determined locally (Proposition~\ref{prop:intro:det locally}). The same is true for coarse structures associated with coarse actions. Namely,  let $(G,\CE)$ be a coarse group and let $\CF,\CF'$ be two coarse structures on $G$ so that the left multiplication defines coarse actions $(G,\CE)\curvearrowright(G,\CF)$ and $(G,\CE)\curvearrowright(G,\CF')$ (such coarse structures are said to be \emph{equi left invariant}). Then $\CF\subseteq \CF'$ if and only if they every $\CF$\=/bounded subset of $G$ is also $\CF'$\=/bounded. In particular, to check whether the two coarse structures coincide it is enough to verify that they have the same bounded sets. We claim that this simple observation can be seen as the heart of the Milnor--Schwarz  Lemma.

More precisely the Milnor--Schwarz  Lemma states that if a set\=/group $G$ has a proper cocompact action on a proper geodesic metric space $(Y,d_Y)$ then it is finitely generated and quasi\=/isometric to $(Y,d_Y)$. To prove this theorem, we start by remarking that cocompactness implies that the coarse action $(G,\mincrs)\curvearrowright(Y,\CE_{d_Y})$ is cobounded, therefore the orbit map is a coarse equivalence between $(G,\alpha_{y}^\ast(\CE_{d_Y}))$ and $(Y,d_Y)$. Since the action is proper, the bounded subsets of $(G,\alpha_{y}^\ast(\CE_{d_Y}))$ are the finite ones. By the coarse geodesicity of $\alpha_{y}^\ast(\CE_{d_Y})$, it follows that $G$ is finitely generated. Let $d_G$ be a word metric associated with a generating set, then $(G,\mincrs)\curvearrowright(G,\CE_{d_G})$ is also a coarse action. Since the bounded subsets of $(G,\CE_{d_G})$ are the finite ones, it follows that $\CE_{d_G}=\alpha_{y}^\ast(\CE_{d_Y})$ and hence the orbit map is a coarse equivalence between $(G,\CE_{d_G})$ and $(Y,\CE_{d_Y})$. Finally, we already remarked that coarse equivalences between geodesic metric spaces are always quasi\=/isometries.
This argument is explained in detail in Subsection~\ref{sec:p1:Milnor_svarc} (see also Section~\ref{sec:appendix:proper coarse actions} for a brief discussion of analogous statements in more general settings).

One may say that the above proof of the Milnor--Schwarz  Lemma bypasses most of the difficulties of the more standard proofs by avoiding working in the metric space $Y$. Instead, it is based on the observation that the relevant ``coarse properties'' are well behaved under taking pull\=/back and this allows us to reduce everything to the study of equi left invariant coarse structures on set\=/groups. We can also appreciate how the argument splits into many independent observations, showing precisely which hypothesis is necessary at any single step. By extension, we anticipate that a coarse geometric perspective will bring a fair amount of clarity into various complicated geometric constructions.

\section{List of Findings II: Selected Topics}

We shall now briefly overview the topics covered in the second part of the monograph. Further motivation and references are provided in the relevant sections.

\

\noindent{\bf Coarsified set\=/groups.}
The most natural connection between the theory of coarse groups and classical set\=/group theory is certainly to be found in the study of coarsified set\=/groups. Namely, those coarse group $(G,\CE)$ where $G$ is a set\=/group. Every set\=/group $G$ has two extreme coarsifications: the trivial one $(G,\mincrs)$ and the bounded one $(G,\maxcrs)$ (the latter is the coarse structure where the whole of $G$ is bounded). From the coarse point of view, neither of these is an interesting coarsification: the former reduces to classical set\=/group theory, the latter is always isomorphic to the coarse group having only one point $(\{e\},\maxcrs)$. We are therefore interested in intermediate coarsifications.

To avoid trivial coarse structures, we will be interested in \emph{coarsely connected} coarsifications. These are coarsifications $(G,\CE)$ so that every element $g\in G$ is $\CE$\=/close to the unit $e\in G$. Typical examples are given by metric coarsifications: if $d$ is some metric on $G$ then for every $g\in G$ the distance $d(g,e)$ is some finite number and hence $g$ is $\CE_d$\=/close to $e$. The reason why coarsely connected coarsifications are particularly interesting for us is that any coarsification $(G,\CE)$ can be described in terms of a connected coarsification and a trivially coarse quotient. That is, the set of points $g\in G$ that are $\CE$\=/close to $e$ is a normal subgroup $G_0\trianglelefteq G$ and the coarse group $(G,\CE)$ factors through the quotient $(G,\CE)\to(G/G_0,\mincrs)$ (Corollary~\ref{cor:p1:ses quotient by id comp}).
The general problem we wish to address is:

\begin{problem}\label{prob:intro:unbounded coarsification}
 Which set\=/groups $G$ admit unbounded coarsely connected coarsifications $(G,\CE)$?
\end{problem}

Not every infinite set\=/group has an unbounded coarsely connected coarsification. Specifically, we say that a set\=/group $G$ is \emph{intrinsically bounded} if $(G,\maxcrs)$ is the only coarsely connected coarsification of $G$. Examples of infinite intrinsically bounded set\=/groups include the dihedral group $D_\infty$ and $\Sl(n,\ZZ)$ with $n\geq 3$ (this follows from some bounded generation properties). More generally, it follows from the work of Duszenko and McCammond--Petersen that a Coxeter group is intrinsically bounded if and only if it is of spherical or affine type \cite{duszenko2012reflection,mccammond2011bounding}. Generalizing the case of $\Sl(n,\ZZ)$, Gal, Kedra, Trost and others proved that ``many'' Chevalley groups are intrinsically bounded \cite{gal2011bi,trost2020strong}.

Problem~\ref{prob:intro:unbounded coarsification} is very much related to the existence of unbounded bi\=/invariant metrics. If $d$ is an unbounded bi\=/invariant metric on $G$ then $(G,\CE_d)$ is an unbounded coarsely connected coarsification. For ``small'' set\=/groups this actually characterizes intrinsic boundedness:

\begin{prop}[Proposition~\ref{prop:p2:intrinsically bounded iff no biinv metric}]
\label{prop:intro:bounded iff bounded metric}
 If a set\=/group $G$ is countably generated, then it is intrinsically bounded if and only if it does not admit any unbounded bi\=/invariant metric.
\end{prop}

We should remark that ``most'' infinite groups are not intrinsically bounded. For example, it follows by the work of Epstein--Fujiwara \cite{EF} that every infinite Gromov hyperbolic group has an unbounded bi\=/invariant metric. It follows that a random (in the sense of Gromov) infinite finitely generated group is not intrinsically bounded.

\

\noindent{\bf Coarsifications of $\ZZ$.}
The Euclidean metric shows that $\ZZ$ does have unbounded coarsely connected coarsifications. As it turns out, this is but one of its many coarsifications. It is perhaps unsurprising that any attempt to classify such coarsifications leads to interesting problems in number theory.

One source of coarsifications is obtained by picking word metrics $d_S$ associated with (infinite) generating sets $S\subset \ZZ$. It is often a delicate problem to understand whether two different generating sets $S\neq S'$ give rise to the same coarse structure $\CE_{d_S}=\CE_{d_{S'}}$. For example, if $S=P$ is the set of all prime numbers, the equality $\CE_{d_P}=\maxcrs$ is equivalent to the highly non\=/trivial statement that every natural number is equal to the sum of a bounded number of primes. One may even say that the Goldbach conjecture is a very strong effective version of this result.
In general, it appears that the only coarse structures $\CE_{d_S}$ that are relatively well understood are those where $S$ is a union of geometric sequences.

Nevertheless, choosing sparse enough generating sets $S$ suffices to show that $\ZZ$ has a continuum of distinct coarsifications.
Moreover, one can also combine these coarsifications to obtain even more coarsifications and thus prove that the set of coarse structures on $\ZZ$ is as large as possible:

\begin{prop}[Proposition~\ref{prop:p2:2^continuum of coarsification}]
 There exist $2^{2^{\aleph_0}}$ distinct connected coarsifications of $\ZZ$.
\end{prop}

A different way of defining coarsifications of $\ZZ$ is by observing that every abelian topological group $(G,\tau)$ admits a natural coarse structure $\CE_\tau$ whose bounded sets are the relatively compact subsets of $G$. This creates a link with the study of group topologies on $\ZZ$: any such topology $\tau$ defines a coarsification $\CE_\tau$. Once again, given different topologies $\tau\neq\tau'$ it is generally challenging to distinguish $\CE_\tau$ from $\CE_\tau'$. One viable strategy is to find a sequence $(a_n)_{n\in\NN}$ that $\tau$\=/converges to some point $a_\infty\in \ZZ$, but so that the set $\{a_n\mid n\in\NN\}$ is not relatively compact according to $\tau'$. This technique proves very effective to differentiate between coarsifications arising from profinite completions of $\ZZ$. More precisely, given any set $Q$ of prime numbers we define a group topology $\tau_Q$ on $\ZZ$ via the diagonal embedding in the product of the pro\=/$p$ completions $\ZZ\hookrightarrow\prod_{p\in Q}\ZZ_p$. We can then prove:

\begin{prop}[Proposition~\ref{prop:p2:different profinite coarse structures}]
 Let $Q,Q'$ be sets of primes. Then $\CE_{\tau_Q}=\CE_{\tau_{Q'}}$ if and only if $Q=Q'$. 
\end{prop}

% \begin{cor}
%  There exists a continuum of distinct coarsely connected coarsifications of $\ZZ$.
% \end{cor}

\

\noindent{\bf Bi-invariant word metrics.}
An important means of defining bi\=/invariant metrics on set\=/groups is by taking word metrics associated with conjugation invariant sets. Of course, the word metrics on $\ZZ$ discussed above are one example of such metrics. We also already mentioned bi\=/invariant word metrics on finitely generated set\=/groups, where we use them to define a canonical coarsification $\varcrs{bw}$ and provide examples of non coarsely abelian  coarse groups.
One reason why these metrics are important is that they are relatively easy to define and they tend to yield ``large metrics''. For instance, the metrics used to prove Proposition~\ref{prop:intro:bounded iff bounded metric} for finitely normally generated set\=/groups are bi\=/invariant word metrics. In fact, it is not hard to show that a finitely normally generated group is intrinsically bounded if and only if the bi\=/invariant word metric $d_{\overline{S}}$ associated with a finite normally generating set $S$ is bounded.

It turns out that it is always possible to recognize whether a coarsification of a set\=/group arises from a bi\=/invariant word metric metric. Coarse structures have a notion of a ``generating set of relations'' (paragraph before Lemma~\ref{lem:p1: generated coarse structure}), and we say that a coarsely connected coarse structure is \emph{coarsely geodesic} if it is generated by a single relation (Definition~\ref{def:p1:crs_geod}). Informally, this definition is an analog of the fact that if $(X,d)$ is a geodesic metric space we can estimate the distance $d(x,y)$ between any two points as the length of the shortest sequence $x_0,\ldots,x_n$ so that $x=x_0$, $y=x_n$ and $d(x_i,x_{i+1})\leq 1$ for all $1\leq i\leq n$. The following holds:

\begin{prop}[Proposition~\ref{prop:p2:coarse geodesic coarsifications}]
 Let $G$ be a set\=/group. A coarsification $(G,\CE)$ is coarsely geodesic if and only if there exists a bi\=/invariant word metric metric $d_{\overline{S}}$ so that $\CE=\CE_{d_{\overline{S}}}$.
\end{prop}

At this point we should issue a warning to the geometric group theorists: the inclusion $\ZZ< D_\infty$ shows that the canonical coarse structures $\varcrs{bw}$ are very poorly behaved when taking finite index subgroups! This shows that bi\=/invariant metrics depend rather subtly on the algebraic properties of set\=/groups and opens the way to all sorts of questions. For example, we do not know whether the canonical coarsifications of non\=/abelian free groups $F_n$ and $F_m$ can ever be isomorphic as coarse groups if $n\neq m$.

\

\noindent{\bf Coarse groups that are not coarsified set\=/groups.}
The vast majority of the examples of coarse groups that we gave so far are coarsified set\=/groups. We already remarked that one may find a coarse group $\crse G=(G,\CE)$ so that no choice of representatives for the coarse operations $\cop$ and $\cinversefn$ makes $G$ into a set\=/group. On its own, this statement is rather unsatisfactory: it is easy to obtain such an example by ``perturbing'' a set\=/group, but such a construction misses the point of the coarse geometric approach. Rather, a much more interesting question is whether there exist a coarse group that is not isomorphic (as a coarse group) to any coarsified set\=/group.

We conjecture that this is indeed the case, and we also describe a suitable candidate. Let $(F_2,\varcrs{bw})$ be the canonical metric coarsification of the free group of rank two. Let also $f\colon F_2\to \RR$ be a homomorphism sending one generator to $1$ and the other to some irrational number $\alpha$, and let $K\coloneqq f^{-1}([0,1])$. The asymptotic equivalence class $[K]$ is a coarse subgroup of $(F_2,\varcrs{bw})$ (this is the coarse kernel of the coarse homomorphism $\crse f\colon(F_2,\varcrs{bw})\to(\RR,\CE_{\abs{\mhyphen}})$), hence it uniquely determines a coarse group up to isomorphism. We conjecture that this kernel is not isomorphic to a coarsified set\=/group. 

We prove a criterion to recognize when a coarse group is isomorphic to a coarsified set\=/group (Proposition~\ref{prop:p2:non_coarsification general}). Using this criterion, we can reduce our conjecture to a very elementary---albeit complicated---combinatorial problem (Conjecture~\ref{conj:p2:non coarsification elementary}).

\

\noindent{\bf Coarse homomorphisms and coarse automorphisms.}
Given coarse groups $\crse G$, $\crse H$, it is often interesting to study the set $\chom(\crse{G,H})$ of coarse homomorphisms $\crse{f\colon G\to H}$ and the set\=/group $\cAut(\crse G)$ of coarse automorphisms of $\crse G$. If $\crse G=(G,\CE)$, $\crse H=(H,\CF)$ are coarsified set\=/groups, it is also natural to compare $\chom(\crse{G,H})$ with the set of set\=/group homomorphisms $\Hom(G,H)$. One reason for doing so, is that such comparisons can be interpreted as rigidity phenomena of the relevant set\=/groups (or lack thereof).

To begin with, consider the metric coarsification $(\ZZ,\CE_{\abs{\mhyphen}})$ associated with the Euclidean metric on $\ZZ$. It is not hard to show that a function $\phi\colon \ZZ\to\ZZ$ defines a coarse homomorphism if and only if it is an $\RR$\=/quasimorphism. 
It is well\=/known (and easy to show) that the quotient of the set of $\RR$\=/quasimorphisms of $\ZZ$ by identifying close quasimorphisms is in natural correspondence with $\RR$. We thus see that  $\chom((\ZZ,\CE_{\abs{\mhyphen}}),(\ZZ,\CE_{\abs{\mhyphen}}))=\RR$ in a natural way. 

It is suggestive to note that when we compare coarse homomorphisms $\ZZ\to\ZZ$ with $\RR$\=/quasimorphisms we are implicitly using the isomorphism of coarse groups $(\ZZ,\CE_{\abs{\mhyphen}})\cong (\RR,\CE_{\abs{\mhyphen}})$ to identify $\chom((\ZZ,\CE_{\abs{\mhyphen}}),(\ZZ,\CE_{\abs{\mhyphen}}))$ with $\chom((\ZZ,\CE_{\abs{\mhyphen}}),(\RR,\CE_{\abs{\mhyphen}}))$. 
Going one step further, we may also replace $\ZZ$ by $\RR$ in the domain and observe that  
\[
 \chom((\RR,\CE_{\abs{\mhyphen}}),(\RR,\CE_{\abs{\mhyphen}})) = \RR
\]
In other words, every coarse endomorphisms of $(\RR,\CE_{\abs{\mhyphen}})$ is close to a linear transformation. This fact is not special about $\RR$: with some care it can be extended to other normed vector spaces:

\begin{prop}[Proposition~\ref{prop:p2:banach chom are linear}]
 If $V_1$, $V_2$ are real Banach spaces then the set of coarse homomorphisms can be identified with the space of bounded linear operators
 \[\chom((V_1,\CE_{\norm{\mhyphen}_{V_1}}),(V_2,\CE_{\norm{\mhyphen}_{V_2}}))= \mathfrak{B}(V_1,V_2). \]
\end{prop}

\begin{cor}
 For every $n\in\NN$ we have $\cAut(\ZZ^n,\CE_{\norm{\mhyphen}})\cong \Gl(n,\RR)$.
\end{cor}

Recall that an arbitrary coarse homomorphism need not have a well\=/defined coarse kernel. It is interesting to note that in the linear setting we know precisely when this is the case:
 
\begin{thm}[Theorem~\ref{thm:p2:banach chom and kernels}]
 Let $\crse T\colon (V_1,\CE_{\norm{\mhyphen}_1})\to(V_2,\CE_{\norm{\mhyphen}_2})$ be a coarse homomorphism between Banach spaces and let $T$ be its linear representative. Then $\crse T$ has a coarse kernel if and only if $T(V_1)$ is a closed subspace of $V_2$. When this is the case, $\cker(\crse T)=[\ker(T)]$.
\end{thm}

On the non\=/abelian side, one case of special interest is where $\crse G=(G,\varcrs{bw})$ and $\crse H=(H,\varcrs{bw})$ are the canonical coarsifications of finitely generated set\=/groups. In this case every set\=/group homomorphism is controlled, so there is a natural homomorphism $\Phi\colon\Hom(G,H)\to\chom(\crse G,\crse H)$. It follows from the bi\=/invariance of $\varcrs{bw}$ that for every $h\in H$ any homomorphism $f\colon G\to H$ is close to its conjugate $c_h\circ f(x)\coloneqq h f(x) h^{-1}$. We thus obtain a quotient map 
\[
 \overline\Phi\colon\Hom(G,H)/\text{conj.}\to\chom(\crse G,\crse H).
\]
In general, $\overline{\Phi}$ need not be injective nor surjective. However, it can be studied on a case-by-case basis to give an idea of the level of compatibility of the algebraic/coarse geometric properties of $G$ and $H$.

In the case $G=H$, this $\overline{\Phi}$ also defines a homomorphism $\overline{\Phi}\colon \Out(G)\to\cAut(G,\varcrs{bw})$. Note that for $G=\ZZ^n$ the coarse structure $\varcrs{bw}$ coincides with the natural metric coarse structure $\CE_{\norm{\mhyphen}}$ and hence $\cAut(\ZZ^n,\varcrs{bw})\cong\Gl(n,\RR)$. We then see that $\overline{\Phi}$ is simply given by the inclusion $\Gl(n,\ZZ)\subset\Gl(n,\RR)$. This shows that in this case $\overline{\Phi}$ is injective and it is very far from being surjective. One may also say that there are many ``exotic'' coarse automorphisms that do not arise from automorphisms of $\ZZ^n$.

Another case of interest is that of the non\=/abelian free set\=/group $F_n$. The set\=/group $\cAut(F_n,\varcrs{bw})$ for $n\geq 2$ is rather mysterious. However, we can use the work of Hartnick--Schweitzer \cite{HS} to prove the following.

\begin{prop}[Subsection~\ref{sec:p2:coarse hom of grps with canc met}]
 The map $\overline{\Phi}\colon \Out(F_n)\to \cAut(F_n,\varcrs{bw})$ is injective. Furthermore, $\cAut(F_n,\varcrs{bw})$ is uncountable and contains torsion of arbitrary order.
\end{prop}

\begin{rmk}
 Hartnick--Schweitzer define a notion of quasimorphism between non\=/abelian set\=/groups. We can show that every set\=/group $G$ has a canonical coarsification $(G,\varcrs{G\to\RR})$ such that the space of quasimorphisms from $G$ to $H$ á la Hartnick--Schweitzer coincides with the space of coarse homomorphisms $\chom((G,\varcrs{G\to\RR}),(H,\varcrs{H\to\RR}))$ (Proposition~\ref{prop:p2:quasim_iff_chom pullback}).
\end{rmk}

\

\noindent{\bf Fragmentary coarse spaces}
The last topic we address is the study of the space of controlled functions between coarse spaces. Ideally, one would like to define a natural coarse structure on the set of controlled functions $f\colon(X,\CE)\to(Y,\CF)$. Unfortunately, it is not possible to define such a coarse structure that is well\=/behaved under taking compositions.

For an idea of what goes wrong, consider the following example. Let $X$ be a metric space and $(f_i)_{i\in I}$, $(f_i')_{i\in I}$ be families of Lipschitz functions from $X$ to itself. One may then say that these two families are uniformly close if there is some constant $D$ so that $d(f_i(x),f_i'(x))\leq D$ for every $i\in I$, $x\in X$. When this happens, we think of these families to be coarsely indistinguishable. However, if we are now given two more families of uniformly close Lipschitz functions $(g_i)_{i\in I}$, $(g_i')_{i\in I}$ there is no reason why the compositions $(g_i\circ f_i)_{i\in I}$ and $(g_i'\circ f_i')_{i\in I}$ should again be uniformly close. We are then in the problematic situation where the composition of two indistinguishable families of functions are no longer indistinguishable.
One way to solve this issue is to only consider families of functions with uniformly bounded Lipschitz constants (albeit the bound may be arbitrarily large). As it turns out, this solution can be implemented in the full generality of coarse spaces via means of ``fragmentary coarse structures''.

A \emph{fragmentary coarse structure} $\CE$ on a set $X$ is a family of subsets of $X\times X$ that satisfies some axioms which are only marginally weaker than those of a coarse structure (specifically, we do not require that $\CE$ contains the diagonal $\Delta_X$, see Definition~\ref{def:p2:fragmentary coarse.structure}). In particular, every coarse structure is also a fragmentary coarse structure.
In analogy with coarse spaces, we may then define a category of frag\=/coarse spaces. Crucially, this category is much better suited for studying spaces of controlled functions.

If $\crse X=(X,\CE)$ and $\crse Y=(Y,\CF)$ are (frag-)coarse spaces, we define a frag\=/coarse structure $\CF^\CE$ on the set of all controlled functions $(X,\CE)\to(Y,\CF)$ (Definition~\ref{def:p2:exponential fcrse strucuture}). We denote by $\crse{Y^{X}}$ the resulting frag\=/coarse space. This definition is natural in the following sense:

\begin{thm}[Theorem~\ref{thm:p2:exponential}]\label{thm:intro:exponential}
 Let $\crse X$, $\crse Y$, $\crse Z$ be (frag-)coarse spaces. Then:
 \begin{enumerate}
  \item the evaluation map $\crse{Y^{X}\times X}\to \crse Y$ sending $(f,x)$ to $f(x)$ is controlled;
  \item the composition map $\crse{Z^{Y}\times Y^{X}\to Z^{X}}$ is controlled;
  \item there is a natural bijection 
  \[
   \braces{\text{(frag-)coarse maps }\crse{Z\times X\to Y}} \longleftrightarrow
   \braces{\text{frag-coarse maps }\crse{Z\to Y^{X}}}. 
  \]
 \end{enumerate}
\end{thm}

\begin{rmk}
 In categorical terms, Theorem~\ref{thm:intro:exponential} means that the category \Cat{FragCrs} of frag\=/coarse spaces is Cartesian closed and that the category of coarse spaces can be enriched over \Cat{FragCrs}.
\end{rmk}

Theorem~\ref{thm:intro:exponential} allows us to take a different point of view on coarse actions. For a given coarse space $\crse Y$, the frag\=/coarse space $\crse{Y^{Y}}$ is a \emph{frag\=/coarse monoid}. One may then describe coarse actions $\crse{G\curvearrowright Y}$ as homomorphisms of frag\=/coarse monoids $\crse{G\to Y^{Y}}$ (Proposition~\ref{prop:p2:actions as monoid homomorphisms}).

We may also carefully define a frag\=/coarse space of self coarse equivalences $\crse{\ctraut{ Y}}$ which is not just a frag\=/coarse monoid, but also a frag\=/coarse group (this is somewhat more subtle than just taking the subspace of $\crse{Y^{Y}}$ of functions that have a coarse inverse, see Definition~\ref{def:p2:frag space of ceq}). 
We may again identify coarse actions $\crse{G\curvearrowright Y}$ with coarse homomorphisms $\crse{G\to \ctraut{Y}}$. With some work, we may then use the coarse homomorphism induced by the coarse action by left translation of a coarse group $\crse{G}$ on itself to prove the following:

\begin{thm}[Theorem~\ref{thm:p2:coarse groups as subgroups}]
 Every coarse group is isomorphic to a coarse subgroup of a frag\=/coarsified set\=/group. 
\end{thm}

\section{Acknowledgments}

It is a pleasure to thank Uri Bader for suggesting this research topic, and for many stimulating conversations. We wish to thank Nicolaus Heuer for his helpful insights. We are also grateful to Samuel Evington, Jarek K\k{e}dra, Christian Rosendal and Katrin Tent for their comments and suggestions. % TODO: Add thanks to Merlin

We also wish to thank the referee for their valuable comments, and for answering several questions we asked in an earlier draft.
The current layout of Chapter~\ref{ch:p2:coarse structures on set groups} was suggested by him/her.

Much of the work was carried out while the authors were at the Weizmann Institute of Science, which we thank for the gracious hospitality. 
Both authors were partially supported by ISF grant 704/08. The second author was also Funded by the Deutsche Forschungsgemeinschaft (DFG, German Research Foundation) under Germany's Excellence Strategy EXC 2044 –390685587, Mathematics Münster: Dynamics–Geometry–Structure.

%% file: BasicTheory.tex
\part{Basic Theory}

\chapter{Introduction to the Coarse Category}\label{ch:p1:intro to coarse category}
In this chapter we introduce the category of coarse spaces, and define basic notions. We also give examples of coarse structures that are frequently useful. 
Most of the material in this chapter is standard, and appears in several sources, including \cite{roe_lectures_2003}.  For this reason we will not prove everything in detail.  The informed reader can skim this chapter.

% There are two main approaches to the theory of coarse spaces: via families of controlled coverings and coarse structures.
% Different approaches prove more advantageous in different situations, and in this monograph we will use both of them.

We follow two main approaches to the theory of coarse spaces: via families of controlled coverings and coarse structures.
These different approaches prove more advantageous in different situations.
It should certainly be noted that there are other possible formalisms, most notably the language of \emph{balleans} or \emph{ball structures}.
This was introduced by Protasov and Banakh \cite{protasov2003ball} and further developed by Protasov and Zarichnyi in \cite{protasovgeneral}.
Led by Protasov, the Ukrainian school produced an impressive body of work dealing with general ball structures on sets and groups.
We also recommend the paper \cite{dikranjan_categorical_2017} for a useful comparison of these languages.

\section{Some Notation for Subsets and Relations}
We begin by setting some notation. Given a set $X$, we denote families\footnote{%
In this text the meaning of ``family'' is---on purpose---a little ambiguous.
In most cases, it simply means ``subset of the power set''.
More rarely, it denotes an \emph{indexed family}, \emph{i.e.}~a possibly non-injective function from some index set to power set.
It will be easy to tell which is the case, as indexed families keep the index in the notation: $(U_i)_{i\in I}$.}

of subsets of $X$ with lowercase fraktur letters ($\fka,\fkb,\dots$) \nomenclature[:COV]{$\fka, \fkb$}{families of subsets of $X$ (a.k.a. partial coverings)}%
and families of subsets of $X\times X$ with uppercase calligraphic letters ($\CE,\CF,\dots)$. Subsets of $X\times X$ are called \emph{relations}\index{relation} on $X$. Given a relation $E\subset X\times X$ and $x,y\in X$, we will write $x\torel{E} y$\nomenclature{$x\torel{E} y$}{$x$ is related to $y$ by $E$ ($(x,y)\in E$)} if $(x,y)\in E$. We also write $x\rel{E} y$ if $x\torel{E}y$ and $y\torel{E} x$.

A \emph{partial covering}\index{covering!partial} is any family $\fka$ of subsets of $X$.  It is a \emph{covering} \index{covering} if $X=\bigcup\braces{A\mid A\in\fka}$.
Let $\fka$ and $\fka'$ be two partial coverings of $X$, we say $\fka$ is a \emph{refinement}\index{covering!refinement} of $\fka'$ if for every $A\in\fka$ there is an $A'\in\fka'$ such that $A\subseteq A'$.
We call $\pts{X}\coloneqq\{\{x\}\mid x\in X\}$ the \emph{minimal covering}\index{covering!minimal} of $X$ (it is a refinement of every other covering).

For any subset $A\subseteq X$ and partial covering $\fkb$ of $X$, the \emph{star}\index{covering!star} of $A$ with respect to $\fkb$ is the set $\st(A,\fkb)\coloneqq \bigcup\braces{B\in\fkb\mid B\cap A\neq \emptyset}$.  \nomenclature[:COV]{$\st(A,\fkb)\coloneqq \bigcup\braces{B\in\fkb\mid B\cap A\neq \emptyset}$}{star of $A$ with respect to $\fkb$}%
Note that $A\subseteq \st(A,\fkb)$ if and only if $\fkb$ covers $A$---if $\fkb$ is a covering this is always the case. 
If $\fka$ is another partial covering of $X$, the \emph{star of $\fka$} with respect to $\fkb$ is the family $\st(\fka,\fkb)\coloneqq\braces{\st(A,\fkb)\mid A\in\fka}$.  \nomenclature[:COV]{$\st(\fka,\fkb)\coloneqq\braces{\st(A,\fkb)\mid A\in\fka}$}{star of $\fka$ with respect to $\fkb$} 
This is a special instance of a convention we will use often: if $F\colon \CP(X)\to\CP(Y)$ is a function sending subsets of $X$ to subsets of $Y$ and $\fka$ is a partial covering of $X$, we denote by $F(\fka)$ the family of images $\fka\coloneqq\braces{F(A)\mid A\in\fka}$.

Now we will describe some operations on relations.
Fix a set $X$ be set and denote by $\pi_1\colon X\times X\to X$ and $\pi_2\colon X\times X\to X$ the projections to the first and second coordinate.  
Given a relation $E\subseteq X\times X$ we let $\op{E}\coloneqq\braces{(y,x)\mid (x,y)\in E}\subseteq X\times X$ denote its symmetric.  \nomenclature[:R]{$\op{E}\coloneqq\braces{(y,x)\mid (x,y) \in E}$}{symmetric relation}
For any subset $Z\subseteq X$ we define $E(Z)\subseteq X$ as \nomenclature[:R]{$E(Z)$}{section of a relation}
\[
 E(Z)\coloneqq \braces{x\in X\mid\exists z\in Z\text{ s.t. }x\torel{E}z}=\pi_1\paren{\pi_2^{-1}(Z)\cap E}.
\]
Given points $x\in X$ and $y\in X$ we define \emph{sections}\index{section (of relation)} 
\[
{}_xE\coloneqq \braces{z\in X\mid x\torel{E}z}\subseteq X
 \qquad\text{ and }\qquad
 E_y\coloneqq \braces{z\in X\mid z\torel{E}y}\subseteq X.
\]
That is, $E_y= E(\{y\})$ and ${}_xE= \op{E}(\{x\})$. Later we will also abuse notation and write $E(y)$.
Given $E_1,E_2\subseteq X\times X$ we denote their \emph{composition}\index{composition (of relations)} by:\footnote{%
We follow Roe's convention on the order of composition of relations \cite{roe_lectures_2003}. This is coherent with the pair groupoid structure on $X\times X$ and it is so that $E_1\cmp E_2(Z)=E_1(E_2(Z))$.
One can see functions as relations by identify a function $f\colon X\to X$ with its graphs ${\rm graph}\subseteq X\times X$. However, our choice of the order of composition is not coherent with the usual composition of functions. In particular, our conventions are so that $f(x)={}_x({\rm graph}(f))=(\op{{\rm graph}(f)})(\{x\})$.
}
\[
 E_1\cmp E_2\coloneqq\braces{(x_1,x_2)\mid \exists x'\in X\text{ s.t }x_1\torel{E_1}x' \torel{E_2}x_2}\subseteq X\times X.
\]

\

We denote the diagonal by $\Delta_{X} \subseteq X\times X$. \nomenclature[:R]{$\Delta_{X},\ \Delta_{G}$}{diagonal} Given $A\subseteq X$, we let $\Delta_A\coloneqq \Delta_{X}\cap (A\times A)$.
In the sequel it will be convenient to move back and forth between relations and partial coverings. Note that the following procedure works particularly well in dealing with relations that contain the diagonal. A relation $E\subseteq X\times X$ naturally defines a partial covering $E(\pts{X})=\braces{E_x\mid x\in X}$.
Given a partial covering $\fka$ of $X$, we define the ``blocky diagonal'' by
\[ 
 \diag(\fka)\coloneqq \bigcup\braces{A\times A\mid A\in\fka} \subseteq X \times X. 
\]
\nomenclature[:COV]{$\diag(\fka)$}{blocky diagonal}
Note that $(\diag(\fka))(\pts X)=\st(\pts X,\fka)$, hence $\fka$ is a refinement of $(\diag(\fka))(\pts X)$. Vice versa, if $E\subseteq X\times X$ contains the diagonal $\Delta_{X}$, then $E\subseteq \diag(E(\pts X))$.

The composition and the star operations are related to one another as follows:

\begin{lem}
\label{lem:p1:composition vs star}
 Fix a set $X$. 
 \begin{enumerate}[(i)]
  \item Given $E,F\subseteq X\times X$ with $\Delta_{X}\subseteq E$, then $E\cmp F(\pts{X})$ is a refinement of $\st(F(\pts X),E(\pts X))$.
  \item Given partial coverings $\fka,\fkb$, then $\diag(\st(\fka,\fkb))=\diag(\fkb)\cmp\diag(\fka)\cmp\diag(\fkb)$.
 \end{enumerate}
\end{lem}
\begin{proof}
 \emph{(i):} Note that
 \[
  (E\cmp F)_x=\braces{y\mid \exists z \text{ s.t. }y\torel{E}z\torel{F}x}=\bigcup\braces{E_z\mid z\in F_x}.
 \]
 Since $\Delta_X\subseteq E$, the point $z$ belongs to $E_z$ hence $(E\cmp F)_x\subseteq \st(F_x,E(\pts X))$.
 
 \emph{(ii):} Note that $x\torel{\diag(\fkb)\cmp\diag(\fka)\cmp\diag(\fkb)}y$ if and only if there exists $A\in \fka$ and $B_1,B_2\in \fkb$ with
 \[
  x\torel{\diag(B_1)}z\torel{\diag(A)}w\torel{\diag(B_2)}y.
 \]
 That is, $x,z\in B_1$, $z,w\in A$ and $w,y\in B_2$, which happens if and only if $x,y\in \st(A,\fkb)$.
\end{proof}

Note that for every function $f\colon X\to Y$, and partial coverings $\fka,\fkb$ of $X$ the image $f\paren{\st(\fka,\fkb)}$ is a refinement of $\st\paren{f(\fka),f(\fkb)}$. 
Similarly, if we consider the product function $f\times f\colon X\times X\to Y\times Y$ then for any two relations $E_1,E_2$ of $X\times X$ we have 
\[
f\times f(E_1\cmp E_2)\subseteq \paren{f\times f(E_1)}\cmp \paren{f\times f(E_2)}.
\]

\

For later reference, we state the following:

\begin{de}
\label{def:p1:ideal}
 If $X$ is any set, an \emph{ideal}\index{ideal (of subsets)} on $X$ is a family $\fkI$ of subsets of $X$ such that 
 \begin{enumerate}[(i)]
  \item $\emptyset\in \fkI$;
  \item if $A\in \fkI$ and $B\subseteq A$, then $B\in\fkI$;
  \item if $A,B\in \fkI$ and $A\cup B\in\fkI$.
 \end{enumerate}
 (Equivalently, this is an ideal in the commutative ring $\mathcal{P}(X)=\ZZ/2\ZZ^X$.)
\end{de}

\section{Coarse Structures}
We follow \cite{roe_lectures_2003} to define coarse structures and coarse spaces. 
In the next section we will also explain how to describe coarse structures in terms of controlled partial covering.

\begin{definition}[See \cite{roe_lectures_2003}]\label{def:p1:coarse.structure}
A \emph{coarse structure} \index{coarse structure} on a set $X$ is an ideal $\CE$ on $X\times X$ (\emph{i.e.}\ a non\=/empty collection of relations on $X$ closed under taking subsets and finite unions) such that
\begin{enumerate}[(i)]
 \item $\Delta_{X}\in\CE$ (contains the diagonal);
 \item if $E\in\CE$ then $\op{E}\in\CE$ (closed under symmetry); 
 \item if $E_1,E_2\in\CE$ then $E_1\cmp E_2\in\CE$ (closed under composition).
\end{enumerate}
The relations $E\in\CE$ are called \emph{controlled sets} \index{controlled!set/entourage} (or \emph{entourages}\index{entourage}) for the coarse structure $\CE$. A set equipped with a coarse structure is called a \emph{coarse space}. \index{coarse!space}
\end{definition}

\begin{example}
 The two extreme examples of coarse structures on a set $X$ are the \emph{minimal coarse structure} \index{coarse structure!minimal} $\mincrs=\braces{\Delta_Z\mid Z\subseteq X}$ \nomenclature[:CE1]{$\mincrs$}{minimal coarse structure} and the \emph{maximal coarse structure}\index{coarse structure!maximal} $\maxcrs=\CP(X\times X)$ \nomenclature[:CE1]{$\maxcrs$}{maximal coarse structure}  (the maximal coarse structure is the only coarse structure containing $X\times X$).
 A slightly less trivial example is $\varcrs{fin}\coloneqq\braces{E\mid E\smallsetminus \Delta_{X}\text{ is finite}}$. \nomenclature[:CE1]{$\varcrs{fin}$}{finite off\=/diagonal coarse structure} This is the \emph{finite off\=/diagonal} coarse structure (called \emph{discrete} coarse structure in \cite{roe_lectures_2003}). \index{coarse structure!finite off-diagonal}
\end{example}

\begin{example}
 If $(X,d)$ is a metric space, the \emph{metric coarse structure}\nomenclature[:CE1]{$\CE_d$}{metric coarse structure on a metric space $(X,d)$} $\CE_d$ is the family of all relations that are within bounded distance of the diagonal:
 \[
  \CE_d\coloneqq\braces{E\subseteq X\times X \mid \exists r\in\RR_+ \text{ s.t. }d(x,x')\leq r  \text{ for all } (x,x')\in E}
 \]
 We say that $(X,\CE_d)$ is a \emph{metric coarse space}\index{metric coarse space}. Metric coarse spaces are the prototypical examples of coarse spaces.
\end{example}

\begin{de}\label{def:p1:connected coarse structure}
 A coarse structure $\CE$ on $X$ is \emph{connected} \index{coarse structure!connected} if every pair $(x,x') \in X \times X$ is contained in some controlled set. If $\CE$ is connected, we say that the coarse space $(X,\CE)$ is \emph{coarsely connected}. \index{coarsely!connected}
 Every coarse space $(X,\CE)$ can be partitioned into a disjoint union of \emph{coarsely connected components}. \index{coarsely!connected!component}
\end{de}

\begin{rmk}
 Note that for every metric space $(X,d)$ the metric coarse structure $\CE_d$ is connected. 
 The ``finite off\=/diagonal coarse structure'' described above is the minimal connected coarse structure on $X$.
\end{rmk}

In view of the above, it can be convenient to consider metrics $d$ that take values in the extended positive reals $[0,+\infty]$. To avoid confusion, we will say that such a $d$ is an \emph{extended metric}.\index{extended metric} We may then define the coarse structure $\CE_d$ associated with any extended metric just as we did for normal metrics. In this case, the coarsely connected components are the subsets of points at finite distance from one another. 

\begin{rmk}
 One may also define \emph{hypermetrics} by considering positive valued functions with hyperreal values $d\colon X\times X\to {}^*\RR_{\geq 0}$ that satisfy the triangle inequality. Such a hypermetric defines a coarse structure $\CE_d$ just as a normal metrics does. Two points $x,y\in X$ belong to the same coarsely connected component if and only if their distance is a finite hyperreal.
 
 This point of view may be useful when dealing with ultralimits of metric spaces. For example, the asymptotic cone ${\rm Cone}(X,e,\lambda)$ is the coarsely connected component of the ultralimit $\omega\mhyphen\lim_i (X,\lambda_i d)$ containing the point $e=\omega\mhyphen\lim_i e_i$ \cite[Chapter 10]{drutu_geometric_2018}.
\end{rmk}

\begin{de}
 Let $(X,\CE)$ be a coarse space.
A subset $A\subseteq X$ is \emph{bounded} if $A\times A\in\CE$. If we want to specify that a set is bounded \index{bounded!subset} with respect to a specific coarse structure $\CE$, we say that it is \emph{$\CE$\=/bounded}.
\end{de}

Note that the whole space $X$ is $\CE$\=/bounded if and only if $\CE$ is the maximal coarse structure.

\begin{example}\label{exmp:p1:compact coarse structures}
 Let $X$ be a topological space.\footnote{%
 We are not requiring $X$ to be Hausdorff. For us a set is compact if every open cover has a finite subcover (in literature these are also called quasi\=/compact). A set is relatively compact if it is contained in a compact set.
 } 
 Then the family of sets 
 \[
  \varcrs{cpt}\coloneqq\braces{E\subseteq X\times X\mid \exists K\subseteq X\text{ compact, } E\subseteq (K\times K)\cup \Delta_{X}} 
 \nomenclature[:CE1]{$ \varcrs{cpt}$}{compact coarse structure}
 \]
 is the \emph{compact coarse structure}\index{coarse structure!compact}. If $X$ is equipped with the discrete topology then $\varcrs{cpt}$ coincides with the finite off\=/diagonal coarse structure $\varcrs{fin}$ (hence why the latter is also called discrete coarse structure).
 
 Another natural coarse structure on $X$ is the \emph{compact fiber coarse structure}: \index{coarse structure!compact fiber}
 \[
  \varcrs[fib]{cpt}\coloneqq\braces{E\subseteq X\times X\mid {}_xE\text{ and }E_y \text{ are relatively compact } \forall x,y\in X}
 \nomenclature[:CE1]{$ \varcrs[fib]{cpt}$}{compact fiber coarse structure}
 \]
 Here is a sketch of a proof that $\varcrs[fib]{cpt}$ is closed under composition: since $(E\cmp F)_y= E(F_y)=\pi_1(E\cap \pi_2^{-1}(F_y))$, it is enough to show that $E\cap \pi_2^{-1}(F_y)$ is relatively compact. 

 By assumption, $F_y$ is contained in a compact $K$. Enlarging $E$ if necessary, we may assume that every fiber $E_z$ is compact and we have that $E\cap \pi_2^{-1}(K)$ is compact.\footnote{%
 This can be shown as follows. Let $C\coloneqq E\cap \pi_2^{-1}(K)$. By assumption, $C_z\coloneqq\pi_2^{-1}(z)\cap C=E_z$ is compact for every $z\in K$.
 Let $U_i$ be an open cover of $C$, we may assume that $U_i=A_i\times B_i$, as the products of open sets form a basis for the topology of $X\times K$. For each $z\in K$, there exists a finite set of indices $i_1,\ldots ,i_n$ such that $C_z\subseteq U_1\cup\cdots\cup U_n$. It follows that the finite intersection $V_z\coloneqq B_{i_1}\cap\cdots\cap B_{i_n}$ is an open neighborhood of $z$ such that $\pi_2^{-1}(V_z)$ is contained in the finite union $U_{i_1}\cup\ldots \cup U_{i_n}$. The family $\braces{V_z\mid z\in K}$ is an open cover of $K$ and hence admits a finite subcover. We obtain a finite subcover of $C$ by considering the relevant sets $U_i$.
 }
 
 In general, the coarse structures $\varcrs[]{cpt}$ and $\varcrs[fib]{cpt}$ can be seen as extreme examples among the family of coarse structures on $X$ whose bounded sets are precisely the relatively compact subsets of $X$.
 If $(X,d)$ is a proper metric space (closed balls are compact), then the controlled entourages of the metric coarse structure $\CE_d$ have compact fibers. It follows that $\varcrs[]{cpt}\subseteq\CE_d\subseteq\varcrs[fib]{cpt}$. If $X$ has infinite diameter one can also check that the containments are strict. We leave it as an exercise to the reader to check these assertions.
\end{example}

\begin{convention}
 Example~\ref{exmp:p1:compact coarse structures} illustrates a notational convention that we will often use to denote special coarse structures: we will denote by $\CE^{mod}_{prop}$\nomenclature[:CE0]{$\CE^{mod}_{prop}$}{ coarse structure whose controlled/bounded sets are described by the property ``$prop$'', possibly with the modifier ``$mod$''} the coarse structure whose controlled/bounded sets are described by the property ``$prop$'', possibly with the modifier ``$mod$''.
\end{convention}

\begin{convention}
\label{conv:p1:x_torel_CE}
 Given a coarse structure $\CE$ on $X$, we write 
 \[
  x\torel{\CE} y \qquad \forall x,y\text{ satisfying ``some property $P$''}
 \]
 to mean that there exists a fixed controlled set $E\in \CE$ such that $x\torel{E} y$ for every $x,y$ satisfying $P$. Note that the order of the quantifiers is important: writing
 \[
  x\torel{\CE} y \qquad \forall x,y\in X
 \]
 implies that $X\times X\in \CE$, hence $(X,\CE)$ is bounded. This is much stronger than saying that for every $x,y\in X$ there is an $E\in\CE$ with $x\torel{E}y$ (\emph{i.e.}\ $(X,\CE)$ is coarsely connected).
 
 This convention will be very handy in the sequel to write statements such as
 \[
  (x\ast y)\ast z \rel{\CE} x \ast (y\ast z) \qquad \forall x,y,z\in X,
 \]
 where $\ast$ is some fixed function from $X\times X$ to $X$.
\end{convention}

For any family $(\CE_i)_{i\in I}$ of coarse structures on $X$, the  intersection $\bigcap_{i\in\ I}\CE_i$ is a coarse structure on $X$.  
For any family of relations $(E_j)_{j\in J}$ we can thus define the coarse structure $\angles{E_j\mid j\in J}$ \emph{generated} by the relations $E_j$ as the minimal coarse structure containing $E_j$ for every $j\in J$. \index{coarse structure! generated}

\begin{lem}\label{lem:p1: generated coarse structure}
 A relation $F$ is contained in $\angles{E_i\mid i\in I}$ if and only if it is contained in a finite composition $F_1\cmp\cdots\cmp F_n$ where each $F_j$ is equal to $E_i\cup \Delta_{X}$ or $\op{E_i}\cup\Delta_{X}$ for some $i\in I$.
\end{lem}

\begin{proof}[Sketch of proof]
 Let $\CF$ be the collection of relations contained in some $E\subseteq F_1\cmp \cdots \cmp F_n$ where the $F_j$ are finite unions of $E_i$, $\op{E_i}$ and $\Delta_{X}$. We claim that $\CF=\angles{E_i\mid i\in I}$. Every relation in $\CF$ must belong to $\angles{E_i\mid i\in I}$, it hence suffices to show that $\CF$ is a coarse structure. It is obvious that $\CF$ contains $\Delta_{X}$ and is closed under taking compositions and subsets. Composition and union are distributive operations, \emph{i.e.}\ for any $E_1,E_2,F_1,F_2\subseteq X\times X$ we have
 \[
  (E_1\cup E_2)\cmp(F_1\cup F_2)
  =(E_1\cmp F_1)\cup(E_1\cmp F_2)\cup(E_2\cmp F_1)\cup(E_2\cmp F_2).
 \]
 It follows that $\CF$ is also closed under finite unions.
 
 To finish the proof it suffices to show that any finite union of $E_i$, $\op{E_i}$ and $\Delta_{X}$ is contained in finite compositions of $E_i\cup\Delta_{X}$ and $\op{E_i}\cup\Delta_{X}$. This follows immediately from observing that
 \[
  E\cup F\subseteq (E\cup\Delta_{X})\cmp(F\cup\Delta_{X}).
 \]
 for any pair of relations $E,F$ on $X$.
\end{proof}

\begin{exmp}
\label{exmp:p1:disconnected union}
 Given a family of coarse spaces $(X_i,\CE_i)$ indexed over a set $I$, we can define their \emph{disconnected union}\index{disconnected union} 
 as the coarse space obtained by giving the disjoint union $X\coloneqq \bigsqcup_{i\in I} X_i$ the coarse structure $\varcrs{\sqcup_i}\coloneqq \angles{\CE_i\mid i\in I}$. Concretely, we see that $E\in \varcrs{\sqcup_i}$ if and only if there exist finitely many indices $i_1,\ldots, i_n$  and controlled sets $E_i\in \CE_i$ such that $E\subseteq \Delta_{X}\cup \bigcup_{i=1}^nE_i$. As the name suggests, each $X_i$ is coarsely disconnected from any other $X_j$.
\end{exmp}

 Given a connected coarse structure $\CE$ on $X$, it is not hard to show that there exists a metric $d$ on $X$ such that $\CE=\CE_d$ if and only if $\CE$ is generated by a countable family of controlled sets $E_n\in\CE$ (see \cite[Theorem 2.55]{roe_lectures_2003}, \cite[Theorem 9.1]{protasov2003ball} or Lemma~\ref{lem:p2:metric coarsifications} for a full argument).
 Similarly, one may check that a (possibly disconnected) coarse structure can be induced by an extended metric on $X$ if and only if it is countably generated. 
 Note that  if $X$ is an uncountable set then the finite off\=/diagonal coarse structure $\varcrs{fin}$ is not countably generated. Therefore $\varcrs{fin}$ cannot be realized as a metric coarse structure for any choice of metric (or extended metric) on $X$. 
 
 If a coarse structure $\CE$ is generated by a finite number of relations $E_1,\ldots,E_n$, then taking the union $E_1\cup\cdots \cup E_n$ shows that $\CE$ is actually generated by a single relation. We then say that $\CE$ is \emph{monogenic}\index{coarse structure!monogenic}. Monogenic coarse structures are particularly easy to control. Note that if $\CE=\angles{\bar E}$, we may as well assume that $\bar E=\op{\bar E}$ and that $\Delta_X\subseteq \bar E$. In this case we see that a relation $E$ is in $\CE$ if and only if there is some $n\in\NN$ so that $E$ is a subset of the $n$\=/fold composition $\bar E^{\cmp n}\coloneqq\bar E\cmp\cdots\cmp \bar E$.

\begin{de}\label{def:p1:crs_geod}
 A coarse space $(X,\CE)$ is \emph{coarsely geodesic}\index{coarsely!geodesic} if it is coarsely connected and the coarse structure $\CE$ is monogenic. We then say that $\CE$ is a \emph{geodesic coarse structure}\index{coarse structure! geodesic}.
\end{de}

 If $(X,\CE)$ is a coarse geodesic space it is easy to describe a metric so that $\CE=\CE_d$. In fact, if $\bar E$ is a generating relation with $\bar E=\op{\bar E}$ and $\Delta_X\subseteq \bar E$, we may then let $d(x,x')\coloneqq \min\braces{n\mid x\torel{\bar E^{\cmp n}} x'}$ for every $x\neq x'$.
 
\begin{exmp}
 Let $(X,d)$ be a metric space. Recall that the length of a curve $\gamma\colon [0,1]\to X$ is defined as 
 \[
  \norm{\gamma}\coloneqq \sup\Bigbraces{\sum_{i=1}^n d(\gamma(t_{i-1}),\gamma(t_i))\Bigmid n\in\NN,\ 0=t_0\leq\cdots\leq t_n=1}.
 \]
 The metric space is called \emph{geodesic} if for every pair of points $x,x'\in X$ there exists a curve between them whose length realizes their distance $\norm{\gamma}=d(x,x')$ (such a curve is a geodesic between $x$ and $x'$).
 In this case it is easy to see that $(X,\CE_d)$ is coarsely geodesic because $\CE_d$ is generated \emph{e.g.}\ by the $1$\=/neighborhood of the diagonal $\bar E\coloneqq\braces{(x,x')\mid d(x,x')\leq 1}\subseteq X\times X$.
 
 One can also show that a metric coarse structure $\CE_d$ is geodesic if and only if there is some fixed $r>0$ such that for every $R>0$ there is an $n=n(R)$ large enough so that for every pair of points $x,x'\in X$ with $d(x,x')\leq R$ there is a sequence $x=x_0,\ldots, x_n=x'$ with $d(x_{i-1},x_i)\leq r$ for every $i=1,\ldots ,n$. 
\end{exmp}

\begin{rmk}
 It is not hard to show (see \cite[Proposition 2.57]{roe_lectures_2003} or Appendix~\ref{sec:appendix:quasifications}) that $(X,\CE)$ is coarsely geodesic if and only if it is coarsely equivalent (Definition~\ref{def:p1:coarse_equivalence}) to $(Y,\CE_d)$, where $(Y,d)$ is some geodesic metric space.
\end{rmk}

\section{Coarse Structures via Partial Coverings}
\label{sec:p1:coarse structures via coverings}

The ability to pass from relations to partial coverings and back enables us to describe coarse structures in terms of systems of \emph{controlled coverings} of $X$. This alternative description is essentially due to Dydak--Hoffland \cite{dydak_alternative_2008}.

\begin{de}\label{def:p1:controlled.covering}
 Given a coarse structure $\CE$ on a set $X$, a \emph{controlled}\footnote{%
 One could refer to controlled coverings as ``coverings by uniformly bounded sets''. We prefer the term `controlled' because the adjective `uniformly' is somewhat overused.
 }
 \emph{(partial) covering} \index{covering!controlled} of $X$ is a (partial) covering $\fka$ such that $\diag(\fka)\in \CE$. We will denote by $\fkC(\CE)$ \nomenclature[:COV]{$\fkC(\CE)$}{family of controlled partial coverings} the family of controlled partial coverings associated with $\CE$. If we want to specify that a (partial) covering is controlled with respect to a certain coarse structure $\CE$, we say that it is \emph{$\CE$\=/controlled}.
\end{de}

\begin{example}
 If $(X,d)$ is a metric space and $\CE_d$ is the associated bounded coarse structure, then $\fkC(\CE_d)$ is the family of partial coverings by sets with uniformly bounded diameter. That is, $\fka\in\fkC(\CE_d)$ if and only if there is a $r\geq 0$ such that $\diam(A)\leq r$ for every $A\in\fka$.
\end{example}

Recall that for every relation $E$ we denote by $E(\pts X)$ the partial covering $\braces{E_x\mid x\in X}$. 
For every relation $E\in\CE$ and $x\in X$ the relation $E_x\times E_x$ is contained in $E\cmp \op{E}$. 
It follows that if $E$ is $\CE$\=/controlled then $E(\pts X)\in\fkC(\CE)$. 
The converse is not true in general (Figure~\ref{fig:p1:controlled coverings and sections}), but it does hold under the assumption that $E$ contains the diagonal. In fact, we already remarked that $E\subseteq \diag(E(\pts{X}))$ whenever $\Delta_{X}\subseteq E$.  
 
\begin{figure}
 \centering
 \includegraphics{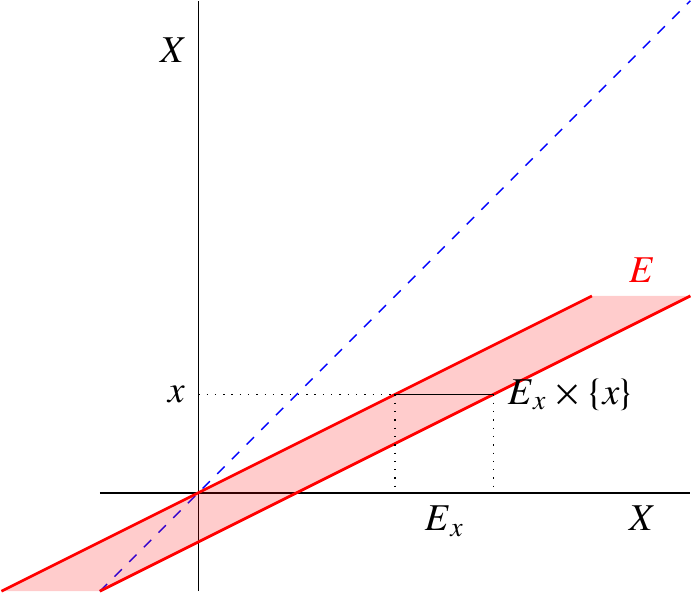}
 \caption{On $\crse X=(\RR,\CE_{\abs{\mhyphen}})$, the set $E=\braces{(x_1,x_2)\mid x_1-2\leq 2x_2\leq x_1}$ is not controlled. However, its sections $E(\pts{X})$ are the controlled covering $\braces{[t,t+1]\mid t\in \RR}$.}
 \label{fig:p1:controlled coverings and sections}
\end{figure}

Note that the family of controlled partial coverings $\fkC(\CE)$  associated with a coarse structure $\CE$ satisfies:
\begin{itemize}
 \item [(C0)] if $\fka\in\fkC(\CE)$ then $\fka\cup\pts{X}\in\fkC(\CE)$;
 \item [(C1)] if $\fka\in\fkC(\CE)$ and $\fka'$ is a refinement of $\fka$, then $\fka'\in\fkC(\CE)$;
 \item [(C2)] given $\fka,\fkb\in\fkC(\CE)$, then $\st(\fka,\fkb)\in\fkC(\CE)$.
\end{itemize}
Conditions (C1) and (C2) follow from the fact that $\CE$ is closed under taking subsets and compositions respectively. Since $\fka\cup\fkb$ is a refinement of $\st(\fka\cup\pts{X},\fkb\cup\pts{X})$, we deduce that $\fkC(\CE)$ is also closed under finite unions (this can also be deduced from the fact that $\CE$ is closed under finite unions).
As it turns out, a coarse structure $\CE$ is completely determined by $\fkC(\CE)$. Moreover, conditions (C0)--(C2) determine the families of partial coverings that are $\fkC(\CE)$ for some coarse structure $\CE$. 
We collect these observations in the following.

\begin{prop} 
\label{prop:p1:crse structures vs ctrl coverings}
 For any set $X$, there is a one\=/to\=/one correspondence
 \[
  \bigbraces{\text{coarse structures on $X$}}\longleftrightarrow
  \bigbraces{\text{families of partial coverings of $X$ satisfying {\upshape(C0)--(C2)}}}.
 \]
\end{prop}

\begin{proof}
Given a family of partial coverings $\fkC$, let $\CE(\fkC)$ be the family of relations that are contained in $\diag(\fka)$ for some $\fka\in \fkC$
\[
 \CE(\fkC)\coloneqq\braces{E\subseteq X\times X\mid \exists \fka\in\fkC,\ E\subseteq \diag(\fka)}.
\]
Such a family of relations is always closed under subsets and symmetry. It contains the diagonal $\Delta_X$ if and only $\fkC$ contains some covering of $X$ (as opposed to a partial covering). It follows from Lemma~\ref{lem:p1:composition vs star}(ii) that if $\fkC$ satisfies (C0)--(C2) then $\CE(\fkC)$ is a coarse structure. Since $(\diag(\fka))(\pts X)=\st(\pts X,\fka)$, conditions (C0)--(C2) also imply that $\fkC(\CE(\fkC))=\fkC$.
 
Now let $\CE$ be any coarse structure. It remains to show that $\CE(\fkC(\CE))=\CE$. It is clear that $\CE\subseteq \CE(\fkC(\CE))$. For the other containment it suffices to observe that if $\Delta_X\subseteq E\in \CE$ then $E\subseteq(\diag(E(\pts{X})))$.
\end{proof}

As a consequence of Proposition~\ref{prop:p1:crse structures vs ctrl coverings} we deduce that every statement  about coarse structures and coarse spaces can be translated into a statement concerning families of controlled partial covers (and vice versa). For instance, given any family of partial coverings $(\fka_i)_{i\in I}$ one may define a generated family of controlled partial coverings by considering the smallest $\fkC$ that contains every $\fka_i$ and satisfies (C0)--(C2). Of course, we would then have $\fkC=\fkC(\angles{\diag(\fka_a)\mid i\in I})$.

Throughout this monograph, it is a recommended exercise to try to rephrase statements from one language to the other. It is convenient to master both languages, as they have different strengths. In our experience, coarse structures seem to lead to briefer proofs while controlled coverings are pictorially more intuitive. 

\begin{exmp}
 Given a family $A_i$ of subsets of $X$, its intersection graph is the graph with one vertex for each $A_i$ and connecting two vertices if the intersection of the associated subsets is not empty. We also say that the radius of a graph $\CG$ is the smallest integer $n$ so that that there is a vertex that can be connected with any other vertex of $\CG$ via a path crossing at most $n$ edges.

 A coarse space $(X,\CE)$ is monogenic if and only if there exists a controlled covering $\fka$ such that for every other $\fkb\in\fkC(\CE)$ there exists an $n\in\NN$ so that $\fkb$ is a refinement of the $n$\=/th iteration of the star operation $\st(\fka,\st(\cdots,\st(\fka,\fka)))$. Concretely, this is the case if and only if every element $B\in\fkb$ can be covered with sets  $A_i\in\fka$ whose intersection graph has radius at most $n$ (Figure~\ref{fig:p1:monogenic via coverings}).
\end{exmp}

\begin{figure}
 \centering

 \includegraphics{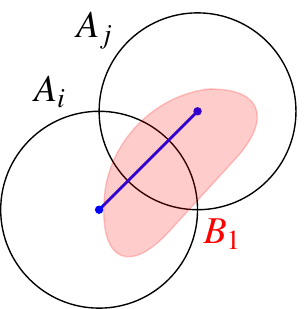}
\qquad 
 \includegraphics{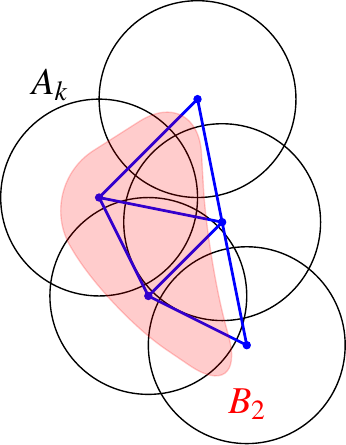}
\qquad 
 \includegraphics{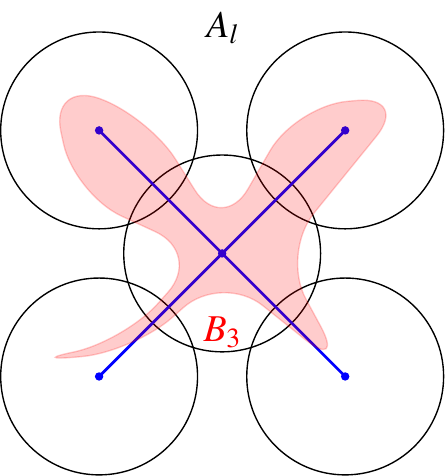}
 \caption{The partial covering $\fkb=\braces{B_1,B_2,B_3}$ is a refinement of $\st(\fka,\fka)$ because each $B_i$ is covered by sets $A_i$ with intersection graph of radius $1$.}
 \label{fig:p1:monogenic via coverings}
\end{figure}

\begin{rmk}
In addition to using controlled sets and partial coverings, one can define coarse structures in terms of families of pseudo\=/metrics. Namely, to a family $d_i\colon X\times X\to [0,\infty]$ with $i$ ranging is some index set $I$ one can associate the generated coarse structure $\angles{\CE_{d_i}\mid i\in I}$. Conversely, given any relation $E\in \CE$ one can define a pseudo metric $d_E$ as the largest pseudo\=/metric such that $d_E(x,y)\leq 1$ whenever $(x,y)\in E$. We may then associate with a coarse structure $\CE$ the family of $d_E$ as $E$ ranges in $\CE$.
We will not make use of this point of view. 
\end{rmk}

\section{Controlled Maps}

In this section we describe how functions $f\colon X\to Y$ interact with coarse structures $\CE$ and $\CF$ on $X$ and $Y$. 
We will use the notation $f\colon(X,\CE)\to(Y,\CF)$ or $f\colon X\to(Y,\CF)$ if we want to specify that some property of the function $f\colon X\to Y$ depends on $\CE$ and/or $\CF$.

\begin{de}\label{def:p1:controlled_map}
 A function $f\colon (X,\CE)\to (Y,\CF)$ is a \emph{controlled map} \index{controlled!map}
 \footnote{%
    In \cite{roe_lectures_2003} a map as in Definition~\ref{def:p1:controlled_map} is called ``bornologous''. 
    We find this nomenclature somewhat misleading, as it suggests that its defining property is that it preserves boundedness of subsets. 
    }
 if $f\times f\colon X\times X\to Y\times Y$ sends controlled sets to controlled sets, \emph{i.e.}\ $(f\times f) (E)\in \CF$ for every $E\in\CE$. 
 Equivalently, $f$ is a controlled map if it preserves controlled partial coverings: i.e.,  $f(\fka)\in\fkC(\CF)$ for every $\fka\in\fkC(\CE)$.
 
\end{de}

\begin{convention}
 For the sake of clarity, we will use the word `function' to denote any (possibly non controlled) function $f\colon X\to Y$. We will reserve the word `map' for controlled maps. 
\end{convention}

\begin{exmp}\label{exmp:p1:metric controlled map}
 Let $(X,d_{X})$ and $(Y,d_Y)$ be metric spaces. Then a function $f\colon(X,\CE_{d_{X}})\to (Y,\CE_{d_Y})$ is controlled if and only if there is a \emph{control function}, \index{control function} \emph{i.e.}\ an increasing function $\rho_+\colon [0,\infty)\to[0,\infty)$ such that
 \[
  d_Y(f(x),f(x'))\leq \rho_+(d_{X}(x,x'))
 \]
 for every $x,x'\in X$. 
 Note that $f$ is a Lipschitz function if and only if it is controlled by a linear control function. We define $f$ to be \emph{quasi-Lipschitz}\index{quasi-!Lipschitz} if it is controlled by an affine function $\rho_+(t)\coloneqq Lt+A$ for some $L,A\geq 0$. Quasi\=/Lipschitz functions are the building blocks for the theory of quasi-isometries (Appendix~\ref{sec:appendix:quasifications}), which is a mainstream tool in geometric group theory. Of course, quasi-Lipschitz functions are controlled.
\end{exmp}

\begin{de}\label{def:p1:close.functions}
 Two functions $f,g\colon X\to (Y,\CF)$ are \emph{close} \index{close function}  (denoted $f\closefn g$) \nomenclature[:REL]{$f\closefn g$}{close functions} if $f\times g\colon X\times X\to Y\times Y$ sends the diagonal $\Delta_{X}$ to a controlled set $F\in\CF$ (\emph{i.e.}\ $f(x)\torel{\CF}g(x)$ for every $x\in X$). 
 Equivalently, $f$ and $g$ are close if the family $(f\cup g)(\pts{X})\coloneqq\braces{\{f(x),g(x)\}\mid x\in X}$ is a controlled partial covering of $Y$.
\end{de}

Closeness is an equivalence relation on the set of functions from $X$ to $Y$. 
If $f_1\colon (X,\CE)\to (Y,\CF)$ is controlled and $f_2\colon X\to Y$ is close to $f_1$, then also $f_2$ is controlled.
If $f_1,f_2\colon X\to (Y,\CF)$ are close and $g\colon Z\to X$ is \emph{any} function, then $f_1\circ g$ and $f_2\circ g$ are close.
If $g'\colon (Y,\CF)\to (Z',\CD')$ is controlled, then also $g'\circ f_1$ and $g'\circ f_2$ are close.

\begin{de}\label{def:p1:coarse_equivalence}
 Let $f\colon (X,\CE)\to (Y,\CF)$ and $g\colon (Y,\CF)\to (X,\CE)$ be controlled maps. We say that $g$ is a \emph{right coarse inverse} of $f$ if the composition $f\circ g$ is close to $\id_{Y}$; and $g$ is a \emph{left coarse inverse} of $f$ if $g\circ f$ is close to $\id_X$. The function $g$ is a \emph{coarse inverse} \index{coarse!inverse} of $f$ if it is both a left and right coarse inverse (in which case $f$ is also a coarse inverse of $g$).
 A \emph{coarse equivalence}\index{coarse!equivalence} is a controlled map with a coarse inverse. We say that are $(X,\CE)$ and $(Y,\CF)$ \emph{coarsely equivalent}\index{coarsely!equivalent} if there is a coarse equivalence between them. 
\end{de}

\begin{de}\label{def:p1:proper.map}
 A function between coarse spaces $f\colon  (X,\CE) \to (Y,\CF)$ is \emph{proper}\index{proper!function} if the preimage of every bounded subset of $Y$ is bounded in $X$. It is \emph{controlledly proper} if $(f\times f)^{-1}(F)\in\CE$ for every $F\in\CF$. 
 Equivalently, $f$ is controlledly proper \index{controlledly proper} if the preimage of every controlled partial covering of $Y$ is a controlled partial covering of $X$. 
 A controlled, controlledly proper map is a \emph{coarse embedding} \index{coarse!embedding}.
\end{de}

If $Z$ is a subset of a coarse space $(X,\CE)$, it inherits a \emph{subspace coarse structure} \index{coarse structure!subspace} $\CE|_Z\coloneqq \braces{E\cap Z\times Z\mid E\in\CE}$. \nomenclature[:CE1]{$\CE{\vert}_Z$}{subspace coarse structure} It is not hard to check that $f\colon(X,\CE)\to(Y,\CF)$ is a coarse embedding if and only if $f\colon(X,\CE)\to(\im(f),\CF|_{\im(f)})$ is a coarse equivalence.
If $f \colon (X, \CE) \to (Y, \CF)$ is a controlled function and $f \closefn g$, then if $f$ is proper (resp. controlledly proper) then so is $g$ (resp. controlledly proper).

\

We say that a subset $Z\subseteq X$ is \emph{coarsely dense}\index{coarsely!dense} if there is a controlled set $E\in\CE$ such that $E(Z)=X$. In terms of controlled coverings, $Z$ is coarsely dense if there exists an $\fka\in\fkC(\CE)$ such that $\st(Z,\fka)=X$ or, equivalently, $\pts{X}$ is a refinement of $\st\paren{\pts{Z},\fka}$.

\begin{de}
 A function $f\colon X\to (Y,\CF)$ is \emph{coarsely surjective}\index{coarsely!surjective} if $f(X)$ is coarsely dense in $Y$.
\end{de}

\begin{exmp}
For metric coarse structures, it is helpful to rephrase the above properties in terms of distances.
 Let $(X,d_{X})$ and $(Y,d_Y)$ be metric spaces. Two functions $f_1,f_2\colon(X,\CE_{d_{X}})\to (Y,\CE_{d_Y})$ are close if there exists some $R\geq 0$ such that $d_Y(f_1(x),f_2(x))\leq R$ for every $x\in X$. A function $f\colon(X,\CE_{d_{X}})\to (Y,\CE_{d_Y})$ is proper if the preimage of a set with finite diameter has finite diameter. The function $f$ is coarsely surjective if there is an $R\geq 0$ so that its image is an $R$\=/dense net of $Y$, \emph{i.e.}\ the balls $B(f(x);R)$ cover $Y$.
 
 Also note that $f$ is a controlledly proper if for every $r\geq 0$ there is a $R\geq 0$ such that $\diam( f^{-1}(A) )\leq R$ for every subset $A\subseteq Y$ of diameter at most $r$.
 Equivalently, this is the case if and only if there exists an increasing and unbounded control function $\rho_-\colon [0,\infty)\to[0,\infty)$ such that 
 \[
  \rho_-(d_{X}(x,x'))\leq  d_Y(f(x),f(x')).
 \] 
 Together with Example~\ref{exmp:p1:metric controlled map}, this shows that $f$ is a coarse embedding if and only if 
 \[
  \rho_-(d_{X}(x,x'))\leq  d_Y(f(x),f(x'))\leq \rho_+(d_{X}(x,x')) 
 \]
 for some appropriately chosen control functions $\rho_-$ and $\rho_+$. This is the standard definition of coarse embedding for metric spaces known by geometric group theorists (sometimes also referred to as ``uniform embedding'' \cite[Section 7.E]{niblo_geometric_1993}). Similarly, we also see that $(X,\CE_{d_{X}})$ and $(Y,\CE_{d_Y})$ are coarsely equivalent as coarse spaces if and only if $(X,d_{X})$ and $(Y,d_Y)$ are coarsely equivalent as metric spaces \cite[Section 1.4]{nowak2012large}. 
Two metrics $d,d'$ on the same set $X$ are \emph{coarsely equivalent}\index{coarsely!equivalent!metric} if the identity function $\id_X\colon(X,d)\to(X,d')$ is a coarse equivalence.
\end{exmp}

The following will not come as a surprise to a geometric group theorist:

\begin{lem}\label{lem:p1:coarse.eq.iff.surjective.coarse.emb}
 A controlled map $f\colon (X,\CE)\to (Y,\CF)$ is a coarse equivalence if and only if it is a coarsely surjective coarse embedding.
\end{lem}
\begin{proof}
It is easy to check that if $f$ has a right coarse inverse then it is coarsely surjective. Moreover, if $f$ has a left coarse inverse, $g$, then $f$ must be controlledly proper. This is because if $x\in f^{-1}(y)$, $x'\in f^{-1}(y')$ and $y\torel{F}y'$, then
\[
 x\torel{\paren{\id_X\times(g\circ f)}(\Delta_X)}g(y)
 \torel{g\times g(F)} g(y')
 \torel{\paren{(g\circ f)\times\id_X}(\Delta_X)}x'.
\]
This proves one implication.

For the other one, let $\fkb_0$ be a covering of $Y$ such that $\st(f(X),\fkb_0)=Y$. For every $y\in Y$ define $g(y)$ by choosing a point $x\in X$ such that both $f(x)$ and $y$ belong to the same $B\in\fkb_0$. 
 The function $g$ is a controlled map because for every controlled covering $\fkb$ of $Y$, the image $g(\fkb)$ is a refinement of $f^{-1}(\st(\fkb,\fkb_0))$. 
 The map $g$ is a coarse inverse of $f$ since $\paren{\id_Y\cup (f\circ g)}(\pts{Y})$ is a refinement of $\fkb_0$ and $\paren{\id_{X}\cup(g\circ f)}(\pts{X})$ is a refinement of $f^{-1}(\fkb_0)$.
\end{proof}

\begin{rmk}
 Of course, coarse surjectivity does not imply that there exists a (controlled) left coarse inverse in general. Consider for instance the controlled surjection $\id_\RR\colon(\RR,\mincrs)\to(\RR,\CE_d)$, where $\CE_d$ is the coarse structure induced by the Euclidean metric.
 Similarly, coarse embeddings need not admit right  coarse inverses. In fact, one can show that $f\colon X\to Y$ has a right coarse inverse if and only if $f$ is controlledly proper and $Y$ admits a ``coarse retraction'' to $f(X)$, \emph{i.e.}\ there is a controlled map $p\colon (Y,\CF)\to (f(X),\CF|_{f(X)})$ that is a right coarse inverse for the inclusion $f(X)\hookrightarrow Y$.
 Note that $(\RR,\CE_d)$ does not coarsely retract to $\{2^n\mid n\in\NN\}\subset\RR$.
\end{rmk}

\section{The Category of Coarse Spaces}
\label{sec:p1:category of coarse spaces}
Next we define the category of coarse spaces and introduce some important notation and conventions that we will use throughout the text. Before doing so, we need one last definition:

\begin{de}
A \emph{coarse map}\index{coarse!map} between two coarse spaces $(X,\CE)$ and $(Y,\CF)$ is an equivalence class of controlled maps $(X,\CE)\to(Y,\CF)$, where $f,g\colon X\to Y$ are equivalent if they are close (denoted $f\closefn g$).
\end{de}

\begin{rmk}
 Here our conventions diverge from Roe's. In \cite{roe_lectures_2003} a coarse map is defined as a controlled map (not an equivalence class) that is also proper.  
\end{rmk}

The composition of coarse maps is a well\=/defined coarse map because the closeness relation is well behaved under composition with controlled maps. We may hence give the following:

\begin{de}
  The \emph{category of coarse spaces}\index{category of!coarse spaces} is the category \Cat{Coarse}\nomenclature[:CAT]{\Cat{Coarse}}{category of coarse spaces} whose objects are coarse spaces $\crse{X}=(X,\CE)$ and whose morphisms are coarse maps:
  \[
  \cmor(\crse{X},\crse{Y})
  \coloneqq\bigbraces{\crse{f}\colon\crse{X}\to\crse{Y}\mid \crse{f}\text{ coarse map}}
  =\bigbraces{f\colon (X,\CE)\to (Y,\CF)\mid f\text{ controlled map}\,}/\closefn.
  \]
\end{de}

Of course, the identity morphism of $\crse{X}$ is the equivalence class of the identity function $\cid_{\crse X}\coloneqq [\id_{X}]$.\nomenclature[:CN]{$\cid_{\crse X}$}{identity coarse map} 
Given two coarse spaces $\crse X=(X,\CE)$, $\crse Y=(Y,\CF)$ and coarse maps $\crse f =[f]\colon(X,\CE)\to(Y,\CF)$ and $\crse g=[g]\colon(Y,\CF)\to(X,\CE)$, it follows from the definitions that $g$ is a left coarse inverse of $f$ if and only if $ \crse{ g \circ f} =\cid_{\crse X} $. That is, the morphism $\crse g$ is a left inverse of $\crse f$ in the category \Cat{Coarse}.
In particular, $\crse X$ and $\crse Y$ are isomorphic as objects in \Cat{Coarse} if and only if they are coarsely equivalent.

\begin{convention}\label{conv:p1:bold for coarse}
 We use bold symbols when we want to stress that we are working in the \Cat{Coarse} category. So that $\crse{X}$ will be a coarse space  and $\crse{f \colon X\to Y}$ a coarse map. 
 We will keep writing $(X,\CE)$ when we want to highlight that a coarse space is a specific set $X$ equipped with the coarse structure $\CE$: this can be helpful \emph{e.g.}\ when describing functions on $X$ or in the arguments involving more than one coarse structure. 
 
 We will use the convention that non-bold symbols denote a choice of representatives for bold symbols. For instance, if $\crse {f\colon X\to Y} $ is a coarse map then $f\colon X\to Y$ will denote a representative for $\crse f$ (\emph{i.e.}\ $f$ is controlled and $\crse f=[f]$ is the equivalence class of all functions that are close to $f$). There are a few instances where this notation is ambiguous---for example when a set is given more than one coarse structure---in these cases we will simply use the ``equivalence class'' notation $[\mhyphen]$. If confusion may arise, we will further decorate the notation with the relevant coarse structure. For instance, $[f]_\CF$ denotes the equivalence class of functions $f\colon X\to (Y,\CF)$ up to $\CF$\=/closeness.
\end{convention}

\begin{rmk}
    The notions of coarse left/right inverse, coarse embedding, coarse surjection and properness are all invariant under taking close functions. In particular, they define properties for coarse maps. 
    In general, we use the adjective `coarse' to denote properties that make sense in \Cat{Coarse}. For instance, a `coarse' property must certainly be invariant under coarse equivalence.
\end{rmk}

\

Every set $I$ can be seen as a coarse space equipped with the minimal coarse structure $\mincrs$. Every function $f\colon (I,\mincrs)\to (X,\CE)$ is controlled. Further, two functions $f,g\colon (X,\CE)=(I,\mincrs)$ are close if and only if $f=g$. This shows that \Cat{Coarse} contains \Cat{Set} in a trivial manner. 

\begin{de}
 A set equipped with the minimal coarse structure is said to be a \emph{trivially coarse space}. \index{trivially coarse!space}
\end{de}

\begin{convention}
 Throughout the monograph, we reserve the adjective `trivial' for the minimal coarse structure $\mincrs$. The maximal coarse structure $\maxcrs$ on a set consisting of more than one point will never be referred to as trivial. Instead, it might be called \emph{bounded}. 
\end{convention}

For later reference, we note that every coarse object $\crse X$ is uniquely determined by the coarse maps from trivially coarse spaces into $\crse{X}$. Namely, any coarse map $\crse f\colon \crse X\to \crse Y$ induces a map 
\[
 \crse{f}_I\colon\cmor((I,\mincrs),\crse{X})\to\cmor((I,\mincrs),\crse{Y})
\]
by composition and we have:

\begin{lem}\label{lem:p1:trivially coarse spaces characterise maps}
 Let $\crse f,\crse g\colon \crse X\to \crse Y$ be coarse maps, then $\crse f=\crse g$ if and only if $\crse f_I=\crse g_I$ for every trivially coarse set $I$.
 Furthermore, $\crse f\colon \crse X\to \crse Y$ is a coarse embedding (resp. coarsely surjective) if and only if $\crse f_I$ is injective (resp. surjective) for every trivially coarse set $I$. 
\end{lem}
\begin{proof}
Since the definition of morphisms in the coarse category is well\=/posed, it is automatic that $\crse f_I=\crse g_I$ whenever $\crse f=\crse g$. 
 The converse implication is also simple. Let $\crse{X}=(X,\CE)$, $\crse Y=(Y,\CF)$ and fix representatives $f,g\colon X\to Y$ for $\crse f$, $\crse g$. Further let $I=X$, and consider the identity function $\id_X\colon (I,\mincrs)\to(X,\CE)$. Then $\crse f_I([\id_{X}]_\CE)=[f]_\CF$ and $\crse g_I([\id_{X}]_\CE)=[g]_\CF$, so that the assumption $\crse f_I([\id_{X}]_\CE)=\crse g_I([\id_{X}]_\CE)$ implies that $[f]_\CF=[g]_\CF$ as controlled maps $(X,\mincrs)\to (Y,\CF)$. Since closeness does not depend on the choice of coarse structure on the domain, it follows that $\crse f=\crse g$.
 
 The `furthermore' statement is only marginally more complicated. It is immediate to check that if $f$ is a coarse embedding then $\crse{f}_I$ is injective. 
 For the converse implication, assume for contradiction that $f$ is not controlledly proper and let $I=\CE$. By assumption, there is a controlled set $F\in\CF$ such that $U\coloneqq(f\times f)^{-1}(F)$ is not controlled in $X$. For every $E\in I =\CE$ there is a pair $(x_1,x_2)\in U\smallsetminus E$ and we let $g_1(E)\coloneqq x_1$ and $g_2(E)\coloneqq x_2$. We thus obtained two functions $g_1,g_2\colon I\to X$ such that $[f\circ g_1]=[f\circ g_2]$ but $[g_1]\neq [g_2]$.
 
 The argument for surjectivity is similar. It is again obvious that when $\crse f$ is coarsely surjective then $\crse{f}_I$ is surjective for every $I$. Vice versa, assume that $Y$ is not contained in $\st(f(X),\fkc)$ for any $\fkc\in \fkC(\CF)$. Let $I=\fkC(\CF)$ and choose for every $\fkc\in\fkC(\CF)$ a $y_\fkc\in Y\smallsetminus\st(f(X),\fkc)$. Then the function $h\colon I \to Y$ sending $\fkc$ to $y_\fkc$ is not close to $f\circ g$ for any choice of $g\colon I\to X$ and hence $[h]\notin \im(\crse{f}_I)$.
\end{proof}

\begin{rmk}
 The fact that composition with $\crse{f}$ induces an injection (resp. surjection) $\cmor(\variable,\crse X)\to\cmor(\variable,\crse Y)$ whenever $\crse f$ is a coarse embedding (resp. coarsely surjective) is also true when considering non\=/trivially coarse spaces. 
\end{rmk}

\begin{rmk}
 We conclude this section by remarking that defining morphisms in \Cat{Coarse} as equivalence classes of \emph{functions} creates a certain degree of imbalance. 
 In fact, properties that can be quantified on the codomain of a function have a clear `coarse' analogue (\emph{e.g.}\ coarse surjectivity is a well\=/defined notion), while properties that need to be quantified on the domain (\emph{e.g.}\ being a well\=/defined function) do not have immediate coarse analogues.
 This can be a nuisance. For example, when constructing a coarse\=/inverse for a coarsely surjective coarse embedding (Lemma~\ref{lem:p1:coarse.eq.iff.surjective.coarse.emb}) there is an obvious `coarsely well\=/defined' candidate but it is then necessary to make some extra unnatural choices to obtain an actual function.
 
 One way to resolve this conundrum would be as follows. A function $f\colon X\to Y$ is identified with a subset of $X\times Y$ whose sections over every point $x\in X$ consist of exactly one point.
 It is this extra requirement of uniqueness that introduces asymmetry in the definitions. One could have a more `symmetric' notion of mapping by removing it altogether and replacing functions with binary relations $E\subseteq X\times Y$. A ``coarsely well\=/defined function'' could then be defined as a relation $E\subseteq X\times Y$ so that the sections ${}_xE$ are a controlled partial covering of $Y$ and $\pi_X(E)$ is coarsely dense in $X$ (\emph{i.e.}\ the ${}_xE$ is non-empty for a coarsely dense set of points).
\end{rmk}

\section{Binary Products}\label{sec:p1:binary products}
Given two coarse spaces $(X,\CE)$ and $(Y,\CF)$ we define the \emph{product coarse structure} \index{coarse structure!product} on $X\times Y$ as
\[
 \CE\otimes\CF\coloneqq \bigbraces{D\subseteq (X\times Y)\times(X\times Y)\mid\pi_{13}(D)\in \CE,\ \pi_{24}(D)\in\CF},
\]
\nomenclature[:CE1]{$\CE \otimes \CF$}{product coarse structure}
where $\pi_{13}\colon(X\times Y)\times(X\times Y)\to X \times X$ and $\pi_{24}\colon(X\times Y)\times(X\times Y)\to Y\times Y$ are the projections to the first \& third, and second \& fourth coordinates respectively.  
Given $E\in\CE$ and $F\in\CF$ we let 
\[
 E\otimes F\coloneqq\pi_{13}^{-1}(E)\cap\pi_{24}^{-1}(F)=\braces{(x_1,y_1,x_2,y_2)\mid (x_1,x_2)\in\CE,\ (y_1,y_2)\in\CF}.
\]
\nomenclature[:R]{$E \otimes F$}{product relation}
With this notation, $D\in\CE\otimes\CF$ if and only if it is contained in $E\otimes F$ for some $E\in\CE$ and $F\in\CF$. In other words, $\CE\otimes\CF$ is the coarse structure generated by taking the closure under subsets of the family $\{E\otimes F\mid E\in\CE,\ F\in\CF\}$.
Using the explicit description of the coarse structure generated by a family of relations (Lemma~\ref{lem:p1: generated coarse structure}), one can also show that $\angles{E_i\mid i\in I}\otimes\angles{E_j\mid j\in J}=\angles{E_i\otimes F_j\mid i\in I,\ j\in J}$.

Given any $E\subseteq X\times X$ and $F\subseteq Y\times Y$, we have:
\begin{equation}\label{eq:p1:otimes as composition}
 E\otimes F=\paren{E\otimes \Delta_Y}\cmp\paren{\Delta_{X}\otimes F}
 =\paren{\Delta_{X}\otimes F}\cmp\paren{E\otimes \Delta_Y}
\end{equation}
The above equation follows directly from the definitions, but it turns out to be very important. We will apply it frequently, since it allows us to decompose products of relations into simpler products.

\begin{rmk}
\label{rmk:p1:control of maps from product generated}
 As an immediate application, \eqref{eq:p1:otimes as composition} implies that when $\CE=\angles{E_i\mid i\in I}$ and $\CF=\angles{F_j\mid j\in J}$ the product $\CE\otimes\CF$ is generated by  $\braces{E_i\otimes \Delta_Y\mid i\in I}\cup\braces{\Delta_{X}\otimes F_j\mid j\in J}$.
 As a consequence, in order to verify that a given function $f\colon (X\times Y,\CE\otimes\CF)\to (Z,\CD)$ is controlled it is enough to show that $f\times f(E_i\otimes \Delta_Y)$ and $f\times f(\Delta_X\otimes F_j)$ belong to $\CD$ for every $i\in I$, $j\in J$.
\end{rmk}

\begin{exmp}
 If $(X,\CE_{d_X})$ and $(Y,\CE_{d_Y})$ are metric coarse spaces then their product $(X,\CE_{d_X})\times(Y,\CE_{d_Y})$ is also a metric coarse space. For instance, the product coarse structure $\CE_{d_X}\otimes \CE_{d_Y}$ is equal to the metric coarse structure $\CE_{d_X+d_Y}$, where $d_X+d_Y$ is the $\ell_1$ product metric:
 \[
  (d_X+d_Y)\bigparen{(x,y),(x',y')}\coloneqq d_X(x,x')+d_Y(y,y').
 \]
 In this case, it is easy to see what Equation~\eqref{eq:p1:otimes as composition} means. It is essentially the observation that 
 \[
  (d_X+d_Y)\bigparen{(x,y),(x',y')}
  = (d_X+d_Y)\bigparen{(x,y),(x',y)}+(d_X+d_Y)\bigparen{(x',y),(x',y')}.
 \]
\end{exmp}

\

To express the product coarse structure in terms of controlled partial coverings, note that
\[
 \fkC(\CE\otimes\CF)=\bigbraces{\fkd\text{ partial covering of }X\times Y\bigmid \pi_1(\fkd)\in\fkC(\CE),\ \pi_2(\fkd)\in\fkC(\CF)}.
\]
Equivalently, $\fkC(\CE\otimes\CF)$ is the system of partial coverings obtained by taking the closure under refinements of $\fkC(\CE)\times\fkC(\CF)$,
where $\fka\times\fkb=\braces{A\times B\mid A\in\fka,\ B\in \fkb}$.

\

It is immediate to check that $\crse X\times\crse  Y\coloneqq (X\times Y,\,\CE\otimes\CF)$ is a product of $\crse X=(X,\CE)$ and $\crse Y=(Y,\CF)$ in the category \Cat{Coarse}. 
Namely, the projections $\crse{\pi_{X}\colon X\times Y\to X}$ and $\crse{\pi_{Y}\colon X\times Y\to Y}$ are coarse maps, and for every pair of coarse maps $\crse{f_{1}\colon Z\to X}$, $\crse{f_{2}\colon Z\to Y}$ their product is a coarse map $\crse{f_{1}\times f_{2}\colon Z\to X\times Y}$.

\begin{rmk}
 It would not be possible to construct products using Roe's original definition of coarse maps: his extra requirement that coarse maps be proper makes it impossible to have projections.
\end{rmk}

\section{Equi Controlled Maps}\label{sec:p1:equi controlled.maps}
In this section we introduce the notion of equi controlled maps.
This will prove to be essential in our study of coarse groups. Observe that if $I$ is some index set and $f_i\colon X\to Y$ for $i\in I$ is a family of functions, then they can be seen as sections of the function $f\colon X\times I \to Y$ defined by $f(x,i)\coloneqq f_i(x)$.

\begin{de}\label{def:p1:equi controlled}
 Let $(I,\CC)$ be a coarse space. A family of controlled maps $f_i\colon (X,\CE)\to (Y,\CF)$ with $i\in I$ is \emph{$(I,\CC)$\=/equi controlled} if $f\colon (X,\CE)\times (I,\CC) \to (Y,\CF)$ is controlled.
 We say that the $f_i$ are \emph{equi controlled} if they are $(I,\mincrs)$\=/equi controlled.\footnote{%
 Families of $(I,\mincrs)$\=/equi controlled functions are called ``uniformly bornologous'' in \cite{brodskiy2008coarse}.
 } 
 \index{equi!controlled (map)}
\end{de}

The following statement follows easily from the definition, but we wish to spell it out explicitly:

\begin{lem}\label{lem:p1:equivalent.definitions.equi controlled}
 Let $f_i\colon (X,\CE)\to (Y,\CF)$ be a family of maps where $i\in I$ is some index set. The following are equivalent:
 \begin{enumerate}[(i)]
  \item the $f_i$ are equi controlled;
  \item for every controlled set $E$ of $X$ the union $\bigcup_{i\in I}\paren{f_i\times f_i}(E)$ is controlled in $Y$;
  \item for every controlled partial covering $\fka$ of $X$ the union $\bigcup_{i\in I}f_i(\fka)$ is a controlled partial covering of $Y$.
 \end{enumerate}
\end{lem}
\begin{proof}
 $(i)\Leftrightarrow(ii)$: the function $f\colon (X,\CE)\times(I,\mincrs) \to (Y,\CF)$ is controlled if and only if $\paren{f\times f}(E\otimes \Delta_I)\in\CF$ for every $E\in \CE$. On the other hand, we have:
 \[
  \paren{f\times f}(E\otimes \Delta_I)
  =\bigcup_{i\in I}\paren{f\times f}(E\otimes\{i\})
  =\bigcup_{i\in I}\paren{f_i\times f_i}(E).
 \]
 $(i)\Leftrightarrow(iii)$: again, $f\colon (X,\CE)\times(I,\mincrs) \to (Y,\CF)$ is controlled if and only if $f(\fka\times \pts{I})$ is a controlled partial covering of $Y$ for every $\fka\in \fkC(\CE)$. The conclusion follows because
 \(
  f(\fka\times\pts{I})
  =\bigcup_{i\in I}f(\fka\times\{i\})
  =\bigcup_{i\in I}f_i(\fka).
 \)
\end{proof}

\begin{rmk}
 Using Convention~\ref{conv:p1:x_torel_CE}, the functions $f_i\colon(X,\CE)\to(Y,\CF)$ are equi controlled if and only if for every $E\in\CE$ we have
 \[
  f_i(x)\rel{\CF}f_i(x')\qquad \forall x\torel{E}x',\ \forall i\in I.
 \]
\end{rmk}

Vice versa, any function $f\colon X\times Y\to Z$ can be seen as a family of functions in two distinct ways: ${}_xf\coloneqq f(x,\variable)\colon Y\to Z$ with $x$ ranging in $X$, and $f_y\coloneqq f(\variable,y)\colon X\to Z$ with $y$ ranging in $Y$. The following lemma is very useful and deceptively simple to prove:

\begin{lem}\label{lem:p1:equi controlled.sections.iff.controlled}
 Let $(X,\CE),(Y,\CF)$ and $(Z,\CD)$ be coarse spaces. A function $f\colon (X,\CE)\times (Y,\CF)\to (Z,\CD)$ is controlled if and only if both families ${}_xf$ and $f_y$ are equi controlled.
\end{lem}

Note that the above statement asks that ${}_xf$ are $(X,\mincrs)$\=/equi controlled (as opposed to $(X,\CE)$\=/equi controlled, in which case the statement would be trivial).

\begin{proof}
One implication is obvious: if $f$ is controlled with respect to $\CE\otimes\CF$ then it is a fortiori controlled with respect to $\paren{\CE\otimes\mincrs}$ and $\paren{\mincrs\otimes\CF}$.

For the other implication, \eqref{eq:p1:otimes as composition} states that we can write $E\otimes F$ as $\paren{E\otimes\Delta_Y}\cmp\paren{\Delta_{X}\otimes F}$.
We therefore see that 
\[
 \paren{f\times f}(E\otimes F)\subseteq \paren{f\times f}\paren{E\otimes\Delta_Y}\cmp \paren{f\times f}\paren{\Delta_{X}\otimes F}
\]
is controlled. This concludes the proof because the relations $E\otimes F$ generate $\CE\otimes\CF$.
\end{proof}

\chapter{Properties of the Category of Coarse Spaces}\label{ch:p1:properties of coarse}

In this chapter we review a few more properties and constructions regarding coarse maps and spaces. Some highlights include notions of coarse images, subspaces and intersections. These concept and the related language will be important in the sequel.

\section{Pull-back and Push-forward}\label{sec:p1:pullback_pushforward}
Coarse structures are well behaved under pull-back\footnote{%
This section describes pull\=/backs and push\=/forwards of structures. We are \emph{not} describing pull\=/backs and push\=/outs of commutative diagrams.
} but less so under push-forward: 

\begin{de}\label{def:p1:pull-back}
 Let $f\colon X\to Y$ be any function and $\CF$ a coarse structure on $Y$. The \emph{pull-back} \index{coarse structure!pull-back} of $\CF$ under $f$ is the coarse structure 
 \[
  f^*(\CF)\coloneqq\braces{E\in X\times X\mid f\times f(E)\in\CF}. \nomenclature[:CE1]{$f^*(\CF)$}{pull back of coarse structure}
 \]
\end{de}

It is easy to verify that the pull back of a coarse structure is indeed a coarse structure.
It follows from the definition that $f\colon (X,f^*(\CF))\to (Y,\CF)$ is a controlled map. In fact, $f\colon (X,\CE)\to (Y,\CF)$ is controlled if and only if $\CE\subseteq f^*(\CF)$. Note also that if $f_1,f_2\colon X\to (Y,\CF)$ are close functions, then $f_1^*(\CF)=f_2^*(\CF)$. 
It follows that if we are given a coarse map $\crse f\colon (X,\CE)\to (Y,\CF)$ then the pull\=/back $f^\ast(\CF)$ contains $\CE$ and does not depend on the choice of representative $f$.

It is also clear that $f\colon (X,f^*(\CF))\to (Y,\CF)$ is a coarse embedding. Lemma~\ref{lem:p1:coarse.eq.iff.surjective.coarse.emb} implies the following:

\begin{cor}\label{cor:p1:pullback.gives.isomorphism}
 If $f\colon X\to (Y,\CF)$ is a coarsely surjective function, then $f$ is a coarse equivalence between $(X,f^*(\CF))$ and $(Y,\CF)$.
\end{cor}

The pull-back is well\=/behaved under composition: $(f_2\circ f_1)^*(\CF)=f_1^*(f_2^*(\CF))$. We implicitly used the notion of pull\=/back when defining products: the product coarse structure on $(X_1,\CE_1)\times (X_2,\CE_2)$ is the intersection of the pullbacks of the projections: $\CE_1\otimes\CE_2 = \pi_1^*(\CE_1)\cap\pi_2^*(\CE_2)$. In turn, this implies that the pull\=/back is also well behaved with respect to taking products: $(f_1\times f_2)^*(\CF_1\otimes\CF_2)=(f_1^*(\CF_1))\otimes(f_2^*(\CF_2))$.

\

Push-forwards are somewhat less well-behaved. Given a function $f\colon (X,\CE)\to Y$, the natural definition for the \emph{push-forward} of $\CE$ under $f$ is $f_*(\CE)\coloneqq\angles{f\times f(E)\mid E\in \CE}$. Note that it is important to take the coarse structure \emph{generated} by $\braces{f\times f(E)\mid E\in \CE}$, as this set need not contain the diagonal nor be closed under composition. In this monograph we will not explicitly use push\=/forwards: they only appear implicitly. 

\section{Controlled Thickenings and Asymptoticity}
\label{sec:p1:thickenings and indexed coverings}
Coarse containments and asymptoticity are useful concepts in the theory of coarse spaces which are key for defining coarse subspaces. The notation introduced in the following definitions will be used fairly often in the sequel.
Let $\crse X=(X,\CE)$ be a coarse space. 
 
\begin{de}\label{def:p1:controlled thickening}
 A \emph{controlled thickening} \index{controlled!thickening} of a subset $A\subseteq X$ is a `star neighborhood' $\st(A,\fka)$ where $\fka\in\fkC(\CE)$ is a controlled partial covering. A set $B\subseteq X$ is \emph{coarsely contained} \index{coarse!containment}  in $A$ (denoted $B\csub A$) \nomenclature[:REL]{$A\csub B$}{coarse containment of subsets} if it is contained in a controlled thickening of $A$. Equivalently, $B\csub A$ if $B\subseteq E(A)$ for some $E\in\CE$. \end{de}

Let $E,F \in \CE$, and $A \subset X$. Since $E(F(A))=(E\cmp F)(A)$, it follows that coarse containment is a preorder on subsets of $X$. This preorder induces an equivalence relation:

\begin{de}\label{def:p1:asymptotic}
 Two subsets $A,B\subseteq X$ are \emph{asymptotic} \index{asymptotic (subsets)} (denoted $A\ceq B$) \nomenclature[:REL]{$A\ceq B$}{asymptotic subsets} if $A\csub B$ and $B\csub A$. We denote the equivalence class of $A$ by $[A]$.
\end{de}

If there is risk of confusion with regard to the coarse structure being used, we will include it in the notation and write $B\csub_\CE A$, $B\ceq_\CE A$ and $[A]_\CE$.

\begin{exmp}
 If we identify points $x\in X$ with singletons $\{x\}\subseteq X$, then $\ceq$ defines an equivalence relation on $X$. It follows from the definition, that $x\ceq x'$ if and only if they belong to the same coarsely connected component of $(X,\CE)$ (Definition~\ref{def:p1:connected coarse structure}).
\end{exmp}

\begin{example}
 If $(X,d)$ is a metric space and $\CE_d$ is the induced coarse structure, then $A\ceq B$ if and only if they are at finite Hausdorff distance. In particular, the notion of `being asymptotic' is coherent with standard terminology (\emph{e.g.}\ infinite geodesic rays in metric spaces are asymptotic subsets if and only if they are asymptotic as paths \cite{bridson_metric_2013,drutu_geometric_2018}).  
\end{example}

The following is lemma is a simple but important observation.

\begin{lem}\label{lem:p1:controlled maps and asymptotic subsets}
 Let $\crse X=(X,\CE)$, $\crse Y=(Y,\CF)$ be coarse spaces. 
 \begin{enumerate}
  \item If two functions $f,f'\colon X\to Y$ are close then $f(A)\ceq f'(A)$ for every subset $A\subseteq X$. 
 \item Given subsets $A,A'\subseteq X$ and a function $f\colon X\to Y$, if $A\csub A'$ and $f$ is controlled then $f(A) \csub f(A')$. In particular, controlled maps preserve asymptoticity.
 \end{enumerate}
\end{lem} 
\begin{proof}
 By definition, if $f$ and $f'$ are close then the partial covering $(f\cup f')(\pts{X})=\braces{\braces{f(x),f'(x)}\mid x\in X}$ is controlled. Noting that $f'(A)\subseteq \st(f(A),(f\cup f')(\pts{X}))$ shows that $f'(A)\csub f(A)$. We similarly see that $f(A)\csub f'(A)$, which proves the claim. 
 
 For the second statement, let $\fkc$ be a controlled covering of $(X,\CE)$ so that $ A \subseteq \st (B, \fkc)$. Then, $f(A) \subseteq f\paren{\st (B, \fkc)}\subseteq\st(f(B),f(\fkc))$. The partial covering $f(\fkc)$ is controlled on $(Y,\CF)$ because $f$ is a controlled map. Therefore $f(A)\csub f(B)$.
\end{proof}

\section{Coarse Subspaces, Restrictions, Images and Quotients}
\label{sec:p1:subspaces and quotients}
When $\crse X=(X,\CE)$ is a coarse space, it is unnatural to work with points and subsets of $X$, since in the category \Cat{Coarse} points and subsets are specified only up to closeness. Instead, we will use the following:

\begin{de}\label{def:p1:coarse subspace}
 A \emph{coarse subspace}\index{coarse!subspace} of a coarse space $\crse X=(X,\CE)$ is an equivalence class $[Y]$ of subsets of $X$ up to $\CE$\=/asymptoticity. Coarse subspaces are denoted with a bold inclusion symbol $\crse {Y\subseteq X}$.\nomenclature[:CN]{$\crse {Y\subseteq X}$}{coarse subspace}
 Given any $x\in X$, we say that $\crse x\coloneqq [x]$ is a \emph{coarse point}\index{coarse!point} in $\crse X$, denoted $\crse{x\in X}$.
\end{de}

\begin{figure}
 \centering

 \includegraphics{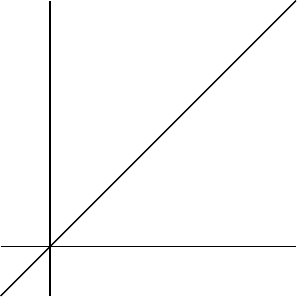}
\qquad 
 \includegraphics{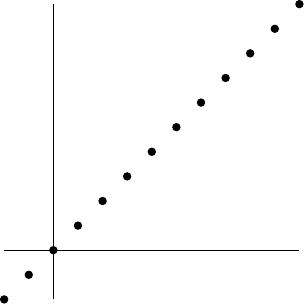}
\qquad 
 \includegraphics{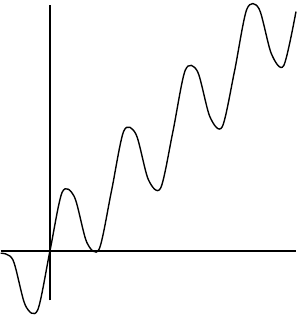}
\qquad 
 \includegraphics{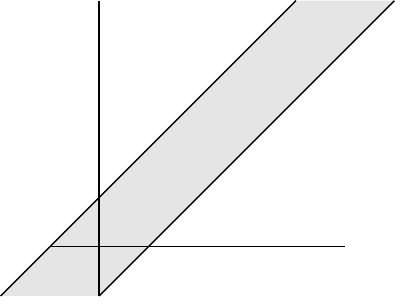}
 \caption{Asymptotic subsets of $\RR^2$ represent the same coarse subspace of $\crse X=(\RR,\CE_{\norm{\mhyphen}})$}
 \label{fig:p1:subspaces of R}
\end{figure}

\begin{rmk}
\label{rmk:p1:coarse points are connected components}
 A coarsely connected space $\crse X$ only has one coarse point $\crse x$, because any two points $x,x'\in X$ are close. More generally, any bounded subset of $X$ is a representative for $\crse{x\in X}$.
 For a general coarse space $\crse X$, the coarse points in $\crse X$ are in correspondence with its coarsely connected components. 
\end{rmk}

Up to coarse equivalence, a coarse subspace uniquely defines a coarse space. Namely, if $Y\subseteq X$ and $Y'\subseteq X$ are two asymptotic subsets in $\crse X$, then there is a natural coarse equivalence between $\crse Y =(Y,\CE|_Y)$ and $\crse Y' =(Y',\CE|_{Y'})$. By this we mean that there exists a unique coarse equivalence $\crse{i\colon Y\to Y'}$ such that for every coarse function $\crse{ f\colon X\to Z}$ the restriction $[f|_{Y}]$ is equal to the composition $[f|_{Y'}]\circ \crse i$.
It is simple to verify this claim directly or, more conceptually, it can be deduced by showing that the definition of coarse subspace is compatible with the categorical notion of subobject, see Appendix~\ref{sec:appendix:subobjects}.

This justifies our choice of notation: if $\crse {Y\subseteq X}$ is a coarse subspace then $\crse Y$ is indeed a coarse space (well\=/defined up to canonical coarse equivalence) and it will be treated as such. 
Since $\crse Y$ is only defined up to coarse equivalence, it is a slight abuse of notation to write $\crse Y=(Y,\CE|_Y)$: it would be more correct to specify that $(Y,\CE|_Y)$ is a choice of representative for $\crse Y$. However, this nuance is inessential and in the sequel we will often write $\crse Y=(Y,\CE|_Y)$ to signify that we fixed a representative $Y\subseteq X$. Importantly, every `coarse' property and construction defined for coarse spaces restricts to a well\=/defined notion for coarse subspaces (meaning that it does not depend on the choice of representative).

\

Given a coarse subspace $\crse{Y\subseteq X}$ and a coarse map $\crse{f\colon X\to Z}$, the restriction $f|_Y\colon Y\to Z$ defines a coarse map $\crse{f|_{Y}\colon Y\to Z}$. Because of the naturality of the coarse equivalence between $(Y,\CE|_Y)$ and any other representative of $\crse {Y\subseteq X}$, we can use the notation $\crse{f|_{Y}}$ without worrying about the choice of representative of $\crse Y$ used define $f|_Y$. We state this as a definition:

\begin{de}
 The coarse map $\crse{f|_{Y}\colon Y\to Z}$ is the \emph{restriction}\index{coarse!restriction (of a map)}\nomenclature[:CN]{$\crse{f{\vert}_{Y}}$}{restriction to a coarse subspace} of $\crse{ f\colon X\to Z}$ to the coarse subspace $\crse Y$.
\end{de}

It follows from Lemma~\ref{lem:p1:controlled maps and asymptotic subsets} that close functions have asymptotic images. We can therefore use coarse subspaces to define the coarse image of a coarse map.

\begin{de}\label{def:p1:coarse image}
Let $\crse f\colon\crse X\to\crse Z$ be a coarse map. The \emph{coarse image of $\crse{f}$}\index{coarse!image}\nomenclature[:CN]{$\cim(\crse f),\ \crse{f(X)}$}{coarse image}
 is the coarse subspace $\cim(\crse f)\crse{\subseteq Z}$ determined by the asymptotic equivalence class of $\im(f)$, \emph{i.e.}\  
\(
 \cim(\crse f)\coloneqq[\im(f)].
\)
The \emph{coarse image $\crse f(\crse Y)$ of a coarse subspace} $\crse Y\subseteq \crse X$ is the coarse subspace
\[
 \crse{f(Y)}\coloneqq [f(Y)]=\cim(\crse{f|_{Y}}).
\]
\end{de}

With an abuse of notation, we may identify $\crse X$ with a coarse subspace of itself. Namely, for a coarse subspace $\crse {Y\subseteq X}$ we write $\crse{ Y= X}$ to signify that the whole space $X$ belongs to the equivalence class $[Y]$ (\emph{i.e.}\ $Y$ is coarsely dense in $\crse X$). This is convenient, because when given a coarse map $\crse{f\colon X\to Z}$ we can write $\cim(\crse f)\crse{=f(X)}$ and say that $\crse f$ is coarsely surjective if and only if $\crse{f(X)=Z}$.
Similarly,  $\crse{f\colon X\to Z}$ is a coarse embedding if and only if the map  $\crse{f\colon X\to f(X)}$ is a coarse equivalence of $\crse X$ with its coarse image.
 
Given two coarse functions $\crse f_1\colon \crse X_1\to \crse X_2$ and $\crse f_2\colon \crse X_2\to \crse X_3$, the coarse image of a coarse subspace $\crse{Y\subseteq X}_1$ under the composition is $\crse f_2\circ \crse f_1(\crse Z)=\crse f_2(\crse f_1(\crse Z))$.

\begin{rmk}
 The coarse image of a coarse map $\crse{f\colon X\to Z}$ does not depend on the coarse structure of $\crse X$. In particular, every function $f\colon X\to Z$ has a well defined coarse image---even if $f$ is not controlled. On the contrary, the coarse image of a coarse subspace of $\crse X$ does depend on the coarse structure of $\crse X$ and the definition of $\crse{f(Y)}$ only makes sense for controlled functions.
\end{rmk}

\

The situation for coarse quotients is simpler, albeit a little counterintuitive.

 \begin{de}\label{def:p1:coarse quotient space}
 A \emph{coarse quotient}\index{coarse!quotient (space)}\index{quotient!coarse space} of a coarse space $\crse X=(X,\CE)$ is a coarse space $(X,\CE')$ where $\CE\subseteq \CE'$. If $\CR$ is any family of relations on $X$, the \emph{coarse quotient of $\crse X$ by $\CR$} is defined as $\crse X/\CR\coloneqq(X,\angles{\CE\cup\CR})$.\nomenclature[:CN]{$\crse X/\CR$}{coarse quotient by relations} 
 \end{de}
 
 According to this convention, the quotient $(X,\CE')$ with $\CE\subseteq \CE'$ can be denoted by $\crse X/\CE'$. In this piece of notation we are not using a bold ``$/$'' symbol because we are specifying the family of relations pointwise. 
 The definition of quotient satisfies the universal property: a coarse map $\crse {f\colon X\to Y}=(Y,\CF)$ factors through $\crse X/\CR$ if and only if $f\times f(R)\in\CF$ for every $R\in\CR$. One can easily verify that this is compatible with the definition of quotient object (Appendix~\ref{ch:appendix:categorical aspects}).
 If $\crse{ q\colon X\to Y}$ is coarsely surjective, then it factors through a coarse equivalence $\crse{X}/q^*(\CF)\to\crse Y$, where $q^*(\CF)$ is the pull\=/back coarse structure. For this reason we may abuse the nomenclature and say that a coarse space $\crse Y$ with coarse surjection $\crse{q\colon X\to Y}$ is a coarse quotient of $\crse X$. 
 
 Coarse quotients are somewhat easier than coarse subspaces because $\crse X/\CR$ is a well defined coarse space. This is opposed to subspaces, which only determine coarse spaces up to coarse equivalence. It may appear peculiar that the base set $X$ does not change when passing to coarse quotients, however, this is a necessity. The informal reason for this is that taking coarse quotients is a less disruptive operation than quotienting sets: it may retain information on ``how close'' are the points that are being identified.  
 More precisely, quotient sets are only defined when quotienting by an equivalence relation $\sim$. The transitivity assumption on $\sim$ holds exactly, no matter how many times it is applied. In contrast, when taking a coarse quotient $\crse X/\{R\}$ for a non transitive relation $R$ we ``remember'' how many times we need to apply $R$ to witness that two points are equivalent in the coarse quotient.
 
 \begin{exmp}
  If $\sim$ is an equivalence relation on a set $X$ and $x_1,x_2,\ldots$ is a sequence of points then $x_i\sim x_{i+1}$ for every $i\in\NN$ if and only if $x_1\sim  x_i$  for every $i\in\NN$. In particular, if we take the quotient set $X/{\sim}$ then the sequence $([x_i])_{i\in \NN}$ is equal to the constant sequence $([x_1])_{i\in\NN}$.
  
  On the other hand, let $R$ be a symmetric, reflexive but non transitive relation on $X$. If $\crse X=(X,\mincrs)$, a sequence of points $(x_i)_{i\in\NN}$ will be close to the constant sequence $(x_1)_{i\in\NN}$ in the coarse quotient $\crse X/\{R\}$ if and only if there is some $n$ such that $x_1\rel{R^{\cmp n}} x_{i}$ for every $i\in\NN$.
  However, if we are given a sequence of points $(x_i)_{i\in \NN}$ such that $x_i\rel{R} x_{i+1}$ for all $i\in \NN$, we can only deduce that $x_1\rel{R^{\cmp i}} x_{i+1}$. Hence $(x_i)_{i\in\NN}$ need not be equivalent to a constant sequence in $\crse X/\{R\}$.
  
  For instance, consider on $\RR$ the relation $R\coloneqq\{(x,y)\mid \abs{x-y}\leq 1\}$. Then $\angles{R}=\CE_{\abs{\mhyphen}}$, hence $(R,\mincrs)/\{R\}=(\RR,\CE_{\abs{\mhyphen}})$. In particular, any unbounded sequence of points $(x_i)_{i\in \NN}$ is not close to a constant sequence, even if each single point $x_i$ is ``equivalent'' to any other. 
 \end{exmp}

 \begin{rmk}
  If $\sim$ is an equivalence relation on $X$ and $\crse X=(X,\mincrs)$ is trivially coarse, then $\crse X/\{\sim\}$ can be naturally identified with $(X/{\sim},\mincrs)$. More generally, if $\crse X=(X,\CE)$ then one can verify that $\crse X/\{\sim\}$ is coarsely equivalent to $(X/{\sim},\pi_*(\CE))$, where $\pi\colon X\to X/{\sim}$ is the quotient map and $\pi_\ast(\CE)$ is the push\=/forward coarse structure.
 \end{rmk}

 \begin{rmk}
  Given a subset $Y\subseteq X$, the relation $Y\times Y$ is an equivalence relation and the quotient of $X$ by $Y\times Y$ is just the quotient collapsing $Y$ to a single point. More interestingly, if $\crse X$ is a coarse space and $Y'\ceq Y$ are asymptotic subsets, then $\crse X/\{Y\times Y\}=\crse X/\{Y'\times Y'\}$. One may thus ``quotient out'' coarse subspaces of $\crse X$ and obtain a well defined coarse quotient.
 \end{rmk}

\section{Containments and Intersections of Coarse Subspaces}\label{sec:p1:containements and intersections}
Coarse spaces are defined as asymptotic equivalence classes of subsets. Asymptoticity is the equivalence relation induced from the preorder given by coarse containments (Definition \ref{def:p1:controlled thickening}). In particular, coarse containment of sets descends to a partial order on the set of coarse subspaces of a coarse space $\crse X$. Namely, we make the following definition:

\begin{de}\label{def:p1:containment}
 Given coarse subspaces $\crse Y$ and $\crse Z$ of $\crse X$, we say that $\crse{Y}$ is \emph{coarsely contained}\index{coarse!containment}\nomenclature[:CN]{$\crse {Y\subseteq Z}$}{coarse containment of coarse subspaces} in $\crse Z$ if $Y\csub Z$. We denote coarse containments by $\crse{ Y\subseteq_{X} Z}$ (we will soon drop the subscript $\crse X$ from the notation).
\end{de}

As announced, $\crse{\subseteq_{X}}$ is a partial order: if $\crse{ Y\subseteq_{X} Z}$ and $\crse{ Z\subseteq_{X} Y}$ then $\crse Y=\crse Z$ are the same coarse subspace of $\crse{X}$. The reason why we can drop the subscript $\crse X$ from the notation is the following observation:

\begin{lem}\label{lem:p1:containments and subspaces}
 Let $\crse{ Y\subseteq X}$ be a coarse subspace. There is a natural correspondence:
 \[
  \braces{\crse{Z\subseteq Y}}
  \longleftrightarrow
  \braces{\crse{Z\subseteq X}\mid\crse{Z\subseteq_{X}Y}}.
 \]
\end{lem}
\begin{proof}
 Let $\crse{\iota\colon Y\to X}$ be the coarse map defined by the inclusion of $Y$ in $X$ (equivalently, $\crse \iota$ is the restriction $\crse{(\cid_{X})|_{Y}}$). 
 Then $\crse{Y=\iota(Y)}$ and for coarse subspace $\crse{Z\subseteq Y}$ the coarse image $\crse{\iota (Z)}$ is a coarse subspace of $\crse X$ coarsely contained in $\crse Y$.
 
 Vice versa, given a coarse subspace of $\crse X$ with $\crse{Z\subseteq_{X}Y}$, we may as well assume that $Z\subseteq Y$. Since the coarse space $\crse Y$ is well defined up to natural coarse equivalence, we see that the resulting coarse subspace $\crse{Z\subseteq Y}$ does not depend on the choice of representative $Y\subseteq X$.
\end{proof}

\begin{exmp}
\label{exmp:p1:correspondence of subspaces concretely}
 The geometric meaning is illuminated by a more hands-on proof of Lemma~\ref{lem:p1:containments and subspaces}.
 Realize the coarse subspace $\crse{Y}\subseteq \crse X$ as a coarse space $\crse Y=(Y,\CE|_Y)$ by choosing a representative $Y\subseteq X$ with $[Y]=\crse Y$. It is clear that every coarse subspace of $\crse Y$ uniquely determines a coarse subspace of $\crse X$.
 
 For the other direction, let $\crse{ Z\subseteq X}$ be a a coarse subspace coarsely contained in $\crse Y$, and choose a representative $Z\subseteq X$ for it. A priori, there is no reason to expect that $Z\subseteq Y$, because we have already fixed a representative for $\crse Y$. In the above proof of Lemma~\ref{lem:p1:containments and subspaces}, we switched to a different representative $Y'$ for $\crse Y$ and used the natural coarse equivalence $(Y',\CE|_{Y'})\to(Y,\CE|_Y)$ to construct the appropriate coarse subspace of $\crse Y$. An alternative approach is as follows. 
 By assumption, there exists a controlled covering $\fka$ so that $Z\subseteq \st(Y,\fka)$. It follows that $Z$ is in fact contained in the $\fka$\=/neighborhood of $Z'\coloneqq Y\cap \st(Z,\fka)$. Hence $Z\ceq_\CE Z'$ are asymptotic subsets of $X$, and we may hence use $Z'$ as a representative for $\crse Z$. Now $Z'$ is a subset of $Y$, so $[Z']_{\CE|_Y}$ is a genuine subspace of $\crse Y$.
\end{exmp}

\begin{figure}
 \centering
 \includegraphics{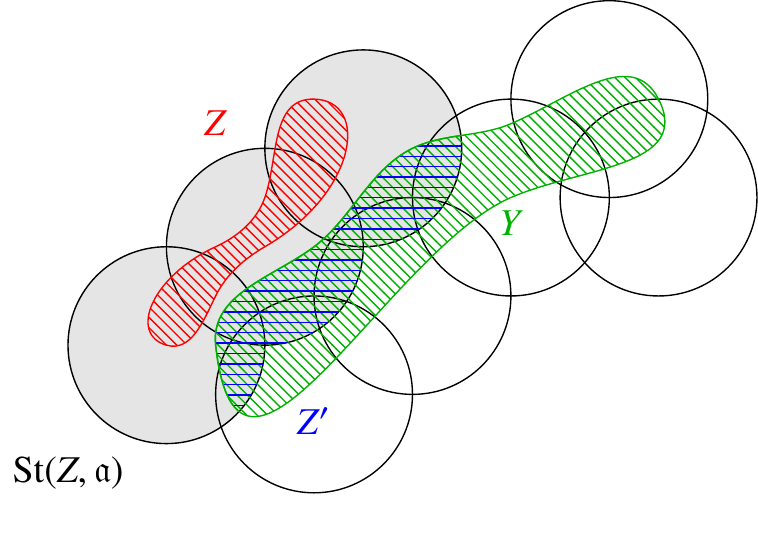}
 \caption{To realize $[Z]$ as a coarse subspace of $Y$, we take a large enough $\fka\in\fkC(\CE)$ so that the intersection $Z'=\st(Z,\fka)\cap Y$ is asymptotic to $Z$. %Example~\ref{exmp:p1:correspondence of subspaces concretely}
 }
 \label{fig:p1:correspondence of subspaces concretely}
\end{figure}

\

The partial ordering on coarse subspaces can be used to define intersections as meets. However, we wish to warn the reader that this is a somewhat delicate matter. We use the following:

\begin{de}\label{def:p1:intersection}
 Let $\crse Y_1,\ldots, \crse Y_n$ be coarse subspaces of $\crse X$. If it exists, their \emph{coarse intersection}\index{coarse!intersection} is a largest common coarse subspace (denoted $\crse Y_1\crse\cap \cdots\crse\cap \crse Y_n$).\nomenclature[:CN]{$\crse Y_1\crse\cap \cdots\crse\cap \crse Y_n$}{coarse intersection} In other words, $\crse Y_1\crse\cap \cdots\crse\cap \crse Y_n$ is a coarse subspace of $\crse X$ such that
 \[
  (\crse Y_1\crse\cap \cdots\crse\cap \crse Y_n)\subseteq \crse Y_i\qquad \text{for each $i=1,\ldots,n$}
 \]
 and so that if $\crse{Z \subseteq Y}_i$ for every $i=1,\ldots,n$ then $\crse{Z\subseteq Y}_1\crse\cap \cdots\crse\cap \crse Y_n$.
\end{de}

\begin{rmk}
 The coarse intersection may be undefined. However, when it exists it is unique. 
\end{rmk}
 
 \begin{lemma}\label{lem:p1:intersection}
  Let $Y_i$ be representatives of $\crse Y_i$ for $i=1,\ldots ,n$. A subset $Z\subseteq X$ is a representative for the intersection $\crse Y_1\crse\cap\ldots\crse\cap\crse Y_n$ if and only if it satisfies the following:
  \begin{itemize}
   \item there is a controlled covering $\fka$ so that $Z\subseteq \st(Y_i,\fka)$ for every $i=1,\ldots,n$;
   \item for every controlled covering $\fka$ there is a controlled covering $\fkb$ such that 
   \[
    \st(Y_1,\fka)\cap\cdots\cap\st(Y_n,\fka)\subseteq \st(Z,\fkb).
   \]
  \end{itemize}
 \end{lemma}
\begin{proof}
 Let $\crse Z\coloneqq [Z]$. The first condition is equivalent to saying that $\crse Z\subseteq \crse Y_i$ for every $i=1,\ldots, n$. Now, fix a controlled covering $\fka$. Then the coarse subspace $[\st(Y_1,\fka)\cap\cdots\cap\st(Y_n,\fka)]$ is coarsely contained in $\crse Y_i$ for each $i=1,\ldots ,n$. If $\crse Z$ is indeed equal to the coarse intersection of the $\crse Y_i$'s, it follows that $\st(Y_1,\fka)\cap\cdots\cap\st(Y_n,\fka)\csub Z$, which is precisely the meaning of the second condition.
 
 Vice versa, assume that the second condition holds. Let $Z'\subseteq X$ be a set so that $Z'\csub Y_i$ for each $i=1,\ldots,n$. In other words, $Z'\subseteq\st(Y_i,\fka_i)$ for some controlled covering $\fka_i$. The finite union $\fka \coloneqq \bigcup_{i=1}^n\fka_i$ is again a controlled covering, and $Z'\subseteq \st(Y_1,\fka)\cap\cdots\cap\st(Y_n,\fka)$. Since the latter is contained in $\st(Z,\fkb)$, we see that $Z'\csub Z$. Hence $\crse Z$ is the coarse intersection. 
\end{proof}

Notice a coarse intersection may be much larger than a set-wise intersection of two \emph{representatives} of coarse subspaces. For example, in $(\RR,\CE_{\abs{\mhyphen}})$ we have $[2\ZZ]=[2\ZZ+1]=[\RR]$, so that the coarse intersection $[2\ZZ]\crse{\cap} [2\ZZ+1]=[\RR]$, even though $2\ZZ$ and $2\ZZ+1$ are disjoint subsets.

The next example shows the coarse intersection may indeed not be defined in general.

\begin{exmp}\label{exmp:p1:no intersection}
 Equip $\RR^2$ with the Euclidean metric and consider the coarse space $(\RR^2,\CE_d)$. Let $A\subset \RR^2$ be the $x$-axis. Choose an increasing sequence $b^{(1)}_k\in\RR$ with $\abs{b^{(1)}_k-b^{(1)}_{k+1}}\to\infty$. Choose a second sequence $b^{(2)}_k$, so that $b^{(1)}_k< b^{(2)}_k <b^{(1)}_{k+1}$ and both $\abs{b^{(1)}_k-b^{(2)}_{k}}$ and $\abs{b^{(2)}_k-b^{(1)}_{k+1}}$ go to infinity. Recursively, keep choosing sequences $b^{(n)}_k$ so that 
 \begin{itemize}
  \item $b^{(n)}_k< b^{(n+1)}_k <b^{(n)}_{k+1}$, 
  \item both $\abs{b^{(n)}_k-b^{(n+1)}_{k}}$ and $\abs{b^{(n+1)}_k-b^{(n)}_{k+1}}$ go to infinity. 
 \end{itemize}

 Now let 
 \(
  B_n\coloneqq \braces{(b_k,n)\mid k\in \NN}
 \)
and $B\coloneqq \bigcup_{n\in\NN}B_n$. Then the coarse intersection $\crse{ A \cap B}$ does not exist. In fact, for any fixed $N\in\NN$, let $Z_N$ be the intersection of the (closed) $N$\=/neighborhoods of $A$ and $B$. Then $Z_N$ contains $B_n$ for every $n\leq N$, and it is contained in a neighborhood of $\bigcup_{n\leq 2N}B_n$.
However, the sets $B_n$ are defined so that $B_{2N+1}$ is not contained in any bounded\=/radius thickening of $\bigcup_{n\leq 2N}B_n$. Hence $Z_{2N+1}\npreceq Z_N$. It follows from Lemma~\ref{lem:p1:intersection} that it is not possible to find a largest coarse subspace of $\crse X$ that is coarsely contained in both $A$ and $B$.
\end{exmp}

\begin{figure}
 \centering
 \includegraphics[scale=0.8]{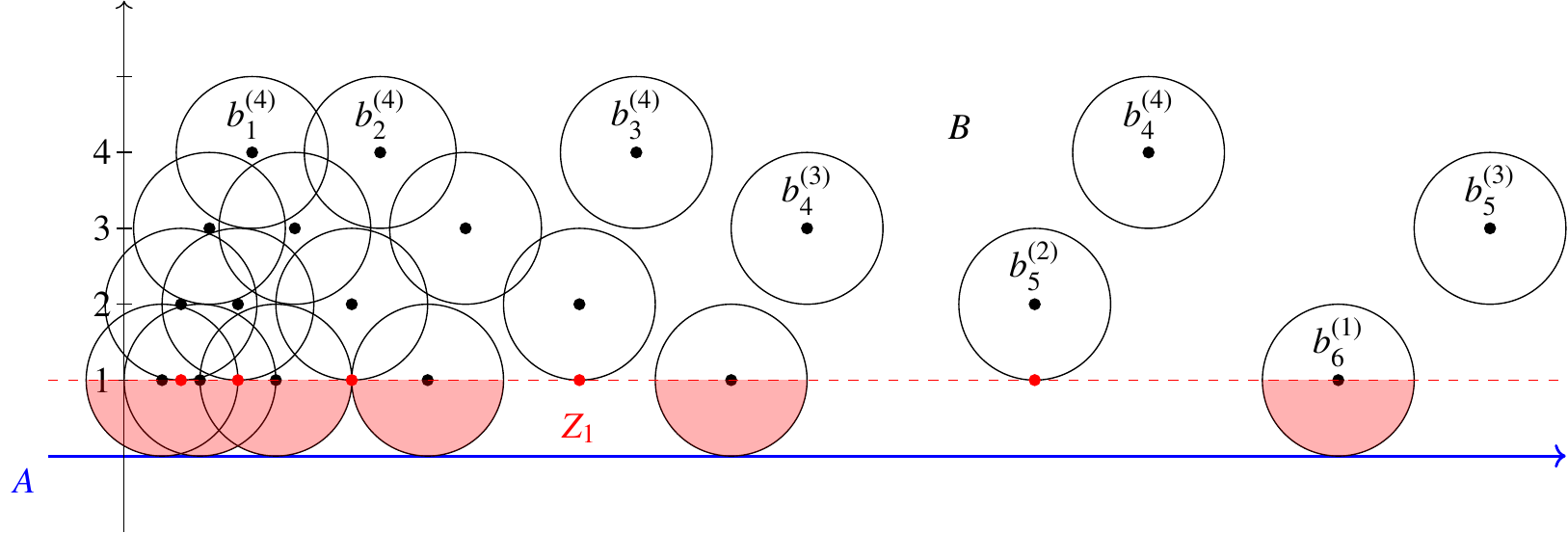}
 \caption{The coarse intersection of $[A]$ and $[B]$ in $(\RR^2,\CE_{\norm{\mhyphen}})$ is not defined because the intersections of the bounded neighborhoods of $A$ and $B$ do not stabilize. 
 %Picture for Example~\ref{exmp:p1:no intersection}
 }
 \label{fig:p1:no intersection}
\end{figure}

\

In a similar spirit, it is also possible to define the coarse preimage of a coarse map:

\begin{de}\label{def:p1:coarse preimage}
 Let $\crse{f\colon X\to Y}$ be a coarse map and $\crse{ Z\subseteq Y}$ a coarse subspace. If it exists, the \emph{coarse preimage of $\crse Z$}\index{coarse!preimage} is a largest coarse subspace of $\crse X$ whose coarse image is coarsely contained in $\crse Z$ (denoted $\crse{f^{-1}}(\crse Z)$).\nomenclature[:CN]{$\crse{f^{-1}}(\crse Z)$}{coarse preimage} In other words, $\crse{f^{-1}}(\crse Z)$ is a coarse subspace of $\crse X$ such that
 \(
  \crse f( \crse{f^{-1}}(\crse Z) )\crse\subseteq \crse Z
 \)
 and so that if $\crse f(\crse{W}) \crse\subseteq \crse Z$ then $\crse {W}\crse\subseteq \crse{f^{-1}}(\crse Z)$.
\end{de}

\begin{rmk}
 Just as for coarse intersection, the coarse preimage may not exist and if it exists it is unique. 
 Note that the coarse intersection $(\crse Y_1\crse\cap \crse Y_2 )\crse\subseteq \crse Y_1$ can be seen as the coarse preimage of $\crse Y_2$ under the coarse embedding $\crse Y_1\crse\hookrightarrow \crse X$.
\end{rmk}

The following result follows easily from the definitions. Its proof is left as an exercise.
\begin{lem}\label{lem:p1:preimage under coarse equivalence}
 If $\crse{f\colon X\to Y}$ is a coarse equivalence and $\crse{g \colon Y\to X }$ is a coarse inverse, then for every coarse subspace $\crse{Z\subseteq Y}$ the coarse preimage $\crse{f^{-1}(Z)}$ exists and is equal to $\crse{g(Z)}$. 
\end{lem}

\begin{rmk}
 One could also define a \emph{coarse union} of subspaces in much the same manner as intersections. These behave better than coarse intersections in that finite coarse unions are always well\=/defined. What one would expect to be true of finite unions holds also in the coarse category. However, in the sequel we will never have occasion to use coarse unions.
\end{rmk}

\chapter{Coarse Groups}\label{ch:p1:coarse.groups}%%%%%%%%%%%%%%%%%%%%%

Having finished the relevant preliminaries on the category of coarse spaces, we are finally ready to introduce the main object of study of this monograph: coarse groups. These are defined as group objects in \Cat{Coarse} (see below).

\begin{convention}
 In the sequel we will work with both groups (in the usual sense) and coarse groups. To avoid confusion, we will always specify the category we are working on by calling the former \emph{set\=/groups}\index{set\=/groups}.
\end{convention}

As far as we are aware, this work is the first to define and study coarse groups as group objects. However, many examples of coarse groups are obtained by equipping set\=/groups with appropriate coarse structures (we call these \emph{coarsified set\=/groups}, see Chapter~\ref{ch:p2:coarse structures on set groups}). Various authors have investigated how set\=/groups interact with coarse structures: as coarsified set\=/groups or via coarse actions (Definition~\ref{def:p1:coarse action}). In that context, a number of elementary results discussed here have already been discovered and we recover them as a special instances of more general statements. Among relevant literature, we should certainly include \cite{brodskiy2008coarse,brodskiy_svarc-milnor_2007,dikranjan2019coarse,dikranjan_categorical_2017, dikranjan2020categories,hernandez2011balleans, nicas2012coarse, protasov2019free, protasov2020coarse, protasov2018bornological, protasov2019sequential,rosendal2017coarse}.

\section{Preliminary: Group Objects in a Category}
\label{sec:p1:group object}
We start by setting some notation and recalling the definition of group object.
Let \Cat{C} be a category with finite products and let $T$ be a terminal object (\emph{i.e.}\ an object such that for every object $A$ in \Cat{C} there is a unique morphism $A\to T$):
$T$ can be defined as the product of the empty family of objects.

A pair of morphisms ${f}_1\colon A\to B_1$ and ${f}_2\colon A\to B_2$ uniquely determines a morphism to the product 
\[
 ({f}_1,{f}_2)\colon A\to B_1\times B_2.
\]
Given morphisms ${f}_1\colon {A}_1\to {B}_1$ and ${f}_2\colon {A}_2\to {B}_2$ we denote their product by 
 \[
  {f}_1\times{f}_2\coloneqq (f_1\circ\pi_1,f_2\circ\pi_2)\colon {A}_1\times {A}_2\to{B}_1\times{B}_2,
 \]
where $\pi_1$ and $\pi_2$ are the projections of $A_1\times A_2$.
By composition, a morphism of the terminal object $ x\colon T\to  A$ induces a ``constant mapping'' of any other coarse space $ B\to T\xrightarrow{ x} A$. We abuse notation, and denote this composition by $ x\colon  B\to A$.

\begin{de}\label{def:p1:group.object}
 A \emph{group object}\index{group object} in the category \Cat{C} is the datum of an object $G$ and morphisms $\ast \colon  G\times  G\to  G$, ${e\colon T\to G}$ and ${\inversefn \colon G\to G}$ such that the following \emph{Group Diagrams}\index{Group Diagrams} commute:
 
\noindent
\begin{minipage}{.36\textwidth}
  \[
  \begin{tikzcd}[row sep=4 em]
    {G\times G\times G} \arrow[r, "\id_G \times \, \ast"] \arrow[d,swap, "{\ast\, \times \id_G}"] 
    & {G\times G}      \arrow[d, "{\ast}"] 
    \\
    {G\times G} \arrow[r, " \ast"]
    &  G  
  \end{tikzcd} 
  \]
    \vspace{-1.5 em}
  \captionof*{figure}{\hspace{1 em}(Associativity)}
  \end{minipage}%
    \begin{minipage}{.32\textwidth}
    \[
    \begin{tikzcd}[row sep=4 em]
        G \arrow[dr, "\id_G"]\arrow[r, "{(\id_G\,,\,e )}"]\arrow[d,swap, "{(e\,,\,\id_G)}"]  & {G\times G}   \arrow[d, " \ast"] \\
        {G\times G} \arrow[r, "\ast"]&  G  
    \end{tikzcd} 
    \]
    \vspace{-1.5 em}
    \captionof*{figure}{\hspace{1.5 em}(Identity)}
    \end{minipage}%
    \begin{minipage}{.32\textwidth}
    \[
    \begin{tikzcd}[row sep=4 em]
        {G}  \arrow[dr, "e"]\arrow[r, "{(\id_{ G}\,,\,\inversefn)}"] \arrow[d,swap, "{(\inversefn\,,\,\id_G)}"] & {G\times G}   \arrow[d, "\ast"] \\
        {G\times G} \arrow[r, " \ast"]&  G.  
    \end{tikzcd} 
    \]
    \vspace{-1.5 em}
    \captionof*{figure}{\hspace{2 em}(Inverse)}
    \end{minipage}
    
\vspace{1 em}
\end{de}

\begin{exmp}
 In the category \Cat{Set}, a terminal object is a singleton $\tobj$ and it is immediate to verify that commutativity of the Group Diagrams is equivalent to the axioms of associativity, identity and inversion. That is, group objects in \Cat{Set} are precisely the set\=/groups (\emph{i.e.}\ groups in the usual sense).
 
 Similarly, topological groups, Lie groups and algebraic groups are group objects in the category of topological spaces, differentiable manifolds and algebraic varieties respectively.
\end{exmp}

Observe that if \Cat{C} is a category so that the morphisms between objects $\mor{C}(A,B)$\nomenclature[:MOR]{$\mor{C}(A,B)$}{morphisms from $A$ to $B$ in the category \Cat{C}} 
are sets, then the set of morphisms with a group object is a set\=/group. Namely, if $G$ is a group object and $A$ is any other object in \Cat{C}, we can define an operation $\odot$\nomenclature[:REL]{$f\odot g$}{product of morphisms}  on set of coarse maps $\mor{C}( A, G)$ as the composition
\[
 f \odot  g \coloneqq A \xrightarrow{( f, g)}  G \times  G \xrightarrow{\ast}  G. 
\]
This is the analogue of the pointwise multiplication of functions. A simple diagram chase proves the following:

\begin{lemma}\label{lem:p1:mor.is.group}
Let $G$ be a group object in a category \Cat{C}. The set $\mor{C}( A, G)$ equipped with the binary operation $\odot$ is a set\=/group, where $ e\in\mor{C}( A, G)$ is the unit and the inverse of an element $ f\in\mor{C}( A, G)$ is the composition $\inversefn\circ  f$.
\end{lemma} 

Notice that for every object $A$ the definition of the set\=/group operation on $\mor{C}(A,G)$ only depends on $G$ and $\ast$. Since in set\=/groups units and inverses are uniquely determined, we obtain as a consequence that the inversion and identity morphisms of a group object are uniquely determined by the multiplication morphism:

\begin{cor}\label{cor:p1:uniqueness.of.unit&inverse} 
Let $G$ be a group object in a category \Cat{C}.
Let $G$ be an object and $\ast\colon G\times G\to G$ a morphism. If there exist morphisms $e\colon T\to G$ and $\inversefn\colon G\to G$ equipped with which $(G,\ast)$ is a group object, then $e$ and $\inversefn$ are unique.
\end{cor} 
\begin{proof}[Proof of the corollary] 
Suppose $ T \xrightarrow{ e}  G$ and $ T \xrightarrow{{e'}}  G$ are two unit morphisms for $ G$. By Lemma \ref{lem:p1:mor.is.group} $\mor{C} ( T, G)$ is a set\=/group and $ e, {e'}$ are both units in $\mor{C}( T,  G)$. Therefore $ e= {e'}$ by uniqueness of the unit in a set\=/group.

Now suppose $\inversefn \colon  G \to  G$ is an inversion morphism.
For every $ f\in\mor{C}( G, G)$, the composition $\inversefn\circ  f$ is the (unique) inverse element of $f$ in the set\=/group $\mor{C}( G, G)$. In particular, $\inversefn$ is the unique inverse element of $\id_G$ in $\mor{C}( G, G)$.
\end{proof}

\section{Coarse Groups: Definition, Notation and Examples}
\label{sec:p1:coarse groups_defnotexmp}

We have shown that the category of coarse spaces has binary products. Any bounded coarse space $(X,\maxcrs)$ is a terminal object in \Cat{Coarse}. For convenience, in the following we will use the singleton $\tobj$ as a preferred choice of terminal object. Since \Cat{Coarse} has products and terminal objects, we may study its group objects.

\begin{de}\label{def:p1:coarse.group}
 A \emph{coarse group}\index{coarse group} is a group object in the category \Cat{Coarse} (see Theorem~\ref{thm:p1:concrete coarse groups} for a concrete description).
\end{de}

Following our convention to use bold characters for coarse spaces and maps, a coarse group is a quadruple $(\crse {G,\cop},\cunit,\cinversefn)$. In view of Corollary~\ref{cor:p1:uniqueness.of.unit&inverse}, we can lighten the notation and denote coarse groups as $(\crse{G},\cop)$. More often than not, we will further simplify the notation by omitting $\cop$ and simply write $\crse G$. 
On the contrary, if we need to make distinctions we will further decorate the symbols and write \emph{e.g.}\ $\crse{\cop_{ G},\ \cunit_{G},\ \cinversefn_{G}}$. 
We call $\cop$\nomenclature[:CO]{$\cop$, $\cop_{\crse G}$}{coarse group operation/multiplication} the \emph{coarse group operation/multiplication}, $\cunit$\nomenclature[:CO]{$\cunit$, $\cunit_{\crse G}$}{unit in a coarse group} the \emph{coarse unit}\index{coarse!unit} and $\cinversefn$\nomenclature[:CO]{$\cinversefn$, $\cinversefn_{\crse G}$}{inversion map in a coarse group} the \emph{inversion coarse map}\index{coarse!inversion map}. 

We will stay true to Convention~\ref{conv:p1:bold for coarse}, so that $G$ will be the set underlying the coarse space $\crse G$ and $\ast,e,\inversefn$ will be representatives for the coarse maps. For any pair of points $x,y\in G$ we will write $x\ast y$ instead of $\ast(x,y)$. There will also be some instances where the choice of representatives is important (\emph{e.g.}\ in Section~\ref{sec:p1:making sets into groups}). In this case we may make them explicit using the notation $(G,\CE,[\ast],[\unit],[\inversefn])$.

\begin{rmk} 
\label{rmk:p1:coarse unit as a coarse point}
 Since we are using the singleton $\tobj$ as terminal object, a representative $e$ for the unit in $\crse G$ is just a choice of a point of $G$. According to Definition~\ref{def:p1:coarse subspace}, we may also denote the coarse unit by $\cunit\crse{\in G}$.
 Notice that the coarse unit of a coarse group is simply a choice of coarsely connected component of $\crse G$ (Remark~\ref{rmk:p1:coarse points are connected components}), however it is often useful to think of it as a coarse point.
\end{rmk}

\begin{rmk}
\label{rmk:p1:connected components are a group}
 Further elaborating on Remark~\ref{rmk:p1:coarse unit as a coarse point}, the set coarse points of $\crse G$ is the set of its coarsely connected components. Identifying coarse points with coarse maps $\tobj\to\crse{G}$ we see that the set of coarsely connected components of $\crse G$ is $\cmor(\tobj,\crse G)$, which is a group by Lemma~\ref{lem:p1:mor.is.group}. This will imply that every coarse group can be decomposed into a set\=/group and a coarsely connected coarse group (see Corollary~\ref{cor:p1:ses quotient by id comp}). 
 
 It is also interesting to apply Lemma~\ref{lem:p1:mor.is.group} to trivially coarse spaces $(I,\mincrs)$.
 If $I$ is a finite set then the set\=/group $\cmor\paren{(I,\mincrs),\crse{G}}$ can be naturally identified with the direct power $\cmor\paren{\tobj,\crse{G}}$. If $I$ is infinite the situation is a little more delicate.
\end{rmk}

If ${f}_1$ and ${f}_2$ are representatives for the coarse maps $\crse f_1$ and $\crse f_2$ then the product ${f}_1\times{f}_2$ is a representative for $\crse f_1\times \crse f_2$. The composition with the diagonal embedding $({f}_1\times {f}_2)\circ \Delta$ is a representative for $(\crse{f}_1,\crse{f}_2)$.
It is helpful to spell out explicitly what it means that the Group Diagrams of Definition~\ref{def:p1:group.object} commute. Once we choose representatives for the coarse maps, a diagram commutes in \Cat{Coarse} if the relevant functions commute up to closeness. Using the notation of Convention~\ref{conv:p1:x_torel_CE}, the following is immediate:

\begin{lem}\label{lem:p1:hearts.diagrams on sets}
 Let $(G,\CE)$ be a coarse space and $\ast\colon G\times G\to G$, $\unit\colon\tobj\to G$ and $\inversefn\colon G\to G$ be fixed functions (possibly not controlled). The relevant Group Diagrams commute up to closeness if and only if
 \[\def\arraystretch{1.5}
 \begin{array}{ccc}
  {\rm (Associativity)}\qquad &
  g_1 \ast (g_2\ast g_3) \rel{\CE} (g_1\ast g_2)\ast g_3 
  & \qquad \forall g_1,g_2,g_3\in G
   \\
   {\rm (Identity)}\qquad   &
   e\ast g \rel{\CE} g\rel{\CE} g\ast e 
   & \qquad \forall g\in G
  \\
  {\rm (Inverse)}\qquad &
  g\ast g^{-1}\rel{\CE} e\rel{\CE} g^{-1}\ast g 
  & \qquad \forall g\in G.
 \end{array}
 \]
\end{lem}

If $G$ is a set\=/group with group operations $\ast,e,\inversefn$, the Group Diagrams commute. Therefore, if $\CE$ is a coarse structure on $G$, then $(G,\CE,[\ast],[e],[\inversefn])$ is a coarse group if and only if $\ast$ and $\inversefn$ are controlled with respect to $\CE$. When this is the case, we say that $\crse{G}$ is a \emph{coarsified set\=/group}\index{coarsified set\=/group}\index{set\=/group!coarsified}\index{coarsification}. Coarsified set\=/groups are the simplest examples of coarse groups, we study them in Chapter~\ref{ch:p2:coarse structures on set groups}.

\begin{exmp}
If $G$ is a set\=/group, then $G$ endowed with the trivial coarse structure $\mincrs$ is a coarse group.
We refer to it as \emph{trivially coarse group}\index{trivially coarse!group}\index{coarse group!trivially coarse}. Vice versa, if $(G,\mincrs)$ is a coarse group then $G$ must be a set\=/group. By themselves, these coarse groups are not particularly interesting examples of coarse groups. However, this does show that there is a correspondence between set\=/groups and coarse groups whose coarsely connected components are singletons (in fact, this is a natural containment of categories $\Cat{SetGroup}\hookrightarrow\Cat{CrsGroup}$). Trivially coarse groups will also give interesting examples of coarse actions (Section~\ref{ch:p1:coarse actions}).
 
 Conversely, a bounded coarse space $G$ endowed with any choice of operations $(G,\maxcrs)$ is a coarse group, which we refer to as \emph{bounded coarse group}\index{coarse group!bounded}. Bounded coarse groups are the coarse analog of the trivial group $\{e\}$. In fact, $\{e\}$ is the only bounded trivially coarse group, and a coarse group is bounded if and only if it is isomorphic to it (as a coarse group, see Section~\ref{sec:p1:coarse homomorphisms_1}).
\end{exmp}

Trivial and bounded coarsifications\index{coarsification!trivial}\index{coarsification!bounded} of set\=/groups are unsurprising examples of coarse groups. Much more interesting examples are given by unbounded connected coarsifications.
The most natural procedure for producing such examples is by considering metric coarsifications of set\=/groups. The following lemma is a simple exercise (a more general statement is proved in the next section, see also Appendix~\ref{sec:appendix:metric groups}).

\begin{lem}\label{lem:p1:biinvariant metrics give coarse groups}
 If $G$ is a set\=/group and $d$ is a bi\=/invariant metric on $G$, then $(G,\CE_d)$ is a coarse group.
\end{lem}

\begin{exmp}\label{exmp:p1:abelian coarse groups_metric}
 If $G$ is an abelian coarse group, every left invariant metric $d$ is bi\=/invariant and hence $(G,\CE_d)$ is a coarse group. It is then easy to obtain the first examples of unbounded connected coarse groups. For instance, $(\ZZ^n,\CE_{\norm{\mhyphen}})$ and $(\RR^n,\CE_{\norm{\mhyphen}})$, where $\norm{\mhyphen}$ is the Euclidean norm. If $(V,\norm{\mhyphen})$ is any normed vector space, then $(V,\CE_{\norm{\mhyphen}})$ is a coarse group. Another interesting example is given by the $p$\=/adic numbers equipped with the $p$\=/adic absolute value $(\QQ_p,\CE_{\abs{\mhyphen}_p})$.
\end{exmp}

 Refining Example~\ref{exmp:p1:abelian coarse groups_metric}, we can also define bi\=/invariant metrics on non\=/abelian groups. 
 Recall that to define a left invariant metric on a set\=/group $G$ it is enough to specify a \emph{length function} $\abs{\variable}\colon G\to \RR_{\geq 0}$ such that:
 \begin{itemize}
  \item $\abs{g}=0$ if and only if $g=e$;
  \item $\abs{g^{-1}}=\abs{g}$ for all $g\in G$;
  \item $\abs{hg}\leq \abs{g}+\abs{h}$ for all $g,h\in G$.
 \end{itemize}
 To any length function is associated a left invariant metric defined by $d(g,h)\coloneqq\abs{g^{-1}h}$.  
 Note that any choice of generating set $S$ for $G$ defines a \emph{word length}\index{length!word}
 \[
  \abs{g}_S\coloneqq\min\braces{n\mid g=s_1^\pm\cdots s_n^\pm,\ s_i\in S}
 \]
and hence a left invariant \emph{word distance}\index{word distance} $d_S$ on $G$. One case of special interest is when the group $G$ is finitely generated, as the word metrics associated with any finite generating set are all coarsely equivalent.
 
 To obtain bi\=/invariant metrics it is enough to notice that the left invariant metric induced by a length function $\abs{\variable}$ is bi\=/invariant if and only if  $\abs{\variable}$ is invariant under conjugation. 
 One way to obtain conjugation invariant length functions is to consider the word length associated with some (infinite) conjugation\=/invariant generating set.
 We will further explore these \emph{bi\=/invariant word metrics}\index{bi-invariant word metric} in Chapter~\ref{ch:p2:canc metrics}, for now we will only focus on finitely normally generated groups.
 
 A set $S\subseteq G$ \emph{normally generates} the set\=/group $G$\index{normally generates} if the normal closure $\overline{S}\coloneqq\{gsg^{-1}\mid s\in S,\ g\in G\}$ is a generating set for $G$. We say that $G$ is \emph{finitely normally generated} if it is normally generated by some finite set $S\subseteq G$. It is an exercise to show that if $G$ is a finitely normally generated set\=/group then different choices of finite normally generating sets $S\subseteq G$ yield coarsely equivalent bi\=/invariant word metrics $\abs{\variable}_{\overline{S}}$\nomenclature[:MET]{$\abs{\variable}_{\overline{S}}$}{conjugation invariant length associated with the normally generating set $S$}\nomenclature[:MET]{$d_{\overline{S}}$}{bi\=/invariant word metric associated with the normally generating set $S$} (see Corollary~\ref{cor:p1:canonical crse_str is E_grp_fin} for a more conceptual proof). In other words, this means that such a $G$ has a canonical metric coarsification:

\begin{de}\label{def:p1:canonical coarse structure}
 If $G$ is a finitely normally generated set\=/group its \emph{canonical coarsification}\index{coarsification!canonical} is given by the metric coarse structure $\varcrs{bw}$\nomenclature[:CE1]{$\varcrs{bw}$}{canonical coarse structure associated with bi\=/invariant word metrics on finitely normally generated set\=/groups} defined by the bi\=/invariant word metric associated with any finite normally generating set. 
\end{de}
 
\begin{example}\label{exmp:p1:cancellation metric}
 Let $F_2=\angles{a,b}$ be the free group on two generators. The bi\=/invariant word length $\abs{\variable}_{\overline{S}}$ associated with the generating set $S=\{a,b\}$ admits a particularly nice description. Namely, the length of a reduced word $w\in F_2$ is equal to the minimal number of letters that is necessary to remove (anywhere in the word) so that all the remaining letters cancel out. For instance, $\abs{a^nba^{-n}}_{\overline{S}}=1$ because everything cancels out when we remove the letter $b$. For this reason, we also call $\abs{\variable}_{\overline{S}}$ the \emph{cancellation length}\index{length!cancellation} (denoted $\abs{\variable}_\times$)\nomenclature[:MET]{$\abs{\variable}_\times$}{cancellation length} and the induced metric is the  \emph{cancellation metric $d_\times$}\index{cancellation metric}\nomenclature[:MET]{$d_\times$}{cancellation metric} on $F_2$.  In particular, $\varcrs{bw}=\CE_{d_\times}$ on $F_2$.  
 
 One interesting feature of the coarse group $(F_2,\CE_{d_\times})$ is that it is not \emph{coarsely abelian}. That is, the controlled maps $(g,h)\mapsto g\ast h$ and $(g,h)\mapsto h\ast g$ are not close as functions $(F_2\times F_2,\CE_{d_\times}\otimes\CE_{d_\times})\to(F_2,\CE_{d_\times})$. 
 To prove this it is enough to that show that $d([a^n, b^n], e)=2n$.  
 By removing every occurrence of $a$ and $\bar a$ from the commutator $[a^n, b^n]$, we get the word $b^n \bar b^n$ which reduces to the empty word.  So $d([a^n, b^n], e)\leq 2n$.  To see $d([a^n, b^n], e)\geq 2n$ one needs to show that this path is optimal, and it is not hard to check it by hand (see also Section~\ref{sec:p2:cancellation metric on free}). 
\end{example}

\begin{rmk}
 In some sense, every bi\=/invariant word metric associated with a generating set $S$ of a set\=/group $G$ can be seen as a cancellation metric. In fact, for every word $w$ in the alphabet $S$ we can say that its cancellation length is equal to the number of letters that is necessary to cancel in order to get a word that represents the trivial element $e\in G$. Then it is easy to show that the bi\=/invariant word length $\abs{g}_{\overline S}$ equals the smallest cancellation length of a word $w$ representing $g\in G$ \cite[Proposition 2.A]{BGKM}. 
\end{rmk}

We give more detailed examples of groups with bi\=/invariant metrics and cancellation metrics in Chapters~\ref{ch:p2:coarse structures on set groups} and \ref{ch:p2:a cgroup not group}. These are a useful source of examples and will play a role in various parts of this monograph. Examples of coarse groups that do not arise as coarsifications of set\=/groups will be given in the sequel, after some more theory is be developed.

\section{Equi Invariant Coarse Structures and Automatic Control}
\label{sec:p1:equi invariant and automatic control}

Given a coarse space $\crse G$ and functions $\ast$, $\inversefn$, it is important to know whether the functions are controlled. In principle this can be rather awkward to check. However, rephrasing the problem in terms of equi control can often help.
Specifically, given a multiplication function $\ast\colon G\times G\to G$ we can see it as a family of \emph{left multiplication} functions ${}_g\ast=\ast(g,\variable)\colon G\to G$ and as a family of \emph{right multiplication} functions $\ast_g=\ast(\variable,g)\colon G\to G$. Then Lemma~\ref{lem:p1:equi controlled.sections.iff.controlled} implies that $\ast\colon (G,\CE)\times (G,\CE)\to (G,\CE)$ is controlled if and only if both the families of left multiplications ${}_g\ast$ and right multiplications $\ast_g$ are equi controlled. 

In other words, $\ast$ is controlled if and only if for every $E\in\CE$ we have
\begin{equation}\label{eq:p1:controlled multiplication}
 \substack{
  \displaystyle g\ast h \rel{\CE} g\ast h' \qquad \forall h\torel{E}h',\ \forall g\in G \hphantom{.}\\
  \displaystyle \vphantom{\int} h\ast g \rel{\CE} h'\ast g \qquad \forall h\torel{E}h',\ \forall g\in G.
  }
\end{equation}

\begin{exmp}
 If $\CE=\CE_d$ is a metric coarse structure, then the condition \eqref{eq:p1:controlled multiplication} is equivalent to saying that for every $r\geq 0$ there exists $R\geq 0$ so that
 \begin{align*}
  & d(g\ast h, g\ast h')\leq R  \qquad \forall h,h'\text{ with }d(h,h')\leq r,\ \forall g\in G \\
  & d(h\ast g, h'\ast g)\leq R  \qquad \forall h,h'\text{ with }d(h,h')\leq r,\ \forall g\in G.
\end{align*}
 It is then clear that if $G$ is a set\=/group and $d$ is a bi\=/invariant metric on it then the multiplication is controlled.
\end{exmp}

It is useful to rephrase the above conditions in terms of controlled entourages. Given two relations $E,F$ on $G$, we let\nomenclature[:R]{$E\ast F$}{multiplication of relations}
\begin{equation}\label{eq:p1:product of relations}
 E\ast F\coloneqq\bigbraces{\paren{e_1\ast f_1,e_2\ast f_2}\bigmid (e_1,e_2)\in E,\ (f_1,f_2)\in F)}
 =(\ast\times\ast)(E\otimes F).
\end{equation}
It follows by the definition that $\ast\colon (G,\CE)\times (G,\CE)\to (G,\CE)$ is controlled if and only if $E\ast F\in\CE$ for every $E,F\in\CE$. In particular, the left multiplications ${}_g\ast\colon(G,\CE)\to(G,\CE)$ are equi controlled if and only if $\Delta_{G}\ast E\in\CE$ for every $E\in\CE$.

Note also that $\Delta_{G}\ast E$ is the relation obtained by taking the union of the ``translations of $E$ under $G$'':
\[
 \Delta_{G}\ast E=\bigcup_{g\in G} (g,g)\ast E=\bigcup_{g\in G}\braces{\paren{g\ast e_1 ,g\ast e_2}\bigmid (e_1,e_2)\in E}
\]
(See Figure~\ref{fig:p1:equi invariant}). This suggests the following nomenclature:
\begin{de}\label{def:p1:equi.bi.invariant coarse structure}
 Given a set $G$ and a multiplication function $\ast\colon G\times G\to G$, a coarse structure $\CE$ on $G$ is \emph{equi left invariant}\index{equi!left invariant} if $\Delta_{G}\ast E\in\CE$ for every $E\in \CE$. Equivalently, $\CE$ is equi left invariant if and only if the left multiplications ${}_g\ast\colon(G,\CE)\to(G,\CE)$ are equi controlled. \emph{Equi right invariance}\index{equi!right invariant} is defined analogously, and we say that $\CE$ is \emph{equi bi\=/invariant}\index{equi!bi\=/invariant} if it is both equi left invariant and equi right invariant.
\end{de}

\begin{figure}
 \centering
 \includegraphics[scale=0.8]{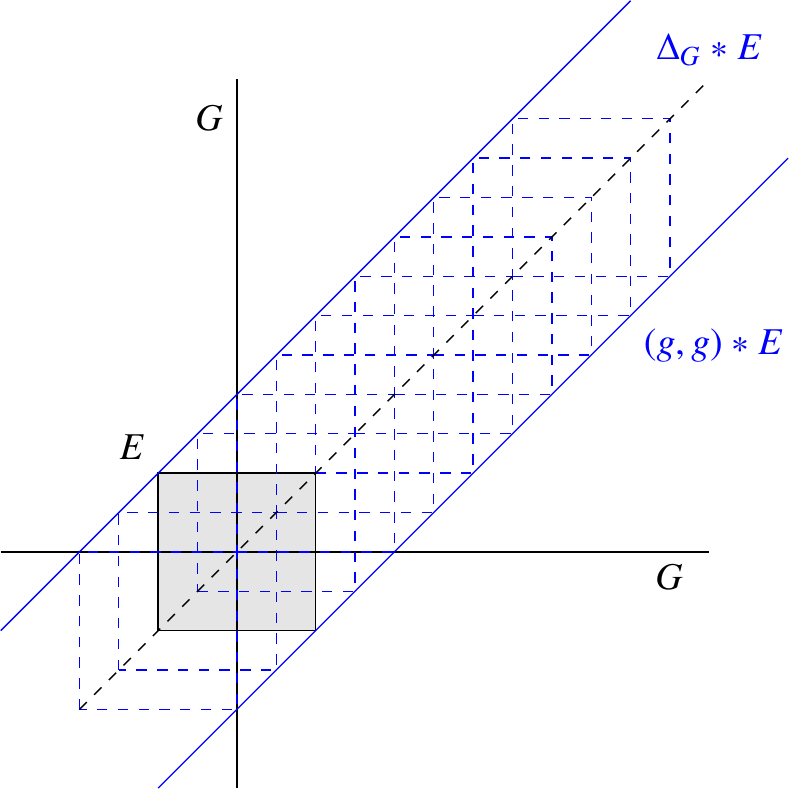}
 \caption{The relation $\Delta_G\ast E$ is the union of the (left) translates of $E$.}
 \label{fig:p1:equi invariant}
\end{figure}

Using this terminology, $\ast\colon (G,\CE)\times(G,\CE)\to (G,\CE)$ is controlled if and only if $\CE$ is equi bi\=/invariant.

\begin{rmk}
 For any relation $E$ on $G$, if $y\torel{E} y'$ then
 \[
  x\ast y\torel{\Delta_G\ast E} x\ast y' \quad \text{ and }\quad y\ast x\torel{E\ast \Delta_G} y'\ast x
 \]
 for every $x\in G$. This simple observation will be used often in the sequel.
\end{rmk}

\begin{rmk}
 In terms of partial coverings, the left multiplications ${}_g\ast$ are equi controlled with respect to $\CE$ if and only if the covering 
 \[
  \pts{G}\ast\fka = \bigcup \braces{g\ast\fka\mid g\in G} = \bigcup \braces{g\ast A \mid g\in G,\ A\in\fka}
 \]
 is controlled whenever $\fka\in\fkC(\CE)$. That is, $\CE$ is equi left invariant if and only if the coverings $g\ast\fka\in\fkC(\CE)$ are `uniformly' controlled as $g\in G$ varies. 
\end{rmk}

It is a pleasant feature of the coarse category that whenever the multiplication function is controlled then the inversion is automatically controlled as well. Namely, the following is true:

\begin{prop}\label{prop:p1:coarse.group.iff.heartsuit.and.equi controlled}
 Let $\crse G=(G,\CE)$ be a coarse space and $\ast,\unit,\inversefn$ be functions so that Group Diagrams commute up to closeness. If $\ast$ is controlled then $\inversefn$ is also controlled and hence $(\crse G,\cop,\cunit,\cinversefn)$ is a coarse group. 
\end{prop}
\begin{proof}
 Fix a controlled entourage $E\in \CE$. Since $_{g}\ast$ are equi controlled,
 \[
  y^{-1}\ast x \rel{\CE} y^{-1}\ast y \qquad  \forall y\torel{E}x.
 \]
 Since $\ast_{g}$ are equi controlled and the diagrams (Identity) and (Inverse) commute up to closeness, we also have that
 \[
   (y^{-1}\ast x)\ast x^{-1} \rel{\CE} (y^{-1}\ast y)\ast x^{-1}  
   \rel{\CE} e\ast x^{-1} \rel{\CE} x^{-1}
   \qquad  \forall y\torel{E}x.  
 \]
 On the other hand, we know from (Associativity) that
 \[
  (y^{-1}\ast x)\ast x^{-1} \rel{\CE}y^{-1}\ast (x\ast x^{-1})  \qquad  \forall y,x\in G.  
 \]
 Using again  (Identity), (Inverse) and equi control of $_{g}\ast$, we have
 \[
 y^{-1}\ast (x\ast x^{-1})  
 \rel{\CE} (y^{-1}\ast e) \rel{\CE} y^{-1}
 \qquad  \forall y\torel{E}x. 
 \]
 Putting everything together we see that
 \[
  x^{-1}\rel{\CE} y^{-1} \qquad  \forall y\torel{E}x,
 \]
 hence $\inversefn$ is controlled.
\end{proof}

\begin{cor}%[{\cite[Proposition 1.15]{dikranjan2020categories} and \cite[Proposition 1.2]{hernandez2011balleans}}]
\label{cor:p1:coarse.group.from.group.iff.equi controlled}
 If $G$ is a set\=/group and $\CE$ is a coarse structure on $G$, then $\crse G$ is a coarse group if and only if $\CE$ is equi bi\=/invariant.
\end{cor}

\begin{rmk}
Proposition~\ref{prop:p1:coarse.group.iff.heartsuit.and.equi controlled} means that if we are given a coarse multiplication $\cop$ and a function $\inversefn$ that is a `candidate inversion' in \Cat{Coarse} then $\inversefn$ is automatically controlled and hence $\cinversefn$ is an actual coarse inversion map.
This is not true in other categories. For example, if $G$ is a set-group equipped with a topology that makes $\ast$ continuous, it is not necessarily the case that $\inversefn$ is continuous (and hence $G$ need not be a topological group).\footnote{%
Set\=/groups equipped with topologies so that the multiplication is continuous are called \emph{paratopological groups}. Paratopological groups need not be topological groups. One such example is the group of real numbers equipped with the Sorgenfrey topology. Other examples can be taken by considering groups of self\=/homeomorphisms equipped with the compact\=/open topology (one such example is given in \cite{dijkstra2005homeomorphism}).}
\end{rmk}

\section{Making Sets into Coarse Groups}\label{sec:p1:making sets into groups}
This is a somewhat technical section that will be useful later. 
It is helpful to reverse our point of view: we first fix some  \emph{set} $G$ (not necessarily a group), functions $\ast\colon G\times G\to G$ and $\inversefn \colon G\to G$ and an element $\unit\in G$. Our goal 
is to characterize the coarse structures $\CE$ on $G$ such that $(G,\CE,[\ast],[\unit],[\inversefn])$ is a coarse group. That is, to find coarse structures $\CE$ so that the multiplication and inversion are controlled functions and the Group Diagrams commute.

Let $\paren{\CE_i}_{i\in I}$ and $\paren{\CF_i}_{i\in I}$ be coarse structures on sets $X$ and $Y$, and $f\colon X\to Y$ be a function such that $f\colon(X,\CE_i)\to(Y,\CF_i)$ is controlled when for each $i\in I$, then $f$ is also controlled when $X$ and $Y$ are equipped with $\bigcap_{i\in I}\CE_i$ and $\bigcap_{i\in I}\CF_i$ respectively.
Moreover, notice that $(\bigcap_{i\in I}\CE_i)\otimes(\bigcap_{i\in I}\CE_i)= \bigcap_{i\in I}(\CE_i\otimes\CE_i)$ as coarse structures on $X\times X$. From these observations it follows that whenever $\paren{\CE_i}_{i\in I}$ are coarse structures on $G$ that make $(G,\ast,\unit,\inversefn)$ into a coarse group, then so does their intersection $\bigcap_{i\in I}\CE_i$ (commutativity of the Group Diagrams is clearly preserved under taking intersections). 
It is also clear that letting $\CE=\CE_{\rm max}$ does make $G$ into a coarse group. We can thus give the following:

\begin{de}
 Given a set $G$, functions $\ast\colon G\times G\to G$, $\inversefn \colon G\to G$ and an element $\unit\in G$ we let $\mingrp$ be the minimal coarse structure on $G$ such that $(G,\mingrp,[\ast],[\unit],[\inversefn])$ is a coarse group. Equivalently, $\mingrp$\nomenclature[:CE1]{$\mingrp$}{minimal group coarse structure}\index{coarse structure!minimal group} is the intersection of all the coarse structures that make $G$ into a coarse group.
\end{de}

\begin{exmp}
 Notice that $\mingrp=\mincrs$ if and only if $G$ is a set\=/group. 
\end{exmp}

By definition, a coarse structure $\CE$ that makes $(G,\ast,\unit,\inversefn)$ into a coarse group must contain $\mingrp$. Vice versa, if $\CE$ is a coarse
 structure containing $\mingrp$ then $\ast$, $\inversefn$ and $e$ make the Group Diagrams commute for $(G,\CE)$. It then follows that such an $\CE$ makes $G$ into a coarse group if and only if $\ast$ is controlled. In other words, we can rephrase this as an extension of Corollary~\ref{cor:p1:coarse.group.from.group.iff.equi controlled}:
 
 \begin{cor}
  A coarse structure $\CE$ makes $(G,\ast,\unit,\inversefn)$ into a coarse group if and only if it contains $\mingrp$ and it is equi bi\=/invariant.
 \end{cor}
 
It is possible to describe $\mingrp$ fairly explicitly as the coarse structure generated by some set of relations. 
Define relations on $G$ as follows:
\begin{align*}
 \assRel' &\coloneqq \bigbraces{
 \paren{g_1 \ast (g_2\ast g_3), (g_1\ast g_2)\ast g_3}\bigmid g_1,g_2,g_3\in G} \\
 \invRel' &\coloneqq 
 \bigbraces{\bigparen{g\ast g^{-1}, \unit}\bigmid g\in G}  \cup
 \bigbraces{\paren{g^{-1}\ast g, \unit}\bigmid g\in G}  \\
 \idRel' &\coloneqq 
 \bigbraces{\paren{g\ast\unit, g}\bigmid g\in G} \cup
 \bigbraces{\paren{\unit\ast g, g}\bigmid g\in G} 
\end{align*}
and let $\assRel\coloneqq\assRel'\cup\op{(\assRel')}$, $\invRel\coloneqq\invRel'\cup\op{(\invRel')}$ and $\idRel\coloneqq\idRel'\cup\op{(\idRel')}$\nomenclature[:R]{$ \idRel$}{associativity relation}\nomenclature[:R]{$ \invRel$}{inversion relation}\nomenclature[:R]{$ \idRel$}{identity relation}.
By definition, these are the relations under which the Group Diagrams commute up to closeness.

\begin{rmk}
 To show that $\assRel$ (resp. $\invRel$, $\idRel$) belongs to a coarse structure $\CE$, it is enough to show to $\assRel'$ (resp. $\invRel'$, $\idRel'$) is in $\CE$. 
 Also recall that to prove that a relation $E$ belongs to a coarse structure $\CE$ it is enough to show that there exist $E_{0},\ldots, E_{n}\in\CE$ such that for every $x\torel{E}y$ there are $z_1,\ldots, z_n\in G$ with
 \[
  x \torel{E_{0}} z_1
  \torel{E_{1}}\cdots 
  \torel{E_{{n-1}}}z_n
  \torel{E_{n}}y.
 \]
\end{rmk}

Recall that the product of two relations is defined as $E\ast F=\braces{(e_1\ast f_1,e_2\ast f_2)\mid (e_1,e_2)\in E,\ (f_1,f_2)\in F}$. We can now explicitly describe $\mingrp$:

\begin{lem}\label{lem:p1:generators of mingrp}
Given a set with multiplication, unit and inversion functions $(G,\ast,\unit,\inversefn)$, the minimal coarse structure making it into a coarse group is
 \[ \mingrp=\angles{{\,\Delta_{G}\ast \assRel\, ,\; (\Delta_{G}\ast \invRel)\ast\Delta
 _{G} \, ,\;
 \Delta_{G}\ast \idRel\, ,\; \idRel \, }}.
 \]
\end{lem}
\begin{proof}
 Let $\CE\coloneqq \angles{{\,\Delta_{G}\ast \assRel\, ,\; (\Delta_{G}\ast \invRel)\ast\Delta
 _{G}\, ,\;
 \Delta_{G}\ast \idRel\, ,\; \idRel \, }}$. It is clear that $\CE\subseteq \mingrp$ because $\assRel$, $\invRel$ and $\idRel$ are contained in $\mingrp$, and the relations obtained by multiplying them by $\Delta_{G}$  must be in $\mingrp$ because it is equi bi\=/invariant.
  
 It remains to show the other containment. That is, we need to show that $\CE$ makes $G$ into a coarse group. Note that $\assRel\in \CE$ because
  \[
   g_1\ast(g_2\ast g_3)
   \rel{\idRel} {\unit\ast (g_1 \ast(g_2\ast g_3))}
   \rel{\Delta_{G}\ast \assRel} {\unit\ast ((g_1\ast g_2)\ast g_3)}
   \rel{\idRel} (g_1\ast g_2)\ast g_3.
  \]
 An analogous argument shows that $\invRel\in\CE$ as well. 
 
 By Proposition~\ref{prop:p1:coarse.group.iff.heartsuit.and.equi controlled}, all it remains to prove is that $\ast\colon (G,\CE)\times (G,\CE)\to (G,\CE)$ is controlled. In order to do so, we need to show that $\Delta_{G}\ast F$ and $F\ast \Delta_{G}$ belong to $\CE$ whenever $F$ is one of $\Delta_{G}\ast \assRel$, $(\Delta_{G}\ast \invRel)\ast\Delta_G$, $\Delta_{G}\ast \idRel$ or $\idRel$ (see Remark~\ref{rmk:p1:control of maps from product generated}). We will prove in detail that $\Delta_{G}\ast (\Delta_{G}\ast \assRel)\in\CE$ and $(\Delta_{G}\ast \assRel )\ast \Delta_{G}\in\CE$; the other cases are similar (and simpler). The former is easy to prove, because
 \begin{align*}
  h_1\ast \bigparen{h_2\ast (g_1 \ast (g_2\ast g_3))}
  &\rel{\assRel} (h_1\ast h_2)\ast (g_1\ast (g_2\ast g_3)) \\
  &\rel{\Delta_{G}\ast \assRel}
  (h_1\ast h_2)\ast ((g_1\ast g_2)\ast g_3) \\
  &\rel{\assRel} h_1\ast \bigparen{h_2\ast ((g_1\ast g_2)\ast g_3)}
 \end{align*}
 for every $h_1,h_2,g_1,g_2,g_3\in G$. The proof of the latter is slightly more involved: 
 \begin{align*}
  \bigparen{h_2\ast  (g_1\ast (g_2\ast g_3))}\ast h_1
  &\rel{\assRel} 
  h_2\ast \bigparen{(g_1\ast (g_2\ast g_3))\ast h_1} \\
  &\rel{\Delta_{G}\ast \assRel}
  h_2\ast \bigparen{g_1\ast ((g_2\ast g_3)\ast h_1)} \\
  &\rel{\Delta_{G}\ast(\Delta_{G}\ast \assRel)}
  h_2\ast \bigparen{g_1\ast (g_2\ast (g_3\ast h_1))} \\
  &\rel{\Delta_{G}\ast \assRel}
  h_2\ast ((g_1\ast g_2)\ast (g_3\ast h_1)) \\
  &\rel{\Delta_{G}\ast \assRel}
  h_2\ast \bigparen{((g_1\ast g_2)\ast g_3)\ast h_1} \\
  &\rel{\assRel}
  \bigparen{h_2\ast ((g_1\ast g_2)\ast g_3)}\ast h_1.
 \end{align*}
 Note that on the third line it was essential that we already knew $\Delta_{G}\ast(\Delta_{G}\ast \assRel)\in\CE$.
 Analogous computations for $\invRel$ and $\idRel$ shows that $\ast$ is controlled and thus concludes the proof.
\end{proof}

 Having studied $\mingrp$, the next step is to study the minimal coarse structure that contains a given set of relations.  
 More precisely, let $\CR\subset\CP(G\times G)$ be any set of relations. We define $\varcrs[grp]{\CR}$ to be the minimal coarse structure that contains $\CR$ and under which $(G,\varcrs[grp]{\CR},[\ast],[\unit],[\inversefn])$ is a coarse group:
 \[
  \varcrs[grp]{\CR}=\bigcap\Bigbraces{\CE\bigmid \CR\subseteq \CE,\ (G,\CE,[\ast],[\unit],[\inversefn])\text{ is a coarse group}}
  \nomenclature[:CE1]{$\varcrs[grp]{\CR}$ }{minimal group coarse structure containing $\CR$}
 \]
 (notice that $\varcrs[grp]{\CR}$ depends on the choice of functions $\ast,\unit,\inversefn$). 
 It turns out that $\varcrs[grp]{\CR}$ is easy to describe in terms of $\CR$ and $\mingrp$. Let $E\ast \CR\coloneqq\braces{E\ast R\mid R\in \CR}$ and $\CR\ast E\coloneqq\braces{R\ast E\mid R\in\CR}$. We can prove the following:

\begin{lem}\label{lem:p1:generated biinvariant coarse structure}
 Let $\CR$ be any set of relations on the set $G$. Then 
 \[
  \varcrs[grp]{\CR}=\angles{(\Delta_{G}\ast \CR)\ast \Delta_{G}\, ,\; \mingrp}=\angles{\Delta_{G}\ast (\CR\ast\Delta_{G})\,,\; \mingrp}.
 \]
\end{lem}

\begin{proof}
 To begin with, it is easy to see that $\angles{(\Delta_{G}\ast \CR)\ast \Delta_{G}, \mingrp}=\angles{\Delta_{G}\ast (\CR\ast\Delta_{G}), \mingrp}$, because $\assRel\in\mingrp$ and for every relation $E\subseteq G\times G$ we have containments $(\Delta_{G}\ast E)\ast \Delta_{G}\subseteq {\assRel} \circ \bigparen{ \Delta_{G}\ast (E\ast \Delta_{G})}\circ \assRel$ and 
 $\Delta_{G}\ast(E\ast \Delta_{G})\subseteq \assRel\circ \bigparen{(\Delta_{G}\ast E)\ast \Delta_{G}}\circ {\assRel}$. 
 We proceed as in the proof of Lemma~\ref{lem:p1:generators of mingrp}: let $\CE\coloneqq \angles{\Delta_{G}\ast (\CR\ast \Delta_{G}), \mingrp}$. It then follows that $\varcrs[grp]{\CR}\subseteq \CE$ and---using Proposition~\ref{prop:p1:coarse.group.iff.heartsuit.and.equi controlled}---it is enough to show that $\ast\colon(G,\CE)\times(G,\CE)\to(G,\CE)$ is controlled to conclude the proof of the lemma.
 
 We first show that the left multiplication ${}_g\ast$ is equi controlled, \emph{i.e.}\ that $\ast\colon(G,\mincrs)\times (G,\CE)\to(G,\CE)$ is controlled. We already know that $\Delta_{G}\ast E\in\mingrp\subseteq \varcrs[grp]{\CR}$ for every $E\in\mingrp$, it is hence enough to show that $\Delta_{G}\ast \paren{\Delta_{G}\ast (R\ast \Delta_{G})} \in\CE$ for every $R\in\CR$. This is readily done: if $h\torel{R}h'$ then 
 \begin{align*}
  g_1\ast(g_2\ast(h\ast g_3))
  &\rel{\assRel} (g_1\ast g_2)\ast(h\ast g_3) \\
  &\torel{\Delta_G\ast (R\ast \Delta_G)} (g_1\ast g_2)\ast(h'\ast g_3)\\
  &\rel{\assRel} g_1\ast(g_2\ast(h'\ast g_3))
 \end{align*}  

 To prove that the right multiplication $\ast_g$ is equi controlled it is enough to write $\CE\coloneqq \angles{(\Delta_{G}\ast \CR)\ast \Delta_{G}, \mingrp}$ and use the same argument. By Lemma~\ref{lem:p1:equi controlled.sections.iff.controlled} it follows that $\ast\colon(G,\CE)\times(G,\CE)\to(G,\CE)$ is controlled.
\end{proof}

\begin{rmk}\label{rmk:p1:generated leftinvariant coarse structure}
 Given a set of relations $\CR$ on $G$ we can also ask what is the smallest equi left invariant coarse structure $\varcrs[left]{\CR}$ containing $\CR$. If we already know that $\mingrp\subseteq\varcrs[left]{\CR}$ (\emph{e.g.}\ because $\mingrp\subseteq\CR$), then it is easy to check that 
 \[
  \varcrs[left]{\CR}=\angles{\Delta_{G}\ast\CR\,,\,\mingrp}.
 \]
 The key point is that $\Delta_{G}\ast(\Delta_{G} \ast E)\subseteq\assRel\circ \bigparen{(\Delta_{G}\ast\Delta_{G}) \ast E}\circ {\assRel}
 \subseteq\assRel\circ \bigparen{\Delta_{G} \ast E}\circ {\assRel}$ for every relation $E\subseteq G\times G$. See Proposition~\ref{prop:p1:quotient action generated by R}. \index{coarse structure!minimal equi left invariant} \nomenclature[:CE1]{$\varcrs[left]{\CR}$}{minimal equi left invariant coarse structure generated by $\CR$}
\end{rmk}

\begin{exmp}\label{exmp:p1:finite sets group coarse structure}
 Recall that the coarse structure of finite off\=/diagonal $\varcrs{fin}$ is the minimal coarsely connected coarse structure. If $G$ is a set\=/group, Lemma~\ref{lem:p1:generated biinvariant coarse structure} implies that $\varcrs[grp]{fin}=\angles{\Delta_{G}\ast E\ast\Delta_{G}\mid E\subseteq G\times G\text{ finite}}$ (the order of multiplication does not matter, as $G$ is a set\=/group). 
 By definition, $\varcrs[grp]{fin}$ is the minimal coarsely connected equi bi\=/invariant coarse structure on $G$. That is, $(G,\varcrs[grp]{fin})$ is the minimal coarsely connected group coarsification of $G$.

 In Section~\ref{sec:p1:determined locally}  we provide a simple proof that if $G$ is finitely normally generated then $\varcrs[grp]{fin}$ coincides with the canonical coarse structure $\varcrs{bw}$ (Definition~\ref{def:p1:canonical coarse structure}). 
 In particular, this proves that the canonical coarse structure $\varcrs{bw}$ is indeed independent of the choice of finite normally generating set. 
\end{exmp}

For Lemma~\ref{lem:p1:generators of mingrp} it is truly necessary to add $\idRel$ in the set of generators of $\mingrp$ (as opposed to keeping only $\Delta_{G}\ast\idRel$). This can be easily seen by considering a function $\ast$ that is not surjective. For the sake of concreteness, here is an example:

\begin{exmp}\label{exmp:p1:non.surjective multiplication}
 For example, let $(G,\ast)$ be a set\=/group. Consider the set $G''=G\times\{\pm 1\}$ and extend the functions $\ast$, $\inversefn$ as $(x,\pm 1)\ast(y,\pm 1)\coloneqq (x\ast y,+1)$ and $(x,\pm 1)^{-1}\coloneqq (x^{-1},+1)$. That is, $G''$ is obtained by ``doubling'' all the elements in $G$ and extending the multiplication and inverse functions in such a way that their image is always contained in $G\times\{+1\}$. It is not hard to show that $\mingrp$ is $\mincrs\otimes\CP(\{\pm1\})$. On the other hand, any relation of the form $E\ast F$ is contained in $(G\times\{+1\})\times (G\times\{+1\})$, this shows that $\mingrp$ cannot be generated using only ``product relations''. 
\end{exmp}

Example~\ref{exmp:p1:non.surjective multiplication} is admittedly artificial, but it does serve to point at some technical difficulties that can arise when studying general coarse groups. These may compound to make proofs very technical and complicated. For this reason it can be useful to assume some `decency' condition on the operation functions, under which they behave more like group operations.

\begin{de}\label{def:p1:well.adapted}
Let $G$ be a set.  A choice of functions $\ast\colon G\times G\to G$, $\inversefn \colon G\to G$ and $\unit\in G$ are \emph{adapted}\index{adapted operations} if $g\ast e=e\ast g=g$ and $g\ast g^{-1}=g^{-1}\ast g=e$ for every $g\in G$.
\end{de}

The assumption that functions are adapted can simplify technical statements. For example, Lemmas~\ref{lem:p1:generators of mingrp} and \ref{lem:p1:generated biinvariant coarse structure} can be rephrased as:

\begin{cor}\label{cor:p1:assrel for well_adapted}
 Let $G$ be a set. If $\ast,\unit$ and $ \inversefn$ are adapted operations on $G$ then $\mingrp=\angles{\Delta_{G}\ast\assRel}$. If $\CR$ is any family of relations on $G$ then 
 \[
 \varcrs[grp]{\CR}=\angles{\Delta_{G}\ast(\CR\ast \Delta_{G})\;,\;\Delta_{G}\ast\assRel}
 =\angles{(\Delta_{G}\ast\CR)\ast \Delta_{G}\;,\;\Delta_{G}\ast\assRel}
 \quad\text{ and }\quad
 \varcrs[left]{\CR}=\angles{\Delta_{G}\ast\CR\;,\;\Delta_{G}\ast\assRel}.
 \]
\end{cor}

It is easy to see that we can always assume that the operations are adapted if $G$ is a coarse group:

\begin{lem}\label{lem:p1:adapted representatives exist}
Every coarse group admits adapted representatives for the coarse operations.
\end{lem}
\begin{proof}
 Let $(G,\CE,[\ast],[e],[\inversefn])$ be a coarse group. In order to produce adapted representatives we will first modify the inversion function to ensure that $g^{-1} \neq e$ whenever $g \neq e$. Define a new inversion function $(\mhyphen)^{(-1)'}$ by letting
 \[
  g^{(-1)'}\coloneqq
  \left\{\begin{array}{ll}
          e & \text{if }g = e \\
          g^{-1} & \text{if }g\neq e \text{ and }g^{-1}\neq e \\
          g & \text{if }g\neq e \text{ and }g^{-1}= e.
         \end{array}
 \right.
 \]
 This function is close to $\inversefn$, and it is therefore another representative for $[\inversefn]$. Replacing $\inversefn$ with $(\mhyphen)^{(-1)'}$, we may thus assume that $g^{-1}=e$ if and only if $g=e$.
 Thus the following definition for the function $\ast'$ is well-posed:
 \[
  g\ast' h\coloneqq
  \left\{\begin{array}{ll}
          g & \text{if }h = e \\
          h & \text{if }g = e \\
          e & \text{if }g= h^{-1} \text{ or }h=g^{-1} \\
          g\ast h & \text{otherwise.}
         \end{array}
 \right.
 \]
 Once again, $\ast'$ is close to $\ast$. We may thus replace $\ast$ with $\ast'$, and the representatives $\ast,e,\inversefn$ are adapted   by construction.
\end{proof}

 For cardinality reasons, it is not always possible to assume that $\inversefn$ is an involution:
 
 \begin{example}\label{exmp:p1:no cancellative reps}
  Let $G\coloneqq A_0\sqcup A_1\sqcup A_3$ where the $A_i$ are three non\=/empty sets of differing cardinalities.  Take the coarse structure $\angles{A_0\times A_0,A_1\times A_1, A_2\times A_2}$, so that $G$ is coarsely equivalent to the set $\{0,1,2\}$ with the trivial coarse structure. Pick three points $a_0\in A_0$, $a_1\in A_1$, $a_2\in A_2$ and define
  \[
  g\ast h\coloneqq a_{i+j}
  \]
  where $g\in A_i$, $h\in A_j$ and the addition is taken mod(3). Similarly define 
  \[
   g^{-1}\coloneqq a_{-i}
  \]
  for $g\in A_i$, and let $e\coloneqq a_0$ (\emph{i.e.}\ we equip $G$ with coarse group operations making it isomorphic to $(\ZZ/3\ZZ,\mincrs)$). Since $\abs{A_1}\neq \abs{A_2}$ the inversion $\inversefn$ cannot be bijective, and hence it is not an involution. This example also shows that the multiplication functions ${}_g\ast$ and $\ast_g$ cannot be assumed to be bijections. 
 \end{example} 

\section{Coarse Groups are Determined Locally}
\label{sec:p1:determined locally}
The main result of this section is the coarse analogue of the observation that the topology of a topological group is determined by the set of open neighborhoods of the identity.  Interestingly, we will later see that this fact lies at the heart of the fundamental observation of geometric group theory (the Milnor--Schwarz  Lemma).

As in the previous section, let $(G,\ast,\unit,\inversefn)$ be a set with operations and $\mingrp$ be the minimal coarse structure making it into a coarse group. For any coarse structure $\CE$ on $G$ let $\fkU_e(\CE)$ be the set of \emph{bounded neighborhoods of the identity $e\in G$} \index{bounded!neighborhoods of the identity}\nomenclature[:BOR]{$ \fkU_e(\CE)$}{bounded neighborhoods of the identity}
\begin{equation*}
 \fkU_e(\CE)\coloneqq\braces{U \mid U\subseteq G\text{ bounded},\ e\in U}
\end{equation*}
For every $U\in\fkU_e(\CE)$ we have a relation $U\times U\in \CE$. We let $\CU_e(\CE)$ be the set of such relations:
\[
 \CU_e(\CE)\coloneqq\braces{U\times U\mid U\in\fkU(\CE)}.
\]

\begin{rmk}
 The word ``neighborhood'' is used quite loosely. If $B\subseteq G$ is any bounded set contained in the same coarsely connected component of $e$, then $U=B\cup\{e\}$ is a bounded neighborhood of identity. In particular, if $(G,\CE)$ is coarsely connected then \emph{every} bounded set is `morally' a bounded neighborhood of $e$. 
 Note also that the assumption $e\in U$ is made out of convenience, but it is somewhat unnatural within the coarse category.
\end{rmk}

Recall that a coarse structure $\CE$ on $G$ is equi left invariant if $\Delta_{G}\ast E\in\CE$ for every $E\in\CE$ (Definition~\ref{def:p1:equi.bi.invariant coarse structure}). The following result shows that equi left invariant coarse structures are determined by their bounded sets:

\begin{prop}\label{prop:p1:neighbourhoods.of.identity.generate}
 Given $(G,\ast,\unit,\inversefn)$ as above and an equi left invariant coarse structure $\CE$ such that $\mingrp\subseteq \CE$, then
 \[
  \CE=\angles{\Delta_{G}\ast\CU_e(\CE)\, ,\; \mingrp}.
 \]
In particular, $\CE$ is completely determined by its set of bounded neighborhoods of the identity.
\end{prop}

\begin{proof}
 Let $\CE'\coloneqq\angles{\Delta_{G}\ast\CU_e(\CE)\, ,\; \mingrp}$, it is clear that $\CE'\subseteq \CE$. We thus have to prove that $\CE\subseteq \CE'$.
 Let $E\in\CE$ be fixed. Note that if $g_1\torel{E} g_2$ then 
 \[
  (g_1^{-1}\ast g_2)
  \torel{\Delta_{G}\ast E} (g_1^{-1}\ast g_1) 
  \rel{\idRel} e.
 \]
 It follows that the set $U\coloneqq\braces{g_1^{-1}\ast g_2\mid (g_1,g_2)\in E}\cup\{e\}$ is a $\CE$\=/bounded neighborhood of $e\in G$ and hence $U\times U\in\CU_e(\CE)$. We can now write
 \[
  g_2 \rel{\idRel\cmp\assRel} g_1\ast(g_1^{-1}\ast g_2)
  \rel{\Delta_{G}\ast (U\times U)} g_1\ast e
  \rel{\idRel} g_1,
 \]
 and deduce that $E\in\CE'$.
\end{proof}

Obviously, an analogous statement holds if $\CE$ is equi right invariant. We thus have the following.

\begin{cor}\label{cor:p1:neighbourhoods.of.identity.generate}
 If $(G,\CE,[\ast],[\unit],[\inversefn])$ is a coarse group, then
 \[
  \CE=\angles{\Delta_{G}\ast\CU_e(\CE)\, ,\; \mingrp}=\angles{\CU_e(\CE)\ast \Delta_{G}\, ,\; \mingrp}.
 \]
\end{cor}

If $\crse G=(G,\CE)$ be a coarsified set\=/group, $\mingrp$ is trivial and therefore we see that $\CE=\angles{\Delta_{G}\ast\CU_e(\CE)}=\angles{\CU_e(\CE)\ast \Delta_{G}}$. In this case the proof of Proposition~\ref{prop:p1:neighbourhoods.of.identity.generate} gives a stronger result:

\begin{cor}\label{cor:p1:neighbourhoods.of.identity.generate_setgroup}
 Let $\crse G=(G,\CE)$ be a coarsified set\=/group. Then
 \[
  \CE=\braces{E\mid \exists U\in\fkU_e(\CE),\ E\subseteq\Delta_{G}\ast \paren{U\times U}}
  =\braces{E\mid \exists U\in\fkU_e(\CE),\ E\subseteq \paren{U\times U}\ast \Delta_{G}}
 \]
\end{cor}

\begin{cor}\label{cor:p1:neighbourhoods.of.identity.generate_setgroup_part cover}
 Let $\crse G=(G,\CE)$ be a coarsified set\=/group. Then a partial covering $\fka$ is controlled if and only if it is a refinement of $\pts{G}\ast U=\braces{gU\mid g\in G}$ for some $U\in \fkU_e(\CE)$.
\end{cor}
\begin{proof}
 One implication is clear, because $\pts{G}\ast U$ is a controlled covering. For the other direction, if $\fka$ is controlled then there is an $V\in\fkU_e(\CE)$ so that $\diag(\fka)\subseteq \Delta_G\ast (V\times V)$. As a consequence, for every $A\in\fka$ and $a_1,a_2\in A$ there are $v_1,v_2\in V$ so that $a_2^{-1}a_1=v_2^{-1}v_1$. It follows that $A\subseteq a_2V^{-1}V$ and we may thus let $U=V^{-1}V$.
\end{proof}

\begin{rmk} 
It is interesting to compare Corollary~\ref{cor:p1:neighbourhoods.of.identity.generate} with Lemma~\ref{lem:p1:generated biinvariant coarse structure}. Specifically, it is a curious fact that if we already know that $\CU_e(\CE)$ is the family of relations obtained from the bounded neighborhoods of $e$ in a coarse group $(G,\CE)$ then the coarse structure $\varcrs[grp]{\CU_e(\CE)}$ is generated by $\mingrp$ and $\Delta_{G}\ast\CU_e(\CE)$ (as opposed to having to take $(\Delta_{G}\ast\CU_e(\CE))\ast\Delta_{G} $). As we shall soon see, this is related with the fact that the family of neighborhoods $\fkU_e(\CE)$ is closed under conjugation. 
\end{rmk}

Proposition~\ref{prop:p1:neighbourhoods.of.identity.generate} and its corollaries imply that any equi left invariant coarse structure $\CE$ containing $\mingrp$ is completely determined by its bounded \emph{sets} (as opposed to entourages). This fact makes it much easier to check whether two equi left invariant coarse structures on $G$ are equal. 

\begin{example}\label{exmp:p1: word metric coincides with finleft}
 Let $G$ be a finitely generated set\=/group with finite generating set $S$, and let $d_S$ be the associated word metric (Example~\ref{exmp:p1:abelian coarse groups_metric}). Bounded sets in $G$ are sets of finite diameter and, since $S$ is finite, this means that the bounded sets are precisely the finite subsets of $G$. 
It then follows from Corollary~\ref{cor:p1:neighbourhoods.of.identity.generate_setgroup} that $\CE_{d_S}$ does not depend on the choice of the finite generating set, $S$. 
 Moreover, let $\varcrs[left]{fin}$ be the minimal left invariant coarse structure so that all the finite sets are $\varcrs[left]{fin}$\=/bounded (see Remark~\ref{rmk:p1:generated leftinvariant coarse structure}). Since all the $\CE_{d_S}$\=/bounded sets are finite, it follows from Corollary~\ref{cor:p1:neighbourhoods.of.identity.generate_setgroup} that
 $\CE_d\subseteq\varcrs[left]{fin}$ and since $\varcrs[left]{fin}$ is minimal, the inclusion is an equality.  
\end{example}

In view of the above, it is important to understand the families $\fkU_e(\CE)$ of bounded neighborhoods of the identity, where $\CE$ ranges among the equi left invariant (resp. equi bi\=/invariant) coarse structures on a given set $G$ with operations $\ast,\ e,\ \inversefn$. 
The requirement that the point $e$ belongs to every set in $\fkU_e(\CE)$ complicates matters a little. In what follows, it will be convenient to assume that the operations $\ast,\ e,\ \inversefn$ are adapted.

\begin{lem}\label{lem:p1:properties of neighborhoods of id}
 Let $\ast,e,\inversefn$ be adapted operations on a set $G$ and let $\CE$ be an equi left invariant coarse structure such that $\mingrp\subseteq\CE$. Then the family $\fkU_e(\CE)$ satisfies
\begin{itemize}
 \item [(U0)] $e\in U$ for every $U\in \fkU_e(\CE)$;
 \item [(U1)] if $U\in \fkU_e(\CE)$ and $e\in V\subset U$ then $ V\in \fkU_e(\CE)$;
 \item [(U2)] if $U\in \fkU_e(\CE)$ then $U^{-1}\in \fkU_e(\CE)$;
 \item [(U3)] if $U_1,U_2\in \fkU_e(\CE)$ then $U_1\ast U_2\in \fkU_e(\CE)$.
\end{itemize}
If $\CE$ is also equi bi\=/invariant (\emph{i.e.}\ $\crse G=(G,\CE)$ is a coarse group), then $\fkU_e(\CE)$ also satisfies
\begin{itemize}
 \item [(U4)] if $U\in \fkU_e(\CE)$ then $\bigcup\braces{g\ast(U\ast g^{-1})\mid g\in G}\in \fkU_e(\CE)$. 
\end{itemize}
\end{lem}

Notice that (U4) is equivalent to (U4'): if $U\in \fkU_e(\CE)$ then $\bigcup\braces{(g\ast U)\ast g^{-1}\mid g\in G}\in \fkU_e(\CE)$. 

\begin{proof}
 Properties (U0) and (U1) are clear. 
 For (U2), fix $U\in\fkU_e(\CE)$. By definition, $e\in U$ and $\Delta_G\ast(U\times U)\in \CE$. Since the operations are adapted, it then follows that
 \[
  e\rel{\invRel} (g^{-1}\ast g)
  \rel{\Delta_G\ast(U\times U)} (g^{-1}\ast e) = g^{-1} \qquad \forall g\in U.
 \]
 This shows that $U^{-1}$ is $\CE$\=/bounded. Since $e\in U$ and the operation are adapted, $e=e^{-1}$ is in $U^{-1}$, hence $U^{-1}\in\fkU_e(\CE)$.
 
 For (U3), fix $U_1,U_2\in\fkU_e(\CE)$. By (U2), $U_1^{-1}$ is also in $U_2\in\fkU_e(\CE)$. Hence $\Delta_G\ast(U_1^{-1}\times U_2)\in \CE$. Since the operations are adapted, we have:
 \[
  e = (g_1\ast g_1^{-1})\torel{\Delta_G\ast(U_1^{-1}\times U_2)} (g_1\ast g_2) \qquad \forall g_1\in U_1,\ g_2\in U_2.
 \]
 This shows that $U_1\ast U_2$ is bounded, hence it is in $\fkU_e(\CE)$ because $e\ast e=e$.
 
 Assume now that $\CE$ is equi bi\=/invariant. To prove (U4), fix  $U\in\fkU_e(\CE)$ and notice that for every $g\in G$ the identity $e$ belongs to $g\ast(U\ast g^{-1})$ because the operations are adapted. Recall that $\pts{G}\ast U$ denotes the covering $\braces{g\ast U\mid g\in G}$. Since left and right multiplication are equi controlled, the covering $\pts{G}\ast(U\ast\pts{G})$ is controlled. It is then enough to notice that $\bigcup\braces{g\ast(U\ast g^{-1})\mid g\in G}$ is contained in the star\=/neighborhood $\st(\{e\}\mid \pts{G}\ast(U\ast\pts{G}))$.
\end{proof}

Recall that every normally finitely generated set\=/group admits a canonical coarse structure $\varcrs{bw}$ (see Definition~\ref{def:p1:canonical coarse structure}). Let $\varcrs[grp]{fin}$ denote the minimal equi bi\=/invariant coarse structure containing $\varcrs{fin}$. Using the above lemma it is simple to prove the following:

\begin{cor}[Example~\ref{exmp:p1:finite sets group coarse structure}]\label{cor:p1:canonical crse_str is E_grp_fin}
Let $G$ be a finitely normally generated set\=/group. Then
 \[
  \varcrs{bw}=\varcrs[grp]{fin}.
 \]
 In particular, $\varcrs{bw}$ does not depend on the choice of finite generating set $S$.
\end{cor}
\begin{proof}
 Let $S$ be a finite normally generating set and let $d_{\overline{S}}$ be the bi\=/invariant word\=/metric associated with $S$, so that $\CE_{d_{\overline{S}}}$ is equal to the canonical coarse structure $\varcrs{bw}$. We may assume that $e\in S$, then the $\CE_{d_{\overline{S}}}$\=/bounded neighborhoods of $e$ are contained in sets of the form $\overline{S}^\pm\ast\cdots\ast\overline{S}^\pm$, where $\overline{S}^\pm$ is the union of all the conjugates of $S$ and $S^{-1}$.
 
 Since $S$ is finite, it is a $\varcrs{fin}$\=/bounded neighborhood of $e$. It then follows from Lemma~\ref{lem:p1:properties of neighborhoods of id} that every $\CE_{d_{\overline{S}}}$\=/bounded neighborhood of $e$ is $\varcrs[grp]{fin}$\=/bounded. Then Proposition~\ref{prop:p1:neighbourhoods.of.identity.generate} implies  $\CE_{d_{\overline S}}\subseteq\varcrs[grp]{fin}$, and by minimality, these two coarse structures must be equal. 
\end{proof}

For set\=/groups, conditions (U0)--(U4) completely determine the families of subsets that can arise as bounded neighborhoods of the identity for some coarsification of $G$. Namely, the following is true:

\begin{prop}\label{prop:p1:families of identity neighbourhoods_setgroups}
 Let $G$ be a set\=/group and let $\fku$ be a partial cover satisfying (U0)--(U3). Then there exists a unique equi left invariant coarse structure $\varcrs[left]{\fku}$ such that 
 \(
  \fku=\fkU_e\bigparen{\varcrs[left]{\fku}}.
 \)
 If $\fku$ also satisfies (U4) then $\varcrs[left]{\fku}$ is equi bi\=/invariant and hence $(G,\varcrs[left]{\fku})$ is a coarse group (we may thus write $\varcrs[left]{\fku}=\varcrs[grp]{\fku}$).
\end{prop}
\begin{proof}
The uniqueness part of the statement follows immediately from Proposition~\ref{prop:p1:neighbourhoods.of.identity.generate}, it thus remains to show the existence. This is slightly easier to prove in terms of controlled coverings. 
Let 
\[
 \fkC_\fku\coloneqq\braces{\fka\mid \exists  U\in\fku,\ \fka \text{ is a refinement of }\pts{G}\ast U}.
\]
We claim that if $A\subseteq G$ is so that $e\in A$ and $\{A\}\in \fkC_\fku$, then $A\in\fku$. By definition of $\fkC_\fku$, if $\{A\}\in \fkC_\fku$ then $A$ must be contained in $g\ast U$ for some $U\in\fku$. Since $e\in A$ and $e\in U$, the couple $\{e,g^{-1}\}$ is contained in $U$ and it is hence in $\fku$ by (U1). By (U3), $\{e,g\}$ is also in $\fku$. By (U2), the product $\{e,g\}\ast U=U\cup g\ast U$ is in $\fku$ as well. Since $A\subseteq g\ast U$, we see that $A$ is in $\fku$ by (U1).

All it remains to show is that $\fkC_\fku$ is the family of controlled partial coverings of some equi left invariant coarse structure $\CE$ on $G$. In fact, it would then follow from the paragraph above that $\fku$ is the set of $\CE$\=/bounded neighborhoods of the identity and we can hence let $\varcrs[left]{\fku}=\CE$.
By Proposition~\ref{prop:p1:crse structures vs ctrl coverings}, $\fkC_\fku=\fkC(\CE)$ for some coarse structure $\CE$ if and only if it satisfies
\begin{itemize}
 \item [(C0)] if $\fka\in\fkC_\fku$ then $\fka\cup\pts{X}\in\fkC_\fku$;
 \item [(C1)] if $\fka\in\fkC_\fku$ and $\fka'$ is a refinement of $\fka$, then $\fka'\in\fkC_\fku$;
 \item [(C2)] given $\fka,\fkb\in\fkC_\fku$, then $\st(\fka,\fkb)\in\fkC_\fku$.
\end{itemize}
Since the first two conditions are trivially satisfied, it is enough to check (C2).

Fix $\fka,\fkb\in\fkC_\fku$. Enlarging them if necessary, we may assume that they are coverings of the form $\pts{G}\ast U_A$, $\pts{G}\ast U_B$ for some $U_A,U_B\in\fku$. We need to show that $\st(\pts{G}\ast U_A,\pts{G}\ast U_B)$ is in $\fkC_\fku$. For every fixed $g\in G$,
\begin{align*}
 \st(g\ast U_A,\pts{G}\ast U_B) = g\ast \st(U_A,\pts{G}\ast U_B).
\end{align*}
It is hence enough to verify that $\st(U_A,\pts{G}\ast U_B)\in\fku$. Note that
\begin{align*}
 \st(U_A,\pts{G}\ast U_B)= \bigcup_{h\in U_A}\braces{g\ast U_B\mid h\in g\ast U_B}.
\end{align*}
If $h\in g\ast U_B$, then $g^{-1}\ast h\in U_B$ and therefore $g^{-1}\in U_B\ast U_A^{-1}$. It follows that $\st(U_A,\pts{G}\ast U_B)$ is contained in $U_A\ast U_B^{-1}\ast U_B$, which is in $\fku$ by (U2) and (U3).

We now know that $\fkC_\fku=\fkC(\CE)$ for some coarse structure $\CE$, and $\CE$ is equi left invariant because any $\fka\in\fkC(\CE)$ is a refinement of some $\pts{G}\ast U$, and therefore $\pts{G}\ast\fka$ is a refinement of $\pts{G}\ast(\pts{G}\ast U)=(\pts{G}\ast\pts{G})\ast U=\pts{G}\ast U$.

Let us now show that if $\fku$ also satisfies (U4) then the coarse structure $\CE$ above is equi bi\=/invariant. It is enough to verify that for every $U\in\fku$ the covering $(\pts{G}\ast U)\ast\pts{G}$ is in $\fkC_\fku$. This is readily done. Since
\[
 g_1 \ast U\ast g_2 = (g_1g_2) \ast g_2^{-1} \ast U\ast g_2
 \subseteq (g_1g_2) \ast \Bigparen{\bigcup\braces{g^{-1} \ast U\ast g\mid g\in G}},
\]
letting $\overline{U} \coloneqq\bigcup\braces{g^{-1} \ast U\ast g\mid g\in G}$ we see that $(\pts{G}\ast U)\ast\pts{G}$ is a refinement of $\pts{G}\ast\overline{U}$. The latter belongs to $\fkC_\fku$ by (U4).
\end{proof}

\begin{rmk}
Instead of the ``family of bounded neighborhoods of the identity'',  the bornology $\fkB$ of bounded sets can be used (see Definition~\ref{def:appendix:bornology}).
 However, this formalism only works well for connected coarse structures, because otherwise it is not necessarily true that for every $B_1,B_2\in\fkB$ the union $B_1\cup B_2$ is still in $\fkB$.
 In turn, it is awkward to correctly formulate analogs of (U0)--(U4) such that Proposition~\ref{prop:p1:families of identity neighbourhoods_setgroups} remains true.
 In fact, the unit $e$ is used in a key way, and if there exist $B\in\fkB$ with $e\notin B$ and $\{e\}\cup B\notin \fkB$ then the proof does not follow through.
\end{rmk}

There are instances where the easiest way define a coarsification of set\=/group $G$ is by specifying its bounded neighborhoods of the identity. 
Namely, if we are given a family $\fkU$ of subsets of $G$ satisfying (U0)--(U4), we can immediately apply Proposition~\ref{prop:p1:families of identity neighbourhoods_setgroups} to obtain a well\=/defined equi bi\=/invariant coarse structure $\varcrs[grp]{\fkU}$ on $G$.
Also note that if $G$ is abelian then condition (U4) is always satisfied. This makes it very easy to define coarsifications of abelian set\=/groups.

\begin{exmp}\label{exmp:p1:topological coarse structure abelian}
 Let $G$ be an abelian topological group. Then family $\fkK_e$ of relatively compact subsets of $G$ containing the identity element satisfies the conditions (U0)--(U4) of Section~\ref{sec:p1:determined locally} (the product of two compact sets is compact because the continuous map $\ast\colon G\times G\to G$ sends compact sets to compact sets).
 It follows from Proposition~\ref{prop:p1:families of identity neighbourhoods_setgroups} there exists a unique equi bi\=/invariant coarse structure $\varcrs[grp]{\fkK_e}$ whose set of bounded neighborhoods of the identity $\fkU_e(\varcrs[left]{cpt})$ is $\fkK_e$. We shall denote the coarse structure $\varcrs[grp]{\fkK_e}$ by $\varcrs[grp]{cpt}$.
 Explicitly, we have
\begin{equation}\label{eq:p1:controlled in E_grp_cpt}
 \varcrs[grp]{cpt} = \braces{E\subseteq G\times G\mid \exists K\subseteq G\text{ compact s.t. }\forall(g_1,g_2)\in E,\ g_2^{-1}g_1\in K}
\end{equation}
 The $\varcrs[grp]{cpt}$\=/controlled partial coverings are refinements of coverings of the form $\pts{G}\ast K$ with $K\in\fkK_e$.
 
 As a special case, if $G$ is a discrete abelian topological groups then the compact subsets are finite and thus $\varcrs[grp]{cpt}$ is equal to $\varcrs[grp]{fin}$.
\end{exmp}

\begin{rmk}
 If $G$ is not abelian, one may still define a coarse structure $\varcrs[grp]{cpt}$ as we did for $\varcrs[grp]{fin}$.
 Namely, by taking the minimal equi bi\=/invariant coarse structure so that compact sets are bounded.
 However, this coarse structure may have many non relatively compact bounded sets and it does not have a simple description like \eqref{eq:p1:controlled in E_grp_cpt}. This makes it fairly complicated to understand in general.
\end{rmk}

\section{Summary: A Concrete Description of Coarse Groups}
\label{ssec:p1:recap on crse grps}

Putting together the criteria we proved this far, we may explicitly characterise coarse groups as follows.

\begin{thm}\label{thm:p1:concrete coarse groups}
 Let $(G,\CE)$ be a coarse space, $\ast\colon G\times G\to G$ a binary law, $\inversefn\colon G\to G$ a function and $e\in G$ an element,
 then $(G,\CE,\cop,\cinversefn,\cunit)$ is a coarse group if and only if the following conditions are satisfied:
 \[\def\arraystretch{1.5}
 \begin{array}{lcc}
%   x\ast y\torel{\CE} x\ast y'
%   &\qquad \forall g_1,\in G,\forall g_2\torel{E}g_2'
%   \\
%   x\ast y\torel{\CE} x'\ast y
%   &\qquad \forall g_1\torel{E}g_1',\ \forall g_2\in G
%   \\
  &
  g_1 \ast (g_2\ast g_3) \rel{\CE} (g_1\ast g_2)\ast g_3
  & \qquad \forall g_1,g_2,g_3\in G
   \\
  &
   e\ast g \rel{\CE} g\rel{\CE} g\ast e
   & \qquad \forall g\in G
  \\
  &
  g\ast g^{-1}\rel{\CE} e\rel{\CE} g^{-1}\ast g
  & \qquad \forall g\in G.
  \\
 \text{and for every }E\in \CE & &
 \\
  &
  x\ast y\torel{\CE} x\ast y'
  &\qquad \forall g_1,\in G,\forall g_2\torel{E}g_2'
  \\
  &
  x\ast y\torel{\CE} x'\ast y
  &\qquad \forall g_1\torel{E}g_1',\ \forall g_2\in G.
 \end{array}
 \]
\end{thm}

In concrete cases---especially when $(G,\ast)$ is a set\=/group---it is often convenient to verify that the bounded neighbourhoods of the identity satisfy properties (U0)--(U4), see Lemma~\ref{lem:p1:properties of neighborhoods of id} and Proposition~\ref{prop:p1:families of identity neighbourhoods_setgroups}.

\chapter{Coarse Homomorphisms, Subgroups and Quotients} %%%%%%%%%%%%%%%%%%%%%%
\label{ch:p1:coarse.subgroups.and.quotients}

\section{Definitions and Examples of Coarse Homomorphisms}\label{sec:p1:coarse homomorphisms_1}
Let $\crse G=(G,\CE)$ and $\crse H=(H,\CF)$ be coarse groups with coarse multiplication $\cop_{\crse G}$ and $\cop_{\crse H}$ respectively.  A coarse homomorphism from $\crse G$ to $\crse H$ is a coarse map that is compatible with the coarse multiplication:

\begin{de}\label{def:p1:coarse homomorphism}
 A \emph{coarse homomorphism}\index{coarse!homomorphism} is a coarse map $\crse f\colon \crse G\to\crse H$ such that the following diagram commutes:  
\[ \begin{tikzcd} 
    \crse{G}\times \crse{G} \arrow[r, "\crse f\times \crse f"] \arrow[d, "\cop_{\crse G}"] & \crse{H}\times \crse{H}      \arrow[d, "\cop_{\crse H}"]\\
    \crse G \arrow[r, "\crse f"]& \crse H . 
\end{tikzcd} 
\]
We denote the set of coarse \emph{homomorphisms} by $\chom(\crse{G},\crse{H})$\nomenclature[:MOR]{$\chom(\crse{G},\crse{H})$}{set of coarse homomorphisms of $\crse G$ into $\crse H$ } 
(this should not be confused with the set $\cmor(\crse{G},\crse{H})$ of all the coarse \emph{maps} between the coarse spaces $\crse G$ and $\crse H$).
\end{de}

After fixing representatives, a controlled function $f\colon G\to H$ gives rise to a coarse homomorphism $\crse f \colon \crse{G}\to \crse H$ if and only if $f(x\ast_{G}y)$ and $f(x)\ast_H f(y)$ are uniformly close (\emph{i.e.}\ there is an $F\in\CF$ such that $f(x\ast_{G}y)\rel{F} f(x)\ast_H f(y)$ for every $x,y\in G$).

\begin{example}
 If $f\colon G\to H$ is a function sending $G$ into a bounded neighborhood of $e_H$, then $\crse f$ is a coarse homomorphism. Notice that $\crse f$ is nothing but the trivial coarse homomorphism $\crse f= \cunit_{\crse H}$.
 
 A second trivial source of examples is obtained by considering set\=/homomorphisms between coarsified set\=/groups. Namely, if $G$ and $H$ are set\=/groups equipped with equi bi\=/invariant coarse structures $\CE$ and $\CF$, then a homomorphism of set\=/groups $f\colon G\to H$ defines a coarse homomorphism if and only if $f\colon (G,\CE)\to (H,\CF)$ is a controlled function.
 
 More generally, if $\crse G$ and $\crse H$ are coarsified set groups and $f'\colon G \to H$ is a function that is close to a controlled homomorphism of set\=/groups $f\colon G \to H$, then $\crse f'$ trivially defines a coarse homomorphism because $\crse{f'=f}$.
\end{example}

\begin{rmk}\label{rmk:p1:checking set_homs are controlled}
 Given coarsified set groups $(G,\CE)$ and $(H,\CF)$, it is simple to verify whether a homomorphism of set\=/groups $f\colon G\to H$ is controlled. By Corollary~\ref{cor:p1:neighbourhoods.of.identity.generate_setgroup}, the coarse structures $\CE$ and $\CF$ are equal to $\angles{\Delta_{G}\ast \CU_e(\CE)}$ and $\angles{\Delta_{G}\ast \CU_e(\CF)}$ respectively. Therefore, $f$ is controlled if and only if for every  $\CE$\=/bounded neighborhood $U$ of $e_{G}$ there is a $\CF$\=/bounded neighborhood $V$ of $e_H$ so that 
 \[
  f\times f\bigparen{\Delta_G\ast(U\times U)}\subseteq \Delta_H\ast(V\times V).
 \]
 Since $f$ is a homomorphism of set\=/groups, $f\times f(\Delta_{G}\ast (U\times U))$ is contained in $\Delta_H\ast (f(U)\times f (U))$. It follows that $f$ is controlled if and only if it sends bounded neighborhoods of $e_{G}$ into bounded neighborhoods of $e_H$. We will later see that similar observations apply to more general settings (Proposition~\ref{prop:p1:bornologous map is coarse hom})
\end{rmk}

\begin{rmk}
 If $(G,\mincrs)$ and $(H,\mincrs)$ are trivially coarse groups, then $\crse{f\colon G\to H}$ is a coarse homomorphism if and only if $f$ is a homomorphism of set\=/groups.
\end{rmk}

More interesting example of coarse homomorphisms are given by $\RR$\=/quasimorphisms of set\=/groups.

\begin{de}\label{def:p1:quasimorphism}
 Let $G$ be a set\=/group. A function $\phi\colon G\to\RR$ is a \emph{$\RR$\=/quasimorphism}\index{quasimorphism!$\RR$} of $G$ if there is a constant $D$ such that $\abs{\phi(xy)-\phi(x)-\phi(y)}\leq D$ for every $x,y\in G$. The \emph{defect}\index{defect} of $\phi$ is 
 \[
    D(\phi)  \coloneqq \sup_{g,h\in G}\abs{\phi(gh)-\phi(g)-\phi(h)}.
 \]
\end{de}

Of course, any function that is close to a homomorphism of set\=/groups $G\to \RR$ is an $\RR$\=/quasimorphism. On the other hand, there are many known constructions of $\RR$\=/quasimorphisms that are not close to any set\=/group homomorphism.\footnote{%
For any finitely generated set\=/group $G$, the space of $\RR$\=/quasimorphisms that are not close to set\=/group homomorphism can be naturally identified with the kernel of the comparison map $H^2_b(G;\RR)\to H^2(G,\RR)$ from bounded to ordinary group cohomology. See \emph{e.g.}\ \cite{calegari2009scl,frigerio2017bounded} for more on this beautiful subject.
}
The easiest examples to describe are some $\RR$\=/quasimorphism of the free group:

\begin{exmp}\label{exmp:p1:brooks_quasimorphisms}  
 Let $F_2=\angles{a,b}$ be the non\=/abelian free group. Given any reduced word $w$ in $\{a,b,a^{-1},b^{-1}\}$, we consider the counting function $C_w\colon G\to \NN$ assigning to an element $g\in G$ the number of times that the word $w$ appears in the reduced word representing $g$. The associated \emph{Brooks quasimorphism} is defined as $\phi_w(g)\coloneqq C_w(g)-C_{w^{-1}}(g)$. It is elementary to verify that $\phi_w$ is indeed a $\RR$\=/quasimorphism and we shall soon see that for ``most'' choices of $w$ the resulting $\phi_w$ is not close to any homomorphism of set\=/groups \cite{brooks1981some,mitsumatsu1984bounded}. For instance, let $w=ab$. If $\phi_{ab}$ was within distance $C$ from a homomorphism of set\=/groups $\bar \phi\colon F_2\to \RR$ then we would see that 
 \[
  C\geq \abs{\phi_{ab}((ab)^n)-\bar\phi((ab)^n)}=\abs{n(1-\bar\phi(ab))}
 \]
 for every $n\in\NN$. From this we deduce that $\bar\phi(ab)=1$. However, the same argument would also show that $\bar\phi(a)=\bar\phi(b)=0$, hence $\bar \phi$ cannot be a homomorphism.
  
 A simple variation on the above construction is obtained replacing $C_w$ by the counting function $c_w$ that associates with $g\in F_2$ the maximal number of non\=/overlapping copies of $w$ in the reduced word representing $g$. A rather different construction goes as follows. Given a bounded function $\alpha\colon \ZZ\to \RR$ so that $\alpha(-k)=-\alpha(k)$, we define a function $\phi_\alpha\colon F_2\to\RR$ by
 \[
  \phi_\alpha(s_1^{k_1}s_2^{k_2}\cdots s_n^{k_n})\coloneqq\sum_{i=1}^n\alpha(k_i),
 \]
 where each $s_i$ is one of $a,b,a^{-1},b^{-1}$, the word $s_1^{k_1}s_2^{k_2}\cdots s_n^{k_n}$ is reduced and $s_i\neq s_{i+1}$ for every $i=1,\ldots,n-1$. The function $\phi_\alpha$ is a $\RR$\=/quasimorphism \cite{rolli2009quasi}.
\end{exmp}

We already remarked that $(\RR,\varcrs{\abs{\mhyphen}})$ is a coarse group, where $\varcrs{\abs{\mhyphen}}$ is the metric coarse structure associated with the Euclidean norm. Notice that if $\crse G=(G,\CE)$ is a coarsified set\=/group and $\phi\colon(G,\CE)\to (\RR,\varcrs{\abs{\mhyphen}})$ is a controlled map, then $\phi$ is a coarse homomorphism if and only if it is an $\RR$\=/quasimorphism. In other words, the set of coarse homomorphisms $\crse G\to(\RR,\varcrs{\abs{\mhyphen}})$ is naturally identified with the set of equivalence classes of controlled $\RR$\=/quasimorphisms
\[
 \chom\bigparen{\crse G,(\RR,\varcrs{\abs{\mhyphen}}) }\longleftrightarrow \braces{\phi\colon G\to \RR\mid \text{ controlled $\RR$\=/quasimorphism}}/_\closefn.
\]

It is immediate to observe that Brooks quasimorphisms $(F_2,\varcrs{bw})\to(\RR,\varcrs{\abs{\mhyphen}})$ are controlled (this is in fact true for every $\RR$\=/quasimorphism of normally finitely generated groups, Corollary~\ref{cor:p1:quasimorphisms are controlled}). It then follows from Example~\ref{exmp:p1:brooks_quasimorphisms} that $[\phi_{ab}]\colon (F_2,\varcrs{bw})\to(\RR,\varcrs{\abs{\mhyphen}})$ is a coarse homomorphism that does not admit any homomorphisms of set\=/groups as a representative.

Building on similar ideas, one can construct various examples of coarse homomorphisms of coarsified set\=/groups.  
Below is an interesting construction of a coarse homomorphism $\crse f \colon (F_2,\varcrs{bw})\to (F_2,\varcrs{bw})$ given by Fujiwara--Kapovich \cite[Section 9.2]{FK}.

\begin{exmp}    
Let $F_2=\angles{a,b}$. Following \cite{FK}, choose two ``appropriate"  words, say $u = a^2 b a^2$ and $v = b^2 a b^2$ and let $H\cong F_2$ be the subgroup generated by $u,v$. Define $f_{u,v}\colon F_2 \to H$ letting 
\( f_{u,v} (w) =  t_1 \cdots t_n  \), 
where $t_i \in \{u, u^{-1}, v, v^{-1} \}$ are substrings of the reduced word $w$ sorted in order of appearance. For example, with the choices of $u,v$ above we see:
\[
f_{u,v}( a^2b a^2b a^2b^2ab^2a^2) = uuv \quad \text{ and }\quad f_{u,v}(ba^{-2}b^{-1}a^{-2}b^2ab^{-2}a^{-1}b^{-2})=u^{-1}v^{-1}. 
\]

Notice that our choice of $u$ and $v$ freely generate $H$, so that $H\cong F_2=\angles{u,v}$. We can hence see $f_{u,v}$ as a function $f_{u,v} \colon (F_2,\varcrs{bw})\to (F_2,\varcrs{bw})$ and it is simple to verify that $f_{u,v}$ is a coarse homomorphism (see also Proposition~\ref{prop:p1:bornologous map is coarse hom}).  We further claim that $f_{u,v}$ is not close to a set\=/group homomorphism. Notice that $f_{u,v}\colon F_2\to H$ acts as the identity on the cyclic subgroups $\langle u \rangle$ and $\langle v \rangle$.  So the image of $f_{u,v}$ is unbounded in $(H,\varcrs{bw})$.  However, $f_{u,v}$ sends the subgroups generated by $\langle a \rangle$ and $\langle b \rangle$ to the identity.  Thus any homomorphism close to $f_{u,v}$ must send $\langle a \rangle$ and $\langle b \rangle$ to bounded subgroups of $(H,\varcrs{bw})$, but this is possible only for the trivial homomorphism. 
Alternatively, let $\phi\colon F_2 \to \ZZ $ be the set\=/homomorphism such that $\phi(u) = 1 \in \ZZ$ and $\phi(v) =0$. The astute reader will notice the composition $f_{u,v} \circ \phi$ is the Brooks quasimorphism $\phi_u$. If $f_{u,v}$ was close to a homomorphism, so would be the composition $f_{u,v}\circ\phi$.
\end{exmp}

In Chapter~\ref{ch:p2:coarse autos} we will provide many more examples of coarse homomorphisms between coarsified set groups, \emph{e.g.}\ using quasimorphisms in the sense of Hartnick--Schweitzer \cite{HS}. 
Conjugation gives examples of coarse homomorphisms of a somewhat different flavor:

\begin{exmp}\label{exmp:p1:conjugation automorphisms}
 Let $\crse G=(G,\CE)$ be a coarse group and $g\in G$ a fixed element. Then it is immediate to check that the coarse conjugation ${}_gc\colon G\to G$ given by $h\mapsto (g\ast h)\ast g^{-1}$ defines a coarse homomorphism. However, this coarse homomorphism is often trivial. 
 
 In fact, if $g$ lies in the same coarsely connected component of $e$ (\emph{i.e.}\ $e\torel{E}g$ for some $E\in\CE$) then $h\torel{\Delta_G\ast E} h\ast g$ for every $h\in G$ and hence the right multiplication function $\ast_g$ is close to $\id_{G}$. Analogously, also left multiplication ${}_g\ast$ is close to $\id_{G}$ and hence ${}_gc$ defines the trivial coarse automorphism $[{}_gc]=\cid_{\crse G}$. In particular, this shows that if $\crse G$ is coarsely connected then all the conjugation automorphisms are coarsely trivial.
 
 More in general, recall that the set of coarsely connected components of $\crse G$ is naturally identified with $\cmor(\tobj,\crse G)$, which is a set\=/group by Lemma~\ref{lem:p1:mor.is.group}. It is then easy to see that the set of conjugation automorphisms of $\crse G$ can be identified with the group of inner automorphisms of $\cmor(\tobj,\crse G)$.
\end{exmp}

\

For later reference, we show that coarse homomorphisms from a coarse group $\crse G$ to a Banach space always have a preferred representative (this fact is very well\=/known for $\RR$\=/quasimorphisms).
An $\RR$\=/quasimorphism $\phi\colon G\to \RR$ is \emph{homogeneous} if $\phi(g^k)=k\phi(g)$ for every $k\in\ZZ$ and $g\in G$. It is an important observation that every $\RR$\=/quasimorphism is close to a homogeneous one. The same argument (see \emph{e.g.}\ \cite[Section 2.2]{calegari2009scl}) proves the following:

\begin{prop}\label{prop:p1:homogeneization.of.qmorph}
 Let $\crse G $ be a coarse group, $(V,\norm{\mhyphen})$ a Banach space and $\CE_{\norm{\mhyphen}}$ the associated metric coarse structure on $V$. Let $\crse {f\colon G}\to (V,\CE_{\norm{\mhyphen}})$ be a coarse homomorphisms. Then $\crse f$ has a representative $\bar f$ such that: 
 \begin{itemize}
  \item[$(\spadesuit)$] $\bar f(g\ast g)=2\bar f(g)$ for every $g\in G$.
 \end{itemize}

 Furthermore, if $\crse G$ is a coarsified set group, we may further assume that $\bar f(g^k)=k f(g)$ for every $g\in G$ and $k\in \ZZ$.
\end{prop}
\begin{proof}[Sketch of Proof]
 Fix a representative $f$ for $\crse f$.
 Replacing $f$ with $f'(g)\coloneqq \frac{1}{2}(f(g)-f(g^{-1}))$, we can assume that $f(g^{-1})=-f(g)$ for every $g\in G$. Let $D\coloneqq \sup_{g,h\in G} \norm{f(g\ast h)-f(g)-f(h)}$, this supremum is finite because $f$ is a coarse homomorphism. Given $g\in G$, let $g_0\coloneqq g$ and define by induction $g_{k+1}\coloneqq g_k\ast g_k\in G$. An induction argument shows that 
 \[
  \norm{\frac{f(g_{k+1})}{2^{k+1}}-\frac{f(g_k)}{2^{k}}}\leq \frac{D}{2^{k+1}}.
 \]
 Since $(V,\norm{\mhyphen})$ is complete, the Cauchy sequence ${D}/{2^{k+1}}$ has a limit. We may thus define
 \[
  \bar f(g)\coloneqq \lim_{k\to \infty}\frac{f(g_k)}{2^{k}}.
 \]

 Since $\norm{\bar f(g)-f(g)}\leq D$, we have that $\bar f$ is a representative for $\crse f$. Condition $(\spadesuit)$ holds trivially for $\bar f$.

 For the `furthermore' part of the statement, for every $k\in\NN$ write $(g^k)^{2^n}$ as the $k$\=/fold product $g^{2^n}\ast\cdots\ast g^{2^n}$ and deduce by induction on $k$ that
 \[
 \norm{f(g^{k{2^n}})- k f(g^{2^n})}\leq (k-1)D.
 \]
 It then follows from the definition that $\bar f(g^k)=k\bar f(g)$ for every $k\in \NN$. The same holds for $k\in \ZZ$ because we assumed that $f(g^{-1})=-f(g)$ for every $g\in G$.
\end{proof}

\begin{rmk}  
 Note that the argument for the `furthermore' part Proposition~\ref{prop:p1:homogeneization.of.qmorph} fails dramatically if $G$ is not a coarsified set\=/group, because the element $g^k$ is not well defined: since $\ast$ is not associative it is necessary to specify the order of multiplication of the various copies of $g$. Once an order is fixed, we cannot write $(g^k)^{2^n}=g^{2^n}\ast\cdots \ast g^{2^n}$.
 In fact, we do not know whether every coarse homomorphism $\crse f\colon \crse G\to (\RR,\varcrs{\abs{\mhyphen}})$ has a representative $\bar f$ such that the $k$\=/fold products satisfy
  \[
   \norm{\bar f\bigparen{g\ast(g\ast\cdots (g\ast g))} - \bar f\bigparen{((g\ast g)\ast\cdots g)\ast g}}\leq R
  \]
  for every $g\in G$ and $k\in \NN$ and some constant $R\geq 0$ independent of $g,k$.
 \end{rmk}

\section{Properties of Coarse Homomorphisms}
\label{sec:p1:coarse homomorphisms_2}

In Chapter~\ref{sec:p1:coarse homomorphisms_1}, we defined coarse homomorphisms as controlled maps that preserve the coarse multiplication. From a semantic point of view, it would have been more precise to also require that coarse homomorphisms preserve the coarse unit and inversion. The next lemma shows that these two approaches are equivalent:

\begin{lem}\label{lem:p1:coarse homomorphisms preserve unit and inverse}

 If  $\crse{f}\colon \crse G\to \crse H$ is a coarse homomorphism then the following diagrams commute:
 \[
  \begin{tikzcd}
    \tobj \arrow[r, "\cunit_{\crse G}"] \arrow[dr,swap, "{\cunit_{\crse H}}"] 
     & \crse G     \arrow[d, "{\crse f}"] 
    \\
    & \crse H  
  \end{tikzcd}\qquad\qquad
  \begin{tikzcd}
    \crse G \arrow[r, "\cinversefn_{\crse G}"] \arrow[d,swap, "{\crse f}"] 
     & \crse G     \arrow[d, "{\crse f}"] 
    \\
    \crse H \arrow[r, "\cinversefn_{\crse H}"]  & \crse H  
  \end{tikzcd}
 \]

 In other words, $\crse f$ sends $\cunit_{\crse G}$ to $\cunit_{\crse H}$ and respects the inversion coarse maps.
\end{lem}

\begin{proof}
 It is not hard to check the statement by hand. Alternatively, recall that for every coarse space $\crse X$ we can define an operation $\odot$ on the set of coarse maps $\cmor(\crse{X},\crse G)$ by letting
 \[
  \crse h_1 \odot  \crse h_2 \coloneqq \crse X \xrightarrow{( \crse h_1, \crse h_2)} \crse{ G \times  G} \xrightarrow{\ \cop\ }  \crse G. 
 \]
 and $\cmor(\crse{X},\crse G)$ equipped with this operation is a set\=/group (see Lemma~\ref{lem:p1:mor.is.group}). If $\crse{f}\colon \crse G\to \crse H$ is a coarse homomorphism, the composition with $\crse{f}$ induces a set\=/group homomorphism $\cmor(\crse{X},\crse G)\to\cmor(\crse{X},\crse H)$. 
 
 By Lemma~\ref{lem:p1:mor.is.group}, we know that $\cunit_{\crse G}$ is the identity element of $\cmor(\tobj,\crse G)$. Since homomorphisms of set\=/groups send the unit to the unit, we see that the composition $\crse f\circ \cunit_{\crse G}$ is equal to $\cunit_{\crse H}$.
 Again by Lemma~\ref{lem:p1:mor.is.group}, we also know that the inverse element of $\crse h\in\cmor\paren{\crse X,\crse G}$ is the composition $\cinversefn_{\crse G}\circ \crse h$. In particular, $\cinversefn_{\crse G}$ is the inverse element of $\cid_{\crse G}\in\cmor\paren{\crse G,\crse G}$. Since homomorphisms of set\=/groups send inverses to inverses, the composition $\crse f\circ \cinversefn_{\crse G}$ is the inverse element of $\crse f=\crse f\circ \cid_{\crse G}\in \cmor(\crse {G,H})$, which we know to be $\cinversefn_{\crse H}\circ \crse f$.
\end{proof}

It is immediate to verify that compositions of coarse homomorphisms are coarse homomorphisms. 

\begin{de}
 An \emph{isomorphism of coarse groups}\index{isomorphic!coarse groups} is a coarse homomorphism $\crse{f\colon G\to H}$ so that there is a coarse homomorphism $\crse f^{-1}\colon \crse{H\to G}$ with $\crse f^{-1}\circ \crse f=\cid_{\crse G}$ and $\crse{f\circ f}^{-1}=\cid_{\crse H}$. Two coarse groups are \emph{isomorphic} (denoted $\crse{G\cong H}$) if there is an isomorphism between them. We will generally use `isomorphism' and `coarse isomorphism' as synonyms. An \emph{automorphism} of $\crse G$ is a coarse isomorphism of $\crse G$ with itself. Note that the composition is a set\=/group operation on the set of coarse automorphisms $\cAut(\crse G) $.
\end{de}

Since composition of coarse homomorphisms are coarse homomorphisms, it follows isomorphism of coarse groups is an equivalence relation. It is also immediate to verify that if a coarse homomorphism $\crse{f}\colon \crse G\to \crse H$ is a coarse equivalence, the coarse inverse $\crse f^{-1}\colon\crse { H}\to\crse G$ is also a coarse homomorphism. In other words, we see that a coarse isomorphism is a coarse homomorphism that is a coarse equivalence (\emph{i.e.}\ an isomorphism in \Cat{Coarse}).

\begin{rmk}
 All the observations we made thus far are not specific to coarse groups: they hold for group objects in any category. On the contrary, the next results only hold in \Cat{Coarse}.
\end{rmk}

\begin{exmp}\label{exmp:p1:isomorphisms of abelian topological groups}
 The inclusion $\ZZ\hookrightarrow \RR$ defines an isomorphism of coarse groups $(\ZZ,\CE_{\abs{\mhyphen}})\to(\RR,\CE_{\abs{\mhyphen}})$. More generally, if $G$ is a set\=/group with a bi\=/invariant metric $d$ and $H \leq G$ is a coarsely dense subgroup then $(H,\CE_{d|_H})\hookrightarrow(G,\CE_d)$ is an isomorphism of coarse groups.
 
 Similarly, recall that every abelian topological group $G$ has an equi bi\=/invariant coarse structure $\varcrs[grp]{cpt}$ so that the $\varcrs[grp]{cpt}$\=/bounded subsets are the relatively compact subsets of $G$ (Example~\ref{exmp:p1:topological coarse structure abelian}). If $H\leq G$ is a subgroup so that the quotient $G/H$ is compact then $H$ is coarsely dense in $(G,\varcrs[grp]{cpt})$. If we equip $H$ with the subset topology, then the coarse group $(H,\varcrs[grp]{cpt})$ is isomorphic to $(G,\varcrs[grp]{cpt})$. In particular, if $H$ is discrete in $G$ we see that $(G,\varcrs[grp]{cpt})$ is isomorphic to $(H,\varcrs[grp]{fin})$. As a special case we notice that the coarsification of the group of rational number $(\QQ,\varcrs[grp]{fin})$ is isomorphic to the topological coarsification of the group of adeles $(\AA_\QQ,\varcrs[grp]{cpt})$. 
\end{exmp}

Recall the notion of pull\=/back coarse structure (Definition~\ref{def:p1:pull-back}). The following observation is useful:

\begin{lem}\label{lem:p1:pull-back.under.hom.is.crsegroup}
 Let $(G,\ast_{G})$, $(H,\ast_H)$ be sets with multiplications and $\CF$ an equi left invariant (resp. equi right invariant) coarse structure on $H$. Given a function $f\colon G\to H$ where the composition $f\circ (\variable\ast_{G}\variable)$ is $\CF$\=/close to $\ast_H\circ (f\times f)$, then the pull\=/back $f^*(\CF)$ is equi left invariant (resp. equi right invariant).
\end{lem}
\begin{proof}
 Let $\bar F\in \CF$ be so that $f(x\ast_G y)\rel{\bar F} f(x)\ast_H f(y)$ for all $x,y\in G$. For every fixed $F\in\CF$, we see that if $f(y_1)\torel{F} f(y_2)$ then
 \[
  f(x\ast_G y_1)
  \rel{\bar F} f(x)\ast_H f(y_1)
  \torel{\Delta_H\ast F} f(x)\ast_H f(y_2)
  \rel{\bar F} f(x \ast_G y_2).
 \]
 This shows that $\Delta_G\ast \bigparen{(f\times f)^{-1}(F)}$ belongs to $f^*(F)$, and hence  $f^*(F)$ is equi left invariant. Equi right invariance is analogous.
\end{proof}

\begin{cor}
 Let $\crse f \colon(G,\CE)\to (H,\CF)$ be a coarse homomorphism of coarse groups. Then $f^*(\CF)$ is equi bi\=/invariant and contains $\CE$. In particular, $(G,f^*(\CF))$ is again a coarse group.
\end{cor}
\begin{proof}
 The coarse structure $f^*(\CF)$ is equi bi\=/invariant by Lemma~\ref{lem:p1:pull-back.under.hom.is.crsegroup}. Since $f$ is controlled, $\CE\subseteq f^*(\CF)$. In particular, the Group Diagrams for $\ast_{G},\unit_{G},\inversefn_{G}$ commute up to closeness when $G$ is equipped with the coarse structure $f^*(\CF)$. 
 It then follows from Proposition~\ref{prop:p1:coarse.group.iff.heartsuit.and.equi controlled} that $(G,f^*(\CF))$ is a coarse group.
\end{proof}

\begin{cor}\label{cor:p1:homomorphisms are grp_fin controlled}
 Let $G$ and $H$ be set\=/groups, $f\colon G\to H$ be a homomorphism, and $\CF$ a connected equi bi\=/invariant coarse structure on $H$.  Then $f\colon(G,\varcrs[grp]{fin})\to(H,\CF)$ is controlled and hence $\crse f$ is a coarse homomorphism.
\end{cor}
\begin{proof}
 The pull\=/back $f^*(\CF)$ is equi bi\=/invariant and every finite subset of $G$ is $f^*(\CF)$\=/bounded because $\CF$ is connected. By minimality, it follows that $\varcrs[grp]{fin}\subseteq f^*(\CF)$.
\end{proof}

The same argument also proves:

\begin{cor}\label{cor:p1:quasimorphisms are controlled}
 If $G$ is a finitely normally generated set\=/group and $\phi$ is an $\RR$\=/quasimorphism then $\phi\colon(G,\varcrs{bw})\to(\RR,\CE_{\abs{\mhyphen}})$ is controlled and hence $\crse \phi$ is a coarse homomorphism.
\end{cor}

\begin{example}
 Clearly, the converse of Lemma~\ref{lem:p1:pull-back.under.hom.is.crsegroup} does not hold. That is, the fact that $(G,f^*(\CF))$ is a coarse group does not imply that $\crse f \colon (G,f^*(\CF))\to (H,\CF)$ is a coarse homomorphism. To see this it is enough to choose two coarse groups $(G,\CE)$, $(H,\CF)$ and a coarse embedding $f\colon (G,\CE)\hookrightarrow (H,\CF)$. Then $f^*(\CF)=\CE$ and hence $(G,f^*(\CF))$ is a coarse group, but $f$ might not be a coarse homomorphism. Concretely, equip $\RR$ and $\RR^2=\CCC$ with the Euclidean metric and consider the embedding given by the logarithmic spiral 
 \[
 t\mapsto \log(1+\abs t)\exp\Bigparen{\frac{it}{\log(1+\abs{t})} }
 \]
 (or any other coarse embedding embedding that is not close to a linear embedding).
\end{example}

\begin{example}\label{exmp:p1:pull back coarse structures}
 Given two coarse groups $\crse G,\crse H$, Lemma~\ref{lem:p1:pull-back.under.hom.is.crsegroup} can be used to construct an intermediate coarse group $\crse{G_{H}}$ characterized by the property that every coarse homomorphism $\crse G\to\crse H$ factors through $\crse{G_{H}}$. Namely, if $\crse G=(G,\CE)$ and $\crse H=(H,\CF)$ we let
 \[
  \CE_{\crse G\to\crse H}\coloneqq\bigcap\Bigbraces{f^*(\CF)\bigmid \crse f\colon(G,\CE)\to(H,\CF)\text{ coarse homomorphism }}.
 \]
 Since intersections of equi bi\=/invariant coarse structures are equi bi\=/invariant, it follows from Lemma~\ref{lem:p1:pull-back.under.hom.is.crsegroup} that $\CE_{\crse G\to\crse H}$ is equi bi\=/invariant and hence $\crse{G_{H}}\coloneqq(G,\CE_{\crse G\to\crse H})$ is a coarse group. 
 By construction, $\CE\subseteq \CE_{\crse G\to\crse H}$ and therefore $\cid_{\crse G}\colon\crse G\to\crse{G_{H}}$ is a coarse homomorphism. It is also clear that $\crse f \colon\crse G\to\crse H$ is a coarse homomorphism if and only if $\crse f \colon\crse{G_{H}}\to\crse H$ is a coarse homomorphism. This shows that coarse homomorphisms factor through $\crse{G_{H}}$:
 \[
 \begin{tikzcd}
  \crse G \arrow[r,"{\crse f}"] \arrow{d}[swap]{\cid_{G}} & \crse H. \\
  \crse{G_{H}} \arrow[ur,swap,"{\crse f}"] &
 \end{tikzcd}
 \]
 
 The coarse group $\crse{G_{H}}$ can be used to quantify the lack or abundance of coarse homomorphisms from $\crse{G}$ to $\crse H$. For example, if $\crse G=(G,\mincrs)$ and $\crse H=(H,\mincrs)$ are trivially coarse groups, then $\crse{G_{H}}$ is bounded if and only if the only set\=/group homomorphism $G\to H$ is the trivial one.
 If $G$ is a finitely generated set\=/group and $H$ is the direct sum of all the finite symmetric groups $\bigoplus_{n\in\NN} S_n$ then $G$ is residually finite if and only if $\CE_{(G,\mincrs)\to(H,\mincrs)}=\mincrs$.
\end{example}

Since we know that the coarse structure of a coarse group is determined by the set of bounded neighborhoods of the identity, we can use Lemma~\ref{lem:p1:pull-back.under.hom.is.crsegroup} to prove the following.

\begin{prop}\label{prop:p1:proper.hom.coarse.embedding}
 Let $\crse{f \colon G\to H}$ be a coarse homomorphism. The following are equivalent:
 \begin{enumerate}[(i)]
  \item $\crse f$ is a coarse embedding;
  \item $\crse f$ is proper (Definition~\ref{def:p1:proper.map});
  \item $f^{-1}(U)$ is bounded for every $U\subseteq H$ bounded neighborhood of $\unit_H$.
 \end{enumerate}
\end{prop}
\begin{proof}

 Let $\crse{G}=(G,\CE)$ and $\crse{H}=(H,\CF)$. The only non\=/trivial implication is $(iii)\Rightarrow(i)$. By Lemma~\ref{lem:p1:pull-back.under.hom.is.crsegroup}, $(G,f^*(\CF))$ is a coarse group. By Proposition~\ref{prop:p1:neighbourhoods.of.identity.generate} we have
 \begin{align*}
  \CE&=\angles{\bigbraces{\Delta_{G}\ast (U\times U)\mid U\text{ is a $\CE$\=/bounded neighborhood of }\unit_{G}}\cup\mingrp} \\
  f^*(\CF)&=\angles{\bigbraces{\Delta_{G}\ast (U\times U)\mid U\text{ is a $f^*(\CF)$\=/bounded neighborhood of }\unit_{G}}\cup\mingrp}.
 \end{align*}

 Note that every $\CE$\=/bounded neighborhood of $e_G$ is $f^*(\CF)$ controlled. Vice versa, let $V$ be a $f^*(\CF)$ controlled neighborhood of $e_G$. Since $\crse f(\cunit_{\crse G})=\cunit_{\crse H}$, the union $f(V)\cup \braces{e_H}$ is a $\CF$\=/controlled neighborhood of $e_H$. Condition $(iii)$ then implies that $V$ is also $\CE$\=/bounded. It follows that $\CE=f^*(\CF)$ and hence $f$ is a coarse embedding. 
\end{proof}

The following is another useful observation that interacts nicely with Lemma~\ref{lem:p1:pull-back.under.hom.is.crsegroup}.

\begin{lem}\label{lem:p1:functions are chom}
 Given a set with mutiplication, unit and inversion functions $(G,\ast_{G},\unit_{G},\inversefn_{G})$, a coarse group $\crse H =(H,\CF,[\ast_H],[\unit_H],[\inversefn_H])$ and a function $f\colon G\to H$, the following are equivalent:
 \begin{enumerate}[(i)]
  \item The composition $f\circ (\variable\ast_{G}\variable)$ is close to $\ast_H\circ (f\times f)$; the composition $f\circ \inversefn_{G}$ is close to $\inversefn_H\circ f$; and $f(e_{G})$ is close to $e_H$.
  \item The function $f\colon (G,\mingrp)\to(H,\CF)$ is controlled and $\crse f$ is a coarse homomorphism.  
 \end{enumerate}
\end{lem}
\begin{proof}
 It is clear that $(ii)$ implies $(i)$: the first closeness condition is included in the definition of coarse homomorphism and the other two are a consequence of Lemma~\ref{lem:p1:coarse homomorphisms preserve unit and inverse}.
 For the other implication, we know by Lemma~\ref{lem:p1:generators of mingrp} that 
 \( 
 \mingrp=\angles{{\,\Delta_{G}\ast \assRel\, ,\; (\Delta_{G}\ast \invRel)\ast\Delta
 _{G} \, ,\;
 \Delta_{G}\ast \idRel\, ,\; \idRel \, }}
 \)
 where $\assRel,\invRel,\idRel$ are the relations defined by
 \[\def\arraystretch{1.5}
 \begin{array}{cc}
  g_1 \ast (g_2\ast g_3) \rel{\assRel} (g_1\ast g_2)\ast g_3 
  & \qquad \forall g_1,g_2,g_3\in G
   \\
   e\ast g \rel{\idRel} g\rel{\idRel} g\ast e 
   & \qquad \forall g\in G
  \\
  g\ast g^{-1}\rel{\invRel} e\rel{\invRel} g^{-1}\ast g 
  & \qquad \forall g\in G.
 \end{array}
 \]
 It is thus enough to check that $f\times f$ sends those generating relations into $\CF$. For convenience, we will only show that $f\times f(\assRel)\in\CF$: checking $\Delta_{G}\ast_{G} \assRel$ and the other generators of $\mingrp$ is just as easy but notationally awkward. 

Let $F_{\rm as}$ be the analog of $\assRel$ for $H$.
By assumption, there is a controlled set $F\in\CF$ such that $f(g_1\ast_{G} g_2)\rel{F} f(g_1)\ast_H f(g_2)$ for every $g_1,g_2\in G$. We then see that
\begin{align*}
  f(g_1\ast_{G}(g_2\ast_{G} g_3)) &\rel{F} f(g_1)\ast_H f(g_2\ast_{G} g_3) \\
    &\rel{\Delta_H\ast_H F} f(g_1)\ast_H (f(g_2)\ast_H f(g_3))  \\
    &\rel{F_{\rm as}} (f(g_1)\ast_H f(g_2) )\ast_H f(g_3)\\
    &\rel{ ({F}\ast_H\Delta_H) \cmp {F}} f((g_1\ast_{G} g_2)\ast_{G} g_3).
\end{align*}
This shows that $f\times f(\assRel)\in\CF$.
\end{proof}

\begin{cor}\label{cor:p1:functions are chom}
 Let $G$ and $\crse H=(H,\CF)$  be a set and a coarse group as in Lemma~\ref{lem:p1:functions are chom}, and let $f\colon G\to H$ be a function satisfying item (i) of the same lemma. Then the pull\=/back $f^\ast(\CF)$ is equi bi\=/invariant on $(G,\ast_{G})$.
\end{cor}
\begin{proof}
 By Lemma~\ref{lem:p1:functions are chom} $\crse f \colon(G,\mingrp)\to(H,\CF)$ is a coarse homomorphism and we can hence apply Lemma~\ref{lem:p1:pull-back.under.hom.is.crsegroup}.
\end{proof}

We can combine the results described thus far to prove a criterion that makes it easier to check whether a function defines a coarse homomorphism:

\begin{prop}\label{prop:p1:bornologous map is coarse hom}
 Let $\crse G=(G,\CE,[\ast_{G}],[\unit_{G}],[\inversefn_{G}])$ and $\crse H= (H,\CF,[\ast_H],[\unit_H],[\inversefn_H])$ be coarse groups. If a function $f\colon G\to H$ satisfies
 \begin{enumerate}[(i)]
  \item the composition $f\circ (\variable\ast_{G}\variable)$ is close to $\ast_H\circ (f\times f)$,
  \item $f(U)\ceq_\CF e_H$ for every $\CE$\=/bounded neighborhood $U$ of $\unit_G\in G$,
 \end{enumerate}
 then $f$ is controlled and $\crse f \colon \crse G\to \crse H$ is a coarse homomorphism.
\end{prop}
\begin{proof}
 We claim that the composition $f\circ \inversefn_{G}$ is close to $\inversefn_H\circ f$. Since $\crse G$ is a coarse group, there is a bounded neighborhood $U$ of $e_{G}$ so that $g\ast_{G} g^{-1}\in U$ for every $g\in G$. By the hypotheses on $f$, there are controlled sets $F_1,F_2\in \CF$ such that
 \[
  e_H\torel{F_1}f(g_1\ast g_1^{-1}) \quad\text{and}\quad f(g_2\ast_{G} g_3)\torel{F_2} f(g_2)\ast_H f(g_3)
 \]
 for every $g_1,g_2,g_3\in G$. Hence for every $g\in G$
 \[
  e_H\torel{F_1}f(g\ast g^{-1})\torel{F_2} f(g)\ast_H f(g^{-1}).
 \]
 Since $H$ is a coarse group, the product relation $(\Delta_H)\ast_H(F_1\cmp F_2)$ is in $\CF$. Therefore we have
 \[
  f(g)^{-1}
  \rel{\CF} f(g)^{-1}\ast_H \bigparen{f(g)\ast_H f(g^{-1}) }
  \rel{\CF} f(g^{-1})
  \qquad \forall g\in G,
 \]
 which proves our claim (here we are using Convention~\ref{conv:p1:x_torel_CE}).

 We now see that $f$ satisfies condition $(i)$ of Lemma~\ref{lem:p1:functions are chom}. By Corollary~\ref{cor:p1:functions are chom}, the pull\=/back $f^*(\CF)$ is an equi bi\=/invariant coarse structure on $G$. By hypothesis, we know that $\CE$\=/bounded neighborhoods of $e_{G}$ are also $f^*(\CF)$\=/bounded. It then follows from Proposition~\ref{prop:p1:neighbourhoods.of.identity.generate} that $\CE\subseteq f^*(\CF)$ or, in other words, that $f\colon (G,\CE)\to (H,\CF)$ is controlled.
\end{proof}

\section{Coarse Subgroups}
\label{sec:p1:coarse subgroups}

Recall that in Section~\ref{sec:p1:subspaces and quotients} we defined coarse subspaces of $\crse X=(X,\CE)$ as equivalence classes of subsets of $X$, where two subsets are equivalent if they are $\CE$\=/asymptotic. We also defined notions of coarse containment and image of coarse subsets under coarse maps. We define coarse subgroups of a coarse group $\crse G$ as those coarse subspaces that are coarsely closed under the coarse group operations:

\begin{de}\label{def:p1:coarse subgroup}
 A \emph{coarse subgroup}\index{coarse!subgroup} of a coarse group $\crse G$ is a coarse subspace $\crse H \crse{\subseteq G}$ such that $\crse{( H\cop_{G} H) \subseteq H}$ and $\crse{H^{-1}\subseteq H}$. We denote coarse subgroups by $\crse{H \leq G}$.\nomenclature[:COS]{$\crse{H \leq G}$}{coarse subgroup } 
\end{de}

Fixing representatives, a subset $H\subseteq G$ determines a coarse subgroup of $(G,\CE)$ if and only if $H\ast_{G} H\preccurlyeq_\CE H$ and $H^{-1}\preccurlyeq_\CE H$.

In Section~\ref{sec:p1:subspaces and quotients} we explained that---up to coarse equivalence---a coarse subspace of $\crse{Y\subseteq X}$ uniquely determines a coarse space $\crse Y$. The following proposition shows that the same is true for coarse subgroups. 
Namely, a coarse subgroup $\crse{H\leq G}$ determines a coarse group $\crse H$ uniquely defined up to a natural isomorphism.

\begin{prop}\label{prop:p1:closed.asymp.classes.are.subgroups}
 Let $\crse{H\leq G}$ be a coarse subgroup of $\crse G=(G,\CE)$ and $H\subseteq G$ a representative for the coarse subspace $\crse{H\subseteq G}$. The coarse space $\crse H=(H,\CE|_H)$ admits a unique choice of coarse group operations such that the embedding $\crse{H\hookrightarrow G}$ is a coarse homomorphism.
 If $H'\subseteq G$ is a different choice of representative for $\crse{H\subseteq G}$ and we let $\crse{H'}=(H',\CE|_{H'})$, then the canonical coarse equivalence $\crse{i\colon H\to H'}$ is an isomorphism of coarse groups.
\end{prop}

\begin{proof}
Fix the representative $H$ and let $\crse H$ be the coarse space $(H,\CE|_H)$. Since the inclusion $\crse\iota\colon \crse{H}\to \crse G$ is a coarse embedding, it is clear that if there exists a coarse multiplication $\cop_{\crse H}\colon \crse H\times \crse H\to \crse H$ such that the following diagram commutes 
 \begin{equation}\label{eq:p1:subgroup diagram}
  \begin{tikzcd} 
    \crse{H}\times \crse H \arrow[r, "\cop_{H}"] 
    \arrow[d,swap, "\crse{\iota}\times \crse \iota"] & \crse{H}\arrow[d, "\crse{\iota}"]\\
    \crse{G}\times \crse G\arrow[r,"\cop_{\crse G}"]& \crse G 
    \end{tikzcd} 
 \end{equation}
 then $\cop_{\crse H}$ is unique. 

 To construct $\cop_{\crse H}$, let $\fka_{\rm mult}$ be a controlled covering of $\crse G$ such that $H\ast_{G} H\subseteq \st(H,\fka_{\rm mult})$. For every pair $(h_1,h_2)\in H\times H$ choose an $\bar h\in H$ such that $h_1\ast_{G} h_2\in \st(\bar h,\fka_{\rm mult})$ and let $h_1\ast_H h_2\coloneqq \bar h$. It is easy to check by hand that $\ast_H\colon (H,\CE|_H)\times(H,\CE|_H)\to (H,\CE|_H)$ is a controlled map. Alternatively, we can also note that the composition $\iota\circ \ast_H\colon H\times H\to G$ is close to $\ast_{G}\circ (\iota\times\iota)$ to deduce that $\iota\circ \ast_H$ is controlled and hence conclude that $\ast_H$ is controlled as well (because $\iota$ is a coarse embedding). We may thus let $\cop_{\crse H}=[\ast_H]$ and note that Diagram \eqref{eq:p1:subgroup diagram} commutes by construction.
 
 Diagram~\eqref{eq:p1:subgroup diagram} allows us to deduce associativity of $\cop_{\crse H}$ from the associativity of $\cop_{\crse G}$. 
 Namely, we have the following equality of coarse maps $\crse{H}\times\crse H\times \crse H\to \crse G$:  
 \begin{align*}
  \crse\iota\circ \cop_{\crse H}\circ \paren{\cid_{\crse H}\times \cop_{\crse H} }
  &= \cop_{\crse G}\circ\paren{\cid_{\crse G}\times \cop_{\crse G}}\circ\paren{\crse \iota\times\crse \iota\times\crse \iota}  \\
  &=\cop_{\crse G}\circ\paren{\cop_{\crse G}\times \cid_{\crse G}}\circ\paren{\crse \iota\times\crse \iota\times\crse \iota} \\
  &=\crse \iota\circ \cop_{\crse H}\circ \paren{\cop_{\crse H}\times\cid_{\crse H} }.
 \end{align*}
 
 Alternatively, one can also show explicitly that the controlled maps $\iota\circ \ast_H\circ \paren{\id_H\times \ast_H}$ and $\iota\circ \ast_H\circ \paren{\ast_H\times\id_H}$ are close by making a judicious use of relations or star\=/neighborhoods. Since $\crse{H}=(H,\CE|_H)$ is equipped with the subspace coarse structure, it follows immediately that $\ast_H\circ \paren{\id_H\times \ast_H}$ and $\ast_H\circ \paren{\ast_H\times\id_H}$ are close functions from $H\times H\times H$ to $(H,\CE|_H)$.

 Similarly, let $\inversefn_G$ be a representative for the inversion coarse map of $\crse G$. If $\fka_{\rm inv}$ is a controlled covering of $G$  such that $(H)^{-1}_{G}\subseteq \st(H,\fka_{\rm inv})$, we can define $(h)^{-1}_H$ by choosing a $\bar h\in H$ such that $(h)^{-1}_{G}\in\st(\bar h,\fka_{\rm inv})$. Again, $\inversefn_H$ is controlled and $\iota\circ\inversefn_H$ is close to $\inversefn_{G}|_H$. We can further let $e_H\in H$ be any point in the same coarsely connected component of $e_G\in G$, for example $e_H\coloneqq  h\ast_H (h)_H^{-1}$ for an arbitrarily chosen $h\in H$. An argument analogous to the above shows that the coarse functions $\cop_{\crse H}$, $\cunit_{\crse H}$ and $\cinversefn_{\crse H}$ also satisfy the (Identity) and (Inverse) Group Diagrams, and thus make $\crse{H}$ into a coarse group. The inclusion $\crse \iota\colon\crse H\to\crse G$ is a coarse homomorphism by construction.

 Let now $H'\asymp H$ be another representative for the coarse subspace, and let $\crse{H'}$ be the coarse group obtained by equipping $(H',\CE_{H'})$ with the unique coarse operations so that the inclusion $\crse{\iota'\colon H'\hookrightarrow G}$ is a coarse homomorphism.
 Since $H$ and $H'$ are asymptotic, we know that there is a canonical coarse equivalence $\crse{i\colon H\to H'}$ such that 
 \[
  \begin{tikzcd} 
    \crse{H} \arrow[r, "{\crse \iota}"] \arrow[d,swap, "\crse{i}"] & \crse{G}.\\
    \crse{H'}\arrow[ur,swap,"{\crse{\iota'}}"]&  
  \end{tikzcd}   
 \]
 One may easily check by hand that $\crse i$ commutes with the coarse operations. Alternatively, we may observe that the following diagram commutes
 \[
  \begin{tikzcd} 
    \crse{H}\times \crse H \arrow[dr,swap, "{\crse{\iota\times\iota}}"] \arrow[r, "\crse{ i\times i}"] 
    &\crse{H' \times H'} \arrow[r, "\cop_{\crse{H'}}"] \arrow[d, "\crse{\iota' \times \iota'}"] 
    & \crse{H'}\arrow[d,swap, "\crse{\iota'}"]\arrow[r, "\crse{i^{-1}}"]
    & \crse{H}\arrow[dl, "\crse{\iota}"]
    \\
    &\crse{G}\times \crse G\arrow[r,"\cop_{\crse G}"]
    & \crse G
    &
   \end{tikzcd}   
 \]
 and hence conclude that the composition $\crse{i^{-1}}\circ\cop_{\crse{H'}}\circ (\crse{ i\times i})$ equals $\cop_{\crse H}$ because this is the unique coarse map that is compatible with $\cop_{\crse G}$.
\end{proof}

Similarly to coarse subspaces, we use the convention that $\crse H$ denotes both a coarse subgroup $\crse{H\leq G}$ (\emph{i.e.}\ an equivalence class of subsets of $G$) and a coarse group $\crse H$. We may also write $\crse H=(H,\CE|_H)$ to mean that $H\subseteq G$ is a choice of representative for the coarse subgroup $\crse H$. Thanks to Proposition~\ref{prop:p1:closed.asymp.classes.are.subgroups}, this ambiguity will not cause any issue. 

 Proposition~\ref{prop:p1:closed.asymp.classes.are.subgroups} of $\crse G$ also implies that coarse subgroups can be alternatively described as the coarse images of coarse homomorphisms into $\crse G$:

\begin{cor}\label{cor:p1:coarse subgroup iff coarse image}
 The set of coarse subgroups of $\crse G$ coincides with set of coarse images of coarse homomorphisms into $\crse G$.
\end{cor}
\begin{proof}
By Proposition~\ref{prop:p1:closed.asymp.classes.are.subgroups} a coarse subgroup $\crse{H\leq G}$ is the coarse image of a coarse homomorphism (the inclusion $\crse{ H\hookrightarrow G}$). 
 Vice versa, if $ \crse{f\colon H \to  G}$ is a homomorphism of coarse groups then, by definition, $\crse{f}\circ \cop_{\crse H}$ coincides with $\cop_{\crse G}\circ(\crse f\times \crse f)$ as coarse maps from $\crse{H\times H}$ to $\crse G$. 
 It follows that 
 \[
  \crse{f(H)\cop_{G} f(H)=\cop_{G}\circ(f\times f)(H\times H)=f(H\cop_{H} H)=f(H)}.
 \]
 That is, the coarse image $\cim(\crse{f})$ is coarsely closed under multiplication. The coarse containment $\crse{f(H)^{-1}\subseteq f(H)}$ is analogous.
\end{proof}

 \begin{exmp}
  Consider $\ZZ^2$ with the Euclidean metric $\norm{\mhyphen}_2$. Fix some parameter $\theta\in \RR$ and let 
  \[
   L_\theta\coloneqq\braces{(n,m)\in\ZZ^2\mid \abs{\theta n-m}\leq 1}.
  \]
  It is immediate to observe that $[L_\theta]$ is a coarse subgroup of $(\ZZ^2,\CE_{\norm{\mhyphen}_2})$ and that $L_\theta$ and $L_{\theta'}$ are asymptotic if and only if $\theta=\theta'$.  In particular, we see that for every irrational value $\theta\in \RR\smallsetminus \QQ$ the coarse subgroup $L_\theta$ is not close to any set\=/subgroup of $\ZZ^2$.
 \end{exmp}

The following observation is simple, but useful for producing examples of coarse subgroups:

\begin{lem}\label{lem:p1:subgroup iff in m.HA}
 Let $\crse G$ be a coarse group and $H\subset G$ some subset. If there exists a bounded neighborhood of the identity $e\in A\subseteq G$ such that both $H\ast H$ and $H^{-1}$ are contained in $H\ast A$, then $\crse H$ is a coarse subgroup. If $G$ is a set\=/group, the converse is also true.
\end{lem}

\begin{proof}
 Since $\ast$ is controlled $H\ast A\ceq H\ast e$. By the (Identity) Group Diagram we see that
 \(
  H\ast e \ceq H
 \). It follows that $H\ast H\csub H$ and $H^{-1}\csub H$, and therefore $\crse H$ is a coarse subgroup.
 
It follows easily from Corollary~\ref{cor:p1:neighbourhoods.of.identity.generate_setgroup_part cover} that for set\=/groups the converse is also true: any bounded neighborhood of $H$ must be contained in a set of the form $\st(H,\pts{G}\ast U)$ with $U\in \fkU_e(\CE)$, and it simple to verify that
 \(
  \st(H,\pts{G}\ast U) = H\ast(U^{-1}U).
 \)
\end{proof}

Recall that a subset $A$ of a set\=/group $G$ is a \emph{$k$-approximate subgroup} for some $k\in\NN$ if $A = A^{-1}$ and $A A \subseteq K A$, where $K\subseteq G$ has cardinality at most $k$. That is, $A$ is symmetric and the product of $A$ with itself is covered by $k$ translates of $A$.

\begin{cor}
 If $\crse G$ is a coarsified set\=/group, then every approximate subgroup $H\subseteq G$ determines a coarse subgroup $\crse{H\leq G}$.
\end{cor}

\begin{remark}
 We refer to the recent monograph \cite{cordes2020foundations} for interesting examples of approximate subgroups of various set\=/groups. We should remark that their study of ``geometric approximate group theory'' is only tangentially related with our notion of coarse groups. 
 This is because they are mostly interested in left-invariant metrics on approximate sub\=/groups. Using such metrics, the multiplication function need not be controlled.  However, an infinite approximate subgroup $H$ of a set\=/groups $G$ equipped with \emph{bi-invariant} metrics $d$ is indeed a coarse subgroup of $(G,\CE_d)$.
\end{remark}

We conclude this section by proving that, when it exists, the coarse intersection of coarse subgroups is a coarse subgroup.

\begin{lem}\label{lem:p1:intersection of subgroups}
 If $\crse H_1,\ldots,\crse H_n$ are coarse subgroups of $\crse G$ and the coarse intersection $\crse Z\coloneqq \crse H_1 \crse\cap \cdots \crse\cap \crse H_n$ (Definition~\ref{def:p1:intersection}) exists, then $\crse Z$ is coarse subgroup of $\crse G$ and of each $\crse H_i$ for $i=1,\ldots,n $.
\end{lem}
\begin{proof}
 Let $Z$ and $H_i$ be representatives of $\crse Z$ and $\crse H_i$ respectively. By hypothesis we have $Z\csub H_i$ for all $i=1,\ldots, n$. Since $Z\times Z\csub H_i\times H_i$ and $\ast $ is controlled, it follows from Lemma~\ref{lem:p1:controlled maps and asymptotic subsets} that $Z\ast Z\csub H_i\ast H_i$ and the latter is contained in a controlled thickening of $ H_i$ because it is a coarse subgroup. As this is true for every $i=1,\ldots,n$, it follows by the definition of coarse intersection that $Z\ast Z\csub Z$.
 
 Similar arguments show that $Z^{-1}$ is coarsely contained in $Z$ as well. This proves that $\crse Z$ is indeed a coarse subgroup of $\crse G$. As explained in Section~\ref{sec:p1:containements and intersections}, once $\crse H_i$ is realized as a coarse space (and hence a coarse group), the coarse subspace $\crse Z\subseteq \crse G$ uniquely determines a coarse subspace of $\crse H_i$. It is immediate that this is a coarse subgroup of $\crse H_i$.
\end{proof}

We already remarked that the coarse intersection of two coarse subspaces may not exist (Example~\ref{exmp:p1:no intersection}). However, one may wonder whether the intersection of coarse subgroups always exists. The following example shows that this is not the case.

\begin{exmp}\label{exmp:p1:coarse subgroups no intersection}
 Let $\crse G\coloneqq (\ell_2(\NN),\CE_{\norm{\mhyphen}_2})$, where $\norm{\mhyphen}_2$ is the standard $\ell_2$\=/norm and let $v_0,v_1,\ldots$ be the standard orthonormal basis. Let $w_n\coloneqq v_{2n}+\frac{ 1}{n}v_{2n+1}$ and define the vector subspaces 
 \[
 H\coloneqq\overline{\angles{w_n\mid n\in\NN}_\RR}
 \quad \text{ and } \quad 
 K\coloneqq\overline{\angles{v_{2n}\mid n\in\NN}_\RR}.
 \]
 Then $\crse H$ and $\crse K$ are coarse subgroups of $\crse G$ and we claim that their intersection is not well defined.
 
 In fact, for any fixed $C>0$ the intersection of $H$ and the $C$\=/neighborhood of $K$ in $\ell_2(\NN)$ is the set
 \[
  I_C\coloneqq \left\{\sum_{n\in\NN}a_nw_n\ \middle|\ a_n\in\RR,\ \sum_{n\in\NN}a_n^2<\infty,\ \sum_{n\in\NN}\frac{a_n^2}{n^2} \leq C^2\right\}\subset \ell_2(\NN).
 \]
 To prove the claim it is enough to note that $I_{C+1}$ is not contained in a bounded neighborhood of $I_C$: for each $n\geq 1$ the former contains the vector $n(C+1)w_n$, and its nearest\=/point projection in $I_C$ is $nCw_n$ which is at distance $n\norm{w_n}$ from it.
\end{exmp}

\section{Coarse Quotients}\label{sec:p1:coarse quotients}

 Recall from Section~\ref{sec:p1:subspaces and quotients} that quotients in the category of coarse spaces are obtained by enlarging the coarse structures. We take the same approach for the definition of coarse quotient groups:
 
 \begin{de}\label{def:p1:coarse quotient group}
  A \emph{coarse quotient}\index{quotient!coarse group}\index{coarse group!quotient}  of a coarse group $\crse G=(G,\CE,[\ast]_{\CE},[\unit]_{\CE},[\inversefn]_{\CE})$ is a coarse group $\overline{\crse G}=(G,\CE',[\ast]_{\CE'},[\unit]_{\CE'},[\inversefn])_{\CE'}$ obtained by equipping the same set $G$ with a coarser equi bi\=/invariant coarse structure $\CE'\supseteq \CE$. 
 \end{de} 
 
 \begin{rmk}
 In the above definition we used the notation $[\ast]_\CE$ to emphasize that we are using the same set\=/functions to define the operations in the coarse quotient group, we have just enlarged the coarse structure. We decorated $[\mhyphen]$ with the coarse structure to draw attention to the fact that the closeness equivalence relation on functions depends on the coarse structure.
 Hereafter we will suppress the coarse structure decoration from the group operations.  
 \end{rmk}

 \begin{exmp}
  For every set\=/group $G$, the coarse group $(G,\varcrs[grp]{fin})$ is a coarse quotient of the trivially coarse group $(G,\mincrs)$. If $G$ is equipped with a bi\=/invariant metric $d$, then $(G,\CE_d)$ is a coarse quotient of both $(G,\varcrs[grp]{fin})$ and $(G,\mincrs)$.
 \end{exmp}

 \begin{exmp}
  Let $\crse G$ and $\crse H$ be coarse groups. Then the coarse group $\crse{G_{H}}$ from Example~\ref{exmp:p1:pull back coarse structures} is a coarse quotient of $\crse G$.
 \end{exmp}

If $\crse Q=(Q,\CF)$ is a coarse group and $\crse f \colon \crse G\to\crse Q$ is a coarsely surjective coarse homomorphism, then $\overline{\crse G}\coloneqq (G,f^*(\CF))$ is a coarse group by Lemma~\ref{lem:p1:pull-back.under.hom.is.crsegroup} and hence it is a coarse quotient of $\crse G$. Further, $\crse f$ descends to a natural coarse homomorphism $\overline{\crse f}=[f]\colon \overline{\crse{G}}\to\crse Q$ which is an isomorphism because $\crse f$ is coarsely surjective (Corollary~\ref{cor:p1:pullback.gives.isomorphism}). Thus whenever we are given a coarsely surjective coarse homomorphism $\crse G \to \crse Q$ we are entitled to say that $\crse Q$ is a coarse quotient of $\crse G$.

\begin{rmk}
 In some sense, this approach to the definition of coarse quotient groups is analogous to the standard way quotients of set\=/groups are introduced. It is common to define the quotients of a set\=/group $G$ by fixing a normal subgroup $N\unlhd G$ and noting that the set of cosets $G/N$ admits a group operation. In a second stage one remarks that if $f\colon G\to Q$ is a surjective homomorphism then, by the First Isomorphism Theorem, $Q$ is naturally isomorphic to the quotient $G/\ker(f)$. For this reason one is entitled to call ``quotient'' any surjective image of a set\=/group homomorphism.
 
 Similarly, Definition~\ref{def:p1:coarse quotient group} provides us with a concrete coarse group (set with a coarse structure and coarse operations), just as $G/N$ provided us with a concrete set\=/group. We then observe that every coarsely surjective image of a coarse homomorphism is naturally identified with a coarse quotient. In this sense, Lemma~\ref{lem:p1:pull-back.under.hom.is.crsegroup} can be seen as a coarse analogue of the First Isomorphism Theorem. 
 
 In Section~\ref{sec:p1:cosets of subgroups} we will give a description of coarse quotient groups that is more reminiscent of coset spaces, while in Section~\ref{sec:p1:isomorphism theorems} we will prove some isomorphism theorems which show that coarse quotients by coarsely normal coarse subgroups behave as expected.
\end{rmk}

 The following result is almost a tautology, but it is good to write it explicitly:

 \begin{prop}\label{prop:p1:quotient.subgroup}
  Let $\crse G=(G,\CE)$ be a coarse group and $\CR$ a family of relations on the set $G$. The minimal equi bi\=/invariant coarse structure $\varcrs[grp]{\CE, \CR}$ such that $\CE\cup\CR\subseteq\varcrs[grp]{\CE, \CR}$ is equal to
  \begin{equation}\label{eq:p1:minimal.crse.structure}
   \varcrs[grp]{\CE, \CR}=\angles{(\Delta_{G}\ast\CR)\ast\Delta_{G} ,\CE}=\angles{\Delta_{G}\ast(\CR\ast\Delta_{G} ),\CE}.   
  \end{equation}
  The coarse quotient group $\crse G/\aangles{\CR}\coloneqq (G,\varcrs[grp]{\CE, \CR})$ has the property that a coarse homomorphism $\crse{f}\colon \crse G\to \crse H=(H,\CF)$ descends to $\crse G/\aangles{\CR}$ if and only if $f\times f(R)\in\CF$ for every $R\in\CR$.
 \end{prop}

 \begin{proof}
  By Lemma~\ref{lem:p1:generated biinvariant coarse structure} we know that 
  \[
   \CE=\angles{(\Delta_{G}\ast\CE)\ast\Delta_{G} \;,\; \mingrp} 
   \quad \text{ and }\quad 
   \varcrs[grp]{\CE, \CR}=\angles{(\Delta_{G}\ast(\CR\cup\CE))\ast\Delta_{G} \;,\; \mingrp}.
  \]
  It is then simple to observe that 
  \[
   \angles{(\Delta_{G}\ast\CR)\ast\Delta_{G} ,\CE}=\angles{(\Delta_{G}\ast\CR)\ast\Delta_{G} ,(\Delta_{G}\ast\CE)\ast\Delta_{G} ,\mingrp}=\angles{(\Delta_{G}\ast(\CR\cup\CE))\ast\Delta_{G} ,\mingrp}.
  \]
  The proof that $\varcrs[grp]{\CE, \CR}=\angles{\Delta_{G}\ast(\CR\ast\Delta_{G} ),\CE}$ is identical.
  
  For the second part of the statement, let $\crse{f} \colon \crse G\to \crse H$ be a coarse homomorphism. If $f\times f(R)\in\CF$ for every $R\in\CR$, then $\CR\cup\CE\subseteq f^*(\CF)$ and hence $\varcrs[grp]{\CE, \CR}\subseteq f^*(\CF)$ by Lemma~\ref{lem:p1:pull-back.under.hom.is.crsegroup} and minimality of $\varcrs[grp]{\CE, \CR}$. It follows that $f\colon (G,\varcrs[grp]{\CE, \CR})\to (H,\CF)$ is controlled and hence yields a coarse homomorphism. On the other hand, if $\crse f$ descends to a controlled map then it must send the relations $R\in\CR$ into $\CF$.
 \end{proof}

 \begin{de}\label{de:p1:quotient by a relation}
  Given a coarse group $\crse{G}=(G,\CE)$ and any set $\CR$ of relations on $G$, we say that $\crse G/\aangles{\CR}\coloneqq (G,\varcrs[grp]{\CE, \CR})$\nomenclature[:COS]{$\crse G/\aangles{\CR}\coloneqq (G,\varcrs[grp]{\CE, \CR})$}{quotient of coarse group by relations} is the \emph{quotient of $\crse G$ by $\CR$}.\index{quotient!by relations (coarse groups)} 
 \end{de}

The quotient \emph{coarse group} $\crse G/\aangles{\CR}$ is not to be confused with the quotient \emph{coarse space} $\crse G/\CR$ of Definition~\ref{def:p1:coarse quotient space}. Rather, the above definition has a strong analogy with presentations of set\=/groups. The coarse group $\crse{G}/\aangles{\CR}$ can be thought of as the coarse group obtained by adding the relations $\CR$ to ``a presentation'' of $\crse G$. 
 One case of interest is that of quotients of the form  $\crse G/\aangles{\braces{A\times A}}$ where $A\subseteq G$ is some arbitrary subset. Loosely speaking, this is the quotient of $\crse G$ by a sort of ``coarse normal closure of $A$''. Note however that there is an important difference between this notion of quotient and actual presentations of set\=/groups: namely, if $G$ is a set\=/group and $A\subset G$ is some subset, the set\=/group obtained by adding $A$ to a presentation of $G$ is the set\=/group quotient $G/\aangles{A}$ where the normal closure $\aangles{A}$ is the \emph{subgroup} consisting of \emph{all} finite products of conjugates of elements in $A$ and their inverses. In particular, $\aangles{A}$ itself becomes trivial in the quotient. This is not the case for coarse quotients. Namely, even though each finite product of elements in $A$ will be at bounded distance from $\unit$ in the coarse quotient $\crse G/\aangles{\{A\times A\}}$, this bound is not uniform and hence it may be that the set of all the finite products of elements of $A$ is \emph{not} bounded. For instance, consider $\crse G= (\RR,\CE_{\abs{\mhyphen}})$ and let $A=B(0;1)$ be the unit ball around $0$. Since $A$ generates $\RR$ as a set\=/group, the quotient $\RR/\aangles{A}$ is trivial. However, $A$ is bounded and hence the coarse quotient $\crse G/\aangles{\{A\times A\}}$ is still $\crse G$! Similar observations will also play a role when discussing quotient coarse actions and coarse coset spaces (Subsections~\ref{sec:p1:quotient coarse action}--\ref{sec:p1:cosets of subgroups}).
 
 The above example of a coarse quotient behaving very differently from a set\=/group quotient may be somewhat unsatisfying. In fact, we use that $A\ceq \{\unit\}$ to deduce that the coarse quotient $\crse G/\aangles{\{A\times A\}}$ is simply $\crse{G}$. This raises the question whether it should always be possible to find a (coarse) subgroup $\crse H \crse{\leq G}$ so that $\crse G/\aangles{\{A\times A\}} = \crse G/\aangles{\{H\times H\}}$. The next example shows that this is not the case (this is related to the fact that coarse homomorphisms need not have a coarse kernel, see Chapter~\ref{ch:p1:coarse kernels}).

 \begin{exmp}
  Consider the trivially coarse group $(\RR,\mincrs)$ and let $A\coloneqq B(0;1)\subseteq \RR$ be the unit ball of centered at the origin $0\in\RR$. It is not hard to check that $(\RR,\mincrs)/\aangles{A\times A}$ coincides with the metric coarse group $(\RR,\CE_{\abs{\mhyphen}})$. In fact, since $A$ has bounded diameter it is bounded in $(\RR,\CE_{\abs{\mhyphen}})$, we see by the properties of the coarse quotient that the identity function $(\RR,\mincrs)/\aangles{A\times A}\to (\RR,\CE_{\abs{\mhyphen}})$ is controlled. On the other hand, every $\CE_{\abs{\mhyphen}}$\=/bounded subset of $\RR$ can be covered by finitely many translates of $A$, and it is therefore bounded in $(\RR,\mincrs)/\aangles{A\times A}$. This shows that the identity function is controlled and proper, and hence defines an isomorphism by Proposition~\ref{prop:p1:proper.hom.coarse.embedding}.
  
  Notice however that we cannot realize $(\RR,\CE_{\abs{\mhyphen}})$ as a coarse quotient of $(\RR,\mincrs)$ by a (coarse) subgroup, because every non\=/trivial subgroup of the set\=/group $\RR$ is unbounded in $(\RR,\CE_{\abs{\mhyphen}})$.
 \end{exmp}
 
 The following example is slightly more elaborate, and shows that the above phenomenon is not special to trivially coarse groups.

\begin{exmp}\label{exmp:p1:Hilbert.space.monomorphism}
 Consider the Hilbert space $\ell^2(\NN)$ with the coarse structure $\CE_{\norm{\mhyphen}_2}$ induced by the natural $\ell^2$\=/distance. Let $v_1,v_2,\ldots$ be the natural orthonormal basis of $\ell^2(\NN)$ and let $\norm{\mhyphen}_2'$ be the (non\=/complete) norm on $\ell^2(\NN)$ obtained by rescaling the $\ell_2$\=/norm in such a way that $\norm{v_n}_2'=\frac{1}{n}$ for every $n\in\NN$. Let $\CE_{\norm{\mhyphen}_2'}$ be the coarse structure induced by $\norm{\mhyphen}_2'$, then both $(\ell^2(\NN),\CE_{\norm{\mhyphen}_2})$ and $(\ell^2(\NN),\CE_{\norm{\mhyphen}_2'})$ are abelian metric coarse groups. 
 
 The identity map $\id_{\norm{\mhyphen}_2\to \norm{\mhyphen}_2'}\colon(\ell^2(\NN),\CE_{\norm{\mhyphen}_2})\to(\ell^2(\NN),\CE_{\norm{\mhyphen}_2'})$ is a controlled homomorphism 
 and hence $[\id_{\norm{\mhyphen}_2\to \norm{\mhyphen}_2'}]$ is a coarse homomorphism. As before, we let $A\coloneqq B_{\norm{\mhyphen}_2'}(0;1) \subset \ell^2(\NN)$ be the unit ball for the norm $\norm{\mhyphen}_2'$ and we claim that $(\ell^2(\NN),\CE_{\norm{\mhyphen}_2})/\aangles{\{A\times A\}}$ coincides $(\ell_2(\NN),\CE_{\norm{\mhyphen}_2'})$. 
 To show this it is enough to observe that every $\norm{\mhyphen}_2'$\=/bounded set is contained in finitely many iterated star\=/neighborhoods of the origin by the controlled covering $\pts{\ell^2(\NN)}\ast A$.
 
 We also claim that there is no coarse subgroup $\crse H \crse{\leq }(\ell^2(\NN),\CE_{\norm{\mhyphen}_2})$ such that 
 \[
   (\ell^2(\NN),\CE_{\norm{\mhyphen}_2})/\aangles{\{A\times A\}} = (\ell^2(\NN),\CE_{\norm{\mhyphen}_2})/\aangles{\{H\times H\}}.
 \]
 This is once again due to the fact that every non\=/trivial coarse subgroup of $(\ell^2(\NN),\CE_{\norm{\mhyphen}_2})$ is unbounded also according to the norm $\norm{\mhyphen}_2'$. In fact, by Corollary~\ref{cor:p1:coarse subgroup iff coarse image} we know that coarse subgroups are coarse images of coarse homomorphisms. If $\crse{f\colon G}\to (\ell^2(\NN),\CE_{\norm{\mhyphen}_2})$ is a coarse homomorphism, we can apply Proposition~\ref{prop:p1:homogeneization.of.qmorph} to choose a representative so that $f(g\ast g)=2f(g)$ for every $g\in G$. Using such representatives, we see that if $w\in f(G)\smallsetminus \{0\}$ then so does the sequence $2^nw$ and hence $f(G)$ cannot be $\norm{\mhyphen}_2'$ bounded.
 
 \end{exmp}

\chapter{Coarse Actions}\label{ch:p1:coarse actions}
In this chapter we study coarse actions of coarse groups. One of the main points is to illustrate the connection between coarse group actions and geometric group theory. 

\section{Definition and Examples}\label{sec:p1:coarse actions def}
Let $\crse G=(G,\CE)$ be a coarse group and $\crse Y=(Y,\CF)$ any coarse space. 

\begin{de}\label{def:p1:coarse action}
 A \emph{(left) coarse action}\index{coarse action} is a coarse map $\crse{\alpha}\colon \crse G\times \crse Y\to \crse Y$ such that the following Action Diagrams\index{Action Diagrams} commute
\[ \begin{tikzcd}
    \crse{G}\times \crse G\times \crse Y \arrow[r, "\cid_{\crse G} \times \crse \alpha"] \arrow[d, "\cop\times \cid_{\crse Y}"] & \crse{G}\times \crse Y      \arrow[d, "\crse \alpha"]\\
    \crse{G}\times \crse Y \arrow[r, "\crse \alpha"]& \crse Y 
\end{tikzcd} 
\qquad\qquad 
\begin{tikzcd}
    \crse{Y}\arrow{r}{(\cunit,\cid_{\crse Y})} 
    \arrow[dr, swap, "\cid_{\crse Y}"] & \crse{G}\times \crse Y      
    \arrow[d, "\crse \alpha"]\\
    & \crse Y 
\end{tikzcd} 
\] 
 In analogy with the notation $\alpha\colon G\curvearrowright Y$ for set\=/group actions, we denote coarse actions by $\crse \alpha\colon \crse G\curvearrowright \crse Y$.
\end{de}

As a direct consequence of the definition, every coarse group $\crse G$ admits a few natural examples of coarse actions. First, given any coarse space $\crse Y$ we may define the \emph{trivial coarse action}\index{coarse action!trivial} of $\crse G$ on $\crse Y$ letting $\alpha(g,y)=y$ for every $g\in G$, $y\in Y$ (\emph{i.e.}\ $\crse\alpha$ is the coordinate projection to $\crse G\times \crse Y\to \crse Y$).

Second, the coarse multiplication $\cop\colon \crse G\times \crse G\to \crse G$ can be seen as a left coarse action of $\crse G$ on itself. If $\crse {H\leq G}$ is a coarse subgroup we may consider the restriction of $\cop_{\crse G}\colon\crse{G\times G\to G}$ to $\crse{H\times G}$ and observe that this defines a coarse action of $\crse H$ on $\crse G$. These are examples of \emph{coarse actions by left multiplication}\index{coarse action!by left multiplication} (later we shall also discuss actions by left multiplications on quotient coarse groups and coarse coset spaces, see Sections~\ref{sec:p1:quotients by left mult}--\ref{sec:p1:cosets of subgroups}). Finally, every coarse group also coarsely acts on itself by conjugation: we will discuss this more in detail in Section~\ref{sec:p1:action by conjugation}.

As for usual set\=/group actions, it is often convenient to write coarse actions as binary operations. That is, once a coarse action $\crse{\alpha \colon G\curvearrowright Y}$ is fixed we prefer to write $g\cdot y$ instead of $\alpha(g,y)$. Using this notation, commutativity of the Action Diagrams takes the familiar form:
\[
 g_1\cdot(g_2\cdot y)\rel{\CF} (g_1\ast g_2)\cdot y
 \quad\text{ and }\quad
 e\cdot y\rel{\CF} y \qquad\forall g_1,g_2\in G,\ \forall y\in Y.
\]

It is also convenient to extend the binary notation to relations on $G$ and $Y$: given $E\subseteq G\times G$ and $F\subseteq Y\times Y$ we let\nomenclature[:R]{$E\cdot F$}{action multiplication of relations}
\[
 E\cdot F\coloneqq\braces{(e_1\cdot f_1,e_2\cdot f_2)\mid (e_1,e_2)\in E,\ (f_1,f_2)\in F}=\alpha\times\alpha(E\otimes F)
\]
(this is convention is analogous to the definition of $E_1\ast E_2$).
Recall that a function $\alpha\colon (G,\CE)\times (Y,\CF)\to(Y,\CF)$ is controlled if and only if 
\[
 E\cdot \Delta_Y \in\CF \qquad\text{and}\qquad \Delta_G\cdot F\in \CF
\]
for every $E\in \CE$ and $F\in\CF$ (Lemma~\ref{lem:p1:equi controlled.sections.iff.controlled}). In many cases we consider, it is clear that $E\cdot \Delta_Y \in\CF$ for every $E\in\CE$. Therefore, the more interesting condition is that $\Delta_G\cdot F$ belongs to $\CF$.  Extending Definition~\ref{def:p1:equi.bi.invariant coarse structure}, we shall say that the coarse structure $\CF$ is \emph{equi left invariant}\index{coarse structure!equi left invariant}\index{equi!left invariant} (for the fixed function $\alpha$) if $\Delta_G\cdot F$ belongs to $\CF$ for every $F\in \CF$.

It is instructive to consider examples of coarse actions of coarsified set\=/groups.
Let $\alpha\colon G\curvearrowright Y$ be a set\=/group action. Then the Action Diagrams trivially commute for any choice of coarse structures on $G$ and $Y$. It follows that such an $\alpha$ defines a coarse action $\crse \alpha \colon (G,\CE)\curvearrowright (Y,\CF)$ if and only if it is controlled. 
In particular, $\alpha$ defines a coarse action of the trivially coarse group $(G,\mincrs)$ on $(Y,\CF)$ if and only if the coarse structure $\CF$ is equi left invariant under $\alpha$.

\begin{exmp}\label{exmp:p1:metric action by unif coarse eq}
 Let $\alpha\colon G\curvearrowright Y$ be a set\=/group action on a metric space $(Y,d)$. Then it is simple to verify $\crse \alpha\colon (G,\mincrs)\curvearrowright(Y,\CE_d)$ is a coarse action if and only if $\alpha$ is an action by \emph{uniform coarse equivalences.}  That is, there must be (increasing, unbounded) controlled functions $\rho_-,\rho_+\colon[0,\infty)\to[0,\infty)$ such that 
 \[
 \rho_-\paren{d(y_1,y_2)}\leq d\paren{g\cdot y_1,g\cdot y_2} \leq \rho_+(d(y_1,y_2))
 \]
 for every pair of points $y_1,y_2\in Y$ and $g\in G$ (the upper bound follows directly from checking that the functions $\alpha(g,\variable)$ are equi controlled. The lower bound is obtained by observing that $\alpha(g^{-1},\variable)$ is the inverse of $\alpha(g,\variable)$). Moreover, if $(Y,d)$ is a geodesic metric space this only happens if $\alpha$ is an action by uniform quasi\=/isometries, \emph{i.e.}\ there must be constants $L,A$ such that $\frac{1}{L}d\paren{y_1,y_2}-A\leq d\paren{g\cdot y_1,g\cdot y_2} \leq Ld(y_1,y_2)+A$ for every pair of points $y_1,y_2\in Y$.
 
 More generally, a \emph{quasi\=/action} of a set\=/group $G$ on a metric space $(Y,d)$ is a map $\alpha$ from $G$ to  the set of quasi\=/isometries of $Y$ such that the maps $\alpha(g,\variable)$ are uniform quasi isometries and 
 \[
  \alpha(gh,\variable)\closefn \alpha(g,\variable)\circ\alpha(h,\variable)
  \qquad \alpha(e,\variable)\closefn \id_Y
 \]
 are close functions for every $g,h\in G$ (\emph{i.e.}\ the Action Diagrams commutes up to closeness). As above, we see that every quasi\=/action is a coarse action of the trivially coarse group $(G,\mingrp)\curvearrowright(Y,\CE_d)$. 
\end{exmp}

\begin{rmk}
Special examples of coarse actions---especially quasi\=/actions---have been studied in many guises by various authors. Some examples include the works of Manning \cite{manning2006quasi} and Mosher--Sageev--Whyte \cite{mosher2003quasi, mosher2011quasi} on quasi-actions on trees; and the works of Eskin--Fisher--Whyte \cite{eskin2010quasi,eskin2012coarse,eskin2013coarse}.  An important motivation comes from the theory of quasi\=/isometric rigidity. In fact, quasi\=/actions are often implicit in the study of such rigidity statements.
\end{rmk}

 \begin{example}\label{exmp:p1:topological left multiplication}
We saw in Example~\ref{exmp:p1:topological coarse structure abelian} an abelian topological group $G$ has a topological coarsification $\varcrs[grp]{cpt}$\index{coarsification!topological} whose bounded subsets are the relatively compact ones. The assumption that $G$ is abelian is necessary to ensure that the family $\fkK_e$ of relatively compact subsets of $G$ containing the identity is closed under taking unions on conjugates (condition (U4) of Section~\ref{sec:p1:determined locally}). However, even if $G$ is an arbitrary topological group, $\fkK_e$ still satisfies the conditions (U0)--(U3). It follows from Proposition~\ref{prop:p1:families of identity neighbourhoods_setgroups} there exists a unique equi left invariant coarse structure $\varcrs[left]{cpt}$ whose bounded sets are the relatively compact subsets of $G$. The left multiplication thus defines a natural coarse action of the trivially coarse group $(G,\mincrs)\curvearrowright(G,\varcrs[left]{cpt})$. 
 Notice that we can also describe the coarse structure $\varcrs[left]{cpt}$ very explicitly. Namely,
 \[
  \varcrs[left]{cpt}
  \coloneqq\braces{E\subseteq G\times G\mid \exists K\subseteq G\text{ compact s.t. }\forall(g_1,g_2)\in E,\ g_2^{-1}g_1\in K}.
 \]
 The $\varcrs[left]{cpt}$\=/controlled partial coverings are refinements of coverings of the form $\pts{G}\ast K$ with $K\in\fkK_e$.
 
 We may also extend this construction to actions of topological groups on topological spaces. Namely, if $\alpha\colon G\curvearrowright Y$ is a continuous action, we can consider the coarse structure $\CF_{\rm cpt}^{\rm left}$ whose controlled partial covers are refinements of coverings of the form $\braces{g(K)\mid g\in G, K\text{ compact}}$. Equivalently,
 \[
  \CF_{\rm cpt}^{\rm left}
  \coloneqq\braces{F\subseteq Y\times Y\mid \exists K\subseteq Y\text{ compact s.t. }\forall(y_1,y_2)\in F,\ \exists g\in G\text{ with } y_1,y_2\in g(K)}
 \]
(one may check that $\CF_{\rm cpt}^{\rm left}$ is closed under composition with an argument analogous to that of Example~\ref{exmp:p1:compact coarse structures}).
 With this coarse structure, $\alpha$ becomes a coarse action of the trivially coarse group $\crse \alpha \colon (G,\mincrs)\curvearrowright (Y,\CF_{\rm cpt}^{\rm left})$.
\end{example}

Returning to general coarse groups, let $\crse G=(G,\CE)$.
If $\CF\supseteq\CE$ is a coarser coarse structure on $G$, then the coarse multiplication map $\ast\colon (G,\CE)\times (G,\CF) \to(G,\CF)$ still satisfies the Action Diagrams. 
Therefore $\ast\colon (G,\CE)\times (G,\CF) \to(G,\CF)$ defines a coarse action of $\crse G$ on $(G,\CF)$ if and only if $\CF$ is equi left invariant. 
We will later see that this elementary construction is the prototypical example of a coarse action.

\begin{rmk}
 Let $(G,\CE)$ be a coarse group and $(Y, \CF)$ a coarse space. 
 The set of functions $\alpha\colon G\times Y\to Y$ that define a coarse action $\crse \alpha \colon (G,\CE)\curvearrowright(Y,\CF)$ decreases if $\CE$ is replaced with some coarser coarse structure $\CE'\supset\CE$. The reason for this behavior can be appreciated with an extreme example: it is natural to expect that the bounded coarse group $(G,\maxcrs)$ only admits trivial coarse actions (\emph{i.e.}\ such that $\alpha(g,\variable)$ stays uniformly close to $\id_Y$ as $g$ ranges in $G$).  For a more sophisticated explanation see also Chapter~\ref{ch:p2:spaces of controlled maps}.
\end{rmk}

\begin{rmk}
Another notion that appears in the literature is that of \emph{near actions}\index{near action} of set\=/groups. This is another way to approach the study of mappings that are ``almost'' a group action up to some finite error. In general, a near action cannot be seen as a coarse action, as the space that is being acted upon need not have a compatible coarse structure. Vice versa, a coarse action need not be a near action if the bounded sets are allowed to be infinite. However, there are various instances of mappings that can be seen both as a near actions and coarse actions. Cornulier \cite{cornulier2019near} gives an extensive list of well-explained examples and references, which is an excellent resource for the reader interested in near actions. 
\end{rmk}

\section{Coarse Invariance and Equivariance}
Adapting standard nomenclature from actions of set\=/groups to actions of coarse groups, we may say that a $\crse G$\=/space is a coarse space $\crse Y$ equipped with a $\crse G$\=/action. Coarse equivariant maps and coarse invariant subspaces will then correspond to maps and inclusions of $\crse G$\=/spaces respectively. Explicitly, we use the following:

\begin{de}\label{def:p1:coarse equivariant}
 Let $\crse \alpha\colon \crse G\curvearrowright \crse Y$ and $\crse{\alpha'}\colon \crse G\curvearrowright \crse{Y'}$ be coarse actions. A coarse map $\crse{f}\colon \crse Y\to \crse{ Y'}$ is \emph{coarsely $\crse G$\=/equivariant}\index{coarsely!$\crse G$\=/equivariant (map)}\index{coarsely!equivariant (map)} (or simply \emph{coarsely equivariant} when $\crse G$ is clear from the context) if the following diagram commutes: 
\[ \begin{tikzcd}
    \crse{G}\times\crse{Y} \arrow[r, "{\crse \alpha}"] \arrow[d,swap, "{\cid_{\crse G}\times \crse f}"] & \crse{Y}      \arrow[d, "{\crse{f}}"]\\
    \crse{G}\times \crse {Y'} \arrow[r, "{\crse {\alpha'}}"]& \crse {Y'} . 
\end{tikzcd} 
\] 
That is, if $\crse{Y'}=(Y,\CF')$ then $\crse f$ is coarsely equivariant if
\[
 f(\alpha(g,y))\rel{\CF'}\alpha'(g,f(y))\qquad \forall g\in G, \forall y\in Y.
\]
Two coarse actions $\crse\alpha\colon \crse G\curvearrowright\crse Y$ and $\crse{\alpha'}\colon \crse G\curvearrowright\crse{Y'}$ are \emph{isomorphic}\index{isomorphic!coarse actions} if there exists a coarsely equivariant coarse equivalence $\crse f\colon \crse Y\to\crse{Y'}$ (the coarse inverse of $\crse f$ is automatically coarsely equivariant).
\end{de}

\begin{exmp}
 The additive group $\RR$ acts on $(\RR,\CE_d)$ by addition $(t,x)\mapsto t+x$ and coarsely acts on $(\ZZ,\CE_d)$ by $(t, n)\mapsto \lfloor t+n\rfloor$. The inclusion $\ZZ\hookrightarrow\RR$ is coarsely equivariant (but not set\=/theoretically equivariant), so these coarse actions are isomorphic. 
 The above formulae define coarse actions of both the trivially coarse group $(\RR,\mincrs)$ and the metric coarse group $(\RR,\CE_d)$. In either case, the inclusion gives an isomorphism of coarse actions. In fact, the notion of coarse equivariance only depends on the coarse structures of the target spaces. 
\end{exmp}

The above example is a special case of a general observation: if $\crse{H\leq G}$ is a coarse subgroup, we know that the restriction of $\cop_{\crse G}$ to $\crse{H\times G}$ defines a coarse action by left multiplication $\crse{H\curvearrowright G}$. Since the coarse multiplication $\cop_{\crse H}$ is compatible with $\cop_{\crse G}$, it follows that the inclusion $\crse{H\hookrightarrow G}$ is $\crse H$\=/equivariant. It is then clear that if $H$ is coarsely dense in $\crse G$ (\emph{i.e.}\ $\crse H=\crse G$ as a coarse subspace) then $\crse{H\curvearrowright H}$ and $\crse{H\curvearrowright G}$ are isomorphic $\crse H$\=/actions. For instance, this shows that $\Gamma$ is a cocompact lattice in an abelian topological group $G$ then $(\Gamma,\varcrs[grp]{fin})\curvearrowright (\Gamma,\varcrs[grp]{fin})$ is coarsely isomorphic to $(\Gamma,\varcrs[grp]{fin})\curvearrowright (G,\varcrs[grp]{cpt})$ (see Example~\ref{exmp:p1:topological coarse structure abelian}).

\begin{exmp}\label{exmp:p1:conjugate left invariant coarse structures}
 Let $\crse G=(G,\CE)$ be a coarse group and $\CF\supseteq\CE$ an equi left invariant coarse structure, so that $\cop \colon \crse G\curvearrowright(G,\CF)$ is a coarse action. For a fixed $g\in G$, let $\CF\ast (g,g)\coloneqq \angles{\braces{F\ast(g,g)\mid F\in \CF}\cup \CE}$ be the ``right translate'' of $\CF$. Here we added $\CE$ to the set of generators because we need $\CE\subseteq \CF\ast(g,g)$ and this is not necessarily the case \emph{e.g.}\ if $\ast_g$ is not surjective. 
 Since $\CE\subseteq \CF\ast(g,g)$, associativity of $\cop$ implies that $\Delta_{G}\ast(F\ast(g,g))$ is in  $\CF\ast(g,g)$ for every $F\in\CF$ and hence the coarse structure $\CF\ast(g,g)$ is equi left invariant. We thus obtain a new coarse action $\cop \colon \crse G\curvearrowright(G,\CF\ast(g,g))$.
 However, $\CF\ast(g,g)$ is defined so that the right multiplication $\ast_g\colon(G,\CF)\to(G,\CF\ast(g,g)$ is controlled and, by the coarse group axioms, $\ast_{g^{-1}}\colon(G,\CF\ast(g,g))\to(G,\CF)$ is controlled as well. These two controlled maps are coarsely equivariant and they are (coarse) inverses of one another. They hence determine an equivariant coarse equivalence, \emph{i.e.}\ they give an isomorphism between the two coarse actions $\crse G\curvearrowright(G,\CF)$ and $\crse G\curvearrowright (G,\CF\ast(g,g))$. 
\end{exmp}

 Expanding this example, we define the coarse structure $(g,g)\ast \CF$, however this will be trivially equal to $\CF$ itself because $\CF$ is (equi)left invariant. In particular, $\CF\ast(g^{-1},g^{-1})=(g,g)\ast\CF\ast(g^{-1},g^{-1})$ (we can drop the parentheses from the notation because $\ast$ is coarsely associative). So we can view every right translate of $\CF$ as a conjugation. This notion will be useful later on, we thus state it as a definition.
 
 \begin{de}\label{def:p1:conjugate left invariant coarse structures}
  Let $\crse G=(G,\CE)$ be a coarse group and $\CF$ an equi left invariant coarse structure containing $\CE$. The equi left invariant coarse structure
  \[
   (g,g)\ast\CF\ast(g^{-1},g^{-1})\coloneqq \angles{\braces{(g,g)\ast F\ast(g^{-1},g^{-1})\mid F\in \CF}\cup \CE} =\CF\ast (g^{-1},g^{-1})
  \]
  is \emph{conjugate}\index{coarse structure!conjugate} to $\CF$. We say that $\CF$ and $(g,g)\ast\CF\ast(g^{-1},g^{-1})$ are \emph{conjugate equi left invariant coarse structures} (this is an equivalence relation).
 \end{de}

 \begin{rmk}
 If the set $\{g,e\}$ is $\CF$\=/bounded then $\CF$ and $(g,g)\ast\CF\ast(g^{-1},g^{-1})$ coincide (this is analogous to the observation that conjugation ${}_gc$ by an element close to the unit is close to the identity function, see Example~\ref{exmp:p1:conjugation automorphisms}). 
In particular, if $\CF$ is a connected coarse structure then it is always preserved by conjugation. On the contrary, disconnected equi left invariant coarse structures need not be preserved under conjugation. For example, let $F_2=\angles{a,b}$ be the $2$\=/generated free group and let $\CF$ be the coarse structure determined by the extended metric where $d(w,wa^n)=\abs{n}$ and $d(w,w')=\infty$ if $w^{-1}w'\notin \angles{a}$ (the coarse space $(F_2,\CF)$ can be seen as a disconnected union of lines indexed over the set of cosets $F_2/\angles{a}$). The set $\{e,a\}$ is $\CF$\=/bounded and hence $\{b,ab\}$ is $(\CF\ast(b,b))$\=/bounded, but $\{b,ab\}$ is not $\CF$\=/bounded, so $\CF\neq\CF\ast(b,b)$.  
 \end{rmk}

We may now discuss coarsely invariant subspaces.

\begin{de}
 If $\crse \alpha\colon\crse G\curvearrowright\crse Y$ is a coarse action and $\crse{ Z\subseteq Y}$ is a coarse subspace, we can define $\crse G\cact \crse Z$ as the coarse image $\crse{\alpha (G\times Z)\subseteq Y}$. We say that $\crse Z$ is \emph{coarsely $\crse G$\=/invariant}\index{coarsely!$\crse G$\=/invariant} if $\crse{G \cact Z=Z}$. 
\end{de}

The same proof of Proposition~\ref{prop:p1:closed.asymp.classes.are.subgroups} shows that if $\crse{Z\subseteq Y}$ is coarsely $\crse G$\=/invariant then we can restrict $\crse \alpha$ to a coarse action $\crse{\alpha|_{Z}}\colon\crse G\curvearrowright \crse Z$ such that the inclusion $\crse {Z}\hookrightarrow \crse{Y}$ is coarsely equivariant. Further, the coarse action $\crse{\alpha|_{Z}}$ is uniquely defined up to coarsely equivariant coarse equivalences. This action is the \emph{restriction} of $\crse \alpha$ to $\crse Z$. Concretely, we can construct a controlled map $\alpha|_Z\colon G\times Z\to Z$ by choosing $\alpha|_Z(g,z)\in Z$ uniformly close to $\alpha(g,z)\in Y$.

\begin{exmp}
 If $\crse{H\leq G}$ is a coarse subgroup and $\crse{H\curvearrowright G}$ is the coarse action by left multiplication, then $\crse H$ is a $\crse{H}$\=/invariant subspace of $\crse G$. The multiplication map $\cop_{\crse H}\colon\crse{H\times H\to H}$ may also be seen as the restriction to $\crse H$ of the coarse action of $\crse H$ on $\crse G$.
\end{exmp}

\section{Coarse Action by Conjugation}
\label{sec:p1:action by conjugation}
A coarse action of special interest is given by conjugation.
Let $\crse G=(G,\CE)$ be a coarse group. We already remarked that conjugation by a fixed $g\in G$ defines a coarse automorphism $\crse{ {}_{g}c}\colon\crse G\to\crse G$ (Example~\ref{exmp:p1:conjugation automorphisms}). This can be improved to show that conjugation defines a left coarse action:

\begin{de}\label{def:p1:action by conjugation}
 The \emph{action by conjugation}\index{coarse action!by conjugation} of a coarse group $\crse G$ on itself is the left coarse action 
$\crse c \colon \crse G\curvearrowright \crse G$ where $c\colon G\times G\to G$ is defined by $c(g,x)\coloneqq g\ast x\ast g^{-1}$ (up to closeness, this definition is independent of the order of multiplication in $g\ast h\ast g^{-1}$).
\end{de} 
If $c$ is controlled, it follows easily from the coarse associativity of $\ast$ that $\crse c$ is indeed a left coarse action. To prove that $c$ is controlled one needs to show that $(c\times c)(E\otimes\Delta_{G})$ and $(c\times c)(\Delta_{G}\otimes E)$ are controlled whenever $E\in \CE$. To prove the former, note that if $x\torel{E} y$ then $x^{-1}\ast y\rel{E'}e$ for some $E'\in\CE$ and hence
\[
 g\ast x\ast g^{-1}
 \rel{\CE} (g\ast x)\ast e \ast g^{-1}
 \rel{\Delta_{G}\ast E'\ast \Delta_{G}} (g\ast x)\ast (x^{-1}\ast y)\ast g^{-1}
 \rel{\CE} g\ast y\ast g^{-1} 
 \qquad \forall g\in G, x\torel{E} y.
\]
Showing that $c\times c(\Delta_{G}\otimes E)$ is similar (and simpler). Alternatively, controlledness could be proved by diagram chasing by realizing $\crse c$ as an appropriate composition of diagonal embeddings, permutation of coordinates, taking inverses and left/right multiplications. Of course, the action by conjugation is well\=/behaved under coarse homomorphism: we leave the proof of the following as an exercise.

\begin{lem}
 If $\crse f\colon \crse G\to\crse H$ is a coarse homomorphism of coarse groups, then the composition 
 \[
  \crse G\times \crse H\xrightarrow{\crse f\times \cid_{\crse H}}\crse H\times \crse H \xrightarrow{\ \crse c\ } \crse H
 \]
 is a coarse action and $\crse f\colon \crse G\to \crse H$ is coarsely equivariant between $\crse c \colon \crse G\curvearrowright \crse G$ and $\crse c\circ (\crse f\times \cid_{\crse H})\colon\crse G\curvearrowright \crse H$.
\end{lem}

 We already pointed out that if $\crse G$ is coarsely connected then $\crse{{}_{g} c} =\cid_{\crse{G}}$  for every fixed $g\in G$. However, this does not mean that the action by conjugation is trivial! In fact, the maps ${}_g c$ need not be \emph{uniformly} close to the identity. More precisely, it is easy to see that the action by conjugation is trivial if and only if the multiplication coarsely commutes:
 
 \begin{de}\label{def:p1:coarsely abelian}
  A coarse group $\crse G=(G,\CE)$ is \emph{coarsely abelian} if any of the following equivalent statements hold:
  \begin{itemize}
   \item the coarse action by conjugation $\crse c\colon \crse G\curvearrowright\crse G$ is trivial (\emph{i.e.}there exists $E\in \CE$ such that 
   \[
    g\ast h\ast  g^{-1}\rel{E} g
   \]
    for every $g,h\in G$);
   \item there exists $E\in \CE$ such that $g_1\ast g_2\rel{E} g_2\ast g_1$ for every $g_1,g_2\in G$. 
  \end{itemize}
 \end{de}

\begin{rmk}\label{rmk:p1:on coarse abelian groups}
 Recall that the cancellation metric on the free group (Example~\ref{exmp:p1:cancellation metric}) provides an example of a non coarsely abelian coarsely connected coarse group.  
 Obviously, a coarsified set\=/group $\crse G=(G,\CE)$ can be coarsely abelian even if $G$ is not abelian. For one, this is the case if $\CE=\varcrs{max}$. A marginally more interesting example is the case $G=H\times F$ where $H$ is abelian, $F$ is finite and the coarse structure $\CE$ is connected (\emph{e.g.}\ $\CE=\varcrs[grp]{fin}$). 
\end{rmk}

Every coarse group $\crse G=(G,\CE)$ has a \emph{coarse abelianization}\index{coarse!abelianization}, \emph{i.e.}\ a coarse quotient $\crse{ \overline{ G}}$ such that every coarse homomorphism from $\crse G$ to a coarsely abelian group factors though $\crse{ \overline{ G}}$. Using the notation of Definition~\ref{de:p1:quotient by a relation}, this is the coarse quotient $\crse G/\aangles{\CR_{\rm comm}}$ where $\CR_{\rm comm}\coloneqq\braces{(g_1\ast g_2,g_2\ast g_1)\mid g_1,g_2\in G}$. Obviously, if $G$ is a set\=/group and $G/[G,G]$ is its set\=/group abelianization, the quotient map $\pi\colon G\to G/[G,G]$ factors through the coarse abelianization
 \[
  (G,\mincrs)\xrightarrow{\cid} (G,\mincrs)/\aangles{\CR_{\rm comm}}\xrightarrow{\crse \pi} (G/[G,G],\mincrs).
 \]
 However, neither of the above maps is an isomorphism of coarse groups in general.
 
\

Having established the meaning of coarse conjugation, we can define the notions of coarse normality and centrality.

\begin{de}\label{def:p1:normal}
 A coarse subspace $\crse A$ of a coarse group $\crse G$ is \emph{coarsely invariant under conjugation}\index{coarsely!invariant under conjugation} if $\crse{c(G\times A)=A}$ (\emph{i.e.}\ it is coarsely invariant under the coarse action by conjugation). We may also say that any representative $A$ is coarsely invariant under conjugation.
 We say that $\crse A$ is in the \emph{coarse centralizer}\index{coarse!centralizer} of $\crse G$ if it is coarsely invariant under conjugation and the restriction $\crse{c|_{A}}\colon\crse G\curvearrowright\crse A$ is the trivial coarse action.
 A coarse subgroup $\crse{H \leq G}$ is \emph{coarsely normal}\index{coarsely!normal} if it is coarsely invariant under conjugation. We will denote it by $\crse{ H\trianglelefteq  G}$. \nomenclature[:COS]{$\crse{H\trianglelefteq G}$}{coarsely normal subgroup}
\end{de}

Explicitly, $A\subseteq G$ is coarsely invariant under conjugation if there is an $E\in\CE$ such that for every $g\in G$ and $a\in A$ there is an $a\in A$ with $g\ast a\rel{E} a'\ast g$. The set $A$ is in the coarse centralizer if $g\ast a\rel{E} a\ast g$ (\emph{i.e.}\ the points in $A$ uniformly coarsely commute with every $g$ in $G$).

\begin{exmp}
 If $\crse G$ is a coarse group and $\crse{G_{e}}$ is the coarsely connected component of the identity, then $\crse{G_{e}}$ is a coarsely normal coarse subgroup of $\crse G$.
\end{exmp}

\section{Coarse Orbits}
\label{sec:p1:coarse.orbits}
Let $\crse\alpha\colon \crse G\curvearrowright \crse Y$ be a coarse action.  Recall our convention ${}_g\alpha(\variable)\coloneqq \alpha(g,\variable)$ and $\alpha_y(\variable)\coloneqq \alpha(\variable,y)$.
Notice that if $y,y'\in Y$ are two close points, then $\alpha_y$ and $\alpha_{y'}$ are close functions. We hence have a well\=/defined coarse map $\crse{\alpha_{y}\colon G\to Y}$, where $\crse {y\in Y}$ denotes a coarse point. 

\begin{de}\label{def:p1:coarse orbits}
 The \emph{coarse orbit} (also called $\crse G$\=/orbit)\index{coarse!orbit} of a coarse point $\crse{ y\in Y}$ is the coarse subspace $\crse{ G\cact y = \alpha(G\times y)\subseteq  Y}$. The associated \emph{orbit map}\index{orbit map} is $\crse{\alpha_{y}}=\colon\crse G\to\crse Y$. \nomenclature[:COS]{$\crse G\cact \crse y $}{coarse orbit}
\end{de}

Equivalently, a coarse point can be seen as a coarse map $\crse y\colon \tobj\to\crse Y$, the orbit map $\crse{\alpha_{y}}$ is the composition
\[
 \crse G\times \tobj\xrightarrow{\cid_{\crse{G}}\times \crse y} \crse G\times\crse Y\xrightarrow{\ \crse \alpha\ }\crse{Y}
\]
and the coarse orbit $\crse{G\cact y}$ is the coarse image of $\crse{\alpha_{y}}$.

 Obviously, a coarse action $\crse{\alpha}\colon\crse G\curvearrowright \crse{Y}$ can have at most as many distinct coarse orbits as there are coarsely connected components in $\crse Y$. In particular, if $\crse{Y}$ is coarsely connected then the coarse orbit is unique up to coarse equivalence. 
 It is worth pointing out that some care is needed when translating between coarse geometric properties of coarse orbits and properties of $\crse \alpha$. The properties of the coarse space $\crse{ G\cdot y}$ are indeed independent from the choice of $y\in Y$, but this does not say whether they hold ``uniformly'' as $y$ varies.  
 For instance, the following is an example of a coarse action $\crse \alpha$ such that all the $\crse G$\=/orbits of $\crse \alpha$ are bounded, but $\crse \alpha$ is not the trivial coarse action (\emph{i.e.}\ the $\crse{G}$\=/orbits are not ``uniformly bounded''):

\begin{exmp}\label{exmp:p1: action by rotation}
 Consider the action by rotation $\alpha\colon\RR\curvearrowright \CCC$, $\alpha(t,z)\coloneqq e^{it}z$. Since it is a set\=/group action by isometries, $\alpha$ gives a coarse action $(\RR,\mincrs)\curvearrowright (\CCC,\CE_{\abs{\mhyphen}})$, where $d$ is the metric on $\CCC$. This coarse action is non\=/trivial (\emph{i.e.}\ it is not close to the projection $\RR\times\CCC\to\CCC$), but it has an unique coarse orbit and this coarse orbit is trivial. 
\end{exmp}

Note that $\crse G$\=/orbits are $\crse{ G}$\=/invariant coarse subspaces because $\crse {G\cact (G\cact y)=(G\cop  G)\cact  y = G\cact y}$. In particular, the coarse action $\crse \alpha$ restricts to a coarse $\crse{G}$\=/action on each $\crse G$\=/orbit of $\crse Y$. The fact that the restriction $\crse{\alpha|_{G\cdot y}}$ is well\=/defined implies that $\crse{G\cact y}$ is actually well defined up to \emph{coarsely $\crse G$\=/equivariant} coarse equivalence (Definition~\ref{def:p1:coarse equivariant}).

\begin{rmk}
 If $\crse{Z\subseteq Y}$ is $\crse G$\=/invariant and $\crse{ z\in Z}$, then the $\crse{G}$\=/orbit of $\crse z$ is coarsely contained in $\crse Z$. Example~\ref{exmp:p1: action by rotation} shows that the converse does not hold: since all the coarse orbits of the coarse action $\crse \alpha \colon(\RR,\mincrs)\curvearrowright (\CCC,\CE_{\abs{\mhyphen}})$ are trivial, every coarse subspace of $\CCC$ coarsely contains the coarse orbit of each of its points. On the other hand, the real axis in $\CCC$ determines a coarse subspace that is not $(\RR,\mincrs)$\=/invariant.
\end{rmk}

The pull\=/back of the coarse structure of $\crse Y$ under any orbit map defines a coarse structure on $G$. That is, given $\crse \alpha \colon (G,\CE)\curvearrowright (Y,\CF)$  and $y\in Y$ we consider the coarse structure $\alpha_y^*(\CF)\coloneqq\alpha(\variable,y)^*(\CF)$. Note that $\alpha_y^*(\CF)$ does not depend on the choice of representatives for $\crse \alpha$ and $\crse y$. Yet, it is notationally convenient to fix such choices.

It follows from the definition of pull\=/back, that $\alpha_y\colon (G,\alpha_y^*(\CF))\to (Y,\CF)$ is a coarse embedding and hence induces a coarse equivalence between $(G,\alpha_y^*(\CF))$ and its coarse image $\crse{ G\cact y}$.
Furthermore, we can show that $\crse G$ coarsely acts on $(G,\alpha_y^*(\CF))$ by left multiplication:

\begin{lem}\label{lem:p1:pullback orbit coarse actions}
 The group operation $\ast$ defines a left coarse action $\cop\colon (G,\CE)\times(G,\alpha_y^*(\CF))\to(G,\alpha_y^*(\CF))$.
 Moreover, the orbit map $\crse{\alpha_{y}}$ is coarsely equivariant from $\cop\colon \crse G\curvearrowright (G,\alpha_y^*(\CF))$ to $\crse \alpha\colon\crse G\curvearrowright \crse Y$.
\end{lem}
\begin{proof}
 We first need to show that $\ast\colon (G,\CE)\times (G,\alpha_y^*(\CF))\to(G,\alpha_y^*(\CF))$ is controlled. By definition, this is the case if and only if the composition $\alpha_y\circ \ast \colon (G,\CE)\times (G,\alpha_y^*(\CF))\to(Y,\CF)$ is controlled. On the other hand, $\alpha_y\circ \ast$ is close to $\alpha(\variable,\alpha(\variable,y))$ by definition of coarse action. The latter is controlled because it is the composition of controlled maps:
 \[
  (G,\CE)\times (G,\alpha_y^*(\CF))\xrightarrow{id_{G}\times \alpha_y}(G,\CE)\times(Y,\CF)\xrightarrow{\quad \alpha\quad} (Y,\CF).
 \]
 It follows that $\alpha_y\circ \ast$ is controlled as well. In particular, $\alpha_y^*(\CF)$ is equi left invariant. Moreover $\CE\subseteq\alpha_y^*(\CF)$ because $\alpha_y\colon (G,\CE)\to(Y,\CF)$ is controlled, therefore $\cop\colon\crse G\curvearrowright (G,\alpha_y^*(\CF))$ is a coarse action. 
 
 To show equivariance of the orbit map, it is enough to write ${\alpha_y}$ as the composition of ${\alpha}$ and $(\id_{G},y)\colon G\to G\times Y$ and note that equivariance follows immediately from the definition of coarse action: 
\[ 
\begin{tikzcd}[column sep = 5em]
    (G,\CE)\times (G,\alpha_y^*(\CF)) \arrow{r}{\cid_{\crse G}\times(\cid_{\crse G},\crse y)} \arrow[d,swap, "\cop"] &
    (G,\CE)\times (G,\alpha_y^*(\CF))\times{Y} \arrow{r}{\cid_{\crse G}\times \crse \alpha} \arrow[swap]{d}{\cop \times \cid_{\crse Y}} & 
    (G,\CE)\times (Y,\CF)      \arrow{d}{\crse \alpha}
    \\
    (G,\alpha_y^*(\CF)) \arrow{r}{(\cid_{\crse G},\crse y)} &
    (G,\alpha_y^*(\CF))\times{Y} \arrow{r}{\crse \alpha} & 
    (Y,\CF) 
\end{tikzcd} 
\] 
(we are also using that commutative diagrams remain coarsely commutative when considered with respect to larger coarse structures).
\end{proof}

The key point in the proof of Lemma~\ref{lem:p1:pullback orbit coarse actions} is the remark that $\ast\colon (G,\CE)\times (G,\alpha_y^*(\CF))\to(G,\alpha_y^*(\CF))$ is controlled. In the language of Section~\ref{sec:p1:equi controlled.maps}, this can be rephrased by saying that the left multiplications ${}_g\ast\colon (G,\alpha_y^*(\CF))\to (G,\alpha_y^*(\CF))$ are $(G,\CE)$\=/equi controlled, or equivalently that the right multiplications $\ast_g\colon (G,\CE)\to(G,\CE)$ are  $(G,\alpha_y^*(\CF))$\=/equi controlled. It is instructive to see what this means in a more concrete example:

\begin{exmp}
 Let $G$ be a set\=/group and $\alpha\colon G \curvearrowright Y$ a free action by isometries on a metric space $(Y,d)$. Equip $Y$ with the metric coarse structure $\CE_d$. Then $\alpha$ defines a coarse action of the trivially coarse group $(G,\mincrs)$. If we identify $G$ with any of its orbits $\alpha_y(G)\subseteq Y$, the metric $d$ induces a metric $d_y$ on $G$. The coarse structure $\alpha_y^*(\CE_d)$ coincides with the induced metric coarse structure (if the action was not free, $d_y$ would be a pseudo\=/metric on $G$, but $\alpha_y^*(\CE_d)$ would still coincide with the induced coarse structure). 
 
 In this example, the fact that the left multiplications ${}_g\ast\colon (G,\alpha_y^*(\CE_d))\to(G,\alpha_y^*(\CE_d))$ are $(G,\mincrs)$\=/equi controlled follows trivially from the observation that ${}_g\ast$ preserves the $d_y$.  
 On the other hand, the fact that the right multiplications ${\ast}_g\colon (G,\mincrs)\to(G,\mincrs)$ are $(G,\alpha_y^*(\CE_d))$\=/equi controlled follows from the observation that ${\ast}_g$ has displacement uniformly bounded in terms of $d_y(e,g)$. That is, for every $h\in G$
 \[
  d_y(h,h\ast g)=d(h\cdot y,h\ast g\cdot y)
  =d(y,g\cdot y)=d_y(e,g).
 \]
 It follows from the triangle inequality that if $(g_i,g_i')$ are pairs of uniformly close elements in $(G,\alpha_y^*(\CE_d))$ then $(h\ast g_i,h\ast g_i')$
 are also uniformly close, \emph{i.e.}\ the right multiplication is $(G,\alpha_y^*(\CE_d))$\=/controlled.
 
 The same argument also shows that in general $(G,\alpha_y^*(\CE_d))$ is \emph{not} a coarse group. In fact, if it was a coarse group then the left multiplications would have to be $(G,\alpha_y^*(\CE_d))$\=/equi controlled. But this is the case only if $d_y(g\ast h,h)$ is uniformly bounded in terms of $d_y(g,e)$, which is usually not the case. 
\end{exmp}

Note that to prove Lemma~\ref{lem:p1:pullback orbit coarse actions} we showed that if $\crse \alpha \colon(G,\CE)\curvearrowright(Y,\CF)$ is any coarse action, taking the pull\=/back allows us to define a coarsely equivariant coarse map between $\crse\alpha$ and a coarse action by left multiplication of $(G,\CE)$ on $G$ equipped with a coarser equi left invariant coarse structure.

Recall that we defined the conjugate of an equi left invariant coarse structure $\CF\supseteq \CE$ as 
\[
 (g,g)\ast\CF\ast(g^{-1},g^{-1})=\angles{\braces{(g,g)\ast F\ast(g^{-1},g^{-1})\mid F\in\CF}\cup\CE},
\]
where $g\in G$ is some fixed element (Definition~\ref{def:p1:conjugate left invariant coarse structures}). We also remarked that the right multiplications $\ast_g$ and $\ast_{g^{-1}}$ define a coarse equivalence between $(G,\CF)$ and $(G,(g,g)\ast\CF\ast(g^{-1},g^{-1}))$. In particular, $(g,g)\ast\CF\ast(g^{-1},g^{-1})$ can also be defined as the pull\=/back of $\CF$ under $\ast_g$.
Now, fix $y\in Y$, $g\in G$ and consider $\alpha_{g\cdot y}^*(\CF)$. Since $ \crse \alpha$ is a coarse action, we can rewrite $\alpha_{g\cdot y}^*(\CF)$ as
 \[
  \alpha_{g\cdot y}^*(\CF)
  =\alpha\paren{\variable,g\cdot y}^*(\CF)
  =\alpha\paren{(\variable)\ast g, y}^*(\CF)
  =\paren{\alpha_y\circ \ast_g}^*(\CF)
  =(g,g)\ast\alpha_y^*(\CF)\ast(g^{-1},g^{-1}).
 \] 
That is, taking the pull back under different points in the same $\crse G$\=/orbit defines conjugate equi left invariant coarse structures on $G$.  

This discussion is particularly interesting when $\crse{G\cact y}=\crse Y$ because it implies that $\crse \alpha$ is \emph{isomorphic} to an action by left multiplication of $\crse G$ on itself equipped with an appropriate coarse structure, and the latter is well defined up to conjugacy. More precisely, we make the following definition:

\begin{de}\label{def:p1:cobounded action}
 A coarse action $\crse \alpha \colon(G,\CE)\curvearrowright(Y,\CF)$ is \emph{cobounded}\index{coarse action!cobounded} if there exists (equivalently, for every) $y\in Y$ such that $\alpha_y(G)$ is coarsely dense in $(Y,\CF)$.
\end{de}

The above is the coarse analogue of transitivity for set-group actions. In analogy with the fact that transitive left actions of a set-group are isomorphic to actions by multiplication on the left cosets of the isotropy group, we can fully describe the family of cobounded coarse actions up to isomorphism:

\begin{prop}\label{prop:p1:cobounded coarse actions isom left mult}
 Let $\crse G=(G,\CE)$ be a coarse group. There is a natural correspondence between conjugacy classes of equi left invariant coarse structures on $G$ containing $\CE$ and isomorphism classes of cobounded left coarse actions:
 \[
  \braces{\CF\text{\upshape equi left invariant coarse structure}\mid \CE\subseteq\CF}/_{\text{\upshape conj.}}
  \longleftrightarrow\braces{\text{\upshape cobounded left coarse actions}}/_{\text{\upshape isom.}}
 \]
 (in particular, the class of cobounded coarse actions up to isomorphism is a set).
\end{prop}
\begin{proof}
 If $\CF$ is equi left invariant and contains $\CE$ then the left multiplication defines a coarse action of $\crse G$ on $(G,\CF)$. 
This action is cobounded, as $G\ast e_{G}$ is coarsely dense in $(G,\CE)$. 
We already noted that $\cop \colon(G,\CE)\curvearrowright(G,\CF)$ and $\cop \colon(G,\CE)\curvearrowright(G,(g,g)\ast\CF\ast(g^{-1},g^{-1}))$ are isomorphic coarse actions (Example~\ref{exmp:p1:conjugate left invariant coarse structures}).

This shows that we have a well\=/defined map from the LHS to the RHS. The converse map is obtained via pull\=/back: if $\crse \alpha \colon(G,\CE)\curvearrowright(Y,\CF)$ is a cobounded coarse action and $y\in Y$ is any fixed point, then $\crse \alpha$ is isomorphic to $\cop \colon(G,\CE)\curvearrowright(G,\alpha_y^*(\CF) )$ by Lemma~\ref{lem:p1:pullback orbit coarse actions}.
 If $z\in Y$ is any other point, then $z$ is close to $g\cdot y$ for some $g\in G$ because $\crse \alpha$ is cobounded. It follows that $\alpha_z^*(\CF)=\alpha_{g\cdot y}^*(\CF)$, and the latter is conjugated to $\alpha_y^*(\CF)$. That is, up to conjugacy, $\alpha_y^*(\CF)$ does not depend on the choice of $y$. 
 Finally, if $\alpha\colon\crse (G,\CE)\curvearrowright (Y,\CF)$ and $\alpha'\colon\crse (G,\CE)\curvearrowright (Y',\CF')$ are isomorphic via an equivariant coarse equivalence $f\colon(Y,\CF)\to(Y',\CF')$, then
 \[
  (\alpha'_{f(y)})^*(\CF')=\alpha_y^*(f^*(\CF'))=\alpha_y^*(\CF).
 \]
 This shows that $\alpha_y^*(\CF)$ is also invariant under taking isomorphic coarse actions.
\end{proof}

\begin{cor}
 If $\crse G=(G,\CE)$ is coarsely connected then every isomorphism class of cobounded left coarse actions contains a unique representative of the form $\cop\colon(G,\CE)\curvearrowright(G,\CF)$ where $\CF\supseteq\CE$ is an equi left invariant coarse structure.
\end{cor}

\section{The Fundamental Observation of Geometric Group Theory}
\label{sec:p1:Milnor_svarc}

Recall that a function between metric spaces $f\colon(X,d_X)\to(Y,d_Y)$ is quasi\=/Lipschitz if there exist constants $L,A$ such that $d_Y(f(x),f(x'))\leq L d_X(x,x')+A$ for every $x,x'\in X$. A \emph{quasi\=/isometry}\index{quasi-!isometry} is a coarse equivalence between metric spaces $f\colon (X,d_X)\to(Y,d_Y)$ such that both $f$ and its coarse inverse $f^{-1}$ are quasi\=/Lipschitz. Further recall that a metric space is a \emph{length space}\index{length space} if it is path connected and the distance between any two points equals the infimum of the lengths of paths connecting them. 

The Milnor--Schwarz  Lemma states that if a set\=/group $G$ acts properly and cocompactly\footnote{%
An action $G\curvearrowright Y$ is cocompact if there exists a compact $K\subseteq Y$ such that $G\cdot K=Y$.
} on a proper length space $(Y,d)$ then $G$ is finitely generated and the orbit map is a quasi\=/isometry between $(Y,d)$ and $G$ equipped with the word metric associated with any finite generating set. Here ``properly'' means that set $\braces{g\in G\mid (g\cdot K)\cap K\neq\emptyset}$ is finite for every compact $K\subseteq Y$. Because of its importance, this result is also known as ``the fundamental observation of geometric group theory''. Although not hard, its proof usually takes several pages. In this section we show that in our settings it is essentially contained in Proposition~\ref{prop:p1:neighbourhoods.of.identity.generate} and Lemma~\ref{lem:p1:pullback orbit coarse actions} (whose proofs further simplify for trivially coarse groups).
Our approach is analogous to an argument of Brodskiy--Dydak--Mitra \cite{brodskiy_svarc-milnor_2007} which is also used in \cite{cornulier2016metric}.

Let $G$ be a set\=/group, so that $\crse G = (G,\mincrs)$ is a trivially coarse group, and let $\crse \alpha\colon(G,\mincrs)\curvearrowright(Y,\CF)$ be a cobounded coarse action. For any fixed $y\in Y$, Lemma~\ref{lem:p1:pullback orbit coarse actions} implies that the orbit map $g\mapsto g\cdot y$ is an equivariant coarse equivalence $(G,\alpha_y^*(\CF))\to(Y,\CF)$. We now give a temporary definition (see Definition~\ref{def:appendix:proper action} for a more conceptual version):

\begin{de*}[Temporary]
 The coarse action $\crse{\alpha \colon G\curvearrowright Y} =(Y,\CF)$ is \emph{proper}\index{proper!coarse action}\index{coarse action!proper} if the set $\braces{g\in G\mid (g\cdot B)\cap B\neq \emptyset}$ is finite for every bounded set $B\subseteq Y$. 
 We also say that a set\=/group action by isometries on a metric space $(Y,d)$ is \emph{metrically\=/proper}\index{proper!action (metrically)} if the induced coarse action $\crse \alpha \colon (G,\mincrs)\curvearrowright (Y,\CE_d)$ is proper.
\end{de*}

\begin{rmk}
 Explicitly, a set\=/group action by isometries $G\curvearrowright (Y,d)$ is metrically proper if for every $R\geq 0$  and $y\in Y$ there are at most finitely many $g\in G$ so that $d(y,g\cdot y)\leq R$.
\end{rmk}

\begin{lem}\label{lem:p1:naive proper iff bounded are finite}
 Let $\crse Y=(Y,\CF)$ be coarsely connected. A coarse action $\crse{\alpha\colon G\curvearrowright Y}$ is proper if and only if every $\alpha_y^*(\CF)$\=/bounded subset of $G$ is finite.  
\end{lem}
\begin{proof}
 By definition, $A\subseteq G$ is $\alpha_y^*(\CF)$\=/bounded if and only if $A\cdot y$ is $\CF$\=/bounded. Letting $B\coloneqq \braces{y}\cup (A\cdot y)$ we see that if $\alpha$ is proper then $A$ must be finite. Vice versa, let $B$ be any $\CF$\=/bounded subset of $Y$. Enlarging $B$ if necessary, we may assume that $y\in B$. Then the set $\braces{g\in G\mid g\cdot B\cap B\neq \emptyset}$ is a subset of $\braces{g\in G\mid g\cdot y\in \st(B,\pts{G}\cdot B)}$, which is $\alpha_y^*(\CF)$\=/bounded.
\end{proof}

Notice also that if $(Y,\CF)$ is coarsely connected and $\crse{G\curvearrowright Y}$ is a coarse action then every finite subset of $G$ is of course $\alpha_y^*(\CF)$\=/bounded. Together with Lemma~\ref{lem:p1:naive proper iff bounded are finite}, this shows that $\crse{G\curvearrowright Y}$ is a coarse action if and only if the set of $\alpha_y^*(\CF)$\=/bounded subsets of $G$ coincides with the set of all the finite subsets of $G$. However, we know from Proposition~\ref{prop:p1:neighbourhoods.of.identity.generate} that an equi left invariant coarse structure is determined by the bounded sets of the coarse structure. In particular, $\alpha_y^*(\CF)=\varcrs[left]{fin}$ does not depend on the specific choice of proper coarse action. We will soon see that this is a key point of the Milnor--Schwarz  Lemma.

Other facets of the Milnor--Schwarz Lemma concern the finite generation of the set\=/group $G$ acting on $(Y,d)$ and showing that the orbit maps are quasi\=/isometries (as opposed to coarse equivalences). Both these points are tightly tied with the assumption that the metric space $(Y,d)$ is a length space. It is a simple observation that a coarse equivalence between length spaces is automatically a quasi\=/isometry (see also Corollary~\ref{cor:appendix:quasi geo quasi Lipschitz}).

Recall that a coarse space $(X,\CE)$ is coarsely geodesic (Definition~\ref{def:p1:crs_geod}) if it is coarsely connected and the coarse structure $\CE$ is generated by a single set $\CE=\angles{\bar E}$. That is, $E\subseteq X\times X$ is in $\CE$ if and only if it is a subset of the $n$\=/fold composition $\bar E\cmp\cdots\cmp \bar E$ for some $n\in\NN$. 
Notice that if $(Y,d)$ is a length metric space then $(Y,\CE_d)$ is coarsely geodesic because $\CE_d$ is generated \emph{e.g.}\ by the $1$\=/neighborhood of the diagonal $\Delta_{Y}\subseteq Y\times Y$. 

\begin{cor}[Milnor--Schwarz  Lemma]\label{cor:p1:Milnor-Svarc}
 If $\alpha\colon G\curvearrowright (Y,d)$ is a metrically\=/proper cocompact action by isometries of a set\=/group $G$ on a length space $Y$, then $G$ is finitely generated and the orbit map is a quasi\=/isometry when $G$ is equipped with a word metric.
\end{cor}
\begin{proof}
 Since the action is by isometries, it defines a coarse action $\crse \alpha \colon (G,\mincrs)\curvearrowright (Y,\CE_d)$. 
 Since $\alpha$ is metrically\=/proper, $ \crse \alpha$ is a proper coarse action and hence $\alpha_y^*(\CE_d)$ is the unique equi left invariant coarse structure on $G$ whose bounded sets are precisely the finite subsets of $G$. 
 Every compact set in $Y$ is $\CE_d$\=/bounded, so cocompactness of $\alpha$ implies that $\crse \alpha$ is a cobounded coarse action. It follows that the orbit map is a coarse equivalence $(G,\alpha_y^*(\CE_d))\to (Y,\CE_d)$. 
 
 Since $Y$ is a length space, $(Y,\CE_d)$ is coarsely geodesic and therefore so is $(G,\alpha_y^*(\CE_d))$. Choose a generator for the coarse structure  $\alpha_y^*(\CE_d)=\angles{\bar E}$. Since $G$ is a set\=/group and $\alpha_y^*(\CE_d)$ is equi left invariant, we can enlarge $\bar E$ to make it into a relation of the form $\bar E=\Delta_{G}\ast (S\times \{e\})$ for some $\alpha_y^*(\CE_d)$\=/bounded set $S$---for instance letting $S$ be the section $S\coloneqq(\Delta_{G}\ast \bar E)_{e}$. Since it is bounded, $S$ is finite. Notice that the $n$\=/fold composition, $\bar E\cmp\cdots \cmp\bar E$ is equal to $\Delta_G\ast(S^n\times\{e\})$. Since $(G,\alpha_y^*(\CE_d))$ is coarsely connected, for every $g\in G$ there exists $n\in\NN$ with $g\torel{\bar E\cmp\cdots \cmp\bar E}e$. This means that $g\in S^n$ and hence $S$ is a finite generating set for $G$.\footnote{%
Compare also with \cite[Theorems 9.2 and 9.7]{protasov2003ball}.}

 Finally, the (coarse) action by left multiplication of $G$ on itself equipped with the metric coarse structure $\CE_{d_S}$ induced by the word metric $d_S$ defined by the finite generating set $S$ is proper and cobounded. By uniqueness, it follows that $\CE_{d_S}=\alpha_y^*(\CE_d)$. That is, the orbit map is a coarse equivalence between $(G,\CE_{d_S})$ and $(Y,\CE_d)$. Since both metrics are length metrics, the orbit map is a quasi\=/isometry.
\end{proof}

\begin{rmk}
 The standard proof of the Milnor--Schwarz  Lemma is obtained by studying how the orbit $G\cdot y$ sits in the metric space $(Y,d)$. Using the metric assumptions on $Y$ and the existence of a proper cocompact action by isometries it is possible to deduce that $G$ itself must be finitely generated and bound the distortion of the orbit map. That is, the $G$\=/action is used to control the coarse geometry of $Y$ and the latter is then used to study the coarse geometry of $G$.
 
 The proof that we just presented displays a subtle change of paradigm. Since coarse geometry is well behaved under pull\=/backs, we can directly use the action to pull\=/back the coarse structure of $Y$ to $G$. The advantage of doing this is that we are then left to study a coarse structure \emph{on a set\=/group} and here we can use uniqueness results `for free'.
\end{rmk}

It is interesting to note that the Milnor--Schwarz  Lemma can be split into two separate statements.  
The first one is implicit in Proposition~\ref{prop:p1:neighbourhoods.of.identity.generate} and implies that a cobounded coarse action is uniquely determined by its `degree of properness'. That is, cobounded coarse actions of $G$ are uniquely determined by the family of subsets of $G$ that have bounded image under the orbit map (see Chapter~\ref{sec:appendix:proper coarse actions} for more on this).

The second part of the statement is more specialized. It is the observation that if the (unique) space admitting a metrically\=/proper cobounded $G$\=/action is coarsely geodesic then the group is finitely generated. These two statements can then be combined by observing that, for a finitely generated group, the action on the Cayley graph is proper and cobounded and is hence the model of every metrically\=/proper cobounded action.

\section{Quotient Coarse Actions}\label{sec:p1:quotient coarse action} 
Let $\crse{\alpha}\colon \crse G\curvearrowright \crse Y$ be a coarse action and $\crse q\colon \crse Y\to\crse Q$ a coarse quotient of coarse spaces (Section \ref{sec:p1:subspaces and quotients}).

\begin{de}
 If there exists a coarse action $\bar{\crse \alpha}\colon \crse G\curvearrowright \crse Q$ such that the diagram commutes:
\[ 
\begin{tikzcd}
    \crse{G}\times \crse Y \arrow[r, "{\crse \alpha}"] \arrow[d,swap, "{\cid_{\crse G}\times \crse q}"] & \crse{Y}      \arrow[d, "{\crse q}"]\\
    \crse{G}\times \crse Q \arrow[r, "\bar{\crse \alpha}"]& \crse Q . 
\end{tikzcd} 
\]
We say that $\bar{\crse \alpha}$ is a \emph{quotient coarse action}\index{quotient!coarse action}\index{coarse action!quotient} for $\crse{\alpha}$. That is, $\bar{\crse{\alpha}}$ is a coarse action such that the quotient map $\crse{q}\colon \crse Y\to \crse Q$ is coarsely equivariant.
\end{de}

\begin{rmk}
 With this language, Proposition~\ref{prop:p1:cobounded coarse actions isom left mult} implies that every cobounded coarse action of $\crse G$ is a coarse quotient of the action by left multiplication of $\crse G$ on itself.
\end{rmk}

Recall from Section~\ref{sec:p1:coarse actions def} that a coarse structure $\CF'$ on $Y$ is equi left invariant if $\Delta_G\cdot F\in \CF'$ for every $F\in \CF'$, where we use the notation
\[
 E\cdot F\coloneqq (\alpha\times\alpha) (E\otimes F).
\]

\begin{de}\label{def:p1:coarsely equivariant sets of relations}
 Let $\crse{\alpha\colon G\curvearrowright Y}=(Y,\CF)$ be a coarse action. A family $\CR$ of relations on $Y$ is \emph{coarsely equivariant}\index{coarsely!equivariant (relations)} (or, \emph{$\crse G$\=/equivariant}) if $\Delta_{G}\cdot R\in \angles{\CF,\CR}$ for every $R\in\CR$ (the definition does not depend on the choice of representative for $\crse \alpha$). 
\end{de}

Since $\crse \alpha$ is a coarse action, $\CF$ is equi left invariant. It follows that $\CR$ is coarsely equivariant if and only if the coarse structure $\angles{\CF,\CR}$ is equi left invariant.
Recall that $\crse Y/\CR$ denotes the quotient coarse space $(Y,\angles{\CF,\CR})$ and let $\crse q=[\id_Y]\colon \crse Y\to\crse Y/\CR$ be the quotient map. Definition~\ref{def:p1:coarsely equivariant sets of relations} is designed so that the following holds.

\begin{lem}
 A family of relations $\CR$ is coarsely equivariant if and only if $\crse \alpha$ descends to a quotient coarse action $\bar{\crse \alpha}\colon \crse G\times (\crse Y/\CR)\to \crse Y/\CR$.
\end{lem}
\begin{proof}
Assume that $\CR$ is $\crse G$\=/equivariant, and let $\crse G=(G,\CE)$. Thus $\angles{\CF,\CR}$ is equi left invariant. Since for every $E\in\CE$ the image $\alpha(E\otimes\Delta_Y)$ is in $\CF\subseteq\angles{\CF,\CR}$, it immediately follows from Lemma~\ref{lem:p1:equi controlled.sections.iff.controlled} that $\alpha\colon (G,\CE)\times(Y, \angles{\CF,\CR})\to (Y, \angles{\CF,\CR})$ is controlled.
Since $\CF\subseteq\angles{\CF,\CR}$ it is also clear that the Action Diagrams for $\alpha$ commute up to $\angles{\CF,\CR}$\=/closeness and hence $\alpha$ defines a quotient coarse action $\bar{\crse\alpha}\colon \crse G\curvearrowright (Y, \angles{\CF,\CR})$.
 
 Vice versa, if there exists a coarse map $\bar{ \crse \alpha}$ so that the diagrams commute, then we must have $\bar{ \crse \alpha}=[\alpha]_{\angles{\CF,\CR}}$ because $\crse q=[\id_Y]_{\angles{\CF,\CR}}$. In particular, we may pick $\alpha$ as a representative for $\bar{\crse\alpha}$, and $\alpha$ is only controlled if $\CR$ is coarsely equivariant. 
\end{proof}

We thus obtain an analog of Proposition~\ref{prop:p1:quotient.subgroup}:

\begin{prop}\label{prop:p1:quotient action generated by R}
 Given a coarse action $\crse{\alpha\colon G\curvearrowright Y}=(Y,\CF)$ and any set of relations $\CR$ on $Y$, the minimal equi left invariant coarse structure $\varFcrs[left]{\CF,\CR}$ containing $\CF$ and $\CR$ is equal to
 \[
  \varFcrs[left]{\CF,\CR}=\angles{\CF,\Delta_{G}\cdot \CR}.
 \]
 The coarse action $\crse \alpha$ descends to a quotient coarse action $\bar{\crse \alpha}\colon\crse G\curvearrowright (Y,\varFcrs[left]{\CF,\CR})$. Every quotient coarse action of $\crse\alpha$ that sends $\CR$ into controlled entourages must be a coarse quotient of $\bar{\crse\alpha}$.
\end{prop}
\begin{proof}
 It is clear that $\varFcrs[left]{\CF,\CR}$ must contain $\angles{\CF,\Delta_{G}\cdot \CR}$, so it is enough to verify that the latter is equi left invariant and contains $\CR$. Notice that the family $\Delta_{G}\cdot \CR = \bigbraces{\Delta_{G}\cdot R\mid R\in\CR}$ is coarsely equivariant because
 \[
  g_1\cdot(g_2\cdot y)
  \rel{\CF} (g_1\ast g_2)\cdot y
  \torel{\Delta_G\cdot R} (g_1\ast g_2)\cdot y'
  \rel{\CF} g_1\cdot (g_2\cdot y')
  \qquad \forall g_1,g_2\in G,\ y\torel{R}y'.
 \]
 We similarly see that $\angles{\CF,\Delta_{G}\cdot \CR}$ contains $\CR$ because
 \[
  y
  \rel{\CF} e_G\cdot y
  \torel{\Delta_G\cdot R} e_G\cdot y'
  \rel{\CF} y'
  \qquad \forall y\torel{R}y'.
 \]
 The second part of the statement is essentially a tautology.
\end{proof}

\begin{exmp}
 If $\alpha\colon G\curvearrowright Y$ is a set\=/group action and $R$ is an \emph{equivalence relation} on $Y$, then $\alpha$ factors through the quotient set $Y/R$ if and only if $R$ is $G$\=/equivariant.
Since $R$ is an equivalence relation, the coarse structure $\angles{R}$ coincides with the set of subsets of $R$. If we now see $\alpha$ as a trivially coarse action $\crse \alpha \colon(G,\mincrs)\curvearrowright(Y,\mincrs)$, then it is easy to observe that an equivalence relation $R$ is $G$\=/equivariant if and only if it is coarsely equivariant. When this happens, we see that $\varFcrs[left]{\mincrs}{}^{}_{,R}=\angles{R}$. We may also observe that the trivially coarse action on the quotient set $(G,\mincrs)\curvearrowright (Y/R,\mincrs)$ is isomorphic (as a coarse action) to the quotient coarse action of $\bar{\crse\alpha}\colon (G,\mincrs)\curvearrowright (Y,\angles{R})$.
\end{exmp}

\begin{exmp}\label{exmp:p1:metric action as quotient of trivial}
 Let $\alpha\colon G\curvearrowright(Y,d)$ be an isometric action on a length space and let $R_1\coloneqq\{(y_1,y_2)\mid d(y_1,y_2)\leq 1\}$. Then $R_1$ is not an equivalence relation, but it is still a $G$\=/equivariant relation on $Y$.
 As a consequence, we see that $\varFcrs[left]{\mincrs}{}^{}_{,R_1}=\angles{R_1}$ and hence $\alpha$ induces a quotient coarse action $\crse \alpha \colon (G,\mincrs)\curvearrowright(Y,\angles{R_1})$. Since $Y$ is a length metric space, $\angles{R_1}$ coincides with the metric coarse structure $\CE_d$, \emph{i.e.}\ the quotient coarse action is nothing but $\crse \alpha \colon (G,\mincrs)\curvearrowright(Y,\CE_d)$. If the action $\alpha$ is not by isometries, $R_1$ need not be $G$\=/invariant. On the other hand, it is easy to see that $R_1$ is coarsely equivariant if and only if the action is by uniform coarse equivalences. This was to be expected, because $\angles{R_1}=\CE_d$ and the latter condition determines those actions such that $\crse \alpha \colon (G,\mincrs)\curvearrowright(Y,\CE_d)$ is a coarse action (Example~\ref{exmp:p1:metric action by unif coarse eq}).
\end{exmp}

\begin{rmk}
 A similar argument shows that the notation introduced here is compatible with that of Example~\ref{exmp:p1:topological left multiplication}: the coarse structure $\varFcrs[left]{cpt}$  defined there is the smallest equi left invariant coarse structure on $Y$ so that compact subsets of $Y$ are bounded.
\end{rmk}

\section{Coarse Quotient Actions of the Action by Left Multiplication}
\label{sec:p1:quotients by left mult}
This is a technical section that will be beneficial in the sequel.
It is useful to study coarse quotients of the coarse action by left multiplication of a coarse group $\crse G =(G,\CE)$ on itself. By Proposition~\ref{prop:p1:quotient action generated by R}, all we need to do is to understand the coarse structures of the form $\varcrs[left]{\CE,\CR}=\angles{\CE,\Delta_G\ast \CR}$ (see also Remark~\ref{rmk:p1:generated leftinvariant coarse structure}).
It is much easier to describe such a coarse structure if the operation functions $\ast,\unit,\inversefn$ are adapted (we are using notations and definitions from Chapter~\ref{sec:p1:making sets into groups}). In fact, we can prove the following:
 
\begin{lem}\label{lem:p1:quotient coarse structure as composition}
 Let $(G,\CE)$ be a coarse group, $\CR$ a set of relations on $G$, and $\ast,\unit,\inversefn$ adapted representatives for the coarse operations.
 Then   $\varcrs[left]{\CE,\CR}=
  \angles{\Delta_{G}\ast(\Delta_{G}\ast \CR)}\cmp \CE$, where 
 \[
  \angles{\Delta_{G}\ast(\Delta_{G}\ast \CR)}\cmp \CE
  \;\coloneqq\;
  \braces{E_1\cmp E_2\mid E_1\in \angles{\Delta_{G}\ast(\Delta_{G}\ast \CR)},\ E_2\in \CE}.
 \]
\end{lem}

\begin{proof}
We only need to show that $\varcrs[left]{\CE,\CR}\subseteq \angles{\Delta_{G}\ast(\Delta_{G}\ast \CR)}\cmp \CE$, as the other containment is obvious.
By Lemma~\ref{lem:p1: generated coarse structure}, $\varcrs[left]{\CE,\CR}$ consists of the relations contained in finite compositions of relations in $\CE$, or of the form $(\Delta_{G}\ast R)\cup \Delta_G$ for some $R\in\CR$:
\begin{equation}\label{eq:p1:composition of relations for left-invariance}
 E_1\cmp\paren{(\Delta_{G}\ast R_1)\cup \Delta_G}\cmp E_2\cmp\paren{(\Delta_{G}\ast R_2)\cup \Delta_G}\cmp\cdots\cmp E_n\cmp\paren{(\Delta_{G}\ast R_n)\cup \Delta_G}.
\end{equation}
To prove the lemma we have to show that, up to replacing $\Delta_G\ast \CR$ with $\Delta_G\ast(\Delta_G\ast \CR)$, we can rearrange the composition in such a way that the contributions from $\CR$ come first and are then followed by composition with some relation in $\CE$.

By distributivity of compositions over unions, $E\cmp\paren{(\Delta_{G}\ast R)\cup \Delta_G}= E\cmp(\Delta_{G}\ast R) \cup E$. If a pair $(x,y)$ belongs to $E\cmp (\Delta_{G}\ast R)$ then there are $(z,w)\in R$ and $g\in G$ such that $y=g\ast w$ and $x\torel{E} g\ast z\torel{\Delta_{G}\ast R}g\ast w=y$.
 Since the operations are adapted, we have:
 \begin{align*}
  x=x\ast (z^{-1}\ast z)
  &\torel{\Delta_{G}\ast(\Delta_{G}\ast R)}x\ast(z^{-1}\ast w) \\
  &\torel{E\ast \Delta_{G}} (g\ast z)\ast(z^{-1}\ast w) \\
  &\torel{\assRel\cmp(\Delta_{G}\ast\assRel)} g\ast((z\ast z^{-1})\ast w)
  =g\ast w=y,
 \end{align*}
 thus showing that $E\cmp (\Delta_{G}\ast R)\subseteq (\Delta_{G}\ast(\Delta_{G}\ast R))\cmp E'$, where $E'=(E\ast \Delta_{G})\cmp \assRel\cmp(\Delta_{G}\ast\assRel)\in \CE$.

 We thus showed that $E_1\cmp\paren{(\Delta_{G}\ast R_1)\cup \Delta_G}$ is contained in some $\paren{\Delta_G\ast(\Delta_{G}\ast R)}\cmp E_1' \cup E_1$. By associativity and distributivity of composition, if we compose the latter with $E_2\cmp\paren{(\Delta_{G}\ast R_2)\cup \Delta_G}$ the resulting relation is contained in
 \begin{align*}
  \Bigparen{\paren{\Delta_G\ast(\Delta_{G}\ast R_1)}\cmp E_1''\cmp\paren{(\Delta_{G}\ast R_2)\cup \Delta_G}}
  \cup
  \Bigparen{E_1''\cmp\paren{(\Delta_{G}\ast R_2)\cup \Delta_G}}
 \end{align*}
 for some $E_1''\in\CE$. By distributivity again, this is contained in
 \begin{align*}
  \Bigparen{\paren{\Delta_G\ast(\Delta_{G}\ast R_1)}\cmp E_1''\cmp(\Delta_{G}\ast R_2)}
  \cup
  \Bigparen{\paren{\Delta_G\ast(\Delta_{G}\ast R_1)}\cmp E_1''}
  \cup
  \Bigparen{E_1''\cmp(\Delta_{G}\ast R_2)}
  \cup
  {E_1''}
 \end{align*}
 and we may then apply the above argument to show that $E_1''\cmp(\Delta_{G}\ast R_2)\subseteq \Delta_G\ast(\Delta_{G}\ast R_2) \cmp E_{2}'$ for some other $E_2'\in\CE$.
 Iterating this process, we deduce that the relation \eqref{eq:p1:composition of relations for left-invariance} is contained in a finite union of finite compositions of the form
 \[
  \paren{\Delta_G\ast(\Delta_{G}\ast R_{i_1})}\cmp
  \paren{\Delta_G\ast(\Delta_{G}\ast R_{i_2})}\cmp
  \paren{\Delta_G\ast(\Delta_{G}\ast R_{i_k})}\cmp E_k'
 \]
 and hence belongs to $\angles{\Delta_{G}\ast(\Delta_{G}\ast \CR)}\cmp \CE$ as claimed.
\end{proof}

Of course, if the representative $\ast$ is associative $\Delta_{G}\ast(\Delta_{G}\ast R)$ is simply equal to $\Delta_{G}\ast R$, hence Lemma~\ref{lem:p1:quotient coarse structure as composition} takes a more pleasing form:

\begin{cor}
 If $\crse G$ is a coarsified set\=/group then $\varcrs[left]{\CE,\CR} =\angles{\Delta_{G}\ast \CR}\cmp \CE$.
\end{cor}

\begin{rmk}\label{rmk:p1:metric on quotient}
 The meaning of Lemma~\ref{lem:p1:quotient coarse structure as composition} is best understood by examining the metric analog. If $(Y,d)$ is a metric space and $R$ is an equivalence relation we define a function on the quotient $\bar d\colon Y/R\times Y/R\to \RR$ by setting $\bar d([y],[y'])$ to be the infimum of the distances between points in $[y]$ and $[y']$. This $\bar d$ may fail to satisfy the triangle inequality. To obtain a distance function on $Y/R$ one may define $\bar d'([y],[y'])$ as 
 %the infimum of the sums of $\bar d$\=/distances of pairs of consecutive points in finite sequences of points in $Y/R$ connecting $[y]$ to $[y']$ 
 \[
 \bar d'([y],[y']) = \inf\Bigbraces{\sum_{i=1}^n \bar d([y_{i}],[y_{i-1}])\Bigmid n\in\NN,\ [y_i]\in Y/R,\ [y_0]=[y],\ [y_n]=[y']}
 \]
(and then check whether $\bar d'([y],[y'])=0$ if and only if $[y]=[y']$).  
This is analogous to writing $\angles{\CE,R}$ as arbitrary finite compositions of relations.
 
 On the other hand, if $R$ is preserved by an isometric group action $G\curvearrowright (Y,d)$ that is transitive on each equivalence class of $R$ (\emph{i.e.}\ $R$ is the orbit equivalence relation of an isometric action), then $\bar d$ does automatically satisfies the triangle inequality and it is hence a pseudo\=/metric. If we assume that the $R$\=/equivalence classes are closed in $Y$ then $\bar d$ is a metric.
 Since $\bar d$ is given by an explicit formula, it is much easier to study metric properties of the quotient $Y/R$. Case in point, the $\bar d$\=/ball of radius $r$ with center $[y]$ in $Y/R$ is the set of $[y']$ that intersect the $d$\=/neighborhood of radius $r$ of $[y]$ in $Y$. This is an analog of the characterization of $\varcrs[left]{\CE,\CR}=\angles{\CE,\CR}$ as $\angles{\Delta_{G}\ast \CR}\cmp \CE$: it says that a $\varcrs[left]{\CE,\CR}$\=/bounded neighborhood of a point $g\in G$ can be recognized by seeing whether it is an $\CE$\=/bounded neighborhood of a $\angles{\Delta_{G}\ast \CR}$\=/bounded neighborhood of $g$.
\end{rmk}

\section{Coarse Cosets Spaces: Subsets}\label{sec:p1:cosets of subsets}
A subclass of special interest among the quotient coarse actions of the left multiplication of $\crse G$ on itself consists of the coarse actions on `cosets spaces'. Let $\crse G=(G,\CE)$ be a coarse group and $\crse A=[A]$ some coarse subset of $\crse G$.

\begin{de}\label{def:p1: coset space}
 The \emph{(left) coarse cosets space $\crse{G/A}$}\index{coarse!coset space} is the coarse quotient 
 \[
   \crse{G / A}\coloneqq\crse{G}/\braces{\Delta_{G}\ast (A\times A)}=(G,\CE^{\rm left}_{\CE,A\times A})
 \]
 \nomenclature[:COS]{$\crse{G/ A}$}{coarse coset space}
 (this is well defined, as $\CE^{\rm left}_{\CE,A\times A}$ does not depend on the choice of representative for $\crse A$ nor $\cop$).
\end{de}

\begin{exmp}\label{exmp:p1:cosets are cosets}
 Let $G$ be a set\=/group and fix $A\subseteq G$. Then $E\in\angles{\Delta_{G}\ast (A\times A)}$ if and only if there exists some $n\in\NN$ such that for every $(x,y)\in E$ we have $x^{-1}\ast y\in (A^{-1}\ast A)^{\ast n}$. In particular, if $A=H$ is a set\=/subgroup then $\Delta_{G}\ast (H\times H)=\Delta_{G}\ast (\{e\}\times H)$  is an equivalence relation where $x\sim y$ if and only if $xH=yH$. Then $\angles{\Delta_{G}\ast (H\times H)}$ is the coarse structure consisting of all the sub\=/relations of $\sim$ and the coarse quotient $(G,\angles{\Delta_{G}\ast (H\times H)})$ is naturally isomorphic to $(G/H,\mincrs)$. That is, if $\crse G=(G,\mincrs)$ and $\crse H=[H]$, then $\crse{G/H}$ is (equivariantly) coarsely equivalent to the trivial coarsification of the usual coset space $G/H$. 
 This justifies our choice of nomenclature.
\end{exmp}

Notice that by Proposition~\ref{prop:p1:quotient action generated by R}, we know that for every coarse subset $\crse {A\subseteq G}$ the coarse group operation $\cop$ descends to a quotient coarse action $\bar{\cop}\colon \crse G\curvearrowright \crse G/\crse A$. This should be thought of as the action by left multiplication on the left cosets of $\crse A$.

\begin{rmk}
 Given a coarse action $\crse{G\curvearrowright Y}=(Y,\CF)$ and coarse point $\crse {y \in Y}$, it follows from Lemma~\ref{lem:p1:pullback orbit coarse actions} that the orbit map $\crse{\alpha_{y}}\colon\crse G\to\crse{G\cact y}$ is coarsely $\crse G$\=/equivariant.
 More precisely, by Proposition~\ref{prop:p1:cobounded coarse actions isom left mult} we see $\crse{\alpha_{y}}$ is an isomorphism between $\cop\colon\crse G\curvearrowright (G,\alpha_y^*(\CF))$ and  the restriction coarse action $\crse{\alpha|_{G\cact y}}\colon\crse G\curvearrowright\crse{G\cact y}$.
 By Proposition~\ref{prop:p1:quotient action generated by R}, $\CE^{\rm left}_{\CE,A\times A}=\angles{\CE, \Delta_{G}\ast (A\times A)}$ is the minimal coarse structure containing $\CE$ and $A\times A$ such that $\ast$ defines a $\crse G$\=/action. In particular, if the coarse image $\crse{\alpha_{y}}(\crse A)$ is bounded in $\crse Y$, it follows that $\CE^{\rm left}_{\CE,A\times A}\subseteq \alpha_y^*(\CF)$. In other words, $\crse{\alpha|_{G\cact y}}$ is a quotient coarse action of the coarse action on the coarse coset space $\crse G\curvearrowright\crse{G/A}$. This can be seen as a version of the Orbit--Stabilizer Theorem. Informally, $\crse{\alpha_{y}}(\crse A)$ is bounded in $\crse Y$ if and only if ``it is contained in the stabilizer of $\crse y$'' and this happens if and only if the coarse action of $\crse G$ on the coarse orbit $\crse{G\cact y}$ factors through the action on the left cosets of $\crse A$.
 
 There are a number of caveats to the above informal description. To begin with, $\crse A$ need not fix $\crse y$ unless we know a priori that $e_{G}\prec A$. This issue can be fixed by replacing $\crse{A}$ with $\crse a^{-1}\cop\crse A$ for some $\crse{a\in A}$ (for every $\crse{ g\in G}$ the coarse cosets spaces $\crse{G/A}$ and $\crse{G/(g\cop A})$ are naturally isomorphic). This fact points to the second caveat: we did not assume that $\crse A$ is a coarse group. 
 More importantly, it is generally not possible to uniquely define a ``coarse stabilizer'' of $\crse y$ because there need not be a ``maximal'' coarse subset $\crse A$ that stabilizes $\crse y$. For instance, let $d$ be a left invariant metric on a set\=/group $G$ as in Example~\ref{exmp:p1:metric action as quotient of trivial}, but this time assume that the coarse structure $\CE_d$ cannot be generated by a single relation (\emph{i.e.}\ $(G,\CE_d)$  is not coarsely geodesic). We then see that $\CE_d\neq \varcrs[left]{A\times A}$ for any $A\subseteq G$. 
 This issue is related to the fact that the coarse preimage of a set under a coarse map need not be well\=/defined. We will encounter similar difficulties when trying to define coarse kernels of coarse homomorphisms (Section~\ref{ch:p1:coarse kernels}). 
 On the positive side, one can find reasonable replacements for stabilizers by considering \emph{bornologies} and proper coarse actions (Section~\ref{sec:appendix:proper coarse actions}).
\end{rmk}

We will later need a technical lemma that gives us some extra control to coset coarse structures. Namely, when specialized to coarse coset spaces Lemma~\ref{lem:p1:quotient coarse structure as composition} can be further improved as follows:

\begin{lem}\label{lem:p1:quotient coarse structure as composition_cosets}
 Let $(G,\CE)$ be a coarse group, $\ast,\unit,\inversefn$ adapted representatives for the coarse operations and $A\subseteq G$ any subset.
 Then 
 \[
  \CE^{\rm left}_{\CE,A\times A}
  \;=\;
  \angles{\Delta_{G}\ast(\{e\}\times (A^{-1}\ast A))}\cmp \CE
   \;=\;
  \CE\cmp\angles{\Delta_{G}\ast((A^{-1}\ast A)\times \{e\})}. 
 \]
 In particular, $E\in\CE^{\rm left}_{\CE,A\times A}$ if and only if there exist $n\in\NN$ and $E'\in\CE$ such that 
 \[
 E\subseteq \bigparen{\Delta_{G}\ast(\{e\}\times (A^{-1}\ast A)^{\ast n})}\cmp E',
 \]
 where $(A^{-1}\ast A)^{\ast n}$ denotes the $n$\=/fold product $\bigparen{\bigparen{(A^{-1}\ast A)\ast (A^{-1}\ast A)}\ast\cdots }\ast(A^{-1}\ast A)$.
\end{lem}
\begin{proof}
 To begin with, note that the equality $\angles{\ldots}\cmp\CE=\CE\cmp\angles{\ldots}$ follows immediately by taking symmetric relations $E\mapsto \op{E}$ and observing that symmetry exchanges the order of composition.
 
One containment is fairly simple to prove. If we assume that $e\in A$, then the set $A^{-1}\ast A$ is a $\CE^{\rm left}_{\CE,A\times A}$\=/bounded neighborhood of $e$ by Lemma~\ref{lem:p1:properties of neighborhoods of id}, hence the right hand side is contained in $\CE^{\rm left}_{\CE,A\times A}$. For an arbitrary set $A$, we can reduce to the case $e\in A$ by choosing an element $a\in A$ and considering $A'= a^{-1}\ast A$. In fact, the product $(A')^{-1}\ast A'$ is $\CE$\=/asymptotic to $A^{-1}\ast A$, so the latter is $\CE^{\rm left}_{\CE,A\times A}$\=/bounded if and only if the former is. Using left invariance and the fact that $(G,\CE)$ is a coarse group, we also see that $\CE^{\rm left}_{\CE,A\times A} =\CE^{\rm left}_{\CE,A'\times A'}$.
 
 For the converse containment, recall that by Lemma~\ref{lem:p1:quotient coarse structure as composition}, any relation in $ \CE^{\rm left}_{\CE,A\times A}$ is contained in a finite composition of the form
 \[
  \bigparen{\Delta_G\ast (\Delta_{G}\ast(A\times A))\;\cup\; \Delta_G}^{\cmp n}\cmp E'.
 \]
 Notice that $\Delta_G\ast (\Delta_{G}\ast(A\times A))\subseteq \assRel\cmp (\Delta_{G}\ast(A\times A))$, and that $\Delta_{G}\ast(\{e\}\times (A^{-1}\ast A))$ always contains the diagonal $\Delta_G$ because the operations are adapted. 
 To prove that $\CE^{\rm left}_{\CE,A\times A} = \angles{\Delta_{G}\ast(\{e\}\times (A^{-1}\ast A))}\cmp \CE$ it is thus enough to show that 
 \[
  E\cmp (\Delta_{G}\ast(A\times A))\subseteq (\Delta_{G}\ast(\{e\}\times (A^{-1}\ast A)))\cmp E'
 \]
 for some $E''\in\CE$.
 If $(x,y)\in E\cmp(\Delta_{G}\ast(A\times A))$, then there are $a_1,a_2\in A$ and $g\in G$ such that $y=g\ast a_2$ and $x\torel{E} g\ast a_1\torel{\Delta_{G}\ast (A\times A)}g\ast a_2=y$.
 Since the operations are adapted,
 \begin{align*}
  x=x\ast e
  &\torel{\Delta_{G}\ast(\{e\}\times(A^{-1}\ast A))}x\ast(a_1^{-1}\ast a_2) \\
  &\torel{E\ast \Delta_{G}} (g\ast a_1)\ast(a_1^{-1}\ast a_2) \\
  &\torel{\assRel\cmp(\Delta_{G}\ast\assRel)} g\ast((a_1\ast a_1^{-1})\ast a_2)
  =g\ast a_2=y
 \end{align*}
 as desired.

For the `in particular' part of the statement, notice that $(A^{-1}\ast A)^{\ast n}\subseteq(A^{-1}\ast A)^{\ast (n+1)}$ because $e\in A^{-1}\ast A$. It is thus enough to show that the $n$\=/fold composition $(\Delta_{G}\ast(\{e\}\times (A^{-1}\ast A)))^{\cmp n}$ is contained in $\Delta_{G}\ast(\{e\}\times (A^{-1}\ast A)^{\ast n})\circ E_n$ for some $E_n\in\CE$.
 If we are given
 \[
  x_0\torel{\Delta_{G}\ast (\{e\}\times(A^{-1}\ast A))} x_1\torel{\Delta_{G}\ast (\{e\}\times(A^{-1}\ast A)))} \cdots \torel{\Delta_{G}\ast (\{e\}\times(A^{-1}\ast A)))}x_n,
 \]
 then there are $b_i\in A^{-1}A$ such that $x_n=((x_0\ast b_1)\ast b_2)\cdots \ast b_n$. Since $((b_1\ast b_2)\cdots \ast b_n)\in (A^{-1}\ast A)^{\ast n}$, we have 
 \[
  x_0\torel{\Delta_{G}\ast(\{e\}\times (A^{-1}\ast A)^{\ast n})}
  x_0 \ast ((b_1\ast b_2)\cdots \ast b_n)
  \rel{\CE} x_n \qquad\forall x_0\in G,\ \forall b_1,\ldots,b_n\in A^{-1}\ast A
 \]
 which proves our claim.
\end{proof}

\begin{rmk}\label{rmk:p1:quotient as composition_cosets_covering}
 The meaning of Lemma~\ref{lem:p1:quotient coarse structure as composition_cosets} is perhaps more clear in terms of controlled coverings.
 A priori, a controlled covering for $\angles{\CE\;,\;\Delta_{G}\ast(A\times A)}$ is obtained by taking arbitrarily many iterated star neighborhoods of $\pts{G}\ast A$ and controlled coverings $\fka\in\fkC(\CE)$. Instead, Lemma~\ref{lem:p1:quotient coarse structure as composition_cosets} ensures us that every controlled covering is a refinement of a covering of the form 
 \[
  \st\Bigparen{\pts{G}\ast (A^{-1}\ast A)^{\ast n}\,,\,\fka}
 \]
 with $\fka\in\fkC(\CE)$. In particular, $\CE^{\rm left}_{\CE,A\times A}$\=/bounded sets are $\CE$\=/coarsely contained into left translates of $(A^{-1}\ast A)^{\ast n}$ for $n$ large enough.
\end{rmk}

\begin{exmp}
 Expanding on Example~\ref{exmp:p1:cosets are cosets} and Remark~\ref{rmk:p1:metric on quotient}, let $G$ be a set\=/group with a left invariant metric $d$ and let $H\leq G$ be a closed subgroup. Then $d$ defines a quotient metric $\bar d$ on the space of cosets $G/H$. 
 Consider now the metric coarse space $\crse G=(G,\CE_d)$ and the coarse coset space $\crse {G/ H}=(G,\CE^{\rm left}_{\CE,H\times H})$. 
 By Lemma~\ref{lem:p1:quotient coarse structure as composition_cosets} we deduce that the $\CE^{\rm left}_{\CE,H\times H}$\=/controlled coverings of $G$ are refinements of coverings by $R$\=/neighborhoods of the cosets $gH\subseteq G$ for some radius $R$ independent of $g\in G$. From this it follows that the set\=/quotient map $G\to G/H$ is an equivariant coarse equivalence from the coarse cosets space $\crse {G/ H}$ to $(G/H,\CE_{\bar d})$.
\end{exmp}

To conclude this section, it is important to know when $\bar{\cop}$ makes a coarse coset space $\crse{G/A}$ into a coarse group. More precisely, If we realize $\crse{G/A}$ as $(G,\CE^{\rm left}_{\CE,H\times H})$ then $\ast$ still denotes a binary operation on $\crse{G/A}$. In general, such $\ast$ will not be controlled with respect to $(G,\CE^{\rm left}_{\CE,H\times H})\otimes(G,\CE^{\rm left}_{\CE,H\times H})$. When it is controlled, we still denote the induced coarse map by $\bar\cop\colon \crse{G/A}\times\crse{G/A}\to\crse{G/A}$. In this case, it follows from Proposition~\ref{prop:p1:coarse.group.iff.heartsuit.and.equi controlled} that $(\crse{G/A},\bar \cop)$ is automatically a coarse group.
In other words, the coarse action $\bar\cop$ defines a coarse group if and only if it is possible to complete the following commutative diagram:
\[ 
\begin{tikzcd}
    \crse{G}\times \crse{G/A} \arrow[r, "{\bar\cop}"] \arrow[d,swap, "{\crse q\times\cid_{\crse{G/A}}}"] & \crse{G/A}.    \\
    \crse{G/A}\times \crse{G/A} \arrow[ur,swap,dashed, "\exists?\bar{\cop}"]&  
\end{tikzcd} 
\]

As one might expect, the next lemma shows that coarse normality is a sufficient condition.

\begin{lem}\label{lem:p1:cosets are groups}
 If $\crse {A\subseteq G}$ is coarsely invariant under conjugation (Definition~\ref{def:p1:normal}), then $\crse{G/A}$ is a coarse group. In the notation of Section~\ref{sec:p1:coarse quotients}, we have $\crse{G/A}=\crse{G}/\aangles{A\times A}=\crse{G}/\aangles{A\times\{e\}}$.
\end{lem}
\begin{proof}

 Fix $A\subseteq G$ and, for convenience, let $R=(A\times A)$ so that $\crse{G/A}= (G,\CE^{\rm left}_{\CE,R})$.
 Since 
 \[
  \ast\colon(G\times G\;,\;\CE\otimes\CE^{\rm left}_{\CE,R})\to (G,\CE^{\rm left}_{\CE,R})
 \]
 is controlled, we only need to check that $\ast\times \ast$ sends $\CE^{\rm left}_{\CE,R}\otimes \Delta_G$ into $\CE^{\rm left}_{\CE,R}$. 
 Writing $\CE^{\rm left}_{\CE,R}=\angles{\CE,\Delta_{G}\ast  R}$, we see that 
 \[
  \CE^{\rm left}_{\CE,R}\otimes\Delta_{G}
  =\angles{\CE\otimes\Delta_{G}\;,\;(\Delta_{G}\ast  R)\otimes\Delta_{G}}
  =\angles{\CE\;,\;(\Delta_{G}\ast  R)\otimes\Delta_{G}}.
 \]
 In other words, we only need to verify that $(\Delta_{G}\ast  R)\ast\Delta_{G}\in \CE^{\rm left}_{\CE,R}$.
 
 Since $A\subseteq G$ is coarsely invariant under conjugation, there exists $E\in\CE$ such that for every $g\in G$ and $a\in A$, $g\ast a\rel{E} a'\ast g$ for some $a'\in A$. Fix an element $\bar a\in A$, then
 \[
  (g_1\ast a)\ast g_2
  \rel{\assRel} g_1 \ast(a\ast g_2)
  \rel{\Delta_{G}\ast E}g_1\ast(g_2\ast a')
  \rel{\assRel} (g_1\ast g_2)\ast a'
  \rel{\Delta_{G}\ast R}(g_1\ast g_2)\ast \bar a.
 \]
 It then follows that
 \[
  (g_1\ast a_1)\ast g_2\rel{\CE^{\rm left}_{\CE,R}} (g_1\ast g_2)\ast \bar a
  \rel{\CE^{\rm left}_{\CE,R}}(g_1\ast a_2)\ast g_2
  \qquad \forall g_1,g_2\in G,\ \forall a_1,a_2\in A,
 \]
 \emph{i.e.}\ $(\Delta_{G}\ast  R)\ast\Delta_{G}\in \CE^{\rm left}_{\CE,R}$. This shows that $\ast$ is controlled, and hence $(\crse{G/A},\bar{\cop})$ is a coarse group.
 
 To connect with the notation of Chapter~\ref{sec:p1:coarse quotients}, recall that $\crse{G}/\aangles{A\times A}=(G,\CE^{\rm grp}_{\CE,R})$. By definition, we always have an inclusion $\CE^{\rm left}_{\CE,R}\subseteq \CE^{\rm grp}_{\CE,R}$, and we just proved that when $A$ is coarsely invariant under conjugation the reverse inclusion also holds. Taking compositions with symmetric relations, we also observe that $\CE^{\rm grp}_{\CE,A\times A}=\CE^{\rm grp}_{\CE,A\times \{e\}}$, hence we may also take $\crse{G}/\aangles{A\times\{e\}}$.
\end{proof}

\section{Coarse Cosets Spaces: Subgroups}\label{sec:p1:cosets of subgroups}
All the results of Section~\ref{sec:p1:cosets of subsets} can be improved when considering cosets of coarse subgroups.
To start, Lemma~\ref{lem:p1:quotient coarse structure as composition_cosets} implies the following:

\begin{cor}\label{cor:p1:quotient coarse structure as composition_cosets}
 Let $\crse G= (G,\CE)$ be a coarse group, $\ast,\unit,\inversefn$ adapted representatives for the coarse operations and $\crse{ H\leq G}$ a coarse subgroup.
 Then 
 \(E\in  \CE^{\rm left}_{\CE,H\times H} \) if and only if it is a subset of either (equivalently, both)
  \[
  \bigparen{\Delta_{G}\ast(\{e\}\times H)}\cmp E'
  \quad \text{ and/or }\quad
  E'\cmp\bigparen{\Delta_{G}\ast(H\times \{e\})}.
 \]
 for some $E'\in\CE$.
 Equivalently, $\CE^{\rm left}_{\CE,H\times H}$\=/controlled covering are refinements of coverings of the form $\st\bigparen{\pts{G}\ast H\,,\,\fka}$ with $\fka\in\fkC(\CE)$.
\end{cor}
\begin{proof}
 Since $H$ is a coarse subgroup, there exists some $E\in\CE$ so that $\{e\}\times(H^{-1}\ast H)^{\ast n}$ is contained in $(\{e\}\times H)\circ E$. Hence $\Delta_{G}\ast\paren{\{e\}\times(H^{-1}\ast H)^{\ast n}}\subseteq (\Delta_{G}\ast(\{e\}\times  H))\circ (\Delta_{G}\ast E)$ and $\Delta_{G}\ast E\in \CE$. 
\end{proof}

Secondly, there is a coarse analogue of the fact that the cosets of a set\=/subgroup partition the parent set\=/group:

\begin{lem}\label{lem:p1:cosets of subgroup partition}
 Let $\crse H$ be a coarse subgroup of $\crse G=(G,\CE)$. If there exists an $\CE$\=/bounded set $B\subseteq G$ such that both $B\cap (g_1\ast H)$ and $B\cap (g_2\ast H)$ are non\=/empty then $(g_1\ast H)\ceq_\CE (g_2\ast H)$.
\end{lem}
\begin{proof}
 Note that $g_1^{-1}\ast\paren{g_1\ast H}\ceq (g_1^{-1} \ast g_2)\ast H \ceq H$ and $g_1^{-1}\ast\paren{g_2\ast H}\ceq (g_1^{-1} \ast g_1)\ast H = g\ast H$, where we let $g\coloneqq g_1^{-1}\ast g_2$.
 Since left multiplication by $g_1^{-1}$ is a coarse equivalence of $(G,\CE)$ with itself, it is therefore enough to show that $H\ceq (g\ast H)$.
 
 By assumption, there exist $h_1,h_2\in H$ such that $h_1\ceq (g\ast h_2)$. Hence $g\ceq (h_1\ast h_2^{-1})\in (H\ast H)\ceq H$. It follows that $g\ast H \csub (H\ast H^{-1}) \ceq H$, as desired (the converse coarse containment is analogous).
\end{proof}

\begin{cor}
 For every $\crse g_1,\crse g_2\crse{\in G}$, either $\crse g_1\crse{\ast H=g}_2\crse{\ast H}$ or $\crse g_1\crse{\ast H}$ and $\crse g_2\crse{\ast H}$ are coarsely disjoint (\emph{i.e.}\ they never intersect the same coarsely connected component of $\crse G$). 
\end{cor}

\begin{cor}\label{cor:p1:bounded sets in cosets}
 Let $\crse{H}$ be a coarse subgroup of $\crse G=(G,\CE)$. Then, for all subsets $A,B\subseteq G$ the following equivalences hold:
 \[
  \Bigparen{A\csub_{\CE^{\rm left}_{\CE,H\times H}} B} \Leftrightarrow \Bigparen{A\csub_\CE H\ast B} \Leftrightarrow \Bigparen{A\csub_\CE B\ast H};
 \]
 \[
  \Bigparen{B\text{ is $\CE^{\rm left}_{\CE,H\times H}$\=/bounded}} \Leftrightarrow \Bigparen{B\csub_\CE g\ast H \text{ for some } g\in G}  \Leftrightarrow \Bigparen{B\csub_\CE g\ast H \text{ for every } g\in B}.  
 \]
\end{cor}
\begin{proof}
 Since the equivalence class $[H\ast B]_\CE$ does not depend on the choice of representative for $[\ast]_\CE$, we can assume that the operations are adapted. If $A\csub_{\CE^{\rm left}_{\CE,H\times H}} B$ it follows Corollary~\ref{cor:p1:quotient coarse structure as composition_cosets} that
 \[
  A\subseteq E\cmp \Bigparen{\Delta_{G}\ast(H\times \{e\})}(B)= E(B\ast H)
 \]
 \emph{i.e.}\ $A\csub_\CE B\ast H$. The converse implication is immediate. To obtain  $A\csub_\CE H\ast B$ one should consider $E\cmp \Bigparen{\Delta_{G}\ast(\{e\}\times H)}(B)$ instead.
 
 The second chain of equivalences follows from the first, as $B$ is $\CE^{\rm left}_{\CE,H\times H}$\=/bounded if and only if $B\csub_{\CE^{\rm left}_{\CE,H\times H}} \{g\}$ for some $g\in G$ or, equivalently, every $g\in B$.
\end{proof}

Finally, in this context we can improve on Lemma~\ref{lem:p1:cosets are groups} to show that coarse normality is also a necessary condition for coset spaces of a coarse subgroups to be coarse groups:

\begin{thm}\label{thm:p1:quotient by normal is a group}
 Let $\crse{H \leq G}$ be a coarse subgroup. The cosets space $\crse{G/H}$ endowed with $\bar\cop$ is a coarse group if and only if $\crse H$ is coarsely normal.
\end{thm}
\begin{proof} 
 Lemma~\ref{lem:p1:cosets are groups} implies that $\crse{G/H}$ is a coarse group whenever $\crse{H\trianglelefteq G}$. Hence we only need to prove the converse implication.

 Assume that the coarse action $\bar\cop$ defines a coarse group structure on $\crse {G/H}$. Let $\crse G=(G,\CE)$, $\crse H=[H]$ and choose adapted representatives $\ast,\unit,\inversefn$ for the coarse operations.
 By hypothesis, the multiplication function $\ast\colon(G,\CE^{\rm left}_{\CE,H\times H})\times(G,\CE^{\rm left}_{\CE,H\times H})\to(G,\CE^{\rm left}_{\CE,H\times H})$ is controlled. In particular, the relation $(\{e\}\times H)\ast\Delta_{G}$ is in $\CE^{\rm left}_{\CE,H\times H}$. By Lemma~\ref{lem:p1:quotient coarse structure as composition_cosets}, there exist some fixed $n\in\NN$ and $E\in\CE$ such that
 \begin{equation}\label{eq:p1:in thm normal subgroup}
  (\{e\}\times H)\ast\Delta_{G}\subseteq \bigparen{\Delta_{G}\ast(\{e\}\times (H^{-1}\ast H)^{\ast n})}\cmp E.
 \end{equation}
 Since $\crse H$ is a coarse subgroup, $(H^{-1}\ast H)^{\ast n}\ceq_\CE H$. In particular, there exists an $E'\in\CE$ such that $(H^{-1}\ast H)^{\ast n}\subseteq E'(H)=\braces{x\mid \exists h\in H,\ x\torel{E'}h}$.

 We see that $(\{e\}\times H)\ast\Delta_G$ is contained in the composition $\bigparen{\Delta_{G}\ast(\{e\}\times E'(H))}\cmp E$. Since the operations are adapted, $(\{e\}\times H)\ast\Delta_G$ is the set of pairs $\braces{(g,h\ast g)\mid g\in G,\ h\in H}$. Therefore, the above containment is equivalent to saying that for every $g\in G$, $h\in H$ there is some $k\in G$ such that
 \[
  g\torel{\Delta_{G}\ast(\{e\}\times E'(H))} k \torel{E} h\ast g.
 \]
 Again, since the operations are adapted this means that $k\in g\ast E'(H)$. In other words, $k=g\ast x$ for some $x\in E'(H)$ and we know that there exists some $h'\in H$ such that $x\torel{E'}h'$. This shows that for every $h\in H$, $g\in G$ there is an $h'\in H$ such that
 \[
  h\ast g\torel{\op{E}} k\;=\;
  g\ast x \torel{\Delta_G\ast E'} g\ast h'.
 \]
 
 It is now simple to deduce that there is some fixed $\bar E\in \CE$ such that $g\ast H\ast g^{-1}$ is contained in the $\bar E$\=/thickening $\bar E(H)$ for every $g\in G$:
 \[
  g\ast(h\ast g^{-1})\torel{\CE} g\ast(g^{-1}\ast h')\torel{\CE} h'
  \qquad \forall h\in H,\ \forall g\in G.
 \]
 This shows that $\crse H$ is coarsely invariant under conjugation.
\end{proof}

\begin{rmk}
It is useful to make an analogy between Theorem~\ref{thm:p1:quotient by normal is a group} and the usual quotients of set\=/groups. If $H$ is an arbitrary subgroup of a set\=/group $G$, the only natural ``quotient'' construction is the set of cosets $G/H$ equipped with the action by multiplication $G\curvearrowright G/H$. When writing cosets as $gH$ with $g\in G$ we are implicitly describing the action, as $gH$ is defined as the image under $g$ of the special coset $H<G$. If $H$ is normal in $G$ then $H$ acts trivially on $G/H$ and hence the action by left multiplication $G\times G/H\to G/H$ descends to a mapping of the quotient $G/H\times G/H\to G/H$ which makes $G/H$ into a set\=/group.
\end{rmk}

\begin{exmp}
 In Theorem~\ref{thm:p1:quotient by normal is a group}, the assumption that $\crse H$ is a coarse subgroup is necessary. For instance, if $\crse G=(G,\mincrs)$ is a trivially coarse group and $A\subseteq G$ is a subset such that $A^n$ is a normal subgroup of $G$ for some fixed $n\in\NN$, then $\crse{G/A}$ is a coarse group isomorphic to the trivially coarse group $(G/A^n,\mincrs)$. However, $A$ need not be invariant under conjugation in $G$. For a concrete example, take $G=S_3$ to be the symmetric group on three elements and $A=\{(1,2),(1,2,3)\}$.
\end{exmp}

\chapter{Coarse Kernels}
\label{ch:p1:coarse kernels}

In this chapter we define the coarse kernel of a coarse homomorphism and show that the coarse analogs of the Isomorphism Theorems hold true. It is important to note that not all coarse homomorphisms have a well\=/defined coarse kernel.

\section{Coarse Preimages and Kernels}
Recall from Section~\ref{sec:p1:containements and intersections} that, given a coarse map $\crse{ f\colon  X\to Y}$ and a coarse subspace $\crse{ Z\subseteq Y}$, the coarse preimage $\crse{ f^{-1}( Z)\subseteq X}$ (Definition~\ref{def:p1:coarse preimage}) is a coarse subspace of $\crse X$ such that 
\begin{itemize}
 \item the coarse image of $\crse{f^{-1}( Z)}$ under $\crse f$ is coarsely contained in $\crse Z$,
 \item if $\crse{f(W)\subseteq Z}$ for some $\crse {W\subseteq X}$ then $\crse{ W\subseteq f^{-1}(Z)}$.
\end{itemize}
Further recall that the coarse preimage need not exist in general. However, when it does exist it is unique, and the next lemma shows that it is well\=/behaved with respect to coarse groups.

\begin{lem}\label{lem:p1:preimage of subgroups}
 Let $\crse f\colon\crse G\to \crse H$ be a coarse homomorphism of coarse groups and $\crse{ K\leq H}$ a coarse subgroup. If the preimage $\crse{f^{-1}}(\crse K)$ exists, then it is a coarse subgroup of $\crse G$. Moreover, if $\crse K$ is coarsely normal (Definition~\ref{def:p1:normal}) then so is $\crse{f^{-1}}(\crse K)$.
\end{lem}
\begin{proof}
 Let $\crse G=(G,\CE)$, $\crse H=(H,\CF)$ and fix some representatives $\crse K=[K]_\CF$, $\crse f=[f]_\CF$, $\crse {f^{-1}}(\crse K)=[Z]_\CE$. Since $\crse f$ is a coarse homomorphism, we see that $f(Z\ast_{G} Z)\ceq_\CF f(Z)\ast_H f(Z)$.  By hypothesis, $f(Z)\csub_\CF K$, hence $f(Z)\ast_H f(Z)\csub_\CF K\ast_H K\ceq_\CF K$. Since $[Z]_\CE=\crse {f^{-1}}(\crse K)$, it follows from the definition of coarse preimage that $Z\ast_{G} Z\csub_\CE Z$. The same argument also shows that $e_{G}$ and $Z^{-1}$ are coarsely contained in $Z$, hence $[Z]_\CE$ is a coarse subgroup.
 
 The proof of the statement on normality is very similar: $\crse{K\leq H}$ is normal if and only if $c(H,K)\ceq_\CF K$, where $c(h,k)= h\ast_H k\ast_H h^{-1}$ (the order of multiplication does not matter up to closeness). It is easy to check that $f(c(G,Z))\ceq_\CF c(f(G),f(Z))\csub_\CF c(H,K)$, hence $c(G,Z)\csub Z$ by definition of coarse preimage. 
\end{proof}

\begin{rmk}
 In the above proof we preferred to fix representatives for the sake of concreteness. However, since coarse images, subgroups, and preimages are all well\=/defined, we could have also argued more implicitly. For example, the chain of coarse containments
 \begin{equation}\label{eq:short containment}
  \crse f\Bigparen{\crse{f^{-1}}(\crse K)\cop_{\crse G}\crse{f^{-1}}(\crse K)}
  \subseteq \crse K\cop_{\crse H}\crse{K} = \crse K
 \end{equation}
 implies that $\crse {f^{-1}}(\crse K)$ is closed under multiplication. All the coarse containments appearing in \eqref{eq:short containment} can be proved from the definitions with some diagram chasing.
\end{rmk}

Of course, one object of special interest is the preimage of the trivial coarse subgroup $\crse{\cunit_{ H}\in H}$.

\begin{de}\label{def:p1:coarse kernel}
 If it exists, the coarse preimage of $\cunit_{\crse H}\crse{\in H}$ under $\crse f$ is the the \emph{coarse kernel}\index{coarse!kernel} of $\crse f$. We denote it by $\cker(\crse f)\crse{\trianglelefteq G}$.\nomenclature[:COS]{$\cker(\crse f)$}{coarse kernel} 
\end{de}

 Concretely, given a subset $K\subseteq G$, its equivalence class $\crse K$ is the coarse kernel of $\crse f$ if and only if for every bounded set $B\subseteq H$ that contains $\unit_H$ the preimage $f^{-1}(B)$ is contained in a controlled thickening of $K$: 
\[
 f^{-1}(B)\subseteq \st(K,\fka),\qquad \fka\in\fkC(\CE).
\]
It follows from Lemma~\ref{lem:p1:preimage of subgroups} that, when it exists, the coarse kernel of a coarse homomorphism is a coarsely normal coarse subgroup of $\crse G$.

\begin{exmp}
 Of course, if $\crse G=(G,\mincrs)$ and $\crse H=(H,\mincrs)$ are trivially coarse groups and $f \colon G\to H$ is a homomorphism then $\cker(\crse f)=\ker (f)$. 
 If $H$ is abelian, then the same remains true if we equip both $G$ and $H$ with the equi bi\=/invariant coarse structure $\varcrs[grp]{fin}$. 
 
 More precisely, let $f\colon G\to H$ be a homomorphism of set\=/groups, so that $\crse f\colon (G,\varcrs[grp]{fin})\to(H,\varcrs[grp]{fin})$ is a coarse homomorphism by Corollary~\ref{cor:p1:homomorphisms are grp_fin controlled}. If $H$ is abelian, its bounded subsets are precisely the finite sets. Given a finite set $B=\braces{h_1,\ldots, h_n}\subseteq H$, the preimage $f^{-1}(B)$ is the union of finitely many left cosets $g_1\ker(f),\ldots,g_n\ker(f)$, where $g_i\in G$ is a fixed arbitrary element in $f^{-1}(h_i)$. In other words, $f^{-1}(B)$ is the set $\braces{g_1,\cdots,g_n}\ast_{G} \ker(f)$, which is a controlled thickening of $\ker(f)$ (and is hence coarsely contained in $\ker(f)$).
 
This is no longer true if $H$ is not abelian. For instance, if $H=D_\infty$ is the infinite dihedral group then $\varcrs[grp]{fin}=\maxcrs$ is the bounded coarse structure (see also Section~\ref{sec:p2:connected coarsification}). It then follows that $\cker(\crse f)$ is equal to all of $G$, irrespective of what $\ker(f)$ is.
\end{exmp}

\begin{exmp}
Let $\crse G=(\RR^n,\CE_{\norm{\mhyphen}_2})$ and $\crse H=(\RR,\CE_{\norm{\mhyphen}_2})$. Then any linear function $f\colon \RR^n\to\RR$ is controlled and defines a coarse homomorphism. It is once again easy to check that $\cker(\crse f)=[\ker(f)]$.

Consider now $\ZZ^n<\RR^n$ for some $n\geq 2$. The inclusion gives an isomorphism of coarse groups $(\ZZ^n,\CE_{\norm{\mhyphen}_2})\cong (\RR^n,\CE_{\norm{\mhyphen}_2})$. In particular, this restricts to an isomorphism between $\cker([f|_{\ZZ^n}])\crse{\leq} (\ZZ^n,\CE_{\norm{\mhyphen}_2})$ and $\cker([f])=[\ker(f)]\crse{\leq}(\RR^n,\CE_{\norm{\mhyphen}_2})$. 
The latter is of course isomorphic to $(\RR^{n-1},\CE_{\norm{\mhyphen}_2})$ and it is therefore unbounded. As a consequence, if we choose an injective homomorphism $f|_{\ZZ^n}\colon\ZZ^n\to \RR^n$ then $\cker([f|_{\ZZ^n}])\neq [\ker(f|_{\ZZ^n})]=[\{e\}]$ (Figure~\ref{fig:p1:coarse kernel linear map}).

More generally, if $\crse G=(\ZZ^n,\CE_{\norm{\mhyphen}_2})$ and $\crse H=(H,\CE_d)$ is a group equipped with a bi\=/invariant proper metric (\emph{i.e.}\ such that closed balls are compact) and $f\colon \ZZ^n\to H$ is a homomorphism, then $\cker([f])= [\ker(f)]$ if and only if $f(\ZZ^n)$ is discrete in $H$. To see this, observe that if $f(\ZZ^n)$ is discrete then the preimage of every bounded subset of $H$ is contained in finitely many cosets of $\ker(f)$. On the contrary, if $f(\ZZ^n)$ is not discrete then the preimage of any open (bounded) neighborhood of $e_H$ contains infinitely many distinct cosets of $\ker(f)$ and it is hence not coarsely contained in $\ker(f)$.
\end{exmp}

\begin{figure}
 \centering
 \includegraphics[scale=0.8]{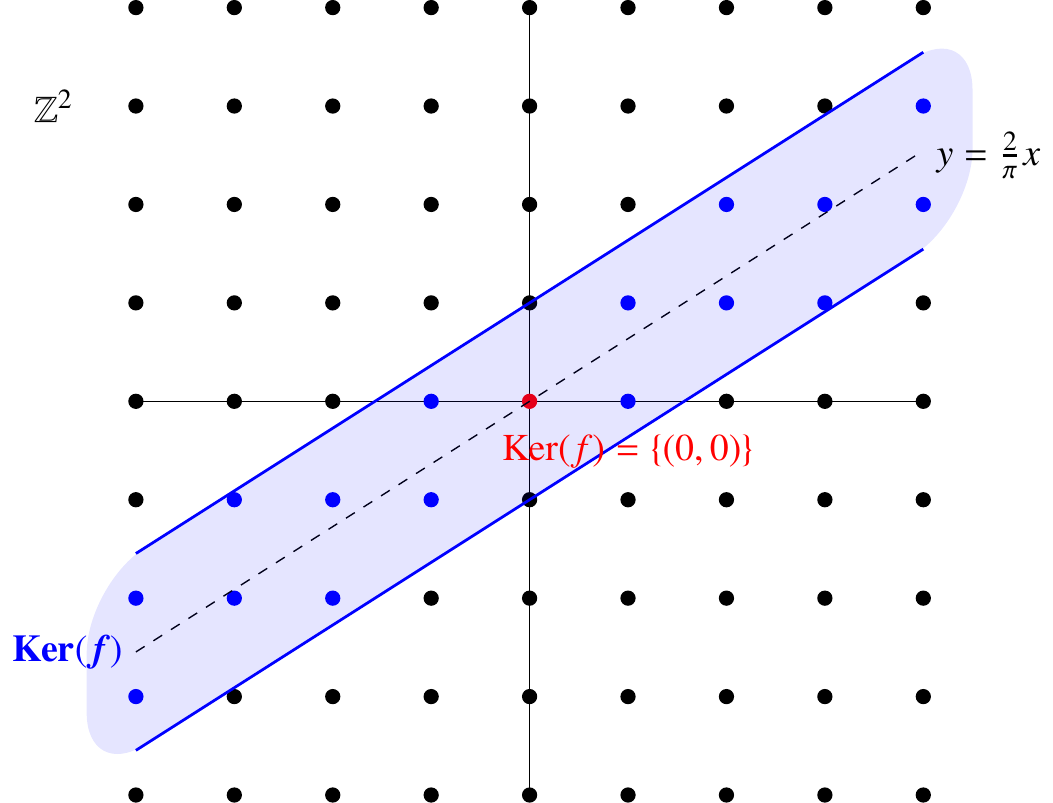}
 \caption{An injective linear mapping $\ZZ^2\to \RR$ given by $f(x,y)\coloneqq y- \frac{2}{\pi}x$, has trivial kernel, but its coarse kernel is coarsely isomorphic to $(\RR,\CE_{\abs{\mhyphen}})$.}
 \label{fig:p1:coarse kernel linear map}
\end{figure}

\

The following examples show that the coarse kernel may not exist.

\begin{lem}\label{lem:p1:no kernel disconnected}
 If $\crse G=(G,\mincrs)$ is a trivially coarse group and $\crse{f\colon G\to H}$ a coarse homomorphism, let $\crse {H_{e}\leq H}$ be the coarsely connected component of $\cim(\crse f)$ containing $\cunit_{\crse H}$. Then $\crse f$ has a coarse kernel if and only if $\crse {H_{e}} =\cunit_{\crse H}$ is bounded in $\crse H$. 
\end{lem}
\begin{proof}
 First, notice that the coarse subgroup $\crse {H_{e}\leq H}$ has a coarse preimage $\crse{f^{-1}(H_{e})}$: we take the preimage under $f$ of the coarsely connected component $H_e\subseteq H$ containing $e$. This is a coarse preimage, as $H_e$ is not contained in any strictly larger controlled neighborhood of itself.
 As a consequence, we see that if $\crse H_{e}$ is bounded then $\crse{f^{-1}(H_{e})}$ is the coarse kernel of $\crse f$.
 
 For the converse implication, if $\crse{ H_{e}}$ is not bounded we see that for every bounded set $B\subset f(G)$ there is a strictly larger bounded set $B\subsetneq B'\subset f(G)$. In turn, this shows that for every subset of $G$ with bounded image there is a strictly larger (and hence non $\mincrs$\=/asymptotic) subset with bounded image. This shows that $\cunit_{\crse H}$ does not have a coarse preimage.
\end{proof}
 
 To give a concrete example, Lemma~\ref{lem:p1:no kernel disconnected} shows that the identity map $\id\colon (\ZZ,\mincrs)\to(\ZZ,\CE_{\abs{\mhyphen}})$ does not have a coarse kernel because $(\ZZ,\CE_{\abs{\mhyphen}})$ is not bounded.

\begin{exmp}\label{exmp:p1:Hilbert.space.monomorphism no ker}
 Recall Example~\ref{exmp:p1:Hilbert.space.monomorphism}, let $\crse{G}=(\ell_2(\NN),\CE_{\norm{\mhyphen}_2})$ and $\crse{H}=(\ell_2(\NN),\CE_{\norm{\mhyphen}_2'})$, where $\norm{\mhyphen}_2'$ is the non\=/complete norm obtained by rescaling the standard unitary base $v_1,v_2,\ldots $ of $\ell_2(\NN)$ so that $\norm{v_n}_2'=\frac 1n$.
 We claim that the identity map $\id_{\norm{\mhyphen}_2\to\norm{\mhyphen}_2'}\colon(\ell_2(\NN),\CE_{\norm{\mhyphen}_2})\to(\ell_2(\NN),\CE_{\norm{\mhyphen}_2'})$ does not have a coarse kernel. 
 In fact, consider the $\norm{\mhyphen}_2'$\=/balls around the origin 
 \[
 B(0;r)\coloneqq\Bigbraces{f\in\ell_2(\NN)\bigmid \sum_{n\in\NN}\frac{f(n)^2}{n^2} < r^2 }\subset\ell_2(\NN).
 \]
 All these balls are bounded neighborhoods of the identity in $ (\ell_2(\NN),\CE_{\norm{\mhyphen}_2'})$. However, they are not asymptotic to one another when seen as subsets of $(\ell_2(\NN),\CE_{\norm{\mhyphen}_2})$. 
 Since every $\norm{\mhyphen}_2'$\=/bounded neighborhood of $0$ is contained in such a ball for some $r$ large enough, this implies that $0\in \ell_2(\NN)$ does not have a well\=/defined coarse preimage under $\id_{\norm{\mhyphen}_2\to\norm{\mhyphen}_2'}$.
\end{exmp}

The following example is quite different in flavor:

\begin{exmp}\label{exmp:p1:quotient F_2/A}
 Let $\crse G\coloneqq (F_2,\varcrs{bw})$ be the coarse group obtained by equipping the non\=/abelian free group $F_2=\angles{a,b}$ with its canonical coarse structure. 
 Recall that $\varcrs{bw}$ coincides with the metric coarse structure defined by the cancellation distance $d_{\times}$ associated with $S=\{a^\pm,b^\pm\}$ (this is the bi\=/invariant metric where the length of a word is the minimal number of letters it is necessary to remove to reduce it to the trivial word, see Example~\ref{exmp:p1:cancellation metric}).
 We identify $F_2$ with the set of (reduced) words in the letters $a$ and $b$. Let $H=\braces{a^k\mid k\in \ZZ}$ be the subgroup generated by the letter $a$ 
 and consider the coarse quotient group $\crse Q\coloneqq \crse G/\aangles{\{H\times H\}}$ (this is the quotient in the sense of Definition~\ref{de:p1:quotient by a relation}, it is not a coarse cosets space). Namely, $\crse Q$ is $(F_2,\CD)$ where
 \[
  \CD=\varcrs[grp]{\mathnormal{d_{\times},H\times H}}=\angles{\CE_{d_\times}\;,\; \Delta_{G}\ast (H\times H)\ast \Delta_{G}},
 \]
 and the quotient map $\crse{q\colon G\to Q}$ is realized by the identify function $\crse q= [\id_{F_2}]$. We claim that $\crse q$ does not have a coarse kernel.
 
 In this specific case, one can show that the coarse structure $\CD$ coincides with the structure induced by a different cancellation metric $\bar{d}_{\times}^a$ on $F_2$, namely the cancellation metric associated with the infinite (non\=/freely) generating set $S=\braces{a^k\mid k\in\ZZ}\cup\braces{b^\pm}$.
 For every $r\in\NN$ consider the set 
 \[
 C_r\coloneqq
 \Bigbraces{a^{n_1}b^{m_1}a^{n_2}b^{m_2}\cdots a^{n_r}b^{m_r}\bigmid n_i,m_i\in \ZZ,\ \textstyle{\sum_{i=1}^r} m_i =0}\subseteq F_2.
 \]
 Note that $C_r$ is contained in the ball $B_{\bar{d}_{\times}^a}(e;r)$ because for every $w\in C_r$ it is possible to remove all the occurrences of the letter $a$ using at most $r$ cancellations and the remainder reduces to the trivial word because $\sum_{i=1}^r m_i=0$. 
 In particular, $C_r$ is a $\CD$\=/bounded neighborhood of the identity.
 However, considering words of the form $(a^nb^n)^ra^nb^{-rn}$ with $n$ arbitrarily large we see that $C_{r+1}$ is not $\CE_{d_\times}$\=/coarsely contained in $B_{\bar{d}_{\times}^a}(e;r)$.
 It follows that for every $r\in \NN$ the ball $B_{\bar{d}_{\times}^a}(e,r+1)$ is not $\CE_{d_\times}$\=/coarsely contained in $B_{\bar{d}_{\times}^a}(e;r)$, hence the coarse kernel does not exist.
 
 It is interesting to note that $\crse H$ is not coarsely normal in $\crse G$ (the sets $gHg^{-1}$ reach arbitrarily $d_\times$\=/far from $H$ as $g$ ranges in $F_2$). This immediately implies that $\crse H$ cannot be the coarse kernel of the quotient $F_2\to F_2/\aangles{\{H\times H\}}$, for coarse kernels are coarsely normal.
 Note also that $\crse Q$ is far from being the group quotient under the normal closure of $H$ in $F_2$: such a quotient would just be isomorphic to $\ZZ$, while it is simple to observe that $(F_2,\bar d_\times^a)$ is not coarsely abelian.
\end{exmp}

\section{The Isomorphism Theorems}\label{sec:p1:isomorphism theorems}
We will now show that the isomorphism theorems hold in the category of coarse groups.

As a preliminary remark, recall that---by definition---if $\crse f\colon \crse G\to \crse H$ is a coarse homomorphism then $\crse f$ is a coarsely $\crse G$\=/equivariant map, where $\crse G$ coarsely acts on itself by left multiplication and it coarsely acts on $\crse H$ multiplying on the left by its image under $\crse f$. In symbols, the coarse action on $\crse H$ is the composition 
\[
 \cop_{\crse H}\circ(\crse{f}\times \cid_{\crse H})\colon \crse {G\times H\to H}.
\]
Let $\crse{K\leq G}$ be a coarse subgroup such that $\crse{f}(\crse K)=\cunit_{\crse{H}}$. Then $\crse f$ descends to a $\crse G$\=/equivariant map of the coarse coset space $\bar{\crse f}\colon\crse{G/K}\to\crse H$.
Our study of coarse coset spaces allows us to easily prove the following:

\begin{thm}[Coarse First Isomorphism Theorem]\label{thm:p1:first iso theorem} 
 Let $\crse f\colon\crse G\to\crse H$ be a coarse homomorphism, $\crse{K\leq G}$ a coarse subgroup such that $\crse f(\crse K)=\cunit_{\crse H}$ and let $\crse{G/K}$ be the coarse coset space. The following are equivalent:
 \begin{enumerate}[(1)]
  \item the coarse map $\bar{\crse f}\colon\crse{G/K}\to\crse H$ is proper;
  \item $\crse K=\cker(\crse f)$;
  \item $\crse K$ is normal in $\crse G$ and $\bar{\crse f}\colon\crse{G/K}\to\cim(\crse f)$ is an isomorphism of coarse groups.
 \end{enumerate}
\end{thm}
\begin{proof}
 Let $\crse G=(G,\CE)$, $\crse H=(H,\CF)$ and let $\ast_{G},\unit_{G},\inversefn$ be adapted representatives for the operations of $\crse G$. We can further assume that $f(e_{G})=e_H$. We realize $\crse{G/K}$ as $(G,\CD)$ where $\CD=\varcrs[left]{\mathnormal{\CE,K\times K}}=\angles{\CE,\Delta_{G}\ast(K\times K)}$. By Corollary~\ref{cor:p1:bounded sets in cosets}, a subset $B\subseteq G$ is $\CD$\=/bounded if and only if $B\csub_\CE g\ast K$ for every $g\in B$.
 
 $(1)\Rightarrow(2)$: 
 If the map $f\colon (G,\CD)\to(H,\CE)$ is proper, the preimage $f^{-1}(B)$ of every bounded neighborhood of the identity $B\subseteq H$ is itself a $\CD$\=/bounded bounded neighborhood of $e_{G}$. It follows that $f^{-1}(B)$ is $\CE$\=/coarsely contained in $K$, which shows $\crse K$ is the coarse kernel.

 $(2)\Rightarrow(1)$:
 Assume that $\crse K$ is the coarse kernel. 
 If $B\subseteq H$ is bounded and $f^{-1}(B)\neq \emptyset$, pick some $g\in f^{-1}(B)$. Since the left multiplication map $\crse{{}_{g^{-1}}\ast_{G}}$ is a coarse equivalence,  Lemma~\ref{lem:p1:preimage under coarse equivalence} implies that for every coarse subspace $\crse{ Z\subseteq G}$ the coarse preimage $(\crse{{}_{g^{-1}}\ast_{G}})^{\crse {-1}}(\crse Z)$ exists. The same is true for $\crse{ {}_{f(g)}\ast_{H}}$.
 Let $\crse B=[B]$. We see that $(\crse{ {}_{f(g)}\ast_{H}})^{\crse{-1}}(\crse B)=\cunit_{\crse H}$ and hence $(\crse{ {}_{f(g)}\ast_{H}} \circ \crse{f})^{\crse{-1}}(\crse B)=\crse K$ because $\crse K$ is the coarse kernel of $\crse f$. 
 Since $\crse f$ is a coarse homomorphism, $\crse{f}=\crse {{}_{f(g)}\ast_{H}} \circ \crse f\circ\crse {{}_{g^{-1}}\ast_{G}}$. Hence we deduce that $\crse{f^{-1}}(\crse B)$ also exists and it is equal to $(\crse{ {}_{g^{-1}}\ast_{G}})^{\crse{-1}}(\crse K)=\crse{ {}_{g}\ast_{G}}(\crse K)$. By definition, this implies that $f^{-1}(B)$ is $\CE$\=/coarsely contained in $g\ast K$ and it is hence $\CD$\=/bounded.

 $(3)\Rightarrow (1)$: This is clear.
 
 $(1)+(2)\Rightarrow (3)$: We already know that $\crse K=\cker(\crse f)$ is a coarsely normal. It follows from Theorem~\ref{thm:p1:quotient by normal is a group} that $\crse{G/K}$ is indeed a coarse group, hence it makes sense to ask whether $\bar{\crse f}$ is an isomorphism. Since we realize $\crse{G/K}$ as $(G,\CD)$ and the maps $\crse f$ and $\bar{\crse f}$ (resp. $\cop_{\crse G}$ and $\bar{\cop}_{\crse G}$) coincide setwise, it follows that $\bar{\crse f}\colon\crse{G/K}\to\crse H$ is a coarse homomorphism. Since $\bar{\crse f}$  is proper, it is a coarse embedding (Proposition~\ref{prop:p1:proper.hom.coarse.embedding}). Since $\bar{\crse f}$ is coarsely surjective on its image, it follows from Lemma~\ref{lem:p1:coarse.eq.iff.surjective.coarse.emb} that $\bar{\crse f}\colon\crse{G/K}\to\cim(\crse f)$ is a coarse equivalence and hence an isomorphism of coarse groups.
\end{proof}

Given two coarse subspaces $\crse{A}$ and $\crse B$ of a coarse group $\crse G$, their product\index{product of coarse subspaces} $\crse{A\cop B}$\nomenclature[:COS]{$\crse{A\cop B}$}{product of coarse subspaces} is the coarse image under $\cop$  of the coarse subspace $\crse A\times \crse B\crse{\subseteq G}\times \crse G$ (equivalently, it is the coarse subspace $[A\ast B]$, where $A\ast B=\braces{a\ast b\mid a\in A,\ b\in B}$). Notice that if $\crse {C\subseteq G}$ is a third coarse subspace then since $\cop$ is controlled $\crse{(A\cop B)\cop C=A\cop (B\cop C)}$.

\begin{thm}[Coarse Second Isomorphism Theorem]\label{thm:p1:second iso theorem}
 Let $\crse H,\crse N$ be coarse subgroups of $\crse G$. If $\crse N$ is coarsely invariant under the coarse action of $\crse H$ by conjugation, then the coarse subspace $\crse{ ( H\cop  N)\subseteq G}$ is a coarse subgroup and $\crse{N \trianglelefteq (H\ast N)}$. Moreover, if the coarse intersection $\crse{H\cap N}$ exists, then
 \begin{enumerate}[(1)]
  \item $\crse{H\cap N}$ is a coarsely normal coarse subgroup of $\crse H$;
  \item there is a canonical isomorphism $(\crse H\cop\crse N)\crse{/N}\to \crse{H/}(\crse{H\cap N})$.
 \end{enumerate}
\end{thm}

\noindent\textit{Proof.}
 Let $\crse G=(G,\CE)$ and fix adapted representatives for the coarse operations. By definition, $\crse N$ is coarsely invariant under conjugation by $\crse H$ if and only if $c(H,N)\csub N$, where $c\colon G\times G\to G$ denotes the conjugation $c(h,g)=(h\ast g)\ast h^{-1}$. 
 Since $[H]=[H^{-1}]$, we also have $c(H^{-1},N)\csub N$ and hence there exists an $E\in \CE$ such that for every $n\in N$ and $h\in H$ there exist $n',n''\in N$ with 
 \[
 (h\ast n_1)\ast h^{-1}\torel{E}n' \qquad\text{and}\qquad (h^{-1}\ast n_1)\ast h \torel{E}n''. 
 \]
 In particular, if $\crse N$ is normalized by $\crse H$ then $\crse{ N\cop H}=\crse{H\cop N}$. In fact $N\ast H\csub H\ast N$ because
 \begin{equation}\label{eq:p1:normalizing n}
   n\ast h \rel{\CE} h \ast \bigparen{(h^{-1} \ast n) \ast h}\torel{\Delta_{G}\ast E} h\ast n'
   \qquad \forall n\in N,\ \forall h\in H.
 \end{equation}
 The coarse containment $H\ast N\csub N\ast H$ is analogous.
 
 Since $\cop$ is associative on coarse subsets, it follows that 
 \[
  \crse{(H\cop N)\cop (H\cop N)= H\cop (N\cop H)\cop N = H\cop (H\cop N)\cop N =H\cop N }. 
 \]
 The coarse containment $(H\ast N)^{-1}\csub (H\ast N)$ follows from the fact that $(h\ast n)^{-1}\rel{\CE} n^{-1}\ast h^{-1}$ for all $h\in H,\ n\in N$. This proves that $\crse{H\cop  N}$ is a coarse subgroup. Both $\crse H$ and $\crse N$ are coarse subgroups of $\crse{H\cop N}$ and it is immediate to verify that $\crse N$ is also coarsely normal.
 For the second part of the statement, assume that $\crse{H\cap N}$ exists. We already know (Lemma~\ref{lem:p1:intersection of subgroups}) that $\crse{H\cap N}$ is a coarse subgroup of $\crse H$ (and of $\crse N$). 
 Normality follows easily from the compatibility of the multiplication in $\crse H$ and $\crse G$. Namely, if $\crse{c^{H}}$ and $\crse{c^{G}}$  denote the conjugation in $\crse H$ and $\crse G$ respectively, when embedding $\crse H$ in $\crse G$ we have 
 \[
  \crse{c^{H}(H,H\cap N)}
  \crse =\crse{c^{G}(H,H\cap N)}.
\]
 The right hand side is coarsely contained in $\crse N$ by assumption and it  is obviously coarsely contained in $\crse H$ as well. It is hence coarsely contained in $\crse{H\cap N}$, as desired.
 
 It remains to prove that $(\crse H\cop_{\crse G}\crse N)\crse{/N}\cong \crse{H/}(\crse{H\cap N})$. Suppose $\crse{H\cap N}=[K]$ for some $K\subseteq H$. For every $g\in H\ast N$ choose $h(g)\in H$ and $n(g)\in N$ with $g=h(g)\ast n(g)$ and define a function $\psi\colon H\ast N\to H$ sending $g$ to $h(g)$. Equip $H$ with the quotient coarse structure $\CD=\varcrs[left]{\mathnormal{\CE|_H, K\times K}} =\angles{\CE|_H,\Delta_H\ast (K\times K)}$. We claim that $\psi\colon (H\ast N,\CE|_{H\ast N})\to (H,\CD)$ is controlled and that the coarse map $\crse\psi$ does not depend on the choices made. 
 
 \begin{claim*}
  For every $E\in\CE$ there exists a $D\in\CD$ such that for every pair of elements $g_1$, $g_2$ such that $g_i=h_i\ast n_i$ with $h_i\in H$ and $n_i\in N$ for $i=1,2$, if $g_1\rel{E} g_2$ then $h_1\rel{D}h_2$.
 \end{claim*}
 \begin{proof}[Proof of Claim]
 Notice that
  \[
  h_2^{-1}\ast h_1\rel{\CE}
  h_2^{-1}\ast (h_1\ast n_1)\ast n_1^{-1}
  \rel{\Delta_{G}\ast E\ast\Delta_{G}}
  h_2^{-1}\ast (h_2\ast n_2)\ast n_1^{-1}
  \rel{\CE} n_2\ast n_1^{-1}
 \]
for all $ h_1\ast n_1\rel{E}h_2\ast n_2$.
 Therefore, there exists an $E_1\in \CE$ depending only on $E$ so that $h_2^{-1}\ast h_1\rel{E_1} n_2\ast n_1^{-1}$ whenever $g_i$, $h_i$, $n_i$ are as in the statement of the claim. 
 
 Let $E_2\in \CE$ be a symmetric relation large enough so that $H^{-1}\ast H\subseteq E_2(H)$ and $N\ast N^{-1}\subseteq E_2(N)$. Then we see that there exist $h'\in\ H$ and $n'\in N$ with $h'\rel{E_2} h_2^{-1}\ast h_1$ and $n'\rel{E_2}n_2\ast n_1^{-1}$. In particular, $h'\in H\cap \paren{ E_2\circ E_1 \circ E_2(N)}$. However, it follows from the definition of coarse intersection that there exists some $E_3\in\CE$ depending only on $E$ such that $H\cap \paren{E_2\circ E_1 \circ E_2(N)}\subseteq E_3(K)$. At this point we are done, because
 \[
  h_1
  \rel{\CE}
  h_2\ast(h_2^{-1}\ast h_1)
  \rel{\Delta_{G}\ast E_2}
  h_2\ast h'
  \rel{\Delta_{G}\ast ( E_3(K)\times\{e\})}
  h_2\ast e
  \rel{\CE}
  h_2
 \]
 showing that $h_1$ and $h_2$ are uniformly $\CD$\=/close.
 \end{proof}

 Applying the claim to the diagonal $E=\Delta_{G}$ we see that $\crse\psi$ is independent of the choices made. Applying the claim to arbitrary $E\in\CE$ shows that $\psi$ is controlled. By \eqref{eq:p1:normalizing n}, the product $(h_1\ast n_1)\ast (h_2\ast n_2)$ is uniformly $\CE$\=/close to $(h_1\ast h_2)\ast n_3$ for some $n_3\in N$. It follows that $\crse\psi$ is a coarse homomorphism.
 
 Now we only need to show that $\crse{N\trianglelefteq H\ast N}$ is the coarse kernel of $\crse\psi$. Explicitly, we need to show that for every $\CD$\=/bounded neighborhood of the identity $e\in B\subseteq H$ the preimage $\psi^{-1}(B)$ is $\CE$\=/coarsely contained in $N$. By the definition of $\psi$, the preimage $\psi^{-1}(B)$ is contained in $B\ast N$.
 On the other hand, it follows from Corollary~\ref{cor:p1:bounded sets in cosets} that the $\CD$\=/bounded set $B$ is $\CE$\=/coarsely contained in $K$.  
 In turn, since $K$ is $\CE$\=/coarsely contained in $N$, we see that $B\ast N\csub_\CE K\ast N\csub_\CE N$ and hence $\psi^{-1}(B)\csub_\CE N$, as desired.
\qed

\begin{thm}[Coarse Third Isomorphism Theorem]\label{thm:p1:third iso theorem}
 Let $\crse{N\trianglelefteq G }$ be a coarsely normal coarse subgroup. There exists a bijective correspondence:
 \[
  \braces{\text{coarse subgroups of }\crse{G/N}}\longleftrightarrow\braces{\crse K\mid \crse{N \leq K\leq G}}.
 \]
 This correspondence preserves normality. Furthermore, for every $\crse{K\trianglelefteq G}$ coarsely containing $\crse N$ there is a canonical isomorphism 
 \[
  \crse{(G/N)/(K/N)\cong G/K}.
 \]
\end{thm}
\begin{proof}
 The projection $\crse G\to \crse{G/N}$ is a coarse homomorphism, hence the coarse image of every coarse subgroup of $\crse G$ is a coarse subgroup of $\crse {G/N}$. We need to show that this map is a bijection when restricted to coarse subgroups of $\crse G$ that coarsely contain $\crse N$.
 
 Let $\crse G=(G,\CE)$, fix adapted representatives for the coarse operations and let $\crse{G/N}=(G,\CD)$ with $\CD=\varcrs[left]{\mathnormal{\CE,N\times N}}\angles{\CE,\Delta_{G}\ast(N\times N)}$. By Corollary~\ref{cor:p1:bounded sets in cosets} we know that $A\csub_\CD B$ if and only if $A\csub_\CE N\ast B$.  
 
 Let us prove injectivity. 
 Let $[K_1]_\CE$ and $[K_2]_\CE$ be coarse subgroups of $\crse G$ coarsely containing $\crse N$. Since the quotient map is the identity, they are sent to $[K_1]_\CD$ and $[K_2]_\CD$. If $K_1\csub_\CD K_2$, it means that $K_1 \csub_\CE N\ast K_2$. Since $N\csub_\CE K_2$, the latter is coarsely contained in $K_2\ast K_2\csub_\CE K_2$. Therefore $[K_1]_\CE\crse{\subseteq} [K_2]_\CE$. Analogously, if $K_2 \csub_\CD K_1$ we deduce that $[K_1]_\CE\crse{\subseteq} [K_2]_\CE$, therefore $[K_1]_\CD\crse{=} [K_2]_\CD$ if and only if $[K_1]_\CE\crse{=} [K_2]_\CE$.
 
 For the surjectivity, let $[K]_\CD$ be a coarse subgroup of $\crse{G/N}$. Note that $N\ast K\ceq_\CD K$, so that we can use $N\ast K$ instead of $K$ as a representative of $[K]_\CD$. This way we ensure that $N\csub_\CE N\ast K$. It remains to see that $[N\ast K]_\CE$ is a coarse subgroup of $\crse G$. By assumption, $(N\ast K)\ast (N\ast K)\csub_\CD K$. Therefore we have
 \[
  (N\ast K)\ast (N\ast K)
  \csub_\CE N\ast K,
 \]
 showing that $[N\ast K]_\CE$ is coarsely closed under multiplication in $(G,\CE)$. Closure under inversion is similar: since $N$ and $K$ are coarse subgroups, $(N\ast K)^{-1}\ceq_\CE K^{-1}\ast N^{-1}\ceq_\CD K\ast N$. Since $\crse N$ is coarsely normal, $K\ast N\ceq_\CE N\ast K$ and we conclude as before that $(N\ast K)^{-1}\csub_\CE N\ast K$. An analogous argument shows that if $[K]_\CD$ is normal in $\crse{G/N}$ then $[N\ast K]_\CE$ is normal in $\crse G$.
 
 In the coarse setting, the last statement is almost a tautology. Fix a representative of $\crse K=[K]_\CE$. We can also assume that $N\subseteq K$. We can realize the quotients as $\crse{G/N}=(G,\CD)$ and $\crse{K/N}=(K,\CD|_{K})$.  Then $\crse{(G/N)/(K/N)}= (G,\varcrs[left]{\mathnormal{\CD,K\times K}})$. Since $N\subseteq K$, we see that 
  \[
   \varcrs[left]{\mathnormal{\CD,K\times K}}
   =\angles{\CD,\Delta_{G}\ast (K\times K)}=\angles{\CE,\Delta_{G}\ast(K\times K)}
   =\varcrs[left]{\mathnormal{\CE,K\times K}},
  \]
  and the latter is the coarse structure that makes $G$ into $\crse{G/K}$.
\end{proof}

\section{Short Exact Sequences of Coarse Groups}
We say that a \emph{short exact sequence}\index{short exact sequence} is a sequence of coarse homomorphisms of coarse groups
 \[
 \begin{tikzcd}
  \tobj\arrow[r] & 
  \crse{K} \arrow[r, "\crse{f}"] &
  \crse{G} \arrow[r, "\crse{h}"] &
  \crse{H} \arrow[r] &
  \tobj
 \end{tikzcd}
\]  
such that $\crse f$ is a coarse embedding, $\crse h$ is coarsely surjective and it has a coarse kernel $\cker(\crse h)$, the coarse kernel $\cker(\crse h)$ coincides with the coarse image $\crse{f(K)}$ (one may similarly define long exact sequences). 

Short exact sequences can be seen as a way to decompose a coarse groups into simpler components, and the Coarse Isomorphism Theorems guarantee that these decompositions are well\=/behaved. As an example, we shall now see that every coarse group can be decomposed into a coarsely connected coarse group and a trivially coarse group.

Recall that the set of coarsely connected components of $\crse G$ can be naturally identified with the set of coarse maps $\cmor(\terobj,\crse G)$ from the bounded coarse space to $\crse G$. Further recall that the latter is naturally a set\=/group when equipped with the operation 
\[
 \crse f_1\odot\crse f_2\colon \tobj\xrightarrow{(\crse f_1,\crse f_2)} \crse{G\times G}\xrightarrow{\crse \ast} \crse G
\]
(see Section~\ref{sec:p1:group object}).

Given a coarse group $\crse G$, let $ G_{e}\subseteq G$ be the coarsely connected component containing the identity. We already remarked that $\crse{G_{e}}$ is a coarsely normal coarse subgroup of $\crse G$ (Subsection~\ref{sec:p1:action by conjugation}). The following elementary observation can be proved in multiple ways. This gives us a good excuse to review some of the techniques that we developed so far.

\begin{lem}\label{lem:p1:quotient by identity component}
 The coarse quotient $\crse{G/G_{e}}$ is isomorphic to the trivially coarse group of coarsely connected components $(\cmor(\terobj,\crse G),\mincrs)$.
\end{lem}

\noindent\emph{Proof.}
 Let 
 \(
 q \colon G\to \cmor(\terobj,\crse G) 
 \)
be the function sending the point $g\in G$ to its coarsely connected component $\crse{g\in G}$, which we identify with the coarse map $\crse g = [{\rm pt} \mapsto g]$ from $\tobj$ to $\crse G$. Note that $q(g_1\ast g_2)= q(g_1)\odot q(g_2)$. 

For brevity, denote $G'\coloneqq\cmor(\terobj,\crse G)$. We can deduce that $q\colon \crse{G/G_{e}}\to (G',\mincrs)$ is controlled using any of the following arguments:
\begin{itemize}
 \item Directly: note that if $A\subseteq G$ is bounded in $\crse{G/G_{e}}$ then $q(g)=q(g')$ for every $g,g' \in A$. It follows that for any controlled partial covering $\fka$ of $\crse G$ the image $q(\fka)$ consists of singletons and it is hence controlled.
 
 \item Since $\crse{G_{e}}$ is coarsely normal, $\crse{G/G_{e}}$ is a coarse group. Using the criterion given in Proposition~\ref{prop:p1:bornologous map is coarse hom} it is hence enough to observe that $q$ is constant on $\crse{G_{e}}$.
 
 \item The left multiplication induces a coarse action of $\crse G$ on $G'$ and $q$ is a representative for the orbit map. The pull\=/back coarse structure $q^*(\mincrs)$ is equi left invariant and $\crse{G_{e}}$ is $q^*(\mincrs)$\=/bounded. By definition of the quotient coarse action, it follows that the coarse structure of $\crse{G/G_{e}}$ is contained in $q^*(\mincrs)$, hence $q$ is controlled.
\end{itemize}

We thus have a surjective coarse homomorphism $\crse {q\colon G/G_{e}}\to (G',\mincrs)$. To prove that it is an isomorphism we may:

\begin{itemize}
 \item Define a function $h\colon G'\to G$ by arbitrarily choosing a representative for each coarsely connected component. The function $h\colon (G',\mincrs)\to \crse{G/G_{e}}$ is automatically controlled and it is easy to verify that it is a coarse inverse for $\crse q$.
 
 \item Notice the only bounded neighborhood of the identity in $G'$ is the singleton $e_{G'}$. It then follows that $\crse {q\colon G/G_{e}}\to (G',\mincrs)$ is proper and hence it is an isomorphism by the coarse Isomorphism Theorem (Theorem~\ref{thm:p1:first iso theorem}). Also note that $\crse{G_{e}}$ is the coarse kernel of $\crse q$.
 
 \item Observe the $q^*(\mincrs)$\=/bounded neighborhoods of $e_G$ are precisely the subsets of the coarsely connected component $G_{e}\subseteq G$. Since equi left invariant coarse structures on $G$ are determined by their bounded sets, $q^*(\mincrs)$ coincides with the coarse structure of the quotient $\crse{G/G_{e}}$. It follows that $\crse q$ is a coarsely surjective coarse embedding, hence a coarse equivalence and therefore an isomorphism. \qed
\end{itemize}

Note that in the proof of Lemma~\ref{lem:p1:quotient by identity component} we also showed that $\crse {G_{e}}$ is the coarse kernel of $\crse q$. This implies the following.

\begin{cor}\label{cor:p1:ses quotient by id comp}
 Every coarse group $\crse G$ fits in a short exact sequence
 \[
 \begin{tikzcd}[row sep=small,
  ar symbol/.style = {draw=none,"\textstyle#1" description,sloped},
  isomorphic/.style = {ar symbol={\cong}},
  ]
  \tobj\arrow[r] & 
  \crse{G_{e}} \arrow[r] &
  \crse{G} \arrow[r, "\crse{q}"] &
  (\cmor(\terobj,\crse G),\mincrs) \arrow[r] \ar[d,isomorphic]&
  \tobj. \\
  &&& \crse{G/G_{e}} &
 \end{tikzcd}
\]
\end{cor}

With some more care, we can understand when coarse cosets groups fit into short exact sequences. Namely, we can prove:

\begin{prop}\label{prop:p1:kernel of cosets spaces}
 Let $\crse{A\subseteq G}$ be a coarse subspace so that the coarse cosets space $\crse{G/A}$ is a coarse group. The quotient homomorphism $\crse {q\colon G\to G/A}$ has a coarse kernel if and only if there is an $n\in \NN$ such that $\crse{(A^{-1} \ast A)}^{\ast n}$ is a normal coarse subgroup of $\crse G$. When this is the case, $\cker(\crse q)\crse{= (A^{-1} \ast A)}^{\ast n}$.
\end{prop}

\begin{proof}
 Assume first that $\crse{(A^{-1} \ast A)}^{\ast n}$ is a normal coarse subgroup of $\crse G$. As in the beginning of the proof of Lemma~\ref{lem:p1:quotient coarse structure as composition_cosets}, we see that $(A^{-1} \ast A)^{\ast n}$ is bounded in $\crse{G/A}$ and, similarly, that $A$ is bounded in $\crse{G/(A^{-1} \ast A)}^{\ast n}$. This implies that $\crse{G/A}$ and $\crse{G/(A^{-1} \ast A)}^{\ast n}$ are quotients of one another, and hence they are the same coarse quotient of $\crse G$. We may then apply the First Coarse Isomorphism Theorem to deduce that $\crse{(A^{-1} \ast A)}^{\ast n}$ is the coarse kernel of the coarse quotient $\crse{q\colon G\to G/A}$.
 
 For the converse implication, assume that $\crse{q}$ has a coarse kernel $\crse{K}\coloneqq \cker(\crse q)\crse{\trianglelefteq G}$. Let $\crse G=(G,\CE)$, fix adapted representatives for the coarse operations and, for convenience, let $B\coloneqq A^{-1} \ast A$. We may assume that $e\in K$.
 Since $K$ is bounded in $\crse{G/A}$, it follows from Lemma~\ref{lem:p1:quotient coarse structure as composition_cosets} that there is a $g\in G$ and a $n\in\NN$ large enough so that $K$ is $\CE$\=/coarsely contained in $g\ast B^{*n}$ (see Remark~\ref{rmk:p1:quotient as composition_cosets_covering}).
 Since both $K$ and $B^{*n}$ contain $e$, the element $g$ is in the coarsely connected component of $\crse G$ containing $e$. We deduce that $K\csub_\CE B^{*n}$.

 On the other hand, $B^{*n}$ is bounded in the coarse quotient $\crse{G/A}$ and therefore $B^{*n}\csub_\CE K$ by definition of the kernel. It follows that $\crse B^{*n} \crse{=K}$ is a normal coarse subgroup of $\crse G$.
\end{proof}

\begin{cor}
  The quotient homomorphism $\crse {q\colon G\to G/A}$ fits in a short exact sequence if and only if it is of the form
 \[
 \begin{tikzcd}
  \tobj\arrow[r] & 
  \crse{(A^{-1} \ast A)}^{\ast n} \arrow[r] &
  \crse{G} \arrow[r, "\crse{q}"] &
  \crse{G/A} \arrow[r]&
  \tobj
 \end{tikzcd}
\]
for some $n\in \NN$ such that $\crse{(A^{-1} \ast A)}^{\ast n}$ is a normal coarse subgroup of $\crse G$.
\end{cor}

\section{A Criterion for the Existence of Coarse Kernels}

We now provide a criterion to determine whether a coarse homomorphism admits a coarse kernel. To do so, we start by adapting the notion of `defect' for $\RR$\=/quasimorphisms to the setting of coarse groups (Definition~\ref{def:p1:quasimorphism}). Recall that the defect of an $\RR$\=/quasimorphism $\phi\colon G\to \RR$ is defined as $D_\phi\coloneqq \sup_{x,y\in G}\abs{\phi(xy)-\phi(x)-\phi(y)}$. It is common to assume that $\phi(x^{-1})=-\phi(x)$ or, at the very least, that $\phi(e_{G})=0$. In this case, the defect has the obvious property that for every $x,y\in G$ 
\[
 \phi(xy)\in \phi(x)+\phi(y)+[-D_\phi,D_\phi]
 \quad\text{ and }\quad
 \phi(x^{-1})\in -\phi(x)+[-D_\phi,D_\phi].
\]
The above is the property which we wish to reproduce for coarse homomorphisms of coarse groups. 
In the coarse setting, it is more convenient to use a notion of `defect relation'.

Let $\crse G =(G,\CE)$ and $\crse H=(H,\CF)$ be coarse groups. The \emph{defect relation} of a function $f\colon G\to H$ is the relation 
\[
 F_f\coloneqq\bigbraces{\bigparen{f(x_1\ast_{G} x_2)\,,\, f(x_1)\ast_H f(x_2)}\mid x_1,x_2\in G}\cup
 \bigbraces{\bigparen{f(x^{-1}),f(x)^{-1}}\mid x\in G}
 \quad\subseteq\  H\times H. 
\]
By definition, $f(x_1\ast_{G} x_2)\torel{F_f}f(x_1)\ast_H f(x_2)$ and $f(x^{-1})\torel{F_f}f(x)^{-1}$ for every $x_1,x_2,x\in G$. 

\begin{rmk}
 If $f$ is controlled, then it is a coarse homomorphism if and only if $F_f\in\CF$. Note that for this characterization of coarse homomorphism it is not really necessary to add to $F_f$ the points $(f(x)^{-1},f(x)^{-1})$. However, doing so will simplify some algebraic manipulations later on.
\end{rmk}

We learned the idea behind the following observation from Heuer--Kielak \cite{kielak_quasi}. Some versions seem to go much further back in time: it is implicit in the work of Meyer \cite{meyer2000algebraic} and it also appears in constructions of quasi\=/crystals and approximate subgroups (see \cite{bjorklund2018approximate,bjorklund2021introduction}). 
Informally, it says that if a subset $B$ of the codomain of a coarse homomorphism $f$ has the property that a ``suitably large'' neighborhood of $B$ can be covered with boundedly many translates $B\ast f(a)$, then the preimage $f^{-1}(F_f(B))$ defines a coarse subgroup. The precise statement is as follows:

\begin{lem}\label{lem:p1:preimage of bounded is subgroup}
 Let $f\colon \crse G \to \crse H$ be a function between coarse groups, fix some $B\subseteq H$ and let $K\coloneqq f^{-1}\paren{F_f(B)}$. Assume that there exists a bounded neighborhood of the identity $e_{G}\ceq A \subseteq G$ such that for every point $y$ in 
 \[
 F_f\Bigparen{\bigparen{F_f(B)\ast_H F_f(B)}\cup ({F_f}(B)} \subseteq H
 \]
 there is an $a\in A$ with $y\ast_H f(a)\in B$. 
 Then $\crse K$ is a coarse subgroup of $\crse G$.
\end{lem}
\begin{proof}
 We begin by showing that $K\ast_{G} K \csub K\ast_{G} A^{-1}$. Fix $x_1,x_2\in K$, then we have
 \[
  f(x_1\ast_{G} x_2)\in F_f\bigparen{f(x_1)\ast_H f(x_2)}\subseteq F_f\bigparen{F_f(B)\ast_H F_f(B)}
 \]
 and therefore there is an $a\in A$ such that $f(x_1\ast_{G} x_2)\ast_H f(a)\in B$. Again, we have
 \[
  f\bigparen{(x_1\ast_{G} x_2)\ast_{G} a}\in F_f\bigparen{f(x_1\ast_{G} x_2)\ast_H f(a)}\subseteq F_f(B)
 \]
 and therefore $(x_1\ast_{G} x_2)\ast_{G} a\in K$. Multiplying by $a^{-1}$, we see that $x_1\ast_{G} x_2$ belongs to a controlled thickening of $K\ast_{G} a^{-1}$ (the control on the thickening is uniform in $x_1,x_2$, as it is obtained by one application of (Associativity), (Inverse) and (Identity)). The same argument also shows that $K^{-1}\csub K\ast_{G} A^{-1}$. Since $A\ceq e_{G}$, this also shows that $K\ast_{G} K\csub K$ and $K^{-1}\csub K$.
\end{proof}

\begin{rmk}\label{rmk:p1:defect set}
 A few comments:
 \begin{enumerate}
  \item Recall that a subset $X\subseteq G$ of a set\=/group $G$ is  an approximate subgroup if $X=X^{-1}$ and there exists a finite $A\subseteq G$ with $X\ast X\subseteq A\ast X$. If $G$ is a set\=/group and $f\colon G\to H$ is a function so that Lemma~\ref{lem:p1:preimage of bounded is subgroup} applies for some finite set $A\subseteq G$, then the above proof shows that $K\coloneqq f^{-1}(F_f(B))$ satisfies $K\ast K\subseteq K\ast A^{-1}$. If one also knows that $K=K^{-1}$ then this implies that $K$ is an approximate subgroup of $G$ (a priori, Lemma~\ref{lem:p1:preimage of bounded is subgroup} only implies that $K^{-1}\subseteq K\ast A^{-1}$). This fact is used in \cite{kielak_quasi} to show that (homogeneous) $\RR$\=/quasimorphisms have an approximate subgroup as ``quasi\=/kernel''. See also \cite{bjorklund2021introduction}.
  
  \item If $H$ is a set\=/group, it is generally simpler to use a ``defect set'' instead of a defect relation. Namely, we define the \emph{defect set}\index{defect} of a function $f\colon G \to H$ as the smallest set $D_f\subseteq H$ such that 
\[
 F_f\subseteq \Delta_H\ast (D_f\times\{e_H\}).
\]
  Now, for every $B\subseteq H$ we have $F_f(B)\subseteq (\Delta_H\ast (D_f\times\{e_H\}))(B) =B D_f$. The proof of Lemma~\ref{lem:p1:preimage of bounded is subgroup} then implies that if there exists $A\ceq e_{G}$ such that $BD_fBD_f^2\subseteq B f(A)^{-1}$ then $[f^{-1}(BD_f)]$ is a coarse subgroup.
 \end{enumerate}
\end{rmk}

Further extending the idea of Lemma~\ref{lem:p1:preimage of bounded is subgroup}, we prove a criterion for the existence of coarse kernels. This is based on the existence of appropriate coarse exhaustions by bounded sets:

\begin{de}\label{def:p1:coarse exhaustion}
 A family of subsets $A_i$, $i\in I$ of a coarse space $\crse X$ \emph{coarsely exhausts $\crse X$}\index{coarsely!exhausting} if there is a controlled covering $\fkb\in\fkC(\CE)$ such that for every bounded subset $C\subseteq \crse X$ there exists $i\in I$ with $C\subseteq \st(A_i,\fkb)$.
\end{de}

\begin{prop}\label{prop:p1:criterion for existence of ker}
 Let $\crse {f \colon G \to H}$ be a coarse homomorphism. If there exists a family of bounded neighborhoods of the identity  $e_{G}\ceq A_i\subseteq G$, $i\in I$ whose image is coarsely exhausting $f(G)$, then $\crse f$ has a coarse kernel.
\end{prop}

\begin{proof}
 We will prove the claim by explicitly constructing a bounded subset of $H$ whose preimage under $f$ is a representative of the coarse kernel of $\crse f$.
 Let $C\coloneqq \braces{y\ast_H y^{-1}\mid y\in H}$, then $C\ceq e_H$ because $\crse H$ is a coarse group. Let $\fkb\in\fkC(\CF)$ be a controlled covering showing that $f(A_i)$ is coarsely exhausting $f(G)$. We claim that if we are given $x\in G$ and $y\in \st(f(x),\fkb)$, then $y\ast_H f(x^{-1})\in \st\bigparen{C,\fkb\ast_H f(x^{-1})}$. In fact, $y\in \st(f(x),\fkb)$ if and only if there is a $B\in\fkb$ with $y,f(x)\in B$. In particular, both $y\ast_H f(x^{-1})$ and $f(x)\ast_H f(x^{-1})$ are in $B\ast_H f(x^{-1})$. The claim follows because
 \[
  B\ast_H f(x^{-1})
  \subseteq \st\bigparen{f(x)\ast_H f(x^{-1})\,,\, \fkb\ast_Hf(x^{-1})}
  \subseteq \st\bigparen{C\,,\,\fkb\ast_H f(x^{-1})}.
 \]
 
 Let $A\coloneqq \bigcup_{i\in I}A_i$. The partial covering $\fkb\ast_H f(\pts{A}^{-1})=\braces{B\ast_H f(a^{-1})\mid B\in\fkb,a\in A}$ is controlled.
 We let 
 \[
  \overline{C}\coloneqq \st\bigparen{C\,,\,\fkb\ast_H f(\pts{A}^{-1})},
 \]
 and note that $ \overline{C}\ceq C\ceq e_H$. Since $\crse f$ is a coarse homomorphism, $F_f(\overline{C})$ is also bounded. Let $K\coloneqq f^{-1}(F_f(\overline{C}))$, we claim that for every bounded $D\subseteq H$ we have $f^{-1}(D)\csub K$. This shows that $\crse K$ is the coarse preimage of $\cunit_{\crse H}$ and hence it is the coarse kernel of $\crse f$.
 
 Fix such a $D$, we can assume that $D\subseteq f(G)$. By hypothesis there is an $i\in I$ such that $D\subseteq \st(f(A_i),\fkb)$. Fix any $x\in f^{-1}(D)$ and let $y\coloneqq f(x)$. By construction, there is an $a\in A_i$ such that $y\in \st(f(a),\fkb)$ and by the previous argument we deduce that 
 \[
  f(x)\ast_H f(a^{-1})  =
  y\ast_H f(a^{-1})\in \overline{C}. 
 \]
 We can now proceed as in the end of the proof of Lemma~\ref{lem:p1:preimage of bounded is subgroup} to see that $x\ast_{G} a^{-1}\in K$ and hence $x$ belongs to a controlled thickening of $K\ast_{G} A_i$. The choice of this last thickening does not depend on $x\in f^{-1}(D)$ and therefore $f^{-1}(D)\csub K\ast_{G} A_i\csub K$, as desired. 
\end{proof}

Recall that if $\crse G=(G,\CE)$ is a coarsified set\=/group then a function $\phi\colon (G,\CE)\to (\RR,\CE_{\abs{\mhyphen}})$ defines a coarse homomorphism if and only if it is a controlled $\RR$\=/quasimorphism. Proposition~\ref{prop:p1:criterion for existence of ker} implies the following.
 
\begin{cor}\label{cor:p1:quasimorphism have ker}
 Let $\crse G=(G,\CE)$ be a connected coarsified set\=/group and $\phi\colon G\to\RR$ an $\RR$\=/quasimorphism. If $\phi\colon (G,\CE)\to(\RR,\CE_d)$ is controlled then $\crse \phi$ has a coarse kernel.
\end{cor}
\begin{proof}
 Assume for convenience that $\phi$ is homogeneous, \emph{i.e.}\ $\phi(x^k)=k\phi(x)$ for every $k\in \ZZ$ (every $\RR$\=/quasimorphism is close to a homogeneous one by Proposition~\ref{prop:p1:homogeneization.of.qmorph}). If $\phi$ is trivial then $\cker(\crse \phi)=\crse G$ and there is nothing to prove.
 
 If $\phi$ is non\=/trivial, there exists an $a\in G$ such that $\phi(a)>0$. Letting $A_n\coloneqq \braces{a^k\mid -n\leq k\leq n}$, we see that $\bigcup_{n\in\NN}\phi(A_n)$ is a lattice in $\RR$ and hence the sequence of sets $\phi(A_n)$ is coarsely exhausting $\RR$. Since $(G,\CE)$ is coarsely connected and $A_n$ is finite for every $n\in\NN$, the sets $A_n$ are bounded neighborhoods of the identity. It follows from Proposition~\ref{prop:p1:criterion for existence of ker} that $\phi$ has a coarse kernel.
\end{proof}

It is convenient to record the following consequence of Proposition~\ref{prop:p1:criterion for existence of ker}.

\begin{cor}\label{cor:p1:criterion for existence of ker}
 Let $\crse G$ be a coarsely connected coarse group, $\crse H$ a coarsified set group and $\crse{f \colon G\to H}$ a coarse homomorphism. 
 If there exists a bounded neighborhood of identity $B\subseteq H$ so that for every bounded set $C\subseteq H$ there exist $g_1,\ldots,g_n\in G$ with \[C\subseteq \bigparen{Bf(g_1)}\cup\cdots \cup \bigparen{Bf(g_n)}\] then $\crse f^{-1}(BD_f)$ is the coarse kernel of $\crse f$ where $D_f\subseteq H$ is the defect set (Remark~\ref{rmk:p1:defect set}).
\end{cor}

\begin{rmk}
 The above corollaries show that if a coarse group $\crse H$ is ``locally small'' then coarse homomorphisms with values in $\crse H$ generally have a coarse kernel. More precisely, if $\crse H$ has \emph{bounded geometry} (\cite[Definition~3.8]{roe_lectures_2003}, \cite[Definition~3.1]{winkel2021geometric}) and $\crse G$ is coarsely connected, then every coarse homomorphism $\crse {f\colon G\to H}$ has a coarse kernel.
\end{rmk}

\section{Some Comments and Questions}

In view of developing a structure theory for coarse groups, it is important to understand which coarse homomorphisms have a coarse kernel. Namely, we are interested in the following general problem:

\begin{problem}\label{prob:when does ker exist?}
 When does a coarse homomorphism have a coarse kernel?
\end{problem}

Proposition~\ref{prop:p1:kernel of cosets spaces} provides a complete answer to the above for quotient maps into coarse coset spaces, while
Lemma~\ref{lem:p1:no kernel disconnected} deals with the coarse homomorphisms $\crse {f\colon G\to H}$ where $\crse G=(G,\mincrs)$ is a trivially coarse group. Moreover, in its proof we show that the coarse preimage $\crse{f^{-1}(H_\cunit)\leq G}$ of the coarse connected component of the identity  of $\crse H$ always exists. In particular, $\crse{f\colon G\to H}$ has a coarse kernel if and only if $\crse {f|_{f^{-1}(H_\cunit)}\colon f^{-1}(H_\cunit)\to H_\cunit}$ does.
We thus find that Problem~\ref{prob:when does ker exist?} is particularly interesting for coarse homomorphisms of coarsely connected coarse groups.

\

Note that the same argument as in Corollary~\ref{cor:p1:quasimorphism have ker} shows that if $\crse{G}$ is a coarsely connected coarse group then every coarse homomorphism $\crse G\to (\RR,\CE_d)$ has a coarse kernel.
Examining the proof suggests that the main reason why this is the case is the fact that $(\RR,\CE_d)$ is ``one dimensional''. 
More generally, it seems plausible that every coarse homomorphism $\crse {f\colon G\to H }$ where $\crse{H}$ is a coarse group of finite asymptotic dimension has a coarse kernel (see \cite[Chapter 9]{roe_lectures_2003} for a definition of asymptotic dimension).

\begin{rmk}
It is also interesting to restrict the scope of Problem~\ref{prob:when does ker exist?} to specific classes of coarse groups. For example,  we will show in Section~\ref{sec:p2:Banach homomorphisms} that a coarse homomorphism between two Banach spaces has a coarse kernel if and only if the image of its homogenization is a closed vector subspace.
\end{rmk}

The above remark shows that Problem~\ref{prob:when does ker exist?} is non trivial also when restricted to controlled homomorphisms of set\=/groups. We pose the following:

\begin{problem}\label{prob:p1:when does ker exist - setgroups}
 Let $\crse{G,\ H}$ be coarsified set\=/groups let $f\colon G\to H$ be a homomorphism that is also controlled, so that $\crse f\colon \crse G\to \crse H$ is a coarse homomorphism. When does $\crse f$ have a coarse kernel?
\end{problem}

This problem should not be understood as an attempt to produce necessary and sufficient conditions for the existence of coarse kernels.
Rather, it highlights the fact that given a controlled homomorphisms between (coarsified) set-groups, it is a natural question to ask whether this specific homomorphism has a coarse kernel.
When this is the case, understanding such a kernel can help to illuminate the whole geometric picture.
One first result in this direction was very recently obtained by Kawasaki--Kimura--Maruyama--Matsushita--Mimura, who investigated certain fine information regarding comparisons of stable mixed commutator lengths and phrased it in terms of coarse kernels of inclusions \cite{kawasaki2023coarse}.

%% file: SelectedTopics.tex
\part{Selected Topics}

\chapter{Coarse Structures on Set-Groups}\label{ch:p2:coarse structures on set groups} 

We begin the second part of this monograph with a closer look at coarsified set\=/groups. There are various reasons for doing so. For one, coarsified set\=/groups provide the simplest constructions of coarse groups and are hence an important source of examples which is necessary to understand and develop the theory of coarse groups.

A second reason is that it is a natural and interesting question to understand if a fixed set\=/group $G$ admits interesting coarsifications.
Namely, the existence of equi bi\=/invariant coarse structures on $G$ is tightly connected with the algebraic properties of $G$, and their study connects the theory of coarse groups to classical group\=/theory.
In the following sections we approach the study of coarsified set\=/groups from the latter perspective and we start investigating the question of ``which set\=/groups admit `interesting' coarsifications?".

%Of course, `interesting' is a rather subjective term. However, 
We know that every set\=/group has at least two obvious coarsifications: %that we can certainly agree are not very interesting.
Namely, the bounded and the trivial coarsifications $\varcrs{max}$, $\varcrs{min}$.
Further recall that every coarse group $\crse G$ fits into the short exact sequence
\begin{equation}\label{eq:p2:ses}
 \begin{tikzcd}
  \tobj\arrow[r] & 
  \crse{G_{e}} \arrow[r] &
  \crse{G} \arrow[r, "\crse{q}"] &
  \crse{G/G_{e}} \arrow[r] &
  \tobj
 \end{tikzcd}
\end{equation}
(see Corollary~\ref{cor:p1:ses quotient by id comp}) and is hence decomposed into a connected (Definition~\ref{def:p1:connected coarse structure}) coarse group and a trivially coarse one.
In view of this, we specialise to only looking for coarsely connected coarsifications.
The very least we can ask from such a coarsification is that it is not bounded. The general problem then becomes:

\begin{problem}\label{prob:p2:which groups are coarsifiable}
 Let $G$ be your favourite set\=/group. Does it admit unbounded, connected coarsifications?
\end{problem}

Not every set\=/group admits unbounded connected coarsifications.
As an obvious example, the only way to make a finite set\=/group into a coarsely connected coarse group is by giving it the bounded coarse structure.
For convenience, we make the following definition.

\begin{de}\label{def:p2:intrinsically bounded}
 A set\=/group $G$ is \emph{intrinsically bounded}\index{intrinsically bounded} if the bounded coarse structure $\varcrs{max}$ is the only connected equi bi\=/invariant coarse structure on $G$.
\end{de}

So, given a group $G$, Problem~\ref{prob:p2:which groups are coarsifiable} asks whether $G$ is intrinsically bounded.
In the following sections we will make a few general observation around intrinsic boundedness and its metric analogues before discussing a few examples.
To begin, we mention instances of Problem~\ref{prob:p2:which groups are coarsifiable} that have already been studied (using other languages),
and that it can be extremely challenging to understand whether a given group is intrinsically bounded.

\begin{rmk}
 Notice that intrinsic boundedness can be seen as rigidity phenomenon: it says that any attempt to `blur' the algebraic structure of the set\=/group collapses to produce something bounded or coarsely disconnected.
 More examples are also given in Chapters \ref{ch:p2:coarse structures on Z} and \ref{ch:p2:canc metrics}, where we study coarsifications of $\ZZ$ and bi\=/invariant word metrics on set\=/groups.
\end{rmk}

\section{Connected Coarsifications of Set-Groups}\label{sec:p2:connected coarsification}
By definition, $(G,\CE)$ is coarsely connected if and only if each pair $\{g_1,g_2\}$ is a $\CE$\=/bounded set. By composition, each finite subset of $G$ must be bounded. Recall that we defined $\varcrs[grp]{fin}$ as the minimal coarse structure for which all the finite sets are bounded (Example~\ref{exmp:p1:finite sets group coarse structure}). It then follows that $G$ is intrinsically bounded if and only if $\varcrs[grp]{fin}=\varcrs{max}$.
In turn, it is simple to prove the following algebraic characterization of intrinsic boundedness:

\begin{prop}\label{prop:p2:intrinsically bounded iff bounded generation}
 A set\=/group $G$ is intrinsically bounded if and only if there exist finitely many conjugacy classes $[g_1],\ldots,[g_n]\subseteq G$ and an $m\in \NN$ such that every element of $G$ can be written as the product of at most $m$ elements in $[g_1],\ldots,[g_n]$.
\end{prop}
\begin{proof}
To prove sufficiency, assume that we are given conjugacy classes $[g_1],\ldots,[g_n]$ as in the statement. Since $\{e,g\}$ is $\varcrs{fin}$\=/bounded for every $g\in G$, it follows from condition (U4) of Lemma~\ref{lem:p1:properties of neighborhoods of id} that every conjugacy class of $G$ is $\varcrs[grp]{fin}$\=/bounded. We then see that $G$ is bounded because it is contained in a finite union of finite products of bounded sets.

 Vice versa, let $\fku$ be the family of all the sets $A\subseteq G$ such that
 \begin{itemize}
  \item $e\in A$;
  \item there exist finitely many conjugacy classes $[g_1],\ldots,[g_n]\subseteq G$ and an $m\in \NN$ such that every element of $A$ can be written as the product of at most $m$ elements in $[g_1],\ldots,[g_n]$
 \end{itemize}
 It is easy to see that $\fku$ satisfies conditions (U0)--(U4) of Section~\ref{sec:p1:determined locally}. By Proposition~\ref{prop:p1:families of identity neighbourhoods_setgroups}, $\fku$ is the family $\fkU_e(\CE)$ of bounded neighborhoods of the identity for some equi bi\=/invariant coarse structure $\CE$. Since $\{e,g\}\in\fku$ for every $g\in G$, the coarse structure $\CE$ is connected.

 On the other hand, it follows from the proof of the other implication that $\fku$ is contained in the family of $\varcrs[grp]{fin}$\=/bounded sets. By minimality of $\varcrs[grp]{fin}$, it follows that $\CE = \varcrs[grp]{fin}$ and hence $\fku$ is equal to the family of $\varcrs[grp]{fin}$\=/bounded sets. In particular, we see that $G$ is $\varcrs[grp]{fin}$\=/bounded if and only if there exist $[g_1],\ldots, [g_n]$ as in the statement.
\end{proof}

As a sample application of the above criterion, we immediately see that the infinite dihedral group
\[
 D_\infty=\ZZ\rtimes \ZZ/2\ZZ=\angles{r,\tau\mid \tau^2=1,\ \tau r=r^{-1}\tau}
\]
is intrinsically bounded. In fact, the elements of the form $\tau r^{2k}$ are all conjugate, and the same is true for those of the form $\tau r^{2k+1}$. It follows that $D_\infty$ is intrinsically bounded because every element in $D_\infty$ is the product of at most two elements of the form $\tau r^n$.

\begin{rmk}
Note that  $\ZZ$ is an index $2$ subgroup of $D_\infty$. This shows that intrinsic boundedness is a rather fine algebraic notion that is not invariant under taking finite index subgroups. It is interesting to compare this with geometric group theory: when working with the \emph{left}\=/invariant word metric of a finitely generated group, taking finite index subgroups or quotients with finite kernel yields quasi\=/isometric groups.
\end{rmk}

\begin{rmk}
 A set\=/group $G$ is \emph{boundedly simple} (also known as \emph{uniformly simple}) if there is some $m\in\NN$ such that any pair of non\=/trivial elements $h,g\in G$, $h$ can be written as a product of at most $m$ conjugates of $g$ or $g^{-1}$.
 This is a notion of model theoretic interest, and it immediately follows from \ref{prop:p2:intrinsically bounded iff bounded generation} that every boundedly simple group is intrinsically bounded.
\end{rmk}

One (almost tautological) tool to show that a group is \emph{not} coarsely bounded is the following:

\begin{lem}\label{lem:p2:criterion for unboundedness - crse hom}
 If there exists a coarse homomorphism $\crse f\colon (G,\mincrs)\to \crse H$ into a coarsely connected coarse group with unbounded image, then $G$ is not intrinsically bounded.
\end{lem}
\begin{proof}
 Taking the pull-back, $(G,f^{*}(\CF))$ is a connected unbounded coarsification of $G$ by Lemma~\ref{lem:p1:pull-back.under.hom.is.crsegroup}.
\end{proof}

\begin{rmk}
 One way of rephrasing the tautological nature of this criterion is by recalling that coarse quotients (Section~\ref{sec:p1:coarse quotients}) of coarse groups are obtained by enlarging the coarse structure.
 Since the coarse image $[f(G)]\crse{\leq H}$ is a coarse quotient of $(G,\mincrs)$, it is unbounded if and only if the induced coarsification of $G$ is unbounded.
\end{rmk}

Despite its simplicity, Lemma~\ref{lem:p2:criterion for unboundedness - crse hom} is one of the most powerful tools for showing that a group is not intrinsically bounded.
More specifically, the coarse homomorphisms are often $\RR$\=/quasimorphisms (Definition~\ref{def:p1:quasimorphism}). See Section~\ref{sec:p2:examples of (un)bounded} for some examples.

\section{Metric coarsifications and Bi-invariant Metrics}\label{sec:p2:biinvariant metrics}
We say that a coarse structure $\CE$ on a set $X$ is \emph{metrizable}\index{coarse structure!metrizable}\index{coarsification!metric} if it coincides with the metric coarse structure $\CE_d$ defined by a metric $d$ on $X$.
Recall that a coarse structure $\CE_d$ induced by a metric $d$ is always coarsely connected.
In particular, any unbounded metrizable coarsification coarsification of $G$ falls within our rather generous idea of ``interesting'' coarsification.

We already know that if $G$ is a set\=/group and $d$ is a bi\=/invariant metric on it, then $(G,\CE_d)$ is a coarse group, so (unbounded) bi\=/invariant metrics give rise to (unbounded) metrizable coarsifications.
In principle, it is not immediately clear whether the converse should also be true.
However, Roe proved that a connected coarse structure is metrizable if and only if it is countably generated\footnote{%
As a consequence, we see that the coarsification $(\RR,\varcrs[grp]{fin})$ cannot be induced by a bi\=/invariant metric of $\RR$.
} (\cite[Theorem 2.55]{roe_lectures_2003}, see also \cite[Theorem 9.1]{protasov2003ball} and \cite[Section 4]{rosendal2017coarse}).
With a little care, one can tweak this characterization and show that every metrizable coarsification of a set\=/group $G$ is indeed induced by a bi\=/invariant metric of $G$:

\begin{lem}\label{lem:p2:metric coarsifications}
 Let $(G,\CE)$ be a coarsely connected coarsified set\=/group. The following are equivalent:
 \begin{enumerate}[(i)]
  \item there exists a bi\=/invariant metric $d$ so that $\CE=\CE_d$;
  \item $\CE$ is metrizable;
  \item $\CE$ is countably generated.
 \end{enumerate}

\end{lem}

\begin{proof}
$(i)\Rightarrow (ii)$ is obvious, for $(ii)\Rightarrow (iii)$ it is enough to remark that a metric coarse structure $\CE_d$ is always generated by the countable family $E_n\coloneqq\braces{(x_1,x_2)\mid d(x_1,x_2)\leq n}$.

 $(iii)\Rightarrow (i)$: let $\CE=\angles{E_n\mid n\in\NN}$. Replacing $E_n$ with $\Delta_{G}\ast (E_n\cup \op{E_n})\ast \Delta_{G}$, we may assume that each $E_n$ is symmetric and invariant under left and right translation. Let also $E_0\coloneqq \Delta_G$. We define a function $d'\colon G\times G\to G$ by
 \[
  d'(x,y)\coloneqq \inf\braces{n\mid (x,y)\in E_n}\cup\{+\infty\}.
 \]
 By construction, $d'$ is symmetric, $d'(x,y)= 0$ if and only if $x=y$, and $d'$ is invariant under both left and right multiplication. It need not be a metric, as it may fail to be transitive and can take infinite values. Both issues are solved by letting
 \[
  d(x,y)\coloneqq\inf_{k,x_0,\ldots,x_k} \sum_{i=1}^k d'(x_{i-1},x_i),
 \]
 where the infimum is taken over all the finite sequences $x_0,\ldots, x_k\in G$ such that $x_0=x$ and $x_k=y$. In fact, transitivity is built into the definition of $d$, and $d$ takes only finite values because we are assuming that $\CE$ is connected. It is easy to verify that $\CE=\CE_d$.
\end{proof}

Thus Problem~\ref{prob:p2:which groups are coarsifiable} can be specialised by the stronger requirement that the group admit an unbounded metrizable coarse structure, this would then take a more familiar form:

\begin{problem}\label{prob:p2:which groups have bi-invariant metrics}
 Let $G$ be your favourite set\=/group. Does it admit unbounded, bi\=/invariant metrics?
\end{problem}

This formulation is much more classical than Problem~\ref{prob:p2:which groups are coarsifiable} (references are provided in the next section, together with some examples).

Recall that if a set $S\subseteq G$ normally generates $G$, the bi\=/invariant word metric $d_{\overline{S}}$ induced by $S$ is the word metric associated with the generating set $\overline{S}=\{gsg^{-1}\mid s\in S,\ g\in G\}$.
This is a bi\=/invariant metric because $\overline{S}$ is invariant under conjugation.
Bi\=/invariant word metrics form an important class of metrics because of the following `maximality' property.
If $d$ is any bi\=/invariant metric on $G$ such that $M\coloneqq\sup\{d(s,e)\mid s\in S\}$ is finite, an iterated application of the triangle inequality implies that
\[
 d(g_1,g_2)\leq M d_{\overline{S}}(g_1,g_2)
\]
for all $g_1,g_2\in G$.

Recall also that if $G$ is a finitely normally generated group then, up to coarse equivalence, the bi\=/invariant word metric $d_{\overline{S}}$ does not depend on the choice of finite normally generating set $S$. The associated metric coarse structure is the \emph{canonical coarse structure} $\varcrs{bw}=\CE_{d_{\overline{S}}}$ of $G$ (and we then showed that $\varcrs{bw}=\varcrs[grp]{fin}$, see Corollary~\ref{cor:p1:canonical crse_str is E_grp_fin}).

One useful way of rephrasing Problem~\ref{prob:p2:which groups have bi-invariant metrics} (suggested to us by the referee) is as follows:

\begin{lem}\label{lem:p2:unbuonded bi-invariant metric iff normal subgroups}
 A set-group $G$ does not admit unbounded bi\=/invariant metrics if and only if the following two conditions are satisfied:

 \vspace{\topsep}\noindent
\begin{minipage}{0.08\textwidth}
 (BB-a)
\end{minipage}
\begin{minipage}{0.92\textwidth}
$G$ is not equal to the union of a strictly increasing sequence of normal subgroups;
\end{minipage}

\vspace{\topsep}\noindent
\begin{minipage}{0.08\textwidth}
 (BB-b)
\end{minipage}
\begin{minipage}{0.92\textwidth}
if $S\subseteq G$ normally generates $G$, then $d_{\overline{S}}$ is bounded.
\end{minipage}
\end{lem}
\begin{proof}
 We show the contrapositive.
 If (BB-b) fails, then $d_{\overline{S}}$ gives an unbounded bi\=/invariant metric, and there is nothing to prove.
 If (BB-a) fails, let $G_1\subsetneq G_2\subsetneq\cdots$ be a strictly increasing sequence of normal subgroups of $G$, so that $G=\bigcup_{n\in\NN}G_n$.
 We then obtain an unbounded bi\=/invariant metric letting $d(g_1,g_2)\coloneqq \min\braces{n\mid g_1^{-1}g_2\in G_n}$
 (this is in fact an ultrametric, as $d(g_1,g_3)\leq \max\{d(g_1,g_2),d(g_2,g_3)\}$).

 In the other direction, let $d$ be an unbounded bi-invariant metric on $G$, and let $B_k\coloneqq B_d(e,k)$ be the ball of radius $k$ around the identity (these are conjugation invariant).
 If there is a $k$ so that $G=\angles{B_k}$ then (BB-b) fails, because $kd_{B_k}\geq d$ is unbounded.
 Otherwise, we can choose a subsequence so that $G_n\coloneqq\angles{B_{k_n}}$ is a strictly increasing sequence of normal subgroups, hence (BB-a) fails.
\end{proof}

One easy consequence is that ``small'' set\=/groups are intrinsically bounded if and only if they do not admit unbounded bi\=/invariant metrics:

\begin{prop}\label{prop:p2:intrinsically bounded iff no biinv metric}
 A countably normally generated set\=/group $G$ is intrinsically bounded if and only if it admits no unbounded bi\=/invariant metrics.
\end{prop}
\begin{proof}
 One implication is clear, so by Lemma~\ref{lem:p2:unbuonded bi-invariant metric iff normal subgroups} it is enough to verify that if $G$ satisfies (BB-a) and (BB-b) then it is intrinsically bounded.
 Let $S=\{s_1,s_2,\ldots\}$ be a countable normally generating subset of $G$ and let $S_n\coloneqq\{s_1,\ldots,s_n\}$.
 We can write $G$ as the countable union of the normal subgroups $G_n\coloneqq\angles{\overline{S_n}}$.
 Therefore, (BB-a) shows that there exists $n$ so that $G=G_n$.

 The above paragraph implies that $G$ is actually finitely normally generated.
 We then know that $\varcrs[grp]{fin}$ coincides with the canonical coarse structure $\varcrs{bw}=\varcrs{d_{\overline{S_n}}}$, and the latter is bounded by (BB-b).
\end{proof}

On the other hand, Proposition~\ref{prop:p2:Bergman not bounded} below shows that there exist ``large'' set\=/groups that have no unbounded bi\=/invariant metrics but are not intrinsically bounded.

\begin{rmk}
 It is often the case that it is relatively simple to verify whether a well\=/behaved group $G$ satisfies (BB-a).
 For example, every normally finitely generated group satisfies it trivially.
 On the other hand, understanding property (BB-b) is often a very delicate matter.
\end{rmk}

To conclude this section, it is also interesting to observe that the coarsifications arising from bi\=/invariant word metrics are precisely the coarsely geodesic coarsifications of $G$.\index{coarsification!coarsely geodesic}
Namely, recall that a coarse structure is coarsely geodesic if it is connected and generated by a single relation (Definition~\ref{def:p1:crs_geod}). We can prove the following.

\begin{prop}\label{prop:p2:coarse geodesic coarsifications}
 Let $(G,\CE)$ be a coarsely connected coarsified set\=/group. Then $\CE$ is coarsely geodesic if and only if there exists a bi\=/invariant word metric $d_{\overline{S}}$ so that $\CE=\CE_d$.
\end{prop}
\begin{proof}
 It is clear from the definition of bi\=/invariant word metric that the relation $E_1\coloneqq\braces{(x_1,x_2)\mid d_{\overline{S}}(x_1,x_2)\leq 1}$ generates $\CE_{d_{\overline{S}}}$. This proves one implication.

 Vice versa, let $\CE$ be generated by a single relation $E$. As before, we may assume that $E$ is symmetric and invariant under left and right translation. Let $S\coloneqq {}_eE=\braces{x\in G\mid (e,x)\in E}$. Note that $S$ is invariant under conjugation: for every $s\in S$ and $g\in G$ the pair $(e,g s g^{-1})$ belongs to $E$ because
 \[
  (e,g s g^{-1})=(g,g)\ast (e,s)\ast (g^{-1}, g^{-1})\in \Delta_{G}\ast E\ast \Delta_{G}.
 \]
 It follows that the bi\=/invariant word metric $d_{\overline{S}}$ associated with $S$ is in fact equal to the word metric $d_S$.

 Also note that $E=\Delta_{G}\ast (\{e\}\times S)$. We then see that the formula for $d$ in the proof of Lemma~\ref{lem:p2:metric coarsifications} defines precisely the word metric $d_S$. As before, we now see that $\CE=\CE_{d_S}=\CE_{d_{\overline{S}}}$.
\end{proof}

\begin{rmk}
 Proposition~\ref{prop:p2:coarse geodesic coarsifications} provides an additional reason to be especially interested in bi\=/invariant \emph{word} metrics. In fact, coarsely geodesic coarse structures are much simpler to deal with than arbitrary coarse structures. For instance, in this setting, quasi\=/isometry invariants are invariants of coarse geometry (see Appendix~\ref{sec:appendix:quasifications}).
\end{rmk}

\section{Some examples of (un)bounded set-groups}
\label{sec:p2:examples of (un)bounded}

For the most part, the examples that we discuss are found in literature regarding the study of bi\=/invariant metrics on groups.
This is an abundant source, as this is a natural topic that makes its appearance in many unrelated settings (it is often phrased in terms of length functions\footnote{In various pieces of literature length functions are called norms.} that are invariant under conjugation).
Besides bi\=/invariant word metrics, there are numerous instances of groups with bi\=/invariant metrics of special significance.
For instance, \emph{e.g.}\ the Hamming distance in symmetric groups, the norm distance in the group of bounded operators on a Banach space, the Hofer norm on groups of Hamiltonian diffeomorphisms of symplectic manifolds \cite{Hof,LM}.
For topological groups, it is also natural to study bi\=/invariant metrics that are compatible with the topology: one recent work in this direction is \cite{polterovich2021norm}.
We refer to \cite{BIP} for a survey on bi\=/invariant metrics on groups of diffeomorphisms.
At the time of writing, Jarek K\c{e}dra is also preparing an extensive survey on bi-invariant metrics of groups.

\begin{rmk}
 Note that if $G$ is finitely normally generated by a set $S$, then the bi-invariant word metric $d_{\overline S}$ is bounded if and only if $G$ is intrinsically bounded
 (this is easily deduced from Proposition~\ref{prop:p2:intrinsically bounded iff bounded generation} or from the combination of Proposition~\ref{prop:p2:intrinsically bounded iff no biinv metric} and Lemma~\ref{lem:p2:unbuonded bi-invariant metric iff normal subgroups}).
 It is good to keep this fact in mind while browsing the literature.
\end{rmk}

As was the case for the study of unbounded coarsifications, one of the most versatile tools for showing that a group has unbounded bi-invariant metrics is by using coarse homomorphisms.
Namely, if $\crse H=(H,\CE_d)$ is a metrizable coarse group and $\crse f\colon (G,\mincrs)\to\crse H$ is a coarse homomorphism with unbounded image, then Lemma~\ref{lem:p2:criterion for unboundedness - crse hom} yields an unbounded connected coarsification of $G$ which is easily seen to be countably generated.
By Lemma~\ref{lem:p2:metric coarsifications}, this coarsification is hence induced by an unbounded bi-invariant metric.
The main player in this strategy is the study of $\RR$-quasimorphisms. For instance this is used in the following.

\begin{example}\label{exmp:p2:hyperbolic groups are unbounded}
Epstein--Fujiwara proved that every non-elementary hyperbolic group $G$ has unbounded $\RR$\=/quasimorphisms \cite{EF}.
 Therefore, such set\=/groups admit unbounded bi\=/invariant metrics.
%  The study of $\RR$\=/quasimorphisms of a given set\=/group $G$ is a fairly classical problem. Drawing from the literature, it is immediate to extend Corollary~\ref{cor:p2:hyperbolic groups are unbounded} to
 Similar results have been proved for other classes of groups of geometric interest as well.
 These include mapping class groups, relatively hyperbolic groups and, more generally, acylindrically hyperbolic groups.
 See \emph{e.g.}\ \cite{bestvina2016bounded, bestvina2002bounded, bestvina2009characterization, dahmani2017hyperbolically, hamenstadt2008bounded, hamenstadt2012isometry, hamenstadt2014lines, hull2013induced}.
%  Since $G$ is finitely generated, the boundedness condition is trivially satisfied.

 Examples of intrinsically bounded torsion free groups can be obtained by \emph{e.g.}~considering finite index subgroups of Chevalley groups \cite{gal2018finite} or certain set\=/groups of transformations that preserve a linear order \cite{calderoni2021simplicity, gal2017uniform}.
\end{example}

% \begin{rmk}
%  In Lemma~\ref{lem:p2:exists unbounded metric iff biinvariant word metric} we used the finiteness assumption to deduce that $\sup\{d(s,e)\mid s\in S\}$ is finite.
%  For more general groups one can still deduce that the bi\=/invariant word metric $d_{\overline{S}}$ bounds linearly every bi\=/invariant metric $d$ such that $\sup\{d(s,e)\mid s\in S\}$ is finite.
%  This can be applied \emph{e.g.}\ if $G$ is a topological group, $S$ is compact and $d$ is assumed to be compatible with the topology.
% \end{rmk}

\begin{example}
Conversely, interesting examples of intrinsically bounded set\=/groups are the special linear groups $\Sl(n,\ZZ)$ with $n\geq 3$: this follows from the fact that these groups are boundedly generated by the set of elementary matrices \cite{carter1983bounded,morris2005bounded}. More precisely, let $E_{ij}$ be the matrix with entries $1$ on the diagonal and in position $(i,j)$ and $0$ elsewhere. One can then show that every elementary matrix in $\Sl(n,\ZZ)$, $n\geq 3$ can be written as a commutator $[E_{i,j},E_{k,l}^m]$ (see \cite[Example 1.1]{BIP}) and hence the finite set of matrices $E_{ij}^{\pm 1}$ satisfies the hypotheses of Proposition~\ref{prop:p2:intrinsically bounded iff bounded generation}.
 With a lot more work, it can be shown that `many' Chevalley groups (listed in the references) are intrinsically bounded \cite{gal2011bi, gal2018finite, kedra2021strong, polterovich2021norm, trost2020strong, trost2020explicit, trost2021strong}.
\end{example}

\begin{rmk}
 In most of the works mentioned above, an important intermediate step in proving intrinsic boundedness is to show that the set\=/groups are boundedly generated.
 The latter is an important algebraic notion that is interesting in its own right.
\end{rmk}

\begin{rmk}
 It is not known whether a lattice $\Gamma$ in higher rank semisimple real Lie groups with finite center must be intrinsically bounded (\cite[Question 1.2]{gal2011bi}). If that was the case, it would follow from the Margulis Normal Subgroup Theorem that such a $\Gamma$ is ``strongly intrinsically bounded'' in the sense that the coarsely connected components of any equi bi\=/invariant coarse structure on $\Gamma$ are bounded. This should be contrasted with the infinite dihedral group $D_\infty$, which is intrinsically bounded but can be given a (disconnected) coarse structure making it coarsely equivalent to the disconnected union of two copies of $(\ZZ,\CE_{\abs{\mhyphen}})$.
\end{rmk}

\begin{rmk}
 We elaborate further on the short exact sequence
 \[
 \begin{tikzcd}
  \tobj\arrow[r] &
  \crse{G_{e}} \arrow[r] &
  \crse{G} \arrow[r, "\crse{q}"] &
  \crse{G/G_{e}} \arrow[r] &
  \tobj.
 \end{tikzcd}
 \]
 If $\crse G=(G,\CE)$ is a coarsified set\=/group, $G_e$ is a normal subgroup of $G$ and the coarse quotient $\crse{G/G_{e}}$ is naturally identified with the trivially coarse group $G/G_e$. This implies that the coarse structure $\CE$ restricts to an equi bi\=/invariant coarse structure on the set\=/group $G_{e}$, each coarsely connected component of $\crse G$ is a coset of $G_e$ and it is coarsely equivalent to $(G_e,\CE|_{G_{e}})$ (in a non-canonical way). The coarse space $(G,\CE)$ is identified with the disconnected union of the cosets of $G_{e}$.

 However, it is not always true that if $N\lhd G$ is a normal subgroup and $\CF$ is an equi bi\=/invariant coarse structure on the set\=/group $N$, then we can construct a coarsification of $G$ by equipping each coset of $N$ with a copy of the coarse structure $\CF$ and letting $(G,\CE)$ be the disconnected union of these cosets. In fact, the assumption that $\CF$ is equi bi\=/invariant under multiplication in $N$ is not enough to imply that the action by conjugation $G\curvearrowright N$ defines a coarse action $(G,\mincrs)\curvearrowright (N,\CF)$.

For instance, let $G$ be the semi\=/direct product $\ZZ^2\rtimes \Sl(2,\ZZ)$, where the action of $\Sl(2,\ZZ)$ on $\ZZ^2$ is the natural one. Then $N=\ZZ^2$ is a normal subgroup and the Euclidean metric gives a coarsification $(N,\CE_{\norm{\mhyphen}})$. However, $\CE_{\norm{\mhyphen}}$ is not equi invariant under the action of conjugation by $\Sl(2,\ZZ)<G$.
\end{rmk}

Let us now return to (un)bounded bi-invariant metrics for a different class of set\=/groups.

\begin{example}
 We already remarked that the infinite dihedral group $D_\infty$ is intrinsically bounded.
 More generally, the situation is very well understood for reflection groups.
 Specifically, let $(W,S)$ be a Coxeter group, where $S$ is the generating set used in its standard presentation.
 Then the set of conjugates $\overline{S}$ coincides with the set of reflections in $W$, and the bi\=/invariant word metric $d_{\overline{S}}$ is also called the reflection length.

 As noted in \cite{mccammond2011bounding}, if $S$ is infinite then taking products of distinct elements of $S$ shows that $d_{\overline{S}}$ is unbounded.
 When $S$ finite and $(W,S)$ is either spherical or affine, the main theorem of \cite{mccammond2011bounding}  provides an explicit upper bound on $d_{\overline{S}}$.
 Conversely, the main result of \cite{duszenko2012reflection} shows that $d_{\overline{S}}$ is unbounded in all the other cases.
\end{example}

Since we are interested in group coarsifications, the discussion up to this point was focused on bi-invariant metrics.
Above we provided many examples of groups that are intrinsically bounded even though they are very much not `bounded' from the point of view of geometric group theory.
Namely, since they are infinite order finitely generated groups, their word metric is an unbounded left-invariant metric (equivalently, these are groups that admit isometric actions with unbounded orbits).
However, leaving the world of finitely generated groups, one encounters infinite set\=/groups that do not admit any unbounded left\=/invariant metric!
Such set\=/groups are said to have Bergman's property, see \cite{bergman2006generating}.
More precisely, the main result of that paper is that the group of permutations of any infinite set satisfies the following conditions:

\vspace{\topsep}\noindent
\begin{minipage}{0.08\textwidth}
 (Ber-a)
\end{minipage}
\begin{minipage}{0.92\textwidth}
$G$ is not equal to the union of a strictly increasing sequence of subgroups;
\end{minipage}

\vspace{\topsep}\noindent
\begin{minipage}{0.08\textwidth}
 (Ber-b)
\end{minipage}
\begin{minipage}{0.92\textwidth}
if $S\subseteq G$ generates $G$, then $d_{{S}}$ is bounded.
\end{minipage}

\vspace{\topsep}\noindent
It is simple to observe that these conditions imply that $G$ does not have any unbounded left-invariant metric (see also \cite{cornulier2006strongly}).
In fact, we should remark that conditions (BB-a) and (BB-b) of Lemma~\ref{lem:p2:unbuonded bi-invariant metric iff normal subgroups} are the natural weakenings of (Ber-a)
and (Ber-b) needed to pass from Bergman's property to the absence of unbounded bi-invariant metric.
We refer to \cite{bergman2006generating} and references therein for this fascinating topic.
% Bergman's property is much stronger than what we need in our study of intrinsic boudnedness.
% However, tools and techniques that have been developed in that setting are obviously relevant in this context as well.
% For instance, it is worth pointing out that Lemma~\ref{lem:p2:unbuonded bi-invariant metric iff normal subgroups} is essentially a bi-invariant adaptation of an observation that Bergman already makes in \cite{bergman2006generating}.

Answering a question in a previous version of this manuscript, the referee gave the following example of a set\=/group that is not intrinsically bounded and yet does not admit an unbounded bi\=/invariant metric:

\begin{prop}\label{prop:p2:Bergman not bounded}
 Let $K$ be a finite perfect set\=/group, $I$ an uncountable set and $G$ be the subgroup of $K^I$ consisting of elements with countable support.
 Then $G$ has Bergman's property but is not intrinsically bounded.
\end{prop}
\begin{proof}
 It is proved in \cite{cornulier2006strongly} that the set\=/group $K^J$ has Bergman's property for any choice of index set $J$.
 If $X\subset G$ is a countable subset, there exists a countable $J\subset I$ so that $X\subseteq K^J<K^I$.
 This property implies that $G$ is Bergman.
 In fact, if there existed an unbounded left\=/invariant metric $d$ on $G$, we could choose a sequence $x_n\in G$ with $d(e,x_n)\to \infty$.
 Picking a countable $J$ so that $\{x_n\mid n\in\NN\}\subseteq K^J$, we would then see that the restriction of $d$ to $K^J$ is an unbounded left\=/invariant metric.

 It remains to show that $G$ is not intrinsically bounded.
 This can be done by specifying a set of bounded neighborhoods of the identity (Proposition~\ref{prop:p1:families of identity neighbourhoods_setgroups}).
 Namely, we declare $e\in U\subset G$ to be bounded if $U$ is contained in $K^J$ for some countable $J\subset I$.
 This family clearly satisfies (U0)--(U3) and, since $K^J\lhd K^I$ is normal, also (U4).
\end{proof}

\chapter{Coarse Structures on $\ZZ$}%%%%%%%%%%%%%%%%%%%%%%%
\label{ch:p2:coarse structures on Z}

The aim of this chapter is to construct coarsifications of $\ZZ$.
As we shall soon see, $\ZZ$ admits a great abundance of distinct coarsifications.
This is in contrast with the fact that the infinite dihedral group $D_\infty$ is intrinsically bounded (see Section~\ref{sec:p2:connected coarsification}). Since $\ZZ<D_\infty$ is a finite index subgroup, this indicates that determining the possible coarsifications of a given set\=/group requires some finesse in general.
We will use two approaches for generating a coarse group structure on $\ZZ$:
by choosing (infinite) sets of generators for the Cayley graph of $\ZZ$, or using topological coarsifications $\varcrs[left]{cpt}$ as in Example~\ref{exmp:p1:topological coarse structure abelian}.

\section{Coarse Structures Generated by Cayley Graphs}
\label{sec:p2:coarse structures on Z gen Cayley}
% Another method to define interesting coarse structures on $\ZZ$ is to consider word lengths associated with infinite generating sets. In other words, f
For any choice of (possibly infinite) generating set $S$ we can construct the Cayley graph ${\rm Cay}(\ZZ,S)$ and use the graph metric to define a coarse structure $\varcrs{Cay(S)}$\nomenclature[:CE1]{$\varcrs{Cay(S)}$}{metric coarse structure from the Cayley graph} on $\ZZ$. Of course, $\varcrs{Cay(S)}$ is equi bi\=/invariant because the metric is left invariant and $\ZZ$ is abelian.
If $S$ is finite we did not discover anything new: $\varcrs{Cay(S)}$ is nothing but $\varcrs[left]{fin}=\CE_{\abs{\mhyphen}}$.
However, when $S$ is infinite $\varcrs{Cay(S)}$ is a strictly larger coarse structure: the Cayley graph is not locally finite and there are bounded sets of infinite cardinality.
This is obviously the case if we let $G=S$ or any other ridiculously large set. On the other hand, if $S$ is sparse enough then we do get interesting coarse structures.

We thank the referee for pointing out the following facts:
%\fnote{I mostly edited the following because it bothers me to copy-paste the referee's argument}

\begin{lem}\label{lem:p2:distinct Cayley metrics}
 Let $W = \{ n! \mid n \geq 1 \}$  and let $S, T$ be subsets of $W$ containing $1$.
 Then $T$ is bounded under the metric $\abs{\mhyphen}_S$ if and only if $T \smallsetminus S$ is finite.
\end{lem}
\begin{proof}
We claim that for each $n \geq 2$ the element $n!$ has word length $n$ with respect to the word metric associated with $W' \coloneqq W \smallsetminus \{ n!\}$.
The lemma immediately follows from this claim, as taking any $n\in T\smallsetminus S$ gives an element of $T$ with $\abs{n}_S\geq\abs{n}_{W'}=n$.

It is clear that $\abs{ n!}_{W'} \leq n$.
To show the reverse inequality, write $n! = \sum _ { i\in I} v_i$ with $v_i \in \pm W'$, and suppose for contradiction that $\abs{I} < n$.
Let $J \coloneqq \{ i \mid \abs{v_i} <n! \}$ and $K = I \smallsetminus J$.
For $L \subseteq I$, denote $v_L \coloneqq \sum_{i \in L} v_i$, so that $n! = v_J + v_K$.

Since $(n+1)! $ divides $v_i$ for every $i \in K$, if $v_K \neq 0$ then $\abs{v_K} \geq (n+1)!$.
On the other hand, $\abs{v_J} \leq \abs{J} (n-1)!$, so the equality $v_J=-v_K +n!$ yields
\[
 \abs{J} (n-1)!\geq (n+1)!-n!
\]
and hence $\abs{J}\geq n(n+1)-n > n$, a contradiction.
Thus $v_K =0$, but then $n! = v_I \leq |J| (n-1)!$, thus $|J| \geq n$.
\end{proof}

Taking the word metrics associated with subsets $S\subseteq W$ proves the following.

\begin{cor}\label{cor:p2:continuum of coarsification}
 There is a continuum of distinct coarsely geodesic coarsifications on $\ZZ$.
\end{cor}

Note that the cardinality of the set of monogenic coarsifications of $\ZZ$ is bounded by $\abs{\CP(\ZZ\times\ZZ)}=2^{\aleph_0}$, so Corollary~\ref{cor:p2:continuum of coarsification} realizes the upper bound.
Leaving the realm of coarsely geodesic coarse structures, the trivial upper bound on the number of coarse structures on $\ZZ$ is $\abs{\CP(\CP(\ZZ\times\ZZ))}=2^{2^{\aleph_0}}$.
Assembling together Cayley metrics, it is not hard to show that this bound can also be attained:

\begin{prop}\label{prop:p2:2^continuum of coarsification}
 There are $2^{2^{\aleph_0}}$ distinct coarsifications of $\ZZ$.
\end{prop}
\begin{proof} 
 \ref{lem:p2:distinct Cayley metrics}
For a non principal ultrafilter $\mu$ on $W = \{ n! \mid n \geq 1 \}$, let
\[
 \CE_\mu\coloneqq\angles{\varcrs{\abs{\mhyphen}_S}\mid S\subset W,\ S\notin\mu}
\]
be the coarse structure generated by $\abs{\mhyphen}_S$ with $S\notin \mu$.
By the ultrafilter property, if $S_1,\ldots, S_n$ are not in $\mu$, then neither is their union $S_1\cup\cdots \cup S_n$.
It then follows that a subset of $\ZZ$ is $\CE_\mu$\=/bounded if and only if it has finite diameter with respect to $\abs{\mhyphen}_S$ for some $S\subset W$ that is not in $\mu$.

If $T \in \mu$ and $S \not \in \mu$ then $T \smallsetminus S$ is infinite, so Lemma~\ref{lem:p2:distinct Cayley metrics} shows that $T$ is not $\abs{\mhyphen}_S$-bounded.
Therefore, $T$ is not $\CE_\mu$\=/bounded either.
Thus, if $\mu, \zeta$ are distinct ultrafilters, there exists $T$ in $\mu \smallsetminus \zeta$ and we see that $T$ is $\CE_\mu$\=/bounded but not $\CE_\zeta$ bounded.
This shows that the coarse structures $\CE_\mu$ are pairwise distinct and hence we are done because there are $2^{2^{\aleph_0}}$ many ultrafilters on $W$.
\end{proof} 

\begin{rmk}
 The neat arguments described above were given by the referee in reply to a question we posed in a previous version of this manuscript.
 Ultrafilters can be avoided by arbitrarily selecting an uncountable family $\Omega\subset \CP(W)$ so that any two $S,T\in\Omega$ have finite intersection.
 Given a subset $\Sigma\subseteq \Omega$, we may define $\CE_\Sigma\coloneqq\angles{\varcrs{\abs{\mhyphen}_S}\mid S\in\Sigma}$ and directly verify that $\CE_\Sigma =\CE_{\Sigma'}$ if and only if $\Sigma=\Sigma'$.
\end{rmk}

\section{Coarse Structures Generated by Topologies}
\label{sec:p2:coarse structures on Z gen by top}

We saw in Section~\ref{sec:p1:determined locally} that an abelian topological group $G$ has a natural topological coarse structure:
\begin{equation*}
 \varcrs[grp]{cpt} = 
 \varcrs[left]{cpt} =\braces{E\subseteq G\times G\mid \exists K\subseteq G\text{ compact s.t. }\forall(g_1,g_2)\in E,\ g_2^{-1}g_1\in K},
\end{equation*}
see Example~\ref{exmp:p1:topological coarse structure abelian}.
The $\varcrs[grp]{cpt}$\=/bounded sets are precisely the relatively compact subsets of $G$ (we are using the convention that a set $A$ is relatively compact if and only if it is a subset of a compact set $K$---which need not be the closure of $A$ if the topology is not Hausdorff).
In particular, we can use topologies on the set of integers to define coarsifications of $\ZZ$.\index{coarsification!topological}

In the following, $\tau$ will denote a group topology on $\ZZ$ (\emph{i.e.}\ a topology such that addition and inversion are continuous) and $\varcrs{\tau}\coloneqq\varcrs[left]{cpt}$\nomenclature[:CE1]{$\varcrs{\tau}$}{topological group coarse structure $\varcrs[grp]{cpt}$ on $\ZZ$ w.r.t. the topology $\tau$} will be the induced equi bi\=/invariant coarse structure. 
The $\varcrs{\tau}$\=/bounded subsets are the $\tau$\=/relatively compact sets. By Proposition~\ref{prop:p1:neighbourhoods.of.identity.generate}, two group topologies give rise to different coarse structures on $\ZZ$ if and only if they have different relatively compact sets. 
Note that if a topology $\tau$ is contained in a finer topology $\tau'$ then every $\tau'$\=/compact set is also $\tau$\=/compact, hence $\CE_{\tau'}\subseteq \CE_\tau$.

\begin{rmk}
 If a topology $\tau$ is induced by a bi\=/invariant metric $d$ then $\CE_\tau\subseteq \CE_d$, but in general the containment may be strict (the equality holds if and only if $d$ is proper). 
\end{rmk}

There are two obvious examples of coarse structures of topological origin. If $\tau$ is a topology so that $\ZZ$ is compact (\emph{e.g.}\ the indiscrete topology) then $\ZZ$ is bounded and $\CE_\tau=\maxcrs$. If $\tau$ is the discrete topology then $\CE_\tau=\varcrs[grp]{fin}=\CE_{\abs{\mhyphen}}$ coincides with the standard coarse structure defined by the Euclidean metric. Now we will explore more exotic topologies.

\begin{exmp}\label{exmp:p2:furstenberg}
 An interesting group topology on $\ZZ$ is the \emph{Furstenberg topology}\index{Furstenberg topology} $\tau_{\rm F}$. This is the topology with infinite arithmetic progressions $\braces{a\ZZ+b\mid a,b\in \ZZ}$ as basis for the open sets \cite{furstenberg1955primes} (one may also note that this topology is the subset topology defined by the inclusion of $\ZZ$ in the profinite completion $\widehat \ZZ$, Remark~\ref{rmk:p2:furstenberg is profinite}).

 This topology is a nice, metrizable topology on $\ZZ$ and---importantly for us---$\ZZ$ is not compact in this topology (this can be seen \emph{e.g.}\ by considering the infinite covering by open sets $\braces{2^{n+1}\ZZ -2^n\mid n\geq 0}$). It follows that $\ZZ$ is not $\CE_{\tau_{\rm F}}$\=/bounded and hence $\CE_{\tau_{\rm F}}\subsetneq \maxcrs$.
 At the same time, one may easily verify that the sequence $n!$ converges to $0$ in $\tau_{\rm F}$. It follows that the infinite set $\braces{n!\mid n\in \NN}\cup \{0\}$ is compact and hence $\CE_{\tau_{\rm F}}$\=/bounded. We then see that we have strict inclusions 
 \[
  \varcrs[grp]{fin} \subsetneq \CE_{\tau_{\rm F}}\subsetneq \maxcrs,
 \]
 and hence $(\ZZ,\CE_{\tau_{\rm F}})$ is a genuinely new coarsification of $\ZZ$. At this point we do not know much else about the coarse structure $\CE_{\tau_{\rm F}}$: to understand it better it would be useful to have some characterization of the $\tau_{\rm F}$\=/compact subsets of $\ZZ$. We refer to \cite{lovas2015some} and references therein for more properties of the Furstenberg topology.
\end{exmp}

Example~\ref{exmp:p2:furstenberg} is but the tip of the iceberg: infinite abelian groups tend to have a huge variety of group topologies (see \emph{e.g.}\ \cite{berhanu1985counting} and references therein). This information on its own is not enough to deduce that $\ZZ$ has many different coarsifications. For instance, many of these topologies can give rise to the usual coarse structures $\varcrs[grp]{fin}$ and $\maxcrs$.
As shown in Example~\ref{exmp:p2:furstenberg}, one convenient way to make sure that a topological space has infinite compact sets is by showing that it admits some infinite converging sequence. Using group invariance, one may also focus on sequences converging to $0$. This point of view helps navigate the literature. For example, in \cite{zelenyuk1991topologies} Zelenyuk and Protasov completely characterize those sequence $(a_n)_{n\in\NN}$ for which there exists a Hausdorff topology $\tau$ so that $a_n\xrightarrow{\tau}0$ (they call them \emph{T\=/sequences}). By construction, all the topologies witnessing that some sequence is a T\=/sequence generate coarse structures that strictly contain $\varcrs[grp]{fin}$. 
A partial list of related references includes \cite{barbieri2003answer,dikranjan2005characterization,hegyvari2005arithmetical,hruvsak2012precompact,nienhuys1972construction,skresanov2020group}

The above discussion leaves us with a wide variety of candidate topologies on $\ZZ$. However, this also leaves us with the task of understanding their compact sets, which seems to be an interesting but difficult question.

% In the next section we will look more in depth into coarse structures induced by profinite on $\ZZ$.
% \section{Profinite Coarse Structures}\label{sec:p2:profinite coarse structures}

We now discuss an interesting class of group topologies on $\ZZ$, namely the profinite ones.
Let $Q$ be a non\=/empty set of primes and consider the family $Q^{*}\coloneqq\braces{m\in\NN\mid (m,p)=1 \text{ $\forall p$ prime }p\notin Q}$ of positive numbers whose factorization consists of powers of primes in $Q$. The family of quotients $\ZZ/m\ZZ$ with $m\in Q^{*}$ forms an inverse system of finite groups.
We define the \emph{pro-$Q$ completion}\index{pro-$Q$ completion of $\ZZ$} of $\ZZ$ as the inverse limit
\[
 \ZZ_Q\coloneqq \varprojlim_{m\in Q^{*}} \ZZ/m\ZZ.
\]
The limit topology is a metrizable, compact group topology on $\ZZ_Q$ (one may realize $\ZZ_Q$ as a closed subgroup of the infinite product $\prod \ZZ/m\ZZ$ and the limit topology coincides with the subspace topology).

Two examples of special interest are the \emph{profinite completion} $\widehat \ZZ\coloneqq\ZZ_{\{\text{all primes}\}}$ and the \emph{pro\=/$p$ completion} $\ZZ_p\coloneqq \ZZ_{\{p\}}$ (this is also known as the group of $p$\=/adic integers).
Using the Chinese Reminder Theorem, one can show that for any $Q$, the pro\=/$Q$ completion is isomorphic (as a topological group) to a product of pro\=/$p$ completions:
\[
 \ZZ_Q\cong\prod_{p\in Q}\ZZ_p.
\]

For any choice of $Q$, the natural homomorphism $\ZZ\to\ZZ_Q$ is an embedding with dense image. When identifying $\ZZ_Q$ with the product of pro\=/$p$ completions, the embedding $\ZZ\hookrightarrow\prod_{p\in Q}\ZZ_p$ is the diagonal embedding. Since $\ZZ$ is a dense proper subset of $\ZZ_Q$ and the latter is Hausdorff, $\ZZ$ is not compact in $\ZZ_Q$.

\begin{de}\label{def:p2:proQ coarse structure}
 Let $\tau_Q$ denote the subspace topology on $\ZZ\subset \ZZ_Q$. The \emph{pro\=/Q coarse structure}\index{coarse structure!pro-$Q$} on $\ZZ$ is the equi left invariant coarse structure $\CE_Q\coloneqq \CE_{\tau_Q}$\nomenclature[:CE1]{$\CE_Q$}{pro\=/$Q$ coarse structure} induced by the pro\=/$Q$ topology $\tau_Q$.
\end{de}

Concretely, a subset $A\subseteq\ZZ$ is $\CE_Q$\=/bounded if and only if its closure $\overline{A}^{\tau_Q}\subseteq \ZZ$ is compact. Note that this definition is not trivial, because the closure $\overline{A}^{\tau_Q}=\overline{A}^{\ZZ_Q}\cap \ZZ$ is generally not closed in $\ZZ_Q$ and hence not compact. In particular, $\CE_Q$ is never equal to the bounded coarse structure $\maxcrs$ because $\ZZ$ is not compact in $\ZZ_Q$.

\begin{rmk}\label{rmk:p2:furstenberg is profinite}
 It is not hard to see that the sets of the form $m\ZZ+k$ with $m\in Q^{*}$ and $k\in \NN$ are a basis of open sets for $\tau_Q$. In particular, this shows that the Furstenberg topology $\tau_{\rm F}$ is equal to the topology $\tau_{\{\text{all primes}\}}$ induced by the profinite completion of $\ZZ$.
\end{rmk}

\begin{prop}\label{prop:p2:different profinite coarse structures}
 If $Q$ and $Q'$ are two sets of primes, then $\CE_Q\subseteq \CE_{Q'}$ if and only if $Q\supseteq Q'$.
\end{prop}
\begin{proof}
 If $Q'$ is contained in $Q$ then the topology $\tau_Q$ is finer than $\tau_{Q'}$, hence $\CE_Q\subseteq \CE_{Q'}$. 
 It remains to show that if $Q'$ is \emph{not} contained in $Q$ then $\CE_Q\nsubseteq \CE_{Q'}$. Let $q$ be a prime in $Q'\smallsetminus Q$. We need to find a $\tau_Q$\=/compact set $K\subset \ZZ$ that is not contained in any $\tau_{Q'}$\=/compact set. Since $\tau_q\subseteq \tau_Q'$, it is enough to check that $K$ is not contained in any $\tau_q$\=/compact set.
 
 We will do so by constructing a sequence $(k_n)_{n\in\NN}$ such that $k_n\to 0$ with respect to $\tau_Q$ but it converges to some point $\xi\in \ZZ_q\smallsetminus\ZZ$ with respect to the pro\=/$q$ topology. 
 Assuming that such a sequence exists, we see that $K=\bigcup\braces{k_n\mid n\in\NN}\cup \{0\}$ is $\tau_Q$\=/compact. Since $\ZZ_q$ is Hausdorff, $\xi$ is the unique accumulation point of $K$ in $\ZZ_q$. It follows that $K$ has no $\tau_q$\=/accumulation point in $\ZZ$ and hence it is not $\tau_q$\=/compact. On the other hand, $K\cup\{\xi\}$ is compact in $\ZZ_q$, hence $K$ is $\tau_q$\=/closed. It follows that $K$ is not contained in any $\tau_q$\=/compact set, otherwise $K$ would be compact.

 It remains to find an appropriate sequence. Note that for every $l\in \NN$ any product of powers of primes in $Q$ is coprime with $q^l$ and hence belongs to the group of units $(\ZZ/q^l\ZZ)^{*}$. For any $x\in (\ZZ/q^l\ZZ)^{*}$, let ${\rm ord}_{q^l}(x)$ be the order of $x$ in $(\ZZ/q^l\ZZ)^{*}$ (\emph{i.e.}\ the smallest exponent so that $x^m\equiv 1 \text{ mod}(q^l)$).

 We shall now construct the sequence iteratively. We will deal with the case where $Q$ is infinite, the finite case is simpler and can be dealt with similarly. Choose an ordering $p_1,p_2,\ldots $ for the elements of $Q$ and let $a_1\coloneqq {\rm ord}_q(p_1)$. We define $k_1\coloneqq p_1^{a_1}$.
 
 Choose some $l_1$ large enough so that $k_1\leq q^{l_1}$ and let $a_2\coloneqq{\rm ord}_{q^{l_1}}(p_1p_2)$. Define $k_2\coloneqq k_1(p_1p_2)^{a_2}$. With this choice, we see that
 \[
  k_2\equiv k_1 \equiv 1 \text{ mod}(q) \qquad \text{and} \qquad
  k_2\equiv k_1 \not\equiv 1 \text{ mod}(q^{l_1}).
 \]
 Moreover, the power of $p_1$ appearing in the factorization of $k_2$ is strictly larger than that of $k_1$.
 
 Inductively, assume that $k_{n-1}$ has already been defined. Choose $l_n$ so that $k_{n-1}\leq q^{l_n}$, let $a_n\coloneqq{\rm ord}_{q^{l_n}}(p_1p_2\cdots p_n)$ and define $k_n\coloneqq k_{n-1}(p_1p_2\cdots p_n)^{a_n}$. We claim that the sequence $k_n$ does what we need.
 
 For each $i$ and each $N$, we see that $p_i^N$ divides $k_n$ for every $n$ large enough. By definition of the pro\=/$Q$ topology, this shows that $k_n\to 0$ in $\ZZ_Q$. On the other hand, for each $j\in \NN$ we see that $k_n\equiv k_j \text{ mod}(q^{l_j})$ for every $n\geq j$. This shows that the sequence 
 \[
  ([k_1],[k_2],[k_3],\ldots)\in \prod_{j\in\NN} \ZZ/q^{l_j}\ZZ
 \]
 defines an element $\xi\in \ZZ_q=\varprojlim_{j\in\NN}\ZZ/q^{l_j}\ZZ$ and that $k_n\to \xi$ in $\ZZ_q$. Finally, this $\xi$ does not belong to $\ZZ$ because it is not eventually constant: for every $j\in\NN$ and $n<j$ we have $k_{n}\not\equiv k_{n-1} \text{ mod}(q^{l_j})$.
\end{proof}

\begin{cor}
 The pro\=/$Q$ coarse structures provide a continuum of distinct coarsifications of $\ZZ$.\index{coarsification!pro\=/$Q$}
\end{cor}

\begin{rmk}
 Since $\ZZ_Q\cong\prod_{p\in Q}\ZZ_p$ and the projections $\prod_{p\in Q}\ZZ_p\to \ZZ_p$ are continuous, we see that any $\tau_Q$\=/compact set is $\tau_p$\=/compact for every $p\in Q$. This implies that $\CE_{Q}\subseteq \bigcap_{p\in Q}\CE_p$. It is curious to observe that this containment is strict: we are grateful to Samuel Evington for pointing out the following example. Let $Q=\{2,3\}$, then there exist sequences of integers $a_n,\ b_n$ such that $a_n2^n+b_n3^n=1$ for every $n\in\NN$. It follows that the sequence $a_n 2^n$ converges to $0$ in $\ZZ_2$ and to $-1$ in $\ZZ_3$. In particular, the set $\{a_n2^n\mid n\in\NN\}$ is relatively compact with respect to both $\tau_2$ and $\tau_3$ and it is hence bounded with respect to $\CE_2\cap\CE_3$.
 
 Considering the diagonal embedding $\ZZ\hookrightarrow \ZZ_2\times \ZZ_3\cong \ZZ_{\{2,3\}}$, we see that the sequence $a_n2^n$ converges to the point $\xi = (0,-1)$. Since $\xi$ is off\=/diagonal, it belongs to $\ZZ_{\{2,3\}}\smallsetminus \ZZ$. As in the proof of Proposition~\ref{prop:p2:different profinite coarse structures}, we deduce that $\{a_n2^n\mid n\in\NN\}$ is not relatively compact with respect to the topology $\tau_{\{2,3\}}$ and hence it is not $\CE_{\{2,3\}}$\=/bounded.
\end{rmk}

\section{Some Questions}%%%%%%%%%%%%%%%%%%%%

In Section~\ref{sec:p2:coarse structures on Z gen Cayley} we used some specific choices of word metrics $\abs{\mhyphen}_S$ to show that the cardinality of distinct coarse structures on $\ZZ$ is as large as possible.
On the other hand, it is in general a very difficult problem to understand the coarse structure $\varcrs{Cay(S)}$ induced by some arbitrary set $S$.
So much so, that is is often complicated even to understand whether $\varcrs{Cay(S)}=\maxcrs$.
Explicitly, $\varcrs{Cay(S)}$ is equal to the maximal coarse structure if and only if there exists some fixed $N$ so that every integer $k\in \ZZ$ can be written as a sum or difference of at most $N$ elements of $S$. In particular, if $S=P$ is the set of all primes the equality $\varcrs{Cay(P)}=\maxcrs$ becomes a very weak version of the Goldbach conjecture. It is a classical fact that every integer is equal to the sum of a bounded number of primes \cite{vinogradov1937representation}. In other words, it is already known that ${\rm Cay}(\ZZ,P)$ does indeed have finite diameter (its value is also known as the Shnirelman constant). Helfgott's proof of the ternary Goldbach conjecture \cite{helfgott2015ternary} implies that this diameter is at most $4$ (the strong Goldbach conjecture would imply that it is $3$).

On the positive side, it is fairly simple to see that if $S=\braces{a^n\mid n\in\NN}$ is a geometric progression for some $a>1$ then $\ZZ$ is not $\varcrs{Cay(S)}$\=/bounded. It is also true that $\ZZ$ is unbounded if $S$ is the set of all numbers whose factorization only contains primes in some fixed finite set \cite[Theorem~3]{nathanson2011geometric}. In particular, it follows that also that sets like $S=\braces{2^n\mid n\in \NN}\cup \braces{3^n\mid n\in\NN}$ give rise to a Cayley graph of infinite diameter. For a given set $S$, deciding whether $\ZZ$ is $\varcrs{Cay(S)}$\=/bounded seems to be a hard number theoretic problem.

It is also an interesting problem to understand when two different generating sets $S_1$, $S_2$ give rise to different coarse structures. This problem is solved in \cite{nathanson2011bi} in the case where both $S_1$ and $S_2$ are geometric progressions: for $S_1 =\braces{a^n\mid n\in\NN}$ and $S_2=\braces{b^n\mid n\in\NN}$, $\varcrs{Cay(S_1)}=\varcrs{Cay(S_2)}$ if and only if $a^n=b^m$ for some $n,m\in\NN$. Little else is known otherwise (see also \cite{bell2017locally,nathanson2011problems} for related work).

Also note that if we are given two coarsifications $(\ZZ,\CE_1)$ and $(\ZZ,\CE_2)$ with $\CE_1\neq \CE_2$ it may still be the case that the coarse group $(\ZZ,\CE_1)$ and $(\ZZ,\CE_2)$ are isomorphic via some function that is not the identity.
This raises an obvious question:

\begin{qu}\label{qu:p2:Schwartz}
 Let $S_1 =\braces{2^n\mid n\in\NN}$ and $S_2=\braces{3^n\mid n\in\NN}$, are $(\ZZ,\varcrs{Cay(S_1)})$ and $(\ZZ,\varcrs{Cay(S_2)})$ isomorphic?
\end{qu}

A version of this question has already appeared in the literature. Namely, it is a question attributed to Richard E. Schwarz whether the Cayley graphs associated with the sets $S_1 =\braces{2^n\mid n\in\NN}$ and $S_2=\braces{3^n\mid n\in\NN}$ are quasi\=/isometric. Note that in this case the Cayley graphs are quasi\=/isometric if and only if they are coarsely equivalent (this is because they are geodesic metric spaces, see Appendix~\ref{sec:appendix:quasifications}).
Our question differs from Schwartz's in that we ask for a coarse equivalence that is also a coarse homomorphism. Both questions are open.

%We end this chapter with a few open questions. We showed that there exists at least a continuum $2^{\aleph_0}$ of different coarsifications, but the set of coarse structures on $\ZZ$ could in principle be as large as $2^{2^{\aleph_0}}$.
% \editA{We end this chapter with a few open questions. We showed that there exists at least $2^{2^{\aleph_0}}$ different coarsifications of $\ZZ$.
%
% \begin{qu}
%  What is the cardinality of the set of distinct coarsifications of $\ZZ$?\fnote{now we know it}
% \end{qu}
% }

We also ask the analog of Question~\ref{qu:p2:Schwartz} for other meaningful coarse structures. For instance:

\begin{qu}
 Given two different sets of primes $Q\neq Q'$, can $(\ZZ,\CE_{Q})$ and $(\ZZ,\CE_{Q'})$ be isomorphic coarse groups?
 How about $Q=\{3\}$ and $Q'=\{5\}$?
\end{qu}

It seems to us that replying the the questions above requires developing some invariants of coarse groups, which we find would be very interesting in its own right.

\

The relation between Cayley and topological coarsifications eludes us.
More precisely, we know that the Cayley coarsifications $\varcrs{Cay(S)}$ are exactly the coarsely geodesic ones (Proposition~\ref{prop:p2:coarse geodesic coarsifications}), but in concrete examples we do not know how to verify whether a given coarse structure is geodesic or not.
We pose it as a problem for topological coarsifications:

\begin{problem}
 Develop criteria to establish whether a given topological coarsification $\varcrs{\tau}$ of $\ZZ$ is coarsely geodesic.
\end{problem}

\begin{rmk}
 We do not know how to solve the above for profinite coarsifications.
\end{rmk}

\

Observing that $(\ZZ,\CE_{\abs{\mhyphen}})$ is isomorphic to $(\RR,\CE_{\abs{\mhyphen}})$, it is not hard to show that the only coarse subgroups of $(\ZZ,\CE_{\abs{\mhyphen}})$ are the trivial ones: $\{0\}$ and $\ZZ$ (one can use Proposition~\ref{prop:p1:homogeneization.of.qmorph} to show that the image of a coarse homomorphism must be bounded or coarsely dense). We do not know whether this feature is specific to the canonical coarse structure $\CE_{\abs{\mhyphen}}$ or if it is a general fact about $\ZZ$. Namely, we ask the following:

\begin{problem}
 Does there exist a connected coarsification of $\ZZ$ such that $(\ZZ,\CE)$ has unbounded, non coarsely dense coarse subgroups $\{0\}\crse{< H<} \ZZ$?
\end{problem}

\chapter{On Bi-invariant Word Metrics}\label{ch:p2:canc metrics}
Many special examples of bi\=/invariant word metrics have been studied, generally under different names. Below we collect some of the instances that we are aware of, but the list is far from exhaustive.

\begin{enumerate}
 \item If $G$ is a perfect group (\emph{i.e.}\ $G=[G,G]$) then the \emph{commutator length}\index{length!commutator} is the word length associated with the set of commutators---which is invariant under conjugation. By now, the commutator length is a rather classical object of study (especially in its ``stabilized'' form) which has beautiful connections with bounded cohomology, $\RR$\=/quasimorphisms and low\=/dimensional topology. We refer to \cite{calegari2009scl} for an introduction to the subject.
 
 If  $G$ is not perfect, the commutator length still gives a bi\=/invariant word metric on the commutator subgroup $G'\coloneqq[G,G]$. It is interesting to note that if $G'$ is unbounded with respect to this metric then it is also possible to define an unbounded bi\=/invariant metric on $G$ \cite[Proposition 1.4]{BIP}.
 
 \item If $G$ is generated by torsion elements, one can consider the bi\=/invariant word metric associated with the set of all the torsion elements. This is sometimes called \emph{torsion length}\index{length!torsion} (\cite{kotschick2004quasi}).
 
 \item Let $G=\pi_1(\Sigma)$ be the fundamental group of a closed surface of genus at least $2$. The set $S\subset G$ of loops homotopic to simple closed curves is invariant under conjugation and generates $G$. Answering a question of Farb, Calegari showed in \cite{Cal} that the bi\=/invariant word metric associated with $S$ is unbounded on $G$ (see also \cite{BH}).
 
 \item The \emph{reflection length}\index{length!reflection} on a Coxeter group $W$ is defined as the bi\=/invariant word length associated with the standard generating set \cite{mccammond2011bounding}. It is known that the reflection length on $W$ is bounded if and only if $W$ is spherical or of affine type \cite{duszenko2012reflection}. See also \cite{drake2021upper,lewis2019computing} and references therein for some recent works on reflection lengths.
 
 \item Let $w$ be a word in some letters $x_1,\ldots ,x_n$. We can see $w$ as a functions with variables $x_1,\ldots ,x_n$ and, for any given a group $G$, we can then consider the set $S$ of elements in $G$ that can be obtained by replacing each $x_i$ with some $g_i\in G$. This set is conjugation invariant, and the word length associated with it is also called \emph{$w$\=/length} or \emph{verbal length} of $G$ (technically, the $w$\=/length is only a length function if $S$ generates $G$).
 For example, when $w=x_1x_2x_1^{-1}x_2^{-1}$ the $w$\=/length coincides with the commutator length. An initial study of some (stable) $w$\=/lengths is carried out in \cite{CZ}. See also \cite{bestvina2019verbal} and references therein for some recent results.
 
 \item Let $F_S$ be the free group generated by $S$. Jiang defined the \emph{width}\index{width of a word} of a word $w\in F_S$ to be its bi\=/invariant word length relative to the set $S'\coloneqq\braces{s^k\mid s\in S,\ k\in \ZZ}$. This quantity is relevant to count the number of fixed points of a self\=/homeomorphism of a surface \cite{jiang1989surface} (see also \cite{grigorchuk1991width}).
 
 \item Bi\=/invariant word metrics on a free group $F_S$ can be seen as \emph{combinatorial areas}\index{combinatorial area}. Namely, a subset $R\subseteq F_S$ normally generates $F_S$ if and only if $\angles{S\mid R}$ is a presentation of the trivial group. The length of an element $w\in F_S$ with respect to the bi\=/invariant word length $\abs{\variable}_{\overline{R}}$ is then equal to the minimal number of conjugates of relations in $R$ that is necessary to multiply in order to show the word $w$ represents the unit element in $\angles{S\mid R}$. In other words, $\abs{w}_{\overline{R}}$ equals the minimal area of a van Kampen diagram of $w$. This is the point of view taken by Riley in \cite{riley2017computing}.
 
 \item Let $M$ be a smooth manifold and $G\coloneqq {\rm Diff_0}(M)$ be the group of diffeomorphisms isotopic to the identity. The Fragmentation Lemma states that $G$ is generated by the subset of diffeomorphisms that are supported on embedded open balls \cite{banyaga2013structure}. The associated bi\=/invariant word length is known as the \emph{fragmentation norm}\index{norm!fragmentation} on $G$. It is an interesting problem to understand which manifolds $M$ give rise to diffeomorphism groups with unbounded fragmentation norm. It is known that ${\rm Diff_0}(M)$ is bounded if $M$ is a closed manifold of dimension $\neq 2,4$, or if $M=\SS^2$ is the two dimensional sphere \cite{BIP,tsuboi2008uniform}. In contrast, ${\rm Diff_0}(\Sigma)$ is unbounded if $\Sigma$ is an orientable surface of positive genus \cite{bowden2022quasi}. See also \cite{brandenbursky2018fragmentation,polterovich2021norm}.
 
 \item Let $G\coloneqq {\rm Diff_0}(\Sigma,{\rm Area})$ be the group of area preserving diffeomorphisms of a closed surface that are isotopic to the identity. This $G$ is generated by a special class of ``autonomous'' diffeomorphisms. The \emph{autonomous norm}\index{norm!autonomous} is the bi\=/invariant word length associated with the set of autonomous diffeomorphisms. It can be shown that the autonomous norm is unbounded on $G$ \cite{BK,Bra,gambaudo2004commutators}. 
\end{enumerate}

\section{Computing the Cancellation Metric on Free Groups} 
\label{sec:p2:cancellation metric on free}
 In this section we will discuss in greater detail the \emph{natural bi\=/invariant word} metric on a free group. This is largely done in order to introduce some notation that will be used in Section~\ref{sec:p2:conjecture on coarsified groups}. However, it is worth remarking that this metric is of special interest in surprisingly diverse topics ranging from biology to engineering: it appears in the study of RNA\=/folding \cite{nussinov1980fast} and the design of liquid crystals \cite{majumdar2010tangent}. See also \cite{riley2017computing} for a short survey of these application.

 Let $S$ be a finite set and $F_S$ be the the group freely generated by it. As explained in Example~\ref{exmp:p1:cancellation metric}, the bi\=/invariant word length $\abs{\variable}_{\overline{S}}$ can be easily described in terms of cancellation of letters. More precisely, let $w$ a word with letters in $S\cup S^{-1}$. The cancellation length $\abs{w}_{\times}$ is defined as the number of letters that is necessary to cancel from $w$ in order to obtain a word whose reduced representative is the empty word. 
 One can verify that the word length $\abs{g}_{\overline{S}}$ of an element $g\in F_S$ is equal to the cancellation length any word $w$ representing it---in particular the cancellation length does not depend on the choice of $w$.\footnote{%
 This is not generally the case for arbitrary groups and generating sets. For one, this property depends on the choice of generating sets: if we take $S'=\braces{a,b,a^2}$ as a generating set for $F_2$ then the two words $aa$ and $a^2$ have different cancellation lengths (we are seeing the latter as a $1$\=/letter word in $S'$). For a more interesting example, notice that for the Baumslag--Solitar groups ${\rm BS}(2,5)=\angles{a,b\mid ba^2b^{-1}=a^5}$ the words $ba^2b^{-1}$ and $a^5$ represent the same element but have different cancellation length \cite[Example 2.G]{BGKM}.
Besides free\=/groups, other set\=/groups that have well\=/behaved cancellation lengths are right angled Artin and Coxeter groups, and, more generally, all groups with a ``balanced'' presentation \cite[Theorem 1.C]{BGKM}. }
 
 For this reason, in this section we will call the metric $d_{\overline{S}}$ the \emph{cancellation metric} of $F_S$. Also recall that $d_{\overline{S}}$ induces the canonical coarsification: $\CE_{d_{\overline{S}}}=\varcrs{bw}$.

 We introduce some notation:
 
 \begin{notation*} 
 We will work with (possibly non reduced) words with letters in the alphabet $S\sqcup S^{-1}$. We will denote $\rd(w)$\nomenclature[:z]{$\rd(w)$}{reduced word} the reduced word obtained by reducing the word $w$. We will \emph{not} use  notation $[w]$ to differentiate between words and elements of $F_S$. However, we will write $v\equiv w$ to clarify that $v$ and $w$ are equal \emph{as words}, while $v=w$ means that they are equal as elements of $F_S$ (\emph{i.e.}\ $v=w$ if and only if $\rd(v)\equiv\rd(w)$). We may sometimes stress the difference between $\equiv$ and $=$ by saying that the latter holds in $F_S$.
 
 The concatenation of two words is denoted $vw$. By $v^{-1}$ we mean the word obtained by reversing the order of the letters $v$ and inverting each letter. With our conventions, $vv^{-1}=\emptyset$ but $vv^{-1}\not\equiv\emptyset$. A word $w$  \emph{subword}\index{subword} of $v$ if $v\equiv awb$. It is a \emph{proper subword}\index{subword!proper} if at least one of $a$ and $b$ is non\=/empty ($w$ itself may be empty). If $a\equiv\emptyset$, $w$ is a \emph{starting subword}.\index{subword!starting} It is an \emph{ending subword}\index{subword!ending} if $b\equiv\emptyset$.
 \end{notation*}
 
By definition, the distance between two reduced words $w_1,w_2$ is $d_{\overline{S}}(w_1,w_2)=\abs{w_2^{-1}w_1}_{\overline{S}}$. However, it is convenient to have a more explicit characterization. Given a letter $x$ in a word $w$, we denote by $w\smallsetminus x$ the word obtained by omitting $x$. We say that a \emph{cancellation move}\index{move!cancellation} $M$ between reduced words is a transformation sending a reduced word $v\equiv w_1uxu^{-1}w_2$ to $w_1w_2$, where $x\in S\sqcup S^{-1}$. That is, $M$ takes $v$ to $\rd(v\smallsetminus x)$. An \emph{addition move}\index{move!addition} is the inverse of a cancellation move, \emph{i.e.}\ it adds a conjugate of a letter in such a way that the end result is a reduced word. 
In either case, we say that the subwords $w_1$ and $w_2$ are \emph{preserved} by the move.

\begin{lem}\label{lem:p2:cancellation distance as moves}
 Given reduced words $w,w'$, let 
\(
 \tilde d(w,w')\coloneqq\min\braces{n\mid \text{can take $w$ to $w'$ with $n$ moves}}.
\)
Then $\tilde d=d_{\overline{S}}$.
\end{lem}
\begin{proof}
 Both $\tilde d$ and $d_{\overline{S}}$ can be seen as graph metrics on the set of reduced words, so it is enough to verify that $\tilde d(w,w')=1$ if and only if $d_{\overline{S}}(w,w')=1$. Note that $d_{\overline{S}}(w,w')=1$ if and only if $w'= wuxu^{-1}$ for some $x\in S\cup S^{-1}$.

 Up to reordering, if $\tilde d(w,w')=1$ we can assume that $w'$ is reached with a single addition move: $w\equiv w_1w_2$ and $w'\equiv w_1uxu^{-1}w_2$. Then we see that $d_{\overline{S}}(w,w')=1$ because 
 \[
  w'=(ww_2^{-1})uxu^{-1}w_2=w(w_2^{-1}u)x(w_2^{-1}u)^{-1}.
 \]
 
 Conversely, if $d_{\overline{S}}(w,w')=1$ then $w'\equiv \rd(wvxv^{-1})$. We can assume that the word $vxv^{-1}$ is reduced (otherwise it can be written as a shorter conjugate of $x$). If $wvxv^{-1}$ is reduced there is nothing to prove. If it is not reduced, let $h$ be the largest starting subword of $vxv^{-1}$ that is cancelled when reducing $wvxv^{-1}$. We thus have $w\equiv w_1h^{-1}$.
 
 There are now two cases. If $h$ is a subword of $v$ then $v\equiv hv_1$ and $w'\equiv w_1v_1xv_1^{-1}h^{-1}$ is reduced. Thus $w'$ is obtained by adding $v_1xv_1^{-1}$ into $w_1h^{-1}=w$. If $h$ is not a subword of $v$, then we decompose $v\equiv v_1v_2$ in such a way that $h\equiv v_1v_2xv_2^{-1}$. We then see that $w'\equiv w_1v_1^{-1}$ is obtained from $w=w_1 v_2x^{-1}v_2^{-1}v_1^{-1}$ by cancelling the letter $x^{-1}$ and reducing the word.
\end{proof}

\begin{rmk}
 Since we are not differentiating between words and group elements, the notation $d_{\overline{S}}(v,w)$ is also defined for non\=/reduced words and defines a pseudo\=/distance. We could have also extended the definition of $\tilde d$ by declaring that $\tilde d(v,w)=0$ if $\rd(v)\equiv\rd(w)$. Of course, this $\tilde d$ is still equal to $d_{\overline{S}}$.
\end{rmk}

\begin{rmk}\label{rmk:p2:cancellation or addition}
 If $w\equiv w_1w_2$ and $uxu^{-1}$ are reduced and $w'$ is the---possibly non reduced---word $w_1uxu^{-1}w_2$, then there can be some cancellation between $w_1$ and $u$ or between $u^{-1}$ and $w_2$, but not both. 
 Assume that the cancellation happens between $w_1$ and $u$, and let $h$ be the initial substring of $uxu^{-1}w_2$ that is cancelled when doing the reduction.
 If $h$ is contained in $u$, we then see as in the proof of Lemma~\ref{lem:p2:cancellation distance as moves} that $w$ can be taken to $\rd(w')$ with an addition move. Otherwise, $\rd(w')$ is obtained from $w$ via a cancellation move.
\end{rmk}

\begin{lem}\label{lem:p2:reordering moves}
 The distance $d_{\overline{S}} (w,w')$ can always be realized by a sequence of moves $M_1,\ldots,M_n$ such that all the cancellations are performed first.
\end{lem}

Proving Lemma~\ref{lem:p2:reordering moves} directly is surprisingly tedious. We found it simpler to prove it by working with the pseudo\=/metric defined on the set of (possibly non reduced) words. Since this argument requires introducing extra notation and is not relevant to the rest of the monograph, it appears in Appendix~\ref{sec:appendix:canc metric step by step}.
As corollary of Lemma~\ref{lem:p2:reordering moves}, we obtain a formal proof of the fact that the coarse group $(F_S,\CE_{d_{\overline{S}}})$ is not coarsely abelian (Definition~\ref{def:p1:coarsely abelian}):

\begin{cor}\label{cor:p2:lower.bound.commutator}
 For every $n<m\in \NN$ and $x_1\neq x_2\in S$, the cancellation norm of the commutator $[x_1^n,x_2^m]$ equals $2n$. In particular, the coarse group $(F_S,\varcrs{bw})$ is not coarsely abelian for $\abs{S}\geq 2$.
\end{cor}

\begin{rmk}
 The reader should be warned that the cancellation distance can be quite deceptive. For example, it is generally not true that $\abs{w^2}_{\overline{S}}= 2\abs{w}_{\overline{S}}$. In fact, let $S=\{a,b\}$ and $A=a^{-1}$, $B=b^{-1}$. One can see that the word $w=abAAB$ has cancellation norm $3$ while $\abs{abAABabAAB}_{\overline{S}}=4$.
 However, there exists an easy\=/to\=/implement polynomial time algorithm that computes $\abs{w}_{\overline{S}}$ (see below). This can be very helpful to test conjectures involving the cancellation metric.
\end{rmk}

\begin{rmk}
 A cubic time algorithm for computing the cancellation length of words was described in \cite{nussinov1980fast} and was rediscovered in \cite{BGKM}. Faster algorithm are also known: an algorithm with less than cubic running time is described in \cite{bringmann2019truly}. See also the discussion in \cite{riley2017computing}. 
\end{rmk}

\section{The Canonical Coarsification of Finitely Generated Set\=/Groups and their Subgroups}
The starting point of Geometric Group Theory is the observation that every finitely generated set\=/group can be given a left\=/invariant word metric and that this choice does not depend on the generating set up to coarse equivalence (or quasi\=/isometry).
We already remarked that the same is true for bi\=/invariant word metrics: choosing a different finite (normally) generating set will give rise to a coarsely equivalent (in fact, bi\=/Lipschitz equivalent) bi\=/invariant word metric. Recall the induced coarse structure is the canonical coarse structure $\varcrs{bw}=\CE_{d_{\overline{S}}}$.

However, these metrics behave quite differently from the left\=/invariant word metrics.
For instance, it is a basic fact that a finitely generated set\=/group equipped with its word metric is quasi\=/isometric to any of its finite index subgroups. This fact fails dramatically for the bi\=/invariant word metric: to see this it is enough to recall that the infinite dihedral group $D_\infty$ is intrinsically bounded even though it has $\ZZ$ as an index\=/2 subgroup and $(\ZZ,\varcrs{bw})=(\ZZ,\CE_{\abs{\mhyphen}})$ is unbounded.

\begin{rmk}
 On the other hand, it is still true that if $N\lhd G$ is a finite normal subgroup and $Q=G/N$ is the quotient then $(G,\varcrs{bw})$ and $(Q,\varcrs{bw})$ are isomorphic coarse groups. 
 In fact, if $S$ is a finite generating set of $Q$ and $q\colon G\to Q$ is the quotient map, it is immediate to verify that the bi\=/invariant word metric on $G$ associated with the finite set $q^{-1}(S)$ generates the pull\=/back coarse structure $q^{*}(\CE_{d_{\overline{S}}})$.
\end{rmk}

Torsion free examples of poorly behaved finite index subgroups are also easy to find:

\begin{exmp}\label{exmp:p2:finite index subgroup free group}
 Let $F_2\coloneqq\angles{a,b}$ be the $2$\=/generated free group and let $H\lhd F_2$ be the index\=/2 subgroup consisting of all the elements represented  by words of even length. 
We claim that the inclusion $\iota\colon(H,\varcrs{bw}^{H}) \hookrightarrow (F_2,\varcrs{bw}^{F_2})$ is not an isomorphism of coarse groups (where $\varcrs{bw}^{H}$ and $\varcrs{bw}^{F_2}$ denote the canonical coarse structures of $H$ and $F_2$ respectively).
 
 The group $H$ is a rank\=/3 free group generated by the elements $x= a^2$, $y=ab$ and $z=ba$. Let $S=\{a,b\}$ and $T=\{x,y,z\}$, so that $d_{\overline{S}}$ and $d_{\overline{T}}$ induce the canonical coarse structures of $F_2$ and $H$ respectively. 
 It is immediate to verify that $\iota$ is a controlled map. Since $H$ is a finite index subgroup and the coarse structure $\varcrs{bw}^{F_2}$ is connected, it is also clear that $\iota$ is coarsely surjective. However, $\iota$ is not a proper map. In other words, we claim that there is a strict containment $\CE_{d_{\overline{T}}}\subsetneq (\CE_{d_{\overline{S}}})|_{H}$. To see this, note that the set 
 \[
  B\coloneqq\braces{a (ab)^n a (ab)^{-n}\mid n\in\NN}\subset H
 \]
 is a $d_{\overline{S}}$\=/bounded neighborhood of the identity. On the other hand, writing the set $B$ in terms of the generators $x,y,z$ we see that 
 \(
  B=\braces{x z^n y^{-n}\mid n\in\NN}
 \)
 and, using Lemma~\ref{lem:p2:reordering moves}, it is easy to show that $B$ is not $d_{{\overline{T}}}$\=/bounded.
 
 This example can also be used to show the somewhat surprising fact that a self\=/commensuration of a coarsified set\=/group need not give rise to a coarse group automorphism. 
 More precisely, it is simple to observe that whenever $\phi\colon H\to H$ is an isomorphism of set\=/groups then $\crse \phi \colon(H,\varcrs{bw}^H)\to(H,\varcrs{bw}^H)$ is an isomorphism of coarse groups. Since $H$ is coarsely dense in $F_2$, one may expect that it should be possible to extend $\crse\phi$ to a coarse homomorphism $\hat{\crse{\phi}}\colon (F_2,\varcrs{bw}^{F_2})\to (F_2,\varcrs{bw}^{F_2})$. However, the fact that $\varcrs{bw}^{H}\neq (\varcrs{bw}^{F_2})|_{H}$ warns us that this is a naive hope. 
 In fact, consider the set\=/group isomorphism $\phi\colon H\to H$ fixing $x,z$ and sending $y$ to $y^{-1}$. The $d_{\overline{S}}$\=/bounded set $B$ defined before is then sent to $\phi(B)=\braces{x z^n y^{n}\mid n\in\NN}=\braces{a (ab)^n a (ab)^{n}\mid n\in\NN}$, which is clearly unbounded in $(F_2,\varcrs{bw}^{F_2})$. This shows that $\iota\circ\phi\colon (H,(\varcrs{bw}^{F_2})|_{H})\to(F_2,\varcrs{bw}^{F_2})$ is not controlled and hence cannot be extended to an isomorphism of the coarse group $(F_2,\varcrs{bw}^{F_2})$.
\end{exmp}

When dealing with finitely (normally) generated groups, it would be interesting to understand under what circumstances is the canonical coarsification of a subgroup $H<G$ is compatible with the restriction of the canonical coarsification of $G$. As a preliminary step, it would be interesting to investigate this for finitely generated normal subgroups:

\begin{problem}\label{qu:p2:subgroup with compatible biinvariant word metric}
 Let $G$ be a finitely generated set\=/group, $N\lhd G$ a finitely generated normal subgroup and $\varcrs{bw}^G$, $\varcrs{bw}^N$ the associated canonical coarse structures. Under what conditions is the containment $\varcrs{bw}^N \subseteq(\varcrs{bw}^G)|_N$ an equality?
\end{problem}

In Example~\ref{exmp:p2:finite index subgroup free group} we showed that the inclusion $F_3\hookrightarrow F_2$ as the index\=/2 subgroup of words of even length is not an isomorphism of coarse groups. However, this does not rule out the possibility that $(F_2,\varcrs{bw}^{F_2})$ and $(F_3,\varcrs{bw}^{F_3})$ be isomorphic coarse groups via some other map. More generally, we do not know the answer to the following:

\begin{qu}
 Let $F_n$ and $F_m$ be free groups with $n, m\geq 2$. Is it true that if $(F_n,\varcrs{bw}^{F_n})$ and $(F_m,\varcrs{bw}^{F_m})$ are isomorphic coarse groups then $n= m$?
\end{qu}

\begin{rmk}
 We now know that commensurable set\=/groups need not be isomorphic coarse groups when equipped with their canonical coarse structures. A fortiori, we would then expect that the knowledge that two  groups $G_1$, $G_2$ are quasi\=/isometric when equipped with their word lengths should give us remarkably little information on the coarse groups $(G_1,\varcrs{bw})$ and $(G_2,\varcrs{bw})$.
 
 However, there are instances where knowledge about the quasi\=/isometry type of the word length of a set\=/group give non\=/trivial information about its canonical coarsification. One such example is given by hyperbolic groups: the property of being non\=/elementary hyperbolic is invariant under quasi\=/isometry and Example~\ref{exmp:p2:hyperbolic groups are unbounded} shows that any such group has an unbounded bi\=/invariant word metric.
 
 In hindsight, the above fact is very surprising. We would like to know under what circumstances the canonical coarsification influences the quasi\=/isometry type of the word metric or vice versa.
\end{rmk}

\chapter{A Quest for Coarse Groups that are not Coarsified Set-Groups}
\label{ch:p2:a cgroup not group}

Many of the explicit examples of coarse groups that we have given so far were obtained by equipping a set\=/group with an equi bi\=/invariant coarse structure.  
It is natural to ask whether every coarse group is a coarsified set\=/group. As stated, the answer to this question is clearly negative. For example, we already remarked that it is easy to construct a coarse group so that the inversion operation is not bijective (Example~\ref{exmp:p1:no cancellative reps}). However, that example is still isomorphic (as a coarse group) to a coarsified set\=/group. The correct question to ask is:

\begin{qu}
 Is every coarse group isomorphic to a coarsified set\=/group?
\end{qu}

We conjecture that the answer is negative (see Conjecture~\ref{conj:p2:non coarsification}). 
In this chapter we will develop some tools that can be used to understand when a coarse group is isomorphic to a coarsification of a set\=/group and provide some evidence towards our conjecture.

\begin{rmk}
 Given a coarse group $\crse G$, we believe that it is generally important to know whether $\crse{G}$ is isomorphic to a coarsification of a set\=/group: when this is not the case we see that the operation is in some way ``intrinsically not associative''. Such a feature should surely have important consequences for the properties of $\crse G$ and its coarse actions. 
\end{rmk}

\section{General Observations}\label{sec:p2:general observations}
In this section we provide a general criterion characterizing which coarse groups are isomorphic to coarsifications of set\=/groups. Given any---possibly infinite---set $X$, we denote by $F_X$ the free set\=/group generated by $X$. Elements of $F_X$ can be identified with reduced words in the alphabet $X\cup X^{-1}$. For any $x\in X$ we also denote by $x$ the one\=/letter word ``$x$''. The other elements of $F_X$ are generally denoted by $w$.

\begin{prop}\label{prop:p2:non_coarsification general}
 Let $\crse G\coloneqq(G,\CE)$ be a coarse group. Then $\crse G$ is isomorphic to a coarsification of a set\=/group if and only if there exists a coarse homomorphism $\crse \Psi \colon (F_{G},\mincrs)\to\crse G$ with a representative $\Psi$ such that $\Psi(g)=g$ for every $g\in G$.
\end{prop}
\begin{proof}
 One implication is clear: if there exists such a $\Psi$ then the pull\=/back $\Psi^{*}(\CE)$ is equi bi\=/invariant on $F_{G}$ by Lemma~\ref{lem:p1:pull-back.under.hom.is.crsegroup} and $\crse G$ is isomorphic to $(F_{G},\Psi^{*}(\CE))$.
 
 For the converse implication, assume that there exists a coarsified set\=/group $\crse H$ and an isomorphism $\crse {f \colon G\to H}$. Fix representatives $f$ for $\crse f$ and $\bar f$ for $\crse{f^{-1}}$. Since $H$ is a set\=/group, there exists a unique set\=/group homomorphism $\Phi\colon F_{G}\to H$ sending $g\in G$ to $f(g)$. A fortiori, such $\Phi$ defines a coarse\=/group homomorphism $\Phi\colon(F_{G},\mincrs)\to \crse H$. The composition $\crse{ f^{-1}}\circ \crse{\Phi}=[\bar {f}\circ \Phi]$ is therefore a coarse homomorphism  from $(F_{G},\mincrs)$ to $\crse G$.
 
 Define $\Psi$ by 
 \[
  \Psi(w)\coloneqq
  \left\{\begin{array}{ll}
   g & \text{if $w=g$ for some $g\in G$,}  \\
   (\bar f\circ \Phi)(w) &\text{otherwise.} 
  \end{array}
  \right.
 \] 
 By construction, $\Psi$ differs from $\bar f\circ\Phi$ only on the $1$\=/letter words in $F_{G}$. However, since $\Phi(g)=f(g)$ and $\bar f\circ f$ is close to $\id_{G}$, we see that $\Psi$ is still close to $\bar f\circ \Phi$. In particular, $\crse \Psi$ is a coarse homomorphism and has the required properties.
\end{proof}

The interesting part of the above criterion is that it can make a rather abstract problem into a very elementary one. Here ``elementary'' means ``easy to state'', but alas not easy to solve! In particular, Proposition~\ref{prop:p2:non_coarsification general} implies the following.

\begin{cor}\label{cor:p2:non_coarsification metric}
 Let $\crse G\coloneqq(G,\CE_d)$ be a coarse group where the coarse structure is induced by a metric $d$. Then $\crse G$ is isomorphic to a coarsification of a set\=/group if and only if there exists a constant $D>0$  and a function $\Psi\colon F_{G}\to G$ such that 
 \begin{itemize}
  \item $\Psi(g)=g$ for every $g\in G$,
  \item $d\bigparen{\Psi(w_1w_2)\,,\,\Psi(w_1)\ast\Psi(w_2)}\leq D$ for every $w_1,w_2\in F_{G}$.
 \end{itemize}
\end{cor}

We will later use this criterion to give evidence to support our conjecture that some coarse subgroups of the coarse group $(F_2,\varcrs{bw})$ are not isomorphic to coarsifications  of set\=/groups.

\

One can apply Proposition~\ref{prop:p2:non_coarsification general} to prove that a coarse group is isomorphic to a coarsified set\=/group if one can find ``special'' representatives for the multiplication function:

\begin{de}\label{def:p2:multi-associative}
 Let $\crse X=(X,\CE)$ be a coarse space. A controlled function $\ast\colon (X,\CE)\times (X,\CE)\to (X,\CE)$ is \emph{coarsely multi\=/associative}\index{coarsely!multi\=/associative} if there exists a bounded set $B\subseteq X$ such that
 \[
  \braces{x_1\ast\cdots \ast x_n\mid \text{associated in any order}}\subseteq B
 \]
 for every $n\in\NN$ and every choice of $ x_i\in X$.
\end{de}

\begin{cor}\label{cor:p2:multi associative}
 If $\crse G$ is a coarse group and the coarse multiplication $\cop$ has a coarsely multi\=/associative representative $\ast$, then $\crse G$ is isomorphic to a coarsified set\=/group.
\end{cor}
\begin{proof}
 Define the map $\Psi\colon F_{G}\to G$ by sending a reduced word $w=g_1\cdots g_n$ to $g_1\ast\cdots \ast g_n$ (with an arbitrarily chosen order of association).
\end{proof}

It is easy to produce examples $\crse G$ where we can choose a representative $\ast$ that is not coarsely multi\=/associative and yet $\crse G$ is isomorphic to the coarsification of a coarse group. For instance, equip $(G,\CE)\coloneqq(\ZZ,\CE_{\abs{\mhyphen}})\times(\ZZ/\ZZ_2,\mincrs)$ with the following operation:
\[
 (a,0)\ast(b,0)\coloneqq (a+b,0)\qquad
 (a,0)\ast(b,1)\coloneqq (a+b+1,1)\hphantom{.}
\]
\[
 (a,1)\ast(b,1)\coloneqq (a+b,0) \qquad
 (a,1)\ast(b,0)\coloneqq (a+b+1,1).
\] 
Obviously, $\ast$ is close to the natural group operation $+$, therefore $\crse G$ is a coarsified set\=/group. However by multiplying $(1,1)$ with $n$ copies of $(1,0)$ one immediately sees that  $\ast$ is not coarsely multi\=/associative. On the other hand, we are not sure whether the converse of Corollary~\ref{cor:p2:multi associative} holds true. That is, it is unclear whether the existence of a coarsely multi\=/associative representative for $\cop$ is a necessary requirement for $\crse G$ to be isomorphic to a coarsified set\=/group.

\section{A Conjecture on Coarse Groups that are not Coarsified Set\=/Groups}\label{sec:p2:conjecture on coarsified groups}

Let $F_2$ be the rank\=/2 free group generated by the set $S=\{a,b\}$. Since $S$ is fixed, in this section we will simply denote the cancellation metric $d_{\overline{S}}$ by $d_\times$. In particular, the canonical coarse structure $\varcrs{bw}$ on $F_2$ coincides with the metric coarse structure $\CE_\times$. 
In this setting, every set\=/group homomorphism $\phi\colon F_2\to \RR$ is obviously controlled as a function $\phi\colon (F_2,\varcrs{bw})\to (\RR,\CE_{\abs{\mhyphen}})$ and hence defines a coarse homomorphism $\crse \phi$ (Corollary~\ref{cor:p1:homomorphisms are grp_fin controlled}).
We already showed (Corollary~\ref{cor:p1:quasimorphism have ker}) that such a $\crse \phi$ always has a coarse kernel $\cker(\crse \phi)\crse\trianglelefteq (F_2,\varcrs{bw})$.
The goal of this section is to justify the following.

\begin{conj}\label{conj:p2:non coarsification}
 Let $\phi\colon F_2\to \RR$ be a set\=/group homomorphism and $\crse\phi\colon (F_2,\varcrs{bw})\to (\RR,\CE_{\abs{\mhyphen}})$ the induced coarse homomorphism. The coarse kernel $\cker(\crse\phi)$ is isomorphic to a coarsification of a set\=/group if and only if the image of $\phi$ is discrete in $\RR$. 
\end{conj}

One implication is easy: assume that $\phi$ has discrete image, $\im(\phi)=\braces{k\theta\mid k\in\ZZ}$. It follows from Corollary~\ref{cor:p1:criterion for existence of ker} that $\cker(\crse \phi)=\crse [\phi^{-1}((-\theta,\theta))]=[\ker(\phi)]$. In particular, $\cker(\crse \phi)$ is isomorphic to $(\ker(\phi),\CE_{d_\times}|_{\ker(\phi)})$, which is a coarsified set\=/group.

On the other hand, let $\phi$ be a homomorphism with dense image. Up to rescaling or replacing $a,b$ with $a^{-1},b^{-1}$, we may assume that $\phi$ sends $a$ to $-1$ and $b$ to $\beta\in\RR_{>1}\smallsetminus\QQ$. Let $K\coloneqq \phi^{-1}([0,1))$, then $\crse K=\cker(\crse \phi)$. The coarse structure on $\crse K$ is induced by the restriction to $K$ of the cancellation metric $d_\times$ of $F_2$, and the conjecture states that $\crse K\coloneqq(K,\CE_{d_\times})$ is not isomorphic to a coarsification of a set\=/group. We rephrase this using Corollary~\ref{cor:p2:non_coarsification metric}:

\begin{conj}\label{conj:p2:non coarsification elementary}
 There do not exist a function $\Psi\colon F_K\to K$ and a constant $D\geq 0$ such that
 \begin{enumerate}[(1)]
  \item $\Psi(k)=k$ for every $k\in K$;
  \item $d_\times\bigparen{\Psi(w_1w_2)\,,\,\Psi(w_1)\Psi(w_2)}\leq D$ for every $w_1,w_2\in F_{K}$.
 \end{enumerate}
\end{conj}

\begin{rmk}
 Technically, to see $\crse K$ as a coarse group we have to choose operations $\ast,\inversefn$ on $K$ as in Section~\ref{sec:p1:coarse subgroups}. We should then apply Corollary~\ref{cor:p2:non_coarsification metric} using this choice of $\ast$.
 However, since we know a priori that $\ast\colon K\times K\to K\subset F_2$ is close to the usual multiplication in $F_S$, we can instead work with the latter (possibly at the cost of choosing a larger constant $D$). This is why the second condition in Conjecture~\ref{conj:p2:non coarsification elementary} only involves the set\=/group multiplications of $F_K$ and $F_2$, and the operation $\ast$ on $K$ does not play any role. 
\end{rmk}

We now use the remainder of this section to outline a strategy of proof (for which we are indebted to Nicolaus Heuer).  The overall strategy is to obtain a contradiction by showing a certain word is ``almost periodic''. We were unfortunately unable to carry out this plan because we lack appropriate tools to compute the cancellation distance $d_\times$ on $F_2$. For every $n\in\NN$ let
\[
 u_n\coloneqq a^{\lfloor n\beta\rfloor}b^n\in F_2
\]
 and note that $u_n$ belongs to $K$ (recall that $\lfloor n\beta\rfloor$ denotes the floor). We claim that to prove Conjecture~\ref{conj:p2:non coarsification elementary} it would be enough to prove the following statement:

\vspace{1 em}

\noindent
\begin{minipage}{0.05\textwidth}
 $(\clubsuit)$
\end{minipage}
\begin{minipage}{0.95\textwidth}
 \textit{
 If $\Psi$ and $D$ satisfy conditions $(1)$ and $(2)$, then there exist some constants $C\gg D$ and $n\in\NN$ large enough so that 
 \[
  d_\times\bigparen{\Psi\paren{u_n^N}\;,\;u_n^N}\leq C
 \]
 for every $N\in\NN$ (in the above inequality, the $u_n^N$ that appears on the left denotes the word consisting of $N$ repetitions of the letter $u_n\in K$. The $u_n^N$ that appears on the right is the element of $F_2$ obtained by multiplying $N$ copies of $u_n$).}
\end{minipage}
\vspace{1 em}

\begin{lem}
 The statement $(\clubsuit)$ implies Conjecture~\ref{conj:p2:non coarsification}.
\end{lem}

\begin{proof}
 We argue by contrapositive.
 Assume that Conjecture~\ref{conj:p2:non coarsification} is false. We may then find $\phi,\ K$ as above and $\Psi\colon F_K\to K$ satisfying conditions (1) and (2) of Conjecture~\ref{conj:p2:non coarsification elementary}.
 For every $g_1,g_2\in F_2$ the difference $\abs{\phi(g_1)-\phi(g_2)}$ is bounded by $\beta d_\times(g_1,g_2)$. If $(\clubsuit)$ true, this implies that 
\[
 \abs{\phi\bigparen{\Psi\paren{u_n^N}}-\phi(u_n^N)}\leq \beta C
\]
for every $N\in\NN$.
Since $\abs{\phi(u_n)^N}= N\abs{\phi(u_n)}$ goes to infinity with $N$, we deduce that for $N$ large enough $\Psi\paren{u_n^N}$ does not belong to $K$, against the hypotheses.
\end{proof}

We now wish to explain why we think that $(\clubsuit)$ is a plausible claim. By hypothesis, both $\Psi(u_n)\Psi(u_n^N)$ and $\Psi(u_n^N)\Psi(u_n)$ are at most at distance $D$ from $\Psi(u_n^{N+1})$. Since $\Psi(u_n)=u_n$, this means that the word $V_{n,N}\coloneqq\Psi(u_n^N)\in F_2$ almost commutes with $u_n$:
\begin{equation}\label{eq:p2:alomst commuting}
 d_\times\bigparen{u_nV_{n,N}\, ,\, V_{n,N}u_n}\leq 2D.
\end{equation}
This poses strong restrictions on the possible choices for $V_{n,N}$. For instance, it is simple to observe that when $D=0$ condition \eqref{eq:p2:alomst commuting} implies that $V_{n,N}\equiv u_n^N$: since the word $u_n$ is primitive and cyclically reduced its centralizer is the set powers of $u_n$. Intuitively, it seems very plausible that something similar should hold also for $D>0$ (so long as $D\ll n$). 

As a warning, we point out that it is definitely not the case that a word commuting with $u_n$ must be close to being a power of $u_n$. For example, letting
\(
 W\coloneqq (au_n)^k
\)
for any $k\in \NN$, we see that 
\(
 d_\times\bigparen{u_nW , Wu_n}= 2
\)
even though $d_\times(u_n^k,W)=k$ is unbounded. However the following might be true: 

\vspace{1 em}

\noindent
\begin{minipage}{0.05\textwidth}
 $(\spadesuit)$
\end{minipage}
\begin{minipage}{0.95\textwidth}
 \textit{
 So long as $D\ll n$, there exists a constant $C'$ depending only on $D$ so that every reduced word $W$ satisfying $ d_\times\bigparen{u_nW , Wu_n}\leq D$ is of the form:
\(
 W = v_1^{k_1}\cdots v_{C'}^{k_{C'}},
\)
where $k_i\in\ZZ$ and $d_\times(u_n,v_i)\leq C'$ for every $i=1,\ldots,C'$.
Importantly, $C'$ does not depend on $n$.
}
\end{minipage}

\vspace{1 em}

\begin{rmk}
 A more audacious claim could be as follows. Say that a reduced word $u$ has \emph{$D$\=/stable almost centralizer} if $(\spadesuit)$ holds true (namely, there exists a constant $C'$ such that every word satisfying $d_\times(uW,Wu)\leq D$ can be decomposed as a product of at most $C'$ powers of words that are $C'$\=/close to $u$). Further, say that a word $u$ is \emph{$R$\=/cyclically reduced} if it is cyclically reduced and every word obtained by cancelling up to $R$ letters from $u$ is still cyclically reduced.
 Similarly, we say that it is \emph{$R$\=/primitive} if every word within $d_\times$\=/distance $R$ of it is primitive. Note that the words $u_n=a^{\floor{n\beta}}b^n$ are $R$\=/cyclically reduced and $R$\=/primitive for every $R<n$.
 
 It seems plausible that for every $D$ there is a $R\gg D$ such that every $R$\=/cyclically reduced $R$\=/primitive word has $D$\=/stable almost centralizer.
\end{rmk}

Claim $(\spadesuit)$ does not directly imply $(\clubsuit)$. However, it should be possible to find some trick to use it to prove $(\clubsuit)$. Specifically, fix some $N\gg n$. For every $m\in\NN$, we know that $\Psi(u_n^m)=v_{m,1}^{k_{m,1}}\cdots v_{m,C'}^{k_{m,C'}}$. Since $C'$ is fixed while $\abs{\Psi(u_n^m)}_\times$ goes to infinity with $m$, we deduce that for some $M\gg N$ there must exist an index $1\leq i\leq C'$ so that $k_{M,i}\gg N$. Let $M=M_1+N+M_2$ for some $M_1,M_2\geq 0$. Then
\[
 \Psi(u_n^{M_1})\Psi(u_n^N)\Psi(u_n^{M_2})\approx_{2D}\Psi(u_n^M)=v_{M,1}^{k_{M,1}}\cdots v_{M,C'}^{k_{M,C'}}.
\]
The idea would then be to find appropriate $M_1$ and $M_2$ so that the most efficient way to take the word $\Psi(u_n^{M_1})\Psi(u_n^N)\Psi(u_n^{M_2})$ to $v_{M,1}^{k_{M,1}}\cdots v_{M,C'}^{k_{M,C'}}$ takes $\Psi(u_n^N)$ to a subword of $v_{M,i}^{k_{M,i}}$. This would mean that $\Psi(u_n^N)$ is within distance $2D$ from a subword of a big power of a word that is $C'$\=/close to $u_n$.
At this point it should not be hard to prove $(\clubsuit)$ or some modification of it which is still sufficient to prove Conjecture~\ref{conj:p2:non coarsification elementary}. 

The problem with the last part of the plan is that it is difficult to control the amount of cancellation that can happen when multiplying $\Psi(u_n^N)$ with $\Psi(u_n^{M_1})$ and $\Psi(u_n^{M_2})$.
More generally, we find that we currently lack some technical tools allowing us to study the cancellation distance. This seems to be a rather interesting combinatorial problem.

\begin{rmk}
 If Conjecture~\ref{conj:p2:non coarsification} holds true, it most likely holds also when $F_2$ is replaced with any finitely generated free group $F_S$ and  $\phi\colon F_S\to V$ is a homomorphism into any Banach space $(V,\norm{\variable})$.
\end{rmk}

\section{A Few More Questions}
As already remarked, we believe that it is an important problem to understand precisely which coarse groups are isomorphic to coarsified set\=/groups. To be a little more specific, let $\crse G$ be a fixed coarse group (possibly a coarsified set\=/group).
A general problem that is worth investigating is the following:

\begin{problem}
 When is the coarse kernel of a coarse homomorphism $\crse{ f\colon G\to H}$ isomorphic to the coarsification of a set\=/group?
\end{problem}

Any kind of ``if and only if'' characterization would most likely be useful for any attempt of classification of coarse groups.
It would already be interesting to answer the above question for some special classes of coarse homomorphism. 

For instance, if $G$ is a finitely normally generated set\=/group then there is a one\=/to\=/one correspondence between coarse homomorphisms $\crse{\phi\colon }(G,\varcrs{bw})\to(\RR,\CE_{\abs{\mhyphen}})$ and closeness\=/classes of $\RR$\=/quasimorphisms $\phi\colon G\to \RR$ (Definition~\ref{def:p1:quasimorphism}). Corollary~\ref{cor:p1:quasimorphism have ker} shows that these coarse homomorphisms always have a coarse kernel, and we thus ask the following:

\begin{problem}\label{prob:p2:ker of qmorph is coarsfied group}
 Let $\phi\colon G\to \RR$ be an $\RR$\=/quasimorphism. When is $\cker(\crse \phi)\crse \trianglelefteq (G,\varcrs{bw})$ isomorphic to the coarsification of a coarse group?
\end{problem}

Note that Conjecture~\ref{conj:p2:non coarsification} is our attempt to solve Problem~\ref{prob:p2:ker of qmorph is coarsfied group} for ``trivial'' $\RR$\=/quasimorphisms of $F_2$, \emph{i.e.}~$\RR$\=/quasimorphisms that are close to homomorphisms of a set\=/group.  Another case of special interest is the following:

\begin{qu}\label{qu:p2:ker of brooks_quasimorphisms}
 Let $\phi\colon F_2\to \RR$ be a Brooks quasimorphism (Example~\ref{exmp:p1:brooks_quasimorphisms}). Is $\cker(\crse \phi)$ isomorphic to the coarsification of a coarse group?
\end{qu}

To explain the interest of the above, recall that an $\RR$\=/quasimorphism is homogeneous if $\phi(g^k)=k\phi(g)$ for every $k\in\ZZ$. It is a classical fact that every $\RR$\=/quasimorphism of $F_2$ is close to a unique homogeneous one (\cite[Section 2.2]{calegari2009scl}). In particular, an $\RR$\=/quasimorphism $\phi$ is close to a homomorphism $\bar\phi$ if and only if $\bar\phi$ is the homogenization of $\phi$. One may then adapt the statement of Conjecture~\ref{conj:p2:non coarsification} to arbitrary $\RR$\=/quasimorphisms by asking whether it is true that $\cker(\crse \phi)$ is isomorphic to a coarsified set\=/group if and only if the homogenization of $\phi$ has dense image in $\RR$.

In view of this, Question~\ref{qu:p2:ker of brooks_quasimorphisms} is especially interesting because the homogenization of a Brooks quasimorphism has discrete image in $\RR$ (hence $\cker(\crse \phi)= [\bar {\phi}^{-1}(0)]$, where $\bar\phi$ is the homogenization of $\phi$). If it turns out that these kernels are not isomorphic to the coarsification of a coarse group, it would mean that the natural extension of Conjecture~\ref{conj:p2:non coarsification} to (homogeneous) $\RR$\=/quasimorphisms is not true.

\

On the other hand, one would expect the situation to be much simpler for coarsely abelian coarse groups. In view of this, we would also like to know the answer to the following:

\begin{qu}
 Is every coarsely abelian coarse group isomorphic to a coarsification of a set\=/group?
\end{qu}

\chapter{On Coarse Homomorphisms and Coarse Automorphisms}
\label{ch:p2:coarse autos}
In this chapter we study the group of coarse automorphisms of (coarsely connected) coarse groups. If $\crse G=(G,\CE)$ is a coarse group, we define $\cAut(\crse G)=\cAut(G,\CE)$\nomenclature[:MOR]{$\cAut(\crse G),\ \cAut(G,\CE)$}{coarse automorphisms of a coarse group} as the group of self\=/isomorphisms of $\crse G$.
That is, $\cAut(\crse G)\leq\cmor(\crse G,\crse G)$ is the subset of coarse\=/homomorphisms that have a coarse inverse.

\section{Elementary Constructions of Coarse Automorphisms}
The most elementary examples of coarse automorphisms for a coarse group $\crse G$ are the conjugations ${}_gc(h)\coloneqq (g\ast h)\ast g^{-1}$. However, we already noted that, whenever $g$ lies in the same coarsely connected component of $\unit$, the map ${}_gc$ is close to the identity (Example~\ref{exmp:p1:conjugation automorphisms}). In particular, in the case of special interest where $\crse G$ is coarsely connected, the conjugation induces a trivial map $G\to \cAut(\crse G)$ sending every $g\in G$ to $\cid_{\crse G}$.

\begin{rmk}
We emphasize that $[{}_gc]=\cid_{\crse G}$ for each $g\in G$ does not imply that the coarse action by conjugation $\crse{G\curvearrowright G}$ is trivial (see Section~\ref{sec:p1:action by conjugation}). This hints to the fact that the set\=/group $\cAut(\crse X)$ is ill\=/suited to describe coarse actions $\crse{ G\curvearrowright X}$. In a sense, this is because $\cAut(\crse X)$ is the set\=/group of ``coarsely connected components of the space of transformations of $\crse X$''. From this point of view, it is clear that $\cAut(\crse X)$ cannot detect the coarse actions of a coarsely connected coarse group $\crse G$. This idea is made precise and further developed in Sections~\ref{sec:p2:coarse actions revisited} and \ref{sec:p2:frag coarse groups of transformations}, where we construct a suitable ``space of transformations'' of a coarse space.
\end{rmk}

If $\crse G$ is a coarsified set\=/group, every controlled homomorphism $f\colon G\to G$ determines a coarse homomorphism $\crse f \colon\crse{G}\to\crse G$. In particular,if $f\in \aut(G)$ is an automorphism such that both $f$ and $f^{-1}$ are controlled then $\crse f \in\cAut(\crse G)$. 

\begin{rmk}
 Notice that set\=/group automorphisms need not be controlled in general. 
For example, let $G$ be the abelian group $\bigoplus_{n\in\NN}\RR$ equipped with the (non-complete!) $\ell^2$-norm $\norm{\variable}_2$. Then the linear function $f\colon G\to G$ sending the sequence $(a_n)_{n\in\NN}$ to $(na_n)_{n\in\NN}$ is a set\=/group automorphism of $G$, but it is not controlled with respect to the metric coarse structure.
\end{rmk}

In some cases one can say a priori that every set\=/group automorphism is automatically controlled.
For example, we already saw that this is the case for $\crse G=(G,\varcrs[grp]{fin})$, where $\varcrs[grp]{fin}$ is the minimal connected equi bi\=/invariant coarse structure (recall that the pull\=/back $f^{*}(\varcrs[grp]{fin})$ for a set\=/group homomorphism $f\colon G\to G$ is equi bi\=/invariant by Lemma~\ref{lem:p1:pull-back.under.hom.is.crsegroup}. Since $f$ sends finite sets to finite sets, we see that $\varcrs[grp]{fin}\subseteq f^{*}(\varcrs[grp]{fin})$ by minimality).

In particular, if $G$ is a finitely (normally) generated group, $\varcrs[grp]{fin} = \varcrs{bw}$ is the canonical coarsification of $G$ and we obtain a natural homomorphism $\aut(G)\to\cAut(G,\varcrs{bw})$. Since the coarse structure $\varcrs{bw}$ is connected, conjugation automorphisms are close to the identity. Therefore, the above homomorphisms quotient through the space of outer automorphisms $\Out(G)\to\cAut(G,\varcrs{bw})$.

\begin{exmp}
 An example of special interest is the set\=/group $\ZZ^n$.
 Here the canonical coarsification $\varcrs{bw}$ is equal to the coarse structure defined by the Euclidean metric $\CE_{\norm{\mhyphen}}$, and we thus obtain a set\=/group homomorphism $\out(\ZZ^n)\to\cAut(\ZZ^n,\CE_d)$. Moreover, $\ZZ^n$ is coarsely dense in $(\RR^n,\CE_{\norm{\mhyphen}})$ and therefore the inclusion $\ZZ^n\hookrightarrow\RR^n$ gives a natural isomorphism of coarse groups $(\ZZ^n,\CE_{\norm{\mhyphen}})\cong(\RR^n,\CE_{\norm{\mhyphen}})$. Noting that the set\=/group of outer automorphisms of $\ZZ^n$ is $\out(\ZZ^n)=\aut(\ZZ^n)\cong \Gl(n,\ZZ)$, we obtain a homomorphism
 \[
  \Gl(n,\ZZ)\to \cAut(\ZZ^n,\CE_{\norm{\mhyphen}})\cong \cAut(\RR^n,\CE_{\norm{\mhyphen}}).
 \]
 We will see in Corollary \ref{cAut_GL} that the right hand side is in fact isomorphic to $\Gl(n,\RR)$ and the above map is just the inclusion $\Gl(n,\ZZ)<\Gl(n,\RR)$.
\end{exmp}

It is hard to over estimate the importance of groups of outer automorphisms of finitely generated groups.
Since we always have a homomorphism $\out(G)\to\cAut(G,\varcrs{bw})$, it is natural to ask whether this map is an inclusion:

\begin{problem}\label{qu:p2:injection of out}
 Let $G$ be a finitely (normally) generated group. When is the homomorphism $\out(G)\to\cAut(G,\varcrs{bw})$ injective?
\end{problem}

We will later see that this homomorphism is injective for finite rank free groups (Corollary~\ref{cor:p2:outFn in cAut}).
On the other hand, if $K$ is a finite set\=/group (or, more generally, intrinsically bounded) then $\cAut(K,\varcrs{bw})=\{\cid_{\crse K}\}$ is trivial. If $\out(K)$ is non\=/trivial then $\out(K)\to\cAut(K,\varcrs{bw})$ is clearly not injective. Similarly, if $G=H\times K$ for some $K$ as above then $\out(K)\leq\out(G)$ while $(G,\varcrs{bw})\cong(H,\varcrs{bw})$ and again we see that
\[
 \out(G)\to \cAut(G,\varcrs{bw})\cong \cAut(H,\varcrs{bw}) 
\]
is not injective.

\begin{rmk}
In geometric group theory there are various results that are proved for (torsion free) finitely generated groups $G$ with finite or trivial $\out(G)$ (\emph{e.g.}\ \cite{huang2018commensurability}). It would not be surprising if some of those arguments can be rephrased  by saying that some finiteness condition on $\cAut(G,\varcrs{bw})$ implies some degree of rigidity for $G$.
\end{rmk}

\section{Coarse Homomorphisms into Banach Spaces}
\label{sec:p2:Banach homomorphisms}
It is generally easier to study (coarsely) abelian coarse groups. The easiest among them are those obtained as coarsifications of finite dimensional normed vector spaces or, more generally, Banach spaces. In this setting we can understand coarse homomorphisms and coarse automorphisms very precisely.

Let $(V_1,\norm{\mhyphen}_1)$ and $(V_2,\norm{\mhyphen}_2)$ be normed vector spaces (for concreteness, we will only deal with vector spaces over $\RR$). The norms define a natural coarsification of their additive groups: $\crse V_1\coloneqq (V_1,\CE_{\norm{\mhyphen}_1})$ and $\crse V_2\coloneqq (V_2,\CE_{\norm{\mhyphen}_2})$. 
It follows by the definition that a linear operator $T\colon (V_1,\norm{\mhyphen}_1)\to(V_2,\norm{\mhyphen}_2)$ is controlled if and only if it is bounded (\emph{i.e.}\ continuous).
We therefore have a natural homomorphism
\[
 \CB\bigparen{(V_1,\norm{\mhyphen}_1),(V_2,\norm{\mhyphen}_2)}\to 
 \chom\bigparen{\crse V_1,\crse V_2},
\]
where $\CB$ is the space of bounded\footnote{%
Of course, here ``bounded'' is meant in the sense of linear operators (\emph{i.e.}~the operator norm $\norm{T}$ is bounded). This does \emph{not} mean that $T$ is a bounded \emph{function}.
} linear operators.

Note that the above homomorphism is always injective. Furthermore, if $(V_2,\norm{\mhyphen}_2)$ is complete (\emph{i.e.}\ a Banach space) then it is also surjective:

\begin{prop}\label{prop:p2:banach chom are linear}
 Let $(V_1,\norm{\mhyphen}_1)$ be a normed vector space and $(V_2,\norm{\mhyphen}_2)$ a Banach space. If $\crse f \colon \crse V_1\to \crse V_2$ is a coarse homomorphism of the coarsified additive groups, then there exists a representative $\bar f$ for $\crse f$ that is linear and continuous. 
\end{prop}
\begin{proof}
 Define $\bar f$ by
 \[
  \bar f(v)\coloneqq \lim_{k\to \infty} \frac{f(2^k v)}{2^k}
  =\lim_{n\to \infty} \frac{f(nv)}{n}
 \]
 for every $v\in V_1$. We know from Proposition~\ref{prop:p1:homogeneization.of.qmorph} that $\bar f$ is well\=/defined, it is close to $f$ and satisfies $f(kv)=k f(v)$ for every $k\in\ZZ$. This also implies that $f(\frac 1kv)=\frac 1k f(v)$.
 
 Let $0\leq D<\infty$ be the defect of $f$. That is, $D=\sup_{v,w\in V_1}\norm{f(v_1+v_2)-f(v_1)-f(v_2)}_2$. Then
 \[
  \norm{\bar f(v+w)-\bar f(v)-\bar f(w)}_2
  =\lim_{n\to\infty} \frac 1n\norm{f(n(v+w))-f(nv)-f(nw)}_2
  \leq\lim_{n\to\infty} \frac Dn =0.
 \]
 This shows that $\bar f$ is additive (and hence linear over $\QQ$).

 Let $B_{V_i}(r)$ denote the ball of radius $r$ centered at the origin of $(V_i,\norm{\variable}_i)$.
 Since $f$ is controlled and $\bar f$ is close to it, $\bar f$ is also controlled. Therefore there is some $R> 0 $ such that $\bar f$ sends the ball of radius $1$ centered at the origin of $V_1$ into the ball of radius $R$ centered at the origin of $V_2$. It follows that $\bar f(B_{V_1}(\frac 1k))\subseteq B_{V_2}(\frac Rk)$ for every $k\in \NN$ and hence $\bar f$ is continuous at $0$. Together with $\QQ$\=/linearity, this implies that $\bar f$ is $\RR$\=/linear and continuous.
\end{proof}

\begin{cor}\label{cor:p2:coarse hom Banach space are linear}
 If $(V_2,\norm{\mhyphen}_2)$ is a Banach space then
\[
 \chom\bigparen{\crse V_1,\crse V_2}
 \cong \CB\bigparen{(V_1,\norm{\mhyphen}_1),(V_2,\norm{\mhyphen}_2)}.
\]
\end{cor}

\begin{cor}\label{cAut_GL}
 If $\norm{\mhyphen}$ is a norm on $\RR^n$ then 
\[
 \cAut(\RR^n,\CE_{\norm{\mhyphen}})\cong \Gl(n,\RR).
\]
\end{cor}

Recall that the Open Mapping Theorem states that if $(V_1,\norm{\mhyphen}_1),(V_2,\norm{\mhyphen}_2)$ are Banach spaces and $T\colon V_1\to V_2$ is a surjective continuous operator, then it must be an open mapping. In particular, if such a $T$ is bijective then the inverse is automatically continuous. Noticing that a continuous group\=/homomorphism $V_1\to V_2$ must be a bounded $\RR$\=/linear operator, we deduce the following:

\begin{cor}
 If $(V,\norm{\mhyphen})$ is a Banach space then $\cAut(\crse V)$ is naturally isomorphic to the group of continuous group automorphisms of the additive group $V$:
 \[
  \cAut(\crse V)\cong \braces{T\in \aut(V)\mid T\text{ is continuous}}.
 \]
\end{cor}

Using Corollary~\ref{cor:p2:coarse hom Banach space are linear} and the Open Mapping Theorem, we can also characterize those coarse homomorphisms between Banach spaces which admit a coarse kernel (Definition~\ref{def:p1:coarse kernel}):

\begin{thm}\label{thm:p2:banach chom and kernels}
 Let $\crse T \colon \crse V_1\to\crse V_2$ be a coarse homomorphism between coarsified Banach spaces and let $T$ be its linear representative. Then $\crse T$ has a coarse kernel if and only if $T(V_1)$ is a closed subspace in $V_2$. When this is the case, $\cker(\crse T)=[\ker(T)]$.
\end{thm}
\begin{proof}
 We start by showing that if the coarse kernel exists then it must be $[\ker(T)]$. This argument is similar to that of Example~\ref{exmp:p1:Hilbert.space.monomorphism}: let $\crse H$ be any coarse subgroup of $\crse V_1$. Since the inclusion $\iota\colon H\hookrightarrow (V_1,\CE_{\norm{\mhyphen}_1})$ is a coarse homomorphism, we can hence use Proposition~\ref{prop:p1:homogeneization.of.qmorph} to construct $\bar\iota\colon H\to V_1$ that is close to $\iota$ and such that $\bar\iota(h\ast h)=2\bar\iota(h)$. It follows that the image $\bar\iota(H)$ is closed under scalar multiplication by $2^k$. In order for $T\circ\bar \iota$ to have bounded image it is hence necessary that $\bar\iota(H)\subseteq \ker(T)$ and hence $H\csub \ker(T)$. This proves our claim.
 
 It follows from the definition that $[\ker(T)]$ is the coarse kernel of $\crse T$ if and only if for every $r>0$ there is an $R>0$ so that the preimage of $B_{V_2}(r)$ is contained in the $R$\=/neighborhood of $\ker(T)$. In other words, we must have
 \begin{equation}\label{eq:p2:open mapping}
  T^{-1}\paren{B_{V_2}(r)}\subseteq \bigparen{B_{V_1}(R) + \ker(T)}.
 \end{equation}
 
 Assume now that $T(V_1)$ is closed in $V_2$. Then $T(V_1)$ is a Banach space and we can apply the Open Mapping Theorem to $T\colon V_1\to T(V_1)$. It follows that for every $r>0$ there is an $R>0$ so that $B_{T(V_1)}(r)\subseteq T(B_{V_1}(R))$. 
 In turn, this implies that $T^{-1}\paren{B_{V_2}(r)}=T^{-1}\paren{B_{T(V_1)}(r)}\subseteq B_{V_1}(R)+\ker(T)$.
This proves that $[\ker(T)]$ is the coarse kernel of $\crse T$.
 
For the converse implication, \eqref{eq:p2:open mapping} implies that $T\colon V_1\to T(V_1)$ is an open map and hence descends to a homeomorphism of the quotient $\overline{T}\colon V_1/\ker(T)\to T(V_1)$. Since $\ker(T)$ is a closed subspace, $V_1/\ker(T)$ is a Banach space. Then $T(V_1)$  is also complete and therefore closed in $V_2$.
\end{proof}

\section{Hartnick-Schweitzer Quasimorphisms}\label{sec:p2:HS_qmorph}
Recall that an $\RR$\=/quasimorphism\index{quasimorphism!$\RR$} on a set-group $G$ is a function $\phi\colon G \to \RR$ which has the property that $|\phi(gh) - \phi (g) - \phi(h)|$ is uniformly bounded (Definition~\ref{def:p1:quasimorphism}). We encountered $\RR$\=/quasimorphisms multiple times, as they provide natural examples of coarse homomorphisms. More precisely, we already noted that if $\crse G=(G,\CE)$ is a coarsification of $G$, then a function $f\colon G\to \RR$ defines a coarse homomorphism $\crse f\colon \crse G\to(\RR,\CE_{\abs{\mhyphen}})$ if and only if it is a controlled $\RR$\=/quasimorphism. Also note that every $\RR$\=/quasimorphism is controlled if $\CE$ is $\mincrs$ or $\varcrs[grp]{fin}$, so that
\[
 \chom\bigparen{(G,\mincrs)\;,\;(\RR,\CE_{\abs{\mhyphen}}}=
 \chom\bigparen{(G,\varcrs[grp]{fin})\;,\;(\RR,\CE_{\abs{\mhyphen}}}=
 \bigbraces{\text{$\RR$\=/quasimorphisms of $G$}}/_{\closefn}.
\]
The above observation plays a key role in this section.

Hartnick and Schweitzer \cite{HS} introduced a notion of quasimorphism between set\=/groups which will allow us to construct a zoo of coarse automorphisms. These quasimorphisms (which we will call HS\=/quasimorphisms) are functions that are well\=/behaved with respect to composition with $\RR$\=/quasimorphisms.

\begin{de}[{\cite{HS}}] 
 A \emph{HS\=/quasimorphism}\index{quasimorphism!HS} between set\=/groups $G,H$ is a function $f\colon G \to H$ such that if $\phi$ is a $\RR$\=/quasimorphism on $H$ then $f\circ\phi$ is a $\RR$\=/quasimorphism on $G$. Two HS\=/quasimorphisms $f,g\colon G\to H$ are \emph{equivalent}\index{quasimorphism!HS!equivalent} if $\abs{\phi\circ f -\phi\circ g}$ is a bounded function for every such $\phi$ (equivalently, $(\phi\circ f)\closefn (\phi\circ g)$ as functions into $(\RR,\CE_{\abs{\mhyphen}})$).
\end{de}

It is immediate from the definition that compositions of HS\=/quasimorphisms are HS\=/quasimorphisms, and that equivalence of HS\=/quasimorphisms is preserved under composition. Hartnick and Schweitzer designed their definition to be ``categorical'', and this turns out to fit very naturally in the theory of coarse groups and coarse homomorphism. More precisely, the link between HS\=/quasimorphism and coarse homomorphisms is provided by the pull\=/back coarse structures defined in Example~\ref{exmp:p1:pull back coarse structures}. 
Namely, any set\=/group $G$ can be equipped with the coarse structure 
 \begin{align*}
   \CE_{(G,\mincrs)\to(\RR,\CE_{\abs{\mhyphen}})}
  \coloneqq & \bigcap\Bigbraces{\phi^{*}(\CE_{\abs{\mhyphen}})\bigmid \crse\phi\colon(G,\mincrs)\to(\RR,\CE_{\abs{\mhyphen}})\text{ coarse homomorphism}}   \\
  =&\bigcap\Bigbraces{\phi^{*}(\CE_{\abs{\mhyphen}})\bigmid \phi\colon G\to \RR,\text{ $\RR$\=/quasimorphism}}
 \end{align*}
 (this can be described as the maximal coarse structure for which quasimorphisms are controlled). For brevity, we will denote this coarse structure by $\CE_{G\to\RR}$. The following is easy to prove.
 
\begin{prop}\label{prop:p2:quasim_iff_chom pullback} A function between set\=/groups $f \colon G \to H$ is a HS\=/quasimorphism if and only if $\crse f \colon(G,\CE_{G\to\RR})\to(H,\CE_{H\to\RR})$ is a coarse homomorphism. 
\end{prop}
\begin{proof} 
 Let $\phi\colon H\to \RR$ be an $\RR$\=/quasimorphism. Then, by definition, $\crse \phi \colon(H,\CE_{H\to\RR})\to(\RR,\CE_{\abs{\mhyphen}})$ is controlled and hence a coarse homomorphism. If $\crse f$ is a coarse homomorphism then the composition $\crse{\phi\circ f}$ is a coarse homomorphism and hence an $\RR$\=/quasimorphism. This shows that $f$ is indeed an HS\=/quasimorphism.
 
 Assume now that $f$ is an HS\=/quasimorphism. Then 
 \begin{align*}
  \CE_{G\to\RR}
  &\subseteq \bigcap\Bigbraces{(\phi\circ f)^{*}(\CE_{\abs{\mhyphen}})\bigmid \phi\colon G\to \RR,\text{ $\RR$\=/quasimorphism}}
  = f^{*}(\CE_{H\to\RR}),
 \end{align*}
 and hence $f\colon (G,\CE_{G\to\RR})\to(H,\CE_{H\to\RR})$ is controlled. It remains to check that 
$f(g_1)\ast_{H} f(g_2)$ and $f(g_1\ast_{G} g_2)$ are uniformly close (with respect to $\CE_{H\to\RR}$). 
 For every $\RR$\=/quasimorphism, both $\phi\circ f$ and $\phi$ are coarse homomorphisms. Therefore, there are $F_1,F_2\in \CE_{\abs{\mhyphen}}$ such that
\[
 (\phi\circ f)(g_1 \ast_{G} g_2)
 \torel{F_1} \bigparen{(\phi\circ f)(g_1)+ (\phi\circ f)(g_2)}
 \torel{F_2}\phi(f(g_1) \ast_{H} f(g_2))
\]
 and hence $f(g_1 \ast_{G} g_2) \torel{(\phi\times\phi)^{-1}(F_1 \circ F_2)}(f(g_1) \ast_{H} f(g_2)$. This shows that 
 \[
  \bigbraces{\bigparen{f(g_1)\ast_{H} f(g_2)\,,\, f(g_1\ast_{G} g_2)}\mid g_1,g_2\in G}\in \phi^{*}(\CE_{\abs{\mhyphen}})
 \]
 for every such $\phi$.
\end{proof}

Amenable set\=/groups manifest some sort of rigidity, which allows Hartnick and Schweitzer to classify HS\=/quasimorphisms in this case (see below for more on this). On the contrary, non\=/abelian free groups seem to be very flexible and admit a huge variety of HS\=/quasimorphisms. In what follows, fix a finite set $2\leq \abs{S}<\infty$ and let $F_S$ be the free group freely generated by $S$. Denote by $x_i$ the elements of $S\cup S^{-1}$. The following examples are taken from \cite[Section 5]{HS}:

\begin{exmp}\label{exmp:p2:HS_qmorph local}
 Fix some $k\in \NN$ and a function $r\colon (S\cup S^{-1})^k\to G$ so that $f(u^{-1})=f(u)^{-1}$ for every $u\in (S\cup S^{-1})^k$. The function $r$ plays the role of a substitution rule and allows us to define a function $f_r\colon F_S\to G$ sending a reduced word $w=x_1\cdots x_n\in F_S$  to $e\in G$ if $n< k$ and to the product 
 \[
  r(x_1,\ldots,r_k)r(x_2,\ldots, r_{k+1})\cdots r(x_{n-k},\ldots,x_n)\in G
 \]
 if $n\geq k$. That is, $f_r(w)$ is computed by substituting each letter $x_i$ of the word $w$ with some element of $G$ according to a rule that only looks at the $k$ letters following $x_i$.
 For any substitution rule $r$, the function $f_r\colon F_S\to G$ is an HS\=/quasimorphism. We call these \emph{local HS\=/quasimorphisms}\index{quasimorphism!HS!local}.   
\end{exmp}

\begin{exmp}\label{exmp:p2:HS_qmorph replacement}
 Two non\=/empty reduced words $w_1,w_2\in F_S$ are \emph{overlapping}\index{overlapping words} if they are subwords of one another or if there is a non\=/empty subword that is initial in one of them and terminal in the other. If $w_1,\ldots,w_k$ is a set of words that are pairwise non\=/overlapping and such that each $w_i$ is not self\=/overlapping, then every reduced word $v\in F_S$ can be uniquely decomposed as a concatenation $v=u_1\cdots u_n$ such that:
 \begin{enumerate}
  \item each  $u_i$ is equal to one of the $w_j$ or it does not have any $w_j$ as a subword;
  \item $n$ is minimal among the decompositions such that (1) holds
 \end{enumerate}
 (see \cite[Lemma~5.13]{HS}).
 
 Let $w_1,w_2$ be two reduced words with the same first and last letter and such that $\{w_1,w_1^{-1},w_2,w_2^{-1}\}$ is a set of pairwise non\=/overlapping and non\=/self\=/overlapping words. We may then use the unique decomposition property of above to define a function $f_{w_1,w_2}\colon F_S\to F_S$ that ``swaps $w_1$ and $w_2$''. Namely, if $w=u_1\cdots u_n$ is the unique decomposition of a reduced word $v$, we let $f_{w_1,w_2}(v)\coloneqq f_{w_1,w_2}(u_1)\cdots f_{w_1,w_2}(u_n)$ where
 \[
  f_{w_1,w_2}(u_i)\coloneqq\left\{ 
  \begin{array}{ll}
   w_2 & \text{if } u_i=w_1 \\
   w_1 & \text{if } u_i=w_2\\
   w_2^{-1} & \text{if } u_i=w_1^{-1}\\
   w_1^{-1} & \text{if } u_i=w_2^{-1}\\
   u_i & \text{otherwise.} \\
  \end{array}
  \right.
 \]

 The function $f_{w_1,w_2}$ is an HS\=/quasimorphism, called \emph{replacement quasimorphism}.\index{quasimorphism!replacement} It is worthwhile noting that $(f_{w_1,w_2})^2=\id_{F_S}$ and that the composition with $f_{w_1,w_2}$ swaps the Brooks counting $\RR$\=/quasimorphisms $\phi_{w_1}$, $\phi_{w_2}$ (the Brooks $\RR$\=/quasimorphisms are defined in Example~\ref{exmp:p1:brooks_quasimorphisms}). The latter fact can be used to show that the group of invertible HS\=/quasimorphisms acts almost transitively on the set of counting $\RR$\=/quasimorphisms associated with non\=/self\=/overlapping words (\cite[Corollary 5.18]{HS}, see also \cite[Theorem~5.21]{HS}). One may also observe that every replacement HS\=/quasimorphism is equivalent to a local HS\=/quasimorphism.
\end{exmp}

\begin{exmp}\label{exmp:p2:HS_qmorph wobbling} 
 Define the \emph{wobbling group}\index{wobbling group} $W(\NN_{>0})$\nomenclature[:z]{$W(\NN_{>0})$}{wobbling group} as the set\=/group of permutations on $\NN_{>0}=\{1,2,\ldots\}$ with bounded displacement (\emph{i.e.}\ those $\sigma$ so that $\abs{\sigma(k)-k}$ is uniformly bounded as $k$ ranges in $\NN_{>0}$). 
 As above, given a non self\=/overlapping cyclically reduced word $v$ we can uniquely decompose any reduced word $w$ as a concatenation $u_0v^{k_1}u_1v^{k_2}u_2\cdots v^{k_n}u_n$ where no $u_i$ has $v$ as a subword, $k_i\neq 0$ for every $i$, and $u_i\neq \emptyset$ for $0<i<n$. For every $\sigma\in W(\NN_{>0})$ we may define a function $f_{v,\sigma}\colon F_S\to F_S$ by acting with $\sigma$ on the exponents $k_i$ appearing in the above decomposition. Namely, we let
 \[
  f_{v,\sigma}(w)\coloneqq u_0v^{\sigma_\pm(k_1)}u_1v^{\sigma_\pm(k_2)}u_2\cdots v^{\sigma_\pm(k_n)}u_n
 \]
 where $\sigma_\pm(k)=\sigma(k)$ if $k$ is positive and $\sigma_\pm(k)=-\sigma(-k)$ otherwise.
 
 The functions $f_{v,\sigma}$ are invertible HS\=/quasimorphisms, as $f_{v,\sigma^{-1}}=f_{v,\sigma}^{-1}$. One may also check that $f_{v,\sigma^{-1}}$ is equivalent to a local HS\=/quasimorphism if and only if $\sigma$ has finite support. Note that the group $W(\NN_{>0})$ is uncountable, and it has elements of arbitrary order.
\end{exmp}

\begin{rmk}
 We may generalize the above example as follows. Let $v$ be a non self\=/overlapping cyclically reduced word, so that any reduced word $w$ has a unique decomposition $w=u_0v^{k_1}u_1\cdots v^{k_n}u_n$ as before. Take a $\ZZ$\=/quasimorphism $f\colon \ZZ \to \ZZ$ that is symmetric (\emph{i.e.}\ so that $f(-k)=-f(k)$ for every $k\in \ZZ$). Then one may check that applying $f$ to the exponents of $v$ in the decompositions of reduced words defines an HS\=/homomorphism $F_S\to F_S$
\[ u_0v^{k_1}u_1v^{k_2}u_2\cdots v^{k_n}u_n \mapsto u_0v^{f(k_1)}u_1v^{f(k_2)}u_2\cdots v^{f(k_n)}u_n.  \]
\end{rmk}

The reason why it is relatively simple to exhibit examples of HS\=/quasimorphisms of the free group is that elements in $F_S$ can be uniquely identified with reduced words. To prove that the above examples are HS\=/quasimorphisms, Hartnick and Schweitzer use the following criterion:

\begin{prop}[{\cite[Proposition 5.3]{HS}}]\label{prop:p2:HS criterion}
 Let $f\colon F_S\to G$ be a function from the free group to an arbitrary set\=/group $G$. If there exists a finite set $B\subset G$ such that
 \begin{enumerate}[(i)]
  \item $f(w_1w_2)\in B\ast f(w_1)\ast B\ast f(w_2)\ast B$ whenever $w_1,w_2$ and their concatenation $w_1w_2$ are reduced words;
  \item $f(w^{-1})\in B\ast f(w)^{-1}\ast B$;
 \end{enumerate}
 then $f$ is an HS\=/quasimorphism.
\end{prop}

We will later see how to deduce Proposition~\ref{prop:p2:HS criterion} from general observations about coarse homomorphisms (Remark~\ref{rmk:p2:proof of HS criterion}).

\

The set of equivalence classes of HS\=/quasimorphisms $f\colon G\to H$ is denoted $Q(G,H)$. The set $Q(G,G)$ equipped with composition is a monoid. Hartnick and Schweitzer define the group of \emph{quasioutomorphisms} $\qOut(G)$ of a set\=/group $G$ as the set\=/group of invertible elements of ${\rm Q}(G,G)$. Using this notation, we may improve on Proposition~\ref{prop:p2:quasim_iff_chom pullback} as follows:

\begin{prop}
 Given set\=/groups $G,H$ we have
 \[
  Q(G,H) = \chom\bigparen{(G,\varcrs{G\to \RR})\,,\, (H,\varcrs{H\to \RR})}.
 \]
\end{prop}
\begin{proof}
 We already showed that a function $f\colon G\to H$ is an HS\=/quasimorphism if and only if it defines a coarse homomorphism. It only remains to observe that two HS\=/quasimorphisms $f_1,f_2$ are equivalent if and only if they are close.
 
 By definition of $\varcrs{H\to \RR}$, a relation $E$ on $H$ is a controlled entourage if and only if for every $\RR$\=/quasimorphism $\phi\colon H\to \RR$ there is an upper bound on the differences $\abs{\phi(x_1)-\phi(x_2)}$ where $x_1\torel{E}x_2$. 
 By definition, the HS\=/quasimorphisms $f_1$, $f_2$ are equivalent if and only if the difference $\abs{\phi\circ f_1(g)-\phi\circ f_2(g)}$ is uniformly bounded, \emph{i.e.}\ if and only if $f_1\times f_2(\Delta_G)$ is in $\varcrs{H\to \RR}$.
\end{proof}

\begin{cor}\label{cor:p2:HS_qmorph are coarse aut}
 For any set\=/group $G$ we have $\qOut(G)=\cAut(G,\CE_{G\to\RR})$.
\end{cor}

As mentioned above, one interesting result proved in \cite{HS} is that amenable groups are somewhat rigid from the point of view of HS\=/quasimorphisms. Namely, they prove the following classification of quasioutomorphisms:

\begin{thm}[{\cite{HS}}]
 If a set\=/group $G$ is amenable and its abelianization has finite rank $k$ then there is a natural isomorphism $\qOut(G)\cong \GL(k,\RR)$.
\end{thm}

On the other hand, the situation for non\=/abelian free groups is much more flexible. The examples of HS\=/quasimorphisms constructed in Example~\ref{exmp:p2:HS_qmorph replacement} and Example~\ref{exmp:p2:HS_qmorph wobbling} are all invertible and hence define elements in $\qOut(F_S)$. Using appropriate $\RR$\=/quasimorphisms, it is not hard to show that these examples are not equivalent. In particular, this shows that that the wobbling group $W(\NN_{>0})$ embeds in $\qOut(F_S)$ \cite[Theorem~5.6]{HS}. 

Of course, every automorphism of a set\=/group $G$ gives rise to an invertible HS\=/quasimorphism, so there is a natural homomorphism $\aut(G)\to\qOut(G)$. This homomorphism factors through the set\=/group of outer automorphisms $\Out(G)$. With some work, it is possible to show that this homomorphism is injective for a finite rank free group:

\begin{prop}[{\cite[Proposition 5.1]{HS}}]\label{prop:p2:out_Fn embeds in HS}
 For every $2\leq \abs{S}<\infty$, the canonical homomorphism 
 \[
  \Out(F_S)\hookrightarrow\qOut(F_S)\cong\cAut(F_S,\CE_{F_S\to\RR})
 \]
 is injective.
\end{prop}

We conclude this section with some ideas for future research: 

\begin{rmk}
 At first sight, it would be natural to guess that the existence of ``exotic'' quasioutomorphisms could be yet another instance of the amenable vs. non\=/amenable divide. However, this is not the case. In fact, we already noted in Section~\ref{sec:p2:examples of (un)bounded}) that the set\=/group $\Sl(n,\ZZ)$ is intrinsically bounded for every $n\geq 3$. Since $\qOut(\Sl(n,\ZZ))\cong\cAut(\Sl(n,\ZZ),\CE_{\Sl(n,\ZZ)\to\RR})$ and  $\CE_{\Sl(n,\ZZ)\to\RR}= \maxcrs$ by intrinsic boundedness, it follows that $\qOut(\Sl(n,\ZZ))=\{e\}$.
 
 It is therefore very unclear which set\=/groups $G$ admit exotic examples of quasioutomorphisms. One issue that we are facing is the lack of examples for non\=/free groups: the examples illustrated in this section make heavy use of the identification between group elements and reduced words. 
 In view of this, it seems fairly natural to start investigating HS\=/quasimorphisms for finitely presented groups that have particularly nice presentations. 
\end{rmk}

\begin{rmk}\label{rmk:p2:HS is coarsely abelian}
 It would be interesting to understand the coarse structures $\CE_{G\to\RR}$ for specific examples of set\=/groups $G$. 
We can of course make a few simple observations. For any set\=/group $G$, $\CE_{G\to\RR}$ is a connected coarse structure because $(\RR,\CE_{\abs{\mhyphen}})$ is itself coarsely connected (in other words, $\varcrs[grp]{fin}\subseteq \CE_{G\to\RR}$). In particular, if $G$ is an intrinsically bounded set\=/group (Definition~\ref{def:p2:intrinsically bounded}) then $\CE_{G\to\RR}=\CE_{\rm max}$ is the bounded coarse structure.

Similarly, it follows from the abelianity of $\RR$ that $(G,\CE_{G\to\RR})$ is necessarily coarsely abelian. In particular, the identity map $\cid \colon (G,\varcrs[grp]{fin})\to(G,\CE_{G\to\RR})$ factors through the coarse abelianization of $(G,\varcrs[grp]{fin})$ (Remark~\ref{rmk:p1:on coarse abelian groups}). It would be interesting to characterize the set\=/groups $G$ for which $(G,\CE_{G\to\RR})$ \emph{is} the coarse abelianization.
By means of Bavard's Duality \cite{bavard1991longueur} (see also \cite[Theorem 2.70]{calegari2009scl}), this problem relates to the question of bounding the stable commutator length in terms of commutator length.
\end{rmk}

\section{Coarse Homomorphisms of Finitely Normally Generated Groups}\label{sec:p2:coarse hom of grps with canc met}
Let $G$ be a finitely normally generated set\=/group and $\varcrs{bw}$ the canonical coarsification obtained by taking the bi\=/invariant word metric associated with a finite normally generating set (Definition~\ref{def:p1:canonical coarse structure}).
The aim of this section is to provide examples of coarse homomorphisms for $(G,\varcrs{bw})$ and explore how they relate to other notions of quasimorphisms.

We begin by observing that for a finite rank free group $F_S$, all the examples of HS\=/quasimorphisms discussed in the previous section (Examples~\ref{exmp:p2:HS_qmorph local}, \ref{exmp:p2:HS_qmorph replacement} and \ref{exmp:p2:HS_qmorph wobbling}) also define coarse homomorphisms of $(F_S,\varcrs{bw})$. This is because the key tool used for proving that they are HS\=/quasimorphisms (Proposition~\ref{prop:p2:HS criterion}) can be extended to cover this setting as well. It is a consequence of the following.

\begin{lem}\label{lem:p2:criterion chom canonical crsificatn}
 Let $\crse G=(G,\CE)$ a coarsely connected coarse group and $f\colon F_S\to G$ be a function such that there exists $E\in\CE$ with
 \begin{enumerate}[(i)]
  \item $f(w_1w_2)\torel{E} f(w_1)\ast f(w_2)$ whenever $w_1,w_2$ and their concatenation $w_1w_2$ are reduced words;
  \item $f(w^{-1})\torel{E} f(w)^{-1}$ for every reduced word $w$;
 \end{enumerate}
 then $\crse f\colon (F_S,\varcrs{bw})\to\crse G$ is a coarse homomorphism. 
\end{lem}
\begin{proof}
 If we can show that $f(g_1 g_2)$ is uniformly close to $f(g_1)\ast f(g_2)$, it follows that condition $(i)$ of Lemma~\ref{lem:p1:functions are chom} is satisfied and hence we deduce that $\crse f\colon (F_S,\mincrs)\to\crse G$ is a coarse homomorphism. The fact that $\crse f\colon (F_S,\varcrs{bw})\to\crse G$ is a homomorphism as well follows from the assumption that $\crse G$ is coarsely connected and $\varcrs{bw}=\varcrs[grp]{fin}$ is the minimal connected coarsification.
 
 Choose reduced words $v_1,\ v_2$ representing the elements $g_1$, $g_2$. We can decompose them as $v_1=w_1u$ and $v_2=u^{-1}w_2$, so that the concatenation $w_1w_2$ is reduced. Now we have both
 \[
  f(g_1 g_2)=f(w_1w_2)\torel{E}f(w_1)\ast f(w_2)
 \]
 and
 \begin{align*}
  f(g_1)\ast f(g_2)&=
  f(w_1u)\ast f(u^{-1}g_2) \\
  &\torel{\CE} 
  f(w_1)\ast f(u)\ast f(u^{-1}) \ast f(w_2) \\
  &\torel{\Delta_{G}\ast E\ast\Delta_{G}}
  f(w_1)\ast f(u)\ast f(u)^{-1} \ast f(w_2)  \\
  &\torel{\CE}
  f(w_1)\ast f(w_2)
 \end{align*}
 for all $g_1,g_2\in F_S$.
\end{proof}

\begin{cor}\label{cor:p2:criterion chom canonical crsificatn}
 If $f\colon F_S\to G$ is as in Proposition~\ref{prop:p2:HS criterion}, then $\crse f\colon (F_S,\varcrs{bw})\to (G,\varcrs{bw})$ is a coarse homomorphism.
\end{cor}
\begin{proof}
 For any fixed finite set $B\subseteq G$, we know that the family $\braces{B\ast f(w_1)\ast B\ast f(w_2)\ast B\mid w_1,w_2\in F_S}$ is a controlled cover of $G$. Assume for simplicity that $e\in B$. If condition $(i)$ of Proposition~\ref{prop:p2:HS criterion} is satisfied, we see that the bounded set $B\ast f(w_1)\ast B\ast f(w_2)\ast B$ contains both $f(w_1w_2)$ and $f(w_1)\ast f(w_2)$. From this we deduce that hypothesis $(i)$ of Lemma~\ref{lem:p2:criterion chom canonical crsificatn} is satisfied as well. A similar argument of course works for condition $(ii)$.
\end{proof}

\begin{cor}
 Any replacement quasioutomorphism $f\colon F_S\to F_S$ defines a coarse automorphism $ \crse f \in\cAut(F_S,\varcrs{bw})$. Example~\ref{exmp:p2:HS_qmorph wobbling} gives a homomorphism of the wobbling group $W(\NN_{>0})\to \cAut(F_S,\varcrs{bw})$.
\end{cor}

\begin{rmk}\label{rmk:p2:proof of HS criterion}
 Corollary~\ref{cor:p2:criterion chom canonical crsificatn} easily implies Proposition~\ref{prop:p2:HS criterion}. In fact, if $\crse f\colon (F_S,\varcrs{bw})\to (G,\varcrs{bw})$ is a coarse homomorphism then $\crse f\colon (F_S,\varcrs{bw})\to (G,\varcrs{G\to\RR})$ is a coarse homomorphism a fortiori. The fact that $\crse f\colon (F_S,\varcrs{F_S\to\RR})\to (G,\varcrs{G\to\RR})$ is also a coarse homomorphism follows as usual from the definition of $\varcrs{G\to\RR}$ and the fact that coarse homomorphisms into $(\RR,\CE_{\abs{\mhyphen}})$ are in correspondence with $\RR$\=/quasimorphisms. See also Remark~\ref{rmk:p2:corse homs give HS qmorph}.
\end{rmk}

\

The reader should not expect that every HS\=/quasimorphism $f\colon G\to G$ is controlled as a function $f\colon(G,\varcrs{bw})\to(G,\varcrs{bw})$. In fact, the coarse structure $\CE_{G\to\RR}$ is generally larger than $\varcrs{bw}$ (for instance, notice that $(F_S,\CE_{F_S\to\RR})$ is coarsely abelian while $(F_S,\varcrs{bw})$ is not). When this is the case, one may then easily find a representative $f$ for a coarse homomorphism $ \crse f \colon(G,\CE_{G\to\RR})\to(G,\CE_{G\to\RR})$ that is not controlled as a function $(G,\varcrs{bw})\to(G,\varcrs{bw})$.

On the contrary, it is often true that coarse homomorphisms of $G$ give rise to HS\=/quasimorphisms:

\begin{lem}
 Let $G,\ H$ be set\=/groups and $(G,\CE)$ any coarsification of $G$. If $\crse f \colon (G,\CE)\to(H,\varcrs[grp]{fin})$ is a coarse homomorphism, then $f$ is an HS\=/quasimorphism.
 Any two representatives $f_1,f_2\in \crse f$ are equivalent HS\=/quasimorphisms.
\end{lem}
\begin{proof}
 Every $\RR$\=/quasimorphism $\phi\colon H\to \RR$ defines a coarse homomorphism $\crse \phi \colon (H,\varcrs[grp]{fin})\to(\RR,\CE_{\abs{\mhyphen}})$. Therefore, the composition $\crse {f\circ\phi} \colon (G,\CE)\to(\RR,\CE_{\abs{\mhyphen}})$ is itself a coarse homomorphism and hence an $\RR$\=/quasimorphism. The statement about equivalence is clear, because $\crse{ \phi\circ f}_1=\crse {\phi \circ f }= \crse {\phi \circ f}_2$ and hence $\abs{\phi\circ f_1-\phi\circ f_2}$ is bounded.
\end{proof}

\begin{rmk}\label{rmk:p2:corse homs give HS qmorph}
 In the above lemma, the coarse structure $(H,\varcrs[grp]{fin})$ can be replaced by any coarsification of $H$ where $\RR$\=/quasimorphisms are automatically controlled.
\end{rmk}

\begin{cor}\label{cor:p2:chom are HS_qmorph}
 If $G$ and $H$ are finitely normally generated set\=/groups, there is a natural map  $\chom((G,\varcrs{bw}),(H,\varcrs{bw}))\to Q(G,H)$ sending coarse homomorphism to equivalence classes of HS\=/quasimorphisms.
\end{cor}

Recall that for every set\=/group $G$ we have natural maps $\out(G)\to\cAut(G,\varcrs{bw})$ and $\out(G)\to\qOut(G)$. By definition, the latter coincides with the composition $\Out(F_S)\to \cAut(F_S,\varcrs{bw})\to Q(F_S)$. Since this is injective by Proposition~\ref{prop:p2:out_Fn embeds in HS}, we deduce the following:

\begin{cor}\label{cor:p2:outFn in cAut}
 The natural homomorphism $\Out(F_S)\to \cAut(F_S,\varcrs{bw})$ is an embedding.
\end{cor}

We already remarked that not every HS\=/quasimorphism $f$ gives rise to a coarse homomorphism.
However, our examples are rather unsatisfactory as such $f$ may still be equivalent to an HS\=/quasimorphism that is a coarse homomorphism.
In other words, we find the following a more natural problem.

\begin{problem}\label{prob:p2:cAut to HS_qmorph inj or surj}
 Let $G,\ H$ be finitely generated groups. Under what circumstances is the map 
 \[
  \chom((G,\varcrs{bw}),(H,\varcrs{bw}))\to Q(G,H)
 \]
 surjective? When is it injective?
\end{problem}

\

Thus far we compared coarse homomorphisms with HS\=/quasimorphism. However one may also consider other notions of $\RR$\=/quasimorphism. Following Hartnick--Schweitzer and Fujiwara--Kapovich, we use the following:

\begin{de}[\cite{HS,FK}]
 Let $G,H$ be set\=/groups. A function $f\colon G\to H$ is
 \begin{itemize}
  \item an \emph{Ulam quasimorphism}\index{quasimorphism!Ulam} if 
  \[
   f(g_1g_2)\in f(g_1)f(g_2)B \qquad \forall g_1,g_2\in G,
  \]
  \item a \emph{middle quasimorphism}\index{quasimorphism!middle} if 
  \[
   f(g_1g_2)\in f(g_1)Bf(g_2) \qquad \forall g_1,g_2\in G,
  \]
  \item a \emph{geometric quasimorphism}\index{quasimorphism!geometric} if 
  \[
   f(g_1g_2)\in f(g_1)Bf(g_2)B \qquad \forall g_1,g_2\in G,
  \]
  \item an \emph{algebraic quasimorphism}\index{quasimorphism!algebraic} if 
  \[
   f(g_1g_2)\in Bf(g_1)Bf(g_2)B \qquad \forall g_1,g_2\in G,
  \]
 \end{itemize}
 where $B\subseteq H$ is some fixed finite set.
\end{de}

If $G$ and $H$ are finitely normally generated, we have the following chain of containments:
\[
\begin{tikzcd}[remember picture,column sep=-7em]
     \braces{f\colon G\to H \mid\text{Ulam q.morph.}}& & \braces{f\colon G\to H \mid\text{middle q.morph.}}\\
     & \braces{f\colon G\to H \mid\text{geometric q.morph.}}& \\
     & \braces{f\colon G\to H \mid\text{algebraic q.morph.}}& \\
     & \braces{f\colon (G,\varcrs{bw})\to (H,\varcrs{bw}) \mid\text{controlled}}& \\
     & \braces{f\colon G\to H \mid\text{HS\=/q.morph.}}.& \\
\end{tikzcd}
\begin{tikzpicture}[overlay,remember picture]
\path (\tikzcdmatrixname-2-2) to node[midway,sloped]{$\subseteq$}(\tikzcdmatrixname-1-1);
\path (\tikzcdmatrixname-1-3) to node[midway,sloped]{$\supseteq$}(\tikzcdmatrixname-2-2);
\path (\tikzcdmatrixname-2-2) to node[midway,sloped]{$\subseteq$}(\tikzcdmatrixname-3-2);
\path (\tikzcdmatrixname-3-2) to node[midway,sloped]{$\subseteq$}(\tikzcdmatrixname-4-2);
\path (\tikzcdmatrixname-4-2) to node[midway,sloped]{$\subseteq$}(\tikzcdmatrixname-5-2);
(\tikzcdmatrixname-2-2);
\end{tikzpicture}
\]

\begin{rmk}
 In \cite{FK} Fujiwara and Kapovich use more general definition of quasimorphisms where they equip $H$ with some left\=/invariant metric $d$ and ask for the set $B$ to be bounded. Their definition coincides with ours if $d$ is proper and discrete. Considering more general metrics complicates the above diagram of inclusions.
\end{rmk}

As before, it is easy to see that these containments are usually strict by adding some perturbation. However, we are mainly interested in \emph{equivalence classes} of quasimorphisms, and more general quasimorphisms which are more tolerant to error should be considered only up to a more general equivalence relation. For example, two geometric quasimorphisms $f_1,\ f_2\colon G\to H$ should be considered equivalent if there is some finite $B'\subseteq H$ so that $f_1(g)\in f_2(g)B'$ for every $g\in G$. In contrast, algebraic quasimorphisms should be equivalent if there is some finite $B'\subseteq H$ so that $f_1(g)\in B'f_2(g)B'$ for every $g\in G$. 

It is now natural to extend Problem~\ref{prob:p2:cAut to HS_qmorph inj or surj} to this setting and ask when it is the case that the maps induced by the above inclusions are injective/surjective when quotienting by the relevant equivalence relation. A special instance of this question appears in \cite{HS}, where they ask whether there exist algebraic quasimorphisms that are not equivalent to any geometric quasimorphism.

It is a result of Fujiwara--Kapovich \cite{FK} that there are very few Ulam quasimorphisms $f\colon F_S\to F_S$. At the opposite end, the examples of Hartnick--Schweitzer show that $Q(F_S)$ is very rich. This shows that the composition of the maps
\[
 \braces{\text{Ulam q.morph.}}
 \to \braces{\text{geom. q.morph.}}/_\sim
 \to \braces{\text{alg. q.morph.}}/_\sim 
 \to \cAut(F_S,\varcrs{bw}) 
 \to Q(F_S) 
\]
is far from being surjective for the free group. However, it is still not clear whether surjectivity fails for all of the above. In general, we find that this area of research still presents a number of open problems.

\begin{rmk}
 We wish to conclude this chapter by remarking that we are not currently aware of any set\=/group $G$ that has an unbounded, coarsely connected coarsification $(G,\CE)$ so that $\cAut(G,\CE)$ is countable or finite.
\end{rmk}

\chapter{Spaces of Controlled Maps}
\label{ch:p2:spaces of controlled maps}
 The aim of this section is to study the space of controlled maps $(X,\CE)\to (Y,\CE)$. This is not to be confused with $\cmor(\crse{X,Y})$, as we are not identifying close functions. Ideally, we would like to define a natural coarse structure on this set. This is not possible, however there is a natural ``fragmentary'' coarse structure which will be useful when working with families of controlled maps.
 Some related ideas appear in \cite{dikranjan2019balleans,zava2019introduction,dikranjan2019hyperballeans}.
 
\section{Fragmentary Coarse Structures}
 A fragmentary coarse structure on a set $X$ is a collection of relations satisfying all the requirements of a coarse structure with the exception that it need not contain the diagonal:\footnote{%
 There is literature (\emph{e.g.}\ \cite{skandalis2002coarse}) where coarse structures are not required to contain the diagonal.
 However, this is rather rare, and we thus preferred to coin the name ``fragmentary coarse structures''.
 Also, we find it convenient that this way the the useful notion of ``fragment'' has a natural name.  
 }

\begin{definition}\label{def:p2:fragmentary coarse.structure}
 A collection $\mathcal{E}$ of relations on $X$ is a \emph{fragmentary coarse structure}\index{fragmentary!coarse!structure} if it is closed under taking subsets, finite unions and
\begin{enumerate}[(i)]
 \item $\{(x,x)\}\in\CE$ for every $x\in X$;
 \item if $E\in\CE$ then $\op{E}\in\CE$; 
 \item if $E_1,E_2\in\CE$ then $E_1\cmp E_2\in\CE$.
\end{enumerate}
The relations $E\in\CE$ are still called controlled sets. A set with a fragmentary coarse structure is a \emph{fragmentary coarse space} (shortened to \emph{frag\=/coarse space}).\index{fragmentary!coarse!space} 

A subset $Y\subseteq X$ is a \emph{fragment}\index{fragment} of the fragmentary coarse space $(X,\CE)$ if $\Delta_{Y}\in \CE$. The \emph{completely fragmented}\index{completely fragmented frag\=/coarse structure} fragmentary coarse structure on $X$ is generated  by the set of diagonal singletons 
\[
 \CE_{\rm Frag}\coloneqq \Bigbraces{\braces{(x_1,x_1),\ldots,(x_n,x_n)}\bigmid x_1,\ldots x_n\in X,\  n\in\NN} 
\]
(this is the minimal fragmentary coarse structure on $X$).
\end{definition}

\begin{rmk}
 If $(X,\CE)$ is a fragmentary coarse space and $Y\subseteq X$ is a fragment, then the restriction of $\CE$ to $Y$ is a genuine coarse structure and hence $(Y,\CE|_{Y})$ is a coarse space.
\end{rmk}

With the obvious modifications, basic properties and definitions for coarse structures hold true for fragmentary coarse structures as well.
For example, Lemma~\ref{lem:p1: generated coarse structure} becomes:

\begin{lem}
 A relation $F$ belongs to the fragmentary coarse structure generated by a family of relations $\braces{E_i\mid i\in I}$ if and only if it is contained in a finite composition $F_1\cmp\cdots\cmp F_n$ where each $F_j$ is a finite union of $E_i$ or $\op{E_i}$.\footnote{%
 In Lemma~\ref{lem:p1: generated coarse structure} the sets $F_j$ are equal to $E_i\cup \Delta_X$ or $\op{E_i}\cup \Delta_X$. Since $\Delta_X$ is not required to be in the generated fragmentary coarse structure, it makes sense that we do not need to add it to the $F_j$ anymore. This has the side effect that we cannot generally obtain finite unions of relations as subsets of compositions, and we are hence forced to let the $F_j$ be arbitrary finite unions.}
\end{lem}

Controlled maps are defined in the usual way. Note that if $(X,\CE)$ is a coarse space and $f\colon(X,\CE)\to(Y,\CF)$ is a controlled map to a fragmentary coarse space then the image $(f(X),\CF|_{f(Y)})$ is still a coarse space, hence $f(X)$ is a fragment of $Y$. 
Bounded sets and controlled partial coverings are also defined as usual. Notice that for any subset of a fragmentary coarse space $Y\subseteq (X,\CE)$ the partial covering $\pts{Y}$ is $\CE$\=/controlled if and only if $Y$ is a fragment of $(X,\CE)$.

The definition of closeness requires a modification, as it relies on the diagonal relation. 

\begin{de}\label{def:p2:closeness frag coarse}
 Two functions $f,g\colon (X,\CE)\to (Y,\CF)$ between fragmentary coarse spaces are \emph{fragmentary close}\index{fragmentary!close functions} (shortened to \emph{frag\=/close}) if  $f\times g(\Delta_Z)\in \CF$ for every fragment $Z\subseteq X$. 
\end{de}

Note that Definition~\ref{def:p2:closeness frag coarse} depends on the fragmentary coarse structure on the domain (more precisely, it depends on the set of fragments of $(X,\CE)$). This was not the case for the usual definition of closeness (Definition~\ref{def:p1:close.functions}), as it was assumed that $\Delta_{X}$ always belong to the coarse structure. It is still the case that frag\=/closeness is preserved under pre or post composition by controlled maps.
If $\Delta_{X}\in\CE$ (\emph{i.e.}\ $(X,\CE)$ is a coarse space) the two notions of closeness coincide.

Since the difference between closeness and frag\=/closeness might create some confusion, we will be zealous and add the adjective ``fragmentary'' to those notions that rely on closeness. For example, a \emph{fragmentary coarse map} (\emph{frag\=/coarse})\index{fragmentary!coarse!map} is an equivalence class of a controlled maps up to frag\=/closeness. Fragmentary coarse equivalence\index{fragmentary!coarse!equivalence} is defined analogously.

\begin{rmk}\label{rmk:p2:about subtleties in frag_coarse containments}
 In the sequel we will not use fragmentary analogues of most of the notions of Chapter~\ref{ch:p1:properties of coarse}, so we do not discuss them in detail. However, we do wish to point out that one should be a little careful when developing a theory of frag\=/coarse subspaces, images, intersections, etc.. The definition of coarse containment and asymptoticity of subsets (Definition~\ref{def:p1:asymptotic}) does not work well as it is, because it is not a reflexive property. Instead, one would have to adjust it as we did for closeness, \emph{i.e.}\ by only checking coarse containment on fragments of the subsets. 
 
 One more reason to do so is to preserve the correspondence between ``set\=/wise'' and ``categorical'' frag\=/coarse subspaces, see Appendix~\ref{sec:appendix:subobjects}. Since the categorical definition in terms of equivalence classes of monics relies on frag\=/closeness, it is certainly necessary to adapt accordingly the definition of frag\=/coarse containment.
\end{rmk}

We continue to use bold symbols to denote fragmentary coarse spaces $\crse X=(X,\CE)$ and frag\=/coarse maps $\crse f\colon\crse X\to\crse Y$. 
Note that if $\crse X$ and $\crse Y$ are two coarse spaces, then two functions $f,g\colon X\to Y$ are frag\=/close if and only they are close in the usual sense. In particular, the bold symbol $\crse f$ represents the same set of equivalent functions regardless of whether we think of $\crse X$ and $\crse Y$ as coarse spaces or frag\=/coarse spaces containing the diagonal.
We denote the set of frag\=/coarse maps $\crse f\colon\crse X\to\crse Y$ by $\fcmor(\crse X,\crse Y)$.

\begin{rmk}
 We can of course define the category \Cat{FragCrs}\nomenclature[:CAT]{\Cat{FragCrs}}{category of fragmentary coarse spaces} of fragmentary coarse spaces, and this category naturally includes the category \Cat{Coarse} of coarse spaces as a subcategory. 
Note that if $\crse X$ and $\crse Y$ are two coarse spaces then $\cmor(\crse{X,Y})=\fcmor(\crse{X,y})$.
 
 The main reason for introducing frag\=/coarse spaces is to have a language suitable to describe spaces of controlled maps. In categorical terms, the key extra features that \Cat{FragCrs} has over \Cat{Coarse} is that it is a Cartesian closed category and that \Cat{Coarse} is enriched over \Cat{FragCrs}. We relegate this discussion to Appendix~\ref{sec:appendix:frag coarse}.
\end{rmk}

\begin{warning}
 The next example shows that a coarse space can be frag\=/coarsely equivalent to a frag\=/coarse space that is not a coarse space. 
\end{warning}

\begin{exmp}\label{exmp:p2:frag bounded}
 It is easy to show that a frag\=/coarse space $(X,\CE)$ is frag\=/coarsely equivalent to the one\=/point coarse space $({\rm pt},\{{\rm pt}\})$ if and only if there exists a point $\bar x\in X$ such that for every fragment $Y\subseteq X$ the product relation $\{\bar x\}\times Y$ is in $\CE$ (we say that such a space is \emph{frag\=/bounded}).\index{fragmentary!bounded} Concretely, the set $\RR$ equipped with the fragmentary coarse structure $\CE=\braces{E\mid \exists t>0,\ E\subseteq [-t,t]\times[-t,t]}$ is frag\=/coarsely equivalent to $\tobj$, but $\Delta_\RR\notin\CE$.
\end{exmp}

There are instances when the diagonal relation is very useful. This is the case for the fundamental equality $E\otimes F= (E\otimes \Delta_{Y})\cmp (\Delta_{X} \otimes F)$ on $(X,\CE)\times (Y,\CF)$. Fortunately, this equality can be refined by writing
\begin{equation}\label{eq:p1:otimes as composition frag}
 E\otimes F= (E\otimes \Delta_{\pi_1(F)})\cmp (\Delta_{\pi_2(E)} \otimes F),
\end{equation}
where $\pi_1$ and $\pi_2$ are the projection to the first and the second coordinate. Importantly, if $E$ is a relation in a fragmentary coarse structure $\CE$ on $X$, then the projections $\pi_1(E)$ and $\pi_2(E)$ are fragments of $X$, so that both terms on the RHS of Equation~\eqref{eq:p1:otimes as composition frag} belong to $\CE\otimes\CF$. We prove a slightly more refined result (which will be useful later):

\begin{lem}\label{lem:p2:projections are fragments}
 Let $\CE$ be a fragmentary coarse structure. Then for every $E\in\CE$ there exists a $\tilde E\in \CE$ such that $E\subseteq \tilde E$ and $\Delta_{\pi_1(\tilde E)\cup\pi_2(\tilde E)}\subseteq \tilde E$. In particular, for every $E\in\CE$ the union $\pi_1(E)\cup\pi_2(E)$ is a fragment of $X$.
\end{lem}
\begin{proof}
 Note that $\pi_1(E_1\cmp E_2)\subseteq\pi_1(E_1)$ and $\pi_2(E_1\cmp E_2)\subseteq\pi_2(E_2)$. Also note that
\[
 \Delta_{\pi_1(E)}\subseteq E\cmp\op{E} \quad \text{ and }\quad 
 \Delta_{\pi_2(E)}\subseteq \op E\cmp {E}.
\]
We may then let 
\[
 \tilde E\coloneqq E\cup (E\cmp \op E)\cup (\op E\cmp E)\cup \op E,
\]
and we are done because $\tilde E$ belongs to $\CE$,
\[
 \pi_1(\tilde E)=\pi_2(\tilde E)= \pi_1(E)\cup\pi_2(E),
\]
and $\Delta_{\pi_1(E)\cup\pi_2(E)}\subseteq \tilde E$.
\end{proof}

\begin{cor}\label{cor:p2:frag_controlled iff controlled on fragments}
 A function between fragmentary coarse spaces $f\colon(X,\CE)\to(Y,\CF)$ is controlled if and only if its restriction to every fragment of $X$ is controlled.
\end{cor}

We can define $(X,\CE)$\=/equi controlled families of maps as in Section~\ref{sec:p1:equi controlled.maps}. That is, an $X$\=/indexed family of controlled functions $f_x\colon (Y,\CF)\to (Z,\CD)$ is \emph{$(X,\CE)$\=/equi controlled}\index{fragmentary!equi controlled} if and only if the induced map $f\colon (X,\CE)\times(Y,\CF)\to(Z,\CD)$ is controlled. 

\begin{warning}
 Recall that a $I$\=/indexed family of functions $f_i\colon(Y,\CF)\to (Z,\CD)$ is equi controlled if they are $(I,\mincrs)$\=/equi controlled, where $\mincrs=\{\Delta_I\}$ as usual. We will use the same nomenclature even if $Y$ and $Z$ are frag\=/coarse spaces.
 In particular, note that if $(X,\CE)$ is a frag\=/coarse space then a family of $(X,\CE)$\=/equi controlled functions need not be equi controlled. 
\end{warning}

The important observation characterizing controlled functions $(X,\CE)\times (Y,\CF)\to (Z,\CD)$ as those functions whose left and right sections are equi controlled (Lemma~\ref{lem:p1:equi controlled.sections.iff.controlled}) cannot hold for frag\=/coarse spaces.
However, since Equation~\eqref{eq:p1:otimes as composition frag} holds true, we can recover the following fact:

\begin{cor}
 Let $(X,\CE),(Y,\CF)$ and $(Z,\CD)$ be fragmentary coarse spaces. A function $f\colon (X,\CE)\times (Y,\CF)\to (Z,\CD)$ is controlled if and only if for every pair of fragments $A\subseteq X$ and $B\subseteq Y$ both families $\braces{{}_xf\colon(Y,\CF)\to (Z,\CD)\mid x\in A}$ and $\braces{f_y\colon(X,\CE)\to (Z,\CD)\mid y\in B}$ are equi controlled.
\end{cor}

\section{The Fragmentary Coarse Space of Controlled Maps} \label{sec:p2:controlled_frag_crs} 
As announced, we introduced fragmentary coarse spaces because they are well\=/suited for studying spaces of controlled maps. We begin with  lemmas and definitions, as we need them to introduce the notation. The heuristics behind these definitions are explained in Remark~\ref{rmk:p2:heuristic on fragmentary coarse actions}.
Let $\crse X=(X,\CE)$ and $\crse Y=(Y,\CF)$ be two fragmentary coarse spaces. With a slight abuse of notation, we denote the set of controlled maps $(X,\CE)\to(Y,\CF)$ by $\ctrl{X}{Y}$ (it would be more correct to write $\ctrl{(Y,\CF)}{(X,\CE)}$). We see this as a subset of the set $Y^{X}$ of all the functions from $X$ to $Y$. Given $E\in\CE$ and $F\in \CF$, we denote 
\[
 F^E\coloneqq\bigbraces{(f_1,f_2)\bigmid f_1,f_2\in \ctrl{X}{Y},\ f_1\times f_2(E)\subseteq F}.
\]
The set $F^E$\nomenclature[:R]{$F^E$}{relation on controlled functions} is a relation on the set of controlled functions from $(X,\CE)$ to $(Y,\CF)$. We may also rewrite it as:
\[
 F^E\coloneqq\bigbraces{(f_1,f_2)\bigmid f_1,f_2\in \ctrl{X}{Y},\ f_1(e_1)\torel{F}f_2(e_2)\ \forall e_1\torel{E} e_2}.
\]
Clearly, $F_1\subseteq F_2$ implies $F_1^E\subseteq F_2^E$ and $E_1\subseteq E_2$ implies $F^{E_1}\supseteq F^{E_2}$. It is also easy to verify that $F^E$ satisfies the following properties:
\begin{enumerate}[(i)]
  \item $F_1^{E}\cup F_2^{E}\subseteq (F_1\cup F_2)^{E}$;
  \item $\op{(F^E)}=(\op{F})^{(\op{E})}$.
\end{enumerate}

For any fixed $E\in\CE$, we let $\CF^E$ denote the set of relations $C\subset \ctrl{X}{Y}\times\ctrl{X}{Y}$ such that $C\subseteq F^E$ for some $F\in\CF$. In other words, $\CF^E$ is the collection of relations obtained by making $\bigcup_{F\in\CF}F^E$ closed under taking subsets.

The following observation is important.

\begin{lem}\label{lem:p2:relations in exponential are equi controlled}
 Let $E\in\CE$ be a relation such that $\Delta_{\pi_1(E)\cup\pi_2(E)}\subseteq E$. Then, for every relation $D\in\CF^E$ there exists a $\tilde F\in\CF$ such that $f\times f(E)\subseteq \tilde F$ for every $f\in\pi_1(D)\cup\pi_2(D)$.
 In other words, there exists $\tilde F$ such that $\Delta_{\pi_1(D)\cup\pi_2(D)}\subseteq \tilde F^E$.
\end{lem}

\begin{proof}
 Fix $D\in\CF^E$. By hypothesis, there is $F\in\CF$ such that $f\times g(E)\subseteq F$ for every pair $(f,g)\in D$. Let $f\in\pi_1(D)$, by definition there is $g\in \ctrl{X}{Y}$ such that $(f,g)\in D$. Note that
 \begin{equation*}\label{eq:p2:fragmentary coarse structure of functions}
  f\times f(E)
  \subseteq \paren{ f\times g (E) }\circ \paren{ g\times f(\Delta_{\pi_2(E)})}
  = \paren{f\times g (E)}\circ \op{\paren{f\times g(\Delta_{\pi_2(E)})} }.
 \end{equation*}
 Since $\Delta_{\pi_2(E)}\subseteq E$, the pair $(f,g)$ is in $F^{\Delta_{\pi_2(E)}}$. That is, $f\times g(\Delta_{\pi_2(E)})\subseteq F$. It follows that $f\times f(E)\subseteq F\cmp \op F$.
 
 An analogous argument shows that $g\times g(E)\subseteq \op F\circ F$ whenever $g\in \pi_2(D)$. To conclude the proof it is enough to let $\tilde F\coloneqq\paren{\op F\circ F}\cup \paren{F\circ\op F}$.
\end{proof}

\begin{lem}\label{lem:p2:exponential is frag coarse struct}
 Let $E\in\CE$ be a relation such that $\Delta_{\pi_1(E)\cup\pi_2(E)}\subseteq E$. Then the collection $\CF^E$ is a fragmentary coarse structure.
\end{lem}
\begin{proof}
 The collection $\CF^E$ is closed under taking subsets by construction. Properties (i) and (ii) show that it is also closed under finite unions and symmetries. To prove that $\CF^E$ is a fragmentary coarse structure, it remains to verify is that it is closed under composition.

 Fix $D_1,D_2\in\CF^E$, then there exist $F_1,F_2\in\CF$ such that $f_1\times f_2(E)\subseteq F_1$ and $g_1\times g_2(E)\subseteq F_2$ for every $f_1\torel{D_1}f_2$ and $g_1\torel{D_2}g_2$. 
 By definition, $f\torel{D_1\cmp D_2} g$ if and only if there exists $h\colon X\to Y$ such that $f\torel{D_1} h$ and $h\torel{ D_2} g$. It follows that for every $(f,g)\in D_1\cmp D_2$ we have
 \[
  f(e_1)
  \torel{f\times h(E)} h(e_2) 
  \torel{\op{\paren{h\times h(E)}} } h(e_1) 
  \torel{h\times g(E)} g(e_2)
  \qquad \forall e_1\torel{E}e_2
 \]
 hence
 \[
  f(e_1)
  \torel{F_1} h(e_2) 
  \torel{\op{\paren{h\times h(E)}} } h(e_1) 
  \torel{F_2} g(e_2)
  \qquad \forall e_1\torel{E}e_2.
 \]
 Using Lemma~\ref{lem:p2:relations in exponential are equi controlled} we can also find a $\tilde F\in\CF$ that contains $\op{\paren{h\times h(E)}}$ for every $h$. It then follows that $D_1\cmp D_2$ is in $\CF^E$. 
\end{proof}

It follows from Lemma~\ref{lem:p2:projections are fragments} that the intersection $\bigcap\braces{\CF^E\mid E\in\CE}$ is equal to the intersection $\bigcap\braces{\CF^E\mid E\in\CE,\ \Delta_{\pi_1(E)\cup\pi_2(E)}\subseteq E}$. By Lemma~\ref{lem:p2:exponential is frag coarse struct}, the latter is an intersection of fragmentary coarse structures and hence it is itself a fragmentary coarse structure. 

\begin{de}\label{def:p2:exponential fcrse strucuture}
 The \emph{exponential fragmentary coarse structure}\index{exponential frag\=/coarse structure} on $\ctrl{X}{Y}$ is the intersection
 \[
  \CF^\CE\coloneqq \bigcap\braces{\CF^E\mid E\in\CE}. \nomenclature[:CE1]{$\CF^\CE\coloneqq \bigcap\braces{\CF^E\mid E\in\CE}$}{exponential fragmentary coarse structure}
 \]
 Given $\crse X=(X,\CE)$ and $\crse Y=(Y,\CF)$, the \emph{fragmentary coarse space of controlled functions from $\crse X$ to $\crse Y$} (or, \emph{exponential frag\=/coarse space}) is the fragmentary coarse space $\crse{Y^{X}}\coloneqq (\ctrl{X}{Y},\CF^\CE)$.\nomenclature[:COS]{$\crse{Y^{X}}$}{frag-coarse space of controlled functions}
\end{de}

Note that if $D\in\CF^\CE$ and $f\in\ctrl{X}{Y}$, then all the functions in $D(f)$ are frag\=/close to $f$ (they are in fact ``equi frag\=/close'' to it). Furthermore, Lemma~\ref{lem:p2:relations in exponential are equi controlled} has an immediate consequence on the exponential fragmentary coarse structure:

\begin{cor}\label{cor:p2:fragments are equi controlled}
For every relation $D\in\CF^\CE$ the functions in $\pi_1(D)\cup \pi_2(D)$ are equi controlled.
\end{cor}
\begin{proof}
 Given $E\in \CE$, choose $E\subseteq \tilde E\in\CE$ with $\Delta_{\pi_1(\tilde E)\cup\pi_2(\tilde E)}\subseteq \tilde E$. Since $D\in \CF^{\tilde E}$, by Lemma~\ref{lem:p2:relations in exponential are equi controlled} there is an $F\in\CF$ such that
 \(
  f\times f(E)\subseteq F
 \)
 for every $f\in\pi_1(D)\cup \pi_2(D)$. 
\end{proof}

\begin{exmp}
 Let $(X,d_{X}),(Y,d_{Y})$ be metric spaces, and let $\CE=\CE_{d_X}$ and $\CF=\CE_{d_Y}$ be the corresponding metric coarse structures. Then $f\colon X\to Y$ is controlled if and only if there is a non\=/decreasing control function $\rho_+\colon [0,\infty)\to[0,\infty)$ with 
 \(
  d_{Y}(f(x),f(x'))\leq \rho_+(d_{X}(x,x'))
 \)
 for every $x,x'\in X$. 
 We write explicitly in terms of the metrics $d_X$, $d_Y$ what it means for a set of controlled functions $\braces{g_i\in \ctrl{X}{Y}\mid i\in I}$ to be contained in a $\CF^\CE$\=/bounded neighborhood of $f\in \crse{Y^{X}}$. Namely, this is the case if and only if for every $r>0$ there is an $R>0$ large enough so that 
 \[
  g_i\bigparen{B_{(X,d_{X})}(x;r)}\subseteq B_{(Y,d_{Y})}(f(x);R)
 \]
 for every $x\in X$, $i\in I$. Notice that this also implies that there is a $\rho'_+\colon [0,\infty)\to[0,\infty)$ such that 
 \(
  d_{Y}(g_i(x),g_i(x'))\leq \rho'_+(d_{X}(x,x'))
 \)
 for every $i\in I$ and $x,x'\in X$ (\emph{i.e.}\ the $g_i$ are equi controlled).  
 One can also check that a set $D=\braces{(g_i,f_i)\mid i\in I}$ is in $\CF^\CE$ if and only if there exists a ``uniform'' control function $\rho_+\colon [0,\infty)\to[0,\infty)$ with 
 \[
  \max\bigbraces{d_{Y}(g_i(x),g_i(x'))\,,\, d_{Y}(f_i(x),f_i(x'))}\leq \rho_+(d_{X}(x,x'))
 \]
 for every $i\in I$, $x,x'\in X$, and a constant $R\geq 0$ such that $d_{Y}(g_i(x),f_i(x))\leq R$ for every $i\in I$, $x\in X$. 
\end{exmp}

\begin{exmp}
 Given two controlled functions $f,g\in\ctrl{X}{Y}$, in order for the singleton $\braces{(f,g)}$ to be in $\CF^\CE$ it is necessary that $f$ and $g$ be frag\=/close. The converse is also true: for every $E\in \CE$ the projection $\pi_2(E)$ is a fragment of $X$ and hence $f\times g(\Delta_{\pi_2(E)})\in\CF$. Since $f\times f(E) \in \CF$ by assumption, it follows that $f\times g(E)\subseteq \paren{f\times f(E)}\cmp\paren{f\times g(\Delta_{\pi_2(E)})}$ is in $\CF$. 
 This show that there is a natural bijection between $\fcmor(\crse X,\crse Y)$ and the set of coarsely connected components of $\crse{Y^{X}}$.
\end{exmp}

\begin{lem}\label{lem:p2:composition of exponential}
 Given three coarse spaces $\crse {X,Y,Z}$, the composition of functions induces a coarse map $\crse{Z^{Y}\times Y^{X}}\to \crse{Z^{X}}$. 
\end{lem}
\begin{proof}
 Let $\crse X=(X,\CE)$, $\crse Y=(Y,\CF)$ and $\crse Z=(Z,\CD)$. We need to check that the composition is a controlled map $\ctrl{Y}{Z}\times \ctrl{X}{Y}\to\ctrl{X}{Z}$, \emph{i.e.}\ that for every pair of controlled sets $C_1\in \CD^\CF$ and $C_2\in \CF^\CE$ the product $C_1\otimes C_2$ is sent to a $\CD^\CE$\=/controlled set. 
 
 Fix any $E\in \CE$. Then there exists $F\in \CF$ with $C_2\subseteq F^E$, and there is a $D\in\CD$ with $C_1\subseteq D^F$. The composition sends $\braces{(f_1,f_2),(g_1,g_2)}\in C_1\otimes C_2$ to $(f_1\circ f_2, g_1\circ g_2)$ and we see that
 \[
  \bigparen{(f_1\circ f_2) \times (g_1\circ g_2) }(E)
  = (f_1\times g_1)\bigparen{(f_2\times g_2)(E)}
  \subseteq (f_1\times g_1)(F)
  \subseteq D.
 \]
 That is, $(\circ\times\circ)(C_1\otimes C_2)\subseteq D^E$. Since $E$ is arbitrary, this concludes the proof.
\end{proof}

\begin{rmk}\label{rmk:p2:heuristic on fragmentary coarse actions}
 The following is a long heuristic remark providing motivation for the definition of $\CF^\CE$.
Informally, the space $\crse{Y^{X}}$ is designed to describe coarse properties of families of controlled functions.  
The following example motivates why we define $\crse{Y^{X}}$ instead of simply using $\cmor(\crse{X,Y})$. For every $t\in\RR$ let $f_t\colon \RR\to\RR$ be the translation by $t$. 
For each fixed $t$, the function $f_t$ is close to the identity, and hence $\crse f_t=\cid_\RR \in\cmor(\RR,\RR)$. That is, the set $\braces{\crse f_t \mid t\in\RR}$ seen as a subset of $\cmor(\RR,\RR)$ is the singleton $\{\cid_\RR\}$. This is rather unsatisfactory because the functions $f_t$ have larger and larger displacement as $\abs{t}$ grows: this non\=/trivial piece of information about the coarse geometry of the family of maps $f_t$ is completely lost in $\cmor(\RR,\RR)$. The key point here is that \emph{each function} $f_t$ is coarsely trivial, but the \emph{family} of the $f_t$'s is not.  In order to describe `coarse properties' of families of controlled functions it is ill advised to immediately quotient out close functions. Rather, it is more appropriate to introduce some coarse structure on the set $\ctrl{X}{Y}$.\footnote{%
By now it should be clear that coarse spaces can be seen as non\=/destructive quotients. Namely, when we equip $X$ with a coarse structure $\CE$ we are saying that close points should be thought of as being equal. However, we do not take the actual quotient by this equivalence relation, because it is useful to keep track of how unbounded \emph{families} of points behave. 
}
The main difficulty is that coarse structures are not flexible enough to describe the current situation, hence we need to pass to fragmentary coarse spaces.

The choice of $\CF^\CE$ might appear arbitrary at first sight, but it is actually natural. To begin with, note that for every $F\in \CF$ the pairs of controlled functions $f,g\colon X\to Y$ such that $(f,g)\in F^{\Delta_{X}}$ are---by definition---pairwise ``uniformly close''. In other words, two families of controlled functions $(f_i)_{i\in I},\ (g_i)_{i\in I}\subseteq \ctrl{X}{Y}$ such that $(f_i,g_i)\in F^{\Delta_{X}}$ for every $i\in I$ should be ``indistinguishable'' from the coarse point of view. Thus the relation $F^{\Delta_{X}}$ should belong to the (fragmentary) coarse structure that we wish to define on $\ctrl{X}{Y}$. A first attempt would be to equip $\ctrl{X}{Y}$ with the coarse structure generated by the sets $F^{\Delta_{X}}$ with $F\in\CF$. However, we will see that this is not quite the correct thing to do. 

One reason why $\CF^{\Delta_{X}}$ does not capture the essence of the ``coarse geometry of functions'' is that it behaves poorly under composition. Namely, say that we are given two ``coarsely indistinguishable'' families of controlled functions $f_i,g_i\colon X\to Y$. If we are given another pair of ``coarsely indistinguishable'' families of controlled functions $f_i',g_i'\colon Z\to X$, we would expect the compositions $f_i\circ f_i'$ and $g_i\circ g_i'$ to be  ``coarsely indistinguishable''. However, we see that the assumptions 
\[\braces{(f_i,g_i)\mid i\in I}\in \CF^{\Delta_{X}} \quad \text{and}\quad \braces{(f_i',g_i')\mid i\in I}\in \CE^{\Delta_Z}\]
are not enough to imply that $\braces{(f_i\circ f_i',g_i\circ g_i')\mid i\in I}$ belongs $\CF^{\Delta_{X}}$. For the latter to hold we need to assume that $\braces{(f_i,g_i)\mid i\in I}\in \CF^{E}$ where $E\in \CE$ is a controlled set so that $\braces{(f_i',g_i')\mid i\in I}\in E^{\Delta_{X}}$. So to obtain a coarse geometric theory of controlled functions that is compatible with compositions we are forced to use the fragmentary coarse structures $\CF^E$.

The above also explains why we consider fragmentary coarse structures. The requirement that the theory be well behaved under composition implies that all the families of functions that we work with must be equi controlled (Corollary~\ref{cor:p2:fragments are equi controlled}). If the coarse space $\crse Y$ is unbounded, we may find a family of controlled functions $f_i\colon X\to Y$ that are not equi controlled and hence the set $\braces{(f_i,f_i)\mid i\in I}$ does not belong to $\CF^\CE$. That is, $\braces{f_i\mid i\in I} $ is not a fragment of $\crse{Y^{X}}$. 
Concretely, if $X=[-1,1]$ and $Y=\RR$ are equipped with the metric coarse structure then the set of maps $m_t(x)\coloneqq tx$ with $t\in [0,\infty)$ is not a fragment of $\RR^{[-1,1]}$ because they are not uniformly Lipschitz. This shows that even well\=/behaved coarse spaces give rise to fragmentary coarse spaces.
\end{rmk}

\

There is a natural \emph{evaluation function}\index{evaluation function} $\ev\colon \ctrl{X}{Y}\times X\to Y$ defined by $\ev(f,x)\coloneqq f(x)$.
Furthermore, every function $f\colon Z\times X\to Y$ admits a \emph{transposed function}\index{transposed function} $\lambda f\colon Z\to Y^{X}$ defined by $\lambda f(z)\coloneqq f(z,\variable)$. Note that $f=\ev\circ(\lambda f\times \id_{X})$.
The following is the main result we prove about fragmentary coarse spaces of controlled functions.
Once the formalism is set, the proof is actually rather simple: most of the ideas already appear in the examples above. 

\begin{thm}\label{thm:p2:exponential}
 Let $\crse X=(X,\CE)$, $\crse Y=(Y,\CF)$, $\crse Z=(Z,\CD)$ be fragmentary coarse spaces. Then,
 \begin{enumerate}[(1)]
  \item the evaluation map $\ev\colon (\ctrl{X}{Y},\CF^\CE)\times (X,\CE)\to (Y,\CF)$ is controlled;
  \item a function $h\colon (Z,\CD)\times (X,\CE)\to (Y,\CF)$ is controlled if and only if its transpose $\lambda h\colon (Z,\CD)\to (\ctrl{X}{Y},\CF^\CE)$ is controlled;
  \item if $h$ is controlled and $\bar h\colon (Z,\CD)\to (\ctrl{X}{Y},\CF^\CE)$ is a function such that $h$ is frag\=/close to $\ev\circ(\bar h\times \id_{X})$, then $\bar h$ must be frag\=/close to $\lambda h$ (and it is hence controlled);
  \item the transposition defines a canonical bijection $\fcmor(\crse Z\times \crse X, \crse Y)\longrightarrow\fcmor(\crse Z,\crse{Y^{X}}) $.
 \end{enumerate}

\end{thm}

\begin{proof}
 (1): The product fragmentary coarse structure on $\ctrl{X}{Y}\times X$ is generated by controlled sets of the form $C\otimes E$ with $C\in\CF^\CE$, $E\in\CE$. 
 It is thus enough to verify that $\ev\times\ev(C\otimes E)\in\CF$. 
 By the definition of $\CF^\CE$, there exists $F\in \CF$ such that $f\times g(E)\subseteq F$ for every $(f,g)\in C$. On the other hand, we have
 \[
  \ev\times\ev(C\otimes E)=\bigcup\braces{f\times g(E)\mid (f,g)\in C}\subseteq F
 \]
 and so we are done.
 
 (2): By construction, $h=\ev\circ(\lambda h\times \id_{X})$. If $\lambda h$ is controlled, then $h$ is composition of controlled maps and hence controlled. 
 Vice versa, assume that $h$ is controlled, then for every $D\in\CD$ and $E\in \CE$ we see:
 \begin{align*}
  \lambda h\times \lambda h(D)
  &=\bigbraces{(h(z_1,\variable),h(z_2,\variable))\mid (z_1,z_2)\in D} \\
  &\subseteq\bigbraces{(f,g)\in \ctrl{X}{Y}\mid f\times g(E)\subseteq h\times h (D\otimes E)}  \\
  &=(h \times h(D\otimes E))^E.
 \end{align*}
 Since $h$ is controlled, $F\coloneqq h \times h(D\otimes E)$ is in $\CF$ and hence $\lambda h\times \lambda h(D)\in \CF^\CE$.

 (3): Assume that $\ev\circ(\bar h\times \id_{X})$ is frag\=/close to $h$, and fix any $E\in\CE$. We can assume that $\pi_1(E)$ is a fragment of $X$. Let $C$ be a fragment of $Z$, so that $\Delta_{C\times \pi_1(E)}=\Delta_{C}\otimes\Delta_{\pi_1(E)}$ is in $\CD\otimes\CE$. By hypothesis, there is an $F\in\CF$ such that $\bigparen{\paren{\ev\circ(\bar h\times \id_{X})}\times h}(\Delta_{C}\otimes\Delta_{\pi_1(E)})\subseteq F$. Then, for every $z\in C$ we have:
 \begin{align*}
  \paren{\bar h(z)\times \lambda h(z)}(E)
  &=\bigbraces{(\bar h(z)(x_1),h(z,x_2))\mid (x_1,x_2)\in E}   \\
  &\subseteq F\cmp \braces{(h(z,x_1),h(z,x_2))\mid (x_1,x_2)\in E}  \\
  &\subseteq F\cmp (h\times h(\Delta_C\otimes E))
 \end{align*}
 and hence $\bar h\times \lambda h(\Delta_C)\in\CF^\CE$ because $h$ is controlled.
 
 (4): 
 It follows from (3) that if $h,h'\colon (Z\times X,\CD\otimes\CE)\to(Y,\otimes \CF)$ are frag\=/close then $\lambda h$ and $\lambda h'$ are frag\=/close. This implies that the transposition $h\mapsto\lambda h$ defines a function 
 \[
  \lambda\colon\fcmor(\crse Z\times \crse X, \crse Y)\rightarrow\fcmor(\crse Z,\crse{Y^{X}}).
 \]
 Since $\ev$ is controlled by (1), composition with $\ev$ preserves frag\=/closeness. We thus obtain a function in the other direction
 \[
  \ev\circ(\variable\times\id_{X})\colon \fcmor(\crse Z,\crse{Y^{X}})\to \fcmor(\crse Z\times \crse X, \crse Y).
 \]
 By construction, $\lambda$ and $\ev\circ(\variable\times\id_{X})$ are inverse of one another (this is even true point\=/wise).  
\end{proof}

Together with Lemma~\ref{lem:p1:equi controlled.sections.iff.controlled}, item (2) of Theorem~\ref{thm:p2:exponential} has the following corollary.

\begin{cor}
 Fix a $Z$\=/indexed family of controlled functions ${}_zf\colon (X,\CE)\to (Y,\CF)$. Then:
 \begin{enumerate}
  \item it is equi controlled if and only if $f\colon (Z,\mincrs)\to(\ctrl{X}{Y},\CF^\CE)$ is controlled (\emph{i.e.}\ the subset $f(Z)\subseteq \ctrl{X}{Y}$ is a fragment);
  \item it is $(Z,\CD)$\=/equi controlled if and only if $f\colon (Z,\CD)\to (\ctrl{X}{Y},\CF^\CE)$ is controlled.
 \end{enumerate}
\end{cor}

\begin{convention}
 To keep true to our convention of using bold letters for frag\=/coarse maps, we will denote $[\lambda f]$ by $\crse{\lambda f}$. A more pedantic approach would be to always see $\lambda$ as a function $\lambda\colon\fcmor(\crse Z\times \crse X, \crse Y)\rightarrow\fcmor(\crse Z,\crse{Y^{X}})$ and hence write $\lambda(\crse \alpha)$ instead of $\crse{\lambda\alpha}$. 
\end{convention}

\section{Coarse Actions Revisited}
\label{sec:p2:coarse actions revisited} 
Depending on the context, it may be more convenient to define set\=/group actions as functions $G\times Y\to Y$ satisfying some commutative diagrams (the Action Diagrams), or as homomorphism of $G$ into the group of bijections $G\to \aut(Y)$.
The frag\=/exponentials provide us with a language to make sense of the latter approach in the context of coarse actions.

For later use, we record the following properties of $\lambda$.

\begin{lem}\label{lem:p2:compositions and transpose}
 The transposition satisfies the following identities.
 \begin{enumerate}
  \item Given frag\=/coarse maps  $\crse{f\colon A \to B}$ and $\crse{g\colon B\times X \to Y}$, consider the composition 
  \[
  \begin{tikzcd}
   \crse{A\times X} \arrow[r,"\crse{f\times \cid_{X}}"] & \crse{B\times X}\arrow[r,"\crse{ g}"] & \crse Y.   
  \end{tikzcd}
  \]
  The following commutes:
  \[ 
 \begin{tikzcd}
    \crse{A} \arrow[dr, "\crse{\lambda(g\circ (f\times \cid_{X}))}"] \arrow[d,swap, "\crse{f}"] & \\
    \crse{B} \arrow[swap]{r}{\crse{\lambda g}}  & \crse{Y^{X}}. 
\end{tikzcd} 
\] 
  \item Given frag\=/coarse maps  $\crse{f\colon A\times Y \to Z}$ and $\crse{g\colon B\times X\to Y}$, consider the composition 
  \[
  \begin{tikzcd}
   \crse{A\times B\times X} \arrow[r,"\crse{\cid_{A}\times  g}"] & \crse{A\times Y}\arrow[r,"\crse{ f}"] & \crse Z.   
  \end{tikzcd}
  \]
   The following commutes:
  \[
  \begin{tikzcd}
    \crse{A\times B} \arrow[dr,swap, "\crse{\lambda(f\circ (\cid_{A}\times g))}"] \arrow[r, "\crse{\lambda f\times \lambda g}"] & \crse{Z^{Y} \times Y^{X}}     \arrow{d}{\ccirc} \\
    & \crse{Z^{X}}. 
\end{tikzcd} 
\] 
 \end{enumerate} 
\end{lem}
\begin{proof}
 Both statements are true pointwise (\emph{i.e.}\ the relevant equalities are already true in \Cat{Set}).
\end{proof}

Let $\crse X$ be any coarse space. By Lemma~\ref{lem:p2:composition of exponential}, the composition of functions defines a frag\=/coarse map $\ccirc\colon \crse{Y^{Y} \times Y^{Y} \to Y^{Y}}$. Since $\circ$ is associative and $\id_{Y}\in Y^{Y}$ is a unit, it follows that $\crse{Y^{Y}}$ is a \emph{frag\=/coarse monoid} (here we are using the obvious definition of monoid object in the category \Cat{FragCrs}). Recall that a homeomorphism between monoids is a map that sends the unit to the unit and commutes with the operation. It is then simple to prove the following.

\begin{prop}\label{prop:p2:actions as monoid homomorphisms}
 Let $\crse G$ be a coarse group. A coarse map $\crse{\alpha \colon G \times Y\to Y}$ is a coarse action if and only if its transpose $\crse {\lambda\alpha\colon G\to Y^{Y}}$ is a coarse homomorphism of frag\=/coarse monoids.
\end{prop}
\begin{proof}
 By Theorem~\ref{thm:p2:exponential} we know that the transposition $\lambda$ gives a bijection between $\cmor(\crse{G\times G\times Y},\crse Y)$ and $\cmor(\crse{G\times G,Y^{Y}})$. Therefore, the diagram
\[ \begin{tikzcd}
    \crse{G}\times \crse G\times \crse Y \arrow[r, "\cid_{\crse G} \times \crse \alpha"] \arrow[d, "\cop\times \cid_{\crse Y}"] & \crse{G}\times \crse Y      \arrow[d, "\crse \alpha"]\\
    \crse{G}\times \crse Y \arrow[r, "\crse \alpha"]& \crse Y 
\end{tikzcd} 
\]
commutes if and only if $\crse{\lambda(\alpha\circ (\cid_{G}\times \alpha))}=\crse{\lambda(\alpha\circ (\cop\times \cid_{Y}))}$. By Lemma~\ref{lem:p2:compositions and transpose}, this happens if and only if the diagram
\[ \begin{tikzcd}
    \crse{G}\times \crse G \arrow[r, "\crse{\lambda\alpha\times\lambda\alpha}"] \arrow[d, "\cop"] & \crse{Y^{Y}\times Y^{Y}}      \arrow{d}{\ccirc} \\
    \crse{G} \arrow[r, "\crse{\lambda \alpha}"]& \crse{Y^{Y}} 
\end{tikzcd} 
\]
commutes.

Similarly, we see that the diagram 
\[
\qquad\qquad 
\begin{tikzcd}
    \crse{Y}\arrow{r}{(\cunit,\cid_{\crse Y})} 
    \arrow[dr, swap, "\cid_{\crse Y}"] & \crse{G}\times \crse Y      
    \arrow[d, "\crse \alpha"]\\
    & \crse Y 
\end{tikzcd} 
\] 
commutes if and only if $\crse{\lambda\alpha(\cunit_{G})=\cid_{Y}}$.
\end{proof}

The above characterization of coarse actions is useful \emph{e.g.}\ to define coarse faithfulness:

\begin{de}\label{def:p2:coarsely faithful}
 A coarse action $\crse{\alpha\colon G\curvearrowright Y}$ is \emph{coarsely faithful}\index{coarsely!faithful} if $\crse{\lambda\alpha\colon G\to Y^{Y}}$ is proper (\emph{i.e.}\ $\cker(\crse{\lambda\alpha})=\cunit_{\crse G}$).
\end{de}

\begin{exmp}\label{exmp:p2:left multiplication is faithful}
 The coarse action by left multiplication $\crse{\cop\colon G\curvearrowright G}$ of a coarse group $\crse G=(G,\CE)$ on itself is always faithful. To see this, let $B\subseteq G$ be a subset whose image under $\lambda{\ast}$ is a bounded neighborhood of $\id_{G}$ in $(\ctrl{G}{G},\CE^\CE)$. We may assume that $e_{G}\in B$. 
 Since $\CE^\CE\subseteq \CE^{\{(e_G,e_G)\}}$, it follows that there is a controlled entourage $E\in\CE$ such that 
 \[
  \paren{(\lambda\ast)(g)}(e_G)\torel{E} \id_G(e_G) \qquad \forall g\in B.
 \]
 That is, $g\ast e_G\torel{E} e_G$ for every $g\in B$. Since $e_G$ is the coarse unit in $G$, it follows that $B$ is a $\CE$\=/bounded neighborhood of $e_G$. 
\end{exmp}

It is of course possible to find an equivalent formulation of Definition~\ref{def:p2:coarsely faithful} without using $\crse{\lambda\alpha}$. Namely, $\crse \alpha \colon(G,\CE)\curvearrowright(Y,\CF)$ is coarsely faithful if and only if every set $B\subseteq G$ such that $\braces{\alpha(B\times \{y\})\mid y\in Y}$ is a controlled partial covering of $Y$ must be $\CE$\=/bounded. We find Definition~\ref{def:p2:coarsely faithful} cleaner.

\section{Frag-coarse Groups of Controlled Transformations}
\label{sec:p2:frag coarse groups of transformations}
Proposition~\ref{prop:p2:actions as monoid homomorphisms} falls short of the characterization of set\=/group actions as homomorphisms $G\to\aut(Y)$, in that we are not considering a coarse analog of the group $\aut(Y)$ but only of the monoid $Y^{Y}=\braces{\text{functions from $Y$ to itself}}$. More precisely, in \Cat{Set} it is trivially true that the subset of invertible functions $Y\to Y$ is itself a set\=/group $\aut(Y)<Y^{Y}$ and that for any set\=/group action $\alpha\colon G\curvearrowright Y$ the image of the homomorphism $\lambda\alpha\colon G\to Y^{Y}$ is contained in $\aut(G)$. In the context of (frag\=/)coarse spaces we still lack a good analogue of $\aut(Y)$.

Let $\crse Y=(Y,\CF)$ be a frag\=/coarse space. For any set\=/monoid, $M$, the subset of invertible elements $M^{*}\subseteq M$ is a set\=/group. Since $\crse{Y^{Y}}$ is a frag\=/coarse monoid, it is tempting to look at the subset of frag\=/coarsely invertible functions $(\ctrl{Y}{Y})^{\crse *}\subseteq \ctrl{Y}{Y}$ (i.e.\ the set of frag\=/coarse equivalences $(Y,\CF)\to (Y,\CF)$).
Clearly, $(\ctrl{Y}{Y})^{\crse *}$ is closed under composition and therefore it defines a frag\=/coarse submonoid $\crse{(Y^{Y})^{*}}$ of $\crse{Y^{Y}}$. However, $\crse{(Y^{Y})^{*}}$ is generally not a frag\=/coarse group because the inversion need not be well\=/behaved with respect to the exponential frag\=/coarse structure. 

More precisely, there are two fundamental issues. The first issue is that even if we assume that every element $f\in (\ctrl{X}{Y})^{\crse *}$ is coarsely invertible, we do not know whether we can \emph{choose} an inversion function $\inversefn\colon (\ctrl{Y}{Y})^{\crse *}\to(\ctrl{Y}{Y})^{\crse *}$ so that the (Inverse) Group Diagram commutes up to frag\=/closeness:
\[ 
\begin{tikzcd}
    \crse{(Y^{Y})^{*}} \arrow{r}{(\cinversefn ,\cid_{\crse{(Y^{Y})^{*}}})} \arrow[swap]{d}{(\cid_{\crse{(Y^{Y})^{*}}},\cinversefn)}  \arrow[dr, "\cid_{\crse Y}"] & \crse{(Y^{Y})^{*}}\times \crse{(Y^{Y})^{*}}      \arrow{d}{\ccirc}\\
    \crse{(Y^{Y})^{*}}\times \crse{(Y^{Y})^{*}} \arrow[swap]{r}{\ccirc} & \crse {(Y^{Y})^{*}}. 
\end{tikzcd} 
\]
For instance, let $\crse Y=(\NN,\CE_d)$ and for each $k\in \NN$ let $f_k(n)\coloneqq n+k$ be a translation by $k$. The set of functions $\braces{f_k\mid k\in \NN}$ is a fragment of $\crse{(Y^{Y})^{*}}$. However, there is no way to find coarse inverses $f_k^{-1}$ such that the compositions $f_k\circ f_k^{-1}$ and $f_k^{-1}\circ f_k$ are uniformly close to $\id_\NN$.

The second issue is that even if there exists an inverse function $\inversefn\colon (\ctrl{Y}{Y})^{\crse *}\to(\ctrl{Y}{Y})^{\crse *}$ so that the diagram commutes, the function would not need to be controlled. To see this, let $\crse Y= (\RR,\CE_d)$ and let $f_t(x)\coloneqq tx$ denote the multiplication by $t$. The set $\braces{f_t\mid 0<t\leq 1}$ is a set of equi controlled coarse equivalences and it is hence a fragment of $\crse{(Y^{Y})^{*}}$. Still, the set of inverses $\braces{f_t^{-1}\mid 0<t\leq 1}=\braces{f_{t^{-1}}\mid 0<t\leq 1}$ does not consist of equi controlled functions and it is hence \emph{not} a fragment.%
\footnote{%
This example shows that Proposition~\ref{prop:p1:coarse.group.iff.heartsuit.and.equi controlled} is not true for frag\=/coarse spaces.
}

We can go around these difficulties as follows. Let $\crse X=(X,\CE)$ and $\crse Y=(Y,\CF)$ be two (frag\=/coarsely equivalent) frag\=/coarse spaces, and define $\ctreq{X}{Y}\subseteq \ctrl{X}{Y}\times \ctrl{Y}{X}$ as the set
\[
 \ctreq{X}{Y}\coloneqq\bigbraces{(f,g)\mid f\in \ctrl{X}{Y},\ g\in\ctrl{Y}{X} \text{ frag-coarse inverses of one another}}.
\]
One advantage of considering such a set is that we have a natural candidate for the inversion function $\inversefn$: namely swapping the order of each pair. The composition $\circ \colon\ctreq{Y}{Z}\times\ctreq{X}{Y}\to \ctreq{X}{Z}$ is defined componentwise. It now remains to define a frag\=/coarse structure such that these operations are controlled and make $\ctreq{Y}{Y}$ into a frag\=/coarse group.

Here we see a second advantage of working with pairs of functions. Since $\ctreq{X}{Y}$ is a subset of the product $\ctrl{X}{Y}\times \ctrl{Y}{X}$, we can consider the restriction of the product frag\=/coarse structure $\CF^\CE\otimes \CE^\CF$. When doing so, we automatically obtain that our candidate $\inversefn\colon\ctreq{X}{Y}\to\ctreq{Y}{X}$ is a controlled function.

We are still not done though. The product coarse structure does not allow us to control ``how close to actual inverses'' are the coarse inverse functions and hence $\inversefn$ need not make the (Inverse) Group Diagram commute. Rather, we will use the following:

\begin{de}\label{def:p2:frag space of ceq}
 A relation $C\subseteq \ctreq{X}{Y}\times \ctreq{X}{Y}$ belongs to the frag\=/coarse structure $\ctreq{\CE}{\CF}$\nomenclature[:CE1]{$\ctreq{X}{Y}$}{fragmentary coarse structure on the set of coarse equivalences}
 if 
 \begin{enumerate}[(Ce1)]
  \item $C\in \CF^\CE\otimes \CE^\CF$;
  \item for every $E\in\CE$ and $F\in\CF$ there exist  $E'\in\CE$ and $F'\in \CF$ such that
  \[
   \id_{X}\times (g_1\circ f_1)(E)\subseteq E' \quad \text{and}\quad
   \id_{Y}\times (f_2\circ g_2)(F)\subseteq F' 
  \]
 for every $(f_1,g_1)\torel{C}(f_2,g_2)$.
 \end{enumerate}
 The \emph{fragmentary coarse space of coarse equivalences}\index{fragmentary!space of coarse equivalences} is $\crse{\ctreq{X}{Y}}=(\ctreq{X}{Y},\ctreq{\CE}{\CF})$.\nomenclature[:COS]{$\crse{\ctreq{X}{Y}}$}{fragmentary coarse space of coarse equivalences} We shall denote $\crse{\ctraut{Y}}\coloneqq\crse{\ctreq{Y}{Y}}$ and $\ctraut{Y}\coloneqq\ctreq{Y}{Y}$
\end{de}

We leave it to reader to verify that the above is a fragmentary coarse structure.
Effectively, to define $\ctreq{\CE}{\CF}$ we selected a subset of $\CF^\CE\otimes \CE^\CF$ in such a way that the fragments of $\crse{\ctreq{X}{Y}}$ are comprised of pairs of functions $(f,g)$ such that $f\circ g$ and $g\circ f$ are uniformly close to $\id_{Y}$ and $\id_{X}$. Namely, we give the following.

\begin{de}
 Let $\crse X=(X,\CE)$ be a frag\=/coarse space. Two families of functions $f_i\colon X\to X$ and $g_i\colon X\to X$ indexed over a set $I$ are \emph{equi frag\=/close}\index{equi!frag-close} if for every fragment $X'\subseteq X$ there is an $E\in \CE$ such that $f_i\times g_i(\Delta_{X'})\subseteq E$ for every $i\in I$. A family $f_i\colon X\to X$ is equi frag\=/close to a fixed function $g\colon X\to X$ if it is equi frag\=/close to the constant family $g_i=g$.
\end{de}

\begin{lem}\label{lem:p2:controlled into ctraut}
 A relation $C\subseteq \ctreq{X}{Y}\times \ctreq{X}{Y}$ belongs to the frag\=/coarse structure $\ctreq{\CE}{\CF}$ if and only if it is in $\CF^\CE\otimes \CE^\CF$ and the families of functions $\{g\circ f\mid (f,g)\in \pi_1(C)\cup\pi_2(C)\}$ and $\{f\circ g\mid (f,g)\in \pi_1(C)\cup\pi_2(C)\}$ are equi frag\=/close to $\id_X$ and $\id_Y$ respectively.
\end{lem}
\begin{proof}
 For the forward implication: let $X'\subseteq X$ be a fragment. Applying condition (Ce2) to the set $E=\Delta_{X'}$ we see that 
 \[
  x \rel{\CE} g_1\circ f_1(x) \qquad \forall (f_1,g_1)\in\pi_1(C),\ x\in X'.
 \]
 Since $C$ is in $\CF^\CE\otimes \CE^\CF$, we also have
 \[
  g_1\circ f_1(x) \rel{\CE} g_2\circ f_2(x) 
  \qquad\forall (f_1,g_1)\torel{C}(f_2,g_2),\ x\in X'.
 \]
 Together, these two facts show that $\{g\circ f\mid (f,g)\in \pi_1(C)\cup\pi_2(C)\}$ are indeed equi frag\=/close to $\id_X$. The same argument shows that $\{f\circ g\mid (f,g)\in \pi_1(C)\cup\pi_2(C)\}$ are equi frag\=/close to $\id_Y$.
 
 For the converse implication we must verify (Ce2). Fix $E\in \CE$. Since $\pi_2(E)\subseteq X$ is a fragment, the equi closeness condition implies
 \[
  x_1 \torel{E} x_2 \rel{\CE} g_1\circ f_1(x_2) \qquad \forall (f_1,g_1)\in\pi_1(C),\ x_1\torel{E}x_2.  
 \]
 Using an analogous argument for $f_2\circ g_2$, we conclude that (Ce2) is satisfied.
\end{proof}

Unraveling notation, we have:

\begin{cor}\label{cor:p2:controlled into ctraut}
 A function $F\colon Z\to (\ctraut{X},\ctraut{\CE})$ defined as $F(z)\coloneqq(f_z,g_z)$ is controlled if and only if $F\colon Z\to (\ctraut{X},\CE^\CE\otimes\CE^\CE)$ is controlled, and for every fragment $Z'\subseteq Z$ the families $\braces{f_z\circ g_z\mid z\in Z'}$ and $\braces{g_z\circ f_z\mid z\in Z'}$ are equi frag\=/close to $\id_X$.
\end{cor}

Now that we defined the frag\=/coarse structure on $\ctreq{X}{Y}$, it is routine to prove the following:

\begin{lem}
 The composition $\circ\colon \ctreq{Y}{Z}\times\ctreq{X}{Y}\to\ctreq{X}{Z}$ and inversion $\inversefn\colon \ctreq{X}{Y}\to\ctreq{Y}{X}$ are controlled maps. Consider also the identity element $(\id_{Y},\id_{Y})\in\ctraut{Y}$,\\ then $(\crse{\ctraut{Y}},\ccirc,[(\id_{Y},\id_{Y})],\cinversefn)$ is a frag\=/coarse group.
\end{lem}

Since we defined the coarse structure $\ctreq{\CE}{\CF}$ as a refinement of $\CF^\CE\otimes \CE^\CF$, it is clear that the projection on the first coordinate $(f,g)\mapsto f $ defines a controlled map $\crse{\ctreq{X}{Y}\to Y^{X}}$. Notice that for $\crse {X=Y}$ this map is in fact a homomorphism of frag\=/coarse monoids. From here, it is easy to obtain the required characterization of coarse actions. 

\begin{prop}\label{prop:p2:actions as group homomorphisms}
 Let $\crse G$ be a coarse group and $\crse Y$ a coarse space. A coarse map $\crse{\alpha\colon G\times Y\to Y}$ is a coarse action if and only if $\crse{\lambda\alpha\colon G\to Y^{Y}}$ lifts to a coarse homomorphism $\crse{G\to \ctraut{Y}}$.
\end{prop}

\begin{proof}

 One implication follows trivially from Proposition~\ref{prop:p2:actions as monoid homomorphisms}: since $\crse{\ctraut{Y}\to Y^{Y}}$ is a frag\=/coarse homomorphism, if $\crse{\lambda\alpha\colon G\to Y^{Y}}$ lifts to a coarse homomorphism $\crse{G\to \ctraut{Y}}$ then it is a coarse homomorphism.

 Now let $\crse \alpha$ be a coarse action. Unpacking the notation, we have $\lambda\alpha(g)={}_g\alpha$ for every $g\in G$. Consider the map $\dblact\colon G\to\ctraut{Y}$ sending $g\mapsto ({}_g\alpha,{}_{g^{-1}}\alpha)$.
 Setwise, the function $\dblact$ is the product of $\lambda\alpha$ and $\lambda\alpha\circ\inversefn$. Since $\crse \alpha$ is a coarse action, it follows from Theorem~\ref{thm:p2:exponential} that $\dblact\colon (G,\CE)\to (\ctraut{Y},\CF^\CF\otimes\CF^\CF)$ is controlled. By Corollary~\ref{cor:p2:controlled into ctraut}, to deduce that $\dblact\colon (G,\CE)\to (\ctraut{Y},\ctraut{\CF})$ is also controlled it remains to check that the functions  ${}_{g}\alpha\circ {}_{g^{-1}}\alpha$ and ${}_{g^{-1}}\alpha\circ {}_{g}\alpha$ are equi close to $\id_Y$. This follows immediately from the coarse group and action axioms:
 \[
  \paren{ {}_{g^{-1}}\alpha\circ {}_{g}\alpha} (y) =g^{-1}\cdot(g\cdot y)
  \rel{\CF} (g^{-1}\ast g)\cdot y
  \rel{\CF} y
  \qquad \forall g\in G,\ y\in Y\phantom{.}
 \]
 \[
  \paren{ {}_{g}\alpha\circ {}_{g^{-1}}\alpha} (y) =g\cdot(g^{-1}\cdot y)
  \rel{\CF} (g\ast g^{-1})\cdot y
  \rel{\CF} y
  \qquad \forall g\in G,\ y\in Y.
 \]

 We now need to show that the diagram 
\[ 
\begin{tikzcd}[column sep = 3 em]
    \crse{G}\times \crse{G} \arrow[r, "\crse {\dblact\times \dblact}"] \arrow[d, "\cop"] & \crse{\ctraut{Y}}\times \crse{\ctraut{Y}}      \arrow[d, "\ccirc"]\\
    \crse G \arrow[r, "\crse{ \dblact}"]& \crse{\ctraut{Y}} 
\end{tikzcd} 
\]
commutes.
 Since $\crse \alpha$ is a coarse action, we already know by Proposition~\ref{prop:p2:actions as monoid homomorphisms} that the diagram commutes up to closeness when we give $\ctraut{Y}$ the frag\=/coarse structure $\CE^\CE\otimes\CE^\CE$.
 It then remains to show that the relation
 \[
  \Bigparen{\paren{\dblact\circ(\variable\ast\variable)}\times\paren{\dblact\circ\dblact}}(\Delta_{G\times G}) =
  \Bigbraces{\bigparen{ ({}_{g_1\ast g_2}\alpha,{}_{(g_1\ast g_2)^{-1}}\alpha)
  \,,\, ({}_{g_1}\alpha \circ {}_{g_2}\alpha\,,\, {}_{g_2^{-1}}\alpha\circ {}_{g_1^{-1}}\alpha) }
  \bigmid g_1,g_2\in G}
 \]
 satisfies condition (Ce2). As in Lemma~\ref{lem:p2:controlled into ctraut}, it is also enough to verify condition (Ce2) in the special case $E=F=\Delta_Y$. In symbols, this amounts to observing that
 \[
  ({}_{(g_1\ast g_2)^{-1}}\alpha)\circ({}_{g_1\ast g_2}\alpha)(y)
  = (g_1\ast g_2)^{-1}\cdot(\paren{g_1\ast g_2}\cdot y)
  \rel{\CE} y \qquad \forall g_1,g_2\in G,\ y\in Y
 \]
 and
 \[
  ({}_{g_1}\alpha \circ {}_{g_2}\alpha) \circ ({}_{g_2^{-1}}\alpha\circ {}_{g_1^{-1}}\alpha)(y)
  = g_1\cdot( g_2\cdot(g_2^{-1}\cdot (g_1^{-1}\cdot y)))
  \rel{\CE} y \qquad \forall g_1,g_2\in G,\ y\in Y.
 \]
 Both the above follow from the coarse group and action axioms.
\end{proof}

\section{Coarse Groups as Coarse Subgroups}
As in Chapter~\ref{ch:p2:a cgroup not group}, we say that a set\=/group equipped with a fragmentary coarse structure with which it is a frag\=/coarse group\footnote{%
The failure of Proposition~\ref{prop:p1:coarse.group.iff.heartsuit.and.equi controlled} for frag coarse spaces shows that condition is \emph{not} equivalent to saying that the frag\=/coarse structure is equi bi\=/invariant. That is, one needs to check it is equi bi\=/invariant \emph{and} that the inverse function is controlled.
} 
is a \emph{frag\=/coarsified set\=/group}.\index{fragmentary!coarsified set\=/group}
In this section we will prove the following consequence of Proposition~\ref{prop:p2:actions as group homomorphisms}:

\begin{thm}\label{thm:p2:coarse groups as subgroups}
 Every coarse group is isomorphic to a frag\=/coarse subgroup of a frag\=/coarsified set\=/group.
\end{thm}

As a preliminary, we start with a general observation on functoriality properties of spaces of controlled functions. Fix coarse spaces $\crse X_1,\ \crse X_2,\ \crse Y_1,\ \crse Y_2$.
If we are given controlled maps $  f_1\colon   Y_1\to   X_1$ and $  f_2\colon   X_2\to  Y_2$ we may define an adjoint coarse map
$ \Ad(  f_2,  f_1)\colon \ctrl{X_1}{{X_2}}\to \ctrl{ Y_1}{{Y_2}}$ by composition:
\[
   \Ad(  f_2,  f_1)(  g)\coloneqq  f_2 \circ g \circ f_1.
\]
It follows from Lemma~\ref{lem:p2:composition of exponential} that $\Ad(f_2,f_1)$ is controlled and hence defines a frag\=/coarse map 
\[
 \crse\Ad(\crse f_2,\crse f_1)\colon {\crse X_2}^{\crse X_1}\to {\crse Y_2}^{\crse Y_1}.
\]
If $\crse f_1,\crse f_2$ are coarse equivalences with coarse inverses $\crse f_1^{-1},\ \crse f_2^{-1}$, then $\crse\Ad(\crse f_2,\crse f_1)$ is a frag\=/coarse equivalence with frag\=/coarse inverse $\crse\Ad(\crse f_2^{-1},\crse f_1^{-1})$. In particular, if $\crse{ f\colon X\to Y}$ is a coarse equivalence we obtain a frag\=/coarse equivalence
\[
 \crse\Ad(\crse f)\coloneqq \crse\Ad(\crse f,\crse f^{-1})\colon\crse{X^{X}\to Y^{Y}}.
\]
Also note that $\crse{\Ad(f)}$ is an isomorphism of frag\=/coarse monoids. These remarks will be useful soon, as they give us a quick way to show that a certain construction yields a coarse action.

\

The key technical point in the proof of Theorem~\ref{thm:p2:coarse groups as subgroups} is a procedure to upgrade coarse actions of coarse groups to coarse actions $\crse{\alpha\colon G\curvearrowright X}$ where left multiplications ${}_g\alpha\colon X\to X$ are bijections for every $g\in G$. We need a piece of nomenclature: a family of (controlled) functions $f_i\colon X\to (Y,\CF)$, $i\in I$ is \emph{equi coarsely surjective}\index{equi!coarsely surjective} if they have uniformly coarsely dense images. That is, there is a controlled entourage $F\in\CF$ such that $F(f_i(X))=Y$ for every $i\in I$.

\begin{lem}\label{lem:p2:upgrade to bijective}
 For every coarse space $\crse X$ there exist a coarse space $\crse X'$ and a coarse equivalence $\crse{\iota\colon X\to X}'$ with a representative $\iota$ such that the following holds. Given any family $D\subseteq \ctrl{X}{X}$ of equi coarsely surjective maps there is a $\Phi\colon D\to \ctrl{X'}{X'}$ such that
 \begin{itemize}
  \item $\Phi(f)\colon X'\to X'$ is bijective for every $f\in D$;
  \item $\Ad(\iota^{-1})\circ \Phi=\id_D$;
  \item $\Phi$ is frag\=/close to the restriction $\Ad(\iota)|_D$.
 \end{itemize}
\end{lem}

\begin{proof}
 Let $Z$ be an infinite set with $\abs{Z}\geq \abs{X}$, let $X'\coloneqq X\times X\times Z$, define $\CE'\coloneqq \CE\otimes\maxcrs\otimes\maxcrs$ on $X'$ and set $\crse X' = \paren{X',\CE' }$.
 Also fix an injection $\iota\colon X\to X\times X\times Z$ such that the composition with the projection onto the first coordinate is the identity: $\pi_1\circ \iota=\id_X$. Since the second and third factors of $\crse X'$ are bounded, it does not matter how $\iota$ is chosen. The map $\iota$ is a coarse equivalence and $\pi_1$ is a coarse inverse for it. With these choices, $\Ad(\iota)(f)= \iota \circ f \circ \pi_1$ and $\Ad(\iota^{-1})(F)= \pi_1 \circ F \circ \iota$.

 Fix one point $\bar z\in Z$. Given any $f\in D$, consider the map $\bar f\colon X'\to X'$ given by 
 \[
  \bar f(x_1,x_2,z)\coloneqq(f(x_1),x_1,\bar z).
 \]
 By construction, $\bar f$ is injective and the preimage of a point is given by
 \[
  \bar f^{-1}(y_1,y_2, z) =
  \left\{ 
  \begin{array}{ll}
   \{y_2\}\times X\times Z \quad & \text{if } y_1=f(y_2)\text{ and }z=\bar z \\
   \emptyset & \text{otherwise. }
  \end{array}
  \right.
 \]
 In particular, as $x$ ranges in $X$ the sets $\bar f^{-1}(f(x),x,\bar z)$ partition $X'$ into subsets of cardinality $\abs{Z}$. Note that $\iota(x)\in f^{-1}(f(x),x,\bar z)$ and also note that if we let $E_1\coloneqq \Delta_X\otimes (X\times X)\otimes (Z\times Z)\in \CE\otimes\maxcrs\otimes\maxcrs$ then
 \begin{equation}\label{eq:p2:close to barf 1}
  (\Ad(\iota)(f))(x') =\iota\circ f\circ \pi_1(x') \rel{E_1} \bar f(x') \qquad \forall x'\in X'
 \end{equation}
 (\emph{i.e.}\ the functions $\bar f$ are equi close to $\Ad(\iota)(f)$ as $f$ ranges in $D$).
 
 By hypothesis, there is an $E\in \CE$ independent of the choice of $f\in D$ such that $X=E(f(X))$. Let $E_2= E\otimes(X\times X)\otimes(Z\times Z)\in\CE\otimes\maxcrs\otimes\maxcrs $.
 We can choose a projection $p_f\colon X'\twoheadrightarrow \bar f(X')$ that is $E_2$\=/close to the identity, in the sense that 
 \begin{equation}\label{eq:p2:close to barf 2}
    x'\rel{E_2} p_f(x') \qquad \forall x'\in X'.
 \end{equation}
 By construction, as $x$ ranges in $X$ the sets $p_f^{-1}(f(x),x,\bar z)$ partition $X'$ in subsets of cardinality $\abs{Z}$. Note that $(f(x),x,\bar z)\in p_f^{-1}(f(x),x,\bar z)$. 
 For each $x\in X$ we may now choose a bijection
 \[
  f'_{x}\colon \bar f^{-1}(f(x),x,\bar z)\to p_f^{-1}(f(x),x,\bar z)
 \]
 because both sets have cardinality $\abs{Z}$. We may also impose that $f'_x(\iota(x))=(f(x),x,\bar z)$. 
 
 We now define $f'\colon X'\to X'$ by $f'(x_1,x_2,z)\coloneqq f'_{x_2}(x_1,x_2,z)$. Since $f'$ is defined using partitions of $X'$, it is bijective. By construction, $f'$ is so that $\Ad(\iota^{-1})(f')= \pi_1\circ f'\circ \iota = f$. It remains to show that $f'$ is controlled (which is rather trivial) and that the mapping $\Phi(f)\coloneqq f'$ is frag\=/close to $\Ad(\iota)|_D$.

 Notice that $p_f\circ f'=\bar f$, hence \eqref{eq:p2:close to barf 2} shows that $f'$ is equi close to $\bar f$ as $f$ ranges in $D$ (this already implies that every $f'$ is controlled). Let $C\subseteq D\subseteq \crse{X^{X}}$ be a fragment, we need to show that 
 \[
  \bigparen{\Phi\times\Ad(\iota)}(\Delta_C)\in \CE'^{\CE'}.
 \]
 Explicitly, we must show that for every $E'\in\CE'$ there is $F'\in\CE'$ such that
 \[
  f'(x'_1)\torel{F'}\iota\circ f\circ \pi_1(x'_2) \qquad \forall x'_1\torel{E'}x'_2,\ \forall f\in C.
 \]

 Fix $E'\in\CE'$. Enlarging it if necessary, we may assume that $E'= E\otimes (X\times X)\otimes(Z\times Z)$. Since $C$ is a fragment, there is an $F\in \CE$ such that $f\times f(E)\subseteq F$ for every $f\in C$. Combining \eqref{eq:p2:close to barf 1} and \eqref{eq:p2:close to barf 2} we obtain:
 \[
  f'(x'_1)\rel{E_2}\bar f(x'_1)
  \torel{F\otimes E\otimes \Delta_Z} \bar f(x'_2)
  \rel{E_1}\iota\circ f\circ \pi_1(x'_2) 
  \qquad \forall x'_1\torel{E'}x'_2,\ \forall f\in C,
 \]
 which concludes the proof.
\end{proof}

\begin{prop}\label{prop:p2:upgrade to bijective action}
 Let $\crse{\alpha \colon G\curvearrowright Y}$ be a coarse action of a coarse group. Then there is an isomorphic coarse action $\crse{\alpha'\colon G\curvearrowright Y'}$ such that left multiplications ${}_g\alpha'\colon Y'\to Y'$ are bijections for every $g\in G$.
\end{prop}
\begin{proof}
 Let $\crse Y'$ be the coarse space constructed in Lemma~\ref{lem:p2:upgrade to bijective}. 
 Since the compositions $\{{}_g\alpha\circ {}_{g^{-1}}\alpha\mid g\in G\}$ are equi close to $\id_Y$, we deduce that $\{\lambda\alpha(g) \mid g\in G\}=\{{}_g\alpha \mid g\in G\}$ is a family of equi coarsely surjective maps. We may then apply Lemma~\ref{lem:p2:upgrade to bijective} to define ${}_g\alpha'\coloneqq \Phi({}_g\alpha)\colon Y'\to Y'$.
 We have constructed a function $\alpha'\colon G\times Y'\to Y'$ so that, by definition, $\lambda\alpha' = \Phi\circ \lambda\alpha$.

 Since $\Phi$ is frag\=/close to $\Ad(\iota)$, the function $\lambda\alpha'$ is frag\=/close to $\Ad(\iota)\circ \lambda\alpha$ and it is therefore controlled. Since $\crse{\Ad(\iota)}$ is an isomorphism of frag\=/coarse monoids, it follows that $\crse{\lambda\alpha}' =\crse{\Ad(\iota)\circ \lambda\alpha}$ is a coarse homomorphism and therefore $\crse\alpha '$ is a coarse action by Proposition~\ref{prop:p2:actions as monoid homomorphisms}. 
 
 By construction, the coarse equivalence $\crse{\iota\colon Y\to Y'}$ is coarsely equivariant from $\crse{\alpha\colon G\curvearrowright Y}$ to $\crse{\alpha'\colon G\curvearrowright Y'}$ and therefore it is an isomorphism of coarse actions.
\end{proof}

Given a (frag-)coarse space $\crse X=(X,\CE)$, let
\[
 \bijctraut X\coloneqq\braces{(f,f^{-1})\mid f\colon X\to X\text{ bijective (frag-)coarse equivalence}}.
\]
This is a subset of $\ctraut X$ and, most importantly, it is a set\=/group (in the above $f^{-1}$ denotes the set\=/theoretic inverse function, not just a coarse inverse).
We can equip it with the restriction of the coarse structure $\ctraut{\CE}$ and thus make it into a frag-coarsified set\=/group $\crse{\bijctraut{X}}$.

The idea now is to use Proposition~\ref{prop:p2:upgrade to bijective action} to upgrade a coarse action of a coarse group $\crse G$ to an action by bijections so that we can use the transpose $\lambda$ to define a coarse homomorphism with image in $\bijctraut Y$. We take this plan to completion for coarse groups $\crse G$ satisfying the following condition:

\vspace{1 em}

\noindent
\begin{minipage}{0.05\textwidth}
 $(\diamondsuit)$
\end{minipage}
\begin{minipage}{0.95\textwidth}
\textit{
 the coarse operations have representatives so that $\inversefn\colon G\to G$ is an involution whose only fixed point is $e_G\in G$.
 }
\end{minipage}

\vspace{1 em}

We showed in Remark~\ref{exmp:p1:no cancellative reps} that coarse groups need not satisfy $(\diamondsuit)$ in general. However, we can still deduce that Theorem~\ref{thm:p2:coarse groups as subgroups} holds for every coarse group because it is always possible to pass to an isomorphic coarse group that does:

\begin{lem}\label{lem:p2:isomorphic to diamondsuit}
 Every coarse group $\crse G$ is isomorphic to a coarse group $\crse {G'}$ that satisfies $(\diamondsuit)$.
\end{lem}
\begin{proof}
 Fix representatives $e\in G$ and $\inversefn\colon G\to G$. Let $G^+\coloneqq G\smallsetminus \{e\}$, let $\overline G{}^{+}$ be a disjoint copy of $G^{+}$, and define $\widehat G\coloneqq G^{+}\sqcup \overline G{}^{+} \sqcup \{e\}$. Let $\iota\colon G\hookrightarrow \widehat G$ be the inclusion and $\pi\colon \widehat G\to G$ the projection 
 \[
  \pi(x)\coloneqq
  \left\{
  \begin{array}{ll}
   g & \text{if }x=g\in G^+\sqcup\{e\} \\
   g^{-1} & \text{if }x=\bar g\in \overline G{}^+.
  \end{array}
  \right.
 \]
 
 Equip $\widehat G$ with the pull\=/back coarse structure $\pi^{\ast}(\CE)$. Since $\pi\colon (\widehat G,\pi^{\ast}(\CE))\to (G,\CE)$ is surjective, it is a coarse equivalence. Since $\iota$ is a section of $\pi$, it is a coarse inverse. We may then use $\pi$ and $\iota$ to define coarse group operations on $\widehat{\crse G}=(\widehat G,\pi^{\ast}(\CE))$ as compositions 
 \[
  \widehat{\crse G}\times \widehat{\crse G}\xrightarrow{\crse {\pi\times \pi}}
  \crse G\times \crse G \xrightarrow{\ \cop\ } \crse G \xrightarrow{\ \crse \iota\ }
  \widehat{\crse G}
  \qquad\qquad
  \widehat{\crse G}\xrightarrow{\ \crse {\pi}\ }
  \crse G \xrightarrow{\cinversefn} \crse G \xrightarrow{\ \crse \iota\ }
  \widehat{\crse G}.
 \]
 By construction, $\widehat{\crse G}$ is a coarse group isomorphic to $\crse G$.
 
 To conclude, it is enough to observe that the involution fixing $e\in \widehat G$ and swapping $g$ with $\bar g$ for every $g\in G^+$ is $\pi^{\ast}(\CE)$\=/close to the composition $\iota\circ\inversefn\circ \pi$, and it is therefore a representative for the coarse inversion in $\widehat{\crse G}$.
\end{proof}

We now have all the necessary tools to prove Theorem~\ref{thm:p2:coarse groups as subgroups}.

\begin{proof}[Proof of Theorem~\ref{thm:p2:coarse groups as subgroups}]
 Fix a faithful coarse action $\crse{\alpha\colon G\curvearrowright Y}=(Y,\CF)$, \emph{e.g.}\ the coarse action by left\=/multiplication $\crse{\cop\colon G\curvearrowright G}$ (Example~\ref{exmp:p2:left multiplication is faithful}). Using Proposition~\ref{prop:p2:upgrade to bijective action}, we may assume that for every $g\in G$ the section ${}_g\alpha\colon Y\to Y$ is bijective.
 
 Assume first that $\crse G$ satisfies condition $(\diamondsuit)$. We may then write $G$ as a disjoint union $G=G^+\sqcup G^{-}\sqcup \{e\}$ so that $\inversefn$ swaps $G^+$ and $G^-$ (for each pair of inverse elements of $G$ we arbitrarily choose one to be in $G^+$ and let the other in $G^-$). We can now slightly modify $\alpha$ as follows: for every $g\in G^+$ let ${}_g\alpha'\coloneqq {}_g\alpha$, for $g\in G^-$ let ${}_g\alpha'\coloneqq \paren{{}_{g^{-1}}\alpha}^{-1} = \paren{{}_{g^{-1}}\alpha'}^{-1}$, and let ${}_e\alpha\coloneqq \id_Y$. This newly defined $\alpha'$ is close to $\alpha$ because
 \[
  {}_g\alpha(y)= {}_g\alpha\circ{}_{g^{-1}}\alpha\circ ({}_{g^{-1}}\alpha)^{-1}(y)
  \rel{\CF} {}_{g\ast g^{-1}}\alpha \circ ({}_{g^{-1}}\alpha)^{-1}(y)
  \rel{\CF} ({}_{g^{-1}}\alpha)^{-1}(y)
  \qquad \forall g\in G^-,\ y\in Y.
 \]
 Hence $\alpha'$ is another representative for $\crse \alpha$.
 
 As in the proof of Proposition~\ref{prop:p2:actions as group homomorphisms}, we may define a function $\dblact'\colon G\to \ctraut{Y}$ sending $g\in G$ to $({}_g\alpha',{}_{g^{-1}}\alpha')$ to obtain a coarse homomorphism $\crse {G\to \ctraut{Y}}$. By construction, the image of $\dblact'$ is contained in $\bijctraut{Y}$. 
 Since the coarse action $\crse{\alpha\colon G\curvearrowright Y}$ is faithful, the map $\dblact'$ is proper. It follows that $\crse{\dblact'}$ defines an isomorphism between $\crse G$ and its coarse image, which is a coarse subgroup of the frag\=/coarsified set\=/group $\crse{\bijctraut{Y}}$.
 
 For an arbitrary coarse group $\crse G$, it is now enough to apply Lemma~\ref{lem:p2:isomorphic to diamondsuit} to obtain an isomorphism $\crse{G\to G}'$ where the coarse group $\crse G'$ satisfies $(\diamondsuit)$ and then realize $\crse G'$ as a coarse subgroup of a coarsified set\=/group.
\end{proof}

As we already pointed out, the definition of the exponential frag\=/coarse structure requires us to leave the category of coarse spaces. As a consequence, we see no obvious way to adapt the proof of the above theorem to upgrade the frag\=/coarsified set\=/group to a genuine coarsified set\=/group. More precisely, we do not know the answer to the following:

\begin{qu}\label{qu:p2:coarse groups as subgroups}
 Is every coarse group isomorphic to a coarse subgroup of a coarsified set\=/group? 
\end{qu}

The above is related to, but independent of, our conjecture that there exist coarse groups that are not isomorphic to any coarsified set\=/group. We do not yet have any intuition for what the answer to Question~\ref{qu:p2:coarse groups as subgroups} should be (note that our candidate example for a coarse group that is not the coarsification of a group is defined as a coarse subgroup of $(F_2,\varcrs{bw})$).

%% file: Appendices.tex
\part{Appendices}

\chapter{Categorical Aspects of \Cat{Coarse}}\label{ch:appendix:categorical aspects}

In this appendix we review some basic properties of the categories \Cat{Coarse} and \Cat{FragCrs}. We will recall most of the relevant categorical definitions, for more details and basic facts we refer to \cite{MacLane}. For the most part, the material concerning \Cat{Coarse} is not novel, we decided to include it to provide a somewhat complete treatment with unified conventions and notation. More details can be found \emph{e.g.}\ \cite{dikranjan_categorical_2017,dikranjan2020categories,luu2007coarse, zava2019cowellpoweredness}. As far as we are aware, all the results concerning fragmentary coarse structures have not been considered anywhere else in the literature.

\begin{warning}
According to our conventions, in the following there will be two sets of bold symbols: we use upright bold characters for categories and italic bold characters for coarse spaces. 
\end{warning}

\section{Limits and Colimits}\label{sec:appendix:limits.and.colimits}
Taking the disconnected union (Example~\ref{exmp:p1:disconnected union}) shows that the category \Cat{Coarse} has arbitrary finite\=/coproducts. Namely, if $\crse X_i=(X_i,\CE_i)$ is a family of coarse spaces over an index set $I$, we can define a coarse space by considering the disjoint union of the $X_i$ with the coarse structure generated by the $\CE_i$: 
\[
 \Bigparen{{\textstyle\bigsqcup_{i\in I}X_i\; ,\; \CE_{\sqcup_i}}}
 =\Bigparen{{\textstyle\bigsqcup_{i\in I}X_i}\; ,\; \angles{\CE_i\mid i\in I}}.
\]
It is obvious that the inclusion $X_i\hookrightarrow\bigsqcup_{i\in I}X_i$ is controlled for every $i$. Given controlled functions $f_i\colon (X_i,\CE_i)\to (Z,\CF)$, their disjoint union $\bigsqcup_{i\in I} f_i\colon \bigsqcup_{i\in I} X_i\to Z$ is controlled because it sends a generating family of controlled sets to controlled sets of $Z$. 
If the set $I$ is finite, it is also true that any other function $\bigsqcup_{i\in I} X_i\to Z$ that is compatible with the $f_i$ must be close to $\bigsqcup_{i\in I} f_i$, therefore $\coprod_{i\in I}\crse X_i \coloneqq\bigparen{{\textstyle\bigsqcup_{i\in I}X_i\; ,\; \CE_{\sqcup_i}}}$ is indeed the coproduct of the finitely many $\crse X_i$. The coproduct of an empty family of coarse spaces is the empty coarse space, which is indeed an initial object in \Cat{Coarse}. 

\

Given two morphisms $f,g\colon X\to Y$ in a category \Cat C, a \emph{coequalizer}\index{coequalizer} is an object $Q$ together with a morphism ${\rm coeq}\colon Y\to Q$ such that ${\rm coeq}\circ f={\rm coeq}\circ g$ and is universal with this property, \emph{i.e.}\ every $f_Z\colon Y\to Z$ such that ${f_Z}\circ f={f_Z}\circ g$ factors through $Q$: 
\[
\begin{tikzcd}
X \arrow[r, shift left, "{f}"] \arrow[r, shift right, swap, "{g}"] & Y \arrow[r, "{\rm coeq}"] \arrow[dr,swap, "{f_Z}"] & Q  \arrow[d,dashed] \\
    &  & Z 
\end{tikzcd} .
\]

\begin{lem}\label{coequalizers} 
 The category \Cat{Coarse} has coequalizers.
\end{lem}
\begin{proof}[Sketch of proof]
 Given $\crse{f,g\colon X\to Y}$, where $\crse Y=(Y,\CF)$, let $\crse Q$ be the set $Y$ equipped with the coarse structure $\CD$ generated by $\CF\cup \braces{\paren{f\times g}(\Delta_X)}$. Then the set-theoretic identity function $\id_Y\colon Y\to Q$ defines a coarse map $\crse{i\colon Y\to Q}$ such that $\crse{i\circ f=i \circ g}$ and is universal with this property.
\end{proof}

Let \Cat{J} be a fixed category. A \emph{diagram of shape} \Cat J in a category \Cat C is a functor $D\colon\Cat J\to\Cat C$. Given a diagram $D\colon\Cat J\to\Cat C$, a \emph{cocone}\index{cocone} is an object $A$ in \Cat C together with morphisms $f_J\colon D(J)\to A$ for every object $J$ of \Cat J and with the property that for every morphism $j\colon J\to J'$ the following commutes:
\[
\begin{tikzcd}[row sep=1 ex]
 D(J) \arrow[dd, swap, "{D(j)}"] \arrow[dr, "{f_J}"] & \\
    & A . \\
 D(J') \arrow[ur,swap, "{f_{J'}}"] & 
\end{tikzcd} 
\]
A \emph{colimit}\index{colimit} for a diagram $D\colon\Cat J\to\Cat C$ is an initial object in the category of cocones. That is, it is a cocone given by an object $\varinjlim D$ and morphisms $p_J\colon D(J)\to\varinjlim D$ with the property that for every other cocone $A$ there is an morphism $\bar f\colon \varinjlim D\to A$ such that 
\[
\begin{tikzcd}
D(J) \arrow[r, "{p_J}"] \arrow[dr,swap, "{f_J}"] & \varinjlim D  \arrow[d,"\bar f"] \\
    &  A 
\end{tikzcd}
\]
commutes. 

\begin{rmk}
Coproducts are colimits where the category \Cat J has no morphisms. Coequalizers are colimits where the category \Cat J consists of two objects and two parallel morphisms.  
\end{rmk}

A category is \emph{small}\index{small} if the collections of its objects and morphisms are sets. A category is \emph{small\=/cocomplete}\index{category!cocomplete} if every small diagram (\emph{i.e.}\ a diagram of shape \Cat J for a small category \Cat J) has a colimit. In particular, a small\=/cocomplete category admits small\=/coproducts (\emph{i.e.}\ coproducts of families of objects indexed over a set) and coequalizers. Vice versa, it is a standard fact that a category with small\=/coproducts and coequalizers is small\=/complete \cite[Chapter V]{MacLane}\index{category!complete}. Analogous facts and definitions are true when `small' is replaced by `finite'. We hence deduce that \Cat{Coarse} is finitely cocomplete. In this case, it is also convenient to show it explicitly:

\begin{lem}
 The category \Cat{Coarse} is finitely cocomplete.
\end{lem}
\begin{proof}[Sketch of proof]
 Let $D$ be a small diagram in \Cat{Coarse} of shape \Cat J. Let $D(J)=\crse X_J=(X_J,\CE_J)$ and for every morphism $j\colon J\to J'$ let  $D(j)=\crse j$. Considering the coproduct $\coprod_J \crse X_J$ we obtain natural inclusions $\crse i_J \colon\crse X_J\hookrightarrow \coprod_J\crse X_J$. To make the coproduct into a cocone it is necessary to change the coarse structure so that $\crse i_J =\crse i_{J'} \circ \crse j$ for every $j\colon J\to J'$. This is readily done by equipping $\coprod_J \crse X_J$ with the coarse structure
 \[
  \varinjlim\CE_J\coloneqq
  \angles*{\vphantom{\bigmid} \bigbraces{\CE_J\mid J\text{ object in \Cat J}}\cup \bigbraces{ i_J\times\paren{ i_{J'}\circ j}(\Delta_{X_J}) \mid j\colon J\to J'\text{ morphism in \Cat J}} }.
 \]
 By definition, $\varinjlim \CE_J$ is the minimal coarse structure making $\bigsqcup_J X_J$ into a cocone and hence $\varinjlim \crse X_J\coloneqq (\bigsqcup_J X_J,\varinjlim \CE_J)$ is a colimit.
\end{proof}

\

We now turn our attention to the dual notions of products, equalizers and limits. A positive result is that \Cat{Coarse} has arbitrary (small) products.

\begin{lem}
Every set of coarse spaces $\crse X_i$, $i\in I$ has a product $\prod_{i\in I}\crse X_i$ in \Cat{Coarse}. 
\end{lem}
\begin{proof}[Sketch of proof]
 Let $\crse X_i=(X_i,\CE_i)$ and consider the Cartesian product $\prod_{i\in I}X_i$. Define the coarse structure $\bigotimes_{i\in I}\CE_i$ containing a set $D\subseteq\paren{\prod_{i\in I}X_i}\times\paren{\prod_{i\in I}X_i}$ if and only if $\pi_j\times\pi_j(D)\in\CE_j$ for every $j\in I$, where $\pi_j$ is the projection to $X_j$.
 
 It is immediate to check that $\bigotimes_{i\in I}\CE_i$ is indeed a coarse structure. The projections $\pi_j\colon \prod_{i\in I}X_i\to X_j$ are controlled by definition and it is also immediate to check that, given controlled functions $f_i\colon Z\to X_i$, the product $\prod_{i\in I}f_i\colon Z\to \prod_{i\in I} X_i$ is controlled.  
\end{proof}

The situation is different for equalizers. An \emph{equalizer}\index{equalizer} of two morphisms $f,g\colon X\to Y$ in a category \Cat C, is an object $S$ together with a morphism ${\rm eq}\colon S\to X$ such that $f\circ{\rm eq}=g\circ{\rm eq}$ and is universal with this property:
\[
\begin{tikzcd}
 S \arrow[r, "{\rm eq}"] & X \arrow[r, shift left, "{f}"] \arrow[r, shift right, swap, "{g}"] & Y \\ 
 Z \arrow[ur,swap, "{f_Z}"]  \arrow[u,dashed]  & &
\end{tikzcd} 
\]

\begin{lem}[\cite{zava2019cowellpoweredness}]\label{noequalizer}
 In \Cat{Coarse}, equalizers do not necessarily exist.
\end{lem}
\begin{proof}[Sketch of proof]
 Consider $\id_\NN,\ubar 0\colon(\NN,\varcrs{min})\to (\NN,\CE_d)$, where $\id_\NN$ is the (set theoretic) identity function, $\ubar{0}$ is the constant function $0$, $\varcrs{min}$ is the minimal coarse structure and $\CE_d$ is the coarse structure coming from the natural metric. For every coarse space $\crse Z$ and controlled map ${f_{Z}\colon Z}\to (\NN,\varcrs{min})$  we have that $\cid_\NN\circ \crse{ f_{Z}}=\ubar{\crse 0}\circ \crse{ f_{Z}}$ if and only if $f_Z(Z)$ is finite. Yet, since $\abs{f_Z(Z)}$ can be arbitrarily large, there cannot exist an ${\rm eq}\colon S\to (\NN,\varcrs{min})$  such that $\cid_\NN\circ  \crse{{\rm eq}}=\ubar{\crse 0}\circ  \crse{{\rm eq}}$ and ${\rm eq}$ is universal.
\end{proof}

Since equalizers are limits, where limits are the dual notion of colimits, we deduce:

\begin{cor}
 The category \Cat{Coarse} need not have (finite) limits. In particular, it is not complete. 
\end{cor}

\begin{rmk}
 The absence of equalizers in \Cat{Coarse} is tied to the fact that it is not possible to define a notion of `intersection of coarse subspaces' that is both general and well behaved. We prefer to use a definition that is very well\=/behaved  but need not exist in general (Subsection~\ref{sec:p1:containements and intersections}).
\end{rmk}

\section{Subobjects and Quotients}
\label{sec:appendix:subobjects}

Recall that a morphism $f\colon X\to Y$ in a category is \emph{monic} \index{monic} (or, a \emph{monomorphism}) \index{monomorphism} if given any two morphisms $g_1,g_2\colon Z\to X$ such that $f\circ g_1=f\circ g_2$ then $g_1=g_2$. Dually, $f\colon X\to Y$ is \emph{epic} \index{epic} (or, an \emph{epimorphism}) \index{epimorphism} if given any two morphisms $g_1,g_2\colon Y\to Z$ such that $g_1\circ f=g_2\circ f$ then $g_1=g_2$.
The following result is proved in detail in \cite{dikranjan_categorical_2017}:

\begin{lem}\label{lem:appendix:monomorphisms.and.epimorphisms}
 Let $\crse {f \colon X\to Y}$ be a coarse map. 
 \begin{enumerate}[(i)]
  \item $\crse {f}$ is a monomorphism if and only if it is a coarse embedding,
  \item $\crse{f} $ is an epimorphism if and only if it is coarsely surjective.
 \end{enumerate}
\end{lem}
\begin{proof}[Sketch of proof]
 $(i)$: it is clear that a coarse embedding is a monomorphism. Lemma~\ref{lem:p1:trivially coarse spaces characterise maps} shows that the converse implication is true in a strong sense (it is enough to check trivially coarse spaces).

 $(ii)$: it is once again clear that a coarsely surjective coarse map is an epimorphism. 
 For the converse implication, fix a representative $f\colon (X,\CE)\to (Y,\CF)$. Let $i_1,i_2\colon Y\to Y\sqcup Y$ be the inclusions of $Y$ into each component of $Y\sqcup Y$, and consider $Y\sqcup Y$ with the coarse structure $\CD$ generated by the partial coverings $i_1(\fka_1), i_2(\fka_2)$ and $(i_1\cup i_2)(f(\pts{X}))\coloneqq\braces{\braces{i_1\circ f(x),i_2\circ f(x)}\mid x\in X}$ where $\fka_1,\fka_2\in\fkC(\CF)$.
 Then $\crse i_1\crse{\circ f}=\crse i_2\crse{\circ f}$, but we claim that $\crse i_1=\crse i_2$ only holds if $f$ is coarsely surjective.

 In fact, let $\fkb\coloneqq(i_1\cup i_2)(f(\pts{X}))\cup \pts{Y_1\sqcup Y_2}$. Note that if $\fka_1$ and $\fka_2$ are partial coverings of $Y$, then
 $\st(i_1(\fka_1)\cup i_2(\fka_2),\fkb)$ is a refinement of $\st(\fkb,i_1(\fka_1\cup \fka_2)\cup i_2(\fka_1\cup\fka_2))$. 
 We can thus deduce from Lemma~\ref{lem:p1: generated coarse structure} that any controlled partial covering in $\fkC(\CD)$ must be a refinement of $\st(\fkb,i_1(\fka)\cup i_2(\fka))$ for some $\fka\in\fkC(\CD)$. One can then show that if $(i_1\cup i_2)(\pts{Y})$ is in $\fkC(\CD)$ then $\pts{Y}$ is a refinement of $\st(f(\pts{X}),\fka)$ and hence $f$ is coarsely surjective.
\end{proof}

Together with Lemma~\ref{lem:p1:coarse.eq.iff.surjective.coarse.emb} this implies the following:

\begin{cor}[{\cite{dikranjan_categorical_2017}}]
 The category \Cat{Coarse} is balanced (\emph{i.e} morphisms that are both monic and epic are isomorphisms).
\end{cor}

\begin{rmk}
  The coarse space $(Y\sqcup Y,\CD)$ used in the proof of Lemma~\ref{lem:appendix:monomorphisms.and.epimorphisms} is the coequalizer of the morphisms $\crse i_1\circ \crse f$ and $\crse i_2\circ \crse f$ going from $\crse X$ to $\crse Y\amalg\crse  Y$, where $\amalg$ denotes the coproduct.
\end{rmk}

\

We will now show that the definitions of coarse subspace and coarse quotient (Subsection~\ref{sec:p1:subspaces and quotients}) are compatible with the notions of subobjects and quotient objects. Recall the following:

\begin{de}\label{def:appendix:subobject}
 In a category \Cat{C}, let $f\colon A \to B$ and $f' \colon A' \to B$ be two monics with the same codomain. Then \emph{$f$ factors through $f'$}if there exists $h \colon A \to A'$ such that $f = f' \circ h$. 
We say that $f$ and $g$ are \emph{equivalent} if $f$ factors through $g$ and $g$ factors through $f$ (this is an equivalence relation). 
 A \emph{subobject}\index{subobject} of an object $B$ in \Cat{C} is an equivalence class of monics into $B$.
\end{de}

\begin{rmk}
 A subobject in a category \Cat{C} is \emph{not} an object in \Cat{C}. However, it does determine an object up to (natural) isomorphism: if $f\colon A\to B$ and $f\colon A'\to B$ are equivalent monics then $A\cong A'$.
\end{rmk}

Applying Definition~\ref{def:appendix:subobject} to the category \Cat{Coarse} provides us with an abstract notion of \emph{coarse subspace}.
This notion is compatible with our definition of coarse subspace (Definition~\ref{def:p1:coarse subspace}). In fact, to any monic  $\crse{f\colon X\to Y}$ we can associate the coarse image $\crse{f(X)\subseteq Y}$. If and $\crse{f' \colon Y'\to  X}$ is an equivalent monic then $\crse{f(Y)= f'(Y')}$.
Vice versa, if we are given a coarse subspace $\crse{Y\subseteq X}$ then the inclusion $\iota\colon Y\hookrightarrow X$ gives a monic $\crse{\iota} \colon \crse Y\to \crse X$. Our claim that a coarse subspace uniquely defines a coarse space up to natural coarse equivalence is essentially equivalent to the statement that any other choice of representative $Y'$ for $\crse Y$ gives rise to an equivalent monic $\crse{\iota'\colon Y'\to X}$.

\

The situation for quotients is analogous: dually to Definition~\ref{def:appendix:subobject}, one can define an equivalence relation on epic morphisms in a category \Cat{C}, and define a \emph{quotient object}\index{quotient!object} as an equivalence class of epimorphisms. Once again we claim the definition of coarse quotient is compatible with the categorical notion of quotient in \Cat{Coarse}.
Given a controlled function $f\colon(X,\CE)\to (Y,\CF)$, the pull\=/back $f^*(\CF)$ is a coarse structure containing $\CE$ (Subsection~\ref{sec:p1:pullback_pushforward}). Since $f^*(\CF)$ is invariant under taking close maps, it follows that for any fixed $\crse X$
 the pull\=/back gives a well\=/defined map 
 \[
  \braces{\crse{f\colon X\to Y}\mid\text{epic}}/_{\text{equivalence}}\xrightarrow{\text{pull\=/back}}\braces{\CE'\text{ coarse structure on $X$}\mid \CE\subseteq\CE'}.
 \] 
 This map is clearly surjective because $\id_X\colon (X,\CE)\to (X,\CE')$ gives an epimorphism in \Cat{Coarse} whenever $\CE\subseteq\CE'$. It is also injective, because $q\colon (X,q^*(\CF))\to (Y,\CF)$ is a coarse equivalence whenever $q \colon (X,\CE)\to (Y,\CF)$ is coarsely surjective, hence $[q] \colon (X,\CE)\to (Y,\CF)$ is equivalent to $[\id_X] \colon (X,\CE)\to (X,q^*(\CF))$.

\section{The Category of Coarse Groups}\label{sec:appendix:category of coarse groups}
Since composition of coarse homomorphisms is a coarse homomorphism, we may define define a category of coarse groups.

\begin{de}\label{def:p1:category.coarse.groups}
 The \emph{category of coarse groups}\index{category of!coarse groups} (denoted \Cat{CrsGroup})\nomenclature[:CAT]{\Cat{CrsGroup}}{category of coarse groups} is the category with objects that are coarse groups and morphisms that are coarse homomorphisms:
 \[
  {\rm Mor}_{\Cat{CrsGroup}}(\crse{G},\crse H)\coloneqq\chom(\crse{G},\crse H).
 \]
\end{de}

This category is not as well behaved as the coarse category. At least, not as far as monomorphisms and epimorphisms are concerned. For set\=/groups, it is well\=/known that the standard ``set-theoretical'' definition of subgroup and quotient group is compatible with the categorical definition of subobject and quotient object in the category of set\=/groups. This is not true for coarse groups.

By Corollary~\ref{cor:p1:coarse subgroup iff coarse image}, the coarse subgroups of a coarse group $\crse G$ coincide with the coarse images of coarse homomorphisms $\crse{f \colon H\to G}$. 
Of course, if a coarse homomorphism is a coarse embedding then it is a monomorphism in \Cat{CrsGroup} (because it is a monomorphism in \Cat{Coarse}). Since the inclusion of a coarse subgroup $\crse{ H\leq G}$ into $\crse G$ is a coarse embedding, we then see that every coarse subgroup can be identified with the coarse image of a monic in \Cat{CrsGroup}. However, there are different monics with the same coarse image! 
For instance, we claim that identity function $\id_\RR\colon (\RR,\mincrs)\to (\RR,\CE_{\abs{\mhyphen}})$ is a monic in \Cat{CrsGroup}. This is because every coarse homomorphisms $\crse f\colon \crse G\to (\RR,\mincrs)$ must satisfy $f(g\ast g)=2f(g)$ for every $g\in G$. Therefore, if we are given coarse homomorphisms $\crse f_1,\crse f_2\colon G\to(\RR,\mincrs)$ such that $f_1$ and $f_2$ are $\CE_{\abs{\mhyphen}}$\=/close then it is easy to see that $f_1$ and $f_2$ must coincide pointwise. As a consequence, the two coarse homomorphisms
\[
 [\id_\RR]\colon(\RR,\mincrs)\to (\RR,\CE_{\abs{\mhyphen}})
 \qquad \qquad 
 [\id_{\RR}]\colon (\RR,\CE_{\abs{\mhyphen}})\to (\RR,\CE_{\abs{\mhyphen}})
\]
are two inequivalent monics with the same coarse image.

\begin{rmk}
 The coarse homomorphism $[\id_\RR]\colon(\RR,\mincrs)\to (\RR,\CE_{\abs{\mhyphen}})$ is a monic that is (set-wise) surjective. In particular, it is also an epimorphism in \Cat{CrsGroup}. Since it is not an isomorphism, this shows that \Cat{CrsGroup} is not a balanced category.
\end{rmk}

 The situation for quotient coarse subgroups seems similar: it is clear that every quotient object in \Cat{CrsGroup} determines a coarse subgroup and we do not see any obvious reason to expect that different quotient objects should give rise to different coarse quotient subgroups.
 However---unlike the case of coarse subgroups---we do not know of any example where this happens. In other words, we do not know the answer to the following:  
 
 \begin{qu}\label{qu:p1:are.epimorphisms.coarsely.surjective}
  If $\crse{q}\colon \crse G\to \crse Q$ is an epimorphism in \Cat{CrsGroup}, is it an epimorphism in \Cat{Coarse}? That is, are epimorphisms of \Cat{CrsGroup} necessarily coarsely surjective?
 \end{qu}

\begin{rmk}
 The easiest way to show that epimorphisms in \Cat{Group} are surjective is by noting that the image of a homomorphism $q\colon G\to Q$ is a subgroup of $Q$ and then considering the amalgamated product $Q\ast_{q(G)}Q$. If $q(G)\neq Q$, the two inclusions of $Q$ into $Q\ast_{q(G)}Q$ show that $q$ is not an epimorphism. The difficulty in adapting this proof to the context of coarse groups is that there does not appear to be a good analogue for the amalgamated product. More generally, it is not clear to us whether \Cat{CrsGroup} admits free products and/or free objects.
\end{rmk}

\section{The Category of Fragmentary Coarse Spaces}
\label{sec:appendix:frag coarse}
The study of spaces of controlled maps naturally led us to introduce a weakening of the notion of coarse structure and to define fragmentary coarse spaces (Section~\ref{ch:p2:spaces of controlled maps}). Here we explain more in detail some of the properties of the category of fragmentary coarse spaces. Recall that a fragmentary coarse structure on $X$ is a set of relations $\CE$ satisfying all the requirements of a coarse structure, except that it may not contain the diagonal. A fragment of $X$ is a subset $Z$ such that $\Delta_Z\in\CE$.
By assumption, every singleton of $X$ is a fragment. Controlled functions are defined as usual, while the notion of closeness is replaced by that of frag\=/closeness (two functions are frag\=/close if their restrictions to every fragment are close).\footnote{%
 More generally, we say that any property $P$ whose definition requires considering the diagonal relation on $X$ has a fragmentary analogue frag\=/$P$ which is obtained by requiring that every fragment of $X$ satisfies $P$. We already provided a few examples with precise definitions in Chapter~\ref{ch:p2:spaces of controlled maps}.
}
The category \Cat{FragCrs} has frag\=/coarse spaces as objects and frag\=/closeness equivalence classes of controlled maps as morphisms.

With regards to limits and colimits, much of the material in Section~\ref{sec:appendix:limits.and.colimits} generalizes to \Cat{FragCrs}. Actually, from this point of view the category \Cat{FragCrs} is better behaved than \Cat{Coarse}. For instance, it contains arbitrary (small) coproducts: if $(X_i,\CE_i)$ is a family of frag\=/coarse spaces indexed over some set $I$ their coproduct is the \emph{fragmented union}\index{fragmented union}
\[
 \coprod_{i\in I}(X_i,\CE_i)\coloneqq\Bigparen{{\textstyle\bigsqcup_{i\in I}X_i}\; ,\; \angles{\CE_i\mid i\in I}_{\rm frag}},
\]
where $\angles{\variable}_{\rm frag}$ denotes the generated frag\=/coarse structure (it differs from the generated coarse structure in that we do not add the diagonal. In particular, each fragment of the coproduct is completely contained in a unique $X_i$). This construction defines a coproduct in \Cat{FragCrs} because the frag\=/closeness condition is verified on one fragment at the time, hence the cardinality of $I$ does not matter.
In \Cat{FragCrs}, the coequalizer of two controlled functions $f,g\colon (X,\CE)\to(Y,\CF)$ between frag\=/coarse spaces can be defined equipping $Q\coloneqq Y$ with the fragmentary coarse structure $\CD$ generated by $\CF\cup \braces{\paren{f\times g}(\Delta_Z)\mid Z\text{ fragment of }X}$ and considering the set-theoretic identity function $\id_Y\colon Y\to Q$. We thus deduce the following:

\begin{lem}
 The category \Cat{FragCrs} is small-cocomplete. 
\end{lem}

With regard to limits, the same construction of \Cat{Coarse} shows that \Cat{FragCrs} has products. In particular, if $\crse X_i$ is a family of coarse spaces then their product in \Cat{Coarse} is also a product in \Cat{FragCrs} (as we saw, this is not the case for coproducts). Unlike \Cat{Coarse}, the category \Cat{FragCrs} also has equalizers:

\begin{lem}
 The category \Cat{FragCrs} has equalizers.
\end{lem}
\begin{proof}[Sketch of proof]
 Let $\crse f,\crse g \colon (X,\CE)\to(Y,\CF)$ be frag\=/coarse maps of frag\=/coarse spaces. For every controlled set $F\in\CF$ let $C_F\coloneqq\braces{x\in X\mid (f(x),g(x))\in F}\subseteq X$. Consider the union 
 \[
  S\coloneqq\bigcup_{F\in\CF} C_F \subseteq X
 \]
 and equip it with the fragmentary coarse structure 
 \(
  \CC\coloneqq\angles{\CE|_{C_F}\mid F\in \CF}_{\rm frag}.
 \)
 
 We claim that the inclusion $\iota\colon (S,\CC)\hookrightarrow (X,\CE)$ is an equalizer. It is clear that $\crse {f\circ \iota}=\crse {g\circ \iota}$. If $f_Z\colon (Z,\CD)\to(X,\CE)$ is a controlled map with $\crse {f\circ f_{Z}}=\crse {g\circ f_{Z}}$, we must have that $f_Z(Z)\subseteq S$. This is because each fragment of $Z$ must be sent into $C_F$ for some large enough $F\in\CF$, and   $Z$ is equal to the union of its fragments. We now see that $f_Z\colon (Z,\CD)\to(S,\CC)$ is controlled because for every $D\in\CD$ the union of the projections $Z'\coloneqq\pi_1(D)\cup\pi_2(D)$ is a fragment of $Z$ by Lemma~\ref{lem:p2:projections are fragments} and hence $f_Z\times f_Z(D)\in\CE|_{f(Z')}\subseteq \CC$.
\end{proof}

\begin{cor}
 The category \Cat{FragCrs} is small\=/complete.
\end{cor}

As we already remarked, various of the concepts and definitions of Chapter~\ref{ch:p1:intro to coarse category} translate without difficulty to frag\=/coarse spaces. One exception is the relation between controlled maps of binary products and equi controlledness (see the discussion following Corollary~\ref{cor:p2:frag_controlled iff controlled on fragments}). Other exceptions are coarse embeddings and coarsely surjective maps: this is because we wish to preserve their relation with monomorphisms and epimorphisms. Namely, it makes sense to say that $f\colon (X,\CE)\to (Y,\CF)$ is a coarse embedding if 
$(f\times f)^{-1}(F)\in\CE$ for every $F\in\CF$, even if $\crse X$ and $\crse Y$ are frag\=/coarse spaces. However, in this case Lemma~\ref{lem:p1:trivially coarse spaces characterise maps} is no longer true (the same happens for coarse surjectivity). In view of this, it is advisable to give the following:

\begin{de}
\label{de:appendix:frag coarse embedding}
 A controlled map between frag\=/coarse spaces $f\colon (X,\CE)\to (Y,\CF)$ is a \emph{frag\=/coarse embedding}\index{fragmentary!coarse embedding} if for every fragment $Z\subseteq X$ and $F\in\CF$ the intersection $(f\times f)^{-1}(F)\cap (Z\times X)$ belongs to $\CE$.
\end{de}
 
\begin{exmp}
 To see why we need to modify the definition of coarse embedding, consider the following example. On $\RR$, consider the frag\=/coarse structure $\CE\coloneqq\braces{E\mid \exists t>0,\ E\subseteq [-t,t]\times[-t,t]}$ (as in Example~\ref{exmp:p2:frag bounded}). The identity function $\id_\RR\colon(\RR,\CE)\to(\RR,\CE_{\abs{\mhyphen}})$ is a monomorphism in \Cat{FragCrs}, but it is not true that $(f\times f)^{-1}(F)$ belongs to $\CE$ when $F\in \CE_{\abs{\mhyphen}}$.
 
 A reasonable guess for the definition of frag\=/coarse embedding would have been to require that the restriction of $f$ to each fragment of $X$ is a coarse embedding. However, taking the frag\=/coarse map $\crse X\amalg\crse X\to \crse X$ which collapses together the fragmented union of two copies of the same frag\=/coarse space shows that the resulting notion does not imply monicity.
 Definition~\ref{de:appendix:frag coarse embedding} takes care of both difficulties.
\end{exmp}

The definition of frag\=/coarse surjectivity requires some care as well:
 
\begin{de}
 A controlled map between frag\=/coarse spaces $f\colon (X,\CE)\to (Y,\CF)$ is \emph{frag\=/coarsely surjective}\index{fragmentary!coarsely surjective} if for every fragment $Y'\subseteq Y$ there exists a fragment $X'\subseteq X$ and $F\in\CF$ so that $Y'\subseteq F(f(X'))$.
\end{de}

Of course, these notions are invariant under taking frag\=/close maps and they are hence properties of frag\=/coarse maps.
Using this definition, it is not hard to adapt the proofs of Lemma~\ref{lem:p1:trivially coarse spaces characterise maps}, Lemma~\ref{lem:p1:coarse.eq.iff.surjective.coarse.emb} and Lemma~\ref{lem:appendix:monomorphisms.and.epimorphisms} to show that their frag\=/coarse analogs hold true. We thus deduce the following:

\begin{prop}
 Let $\crse {f\colon X \to Y}$ be a frag\=/coarse map between frag\=/coarse spaces. Then
 \begin{enumerate}[(i)]
  \item $\crse f$ is a monomorphism if and only if it is a frag\=/coarse embedding,
  \item $\crse f$ is an epimorphism if and only if it is frag\=/coarsely surjective,
  \item $\crse f$ is a frag\=/coarse equivalence if and only if it is a frag\=/coarsely surjective coarse embedding.
 \end{enumerate}
 In particular, \Cat{FragCrs} is a balanced category (\emph{i.e.}\ every morphism that is both monic and epic is an isomorphism).
\end{prop}

\

The main result of Section~\ref{sec:p2:controlled_frag_crs} is a theorem on the properties of exponential frag\=/coarse spaces (Theorem~\ref{thm:p2:exponential}). The language and conventions used there are motivated by the categorical definition of exponential object:  

\begin{definition}
Let $X$ and $Y$ be objects of a category \Cat{C} in which all binary products with $X$ exist. Then an \emph{exponential object}\index{exponential object} is an object $Y^{X}$ in \Cat{C} equipped with an evaluation morphism $\textrm{ev}\colon Y^{X}\times X \to Y$ which is universal in the sense that, given any object $Z$ and morphism $f \colon Z \times X\to Y$, there exists a unique morphism $u\colon Z\to Y^{X}$ such that the composition
\[Z\times X \xrightarrow{ u \times \id_X} Y^{X} \times X \xrightarrow{\textrm{ev}} Y\] 
equals $f$. The morphism $u$ is denoted by $\lambda f$ and is called the \emph{transpose}\index{transposed function} of $f$.
\end{definition} 

If it exists, an exponential object is of course unique up to isomorphism. The main content of Theorem~\ref{thm:p2:exponential} can be rephrased as follows:

\begin{thm}\label{thm:appendix:exponential}
 In \Cat{FragCrs} every pair $\crse X$ and $\crse Y$ has an exponential object $\crse{Y^{X}}$.
\end{thm}

This implies that \Cat{FragCrs} is a particularly `nice' category. More precisely, recall the following:

\begin{definition}   
 A category is \emph{Cartesian closed}\index{Cartesian closed} if it contains a terminal object, it has all binary products, and the exponential of any two objects.  
\end{definition} 

Since \Cat{FragCrs} has all finite products, Theorem~\ref{thm:appendix:exponential} implies that it is Cartesian closed:

\begin{cor} 
 The category \Cat{FragCrs} is Cartesian closed. 
\end{cor} 

\section{Enriched Coarse Categories}\label{sec:appendix:enriched coarse}
Informally, an enriched category\index{category!enriched} is a category where the sets of morphisms have some additional structure. A prototypical example is the category \Cat{Vect} of vector spaces over some fixed field, as the set of linear maps between two vector spaces is itself a vector space. In other words, \Cat{Vect} is enriched over itself. It is a general fact that a Cartesian closed category is enriched over itself (see \cite[Section~1.6]{kelly1982basic}), therefore \Cat{FragCrs} is enriched just as \Cat{Vect} is.
Further, it is easy to see \Cat{Coarse} can be enriched over \Cat{FragCrs}.
Rather than giving a precise definition of enriched category, we will content ourselves with a hands\=/on description of the extra structure on the sets of morphisms of \Cat{Coarse}. The reader interested in the abstract formalism may refer \cite{kelly1982basic}.

As mentioned above, the `enriched' structure of \Cat{Vect} is easy to describe: if $V$ and $W$ are two vector spaces it is easy to define operations on the set of linear maps $\CL(V,W)=\mor{Vect}(V,W)$ endowed with which it is a vector space. 
The enriched structure of \Cat{Coarse} is more subtle, because it is not possible to equip the set  $\cmor(\crse{X,Y})$ with a good (frag\=/)coarse structure. Instead, we need to consider the exponential frag\=/coarse space $\crse{Y^{X}}$. 

More precisely, we define an enriched version of the coarse category as the category \Cat{Coarse$^{en}$} having coarse spaces as objects and so that the ``space of morphisms'' $\mor{Coarse$^{en}$}(\crse{ X,Y})$ is the frag\=/coarse space $\crse{Y^{X}}$.
The composition in \Cat{Coarse$^{en}$} is defined by the composition morphism $[\circ]\colon \crse{Z^{Y}\times Y^{X}\to Z^{X}}$, \emph{i.e.}\ the frag\=/coarse map associated with the composition $\circ\colon(f,g)\mapsto f\circ g$. 
Under this definition \Cat{Coarse$^{en}$} is a \emph{category enriched over \Cat{FragCrs}} (or \emph{\Cat{FragCrs}\=/category}) \cite{kelly1982basic}.

Note that it is not technically correct to talk about ``set of morphisms'' in \Cat{Coarse$^{en}$}. This is because an enriched category is not really a category: the morphisms are not a ``set'', but an object in a different category (in this case \Cat{FragCrs}). In this setting, the correct terminology is that $\crse{Y^{X}}$ is the \emph{hom\=/object}\index{hom\=/object} from $\crse X$ to $\crse Y$. In particular, it is improper to think of an element $f\in \crse{Y^{X}}$ as a morphism in $\mor{Coarse$^{en}$}(\crse{ X,Y})$.
The formal way to recover the original category \Cat{Coarse} from its enriched version \Cat{Coarse$^{en}$} is by recalling that a coarse map $\crse{f\colon X\to Y}$ is a frag\=/closeness equivalence class of controlled functions $f\in \ctrl{X}{Y}$, and that this equivalence class can be described as the frag\=/closeness class of the function $\tobj\to \crse{Y^{X}}$ sending ${\rm pt}$ to $f$.
That is, the sets of morphism between two objects $\crse{X}$, $\crse Y$ in \Cat{Coarse} can be identified with frag\=/coarse maps from the terminal object $\tobj$ to the hom\=/object $\crse{Y^{X}}$:
\[
 \cmor(\crse{X,Y})\cong \fcmor(\tobj,\crse{Y^{X}}).
\]

Of course, the same construction can be performed to define the category \Cat{FragCrs$^{en}$} enriched over \Cat{FragCrs} (this is the standard way of enriching a Cartesian closed category over itself).

\section{The Pre-coarse Categories}\label{sec:appendix:enriched pre_coarse}
At times it may be useful to refrain from identifying close functions. We conclude this categorical appendix by noting that the categories whose morphisms are controlled maps (as opposed to equivalence classes thereof) have a natural enriched structure.

\begin{de}
 The category of \emph{precoarse spaces} \Cat{PreCoarse}\index{category of!precoarse spaces} (resp. of \emph{prefrag\=/coarse spaces} \Cat{PreFragCrs}\index{category of!prefrag-coarse spaces}) has coarse spaces (resp. frag\=/coarse spaces) as objects and controlled maps as morphisms.
\end{de}

\begin{rmk}
 By definition, we have a functor $\Cat{PreCoarse}\to\Cat{Coarse}$ preserving each coarse space and sending a controlled function $f$ to its equivalence class $\crse f=[f]$. The same is of course true for \Cat{PreFragCrs} and \Cat{FragCrs}. 
\end{rmk}

As in Section~\ref{sec:appendix:enriched coarse}, also \Cat{PreCoarse} and \Cat{PreFragCrs} can be enriched over \Cat{PreFragCrs}. In this case the situation is even simpler to describe.
Namely, given coarse spaces $\crse X=(X,\CE)$ and $\crse Y=(Y,\CF)$, the set of morphisms $\mor{PreCoarse}(\crse{X,Y})$ equal to $\ctrl{X}{Y}$ by definition. We can hence directly equip it with the frag\=/coarse structure $\CF^\CE$. Namely, we are entitled to identify $\mor{PreCoarse}(\crse{X,Y})$ with $\crse{Y^{X}}$. 

Unlike for \Cat{Coarse}, here it is legitimate to say that an element $f\in\crse{Y^{X}}$ is a morphism. This is the case because $f$ can be uniquely identified with the morphism $\mor{PreFragCrs}(\tobj,\crse{Y^{X}})$ sending $\rm pt$ to $f$.

\chapter{Metric Groups and Quasifications}%%%%%%%%%%%%%%%%%%%%%%%%%%%%%%
\label{ch:appendix:metric.groups}

\section{Quasifications of Metric Spaces}
\label{sec:appendix:quasifications}

We introduced coarse geometry as a language to describe ``geometric properties'' that are stable up to uniformly bounded error. On metric spaces there is another natural approach to achieve similar results, namely via ``quasifications''. Quasifications are generally favored in geometric group theory.

Recall that a function between metric spaces $f\colon(X,d_X)\to(Y,d_Y)$ is $(L,A)$\=/quasi\=/Lipschitz if $d_Y(f(x),f(x'))\leq L d_X(x,x')+A$ for every $x,x'\in X$. Of course, quasi\=/Lipschitz functions are a special class of controlled maps and the property of being quasi\=/Lipschitz is preserved under passing to close functions. We may define a category of quasi\=/metric spaces \Cat{QuasiMet} having metric spaces as objects and equivalence classes of quasi\=/Lipschitz maps as morphisms, where functions are equivalent if they are close.
Two metric spaces are isomorphic in \Cat{QuasiMet} if and only if they are quasi\=/isometric.

Note that ``quasifications'' of metric spaces generally contain much more information than the ``coarsification''. For instance, if $(X,d)$ is a metric space of infinite diameter and we consider the metric $d'\coloneqq \log(1+d)$, then the identity is a coarse equivalence between $(X,d)$ and $(X,d')$ but not a quasi\=/isometry.
Informally, this loss of information comes from the fact that coarse structures do not encode ``how large'' a controlled set is. Namely, for any subset of $(X,d)$ we can say whether it is bounded but we do not have any information whatsoever on its diameter.

\begin{rmk}
 One may try to build an abstract theory of quasi-spaces that includes non\=/metric examples. The idea would be to consider a coarse space $(X,\CE)$ together with some indexing $\CE\to\RR_+$ which plays the role of the diameter (this indexing would have to satisfy some compatibility conditions, \emph{e.g.}\ being increasing and subadditive under composition). We could then restrict to consider only those morphisms that preserve the indexing up to an affine factor.
 
 Taking this point of view, it could also be interesting to replace $\RR_+$ with some other fixed ordered set $(I,\leq)$ and to consider morphisms that preserve the indexing up to some other equivalence relation.
\end{rmk}

It is worthwhile remarking that the theories of ``quasifications'' and ``coarsification'' are equivalent when restricting to (quasi) geodesic metric spaces. This fact was used more or less implicitly in the discussion of the Milnor--Schwarz Lemma (Subsection~\ref{sec:p1:Milnor_svarc}).

\begin{de}
 A metric space $(X,d)$ is \emph{quasi\=/geodesic}\index{quasi-!geodesic} if it is quasi\=/isometric to a geodesic metric spaces.
\end{de}

Given an interval $I\subseteq\RR$, $L\geq 1$ and $A\geq 0$, a function $\gamma\colon I\to (X,d)$ is an \emph{$(L,A)$\=/quasi\=/geodesic} if it is an $(L,A)$\=/quasi\=/isometric embedding. That is,
\[
 \frac 1L \abs{t-t'}-A \leq d(\gamma(t),\gamma(t'))\leq L \abs{t-t'} +A 
\]
for every $t,t'\in I$.
A function $f\colon(X,d_X)\to(Y,d_Y)$ is an \emph{$(L,A)$\=/quasi\=/isometry}\index{quasi-!isometry} if it is $(L,A)$\=/quasi\=/Lipschitz and it has an $(L,A)$\=/quasi\=/Lipschitz quasi\=/inverse $g\colon (Y,d_Y)\to(X,d_X)$ such that $g\circ f$ and $f\circ g$ are at distance at most $A$ from $\id_X$ and $\id_Y$. It is simple to verify that an $(L,A)$\=/quasi\=/isometry is an $(L,3A)$\=/quasi\=/isometric embedding. 

\begin{lem}
\label{lem:appendix:quasi geo iff quasi geo}
 A metric space $(X,d_X)$ is quasi\=/geodesic if and only if there are $L,A\in\RR_+$ such that any two points of $X$ can be joined by an $(L,A)$\=/quasi\=/geodesic.
\end{lem}
\begin{proof}
 Assume that there exist a geodesic metric space $(Y,d_Y)$ and an $(L,A)$\=/quasi\=/isometry $f\colon(X,d_X)\to(Y,d_Y)$ with quasi\=/inverse $g\colon (Y,d_Y)\to(X,d_X)$. For any two points $x,x'\in X$, let $l\coloneqq d_Y(f(x),f(x'))$ and choose a geodesic $\gamma\colon [0,l]\to (Y,d_Y)$ connecting $f(x)$ to $f(x')$. Define a function $\tilde \gamma\colon[0,l]\to X$ by
 \[
 \tilde\gamma(t)\coloneqq\left\{
  \begin{array}{ll}
   x & t=0; \\
   g(\gamma(t)) & 0<t<l; \\
   x' & t=l.
  \end{array}
 \right.
 \]
 Since $g$ is a $(L,3A)$\=/quasi\=/isometric embedding, it is immediate to verify that $\tilde\gamma$ is a $(L,4A)$\=/quasi-geodesic.
 
 For the other implication, assume that every pair of points $x,x'\in X$ can be joined by a $(L,A)$\=/quasi-geodesic. Construct a graph $Y$ using $X$ as vertex set and connecting two vertices $x,x'$ by an edge if and only if $d(x,x')\leq 2A$. Declare each edge to have length $1$ and let $d_Y$ be the induced path metric on $Y$. It is not hard to verify that the function $(X,d_X)\to(Y,d_Y)$ sending each point to the corresponding vertex is a quasi\=/isometry.
\end{proof}

The fact that the distance between pairs of points in quasi\=/geodesic metric spaces can be ``quasi\=/realized'' by a quasi\=/geodesic has the following consequence:

\begin{lem}
 If $(X,d_X)$ is a quasi\=/geodesic metric space and $f\colon (X,d_X)\to (Y,d_Y)$ is a controlled map, then $f$ is quasi\=/Lipschitz.
\end{lem}
\begin{proof}
 Let $L,A$ be constants so that each pair of points in $(X,d_X)$ can be joined by an $(L,A)$\=/quasi\=/geodesic. Let $C$ be so that $d(f(x),f(x'))\leq C$ whenever $d(x,x')\leq 1+A$.
 Now, for any fixed a pair of points $x,x'\in X$ let $\gamma\colon [0,l]\to (X,d_X)$ be a $(L,A)$\=/quasi\=/geodesic connecting them. Note that $d(x,x')\geq \frac lL -A$.
 Let $n\coloneqq \lfloor lL \rfloor+1$ and pick $0=t_0<\cdots <t_n=l$ with $\abs{t_{i-1},t_i}< \frac 1L$.  
 Then $d_X(f\circ \gamma(t_{i-1}),f\circ \gamma(t_i))\leq C$ for every $i=1,\ldots, n$. By triangle inequality, $d_Y(f(x),f(x'))\leq nC\leq lLC+C\leq d_X(x,x')L^2C+L^2CA+C$.
\end{proof}

\begin{cor}
\label{cor:appendix:quasi geo quasi Lipschitz}
 If $(X,d_X)$ and $(Y,d_Y)$ are quasi\=/geodesic metric spaces and $f\colon (X,d_X)\to (Y,d_Y)$ is a coarse equivalence then $f$ is a quasi\=/isometry.
\end{cor}

In turn, we deduce the following \editC{(see also \cite[Theorem 9.2]{protasov2003ball} and \cite[2.57]{roe_lectures_2003})}:

\begin{prop}
\label{prop:appendix:coarse geo has unique quasi metric}
 If $(X,\CE)$ is coarsely geodesic, then there exists a metric $d$ on $X$ such that $\CE=\CE_d$ and $(X,d)$ is quasi\=/geodesic. The metric $d$ is uniquely defined up to quasi\=/isometry. 
\end{prop}
\begin{proof}
 Let $E$ be a generating relation for $\CE$ and for every $x\neq x'$ let $d(x,x')\coloneqq \min\{n\mid (x,x')\in E^{\cmp n}\}$. Then $\CE=\CE_d$. It follows from the definition that if $d(x,x')=n$ then there exist points $x=x_0,\ldots,x_n=x'$ such that $d(x_{i-1},x_i)=1$ for every $i=1,\ldots, n$. We can use these points to construct a $(1,1)$\=/quasi\=/geodesic between $x$ and $x'$, therefore $(X,d)$ is a quasi\=/geodesic metric space by Lemma~\ref{lem:appendix:quasi geo iff quasi geo}. 
 Uniqueness of $d$ up to quasi\=/isometry follows from Corollary~\ref{cor:appendix:quasi geo quasi Lipschitz}
\end{proof}

\begin{rmk}
 Since the quasi\=/isometry type of the metric $d$ constructed in Proposition~\ref{prop:appendix:coarse geo has unique quasi metric} is unique, this enables us to use quasi\=/isometry invariants as coarse invariants for coarsely geodesic spaces. This can be very useful, as many quantitative quasi\=/isometric invariants (\emph{e.g.}\ the growth type) do not have a coarse analog. 
\end{rmk}

\section{Groups in Metric Categories}
\label{sec:appendix:metric groups}

Even before studying coarse groups, it is natural to wonder if there is a notion of ``metric group'', defined as group object in a category of metric spaces. Interestingly, there is more than one natural choice for the `category of metric spaces'. In fact, even if it is clear what the objects should be, the choice of which morphisms to consider can vary accordingly to one's necessity. In this section we will consider two such choices and characterize the group objects for both.

In order to have group objects (Definition~\ref{def:p1:group.object}), a category must have binary products and a terminal objects. The one\=/point metric space will always provides us with a terminal object, while the situation is a little more delicate with binary products.
Given two metric spaces $(X,d_X)$ and $(Y,d_Y)$, the \emph{$p$\=/product metric}\index{$p$\=/product metric} on the Cartesian product $X\times Y$ is the metric given by
\[
 d_p\paren{(x,y),(x',y')}\coloneqq \paren{d_X(x,x')^p+d_Y(y,y')^p}^{\frac{1}{p}}.
\]
The \emph{$\infty$\=/product metric}\index{$\infty$\=/product metric} is the maximum $d_\infty\paren{(x,y),(x',y')}\coloneqq \max\paren{d_X(x,x'),d_Y(y,y')}$. Depending on the choice of morphisms, these product metrics need not be equivalent and they do not all define a (categorical) product.

\

The category \Cat{1\=/LipMet} is the category whose objects are metric spaces and whose morphisms are $1$\=/Lipschitz maps.\footnote{%
From a categorical point of view it is natural to restrict ones attention to $1$\=/Lipschitz maps: these are the morphisms that arise when considering metric spaces as enriched categories \cite{lawvere1973metric}.}
It follows from the triangle inequality that $(X\times Y,d_\infty)$ is a product of $(X,d_X)$ and $(Y,d_Y)$ in \Cat{1\=/LipMet}. On the other hand, if $p\neq\infty$ and neither $X$ nor $Y$ consist of a single point, then $(X\times Y,d_p)$ is never a product (in the categorical sense). In fact, the projections from $(X\times Y,d_\infty)$ to $X$ and $Y$ do not lift to a $1$\=/Lipschitz map $(X\times Y,d_\infty)\to (X\times Y,d_p)$.

Recall that an \emph{ultrametric}\index{ultrametric} on a set $X$ is a metric $d$ that satisfies the following strengthened triangle inequality:
\[
 d(x,z)\leq \max\{d(x,y),d(y,z)\}
\]
for every $x,y,z\in X$.

\begin{prop}\label{prop:appendix:group.objects.in.1.LipMet}
 A group object in \Cat{1\=/LipMet} is a set\=/group equipped with a bi\=/invariant ultrametric.
\end{prop}
\begin{proof}
 Functions between metric spaces define the same morphism if and only if they coincide pointwise. In particular, functions $\ast,\inversefn ,\unit$ making $(G,d)$ into a metric group must make it into a set\=/group as well.
 That is, the problem of characterizing group objects in \Cat{1\=/LipMet} is equivalent to the problem of characterizing the set of metrics making a set\=/group $G$ into a group object. 
 
 Let $(G,\ast,\unit,\inversefn )$ be a set\=/group. A metric $d$ on $G$ makes it into a group object if and only if it makes $\ast$ and $\inversefn $ into $1$\=/Lipschitz maps. For $\ast$ to be $1$\=/Lipschitz it is certainly necessary that both the left multiplication ${}_g\ast\colon G\to G$ and the right multiplication $\ast_g\colon G\to G$ be $1$\=/Lipschitz. Since their inverses $(\ast_g)^{-1}=\ast_{g^{-1}}$ and $({}_g\ast)^{-1} ={}_{g^{-1}}\ast$ are also $1$\=/Lipschitz, we deduce that the metric $d$ must be bi\=/invariant. Note that bi\=/invariance of $d$ immediately implies that $\inversefn $ is an isometry.
 
 We claim that $d$ must also be an ultrametric. Let $\abs{g}\coloneqq d(g,\unit)=d(\unit,g)$ and note that $\abs{g}=\abs{g^{-1}}$ for every $g\in G$ because $d$ is bi\=/invariant.
 By definition we have
 \[
  d_\infty\bigparen{(\unit,\unit),(g,h)}=\max\braces{d(\unit,g),d(\unit,h)}=\max\braces{\abs{g},\abs{h}}.
 \]
 If $\ast$ is $1$\=/Lipschitz it follows that for every $g,h,k\in G$ we have
 \[
  d(g,h)=\abs{g^{-1}h}\leq d_\infty\bigparen{(\unit,\unit),(g^{-1}k,k^{-1}h)}
  =\max\bigbraces{\abs{g^{-1}k},\abs{k^{-1}h}}=\max\braces{d(g,k),d(k,h)}.
 \]
 
 Vice versa, it is immediate to check that if $d$ is a bi\=/invariant ultrametric then $(G,d,\ast,\unit,\inversefn )$ is a group object.
\end{proof}

\

Another reasonable choice for the category of metric spaces is the category \Cat{LipMet} having all the Lipschitz maps as morphisms. Note that the $p$\=/product metrics are all bi\=/Lipschitz equivalent to one another, hence define isomorphic objects in \Cat{LipMet}. For concreteness, we keep using $\infty$\=/product metric to make $X\times Y$ into a categorical product. 

As for Proposition~\ref{prop:appendix:group.objects.in.1.LipMet}, group objects in \Cat{LipMet} must be set\=/groups. To characterize the metrics making an abstract group into a group object we need to introduce a piece of terminology. A family of maps $f_i\colon (X,d_X)\to (Y,d_Y)$ is \emph{equi Lipschitz}\index{equi!Lipschitz} if there is a constant $L$ such that all the $f_i$ are $L$\=/Lipschitz.

\begin{prop}\label{prop:appendix:group.objects.in.LipMet}
 A group object in \Cat{LipMet} is a group equipped with a metric so that both the left and the right multiplication are families of equi Lipschitz maps.
 Moreover, if $(G,d,\ast,\unit,\inversefn )$ is a group object then there is a bi\=/invariant metric $\tilde d$ that is bi\=/Lipschitz equivalent to $d$. 
\end{prop}
\begin{proof}
 Let $(G,\ast,\unit,\inversefn )$ be a group and $d$ a metric on $G$. If $\ast$ is $L$\=/Lipschitz, then so are ${}_g\ast$ and $\ast_g$ for all $g\in G$. Vice versa, if ${}_g\ast$ and $\ast_g$ are equi Lipschitz, say with Lipschitz constant $L$, then we have
 \[
  d(gh,g'h')\leq d(gh,gh')+d(gh',g'h')\leq L\paren{d(h,h')+d(g,g')}
  =L d_1\bigparen{(g,h),(g'h')}\leq 2L d_\infty\bigparen{(g,h),(g'h')}.
 \]
 The fact that also $\inversefn $ is Lipschitz is immediate.

 For the second part of the statement, let $\tilde d(h,h')\coloneqq \sup\braces{d(g_1hg_2,g_1h'g_2)\mid g_1,g_2\in G}$. Since left and right multiplication are equi Lipschitz with Lipschitz constant $L$, we see that $\tilde d$ is bi\=/Lipschitz equivalent to $d$ because
 \[
  d(h,h')\leq \tilde d(h,h')
  =\sup_{g_1\in G}\Bigparen{\sup_{g_2\in G} d(g_1hg_2,g_1h'g_2)}
  \leq \sup_{g_1\in G}\Bigparen{Ld(g_1h,g_1h')}
  \leq L^2 d(h,h').
 \]
 It is routine to verify that $\tilde d$ satisfies the triangle inequality and is indeed a bi\=/invariant metric on $G$.
\end{proof}

\begin{cor}
 Up to isomorphism, a group object in \Cat{LipMet} is a set\=/group equipped with a bi\=/invariant metric.
\end{cor}

\section{Quasi-Metric Groups}
\label{sec:appendix:quasigroups}
At this point it is all but natural to ask about group objects in the category of quasi\=/metric spaces (Subsection~\ref{sec:appendix:quasifications}).

\begin{de}
 A \emph{quasi\=/metric group}\index{metric group} is a group object in \Cat{QuasiMet}.
\end{de}

The usual arguments show that a quasi\=/metric group is a metric space $(G,d)$ with (equivalence classes of) functions $\ast,e,\inversefn$ for which there are constants $L\geq 1$ and $A\geq 0$ such that
\begin{enumerate}
 \item  ${}_g\ast$, $\ast_g$ and $\inversefn$ are $(L,A)$\=/quasi\=/isometries;
 \item $(g_1\ast g_2)\ast g_3$ is $A$\=/close to $g_1 \ast (g_2\ast g_3)$;
 \item $g\ast e$ and $e\ast g$ are $A$\=/close to $g$;
 \item $g\ast g^{-1}$ and $g^{-1}\ast g$ are $A$\=/close to $e$;
\end{enumerate}
for every $g,g_1,g_2,g_3\in G$.

A fair amount of the theory that we developed for coarse groups can be specialized to the case of quasi\=/metric groups. For instance, the quasification of a set\=/group $G$ with respect to a bi\=/invariant metric $d$ is a quasi\=/metric group. The restriction of $d$ to any approximate subgroup of $H\subset G$ can be used to define a compatible quasi\=/metric group structure of $H$, which is unique up to quasi\=/isometry.

If $G$ is a set\=/group and $d$ is a metric such that $(G,d)$ is a quasi\=/metric group, then we can argue as in the proof of Proposition~\ref{prop:appendix:group.objects.in.LipMet} to deduce that
\[
 \tilde d(h,h')\coloneqq \sup_{g_1,g_2\in G} d(g_1hg_2,g_1h'g_2)
\]
is a bi\=/invariant metric on $G$ that is quasi\=/isometric to $d$.
If $(G,d,[\ast],[e],[\mhyphen]^{-1})$ is an arbitrary quasi\=/metric group, we may fix representatives $\ast$, $e$ and $\inversefn$ and similarly define a new metric $\tilde d$ on $G$. This metric will be quasi\=/isometric to $d$, but it may fail to be bi\=/invariant under $\ast$. On the other hand, it can be shown that with respect to this metric the left and right multiplications are $(1,A')$\=/quasi\=/isometries for some large enough $A'$.

\

One useful feature of coarse groups is that every set\=/group can be realized as a trivially coarse group. To reproduce this feature in the world of quasi\=/metric groups we need to allow ourselves to use extended metric (with the obvious extension of the notations of quasi\=/Lipschitz, quasi\=/isometry, etc.). 
We may then declare that the trivial quasification of a set\=/group $G$ is the quasi\=/metric group obtained giving $G$ the extended metric taking value $+\infty$ on every pair of distinct points.

One reason for caring about trivial quasifications comes from the theory of quasi\=/actions. Quasi\=/actions of quasi\=/metric groups may be defined in analogy with coarse actions. It is then easy to observe that the quasi\=/actions of a set\=/group $G$ as usually understood in the literature are nothing but the quasi\=/actions of the trivial quasification of $G$.

\begin{rmk}
 We have functors
 \[
  \Cat{1-LipMet}\to\Cat{LipMet}\to\Cat{QuasiMet}\to\Cat{Coarse}.
 \]
 The fact that set\=/groups with bi-invariant metrics give rise to quasi\=/metric groups and coarse groups may thus be seen as a consequence of the fact that functors preserve group objects.
 Note that at each step of the above chain of functors there are many new group objects which do not belong to the image of the preceding functor. For instance, $(\RR,\varcrs[grp]{fin})$ cannot be obtained from a quasi\=/metric group because the coarse structure $\varcrs[grp]{fin}$ is not metrizeable. 
\end{rmk}

\chapter{Extra Topics}
\label{ch:appendix:extra topics}

\section{Computing Cancellation Metrics}
\label{sec:appendix:canc metric step by step}

 Here we will give a formal proof of Lemma~\ref{lem:p2:reordering moves}. The main issue is that addition and (especially) cancellation moves are very disruptive operations, which can greatly modify a reduced word. This makes them complicated to deal with. The approach that we take here is to consider also non\=/reduced words and do the various cancellations one step at the time. To do so, we must begin by introducing more notation.

 Fix a set $S$ of letters, $2\leq \abs{S}<\infty$. Define a new letter $\bar x$ for each $x$ in $S$, let $Q=\braces{\bar x\mid x\in S}\cup S$ and let $X$ be the set of words in the alphabet $Q$. We use the convention that $\bar{\bar x}=x$, so that for every $y\in Q$ the symbol $\bar y$ denotes again a letter in $Q$. 
We define a pseudo\=/metric $d_X$ on $X$ as the maximal pseudo\=/metric such that 
\begin{equation}\label{eq:appendix:def psudo metric}
 d_X(wxw', ww') \leq 1 \qquad d_X( w x\bar xw' , ww') =0,
\end{equation}
where $x \in Q$ is any letter and $w,w'$ are words in $X$ (namely, $d_X$ is obtained by imposing that both \eqref{eq:appendix:def psudo metric} and the triangle inequality hold).

Alternatively, $d_X$ can be described as follows. Consider the labeled directed graph having $X$ as the set of vertices and the following set of edges:
\begin{itemize}
\item for every word in $w\in X$ add edges labeled $A^+$ pointing to all the words $w'$ obtained by adding a letter $x\in Q$ to $w$.
Vice versa, for every word $w'$ obtained by removing one letter from $w$, add an edge from $w$ to $w'$ and label it $A^-$.

\item for every word in $w\in X$ add edges labeled by $T^+$ pointing to any word $w'$ obtained by adding to $w$ a \emph{cancelling pair} of letters $x\bar x$ for some $x\in Q$.
Vice versa, if a word $w$ contains a subsequent pair of cancelling letters add an edge from $w$ to the word obtained from $w$ by omitting those letters, and label this edge $T^-$.
\end{itemize} 
We declare that an edge has length $1$ if its label is $A^\pm$, and $0$ otherwise.
The distance $d_X(w,w')$ is the length of the shortest oriented path connecting them (here and in what follows, the \emph{length} of a path is the sum of the lengths of its edges, not the number of edges).

\begin{rmk}
 If identify the free group $F_S$ with the subset of \emph{reduced words} in $X$ (\emph{i.e} those words that do not have cancelling pairs). It is not hard to show that, under this identification, the restriction to $F_2$ of the pseudo\=/metric $d_X$ coincides with the cancellation metric $d_{\overline{S}}$ of Section~\ref{sec:p2:cancellation metric on free} (this can be done by showing that the restriction $d_X|_{F_2}$ is bi\=/invariant and that $d_X(e,g)$ coincides with the cancellation length $\abs{\variable}_{\overline{S}}$).
\end{rmk}

In the sequel it will be convenient to keep track explicitly of the labels of the edges of an oriented path in $X$. For this reason we will denote a path from $w_0$ to $w_n$ by
\[
 w_0\xrightarrow{l_1} w_1\xrightarrow{l_2}\cdots \xrightarrow{l_n}w_n,
\]
where $l_n$ is the label of the $n$\=/th edge. 
Note that there may be multiple edges connecting two words, for example $w=a$ is taken to $w'=aa$ by adding a letter $a$ either at the beginning or at its end of $w$. However, all those edges will have the same label, because the label is uniquely determined by the difference in length of the the two words. 
The following lemma is the reason why we introduced the space $(X,d_X)$:

\begin{lem}\label{lem:appendix:rearranging.labels}
 Let $w_0\xrightarrow{l_1} w_1\xrightarrow{l_2}\cdots \xrightarrow{l_n}w_n$ be any fixed path in $X$. We can then find another path going from $w_0$ to $w_n$ that is at most as long as the original path and has the form
 \[
 w_0\xrightarrow{A^-}\cdots \xrightarrow{A^-}w'_j\xrightarrow{T^-}\cdots \xrightarrow{T^-}w'_k\xrightarrow{T^+}\cdots \xrightarrow{T^+}w'_l\xrightarrow{A^+}\cdots \xrightarrow{A^+}w'_{m}=w_n
 \] 
 for some $0\leq j\leq k\leq l\leq m\leq n$.
\end{lem}
\begin{proof}
 The idea of proof is to show that we can perform a few `local moves' that allow us to `rearrange the labels' of the path without increasing the total length. The list of moves is as follows:
 
 \begin{move}
  Replace $w_{i-1}\xrightarrow{T^-} w_i\xrightarrow{A^-}w_{i+1}$ with $w_{i-1}\xrightarrow{A^-} w'_i\xrightarrow{T^-}w_{i+1}$.
 \end{move}
 
 That is, if $w_{i-1}\xrightarrow{T^-} w_i\xrightarrow{A^-}w_{i+1}$ then there exists a word $w'$ such that $w_{i-1}\xrightarrow{A^-} w'_i\xrightarrow{T^-}w_{i+1}$. We can hence use this deviation to obtain a new path from $w_0$ to $w_n$ having the same total length of the original one. We will use similar conventions to describe the other moves as well.
 
 \begin{move}
   Replace $w_{i-1}\xrightarrow{A^+} w_i\xrightarrow{T^+}w_{i+1}$ with $w_{i-1}\xrightarrow{T^+} w'_i\xrightarrow{A^+}w_{i+1}$.
 \end{move}
  This move is symmetric to Move 1.
  
  \begin{move}
   Replace $w_{i-1}\xrightarrow{T^+} w_i\xrightarrow{A^-}w_{i+1}$ with either $w_{i-1}\xrightarrow{A^-} w'_i\xrightarrow{T^+}w_{i+1}$ or $w_{i-1}\xrightarrow{A^+} w_{i+1}$.
  \end{move}
 Indeed, if we have $w_{i-1}\xrightarrow{T^+} w_i\xrightarrow{A^-}w_{i+1}$ then $w_{i-1}=uv$ and $w_{i}=ux\bar xv$ for some $u,v\in X,\ x\in Q$. There are now two cases: either $w_{i+1}$ is obtained from $w_i$ by removing a letter of $u$ or $v$, or it is obtained by removing either $x$ or $\bar x$. In the first case we can clearly find a $w_i'$ such that $w_{i-1}\xrightarrow{A^-} w'_i\xrightarrow{T^+}w_{i+1}$; in the second case we can skip $w_i$ altogether because $w_{i-1}\xrightarrow{A^+} w_{i+1}$. In either case, we obtained a path from $w_0$ to $w_n$ having the same total length (possibly with fewer steps).
 
 \begin{move}
  Replace $w_{i-1}\xrightarrow{A^+} w_i\xrightarrow{T^-}w_{i+1}$ with either $w_{i-1}\xrightarrow{T^-} w'_i\xrightarrow{A^+}w_{i+1}$ or $w_{i-1}\xrightarrow{A^-} w_{i+1}$.
 \end{move}
 
 This is symmetric to Move $3$.
 
 \begin{move}
  Replace $w_{i-1}\xrightarrow{A^+} w_i\xrightarrow{A^-}w_{i+1}$ with either $w_{i-1}$ or $w_{i-1}\xrightarrow{A^-} w_i'\xrightarrow{A^+}w_{i+1}$.
 \end{move}
 
  In fact, $w_{i-1}\xrightarrow{A^+} w_i\xrightarrow{A^-}w_{i+1}$ means that $w_{i-1}=uv$ and $w_{i}=uxv$ for some $u,v\in X,\ x\in Q$. Now, either $w_{i+1}$ is obtained from $w_i$ by removing $x$ or by removing some other letter. In the first case $w_{i+1}=w_{i-1}$ and the whole sequence can be removed from the path, in the second case there exists $w_i'$ such that $w_{i-1}\xrightarrow{T^+} w_i'\xrightarrow{A^-}w_{i+1}$.
  
 \begin{move}
  Replace $w_{i-1}\xrightarrow{T^+} w_i\xrightarrow{T^-}w_{i+1}$ with either $w_{i+1}=w_{i-1}$ or $w_{i-1}\xrightarrow{T^-} w_i'\xrightarrow{T^+}w_{i+1}$.
 \end{move}
 
 If $w_{i-1}\xrightarrow{T^+} w_i\xrightarrow{T^-}w_{i+1}$ then $w_{i-1}=uv$ and $w_{i}=ux\bar xv$ for some $u,v\in X,\ x\in Q$. There are now four cases: either $T^-$ is the inverse of $T^+$, or $u=u'\bar x$ and $w_{i+1}=u'\bar xv$, or $v=xv'$ and $w_{i+1}=uxv$, or $T^-$ is removing a cancelling pair that does not involve $x$ nor $\bar x$. In the first three cases $w_{i+1}=w_{i-1}$ and the sequence can thus be omitted. In the last case we can find an $w_i'$ such that $w_{i-1}\xrightarrow{T^-} w_i'\xrightarrow{T^+}w_{i+1}$. 

\

It is now simple to complete the proof of the lemma. Applying Moves $1,\ 3$ and $5$ multiple times, we can assume that all the labels $A^-$ in the path appear at the beginning. Using Moves $2$ and $4$ we can then move all the labels $A^+$ to the end of the path. 

Note that when using Move $4$ an $A^-$ might appear, in which case one has to move the label $A^-$ to the beginning of the path. While doing so, an extra $A^+$ might appear as a result of Move $3$. However, each time this happens the number of edges in the path decreases because we removed some edges labeled $T^\pm$. It follows that this process must terminate.

It is now enough to use Move $6$ to rearrange the labels $T^\pm$ as desired.
\end{proof}

Recall that a cancellation move in $F_S$ is defined by cancelling a letter and reducing the resulting word, while an addition move is the inverse operation. Further recall that Lemma~\ref{lem:p2:reordering moves} states:

\begin{lem*}
 Given $w,w'\in F_S$, the distance $d_{\overline{S}} (w,w')$ can always be realized by a sequence of moves $M_1,\ldots,M_n$ such that all the cancellations are performed first.
\end{lem*}

We can now prove it easily using Lemma~\ref{lem:appendix:rearranging.labels}.

\begin{proof}[Proof of Lemma~\ref{lem:p2:reordering moves}]
 By induction on $d_{\overline{S}} (w,w')$: if $d_{\overline{S}} (w,w')=0$ there is nothing to prove. Assume that $d_{\overline{S}} (w,w')>0$. 
 Choose a path in $X$ as by Lemma~\ref{lem:appendix:rearranging.labels} which realizes the distance between $w$ and $w'$. Swapping the role of $w$, $w'$ and reversing the path if necessary, we may assume that the first edge is labeled $A^-$. Let $w_1\in X$ be the word obtained after the first step and consider $\rd(w_1)\in F_S$: then $\rd(w_1)$ is obtained from $w$ by a cancellation move.
 Since
 \[
  d_{\overline{S}} (w,w')= d_{X} (w,w') = 1+ d_{X} (w_1,w') = 1+ d_{\overline{S}} (\rd(w_1),w'),
 \]
 it follows by the induction hypothesis that we can realize the distance from $\rd(w_1)$ to $w'$ with a sequence of cancellation moves followed by addition moves. Since the move from $w$ to $\rd(w_1)$ is a cancellation, this proves the lemma.
\end{proof}

\begin{rmk}
By Lemma~\ref{lem:appendix:rearranging.labels}, we also deduce that for any $w\in F_S$ the cancellation length $\abs{w}_{\overline{S}}$ is equal to the length of shortest path of the form:
 \[
  w=w_0\xrightarrow{A^-}\cdots \xrightarrow{A^-}w_j\xrightarrow{T^-}\cdots \xrightarrow{T^-}\emptyset.
 \]
 This implies that we can define the cancellation length by marking a set of letters to remove and then doing all the cancellations at the same time (as opposed to removing one letter at the time and immediately reducing the result).
\end{rmk}

\section{Proper Coarse Actions}
\label{sec:appendix:proper coarse actions}
Here we will briefly close a circle of ideas by discussing properness and the Milnor--Schwarz Lemma in a general framework.
We do not claim any novelty here: the results we give for coarsified set-groups are very well known (see \emph{e.g}
\cite{brodskiy_svarc-milnor_2007,brodskiy2008coarse,cornulier2012quasi,cornulier2016metric,nicas2012coarse,rosendal2017coarse})
and their extensions to coarse groups are relatively straightforward using the techniques we developed in the main body of text.
Our goal is mainly to set some notation and to show how our conventions translate to more traditional approaches.

Informally, properness is the property of preserving boundedness under taking pre\=/images. The most appropriate language to describe boundedness is that of bornologies.

\begin{de}\label{def:appendix:bornology}
 A \emph{bornology}\index{bornology} $\fkB$\nomenclature[:BOR]{$\fkB$}{bornology} on $X$ is a collection of subsets of $X$ (called \emph{bounded} sets) such that:

 \begin{enumerate}[(B1)]
  \item $\fkB$ covers $X$,
  \item $\fkB$ is closed under taking subsets,
  \item $B_1\cup B_2\in\fkB$ for all $B_1, B_2\in\fkB$ with $B_1\cap B_2\neq \emptyset$.
 \end{enumerate}
 A function between bornological spaces $f\colon (X_1,\fkB_1)\to(X_2,\fkB_2)$ is \emph{bornological}\footnote{%
    Such a map is often called `bounded', but in our setting we find this terminology confusing.} \index{bornological function}
if $f(B_1)\in\fkB_2$ for every $B_1\in\fkB_1$. We also say that $f$ is \emph{proper}\index{proper!function} if $f^{-1}(B_2)\in\fkB_1$ for every $B_2\in\fkB_2$ .
\end{de}

\begin{rmk}
 In the literature, condition (B3) is usually replaced by the stronger requirement that $\fkB$ be closed under finite unions (\emph{e.g.}\ {\cite[Chapter 3]{bourbaki1987vector}}). The latter condition is only equivalent to our (B3) if every finite set belongs to $\fkB$. This is generally the case in the classical examples of bornological spaces. In this work we use the weaker notion (B3) to deal with disconnected coarse spaces.
\end{rmk}

If $\CE$ is a coarse structure on $X$, the set of $\CE$\=/bounded sets is a bornology. That is, to every coarse space is naturally associated a unique bornology (the converse is not true: a given bornology can be induced by distinct coarse structures).
A controlled map between coarse spaces $f\colon(X_1,\CE_1)\to(X_2,\CE_2)$ is bornological with respect to the induced bornologies. Moreover, $f$ is proper if and only if it is proper in the sense of bornologies. For these reasons we can mix coarse structures and bornologies, so that it makes sense to say that a function $f\colon(X,\fkB)\to(Y,\CE)$ from a bornological space to a coarse space is bornological.

The product of two bornological spaces $(X_1,\fkB_1)\times (X_2,\fkB_2)$ is the Cartesian product $X_1\times X_2$ equipped with the bornology $\fkB_1\otimes\fkB_2$ generated by the products $B_1\times B_2$ where $B_1\in\fkB_1$ and $B_2\in\fkB_2$. Equivalently, $B\in\fkB_1\otimes \fkB_2$ if and only if $\pi_1(B)\in\fkB_1$ and $\pi_2(B)\in\fkB_2$.

\begin{de}\label{def:appendix:proper action}
 Let $(G,\CE)$ be a coarse group and let $\mathfrak{B}$ be a bornology on $G$. A coarse action $\crse \alpha \colon (G,\CE)\curvearrowright(Y,\CF)$ is \emph{$\fkB$\=/proper}\index{coarse action!$\fkB$\=/proper} if $(\alpha,\pi_{Y})\colon (G,\fkB)\times (Y,\CF)\to(Y,\CF)\times (Y,\CF)$ is bornological and proper, where $\pi_{Y}$ denotes the projection onto $Y$. Since $\CF$ is a coarse structure, this definition does not depend on the choice of representative for $\crse \alpha$.
\end{de}

\begin{exmp}\label{exmp:appendix:proper action discrete group}
 Let $G$ be a set\=/group and $\crse G= (G,\mincrs)$ a trivially coarse group. In Section~\ref{sec:p1:Milnor_svarc} we gave a temporary naive definition of properness by declaring that a coarse action $\crse{\alpha\colon G\curvearrowright Y}$ of the trivially coarse group is proper if the set
\(
 \braces{g\in G\mid (g\cdot B)\cap B\neq \emptyset}
\)
is finite for every bounded $B\subseteq Y$. Let $\fkF$ be the bornology of all finite subsets of $G$. Then one can verify that $\crse{\alpha\colon G\curvearrowright Y}$ is proper in the naive sense if and only if it is $\fkF$\=/proper.
\end{exmp}

Let $\alpha\colon (G,\CE)\curvearrowright (Y,\CF)$ define a coarse action.
Recall that for any $y\in Y$ we denote by $\alpha_y\colon G\to Y$ the orbit map and by $\alpha_y^*(\CF)$ the pull\=/back coarse structure on $G$ (Subsection~\ref{sec:p1:coarse.orbits}).
It is then easy to verify that $\crse{\alpha}$ if a $\fkB$\=/proper coarse action then $\fkB$ coincides with the family of $\alpha_y^*(\CF)$\=/bounded sets.
In particular, $\fkB$ must contain the bornology of $\CE$\=/bounded sets.

\begin{rmk}
 The definition of $\fkB$\=/proper coarse actions is designed to describe those coarse actions whose orbit maps are not proper. The bornology $\fkB$ ought to be thought of as a way to keep track of those (possibly $\CE$\=/unbounded) subsets of $G$ that are sent to bounded sets in $Y$. One may also think of the bornology $\fkB$ of a $\fkB$\=/proper coarse action as a sort of ``point\=/stabilizer''.
\end{rmk}

Since equi left invariant coarse structures on a coarse group $\crse G$ are determined by their bounded sets,
one can show that the notion of $\fkB$\=/properness of coarse actions completely determines the equivariant coarse geometry of the orbits.
Namely, let $\crse G = (G,\CE)$ be a coarse group, $\fkB$ be a bornology on $G$ and let $\varcrs[left]{\CE,\fkB}$ be the smallest equi left invariant coarse structure containing $\fkB$ and $\CE$. Explicitly,
\[
 \varcrs[left]{\CE,\fkB} = \angles{\Delta_{G}\ast\CR_\fkB\, ,\, \CE}
\]
where $\CR_\fkB\coloneqq\braces{B\times B\mid B\in\fkB}$. Then the following holds true.

\begin{prop}\label{prop:appendix:fkB_proper coarse action}
 If there exists a $\fkB$\=/proper coarse action of $(G,\CE)$ then $\fkB$ must coincide with the set of $\varcrs[left]{\CE,\fkB}$\=/bounded sets.
 Given such a $\fkB$, a coarse action $\crse \alpha \colon (G,\CE)\curvearrowright(Y,\CF)$ is $\fkB$\=/proper if and only if
 \(
    \alpha_y^{*}(\CF)=\varcrs[left]{\CE,\fkB}
 \)
 for every fixed $y\in Y$.
\end{prop}

This implies that cobounded coarse actions are ``determined locally''
(recall that a coarse action $\crse{\alpha\colon G\curvearrowright Y}$ is cobounded if and only if the orbit maps are coarsely dense).
Namely, Proposition~\ref{prop:appendix:fkB_proper coarse action} implies that, if it exists, a $\fkB$\=/proper cobounded coarse action of $(G,\CE)$ must be isomorphic to $\cop \colon (G,\CE)\curvearrowright(G,\varcrs[left]{\CE,\fkB})$.
In particular, such an action is unique up to isomorphism.

In analogy with Section~\ref{sec:p1:determined locally}, there is a characterisation for the bornologies $\fkB$ on a coarsified set\=/group $\crse G=(G,\CE)$
which admit $\fkB$\=/proper actions. Namely, let $\fkB_e$ be the set of $\fkB$\=/bounded neighborhoods of the identity.
Then one can verify that there exists a $\fkB$\=/proper $\crse G$\=/action if and only if $\fkB=\pts{G}\ast\fkB_e$ and
the family $\fkB_e$ contains the $\CE$\=/bounded neighborhoods of the identity and it satisfies 
\begin{itemize}
 \item[(U0)] $e\in U$ for every $U\in \fkB_e$.
 \item[(U1)] if $U\in\fkB_e$ and $e\in V\subseteq U$ then $V\in \fkB_e$.
 \item[(U2)] if $U\in\fkB_e$ then $U^{-1}\in\fkB_e$;
 \item[(U3)] if $U_1,\ U_2\in\fkB_e$ then $U_1\ast U_2\in\fkB_e$.
\end{itemize}

\

It is instructive to use this language to extend the Milnor--Schwarz Lemma to actions of more general topological groups
(this is not an original idea, see for instance the books \cite{cornulier2016metric,rosendal2017coarse} and references therein).
Recall that, a continuous action $\alpha\colon G\curvearrowright Y$ of a topological group $G$ is \emph{topologically proper}\index{proper!action (topologically)} if
\[
 (\alpha\times\pi_{Y})\colon G\times Y\to Y\times Y
\]
is a proper map (preimage of compact sets is compact).
% Let $G$ be a set\=/group, $Y$ a topological space and $\alpha\colon G\curvearrowright Y$ an action by homeomorphisms.
We saw in Example~\ref{exmp:p1:topological coarse structure abelian} that we may define an equi left invariant coarse structure $\CF_{\rm cpt}^{\rm left}$ on $Y$ letting
\[
 \CF_{\rm cpt}^{\rm left} \coloneqq\bigbraces{F\subseteq Y\times Y\mid \exists K\subseteq Y\text{ compact s.t. }\forall(y_1,y_2)\in F,\ \exists g\in G\text{ with } y_1,y_2\in g(K)}.
\]
By construction, $\alpha$ defines a coarse action of the trivially coarse group $(G,\mincrs)\curvearrowright(Y,\CF_{\rm cpt}^{\rm left})$.
Let $\fkK$\nomenclature[:BOR]{$\fkK$}{bornology of relatively compact subsets} be the bornology of relatively compact subsets of $G$ it is simple to prove that
if the action is topologically proper the coarse action $(G,\mincrs)\curvearrowright(Y,\CF_{\rm cpt}^{\rm left})$ is $\fkK$\=/proper.
If $Y$ is Hausdorff the converse is also true.

Note that if $\alpha$ is a topologically proper action, the pull\=/back of $\CF_{\rm cpt}^{\rm left}$ under the orbit map $\alpha_y$ is equal to $\varcrs[left]{cpt}$.
Also note that the coarse action $(G,\mincrs)\curvearrowright (Y,\CF_{\rm cpt}^{\rm left})$ induced by a topological action is cobounded\index{coarse action!cobounded} (Definition~\ref{def:p1:cobounded action}) if and only if there exists a compact set $K\subseteq Y$ such that $\pts{G}\cdot K=\braces{g(K)\mid g\in G}$ covers $Y$.
That is, if and only if $\alpha$ is cocompact. This shows that if $G\curvearrowright Y$ is a proper cocompact action of a topological group, then the orbit map induces an equivariant coarse equivalence between $(G,\varcrs[left]{cpt})$ and $(Y,\CF_{\rm cpt}^{\rm left})$.
This generalizes the observation at the heart of Section~\ref{sec:p1:Milnor_svarc}.

To be more explicit, let $(Y,d)$ be a metric space and denote by $\CE_d$ the induced coarse structure. Recall that a set\=/group action $ G\curvearrowright Y$ induces a coarse action on $(Y,\CE_d)$ if and only if it is an action by ``uniform coarse equivalences'' (see Example~\ref{exmp:p1:metric action by unif coarse eq}). If $\alpha\colon G\curvearrowright Y$ is a continuous action of a topological group which does induce a coarse action of $(G,\mincrs)$ on $(Y,\CE_d)$, then $\CF_{\rm cpt}^{\rm left}\subseteq \CE_d$ because compact sets have finite diameter and the latter is equi left equivariant. If $Y$ is a \emph{proper} metric space (\emph{i.e.}\ closed balls are compact) then $\CE_d\subseteq \CF_{\rm cpt}^{\rm left}$ as well.

\begin{cor}\label{cor:appendix:topological Milnor_svarc_I}
 If $G$ is a topological group, $(Y,d)$ a proper metric space and $\alpha\colon G\curvearrowright Y$ is a continuous action by uniform coarse equivalences that is proper and cocompact, then the orbit map induces a coarsely\=/equivariant coarse equivalence between $(G,\varcrs[left]{cpt})$ and $(Y,\CE_d)$.
\end{cor}

\begin{exmp}
 If $G$ is a set\=/group that is finitely generated by some set $S$, then $\varcrs[left]{cpt}$ is the metric coarse structure associated with the left\=/invariant word metric $d_S$ (see also Example~\ref{exmp:appendix:proper action discrete group}).

 If $G\curvearrowright(Y,d)$ is an isometric action, then $(G,\mincrs)\curvearrowright(Y,d)$ defines a coarse action of the trivially coarse group. If $(Y,d)$ is a proper metric space and the isometric action is also proper and cocompact, then Corollary~\ref{cor:appendix:topological Milnor_svarc_I} implies that the orbit map defines a coarse equivalence between $(G,d_S)$ and $(Y,d)$.
\end{exmp}

The above example shows that Corollary~\ref{cor:appendix:topological Milnor_svarc_I} generalizes one half of the statement of the Milnor--Schwarz Lemma (Corollary~\ref{cor:p1:Milnor-Svarc}).
As remarked at the end of Section~\ref{sec:p1:Milnor_svarc}, the other parts of the statement have much to do with coarse geodesicity.
It can be generalized as follows (compare with \cite[Section 4.C]{cornulier2016metric} and \cite[Chapter 2]{rosendal2017coarse}):

\begin{prop}\label{prop:appendix:topological Milnor_svarc_II}
 Let $G$ be a Hausdorff topological group and $(Y,\CE_d)$ a quasi\=/geodesic proper metric space. If there exists a continuous proper and cocompact action $\alpha\colon G\curvearrowright Y$ by uniform coarse equivalences, then $G$ is compactly generated. Moreover, for any compact generating set $K$ and $y\in Y$ the orbit map $\alpha_y\colon (G,d_K)\to (Y,d)$ is a quasi\=/isometry.
\end{prop}

\begin{rmk}
  The above proposition truly is a generalization of the Milnor--Schwarz  Lemma to general locally compact groups: the original lemma is obtained as the special case where $G$ is a discrete group.
  A proof is also implicit in \cite{brodskiy_svarc-milnor_2007,brodskiy2008coarse}.
\end{rmk}